%% file: EyalKaplan_phdthesis.tex
\documentclass[11pt]{amsart}
\usepackage[margin=1.4in]{geometry}
\usepackage{amsmath,amsfonts,amssymb,mathdots,mathrsfs,esint,mathtools}
\usepackage{setspace}
\usepackage{verbatim}
\usepackage{enumitem}
\usepackage{url}

\usepackage{thesis}

\begin{document}

\newtheorem{conjecture}{Conjecture}[chapter]
\newtheorem{theorem}{Theorem}[chapter]
\newtheorem{proposition}{Proposition}[chapter]
\newtheorem{corollary}[proposition]{Corollary}
\newtheorem{lemma}[proposition]{Lemma}
\newtheorem{claim}[proposition]{Claim}
\theoremstyle{remark}
\newtheorem{remark}{Remark}[chapter]
\newtheorem{example}{Example}[chapter]
\theoremstyle{definition}
\newtheorem{definition}{Definition}[chapter]
\numberwithin{equation}{chapter}

\pagestyle{empty}
\thispagestyle{empty}

\input{thesis_notations}
\begin{center}
\LARGE{Tel Aviv University}

\LARGE{Raymond and Beverly Sackler Faculty of Exact Sciences}

\LARGE{School of Mathematical Sciences}

\vspace{2.9cm}

\renewcommand{\bfdefault}{b}\fontfamily{\familydefault}
\huge{\textbf{On the Local Theory of Rankin-Selberg Convolutions for
$\mathbf{SO_{2\smallidx}\times\GL{\bigidx}}$}}\renewcommand{\bfdefault}{bx}\fontfamily{\familydefault}

\vspace{2.7cm}

\LARGE{Thesis submitted for the degree

``Doctor of Philosophy"

by

Eyal Kaplan}

\vspace{1.8cm}

\LARGE{Prepared under the supervision of

Professor David Soudry}

\vspace{4.2cm}

\LARGE{Submitted to the Senate of Tel Aviv University}

\LARGE{January 2013}
\end{center}

\newpage
\acknowledgements

First and foremost, I would like to express my deepest gratitude to
my advisor Professor David Soudry, who agreed a few years ago to
guide me through the maze of Automorphic Forms, and who has
literally taught me almost all Algebra I know. I am extremely
fortunate to have worked with a rare person, whose wisdom, knowledge
and patience have inspired me, in my research as well as in my daily
life.

It is somewhat strange that I ended up writing a thesis in
mathematics. Years ago, on a trip to South America, I made up my
mind to enroll in psychology at Tel Aviv University and become a
psychologist. It was my family and friends, especially Dr. Tomer
Shlomi, who suggested I enroll in computer science as well. The fact
that my M.Sc. was in computer science probably proves they were
right.

My interest in mathematics has slowly evolved throughout this time,
as I was taking courses in Group and Field Theory, of Professor
David Soudry, mostly for fun. I was then exposed to the miraculous
teaching of David, making the material interesting and simple at the
same time. Near the completion of my M.Sc., I had doubts about my
place and direction. My mother Ora is to blame for me deciding to
continue my studies in mathematics. This turned out to be the best
place for me to be.

I would like to take this opportunity to thank a remarkable group of
people whom I am happy to call colleagues, Dr. Dani Szpruch, Zemer
Kosloff, Edva Roddity, Adam and Lena Gal, Dr. Eli Leher, and Dr.
Tomer Shlomi. Over the years we shared our thoughts, doubts,
frustration and many cups of coffee. It was Dani, who over a phone
call introduced me to the story of the $p$-adic numbers. A short
while ago, in another phone call, this time overseas, Dani
encouraged me to continue on the path I chose.

I thank my family and friends for their love and support throughout
this journey. In particular, my parents Ora and Doron who have
helped, encouraged and inspired me, and my sister Sigal, who I
admire: whenever I needed to look at something of true quality, I
looked at your works. I thank Shachar Agmon, Nati Agmon, Gil Haim
Amit, Nachum Eibschutz, Shmulik Nachum, Yael and Sophie for the laughs we have
had, helping me keep my spirits up. Last but not least, I thank my
lovely wife Tali, who has to put up with me daily.

\vspace{48 mm}

\hspace{50 mm}\emph{A.J.: ``... focus on the good times."}

\hspace{50 mm}\emph{Tony Soprano: ``Don't be sarcastic."}

\hspace{50 mm}\emph{A.J.: ``Isn't that what you said one time..."}

\hspace{50 mm}\emph{Tony Soprano: ``... Well, it's true, I guess." \cite{Sopranos}}

\newpage
\abstract

The Rankin-Selberg method for studying Langlands' automorphic
$L$-functions is to find integral representations, involving certain Fourier coefficients of cusp forms and
Eisenstein series, for these
functions. 
In this thesis we develop the local theory for generic
representations of special orthogonal groups. We study the local
integrals for $SO_{2l}\times GL_{n}$, where $SO_{2l}$ is the special
even orthogonal group, either split or quasi-split, over a local
\nonarchimedean\ field. These integrals admit a functional equation,
which is used to define a $\gamma$-factor. We show that, as
expected, the $\gamma$-factor is identical with
Shahidi's $\gamma$-factor. The analytic properties of the integrals
are condensed into a notion of a greatest common divisor (g.c.d.).
We establish certain bounds on the g.c.d. and relate it to the
$L$-function defined by Shahidi in several cases, thereby providing
another point of view on the $L$-function, linking it to the poles
of the integrals. In particular, in the tempered case under a
reasonable assumption the g.c.d. is equal to the $L$-function.
Finally, this study includes the computation of the integrals with
unramified data. This work may lead to new applications of the
descent method, as well as aid in analyzing the poles of the global
$L$-function.


\newpage\pagenumbering{roman}\pagestyle{plain}\setcounter{page}{3}
\renewcommand{\contentsname}{\Large \textbf{Contents}}

 \doublespacing \tableofcontents \singlespacing

\newpage\pagenumbering{arabic}
\pagestyle{plain}

\input{chapter_introduction}

\input{chapter_prelim}
\input{chapter_the_integrals}

\input{chapter_uniqueness}

\input{chapter_basic_properties}

\input{chapter_local_factors}

\input{chapter_gamma_mult}

\input{chapter_archimedean_results}

\input{chapter_gcd}

\input{chapter_gcd_upper_bounds}

\begin{flushleft}
\clearpage {\LARGE \textbf{References}}\vspace*{12pt}
\end{flushleft}


\providecommand{\bysame}{\leavevmode\hbox
to3em{\hrulefill}\thinspace}
\providecommand{\MR}{\relax\ifhmode\unskip\space\fi MR }
\providecommand{\MRhref}[2]{%
  \href{http://www.ams.org/mathscinet-getitem?mr=#1}{#2}
} \providecommand{\href}[2]{#2}

\end{document}

%% file: thesis_notations.tex


\newcommand{\supercuspidal}{supercuspidal} 
\newcommand{\nonzero}{nonzero} 
\newcommand{\nonsplit}{non-split} 
\newcommand{\nontrivial}{nontrivial} 
\newcommand{\quasisplit}{quasi-split} 
\newcommand{\semidirect}{semi-direct} 
\newcommand{\nonisotropic}{non-isotropic} 
\newcommand{\nonnegative}{non-negative} 
\newcommand{\nonarchimedean}{non-Archimedean} 
\newcommand{\archimedean}{Archimedean} 
\newcommand{\pointwise}{point-wise} 
\newcommand{\rhs}{right-hand side} 
\newcommand{\lhs}{left-hand side} 
\newcommand{\equalun}{\simeq} 

\newcommand{\Integers}{\mathbb{Z}}
\newcommand{\Naturals}{\mathbb{N}}
\newcommand{\R}{\mathbb{R}}
\newcommand{\C}{\mathbb{C}}
\newcommand{\Adele}{\mathbb{A}}
\newcommand{\half}{\frac 1 2}   
\newcommand{\quarter}{\frac 1 4}    
\newcommand{\eighth}{\frac 1 8}     


\newcommand{\setof}[2]{\{#1:#2\}} 
\newcommand{\intrange}[2]{#1,\ldots,#2} 
\newcommand{\intsupto}[1]{[#1]} 
\newcommand{\sequence}[1]{\{#1\}} 
\newcommand{\card}[1]{\left|#1\right|} 

\newcommand{\frestrict}[2]{#1|_{#2}} 
\newcommand{\support}[1]{supp(#1)} 
\newcommand{\psupport}[2]{supp^{#1}(#2)} 


\newcommand{\Span}[2]{{Span}_{#1}#2} 
\newcommand{\Stab}[2]{{Stab}_{#1}(#2)} 
\newcommand{\lmodulo}[2]{#1 \backslash #2} 
\newcommand{\blmodulo}[2]{#1 \Big\backslash #2} 
\newcommand{\orbit}[1]{\mathcal{O}(#1)} 
\newcommand{\splitgrp}[1]{{#1}} 
\newcommand{\quasisplitgrp}[1]{#1^{ns}} 
\newcommand{\rmodulo}[2]{#1 / #2} 
\newcommand{\rconj}[1]{{}^{#1}} 
\newcommand{\algebricroot}{\epsilon}
\newcommand{\norm}[1]{||#1||} 
\newcommand{\supnorm}[1]{||#1||_{\infty}} 


\newcommand{\dualrep}[1]{\widetilde{#1}} 
\newcommand{\intertwining}[2]{M(#2)} 
\newcommand{\intertwiningfull}[2]{M(#1,#2)} 
\newcommand{\nintertwining}[2]{M^*(#2)} 
\newcommand{\nintertwiningfull}[2]{M^*(#1,#2)} 
\newcommand{\nintertwiningrep}[2]{M^*(#1)} 
\newcommand{\cinduced}[3]{Ind_{#1}^{#2}({#3})} 
\newcommand{\induced}[3]{Ind_{#1}^{#2}({#3})} 
\newcommand{\isomorphic}{\cong} 
\newcommand{\Whittaker}[2]{\mathcal{W}(#1,#2)} 
\newcommand{\whittakerfunctional}{\Upsilon} 
\newcommand{\homspace}[3]{Hom_{#1}({#2},{#3})} 
\newcommand{\jacquet}[2]{\mathcal{J}_{#1}({#2})} 
\newcommand{\SObyGL}{SO_{2n}\times GL_n}
\newcommand{\inGLn}[1]{\underline{#1}} 
\newcommand{\bigidx}{n}     
\newcommand{\smallidx}{l}   
\newcommand{\specialvec}{e_{\gamma}}    
\newcommand{\absdet}[1]{|\det#1|} 


\newcommand{\Mat}[1]{M_{#1}} 
\newcommand{\MatF}[2]{M_{#1}(#2)} 
\newcommand{\GL}[1]{GL_{#1}} 
\newcommand{\GLF}[2]{GL_{#1}(#2)} 
\newcommand{\matrow}[2]{{#1}_{#2,\rightarrow}} 
\newcommand{\matcol}[2]{{#1}_{\downarrow,#2}} 
\newcommand{\transpose}[1]{\rconj{t}#1} 
\newcommand{\idblk}[1]{I_{#1}} 


\newcommand{\symgroup}[1]{S_{#1}} 
\newcommand{\rootgroup}[1]{R_{#1}} 
\newcommand{\liftG}[1]{{#1}^{\wedge}} 
\newcommand{\liftH}[1]{\breve{#1}} 

\newcommand{\logvalnoparam}{\vartheta} 
\newcommand{\logval}[1]{\vartheta(#1)} 
\newcommand{\padicintegers}{\mathcal{O}} 

\newcommand{\nequiv}{\not\equiv}
\newcommand{\dotcup}{\coprod}

\newcommand{\textseparator}{********************************************************************} 

%% file: chapter_introduction.tex
\newtheorem{conjecture}{Conjecture}[section]
\newtheorem{theorem}{Theorem}[section]
\newtheorem{proposition}{Proposition}[section]
\newtheorem{corollary}{Corollary}[section]
\newtheorem{lemma}{Lemma}[section]
\newtheorem{claim}{Claim}[section]
\theoremstyle{remark}
\newtheorem{remark}{Remark}[section]
\newtheorem{example}{Example}[section]
\theoremstyle{definition}
\newtheorem{definition}{Definition}[section]
\numberwithin{equation}{section}
\newcommand{\chapter}{\section} 
\input{thesis_notations}
\end{comment}

\chapter{Introduction}\label{chapter:introduction}
Approximately 35 years ago, Piatetski-Shapiro envisioned a converse
theorem for $\GL{\bigidx}$, i.e., a necessary condition for gluing
together local representations into an automorphic representation.
The analytic properties of $L$-functions for pairs of
representations of $GL_n\times GL_m$, with $m\leq n$, were essential
for his approach. This was the motivation for the development of the
Rankin-Selberg theory for $GL_n\times GL_m$ by Jacquet,
Piatetski-Shapiro and Shalika \cite{JPSS}.

Since then several converse theorems for $\GL{\bigidx}$ were
formulated and proved (see 
\cite{Cog3,Cog2}). The focus of
research has turned to questions of Langlands' (local and global)
functoriality (
\cite{Co}). Functoriality is basically a means for reducing problems
of representations of classical groups $G$ to problems of
representations of $GL_n$. A converse theorem along with a theory of
representations of $G\times GL_m$ can be used to produce
functoriality results (e.g. \cite{CKPS}). The study of
representations of $G\times GL_m$ is conducted using local and
global $L$-functions. There are essentially two known methods for
studying $L$-functions, the Langlands-Shahidi method and
Rankin-Selberg method (see 
\cite{Bu}).

The Rankin-Selberg method for studying Langlands' automorphic
$L$-functions is to find integral representations for these
functions. We briefly account the steps of the method, following
Cogdell \cite{Cog5}. The starting point is a global integral
tailored for the automorphic representations at hand. This integral
admits a factorization into an Euler product of local factors,
called the local Rankin-Selberg integrals. In order to relate the
global integral to the $L$-function, one computes the local
integrals with unramified data and shows that they produce local
$L$-functions. This is roughly sufficient to determine the analytic
properties of the partial $L$-function. In order to study the
global $L$-function, the local integrals at the finite ramified
and \archimedean\ ramified places must be studied. The
local analysis typically involves a functional equation and local
factors, namely $L$, $\gamma$ and $\epsilon$-factors.

Gelbart and Piatetski-Shapiro \cite{GPS} developed the
Rankin-Selberg theory for $G\times GL_m$, where $G$ is a split
classical group of rank $m$. Their method was extended by Ginzburg
\cite{G} to the split group $G=SO_{l}$ with $\lfloor l/2\rfloor\geq
m$ and by Soudry \cite{Soudry} to the split group $G=SO_{2l+1}$ with $l<
m$. These constructions (as well as \cite{JPSS}) were all limited to
generic representations. Ginzburg, Piatetski-Shapiro and Rallis
\cite{GPSR} developed Rankin-Selberg integrals for a pair of
representations of $O_m\times\GL{\bigidx}$, where $O_m$ is the
orthogonal group with respect to an arbitrary non-degenerate
quadratic form over a number field, and the representation of $O_m$
is not necessarily generic.

The importance of the Rankin-Selberg method in current research is that within its
frame, it is possible to locate the poles of the $L$-functions.
These are essential for describing the image of the functorial lift
(global or local) from generic representations of classical groups 
to $\GL{N}$, 
as well as in the descent construction
\cite{GRS3,JSd1,CKPS,Soudry4}.

We study the Rankin-Selberg integrals for generic representations
of $SO_{2\smallidx}\times\GL{\bigidx}$. We construct the global
integrals and show that they are Eulerian. The local unramified
factors at the finite places are computed and equated with the local $L$-functions. Then
we turn to local \nonarchimedean\ analysis and focus primarily on
defining local factors and establishing their fundamental
properties. Our work is based on the extensive study of
$SO_{2\smallidx+1}\times\GL{\bigidx}$ by Soudry
\cite{Soudry,Soudry3,Soudry2}, which in turn is based on the works
of Jacquet, Piatetski-Shapiro and Shalika \cite{JPSS,JS3}. The global Rankin-Selberg integrals and their local counterparts apply only to generic representations and so does our study of local factors. There are currently no
definitions for local factors in the general context of \cite{GPSR}.

Ginzburg, Piatetski-Shapiro and Rallis \cite{GPSR} computed the
unramified factors for their general Rankin-Selberg integrals
mentioned above. Their computation of the integral for
$SO_m\times\GL{\bigidx}$ (non-generic case) can be reduced in a
uniform fashion to our present calculation, thereby providing
another proof of their results. This will not be addressed here.

The local integrals satisfy a certain uniqueness property, which is
used to define a $\gamma$-factor. We prove that this factor is
multiplicative. Our results, combined with the \archimedean\ results
of Soudry \cite{Soudry3}, imply that our
$\gamma$-factor is identical with the corresponding $\gamma$-factor
defined by Shahidi in \cite{Sh3}.

Motivated by the work of Jacquet, Piatetski-Shapiro and Shalika
\cite{JPSS}, we define a greatest common divisor (g.c.d.) for the
local integrals. The integrals span a fractional ideal and its
unique generator, which contains any pole which appears in the
integrals, is called the g.c.d. We describe the properties of the
g.c.d. and establish upper and lower bounds for the poles.
The upper bounds are obtained using a generalization of a technique introduced by
\cite{JPSS}: we attach Laurent series to the integrals
and establish functional equations, in terms of these series, which imply the bounds.
In the tempered case we can relate the g.c.d. to the $L$-function of the
representations defined by Shahidi. Results of this study may lead
to a g.c.d. definition for the $L$-function. Such a definition
provides another point of view on the $L$-function. It is expected
to have many applications, since the poles of the integrals indicate
relations between the representations.

Our technique and results regarding the g.c.d. readily adapt to the integrals studied by
Ginzburg \cite{G} and Soudry \cite{Soudry}, due to the similar
nature and technical closeness of the constructions.

The g.c.d. can be used
at the ramified places in order to define a partial $L$-function as a product over all
finite places. Therefore, this work may have applications to the analysis of the global $L$-function. Such results were obtained by Piatetski-Shapiro and Rallis \cite{PS} and Ikeda
\cite{Ik}.

\section{Main results}\label{section:main results}
Let $F$ be a local \nonarchimedean\ field of characteristic zero.
Let $\pi$ be a smooth admissible finitely generated generic
representation of $SO_{2\smallidx}(F)$ (generic, or more specifically $\chi$-generic - has a unique Whittaker model, with respect to a character $\chi$ of a maximal unipotent subgroup). Here the group
$SO_{2\smallidx}(F)$ is either split or \quasisplit, i.e., non-split
over $F$ but split over a quadratic extension of $F$. Let $\tau$ be
a smooth admissible finitely generated generic representation of $\GLF{\bigidx}{F}$, which has a central character $\omega_{\tau}$.

We omit references to the field from the notation.
For a complex parameter $s\in \C$, let $V(\tau,s)$ be the space of the induced
representation
$\cinduced{Q_{\bigidx}}{SO_{2\bigidx+1}}{\tau\absdet{}^{s-\half}}$,
where $Q_{\bigidx}$ is a maximal parabolic subgroup with a Levi part
isomorphic to $\GL{\bigidx}$. Denote by
$\intertwiningfull{\tau}{s}:V(\tau,s)\rightarrow V(\tau^*,1-s)$ the
standard intertwining operator where $\tau^*(b)=\tau(J_{\bigidx}(\transpose{b^{-1}})J_{\bigidx})$,
$J_{\bigidx}$ is the matrix with $1$ on the anti-diagonal and $0$ elsewhere (if $\tau$ is irreducible, $\tau^*$ is isomorphic to the
representation contragredient to $\tau$). Let
$\nintertwiningfull{\tau}{s}$ be the standard normalized intertwining
operator, that is, $\intertwiningfull{\tau}{s}$ multiplied by the
local coefficient. An element $f_s\in V(\tau,s)$ is called a
standard section if its restriction to a certain fixed maximal
compact subgroup is independent of $s$. Denote by $\xi(\tau,std,s)$
the space of standard sections. Let
$\xi(\tau,hol,s)=\C[q^{-s},q^s]\otimes\xi(\tau,std,s)$ be the space
of holomorphic sections (where $q$ is the cardinality of the residue
field of $F$). The space of rational sections is defined by $\xi(\tau,rat,s)=\C(q^{-s})\otimes\xi(\tau,std,s)$. Also fix an additive \nontrivial\ character $\psi$ of $F$.
Throughout, only generic representations are considered. For any $k\geq1$, a representation $\eta$ of $\GL{k}$
is assumed to have a central character denoted by $\omega_{\eta}$.

Let $W$ be a Whittaker function for $\pi$ and let $f_s\in\xi(\tau,rat,s)$.
We construct a local Rankin-Selberg integral $\Psi(W,f_s,s)$. The exact form of the integral is given in Section~\ref{subsection:the integrals}.
It is absolutely convergent for $\Re(s)>>0$. This integral appears as a local factor in
a global Rankin-Selberg integral. The first step in relating the
global integral to the global Langlands $L$-function, is to prove
that the local integrals with unramified data evaluate to the local
$L$-functions. The following result is proved in
Section~\ref{section:Computation of the local integral with
unramified data}.
\begin{theorem}\label{theorem:unramified computation}
Let $\pi$ and $\tau$ be irreducible unramified representations and
assume that all data are unramified. In particular, $W$ and $f_s$
are the normalized unramified elements. Then for $\Re(s)>>0$,
\begin{align*}
\Psi(W,f_s,s)=\frac{L(\pi\times\tau,s)}{L(\tau,Sym^2,2s)}.
\end{align*}
Here the $L$-functions are defined using the Satake parameters of $\pi$ and $\tau$,
$Sym^2$ denotes the symmetric square representation.
\end{theorem}

One of the basic properties of the integral $\Psi(W,f_s,s)$ is that it has a meromorphic continuation to
a function in $\C(q^{-s})$. We prove this in Section~\ref{subsection:meromorphic continuation}.
Theorem~\ref{theorem:unramified computation} confirms this for the unramified case.

We turn to defining local factors. The key observation is that
the integral $\Psi(W,f_s,s)$ may
be regarded as a bilinear form satisfying certain equivariance
properties, placing it in a space which is roughly
one-dimensional. In detail,
\begin{align*}
\Psi(W,f_s,s)\in\begin{dcases}Bil_{SO_{2\smallidx}}(\pi
,V(\tau,s)_{N_{\bigidx-\smallidx},\psi_{\gamma}^{-1}})&\smallidx\leq\bigidx,\\
Bil_{SO_{2\bigidx+1}}(\pi
_{N^{\smallidx-\bigidx},\psi_{\gamma}^{-1}},V(\tau,s))&\smallidx>\bigidx.
\end{dcases}
\end{align*}
Here $Bil_{SO_{2\smallidx}}(\cdot,\cdot)$ (resp. $Bil_{SO_{2\bigidx+1}}(\cdot,\cdot)$) denotes the space of
$SO_{2\smallidx}$-equivariant (resp. $SO_{2\bigidx+1}$-equivariant) bilinear forms,
$V(\tau,s)_{N_{\bigidx-\smallidx},\psi_{\gamma}^{-1}}$ and
$\pi_{N^{\smallidx-\bigidx},\psi_{\gamma}^{-1}}$ are certain Jacquet modules, see Sections~\ref{subsection:The global integral for l<=n}
and \ref{subsection:The global integral for l>n} (resp.) for the precise definitions of the subgroups and characters.  In Chapter~\ref{chapter:uniqueness} we prove,
\begin{theorem}\label{theorem:uniqueness}
Except for a finite set of values of $q^{-s}$, the complex vector spaces $Bil_{SO_{2\smallidx}}(\pi,V(\tau,s)_{N_{\bigidx-\smallidx},\psi_{\gamma}^{-1}})$
and $Bil_{SO_{2\bigidx+1}}(\pi_{N^{\smallidx-\bigidx},\psi_{\gamma}^{-1}},V(\tau,s))$ are at most one-dimensional.
\end{theorem}
We mention that for irreducible representations, this result follows from \cite{AGRS,GGP,MW}, but here our representations may be reducible. The integral \\$\Psi(W,\nintertwiningfull{\tau}{s}f_s,1-s)$ satisfies similar equivariance properties and an immediate
consequence of the theorem, is that
these integrals are related by a functional equation. Namely, there
exists a proportionality factor $\gamma(\pi\times\tau,\psi,s)$
essentially satisfying
\begin{align*}
\gamma(\pi\times\tau,\psi,s)\Psi(W,f_s,s)=\Psi(W,\nintertwiningfull{\tau}{s}f_s,1-s).
\end{align*}
The actual definition of $\gamma(\pi\times\tau,\psi,s)$ in
Section~\ref{subsection:the gamma factor} (see \eqref{eq:gamma def})
involves a certain normalization. One of our goals is to relate the
$\gamma$-factor $\gamma(\pi\times\tau,\psi,s)$ to the corresponding $\gamma$-factor of Shahidi 
on $SO_{2\smallidx}\times\GL{\bigidx}$ defined in
\cite{Sh3}. Using standard global arguments (see \cite{Sh3}
Section~5 and \cite{Soudry2} Section~0), this problem is reduced to
the following three major milestones. Firstly, one proves that both
factors are identical when all data are unramified. This is evident
from Theorem~\ref{theorem:unramified computation} (see Claim~\ref{claim:gamma factor for unramified representations}).
Secondly, these factors should be identical over \archimedean\ fields whenever $SO_{2\smallidx}$ is split
(this is not needed for the \quasisplit\ case, because we can find a global group which is split at all \archimedean\ places). This follows from Soudry \cite{Soudry3}. Thirdly, we need to
establish full multiplicative properties for
$\gamma(\pi\times\tau,\psi,s)$ over a \nonarchimedean\ field.


In Chapter~\ref{section:gamma_mult} we prove the following theorems,
comprising the multiplicative properties. As mentioned above, throughout this work, representations of $SO_{2\smallidx}$ and $\GL{k}$ are taken to be smooth, admissible, finitely generated and generic, and representations of $\GL{k}$ are assumed to have central characters.
\begin{theorem}\label{theorem:multiplicity second var}
Let $\tau_i$ be a representation of $\GL{\bigidx_i}$ for $i=1,2$. Let $\tau$ be a
quotient of a representation parabolically induced from
$\tau_1\otimes\tau_2$. Then
\begin{align*}
\gamma(\pi\times\tau,\psi,s)=\gamma(\pi\times\tau_1,\psi,s)\gamma(\pi\times\tau_2,\psi,s).
\end{align*}
\end{theorem}
\begin{theorem}\label{theorem:multiplicity first var}
Let $\sigma$ be a representation of $\GL{k}$ and $\pi'$ be a
representation of $SO_{2(\smallidx-k)}$, $0<k\leq\smallidx$. 
Let $\pi$
be a quotient of a representation parabolically induced from
$\sigma\otimes\pi'$. Then
\begin{align*}
\gamma(\pi\times\tau,\psi,s)=\omega_{\sigma}(-1)^{\bigidx}\omega_{\tau}(-1)^k[\omega_{\tau}(2\gamma)^{-1}]\gamma(\sigma\times\tau,\psi,s)\gamma(\pi'\times\tau,\psi,s)\gamma(\sigma^*\times\tau,\psi,s).
\end{align*}
Here 
the factor
$\omega_{\tau}(2\gamma)^{-1}$ appears only when $k=\smallidx$ and
$\gamma$ is a global constant depending only on the groups. When
$k=\smallidx$ we define $\gamma(\pi'\times\tau,\psi,s)=1$. The
factors $\gamma(\sigma\times\tau,\psi,s)$ and
$\gamma(\sigma^*\times\tau,\psi,s)$ are the
$\GL{k}\times\GL{\bigidx}$ $\gamma$-factors of Jacquet,
Piatetski-Shapiro and Shalika \cite{JPSS}. 
\end{theorem}


Assume that $\pi$ and $\tau$ are irreducible and note that Shahidi
defined the $\gamma$-factor only for irreducible representations.
Then we can define a normalized $\gamma$-factor
$\Gamma(\pi\times\tau,\psi,s)$. Roughly, we normalize
$\gamma(\pi\times\tau,\psi,s)$ by the central characters appearing in Theorem~\ref{theorem:multiplicity first var}. In
the case where $\pi$ is a representation of $SO_0=\{1\}$ define
$\Gamma(\pi\times\tau,\psi,s)=1$. When $\pi$ and $\tau$ have trivial central characters
and $G_{\smallidx}$ is split, $\Gamma(\pi\times\tau,\psi,s)=\gamma(\pi\times\tau,\psi,s)$. In
Section~\ref{subsection:the gamma factor} we prove,
\begin{theorem}\label{theorem:main multiplicative properties}
Let $\sigma_1\otimes\ldots\otimes\sigma_m\otimes\pi'$ be an
irreducible representation of
$\GL{k_1}\times\ldots\times\GL{k_m}\times SO_{2(\smallidx-k)}$, with
$k=k_1+\ldots+k_m$ ($0\leq k\leq\smallidx$), and $\pi$ be an
irreducible quotient of a representation parabolically induced from
$\sigma_1\otimes\ldots\otimes\sigma_m\otimes\pi'$. Let
$\tau_1\otimes\ldots\otimes\tau_a$ be an irreducible representation
of $\GL{\bigidx_1}\times\ldots\times\GL{\bigidx_a}$ ($1\leq
a\leq\bigidx$) and $\tau$ be an irreducible quotient of a
representation parabolically induced from
$\tau_1\otimes\ldots\otimes\tau_a$. Then
\begin{align*}
&\Gamma(\pi\times\tau,\psi,s)=\prod_{i=1}^a\Gamma(\pi\times\tau_i,\psi,s),\\
&\Gamma(\pi\times\tau,\psi,s)=\Gamma(\pi'\times\tau,\psi,s)\prod_{i=1}^m\gamma(\sigma_i\times\tau,\psi,s)\gamma(\sigma_i^*\times\tau,\psi,s).
\end{align*}
\end{theorem}
In Chapter~\ref{chapter:archimedean results} we briefly describe how
to adapt the \archimedean\ results of Soudry \cite{Soudry3} to our
integrals. As explained above we conclude,
\begin{corollary}\label{corollary:gamma is identical to local coeff in split case}
For 
irreducible representations $\pi$ of $SO_{2\smallidx}$ and $\tau$ of $\GL{\bigidx}$,
$\Gamma(\pi\times\tau,\psi,s)$ is identical with Shahidi's 
$\gamma$-factor.\qed
\end{corollary}
Our definition of $\Gamma(\pi\times\tau,\psi,s)$ implies that 
$\gamma(\pi\times\tau,\psi,s)$ is identical with Shahidi's 
$\gamma$-factor, up to a \nonzero\ complex constant.

We proceed to define the g.c.d. and $\epsilon$-factor. The
notion of a g.c.d. of Rankin-Selberg integrals was introduced by
Jacquet, Piatetski-Shapiro and Shalika \cite{JPSS}, who defined the
$L$-function of the representations as the g.c.d. of their local Rankin-Selberg convolutions.

Let $\tau$ be an irreducible representation ($\pi$ may be reducible, but
finitely generated). Following the method of Piatetski-Shapiro and
Rallis \cite{PR2,PS} we define the set of ``good sections",
\begin{align*}
\xi(\tau,good,s)=\xi(\tau,hol,s)\cup
\nintertwiningfull{\tau^*}{1-s}\xi(\tau^*,hol,1-s).
\end{align*}
The Rankin-Selberg integrals $\Psi(W,f_s,s)$, where $W$ is a
Whittaker function for $\pi$ and $f_s\in\xi(\tau,good,s)$, span a
fractional ideal $\mathcal{I}_{\pi\times\tau}(s)$ of
$\C[q^{-s},q^s]$ which contains the constant $1$, and its unique generator, in the form
$P(q^{-s})^{-1}$ for $P\in\C[X]$ such that $P(0)=1$, is what we call
the g.c.d. of the integrals, $\gcd(\pi\times\tau,s)$. We prove that the g.c.d. does not depend on $\psi$. The
corresponding functional equation that defines
$\epsilon(\pi\times\tau,\psi,s)$ (an exponential) is
\begin{align*}
\frac{\Psi(W,\nintertwiningfull{\tau}{s}f_s,1-s)}{\gcd(\pi\times\tau^*,1-s)}
=\epsilon(\pi\times\tau,\psi,s)\frac{\Psi(W,f_s,s)}{\gcd(\pi\times\tau,s)}
\end{align*}
(here we omitted a normalization factor, the actual equation is given by \eqref{eq:epsilon
def}). We establish several properties of the g.c.d. including lower and
upper bounds and in certain cases we are able to calculate it
explicitly.

The $L$-function of $\pi\times\tau$ is already defined due to the
works of Shahidi on his method of local coefficients (e.g.
\cite{Sh4,Sh3}). Under a certain mild assumption we can show that
in the tempered case, the g.c.d. is equal to this $L$-function.

Throughout, tempered representations are assumed to be irreducible.
Here is our main result on the g.c.d., proved in
Section~\ref{subsection:tempered}. We recall that our representations are always assumed to be generic.
\newpage
\begin{theorem}\label{theorem:gcd for tempered reps}
Let $\pi$ and $\tau$ be tempered representations. 
Let $L(\pi\times\tau,s)$ be the 
$L$-function attached to $\pi$
and $\tau$ by Shahidi. Then $L(\pi\times\tau,s)^{-1}$ divides
$\gcd(\pi\times\tau,s)^{-1}$ and $\gcd(\pi\times\tau,s)\in
L(\pi\times\tau,s)M_{\tau}(s)\C[q^{-s},q^s]$, where
$M_{\tau}(s)$ (defined in Section~\ref{subsection:the factors M
tau s}) contains the poles of $\nintertwiningfull{\tau}{s}$ and
$\nintertwiningfull{\tau^*}{1-s}$,
$M_{\tau}(s)^{-1}\in\C[q^{-s},q^s]$.
Moreover, if the operators
$L(\tau,Sym^2,2s-1)^{-1}\intertwiningfull{\tau}{s}$ and
$L(\tau^*,Sym^2,1-2s)^{-1}\intertwiningfull{\tau^*}{1-s}$ are
holomorphic, 
\begin{align*}
\gcd(\pi\times\tau,s)=L(\pi\times\tau,s).
\end{align*}
\end{theorem}
According to a result of Casselman and Shahidi \cite{CSh}
(Theorem~5.1), $L(\tau,Sym^2,2s-1)^{-1}\intertwiningfull{\tau}{s}$ is
holomorphic for an irreducible supercuspidal $\tau$, and is expected
to be holomorphic for a tempered $\tau$. Therefore,
Theorem~\ref{theorem:gcd for tempered reps} states that in the
tempered case, under a reasonable assumption on the operators
(holding in the supercuspidal case), the g.c.d. definition gives the
``usual" $L$-function. Without the assumption, the g.c.d. and this $L$-function
agree up to the poles of $M_{\tau}(s)$.

An essentially tempered representation $\tau$ is a representation of the form $\absdet{}^v\tau_0$ where
$v\in\C$ and $\tau_0$ is tempered. For instance, an irreducible supercuspidal $\tau$ is in particular essentially tempered.
In Section~\ref{subsection:tempered} we will derive the following immediate
consequence of Theorem~\ref{theorem:gcd for tempered reps} and the definitions.
\begin{corollary}\label{corollary:tempered theorem holds for essentially tempered}
Theorem~\ref{theorem:gcd for tempered reps} also holds if $\tau$ is essentially tempered.
\end{corollary}

We restate Theorem~\ref{theorem:gcd for tempered reps} for the case where it
holds unconditionally.
\begin{corollary}\label{corollary:gcd for pi tempered tau supercuspidal split case}
Let $\pi$ be tempered and $\tau$ be irreducible supercuspidal. 
Then $\gcd(\pi\times\tau,s)=L(\pi\times\tau,s)$.\qed
\end{corollary}

For the general case 
we have the following lower and upper bounds. In
Chapter~\ref{chapter:the gcd} we show,
\begin{theorem}\label{theorem:lower bound for gcd 2}
Let $\pi$ be irreducible and $\tau$ be tempered such that
$L(\tau,Sym^2,2s-1)^{-1}\intertwiningfull{\tau}{s}$ and
$L(\tau^*,Sym^2,1-2s)^{-1}\intertwiningfull{\tau^*}{1-s}$ are
holomorphic. Write $\pi$ as a standard module - a representation
parabolically induced from a representation $\sigma\otimes\pi'$ of
$\GL{k}\times SO_{2(\smallidx-k)}$ whose exact properties are given in Section~\ref{subsection:Lower bound for pi induced tau
tempered} (in particular, $\pi'$ is
tempered). Then if $k<\smallidx$ or $k=\smallidx>\bigidx$,
\begin{align*}
L(\sigma\times\tau,s)\gcd(\pi'\times\tau,s)L(\sigma^*\times\tau,s)\in\gcd(\pi\times\tau,s)\C[q^{-s},q^s].
\end{align*}
\end{theorem}
Note that for a representation $\pi'$ of $SO_0$ we set
$\gcd(\pi'\times\tau,s)=1$. In Chapter~\ref{chapter:upper boubnds on
the gcd} we prove,
\begin{theorem}\label{theorem:gcd sub multiplicity second var}
Let $\tau$ be an irreducible representation 
parabolically induced from a representation
$\tau_1\otimes\ldots\otimes\tau_k$ of
$\GL{\bigidx_1}\times\ldots\times\GL{\bigidx_k}$. Then
\begin{align*}
\gcd(\pi\times\tau,s)\in
(\prod_{i=1}^k\gcd(\pi\times\tau_i,s))M_{\tau_1\otimes\ldots\otimes\tau_k}(s)\C[q^{-s},q^s].
\end{align*}
\end{theorem}
\begin{theorem}\label{theorem:gcd sub multiplicity first var}
Let $\pi$ be a quotient of a representation parabolically induced
from a representation $\sigma\otimes\pi'$ of $\GL{k}\times
SO_{2(\smallidx-k)}$. Let $\tau$ be an irreducible representation as
in Theorem~\ref{theorem:gcd sub multiplicity second var}, induced
from $\tau_1\otimes\ldots\otimes\tau_a$ ($a\geq1$). Then
\begin{align*}
\gcd(\pi\times\tau,s)\in
L(\sigma\times\tau,s)(\prod_{i=1}^a\gcd(\pi'\times\tau_i,s))L(\sigma^*\times\tau,s)M_{\tau_1\otimes\ldots\otimes\tau_a}(s)\C[q^{-s},q^s].
\end{align*}
Here $L(\sigma\times\tau,s)$ and $L(\sigma^*\times\tau,s)$ are the
$L$-factors of \cite{JPSS}. In the case $k=\smallidx\leq\bigidx$,
this relation holds
under the assumptions that $\sigma$ is irreducible and $\tau$ is of
Langlands' type. 
\end{theorem}
According to Zelevinsky \cite{Z3} any irreducible generic
representation $\tau$ (of $\GL{\bigidx}$) is of the form
$\cinduced{P_{\bigidx_1,\ldots,\bigidx_a}}{\GL{\bigidx}}{\tau_1\otimes\ldots\otimes\tau_a}$
for specific inducing data and a parabolic subgroup $P_{\bigidx_1,\ldots,\bigidx_a}$, i.e., $\tau$ is an induced representation
of Langlands' type (see Section~\ref{subsection:sub mult first var proof body k=l} and \cite{JS}) also called a standard module.
Similarly, any irreducible generic representation $\pi$ is (a
standard module) of the form
$\cinduced{P}{SO_{2\smallidx}}{\sigma\otimes\pi'}$ for
an irreducible generic representation $\sigma\otimes\pi'$ of
$\GL{k}\times SO_{2(\smallidx-k)}$, where $\pi'$ is tempered and $P$ is a suitable parabolic subgroup. This
follows from the standard module conjecture of Casselman and Shahidi
\cite{CSh}, proved by Mui\'{c} \cite{Mu3} for classical groups.

The factor $M_{\tau_1\otimes\ldots\otimes\tau_k}(s)$ (defined in
Section~\ref{subsection:the factors M tau s}) designates an upper
bound for the poles of $\nintertwiningfull{\tau}{s}$ and
$\nintertwiningfull{\tau^*}{1-s}$. We also prove more detailed
versions of Theorems~\ref{theorem:gcd sub multiplicity second var}
and \ref{theorem:gcd sub multiplicity first var} (see resp.,
Corollaries~\ref{corollary:refinement gcd second var bound} and
\ref{corollary:first var k<l<n exact sub mult}).

The upper bounds are obtained by associating certain Laurent series to the integrals.
An integral $\Psi(W,f_s,s)$ is represented by a series $\sum_{m\in\Integers}a_mX^m\in\C[[X,X^{-1}]]$ if for $\Re(s)>>0$,
the series $\sum_{m\in\Integers}a_mq^{-sm}$ is absolutely convergent and equals $\Psi(W,f_s,s)$. We will construct
a series, denoted by $\Sigma(W,f_s,s)$, representing $\Psi(W,f_s,s)$.
Let $P\in\C[X,X^{-1}]$ be the polynomial obtained from $\epsilon(\pi\times\tau,\psi,s)\gcd(\pi\times\tau,s)^{-1}$ by
replacing $q^{-s}$ with $X$. Let $\widetilde{P}\in\C[X,X^{-1}]$ be obtained similarly, from $\gcd(\pi\times\tau^*,1-s)^{-1}$.
The functional equation described above can be interpreted as an equation
in $\C[[X,X^{-1}]]$, roughly
\begin{align*}
\widetilde{P}\Sigma(W,\nintertwiningfull{\tau}{s}f_s,1-s)=P\Sigma(W,f_s,s).
\end{align*}
This is actually an equality of the form $\sum_{m<\widetilde{M}}b_mX^m=\sum_{m>M}a_mX^m$ for some $\widetilde{M},M\in\Integers$ and can only be valid when
both sides are polynomials. This implies in particular that the poles of $\Psi(W,f_s,s)$ appear in $P^{-1}$. Then, for example,
Theorem~\ref{theorem:gcd sub multiplicity second var} may be derived if we can replace $P$ with a polynomial obtained from
$\gcd(\pi\times\tau_1,s)^{-1}\cdot\ldots\cdot\gcd(\pi\times\tau_k,s)^{-1}$.

The idea of using Laurent series to establish upper bounds is due to Jacquet, Piatetski-Shapiro and Shalika \cite{JPSS}.
In Section~\ref{section:integrals and series} we extend their construction and provide a general framework for converting integrals into series.
One result we obtain is a certain version of Fubini's Theorem for series (Lemma~\ref{lemma:double iterated series}). The
definitions and results are stated in a general form and can
be applied to other types of Rankin-Selberg integrals, e.g.
the integrals for $SO_{2\smallidx+1}\times\GL{\bigidx}$ of \cite{G,Soudry}.

We have the following two corollaries, demonstrating that the g.c.d.
captures the correct notion of the $L$-function, except perhaps for
additional poles due to the intertwining operator
($M_{\tau}(s)$).
\begin{corollary}\label{corollary:the correct definition for pi
general} Let $\pi$ be irreducible and write it as a standard module -
a representation
parabolically induced from a representation $\sigma\otimes\pi'$ of
$\GL{k}\times SO_{2(\smallidx-k)}$ (see Section~\ref{subsection:Lower bound for pi induced tau
tempered}). Let $\tau$ be tempered. Assume that 
the operators
$L(\tau,Sym^2,2s-1)^{-1}\intertwiningfull{\tau}{s}$ and
$L(\tau^*,Sym^2,1-2s)^{-1}\intertwiningfull{\tau^*}{1-s}$ are
holomorphic.
\begin{enumerate}
\item
$\gcd(\pi\times\tau,s)\in
L(\pi\times\tau,s)M_{\tau}(s)\C[q^{-s},q^s]$.
\item If $k<\smallidx$ or $k=\smallidx>\bigidx$, $L(\pi\times\tau,s)^{-1}$ divides $\gcd(\pi\times\tau,s)^{-1}$.
\end{enumerate}
\end{corollary}
\begin{proof}[Proof of Corollary~\ref{corollary:the correct definition for pi
general}]
Since $\pi$ is a standard module and $\tau$ is tempered,
\begin{align*}
L(\pi\times\tau,s)=L(\sigma\times\tau,s)L(\pi'\times\tau,s)L(\sigma^*\times\tau,s).
\end{align*}
Under the given assumptions, Theorem~\ref{theorem:gcd for tempered reps} applied to
$\pi'\times\tau$ implies $L(\pi'\times\tau,s)=\gcd(\pi'\times\tau,s)$. Now the first assertion
follows from Theorem~\ref{theorem:gcd sub
multiplicity first var}, which is applicable to the full range of $k,\smallidx$ and $\bigidx$ because $\sigma$ is irreducible and
$\tau$ is tempered. The second assertion follows from Theorem~\ref{theorem:lower bound for gcd 2}.
\end{proof}
\begin{corollary}\label{corollary:upper bound irreducibles}
Let $\pi$ and $\tau$ be irreducible and write
$\tau=\cinduced{P_{\bigidx_1,\ldots,\bigidx_a}}{\GL{\bigidx}}{\tau_1\otimes\ldots\otimes\tau_a}$
as a representation of Langlands' type (in particular, each $\tau_i$ is essentially tempered), $a\geq1$. Assume that for each $i$,
$L(\tau_i,Sym^2,2s-1)^{-1}\intertwiningfull{\tau_i}{s}$ and
$L(\tau_i^*,Sym^2,1-2s)^{-1}\intertwiningfull{\tau_i^*}{1-s}$ are
holomorphic. 
Then
\begin{align*}
\gcd(\pi\times\tau,s)\in
L(\pi\times\tau,s)M_{\tau_1\otimes\ldots\otimes\tau_a}(s)\C[q^{-s},q^s].
\end{align*}
\end{corollary}
\begin{proof}[Proof of Corollary~\ref{corollary:upper bound irreducibles}]
Follows immediately from Theorems~\ref{theorem:gcd sub multiplicity first var} and \ref{theorem:gcd for tempered reps} (and
Corollary~\ref{corollary:tempered theorem holds for essentially tempered}).
\end{proof}

The present work is among the few attempts so far to provide a
g.c.d. definition to the $L$-function, and the first attempt to
carry over the work of Jacquet, Piatetski-Shapiro and Shalika
\cite{JPSS} to the Rankin-Selberg convolutions of
$G\times\GL{\bigidx}$ where $G$ is a classical group. Such a study
was suggested by Gelbart and Piatetski-Shapiro \cite{GPS} (p. 136).
Our technique and results readily adapt to the integrals studied by
Ginzburg \cite{G} and Soudry \cite{Soudry}, due to the similar
nature and technical closeness of the constructions. For example, Soudry already
proved in \cite{Soudry,Soudry3,Soudry2} that
the $\gamma$-factor for $SO_{2\smallidx+1}\times\GL{\bigidx}$
defined by the Rankin-Selberg integrals is equal to Shahidi's
$\gamma$-factor. This can be used to obtain an analogue of Theorem~\ref{theorem:gcd for
tempered reps}. Partial arguments towards a result similar to
Theorem~\ref{theorem:lower bound for gcd 2} can be
found in \cite{Soudry}. Additionally,
manipulations of integrals used here to prove
Theorems~\ref{theorem:gcd sub multiplicity second var} and
\ref{theorem:gcd sub multiplicity first var} can be replaced with
results of Soudry \cite{Soudry,Soudry2} to conclude similar upper
bounds.

The results of Chapter~\ref{chapter:the integrals} on the unfolding of the global integral and proof of
Theorem~\ref{theorem:unramified computation}, for $\smallidx\leq\bigidx$, were published in \cite{E,me2}.
The multiplicativity of the $\gamma$-factor (Chapter~\ref{section:gamma_mult}) appeared in \cite{me3}.
Results of Chapters~\ref{chapter:the gcd} and \ref{chapter:upper boubnds on the gcd} regarding the g.c.d. and its
properties were published in \cite{me4}. 

\section{Thesis outline}\label{section:thesis outline}
This thesis is organized as follows.
Chapter~\ref{section:preliminaries} contains preliminaries,
notation and basic definitions. Chapter~\ref{chapter:the integrals}
is devoted to a brief description of the global setting and to the
calculations of the integrals with unramified data. In Chapter~\ref{chapter:uniqueness} we prove the uniqueness properties. The basic tools
used in the study of the integrals are developed in
Chapter~\ref{chapter:Properties of the integrals}. The local factors
are defined in Chapter~\ref{chapter:local factors}. In
Chapter~\ref{section:gamma_mult} we establish the multiplicative
properties of the $\gamma$-factor. The \archimedean\
results are described in Chapter~\ref{chapter:archimedean results}. Chapter~\ref{chapter:the gcd}
contains several results on the g.c.d., including the analysis of the
tempered case. Finally, the upper bounds on the g.c.d. are proved in
Chapter~\ref{chapter:upper boubnds on the gcd}.


%% file: chapter_prelim.tex
\newtheorem{theorem}{Theorem}[section]
\newtheorem{proposition}{Proposition}[section]
\newtheorem{corollary}{Corollary}[section]
\newtheorem{lemma}{Lemma}[section]
\newtheorem{claim}{Claim}[section]
\theoremstyle{remark}
\newtheorem{remark}{Remark}[section]
\newtheorem{example}{Example}[section]
\theoremstyle{definition}
\newtheorem{definition}{Definition}[section]
\numberwithin{equation}{section}
\newcommand{\chapter}{\section} 
\input{thesis_notations}
\end{comment}
\chapter{Preliminaries and Notation}\label{section:preliminaries}

\section{The groups}\label{subsection:groups in study} Let $F$ be a field of characteristic zero. For $k\geq1$
put
\begin{align*}
J_k=\left(\begin {array}{ccc}
  0 &  & 1 \\
  & \iddots &  \\
  1 &  & 0
\end {array}\right)\in\GLF{k}{F}.
\end{align*}
Let $\rho\in F^*$. If $\rho\in F^2$, define
$J_{2\smallidx,\rho}=J_{2\smallidx}$ and set $\rho=\beta^2$.
Otherwise denote
\begin{align}\label{matrix:J 2lrho}
J_{2\smallidx,\rho}=\left(\begin {array}{ccc}  &  & J_{\smallidx-1}
\\  & \begin{matrix}
   1 & 0 \\
   0 & -\rho
\end{matrix} &  \\ J_{\smallidx-1} &  &
\end {array}\right)\in\GLF{2\smallidx}{F}.
\end{align}
We use $\rho$ to define the special even orthogonal group
$SO_{2\smallidx}(F)$. Let
\begin{align*}
G_{\smallidx}(F)=SO_{2\smallidx}(F)=\setof{g\in
\GLF{2\smallidx}{F}}{\det{g}=1,\transpose{g}J_{2\smallidx,\rho}g=J_{2\smallidx,\rho}},
\end{align*}
regarded as an algebraic group over $F$. It is split when
$\rho=\beta^2$, then we will assume $|\beta|=1$. Otherwise it is
\quasisplit, i.e., \nonsplit\ over $F$ and split over a quadratic
extension of $F$ and we assume $|\rho|\geq1$. 
Also let $\gamma=\half\rho$.

\begin{remark}
One may consider \eqref{matrix:J 2lrho} also with $\rho=\beta^2$, then
the special orthogonal group is split.
\end{remark}

Except for Section~\ref{section:construction of the global integral}
and Chapter~\ref{chapter:archimedean results}, our field will always
be local \nonarchimedean. For such a field, denote by $\mathcal{O}$
the ring of integers of $F$, $\mathcal{P}=\varpi\mathcal{O}$ is the
maximal ideal and
$|\varpi|^{-1}=q=|\rmodulo{\mathcal{O}}{\mathcal{P}}|$. The
valuation of $F$ is $\vartheta$, i.e., $|a|=q^{-\vartheta(a)}$ for
$a\ne0$. Throughout, references to the field are omitted.

The following subgroups of $G_{\smallidx}$ will be used repeatedly.
Fix the Borel subgroup
$B_{G_{\smallidx}}=T_{G_{\smallidx}}\ltimes U_{G_{\smallidx}}$,
where $U_{G_{\smallidx}}$ is the subgroup of upper triangular
unipotent matrices in $G_{\smallidx}$. In the split case $T_{G_{\smallidx}}$ is composed of diagonal matrices. When $G_{\smallidx}$ is \quasisplit,
\begin{align*}
T_{G_{\smallidx}}=\setof{diag(t_1,\ldots,t_{\smallidx-1},\left(\begin{array}{cc}a&b\rho\\b&a\\\end{array}\right),t_{\smallidx-1}^{-1},\ldots,t_1^{-1})}{a^2-b^2\rho=1}.
\end{align*}
Let $P_k=L_k\ltimes V_k$ be the standard maximal parabolic subgroup
with
\begin{align*}
L_k=\setof{diag(x,y,J_k(\transpose{x^{-1}})J_k)}{x\in\GL{k},y\in
G_{\smallidx-k}}\isomorphic\GL{k}\times G_{\smallidx-k},
\end{align*}
$V_k<U_{G_{\smallidx}}$.
When $G_{\smallidx}$
is split, we have an additional standard maximal parabolic subgroup 
$\rconj{\kappa}P_{\smallidx}={\kappa}^{-1}P_{\smallidx}\kappa$ where
$\kappa=diag(I_{\smallidx-1},J_2,I_{\smallidx-1})$.
In the \quasisplit\ case the subgroups $P_k$ with $k<\smallidx$
exhaust all standard maximal parabolic subgroups.
For any parabolic subgroup $P<G_{\smallidx}$, $\overline{P}$ is the
parabolic subgroup opposite to $P$ containing the Levi part of $P$.
Specifically, if $P=LV$ where $L$ is the Levi part and $V$ is the
unipotent radical of $P$, $\overline{P}=L\overline{V}$ where
$\overline{V}$ is the unipotent subgroup opposed to $V$. Also
$\delta_P$ denotes the modulus character of $P$.

Let $\Delta_{G_{\smallidx}}$ be the set of simple roots of $G_{\smallidx}$, determined by our selection of $B_{G_{\smallidx}}$. The root system of $G_{\smallidx}$ is of type $D_{\smallidx}$ in the split case and $B_{\smallidx-1}$ in the \quasisplit\ case (see e.g. \cite{Spr} p.~261). Explicitly,
\begin{align*}
\Delta_{G_{\smallidx}}=\begin{cases}\{\epsilon_1-\epsilon_2,\ldots,\epsilon_{\smallidx-1}-\epsilon_{\smallidx},\epsilon_{\smallidx-1}+\epsilon_{\smallidx}\}&\text{split $G_{\smallidx}$,}\\
\{\epsilon_1-\epsilon_2,\ldots,\epsilon_{\smallidx-2}-\epsilon_{\smallidx-1},\epsilon_{\smallidx-1}\}&\text{\quasisplit\
$G_{\smallidx}$.}\end{cases}
\end{align*}
Here $\epsilon_i(t)=t_i$ is the $i$-th coordinate function of $t\in
T_{G_{\smallidx}}$, it is defined for $1\leq i\leq \smallidx$ (resp.
$1\leq i\leq \smallidx-1$) if $G_{\smallidx}$ is split (resp.
\quasisplit).


For any $k\leq\smallidx$, $G_{\smallidx-k}$ is naturally regarded as
a subgroup of $G_{\smallidx}$, embedded in $P_k$. Let
$K_{G_{\smallidx}}$ be a special good maximal compact open subgroup
(e.g. in the split case
$K_{G_{\smallidx}}=G_{\smallidx}(\mathcal{O})$). For any $k>0$ we
have the ``small" compact open subgroup, neighborhood of the
identity
$\mathcal{N}_{G_{\smallidx},k}=(I_{2{\smallidx}}+\MatF{2{\smallidx}\times2{\smallidx}}{\mathcal{P}^k})\cap
G_{\smallidx}$.

The special odd orthogonal group is
\begin{align*}
H_{\bigidx}=SO_{2\bigidx+1}=\setof{g\in
\GL{2\bigidx+1}}{\det{g}=1,\transpose{g}J_{2\bigidx+1}g=J_{2\bigidx+1}}.
\end{align*}

We use a notation similar to the above for $H_{\bigidx}$ (e.g.
$B_{H_{\bigidx}}=T_{H_{\bigidx}}\ltimes U_{H_{\bigidx}}$). The torus
$T_{H_{\bigidx}}$ is the subgroup of diagonal matrices in
$H_{\bigidx}$. Let $Q_k=M_k\ltimes U_k$ be the standard parabolic
subgroup with a Levi part
\begin{align*}
M_k=\setof{diag(x,y,J_k(\transpose{x^{-1}})J_k)}{x\in\GL{k},y\in
H_{\bigidx-k}}\isomorphic\GL{k}\times H_{\bigidx-k},
\end{align*}
$U_k<U_{H_{\bigidx}}$. Here $\Delta_{H_{\bigidx}}$ is of type
$B_{\bigidx}$ and we take
$K_{H_{\bigidx}}=H_{\bigidx}(\mathcal{O})$.
For $k\leq\bigidx$, $H_{\bigidx-k}$ is embedded in
$H_{\bigidx}$ through $M_{k}$. 

In the group $\GL{k}$ fix the Borel subgroup $B_{\GL{k}}=A_k\ltimes
Z_k$, $A_k$ is the diagonal subgroup and $Z_k$ is the subgroup of
upper triangular unipotent matrices. For $m>1$ and
$k_1,\ldots,k_m\geq1$ such that $k_1+\ldots+k_m=k$, let
$P_{k_1,\ldots,k_m}=A_{k_1,\ldots,k_m}\ltimes Z_{k_1,\ldots,k_m}$ be
the standard parabolic subgroup of $\GL{k}$ which corresponds to
$(k_1,\ldots,k_m)$. Its Levi part $A_{k_1,\ldots,k_m}$ is isomorphic
to $\GL{k_1}\times\ldots\times\GL{k_m}$. In the context of induced
representations, we define $P_{k_1,\ldots,k_m}$ with $m=1$ to be
$\GL{k}$. For example if we have for some $m\geq1$, a representation
$\sigma_1\otimes\ldots\otimes\sigma_m$ of $A_{k_1,\ldots,k_m}$, then
$\cinduced{P_{k_1,\ldots,k_m}}{\GL{k}}{\sigma_1\otimes\ldots\otimes\sigma_m}$
will simply be $\sigma_1$ if $m=1$. We take
$K_{\GL{k}}=\GLF{k}{\mathcal{O}}$. Also denote by $Y_k$ the
mirabolic subgroup of $\GL{k}$, i.e., the subgroup of $x\in \GL{k}$
with the last row $(0,\ldots,0,1)$. We shall frequently use the Weyl
element
$\omega_{k_1,k_2}=\bigl(\begin{smallmatrix}&I_{k_1}\\I_{k_2}&\end{smallmatrix}\bigr)\in\GL{k_1+k_2}$
($k_1,k_2\geq0$). For $x\in\GL{k}$,
$x^*=J_k(\transpose{x^{-1}})J_k$.
Oftentimes $\GL{k}$ will be regarded as a subgroup of
$G_{\smallidx}$ (resp. $H_{\bigidx}$), embedded in $L_{j}$ (resp.
$M_{j}$) for $j\geq k$ by $x\mapsto diag(x,I_{2(\smallidx-k)},x^*)$
(resp. $x\mapsto diag(x,I_{2(\bigidx-k)+1},x^*)$).

Since the image of $A_{\bigidx}$ in $M_{\bigidx}$ is
$T_{H_{\bigidx}}$, we identify $T_{H_{\bigidx}}$ with $A_{\bigidx}$.
Regarding $G_{\smallidx}$, 
$T_{G_{\smallidx}}$ is identified with 
$G_1\times A_{\smallidx-1}$.

Let $\mathcal{S}(F^k)$ be the space of Schwartz-Bruhat functions on
$F^k$.

In general we denote for a subgroup $Y$, $y\in Y$ and an element
$g$, $\rconj{g}Y=g^{-1}Yg$ and $\rconj{g}y=g^{-1}yg$ (e.g.
$\rconj{\omega_{k_2,k_1}}P_{k_2,k_1}=\overline{P_{k_1,k_2}}$).

\subsection{Embedding $G_{\smallidx}$ in $H_{\bigidx}$, $\smallidx\leq\bigidx$}\label{subsection:G_l in H_n}
The embedding is described by defining bases for the underlying
vector spaces of $H_{\bigidx}$ and $G_{\smallidx}$. Let $F^k$ be the
$k$-dimensional column space. Denote by $(,)$ the symmetric form
defined on the column space $F^{2\bigidx+1}$ by $J_{2\bigidx+1}$
(i.e. $(u,v)=\transpose{u}J_{2\bigidx+1}v$). Let
\begin{align*}
\mathcal{E}_{H_{\bigidx}}=(e_1,\ldots,e_{\bigidx},e_{\bigidx+1},e_{-\bigidx},\ldots,e_{-1})
\end{align*}
be the standard basis of $F^{2\bigidx+1}$, such that the Gram matrix
of $\mathcal{E}_{H_{\bigidx}}$ with respect to $(,)$ is
$J_{2\bigidx+1}$ (e.g.
$e_{-i}=e_{2\bigidx+2-i}$). 
Set $e_{\gamma}=e_{\bigidx}+\gamma e_{-\bigidx}$. The image of $G_{\smallidx}$ in $H_{\bigidx}$ is $SO(V)$ where $V$ is the
orthogonal complement of
\begin{align*}
\Span{F}{\{e_i,e_{-i}:1\leq
i\leq\bigidx-\smallidx\}}+\Span{F}{\{e_{\gamma}\}}.
\end{align*}

Select a basis $\mathcal{E}_{G_{\smallidx}}$ of $V$ such that its
Gram matrix (with respect to $(,)$) is $J_{2\smallidx,\rho}$.
Elements of $H_{\bigidx}$ (resp. $G_{\smallidx}$) will be written
using $\mathcal{E}_{H_{\bigidx}}$ (resp.
$\mathcal{E}_{G_{\smallidx}}$). Let $\mathcal{E}_{G_{\smallidx}}'$
be the basis of $V+\Span{F}{\{e_{\gamma}\}}$ obtained by adding
$e_{\gamma}$ to $\mathcal{E}_{G_{\smallidx}}$ as the
$(\smallidx+1)$-th vector and $M$ be the transition matrix from
$\mathcal{E}_{G_{\smallidx}}'$ to $\mathcal{E}_{H_{\smallidx}}$.

If $g=\bigl(\begin{smallmatrix}A&B\\C&D\end{smallmatrix}\bigr)\in G_{\smallidx}$ ($A,D\in\Mat{\smallidx\times\smallidx}$), we can
regard $g$ as a linear transform on $V+\Span{F}{\{e_{\gamma}\}}$ by extending it trivially on $e_{\gamma}$.
Then the coordinates of $g$ relative to $\mathcal{E}_{G_{\smallidx}}'$, denoted by $[g]_{\mathcal{E}_{G_{\smallidx}}'}$, satisfy
\begin{align*}
[g]_{\mathcal{E}_{G_{\smallidx}}'}=\left(\begin{array}{ccc}A&0&B\\0&1&0\\C&0&D\end{array}\right).
\end{align*}

To write $g\in G_{\smallidx}$ as an element of $H_{\bigidx}$, define
$ge_{\gamma}=e_{\gamma}$ and compute
\begin{align*}
[g]_{\mathcal{E}_{H_{\bigidx}}}=diag(I_{\bigidx-\smallidx},M^{-1}[g]_{\mathcal{E}_{G_{\smallidx}}'}M,I_{\bigidx-\smallidx}).
\end{align*}
We take
\begin{align*}
M=\begin{dcases}diag(I_{\smallidx-1},\left(\begin {array}{ccc}
\quarter&\frac 1 {2\beta} &-\frac 1 {2\beta^2}\\\half&0&\frac 1
{\beta^2}\\-\half\beta^2&\beta&1\\\end{array}\right),I_{\smallidx-1})&\text{split
$G_{\smallidx}$,}\\diag(I_{\smallidx-1},\left(\begin {array}{ccc}
0&1&0\\\half&0&\half\gamma^{-1}\\\half&0&-\half\gamma^{-1}\\\end{array}\right),I_{\smallidx-1})&\text{\quasisplit\
$G_{\smallidx}$.}\end{dcases}
\end{align*}
\begin{remark}\label{remark:why choose gamma}
The vector $e_{\gamma}$ is defined using $\half\rho$ (instead of
$\rho$), so that when $\rho\in F^2$ (hence $G_{\smallidx}$ is
split), the Witt index of the orthogonal complement of
$\Span{F}{\{e_{\gamma}\}}$ would be $\bigidx$. This is necessary for
embedding split $G_{\smallidx}$ in $H_{\smallidx}$.
\end{remark}
\begin{remark}\label{remark:why choose e gamma}
Assume $\smallidx<\bigidx$. Let
$N_{\bigidx-\smallidx}=Z_{\bigidx-\smallidx}\ltimes
U_{\bigidx-\smallidx}$ (here
$Z_{\bigidx-\smallidx}<M_{\bigidx-\smallidx}$). Fix some additive \nontrivial\
character $\psi$ of the field and extend it to a character of
$Z_{k}$ (for any $k\geq1$) by
$\psi(z)=\psi(\sum_{i=1}^{k-1}z_{i,i+1})$. Define a character
$\psi_{\gamma}$ of $N_{\bigidx-\smallidx}$ by
$\psi_{\gamma}(zy)=\psi(z)\psi(\transpose{e_{\bigidx-\smallidx}}ye_{\gamma})$
where $z\in Z_{\bigidx-\smallidx}$ and $y\in U_{\bigidx-\smallidx}$.
The group $G_{\smallidx}$ is embedded in $H_{\bigidx}$ so that it
normalizes $N_{\bigidx-\smallidx}$ and stabilizes $\psi_{\gamma}$.
The pair $(N_{\bigidx-\smallidx},\psi_{\gamma})$ will be used to
construct the integral for $\smallidx<\bigidx$, see
Section~\ref{subsection:The global integral for l<=n}.
\end{remark}

One property of this embedding is that although
$A_{\smallidx-1}<T_{H_{\bigidx}}$, $T_{G_{\smallidx}}$ is not a
subgroup of $T_{H_{\bigidx}}$. 
We write explicitly the form of an element from $G_1$ under the
embedding in $H_1$. Of course this form carries over to
$H_{\bigidx}$ by the embedding $H_1<H_{\bigidx}$. In the split case,
\begin{align}\label{eq:embedding of SO_2 split}
x=\left(\begin{array}{cc}b&0\\0&b^{-1}\\\end{array}\right)\mapsto\left(\begin{array}{ccc}
\half+\quarter (b+ b^{-1}) &
\frac 1 {2\beta}(b-b^{-1}) & \frac 2 {\beta^2}(\half-\quarter(b+b^{-1}))\\
\quarter\beta (b-b^{-1}) & \half (b + b^{-1}) & -\frac 1 {2\beta} (b-b^{-1})\\
\half\beta^2(\half-\quarter(b+b^{-1}))&-\quarter\beta (b-b^{-1}) &
\half+\quarter (b + b^{-1})
\\\end{array}\right).
\end{align}

Observe that if $|b|=1$ and $|2|=1$, the image lies in $K_{H_1}$
(recall that by assumption $|\beta|=1$). It will be useful to denote
$\lfloor x\rfloor=b$ if $|b|\leq1$ and $\lfloor x\rfloor=b^{-1}$
otherwise. Also let $[x]=\max(|b|,|b|^{-1})$.

We see from \eqref{eq:embedding of SO_2 split} that
$T_{G_{\smallidx}}$ is not a subgroup of $T_{H_{\bigidx}}$. However,
the following lemma shows how to write a torus element of
$G_{\smallidx}$ approximately as a torus element of $H_{\bigidx}$.

\begin{lemma}\label{lemma:torus elements in double cosets}
For any $k_0>0$ there are $h_1,h_2\in H_1$ and $k\geq k_0$ such that
for all 
$x=\bigl(\begin{smallmatrix} b&\\&b^{-1}\end{smallmatrix}\bigr)$
with $[x]>q^{k}$, in $H_1$,
\begin{align*}
x\in\begin{dcases} B_{H_1}h_1\mathcal{N}_{H_1,k_0} & \lfloor x\rfloor=b,\\
B_{H_1}h_2\mathcal{N}_{H_1,k_0} & \lfloor
x\rfloor=b^{-1}.\end{dcases}
\end{align*}
Specifically, we can write $x\in m_xu_xh_i\mathcal{N}_{H_1,k_0}$
with
\begin{align*}
&m_x=\left(\begin{array}{ccc}\lfloor x\rfloor\\&1\\&&\lfloor
x\rfloor^{-1}\\\end{array}\right),\quad
u_x=\left(\begin{array}{ccc}1&c\lfloor x\rfloor^{-1}&-\half
c^2\lfloor x\rfloor^{-2}\\&1&-c \lfloor
x\rfloor^{-1}\\&&1\\\end{array}\right),\quad
\end{align*}
\begin{align*}
c=\begin{dcases} -2\beta^{-1} & \lfloor x\rfloor=b,\\
2\beta^{-1} & \lfloor x\rfloor=b^{-1}.\end{dcases}
\end{align*}
\end{lemma}
\begin{proof}[Proof of Lemma~\ref{lemma:torus elements in double cosets}] 
Let $\epsilon=diag(\beta^{-1},1,\beta)$, let $k'\geq k_0$ be such
that
$\rconj{\epsilon^{-1}}\mathcal{N}_{H_1,k'}<\mathcal{N}_{H_1,k_0}$
and take $k''\geq k'$ such that $q^{-2k''}\leq|2|q^{-k'}$.

In $H_1$, $x$ is given by \eqref{eq:embedding of SO_2 split}. Denote
\begin{align*}
x'=\rconj{\epsilon}x=\left(\begin{array}{ccc}\half+\quarter(b+
b^{-1}) &
\half(b-b^{-1}) & 1-\half(b+b^{-1})\\
\quarter(b-b^{-1}) & \half(b + b^{-1}) & -\half(b-b^{-1})\\
\quarter-\frac18(b+b^{-1})&-\quarter(b-b^{-1})& \half+\quarter(b +
b^{-1})
\\\end{array}\right).
\end{align*}
We will exhibit $m'\in B_{H_1}$ and $h'\in H_1$ such that $m'x'\in
h'\mathcal{N}_{H_1,k'}$, whence $m=\rconj{\epsilon^{-1}}m'$ and
$h=\rconj{\epsilon^{-1}}h'$ satisfy $mx\in h\mathcal{N}_{H_1,k_0}$.
Let
\begin{align*}
&m'=\left(\begin{array}{ccc}t&tv&-\half
tv^2\\&1&-v\\&&t^{-1}\\\end{array}\right),\qquad
u(\xi)=\left(\begin{array}{ccc}
&&1\\
&-1&-\xi\\
1&-\xi&-\half\xi^2\\
\end{array}\right)\quad(\xi\in F).
\end{align*}

We select $k\geq k'$ for which $q^{-k}<|4|q^{-k'}$. Assume
$\lfloor x\rfloor=b$ and $|b|<q^{-k}$. Take
$v=2+4b(1+2\varpi^{k''})$, $t^{-1}=-8b(1+\varpi^{k''})$ and
$h'=u(2)$. 
By matrix multiplication we see that
\begin{align*}
m'x'\in h'\mathcal{N}_{H_1,k'}=\left(\begin{array}{ccc}
\mathcal{P}^{k'}&\mathcal{P}^{k'}&1+\mathcal{P}^{k'}\\
\mathcal{P}^{k'}&-1+\mathcal{P}^{k'}&-2+\mathcal{P}^{k'}\\
1+\mathcal{P}^{k'}&-2+\mathcal{P}^{k'}&-2+\mathcal{P}^{k'}\\
\end{array}\right).
\end{align*}
Note that if one assumes $|2|=1$, to verify this computation it is
enough to check the last two rows of $m'x'$, then use the fact that
$m'x'\in H_1$.

Returning to $x$ we have obtained $x\in
m^{-1}h\mathcal{N}_{H_1,k_0}$ with
\begin{align*}
&m^{-1}=\left(\begin{array}{ccc}-8b(1+\varpi^{k''})&-\frac{2(1+2b+4b\varpi^{k''})}{\beta}&*\\&1&-\frac{1+2b+4b\varpi^{k''}}{4\beta
b(1+\varpi^{k''})}\\&&-\frac1{8b(1+\varpi^{k''})}\\\end{array}\right),\\
&h=\left(\begin{array}{ccc}
&&\beta^{-2}\\
&-1&-2\beta^{-1}\\
\beta^2&-2\beta&-2\\
\end{array}\right).
\end{align*}
Then $m^{-1}=m_xu_xz$ with
\begin{align*}
z=\left(\begin{array}{ccc}-8(1+\varpi^{k''})&-\frac{4(1+2\varpi^{k''})}{\beta}&*\\&1&-\frac{1+3\varpi^{k''}+2\varpi^{2k''}}{2\beta
(1+\varpi^{k''})^2}\\&&-\frac1{8(1+\varpi^{k''})}\\\end{array}\right).
\end{align*}
Let $h_1=zh$. Then $h_1$ depends only on $k'',k',k_0$. Also $x\in
m_xu_xh_1\mathcal{N}_{H_1,k_0}$.

The case $\lfloor x\rfloor=b^{-1}$ follows similarly, with
$v=-2-4b^{-1}(1+2\varpi^{k''})$, $t^{-1}=-8b^{-1}(1+\varpi^{k''})$
and $h'=u(-2)$.
\end{proof} 
\begin{remark}\label{remark:torus elements in double cosets h_1 and
k_0 small} Put $d=diag(8,1,8^{-1})$. Following the construction of
$h_i$ in the proof we see that $zd^{-1},h\in K_{H_{1}}$. Hence if
$k_1$ satisfies $|\varpi|^{k_1}=|8|^2$ and $k_0>k_1$, we get
$\rconj{d^{-1}}\mathcal{N}_{H_1,k_0}<\mathcal{N}_{H_1,k_0-k_1}$ and
$\rconj{h_i^{-1}}\mathcal{N}_{H_1,k_0}<\mathcal{N}_{H_1,k_0-k_1}$.
It follows that for any $c>0$ there is $k_0>0$ such that
$\rconj{h_i^{-1}}\mathcal{N}_{H_1,k_0}<\mathcal{N}_{H_1,c}$ for
$i=1,2$.
\end{remark}

In the \quasisplit\ case the image is given by
\begin{align*}
\left(\begin{array}{cc}a&b\rho\\b&a\\\end{array}\right)\mapsto\left(\begin{array}{ccc}\half(1+a)&b&\half\gamma^{-1}(1-a)\\\gamma b&a&-b\\
\half\gamma(1-a)&-\gamma b&\half(1+a)\end{array}\right).
\end{align*}
It will usually be sufficient to know that the image of $G_1$ in
$H_1$ is a compact subgroup which is covered, for any $k>0$, by a
finite union $\bigcup x_i\mathcal{N}_{H_1,k}$ ($x_i\in H_1$). In the
case where data are unramified, for instance if
$F$ is an unramified extension of $\mathbb{Q}_p$ 
and $|\gamma|=|2|=1$, $1=|a^2-b^2\rho|=\max(|a|^2,|b|^2)$ hence
$|a|,|b|\leq1$ and the image of $G_1$ in
$H_1$ is contained in $K_{H_1}$.

\subsection{Embedding $H_{\bigidx}$ in $G_{\smallidx}$, $\smallidx>\bigidx$}\label{subsection:H_n in G_l}
Here $(,)$ is defined using $J_{2\smallidx,\rho}$. Let
$\mathcal{E}_{G_{\smallidx}}=(e_1,\ldots,e_{\smallidx},e_{-\smallidx},\ldots,e_{-1})$
be the standard basis of $F^{2\smallidx}$ such that the Gram matrix
of $\mathcal{E}_{G_{\smallidx}}$ with respect to $(,)$ is
$J_{2\smallidx,\rho}$. Let
\begin{align*}
e_{\gamma}=\begin{cases}
\quarter e_{\smallidx}-\gamma e_{-\smallidx} & \text{$G_{\smallidx}$ is split,} \\
\half e_{-\smallidx} & \text{$G_{\smallidx}$ is \quasisplit.}
\end{cases}
\end{align*}
The group $H_{\bigidx}$ is embedded in $G_{\smallidx}$ as $SO(V)$, where $V$ is the
orthogonal complement of
\begin{align*}
\Span{F}{\{e_i,e_{-i}:1\leq
i\leq\smallidx-\bigidx-1\}}+\Span{F}{\{e_{\gamma}\}}.
\end{align*}
If $\mathcal{E}_{H_{\bigidx}}$ is a basis of $V$ with a Gram matrix
$J_{2\bigidx+1}$ (with respect to $(,)$),
$\mathcal{E}_{H_{\bigidx}}'$ is obtained from
$\mathcal{E}_{H_{\bigidx}}$ by adding $e_{\gamma}$ as the
$(\bigidx+1)$-th vector. Extend $h\in H_{\bigidx}$ by defining
$he_{\gamma}=e_{\gamma}$, then
\begin{align*}
[h]_{\mathcal{E}_{G_{\smallidx}}}=diag(I_{\smallidx-\bigidx-1},M^{-1}[h]_{\mathcal{E}_{H_{\bigidx}}'}M,I_{\smallidx-\bigidx-1}).
\end{align*} Here
\begin{align*}
M=\begin{dcases}diag(I_{\bigidx},\left(\begin {array}{cc}
2&-\frac1{\beta^2}
\\\beta&\frac1{2\beta}\\\end{array}\right),I_{\bigidx})&\text{split $G_{\smallidx}$,}\\
diag(I_{\bigidx},\left(\begin {array}{cc} 0&2
\\1&0\\\end{array}\right),I_{\bigidx})&\text{\quasisplit\ $G_{\smallidx}$.}\end{dcases}
\end{align*}
%
\begin{remark}\label{remark:why choose e gamma l > n}
Assume $\smallidx>\bigidx+1$. Let
$N^{\smallidx-\bigidx}=Z_{\smallidx-\bigidx-1}\ltimes
V_{\smallidx-\bigidx-1}$, where
$Z_{\smallidx-\bigidx-1}<L_{\smallidx-\bigidx-1}$. Define a
character $\psi_{\gamma}$ of $N^{\smallidx-\bigidx}$ by
$\psi_{\gamma}(zy)=\psi(z)\psi(\transpose{e_{\smallidx-\bigidx-1}}ye_{\gamma})$
($z\in Z_{\smallidx-\bigidx-1}$, $y\in V_{\smallidx-\bigidx-1}$ and
$\psi$ as in Remark~\ref{remark:why choose e gamma}). Then
$H_{\bigidx}$ is embedded in $G_{\smallidx}$ in the normalizer of
$N^{\smallidx-\bigidx}$ and stabilizer of $\psi_{\gamma}$. In
Section~\ref{subsection:The global integral for l>n} we use
$N^{\smallidx-\bigidx}$ and $\psi_{\gamma}$ to construct the
integral.
\end{remark}

\section{Additional notation and symbols}\label{subsection:additional notation and symbols}
For a function $f$ defined on some set $X$, taking values in some
ring (usually in $\C$), $\support{f}$ denotes the support of $f$,
i.e., $\support{f}=\setof{x\in X}{f(x)\ne0}$. If $X$ is a
topological space, $f$ is said to be locally constant if for each
$x$ there is a neighborhood $N_x$ of $x$ such that $f(n_x)=f(x)$ for
all $n_x\in N_x$. Assume that $G$ is a topological group acting on
$X$. The function $f$ is smooth (with respect to the action of $G$
on $X$) if there is a neighborhood $O$ of the identity in $G$, such
that for all $x\in X$ and $o\in O$, $f(o\cdot x)=f(x)$. In the cases
we consider, the notions of being locally constant and smoothness
often coincide (because of the Iwasawa decomposition).

For a measure space $X$ and a measurable subset $Y\subset X$, we denote by $vol(Y)$ the measure of $Y$.

Block diagonal matrices will usually be denoted by $diag(\cdots)$.

Consider the polynomial ring $\C[q^{-s},q^s]$ and its field of
fractions $\C(q^{-s})$. We denote by $\equalun$ an equality of
polynomials or rational functions which holds up to invertible
factors of $\C[q^{-s},q^s]$. For example, $q^{-s}\equalun1$ and for
any $P\in\C[q^{-s},q^s]$ there is some $P_1\in\C[q^{-s}]$ such that
$P\equalun P_1$.

A disjoint union of sets is denoted by $\dotcup$. The sign $\subset$
refers to a weak containment of sets (containment or equality),
$\subsetneq$ denotes proper containment. The symbol $\emptyset$
designates the empty set.

\section{Representations}\label{subsection:notation for reps}
Representations will always be smooth, admissible, finitely
generated, and generic, i.e. admit unique Whittaker models. If a
representation $\pi$ has a Whittaker model with respect to a
character $\chi$, the model is denoted by $\Whittaker{\pi}{\chi}$.
If $\pi$ has a central character, it is denoted by $\omega_{\pi}$.
For any $k\geq1$, a representation of $\GL{k}$ is always assumed to have a central character.

Tempered representations are assumed to be irreducible by
definition. An essentially tempered (resp. essentially square-integrable) representation $\tau$ is a representation of the form $\absdet{}^v\tau_0$ where $v\in\C$ and $\tau_0$ is tempered (resp. square-integrable).
Supercuspidal representations are neither assumed to be
unitary nor irreducible. A unitary irreducible supercuspidal
representation is tempered. Characters are not assumed to be
unitary.

We fix a \nontrivial\ unitary additive character $\psi$ of $F$ and
construct a (unitary) character of $Z_k$ by
$z\mapsto\psi(\sum_{i=1}^{k-1}z_{i,i+1})$ ($\psi$ is non-degenerate
for $k>1$). Then $\psi$ defines a character of any subgroup of $Z_k$
by restriction.

For $s\in\C$ and $b\in\GL{\bigidx}$ denote
$\alpha^s(b)=\absdet{b}^{s-\half}$.

Let $\tau$ be a representation of $\GL{\bigidx}$ on a space $U$. For
a character $\mu$ of $\GL{\bigidx}$, the representation $\tau\mu$ of
$\GL{\bigidx}$ on $U$ is defined by $(\tau\mu)(b)=\mu(b)\tau(b)$.
Denote by $\tau^*$ the representation (of $\GL{\bigidx}$) defined on
$U$ by
$\tau^{*}(b)=\tau(b^*)$ ($b^*=J_{\bigidx}(\transpose{b^{-1}})J_{\bigidx}$).
When $\tau$ is irreducible,
$\Whittaker{\dualrep{\tau}}{\psi}=\Whittaker{\tau^*}{\psi}$ where
$\dualrep{\tau}$ is the representation contragredient to $\tau$. In
general,
\begin{align*}(\cinduced{P_{\bigidx_1,\ldots,\bigidx_k}}{\GL{\bigidx}}{\tau_1\otimes\ldots\otimes\tau_k})^*=
\cinduced{P_{\bigidx_k,\ldots,\bigidx_1}}{\GL{\bigidx}}{\tau_k^*\otimes\ldots\otimes\tau_1^*}.
\end{align*}

For a function $f$ defined on some group and $x$ an element of the
group, $\lambda(x)f$ (resp. $x\cdot f$) denotes  the
left-translation (resp. right-translation) of $f$ by $x$.

\section{Sections}\label{subsection:sections}
We recall the definition of a holomorphic section of a parabolically
induced representation, parameterized by an unramified character of
the Levi part. For a thorough treatment of this subject refer to
Waldspurger \cite{W} (Section~IV) and Mui\'{c} \cite{Mu}.

Let $\tau$ be a representation of $\GL{\bigidx}$ on a space $U$. Set
$K=K_{H_{\bigidx}}=H_{\bigidx}(\mathcal{O})$. Consider the induced
representation
$\cinduced{Q_{\bigidx}\cap K}{K}{\tau}$ whose space is denoted by
$V(\tau)=V_{Q_{\bigidx}\cap K}^{K}(\tau)$. This is the space of
functions $f:K\rightarrow U$ such that $f$ is smooth on
the right, i.e., there is some compact open subgroup $N<K$ such that
$f(ky)=f(k)$ for all $k\in K$ and $y\in Y$, and $f(qk)=\tau(q)f(k)$
for $q\in Q_{\bigidx}\cap K$.

For any fixed $s\in\C$ we have
the representation 
$\cinduced{Q_{\bigidx}}{H_{\bigidx}}{\tau\alpha^s}$ (normalized
induction) on the space
$V(\tau,s)=V_{Q_{\bigidx}}^{H_{\bigidx}}(\tau,s)$. Specifically,
$V(\tau,s)$ consists of the functions $f_s:H_{\bigidx}\rightarrow U$
such that $f_s$ is smooth on the right and
$f_s(qh)=\delta_{Q_{\bigidx}}^{\half}(q)\absdet{a}^{s-\half}\tau(a)f_s(h)$
for $q=diag(a,1,a^*)u\in Q_{\bigidx}$ ($a\in\GL{\bigidx}$, $u\in
U_{\bigidx}$).

Any $f\in V(\tau)$ can be extended to an element of $V(\tau,s)$,
according to the Iwasawa decomposition. This produces an isomorphism
between $V(\tau)$ and $V(\tau,s)$ as $K$-representation spaces. The
image of $f\in V(\tau)$ in $V(\tau,s)$ is denoted by $f_s$.

Consider the following family of functions. For $k\in K$, $N<K$ a
compact open subgroup and $v\in U$ which is invariant by
$(\rconj{k^{-1}}N)\cap Q_{\bigidx}$, define $ch_{kN,v}\in V(\tau)$
by
\begin{align}\label{func:ch kN}
ch_{kN,v}(k')=
\begin{dcases}
\tau(a)v& k'=akn, a\in Q_{\bigidx}\cap K,n\in N, \\
0& \text{otherwise}.
\end{dcases}
\end{align}
These functions span $V(\tau)$ (see 
\cite{BZ1}, 2.24). Then $ch_{kN,v}$ extends to $ch_{kN,v,s}\in
V(\tau,s)$. For $h\in H_{\bigidx}$, write $h=qk'\in Q_{\bigidx}K$
with $q=au$, $a\in\GL{\bigidx}\isomorphic M_{\bigidx}$, $u\in
U_{\bigidx}$. Then
$ch_{kN,v,s}(h)=\delta_{Q_{\bigidx}}^{\half}(q)\absdet{a}^{s-\half}\tau(a)ch_{kN,v}(k')$.

We will usually consider $s$ as a parameter. A function
$f(s,h):\C\times H_{\bigidx}\rightarrow U$ such that for all $s$,
the mapping $h\mapsto f(s,h)$ belongs to $V(\tau,s)$, is called a
section.

A section $f$ is called standard if for any fixed $k\in K$ the
function $s\mapsto f(s,k)$ is independent of $s$ (i.e., it is a
constant function). Let
$\xi_{Q_{\bigidx}}^{H_{\bigidx}}(\tau,std,s)$ be the space of
standard sections. The elements of this space are precisely the
functions $(s,h)\mapsto f_s(h)$ where $f\in V(\tau)$. The subgroup
$K$ acts on this space by $k\cdot f(s,h)=f(s,hk)$. For
$f\in\xi_{Q_{\bigidx}}^{H_{\bigidx}}(\tau,std,s)$ there is a compact
open subgroup $N<K$ such that $f(s,hy)=f(s,h)$ for all $s\in\C$,
$h\in H_{\bigidx}$ and $y\in N$. Furthermore, there is a subset
$B\subset H_{\bigidx}$ such that for all $s$, the support of the
function $h\mapsto f(s,h)$ equals $B$. The subset $B$ is invariant
for multiplication on the right by $N$ and on the left by
$Q_{\bigidx}$.

We usually pick a section $f$ and study the function $h\mapsto
f(s,h)$ as $s$ varies. Therefore we introduce the following
convention. We a priori write $f_s$ instead of $f$ and think of $s$
as a parameter. So for any $f_s\in
\xi_{Q_{\bigidx}}^{H_{\bigidx}}(\tau,std,s)$ there is a compact open
subgroup $N<K$ independent of $s$, such that $f_s$ is
right-invariant by $N$, and the support of $f_s$ in $H_{\bigidx}$ is
independent of $s$. In order to obtain a concrete function on
$H_{\bigidx}$, we must fix $s$. Then we say that for a fixed $s$,
$f_s\in V(\tau,s)$.

The space of holomorphic sections is
$\xi_{Q_{\bigidx}}^{H_{\bigidx}}(\tau,hol,s)=\C[q^{-s},q^{s}]\otimes_{\C}\xi_{Q_{\bigidx}}^{H_{\bigidx}}(\tau,std,s)$.
We abbreviate
$\xi(\tau,\cdot,s)=\xi_{Q_{\bigidx}}^{H_{\bigidx}}(\tau,\cdot,s)$. A
holomorphic section takes the form
$f_s=\sum_{i=1}^mP_i(q^{-s},q^{s})f_s^{(i)}$ where
$P_i\in\C[q^{-s},q^{s}]$, $f_s^{(i)}\in\xi(\tau,std,s)$. The
following claim shows that $\xi(\tau,hol,s)$ is an
$H_{\bigidx}$-space, where $H_{\bigidx}$ acts by right-translations
on the second component of the tensors.
\begin{claim}\label{claim:translation of standard section is nearly standard}
For any $f_s\in\xi(\tau,hol,s)$ and $h\in H_{\bigidx}$, $h\cdot f_s\in\xi(\tau,hol,s)$.
\end{claim}
\begin{proof}[Proof of Claim~\ref{claim:translation of standard section is nearly standard}] 
It is enough to prove the claim for $f_s\in \xi(\tau,std,s)$. Take
$N<K$ as above. Given $h\in H_{\bigidx}$, $h\cdot f_s$ is
right-invariant on $N_h=(\rconj{h^{-1}}N)\cap K$. Let
$k_1,\ldots,k_m\in K$ be distinct representatives of the double
coset space $\rmodulo{\lmodulo{Q_{\bigidx}}{H_{\bigidx}}}{N_h}$. For
any $k_i$, write $k_ih=q_ik_i'\in Q_{\bigidx}K$, $q_i=a_iu_i$ with $a_i\in\GL{\bigidx}$, $u_i\in U_{\bigidx}$. Then
\begin{align*}
h\cdot
f_s(k_i)=f_s(k_ih)=\delta_{Q_{\bigidx}}^{\half}(q_i)\absdet{a_i}^{s-\half}\tau(a_i)
f_s(k_i').
\end{align*}
Set $P_i=\delta_{Q_{\bigidx}}^{\half}(q_i)\absdet{a_i}^{s-\half}\in\C[q^{-s},q^s]$
and $v_i=\tau(a_i)f_s(k_i')\in U$.
Note that $v_i$ is
invariant by $(\rconj{k_i^{-1}}N_h)\cap Q_{\bigidx}$ and
furthermore, it is independent of $s$ because $k_i'\in K$ and $f_s$ is standard. Hence $ch_{k_iN_h,v_i}\in
V(\tau)$ and $h\cdot f_s=\sum_{i=1}^mP_i\cdot
ch_{k_iN_h,v_i,s}\in\xi(\tau,hol,s)$.
\end{proof} 
For an element $f_s\in\xi(\tau,hol,s)$ there is a compact open
subgroup $N<K$ independent of $s$ such that $f_s$ is right-invariant
by $N$. Whenever we take a subgroup by which $f_s$ is
right-invariant, we implicitly mean such a subgroup. The support of
$f_s$ in $H_{\bigidx}$ may depend on $s$.

Let $f_s=\sum_{i=1}^mP_if_s^{(i)}\in\xi(\tau,hol,s)$ with $0\ne
P_i\in\C[q^{-s},q^{s}]$, $f_s^{(i)}\in\xi(\tau,std,s)$. Pick a
subgroup $N$ such that $f_s^{(i)}$ is right-invariant by $N$ for all
$i$. Then if $k_1,\ldots,k_b\in K$ are distinct representatives for
$\rmodulo{\lmodulo{Q_{\bigidx}}{H_{\bigidx}}}{N}$,
$f_s=\sum_{i,j}P_ich_{k_jN,v_{i,j},s}$. Let $U^{k_j}\subset U$ be
the subspace of vectors fixed by $(\rconj{k_j^{-1}}N)\cap Q_{\bigidx}$.
Since the function $\phi:U^{k_j}\rightarrow V(\tau)$ given by
$\phi(v)=ch_{k_jN,v}$ is linear (see \eqref{func:ch kN}), we may rewrite $f_s$ so that for
each $j$, the \nonzero\ vectors in $\{v_{1,j},\ldots,v_{m,j}\}$ are
linearly independent. Consequently $f_s$ can be written as
\begin{align}\label{eq:holomorphic section as combination of disjoint standard sections}
f_s=\sum_{i=1}^{m'}P_i'\cdot ch_{k_i'N,v_i',s},
\end{align}
with $0\ne P_i'\in\C[q^{-s},q^{s}]$, $k_i'\in\{k_1,\ldots,k_b\}$
(the double cosets $Q_{\bigidx}k_i'N$ are not necessarily disjoint)
and the data $(m',k_i',N,v_i')$ do not depend on $s$. Moreover, if
$\{i_1,\ldots,i_c\}$ is a set of indices satisfying
$k_{i_1}'=\ldots=k_{i_c}'$, then $v_{i_1}',\ldots,v_{i_c}'$ are linearly
independent. 

We will mostly be dealing with either holomorphic sections or the
images of such, under specific intertwining operators. This leads us
to consider the space 
$\xi_{Q_{\bigidx}}^{H_{\bigidx}}(\tau,rat,s)=\C(q^{-s})\otimes_{\C}\xi(\tau,std,s)$
of rational sections. It is also an $H_{\bigidx}$-space.

In a slightly more general context, let $G$ be one of the groups
defined in Section~\ref{subsection:groups in study} and $P<G$ be a
parabolic subgroup with a Levi part
$L\isomorphic\GL{k_1}\times\ldots\times\GL{k_m}\times G'$, where
$G'$ is either the trivial group $\{1\}$ or a group of the same type
as $G$. Assume that $\tau$ is a representation of $L$ on a space
$U$. An $m$-tuple $\underline{s}=(s_1,\ldots,s_m)\in\C^m$ defines an
unramified character $\alpha^{\underline{s}}$ of $L$ by
$(g_1,\ldots,g_m,g')\mapsto\alpha^{s_1}(g_1)\cdot\ldots\cdot\alpha^{s_m}(g_m)$.
The space $V_{P\cap K_G}^{K_G}(\tau)$ of the representation
$\cinduced{P\cap K_G}{K_G}{\tau}$ is isomorphic to the space
$V_{P}^{G}(\tau,\underline{s})$ of
$\cinduced{P}{G}{\tau\alpha^{\underline{s}}}$ (as $K_G$-spaces). A
function $f(\underline{s},g):\C^m\times G\rightarrow U$ such that
for all $\underline{s}$, the mapping $g\mapsto f(\underline{s},g)$
belongs to $V_{P}^{G}(\tau,\underline{s})$ is called a section. The
space of standard sections $\xi_{P}^{G}(\tau,std,\underline{s})$
contains precisely the sections $f$ such that for all $k\in K_G$,
the function $\underline{s}\mapsto f(\underline{s},k)$ is constant.
The holomorphic sections are the elements of
$\xi_{P}^{G}(\tau,hol,\underline{s})=\C[q^{\mp s_1},\ldots,q^{\mp
s_m}]\otimes_{\C}\xi_{P}^{G}(\tau,std,\underline{s})$ and rational
sections are defined by
$\xi_{P}^{G}(\tau,rat,\underline{s})=\C(q^{-s_1},\ldots,q^{-s_m})\otimes_{\C}\xi_{P}^{G}(\tau,std,\underline{s})$.
Note that $K_G$ may be chosen so that
$\frestrict{\alpha^{\underline{s}}}{K_G\cap L}\equiv1$.

\section{Properties of Whittaker functions}\label{section:whittaker props}
Let $G$ be one of the groups defined in
Section~\ref{subsection:groups in study}, denote by $T_G$ the
maximal torus of $G$ and let $\pi$ be a representation of $G$
realized in a Whittaker model $\Whittaker{\pi}{\chi}$. We state a
few properties of the Whittaker functions of
$\Whittaker{\pi}{\chi}$, that will be used repeatedly.
\begin{enumerate}[leftmargin=*]
\item A Whittaker function vanishes away from zero: let $W\in\Whittaker{\pi}{\chi}$. There is a constant $c>0$, depending on $W$, such that for all $t\in T_G$, $W(t)=0$ unless $|\alpha(t)|<c$ for all $\alpha\in\Delta_G$. This is a result of Casselman and Shalika \cite{CS2}
(Proposition~6.1). Actually, the constant $c$ in the proof depends
only on the subgroup $\mathcal{N}_{G,k}<G$ which fixes $W$ and on
$\chi$ ($\mathcal{N}_{G,k}$ was defined in Section~\ref{subsection:groups in study}, $k$ is a constant depending on $W$). This implies the following. Let $u$ be a complex parameter.
Assume that for all $u$, $W_u\in\Whittaker{\pi\alpha^u}{\chi}$ and
furthermore, there is a constant $k>0$ such that for all $u$, $W_u$
is right-invariant by $\mathcal{N}_{G,k}$. Then there is some $c$
satisfying for all $u$, $W_u(t)=0$ unless $|\alpha(t)|<c$ for all
$\alpha\in\Delta_G$.

\item The asymptotic expansion of Whittaker functions: the restriction of a Whittaker function to a torus can be written in a simple, convenient form.
First we recall the notion of a finite function. A finite function
on $A_k$ is a locally constant complex-valued function $f$ such that
$\Span{\C}{\setof{a\cdot f}{a\in A_k}}$ is a finite-dimensional
vector space. Any finite function is equal to a finite sum of
products $\prod_{i=1}^{k}\eta_i$ where each $\eta_i$ is a finite
function on $F^*$. Moreover, any finite function $f$ can be
expressed as
$f(a)=\sum_{\xi\in C}\xi(a)P_{\xi}(a)$,
where $C$ is a finite set of characters of $F^*$ and $P_{\xi}$ is a
polynomial in the valuation vector of $a$. I.e., $P_{\xi}(a)$ is a
finite sum of products $\prod_{i=1}^k\vartheta(a_i)^{m_i}$
($a=diag(a_1,\ldots,a_k)$) with $0\leq m_i\in\Integers$. 
See Waldspurger \cite{W} (Section~I.2).

Assume that $G$ is neither $\GL{1}$ nor $G_1$ (in these cases the
form of Whittaker functions is trivial). Write
$\Delta_{G}=\{\alpha_1,\ldots,\alpha_{m}\}$. Let $T_0$ denote the
maximal split torus of $G$, i.e., if $G=\GL{\bigidx}, H_{\bigidx}$
or split $G_{\smallidx}$, $T_0=T_G$. If $G=G_{\smallidx}$ is
\quasisplit, $T_0=A_{\smallidx-1}<T_{G_{\smallidx}}$. There is a
finite set $\mathcal{A}_{\pi}$ of finite functions on $F^*$, such
that for any $W\in\Whittaker{\pi}{\chi}$ there exist functions
$\phi_{\eta_1,\ldots,\eta_m}\in\mathcal{S}(F^{m})$ defined for any
$m$ elements $\eta_1,\ldots,\eta_m\in\mathcal{A}_{\pi}$, such that
for all $t\in T_0$,
\begin{align*}
W(t)=\sum_{\eta_1,\ldots,\eta_m\in\mathcal{A}_{\pi}}\phi_{\eta_1,\ldots,\eta_{m}}(\alpha_1(t),\ldots,\alpha_{m}(t))\prod_{i=1}^{m}\eta_i(\alpha_i(t)).
\end{align*}
In the case $G=\GL{\bigidx}$ this is a result of Jacquet,
Piatetski-Shapiro and Shalika \cite{JPSS2} (Section~2.2). The case
of $G_{\smallidx}$ can be deduced from their general results, as
done by Soudry \cite{Soudry} (Section~2) for $SO_{2\smallidx+1}$.
Note that in the \quasisplit\ case we will usually reduce integrals
over $T_{G_{\smallidx}}$ to integrals over $A_{\smallidx-1}$, so we
only need to evaluate $W$ on $T_0$.
\end{enumerate}

\section{A realization of $\cinduced{Q_{\bigidx_1,\bigidx_2}}{H_{\bigidx}}{(\tau_1\otimes\tau_2)\alpha^{(s_1,s_2)}}$}\label{subsection:realization of induced tau_1 and tau_2}
Let $\tau_i$ be a representation of $\GL{\bigidx_i}$, $i=1,2$, and
set $\bigidx=\bigidx_1+\bigidx_2$. Let
$Q_{\bigidx_1,\bigidx_2}<H_{\bigidx}$ be the standard parabolic
subgroup whose Levi part is
\begin{align*}
\setof{diag(b_1,b_2,1,b_2^*,b_1^*)}{b_i\in\GL{\bigidx_i}}\isomorphic
\GL{\bigidx_1}\times\GL{\bigidx_2}.
\end{align*}
Consider the induced representation
\begin{align*}
\Pi_{s_1,s_2}=\cinduced{Q_{\bigidx_1,\bigidx_2}}{H_{\bigidx}}{(\tau_1\otimes\tau_2)\alpha^{(s_1,s_2)}},
\end{align*}
where $s_1$ and $s_2$ are complex parameters (i.e.,
$\underline{s}=(s_1,s_2)$). As explained in
Section~\ref{subsection:sections} we have spaces of standard,
holomorphic and rational sections for $\Pi_{s_1,s_2}$. Consider also
\begin{align*}
\Pi_{s_1,s_2}'=\cinduced{Q_{\bigidx_1}}{H_{\bigidx}}{(\tau_1\otimes
\cinduced{Q_{\bigidx_2}}{H_{\bigidx_2}}{\tau_2\alpha^{s_2}})\alpha^{s_1}}
\end{align*}
(for $b_1\in\GL{\bigidx_1}$ and $h_2\in H_{\bigidx_2}$,
$\alpha^{s_1}(diag(b_1,h_2,b_1^*))=\alpha^{s_1}(b_1)$). These
representations are isomorphic via
\begin{align}\label{define iso}
(h,b_1,h_2,b_2)\mapsto f(h_2h,b_1,b_2),\qquad (h,b_1,b_2)\mapsto
\varphi(h,b_1,I_{2\bigidx_2+1},b_2).
\end{align}
Here $f$ belongs to the space of $\Pi_{s_1,s_2}$, $\varphi$ belongs
to the space of $\Pi_{s_1,s_2}'$, $h\in H_{\bigidx}$, $h_2\in
H_{\bigidx_2}$ and $b_i\in\GL{\bigidx_i}$. Furthermore \eqref{define
iso} defines isomorphisms between
\begin{align*}
\Pi=\cinduced{Q_{\bigidx_1,\bigidx_2}\cap
K_{H_{\bigidx}}}{K_{H_{\bigidx}}}{\tau_1\otimes\tau_2}
\end{align*}
and
\begin{align*}
\Pi'=\cinduced{Q_{\bigidx_1}\cap
K_{H_{\bigidx}}}{K_{H_{\bigidx}}}{\tau_1\otimes
\cinduced{Q_{\bigidx_2}\cap
K_{H_{\bigidx_2}}}{K_{H_{\bigidx_2}}}{\tau_2}}.
\end{align*}
Then for $\varphi$ in the space of $\Pi'$, we can first regard it as
a function in the space of $\Pi$, extend it to a function in the
space of $\Pi_{s_1,s_2}$ using the Iwasawa decomposition and finally
consider it as a function in the space of $\Pi_{s_1,s_2}'$.
Therefore we may realize the space
$\xi(\tau_1\otimes\tau_2,std,(s_1,s_2))=\xi_{Q_{\bigidx_1,\bigidx_2}}^{H_{\bigidx}}(\tau_1\otimes\tau_2,std,(s_1,s_2))$
by extending functions of the space of $\Pi'$ to functions of the
space of $\Pi_{s_1,s_2}'$. Then we can also realize
$\xi(\tau_1\otimes\tau_2,hol,(s_1,s_2))$ as the space of elements
$\sum_{i=1}^mP_i\varphi_{s_1,s_2}^{(i)}$ where $P_i\in\C[q^{\mp
s_1},q^{\mp s_2}]$, $\varphi_{s_1,s_2}^{(i)}\in
\xi(\tau_1\otimes\tau_2,std,(s_1,s_2))$ and similarly realize
$\xi(\tau_1\otimes\tau_2,rat,(s_1,s_2))$.

\section{The intertwining operators}\label{subsection:the intertwining operator}

\subsection{Intertwining operator for $\cinduced{Q_{\bigidx}}{H_{\bigidx}}{\tau\alpha^s}$}\label{subsection:the intertwining operator for tau}
Let $\tau$ be a representation of $\GL{\bigidx}$ realized in $\Whittaker{\tau}{\psi}$. 
Let
$\intertwiningfull{\tau}{s}:V_{Q_{\bigidx}}^{H_{\bigidx}}(\tau,s)\rightarrow
V_{Q_{\bigidx}}^{H_{\bigidx}}(\tau^*,1-s)$ be the standard
intertwining operator. It is given (formally) by the integral
\begin{align*}
\intertwiningfull{\tau}{s}f_s(h,b)=\int_{U_{\bigidx}}f_s(w_{\bigidx}uh,d_{\bigidx}b^*)du \qquad(h\in H_{\bigidx},b\in\GL{\bigidx}).
\end{align*}
Here
\begin{align*}
w_{\bigidx}=\left(\begin{array}{ccc}&&I_{\bigidx}\\&(-1)^{\bigidx}&\\I_{\bigidx}\\\end{array}\right),\qquad
d_{\bigidx}=diag(-1,1,\ldots,(-1)^{\bigidx})\in\GL{\bigidx}.
\end{align*}
Note that $\tau^*$ is realized in $\Whittaker{\tau^*}{\psi}$. This
integral converges absolutely for $\Re(s)>>0$ and has a meromorphic
continuation to a function in $\C(q^{-s})$. If
$f_s\in\xi(\tau,rat,s)$,
$\intertwiningfull{\tau}{s}f_s\in\xi(\tau^*,rat,1-s)$. Moreover
there is a polynomial $0\ne P\in\C[q^{-s}]$ such that
$P\intertwiningfull{\tau}{s}$ is a holomorphic operator, i.e.,
$P\intertwiningfull{\tau}{s}f_s\in\xi(\tau^*,hol,1-s)$ for any
$f_s\in\xi(\tau,hol,s)$. Since $\tau$ is generic, there is a
proportionality factor $\gamma(\tau,Sym^2,\psi,2s-1)$ such that for
all $f_s\in\xi(\tau,rat,s)$,
\begin{align*}
&\int_{U_{\bigidx}}f_s(d_{\bigidx}w_{\bigidx}u,1)\psi^{-1}(u_{\bigidx,\bigidx+1})du\\\notag&=\gamma(\tau,Sym^2,\psi,2s-1)\int_{U_{\bigidx}}\intertwiningfull{\tau}{s}
f_s(d_{\bigidx}w_{\bigidx}u,1)\psi^{-1}(u_{\bigidx,\bigidx+1})du.
\end{align*}
This is an equality of meromorphic continuations in $\C(q^{-s})$.
Note that the \lhs\ 
defines a Whittaker functional on $V(\tau,s)$ with respect to the
character $\psi(u)=\psi(\sum_{i=1}^{\bigidx}u_{i,i+1})$ of
$U_{H_{\bigidx}}$. This is Shahidi's functional equation (\cite{Sh4}
Theorem~3.1), the factor $\gamma(\tau,Sym^2,\psi,2s-1)$ is Shahidi's
local coefficient defined in \textit{loc. cit.} and $Sym^2$ is the
symmetric square representation.

Denote by $\nintertwiningfull{\tau}{s}$ the standard normalized
intertwining operator,
\begin{align*}
\nintertwiningfull{\tau}{s}=\gamma(\tau,Sym^2,\psi,2s-1)\intertwiningfull{\tau}{s}.
\end{align*}
Then Shahidi's functional equation takes the form
\begin{align}\label{eq:Shahidi local coefficient def}
\int_{U_{\bigidx}}f_s(d_{\bigidx}w_{\bigidx}u,1)\psi^{-1}(u_{\bigidx,\bigidx+1})du=\int_{U_{\bigidx}}\nintertwiningfull{\tau}{s}
f_s(d_{\bigidx}w_{\bigidx}u,1)\psi^{-1}(u_{\bigidx,\bigidx+1})du.
\end{align}

The properties stated above are usually formulated for an
irreducible $\tau$ (e.g. \cite{Sh4,Mu}). However, as shown by
Waldspurger \cite{W} (Section~IV), our assumptions on $\tau$ -
namely that it is smooth, admissible and finitely generated (see
Section~\ref{subsection:notation for reps}) are sufficient to deduce
the absolute convergence in a right half-plane and the meromorphic
properties. The functional equation also holds regardless of the
irreducibility condition, because $\tau$ is generic. The additional
assumption that $\tau$ has a central character $\omega_{\tau}$ is
needed in order to obtain simple multiplicative formulas, see for
example Section~\ref{subsection:the multiplicativity of the
intertwining operator for tau induced}.

Let $P\in\C[X]$ be a polynomial of minimal degree, with $P(0)=1$,
such that $P(q^{-s})\nintertwiningfull{\tau}{s}$ is a holomorphic
operator. We set $\ell_{\tau}(s)=P(q^{-s})^{-1}$. Then for any
$f_s\in\xi(\tau,hol,s)$ there exists
$f_{1-s}^*\in\xi(\tau^*,hol,1-s)$ such that
$\nintertwiningfull{\tau}{s}f_s=\ell_{\tau}(s)f_{1-s}^*$. Note that
$\ell_{\tau^*}(1-s)^{-1}\in\C[q^s]$.

From here until the end of this section assume that $\tau$ is
irreducible. Since $\tau$ is also generic, according to Shahidi
\cite{Sh4} Section~3,
\begin{align}\label{eq:multiplication of normalized intertwiners}
\nintertwiningfull{\tau}{s}\nintertwiningfull{\tau^*}{1-s}=1.
\end{align}
To explain this, note that the irreducibility of $\tau$ implies that
$\nintertwiningfull{\tau}{s}\nintertwiningfull{\tau^*}{1-s}$ is a
\nonzero\ scalar $c(s)$ multiplied by the identity mapping, for
almost all $s$ (\cite{Sh4} Section~2). Then if
$\whittakerfunctional$ denotes the Jacquet integral on the \lhs\ of
\eqref{eq:Shahidi local coefficient def} and $f_{1-s}'\in
V(\tau^*,1-s)$, applying \eqref{eq:Shahidi local coefficient def}
twice yields
\begin{align*}
\whittakerfunctional(f_{1-s}')=\whittakerfunctional(\nintertwiningfull{\tau^*}{1-s}f_{1-s}')=
\whittakerfunctional(\nintertwiningfull{\tau}{s}\nintertwiningfull{\tau^*}{1-s}f_{1-s}')=c(s)\whittakerfunctional(f_{1-s}').
\end{align*}
Hence for almost all $s$, $c(s)=1$, implying
\eqref{eq:multiplication of normalized intertwiners}.

We collect a few results regarding the poles of the intertwining
operator. The $L$-group of the Levi part of $Q_{\bigidx}$ is
$\GLF{\bigidx}{\C}$. The adjoint action of $\GLF{\bigidx}{\C}$ on
the Lie algebra of the $L$-group of $U_{\bigidx}$ is $Sym^2$, which
is irreducible (see e.g. \cite{Sh5} p. 5). Therefore, in this case
the local coefficient $\gamma(\tau,Sym^2,\psi,2s-1)$ and Shahidi's
$\gamma$-factor are equal, up to a unit in $\C[q^{-s},q^s]$
(\cite{Sh3} Theorem~3.5).
According to Shahidi \cite{Sh3} (Section~7), 
\begin{align}\label{eq:gamma up to units}
\gamma(\tau,Sym^2,\psi,2s-1)\equalun\frac{L(\tau^*,Sym^2,2-2s)}{L(\tau,Sym^2,2s-1)}
\end{align}
($\equalun$ means up to invertible factors in $\C[q^{-s},q^s]$).

We have the following result of Shahidi (\textit{loc. cit.}
Proposition~7.2a),
\begin{theorem}\label{theorem:tempered L function holomorphic for half plane}
For a tempered $\tau$, $L(\tau,Sym^2,s)$ is holomorphic for
$\Re(s)>0$.
\end{theorem}
This 
is a part of a more general conjecture of Shahidi (\textit{loc. cit.} Conjecture~7.1), proved by Casselman and Shahidi \cite{CSh} (Section~4) in the broad context of  groups of classical type (e.g. classical groups).

The following result of Casselman and Shahidi \cite{CSh}
(Theorem~5.1) will be used to bound the poles of the intertwining
operator.
\begin{theorem}\label{theorem:poles of normalized intertwining for tau cuspidal}
For a supercuspidal (irreducible) $\tau$,
$L(\tau,Sym^2,2s-1)^{-1}\intertwiningfull{\tau}{s}$ is holomorphic.
\end{theorem}
\begin{remark}
The result is stated in \textit{loc. cit.} for standard modules
which satisfy injectivity at a certain level, in particular it is
valid for standard modules induced from generic irreducible
supercuspidal representations (\textit{loc. cit.} Theorem~3.4).
\end{remark}
For an essentially tempered $\tau$, according to
Theorem~\ref{theorem:tempered L function holomorphic for half
plane}, the quotient on the \rhs\ of \eqref{eq:gamma up to units} is
reduced. Therefore if
$L(\tau,Sym^2,2s-1)^{-1}\intertwiningfull{\tau}{s}$ is holomorphic
(e.g. $\tau$ is irreducible supercuspidal),
$\ell_{\tau}(s)=L(\tau^*,Sym^2,2-2s)$. 

\subsection{Intertwining operator for $\cinduced{P_{\bigidx_1,\bigidx_2}}{\GL{\bigidx}}{\tau_1\alpha^{\half\bigidx+s}\otimes\tau_2^*\alpha^{\half\bigidx+1-s}}$}\label{subsection:the intertwining operator for tau_1 and tau_2}
Let $\tau_i$ be a representation of $\GL{\bigidx_i}$ 
realized in $\Whittaker{\tau_i}{\psi}$, $i=1,2$.
We have the standard intertwining operator
\begin{align*}
\intertwiningfull{\tau_1\otimes\tau_2^*}{(s,1-s)}:&V_{P_{\bigidx_1,\bigidx_2}}^{\GL{\bigidx}}
(\tau_1\absdet{}^{\frac{\bigidx}2}\otimes\tau_2^*\absdet{}^{\frac{\bigidx}2},(s,1-s))\\&\rightarrow
V_{P_{\bigidx_2,\bigidx_1}}^{\GL{\bigidx}}(\tau_2^*\absdet{}^{\frac{\bigidx}2}\otimes\tau_1\absdet{}^{\frac{\bigidx}2},(1-s,s)),
\end{align*}
defined via the meromorphic continuation of
\begin{align*}
(b,b_2,b_1)\mapsto
\int_{Z_{\bigidx_2,\bigidx_1}}\theta_s(\omega_{\bigidx_1,\bigidx_2}zb,b_1,b_2)dz\qquad(b\in\GL{\bigidx},b_i\in\GL{\bigidx_i}).
\end{align*}
The standard normalized intertwining operator
$\nintertwiningfull{\tau_1\otimes\tau_2^*}{(s,1-s)}$ is defined
according to Shahidi's functional equation (\cite{Sh4} Theorem~3.1),
\begin{align}\label{eq:shahidi func eq}
&\int_{Z_{\bigidx_2,\bigidx_1}}\theta_s(\omega_{\bigidx_1,\bigidx_2}z,I_{\bigidx_1},I_{\bigidx_2})\psi^{-1}(z)dz\\\notag
&=\int_{Z_{\bigidx_1,\bigidx_2}}\nintertwiningfull{\tau_1\otimes\tau_2^*}{(s,1-s)}\theta_s(\omega_{\bigidx_2,\bigidx_1}m,I_{\bigidx_2},I_{\bigidx_1})\psi^{-1}(m)dm,
\end{align}
with $\theta_s\in
\xi_{P_{\bigidx_1,\bigidx_2}}^{\GL{\bigidx}}(\tau_1\absdet{}^{\frac{\bigidx}2}\otimes\tau_2^*\absdet{}^{\frac{\bigidx}2},rat,(s,1-s))$.
As in Section~\ref{subsection:the intertwining operator for tau},
this is an equality in $\C(q^{-s})$. Define
$\ell_{\tau_1\otimes\tau_2^*}(s)^{-1}\in\C[q^{-s}]$ similarly to
$\ell_{\tau}(s)$, i.e., such that
$\ell_{\tau_1\otimes\tau_2^*}(s)^{-1}\nintertwiningfull{\tau_1\otimes\tau_2^*}{(s,1-s)}$
is holomorphic.

Assume that $\tau_1$ and $\tau_2$ are irreducible. Then according to
Shahidi \cite{Sh3} (Theorem~3.5), there is some
$e(q^{-s},q^{s})\in\C[q^{-s},q^s]^*$ satisfying
\begin{align*}
\nintertwiningfull{\tau_1\otimes\tau_2^*}{(s,1-s)}=e(q^{-s},q^{s})
\frac{L(\tau_1^*\times\tau_2^*,2-2s)}{L(\tau_1\times\tau_2,2s-1)}
\intertwiningfull{\tau_1\otimes\tau_2^*}{(s,1-s)}.
\end{align*}
The result of Casselman and Shahidi \cite{CSh} (Theorem~5.1) now
reads,
\begin{theorem}\label{theorem:poles of normalized intertwining for tau_1 tau_2 cuspidal}
For irreducible supercuspidal $\tau_1$ and $\tau_2$,
$L(\tau_1\times\tau_2,2s-1)^{-1}\intertwiningfull{\tau_1\otimes\tau_2^*}{(s,1-s)}$
is holomorphic.
\end{theorem}
In the case of the theorem,
$\ell_{\tau_1\otimes\tau_2^*}(s)=L(\tau_1^*\times\tau_2^*,2-2s)$.
\subsection{The multiplicativity of $\nintertwiningfull{\tau}{s}$ for an induced $\tau$}\label{subsection:the multiplicativity of the intertwining operator for tau induced}
Assume
$\tau=\cinduced{P_{\bigidx_1,\bigidx_2}}{\GL{\bigidx}}{\tau_1\otimes\tau_2}$,
where each representation $\tau_i$ 
is realized in $\Whittaker{\tau_i}{\psi}$ and $\tau$ is realized in
the space of the induced representation.

The elements of $V_{Q_{\bigidx}}^{H_{\bigidx}}(\tau,s)$ are locally
constant functions on $H_{\bigidx}$ taking values in the space of
$\tau$, we regard them as locally constant functions
$f_s:H_{\bigidx}\times\GL{\bigidx}\times\GL{\bigidx_1}\times\GL{\bigidx_2}\rightarrow\C$
with the following properties: let $h\in H_{\bigidx}$,
$b\in\GL{\bigidx}$, $b_i\in\GL{\bigidx_i}$.
\begin{enumerate}
\item For $q\in Q_{\bigidx}$ with $q=au$, $a\in\GL{\bigidx}\isomorphic
M_{\bigidx}$ and $u\in U_{\bigidx}$,
\begin{align*}
f_s(qh,b,b_1,b_2)=\delta_{Q_{\bigidx}}^{\half}(q)\absdet{a}^{s-\half}f_s(h,ba,b_1,b_2).
\end{align*}
\item For $a\in P_{\bigidx_1,\bigidx_2}$ with $a=diag(a_1,a_2)z$, $a_i\in\GL{\bigidx_i}$ and $z\in
Z_{\bigidx_1,\bigidx_2}$,
\begin{align*}
f_s(h,ab,b_1,b_2)=\delta_{P_{\bigidx_1,\bigidx_2}}^{\half}(a)f_s(h,b,b_1a_1,b_2a_2).
\end{align*}
\item The mapping $b_1\mapsto f_s(h,b,b_1,b_2)$ belongs to
$\Whittaker{\tau_1}{\psi}$ and $b_2\mapsto f_s(h,b,b_1,b_2)$ belongs
to $\Whittaker{\tau_2}{\psi}$.
\end{enumerate}

Let
$\intertwiningfull{\tau}{s}:V_{Q_{\bigidx}}^{H_{\bigidx}}(\tau,s)\rightarrow
V_{Q_{\bigidx}}^{H_{\bigidx}}(\tau^*,1-s)$ be the standard
intertwining operator, defined by the meromorphic continuation of
the integral
\begin{align*}
\intertwiningfull{\tau}{s}f_s(h,b,b_2,b_1)=\int_{U_{\bigidx}}f_s(w_{\bigidx}u,b^*,d_{\bigidx_1}b_1^*,d_{\bigidx_2}b_2^*)du.
\end{align*}
The standard normalized intertwining operator is defined by the
functional equation
\begin{align*}
&\int_{U_{\bigidx}}\int_{Z_{\bigidx_2,\bigidx_1}}f_s(w_{\bigidx}u,\omega_{\bigidx_1,\bigidx_2}zd_{\bigidx},I_{\bigidx_1},I_{\bigidx_2})\psi^{-1}(z)\psi^{-1}(u)dzdu\\\notag
&=\int_{U_{\bigidx}}\int_{Z_{\bigidx_1,\bigidx_2}}\nintertwiningfull{\tau}{s}f_s(w_{\bigidx}u,\omega_{\bigidx_2,\bigidx_1}zd_{\bigidx},I_{\bigidx_2},I_{\bigidx_1})\psi^{-1}(z)\psi^{-1}(u)dzdu.
\end{align*}
This is an equality in $\C(q^{-s})$. According to the
multiplicativity of the intertwining operators and local
coefficients (\cite{Sh4} Theorem~2.1.1 and Proposition~3.2.1),
\begin{align}\label{eq:multiplicative property of intertwiners}
\nintertwiningfull{\tau}{s}=\nintertwiningfull{\tau_1}{s}\nintertwiningfull{\tau_1\otimes\tau_2^*}{(s,1-s)}\nintertwiningfull{\tau_2}{s}.
\end{align}
We explain how to interpret \eqref{eq:multiplicative property of
intertwiners}.

Let $s_1,s_2$ be complex parameters. Consider the representation
\begin{align*}
\Pi_{s_1,s_2}'=\cinduced{Q_{\bigidx_1}}{H_{\bigidx}}{(\tau_1\otimes
\cinduced{Q_{\bigidx_2}}{H_{\bigidx_2}}{\tau_2\alpha^{s_2}})\alpha^{s_1}}
\end{align*}
defined in Section~\ref{subsection:realization of induced tau_1 and
tau_2}. Denote its space by $V'(\tau_1\otimes\tau_2,(s_1,s_2))$. Let
$\varphi_{s_1,s_2}\in V'(\tau_1\otimes\tau_2,(s_1,s_2))$. Any fixed
$h_1\in H_{\bigidx_1}$ and $b_1\in\GL{\bigidx_1}$ define a mapping
$\varphi_{s_1,s_2}(h_1,b_1,\cdot,\cdot)\in
V_{Q_{\bigidx_2}}^{H_{\bigidx_2}}(\tau_2,s_2)$ by
\begin{align*}
(h_2,b_2)\mapsto\varphi_{s_1,s_2}(h,b_1,h_2,b_2)\qquad (h_2\in
H_{\bigidx_2},b_2\in\GL{\bigidx_2}).
\end{align*}
In addition, any $h\in H_{\bigidx}$ defines a mapping
$h\cdot\varphi_{s_1,s_2}(\cdot,\cdot,I_{2\bigidx_2+1},\cdot)\in
V_{P_{\bigidx_1,\bigidx_2}}^{\GL{\bigidx}}(\tau_1\absdet{}^{\frac{\bigidx}2}\otimes\tau_2\absdet{}^{\frac{\bigidx}2},(s_1,s_2))$
by
\begin{align*}
(b,b_1,b_2)\mapsto
\varphi_{s_1,s_2}(bh,b_1,I_{2\bigidx_2+1},b_2)\qquad
(b\in\GL{\bigidx},b_i\in\GL{\bigidx_i}).
\end{align*}

Let $s\in\C$. The representations
$V_{Q_{\bigidx}}^{H_{\bigidx}}(\tau,s)$ and $\Pi_{s,s}'$ are
isomorphic according to the mappings
\begin{align*}
(h,b_1,h_2,b_2)\mapsto f_s(h_2h,I_{\bigidx},b_1,b_2),\quad
(h,b,b_1,b_2)\mapsto\absdet{b}^{-\half\bigidx-s+\half}\varphi_s(bh,b_1,I_{2\bigidx_2+1},b_2).
\end{align*}
Here $f_s\in V_{Q_{\bigidx}}^{H_{\bigidx}}(\tau,s)$ and
$\varphi_s\in V'(\tau_1\otimes\tau_2,(s,s))$. Denote the image of
$f_s$ in $V'(\tau_1\otimes\tau_2,(s,s))$ by $\imath(f_s)$. 

Assume that $s$ is chosen away from the poles of the intertwining
operators on the \rhs\ of \eqref{eq:multiplicative property of
intertwiners}. Let $\varphi_s\in V'(\tau_1\otimes\tau_2,(s,s))$. The
mapping
\begin{align*}
\intertwiningfull{\tau_2}{s}\varphi_s\in
V'(\tau_1\otimes\tau_2^*,(s,1-s))
\end{align*}
is obtained by applying $\intertwiningfull{\tau_2}{s}$ to
$\varphi_s(h,b_1,\cdot,\cdot)$. Formally, 
\begin{align*}
\intertwiningfull{\tau_2}{s}\varphi_s(h,b_1,h_2,b_2)=\int_{U_{\bigidx_2}}\varphi_s(h,b_1,w_{\bigidx_2}uh_2,d_{\bigidx_2}b_2^*)du.
\end{align*}
Then $\varphi_{s,1-s}\in V'(\tau_1\otimes\tau_2^*,(s,1-s))$ is given
by
\begin{align*}
\varphi_{s,1-s}=\nintertwiningfull{\tau_2}{s}\varphi_s=\gamma(\tau_2,Sym^2,\psi,2s-1)\intertwiningfull{\tau_2}{s}\varphi_s.
\end{align*}

The function
\begin{align*}
\varphi_{1-s,s}'=\nintertwiningfull{\tau_1\otimes\tau_2^*}{(s,1-s)}\varphi_{s,1-s}\in
V'(\tau_2^*\otimes\tau_1,(1-s,s))
\end{align*}
is defined by applying
$\nintertwiningfull{\tau_1\otimes\tau_2^*}{(s,1-s)}$ to the function
$h\cdot\varphi_{s,1-s}(\cdot,\cdot,I_{2\bigidx_2+1},\cdot)$.
Formally,
\begin{align*}
\varphi_{1-s,s}'(h,b_2,h_1,b_1)&=(\nintertwiningfull{\tau_1\otimes\tau_2^*}{(s,1-s)}(h_1h)\cdot\varphi_{s,1-s})(I_{\bigidx},b_2,b_1)\\\notag
&=\gamma(\tau_1\times\tau_2,\psi,2s-1)'\int_{Z_{\bigidx_2,\bigidx_1}}\varphi_{s,1-s}(\omega_{\bigidx_1,\bigidx_2}zh_1h,b_1,I_{2\bigidx_2+1},b_2)dz.
\end{align*}
Here $\gamma(\tau_1\times\tau_2,\psi,2s-1)'$ denotes the local
coefficient. Then
\begin{align*}
\varphi_{1-s}^*=\nintertwiningfull{\tau_1}{s}\varphi_{1-s,s}'\in
V'(\tau_2^*\otimes\tau_1^*,(1-s,1-s))
\end{align*}
is obtained similarly to $\varphi_{s,1-s}$. Collecting the
applications of the intertwining operators,
\begin{align*}
\varphi_{1-s}^*=
\nintertwiningfull{\tau_1}{s}\nintertwiningfull{\tau_1\otimes\tau_2^*}{(s,1-s)}\nintertwiningfull{\tau_2}{s}\varphi_s.
\end{align*}

The interpretation of \eqref{eq:multiplicative property of
intertwiners} is
\begin{align*}
\imath(\nintertwiningfull{\tau}{s}f_s)=\nintertwiningfull{\tau_1}{s}\nintertwiningfull{\tau_1\otimes\tau_2^*}{(s,1-s)}\nintertwiningfull{\tau_2}{s}\imath(f_s),
\end{align*}
or if we put $\varphi_s=\imath(f_s)$,
\begin{align*}
\imath(\nintertwiningfull{\tau}{s}f_s)=\varphi_{1-s}^*.
\end{align*}
As mentioned above, this follows from the results of Shahidi
\cite{Sh4}. However, we include a formal proof of this equality,
showing how the definitions above are combined together.
\begin{claim}\label{claim:verification of multiplicativity intertwiners}
$\imath(\nintertwiningfull{\tau}{s}f_s)=\varphi_{1-s}^*$.
\end{claim}
\begin{proof}[Proof of Claim~\ref{claim:verification of multiplicativity intertwiners}]
In order to prove this equality, we start by showing
\begin{align}\label{eq:non normalized multiplicative property of intertwiners}
\imath(\intertwiningfull{\tau}{s}f_s)=\intertwiningfull{\tau_1}{s}\intertwiningfull{\tau_1\otimes\tau_2^*}{(s,1-s)}\intertwiningfull{\tau_2}{s}\varphi_s.
\end{align}
Both sides are functions in $V'(\tau_2^*\otimes\tau_1^*,(1-s,1-s))$
and it is enough to show that they are equal on $h\in H_{\bigidx}$.
This equality is first interpreted as an equality of integrals for
$\Re(s)>>0$, then it follows for all $s$ by meromorphic
continuation. For $\Re(s)>>0$ we have,
\begin{align*}
&\imath(\intertwiningfull{\tau}{s}f_s)(h,I_{\bigidx_2},I_{2\bigidx_1+1},I_{\bigidx_1})
=\intertwiningfull{\tau}{s}f_s(h,I_{\bigidx},I_{\bigidx_2},I_{\bigidx_1})\\\notag
&=\int_{U_{\bigidx}}f_s(w_{\bigidx}uh,I_{\bigidx},d_{\bigidx_1},d_{\bigidx_2})du
=\int_{U_{\bigidx}}\varphi_s(w_{\bigidx}uh,d_{\bigidx_1},I_{2\bigidx_2+1},d_{\bigidx_2})du.
\end{align*}
Write
$w_{\bigidx}=w_{\bigidx_2}\omega_{\bigidx_1,\bigidx_2}w_{\bigidx_1}$,
where $w_{\bigidx_i}\in H_{\bigidx_i}<M_{\bigidx_{3-i}}$ and
$\omega_{\bigidx_1,\bigidx_2}\in M_{\bigidx}$. Then
$w_{\bigidx}u=w_{\bigidx_2}u_2\omega_{\bigidx_1,\bigidx_2}zw_{\bigidx_1}u_1$
where $u_i\in U_{\bigidx_i}<H_{\bigidx_i}$ and $z\in
Z_{\bigidx_2,\bigidx_1}<M_{\bigidx}$. This decomposition also
implies a decomposition of the measure $du=du_2dzdu_1$.
The last integral equals
\begin{align*}
&\int_{U_{\bigidx_1}}\int_{Z_{\bigidx_2,\bigidx_1}}\int_{U_{\bigidx_2}}
\varphi_s(w_{\bigidx_2}u_2\omega_{\bigidx_1,\bigidx_2}zw_{\bigidx_1}u_1h,d_{\bigidx_1},I_{2\bigidx_2+1},d_{\bigidx_2})du_2dzdu_1\\\notag
&=\int_{U_{\bigidx_1}}\int_{Z_{\bigidx_2,\bigidx_1}}
\intertwiningfull{\tau_2}{s}\varphi_s(\omega_{\bigidx_1,\bigidx_2}zw_{\bigidx_1}u_1h,d_{\bigidx_1},I_{2\bigidx_2+1},I_{\bigidx_2})dzdu_1\\\notag
&=\int_{U_{\bigidx_1}}\intertwiningfull{\tau_1\otimes\tau_2^*}{(s,1-s)}\intertwiningfull{\tau_2}{s}\varphi_s(w_{\bigidx_1}u_1h,I_{\bigidx_2},I_{2\bigidx_1+1},d_{\bigidx_1})du_1\\\notag
&=\intertwiningfull{\tau_1}{s}\intertwiningfull{\tau_1\otimes\tau_2^*}{(s,1-s)}\intertwiningfull{\tau_2}{s}\varphi_s(h,I_{\bigidx_2},I_{2\bigidx_1+1},I_{\bigidx_1}).
\end{align*}
This proves \eqref{eq:non normalized multiplicative property of
intertwiners}.

In order to complete the proof we show that $\varphi_{1-s}^*$
satisfies the same functional equation as
$\nintertwiningfull{\tau}{s}f_s$. I.e.,
\begin{align}\label{int:func equation intertwining operators start}
&\int_{U_{\bigidx}}\int_{Z_{\bigidx_2,\bigidx_1}}\varphi_s(\omega_{\bigidx_1,\bigidx_2}zd_{\bigidx}w_{\bigidx}u,I_{\bigidx_1},I_{2\bigidx_2+1},I_{\bigidx_2})\psi^{-1}(z)\psi^{-1}(u)dzdu\\\notag
&=\int_{U_{\bigidx}}\int_{Z_{\bigidx_1,\bigidx_2}}\varphi_{1-s}^*(\omega_{\bigidx_2,\bigidx_1}zd_{\bigidx}w_{\bigidx}u,I_{\bigidx_2},I_{2\bigidx_1+1},I_{\bigidx_1})\psi^{-1}(z)\psi^{-1}(u)dzdu.
\end{align}
Start with the \rhs. Here our arguments are formal, ignoring
convergence issues. Shift $\omega_{\bigidx_2,\bigidx_1}z$ to the
right of $d_{\bigidx}w_{\bigidx}u$. Put
$d'=diag((-1)^{\bigidx_1}d_{\bigidx_2},I_{\bigidx_1})$. Note that $\psi^{-1}(u)=\psi^{-1}(u_{\bigidx,\bigidx+1})$. We get 
\begin{align*}
&\int_{Z_{\bigidx_2,\bigidx_1}}\int_{U_{\bigidx}}\varphi_{1-s}^*(d'w_{\bigidx}u\omega_{\bigidx_1,\bigidx_2}z,I_{\bigidx_2},I_{2\bigidx_1+1},d_{\bigidx_1})\psi^{-1}(z)\psi^{-1}(u_{\bigidx_1,\bigidx+1})dudz.
\end{align*}
Write
$w_{\bigidx}=w_{\bigidx_1}\omega_{\bigidx_2,\bigidx_1}w_{\bigidx_2}$
and decompose
$U_{\bigidx}=\rconj{(\omega_{\bigidx_2,\bigidx_1}w_{\bigidx_2})}U_{\bigidx_1}\ltimes
V_{\bigidx_1,\bigidx_2}$, where
\begin{align*}
V_{\bigidx_1,\bigidx_2}=\{\left(\begin{array}{ccccc}I_{\bigidx_1}&&0&v_1&0\\&I_{\bigidx_2}&v_2&v_3&v_1'\\&&1&v_2'&0\\&&&I_{\bigidx_2}\\&&&&I_{\bigidx_1}\end{array}\right)\}.
\end{align*}
The integral becomes
\begin{align}\label{int:multiplic intertwining normalized before func 1}
&\int_{Z_{\bigidx_2,\bigidx_1}}\int_{V_{\bigidx_1,\bigidx_2}}\int_{U_{\bigidx_1}}\varphi_{1-s}^*
(d'\omega_{\bigidx_2,\bigidx_1}w_{\bigidx_2}v\omega_{\bigidx_1,\bigidx_2}z,I_{\bigidx_2},I_{2\bigidx_1+1},d_{\bigidx_1}w_{\bigidx_1}u_1)\\\notag&\psi^{-1}(z)\psi^{-1}((-1)^{\bigidx_2}(u_1)_{\bigidx_1,\bigidx_1+1})du_1dvdz.
\end{align}
The $du_1$-integration resembles the \rhs\ of \eqref{eq:Shahidi
local coefficient def}, except for the character
$\psi^{-1}((-1)^{\bigidx_2}(u_1)_{\bigidx_1,\bigidx_1+1})$ instead
of $\psi^{-1}((u_1)_{\bigidx_1,\bigidx_1+1})$. However,
equality~\eqref{eq:Shahidi local coefficient def} also holds with
this change, to see this replace $f_s$ with $y\cdot f_s$ in
\eqref{eq:Shahidi local coefficient def} where\\
$y=diag((-1)^{\bigidx_2}I_{\bigidx_1},1,(-1)^{\bigidx_2}I_{\bigidx_1})\in
H_{\bigidx_1}$ ($\bigidx$ of \eqref{eq:Shahidi local coefficient
def} is now $\bigidx_1$) and use the fact that $\tau_1$ has a
central character.

Thus we see that the last integral is equal, as a meromorphic
continuation, to
\begin{align*}
&\int_{Z_{\bigidx_2,\bigidx_1}}\int_{V_{\bigidx_1,\bigidx_2}}\int_{U_{\bigidx_1}}\varphi_{1-s,s}'
(d'\omega_{\bigidx_2,\bigidx_1}w_{\bigidx_2}v\omega_{\bigidx_1,\bigidx_2}z,I_{\bigidx_2},I_{2\bigidx_1+1},d_{\bigidx_1}w_{\bigidx_1}u_1)\\\notag&\psi^{-1}(z)\psi^{-1}((-1)^{\bigidx_2}(u_1)_{\bigidx_1,\bigidx_1+1})du_1dvdz.
\end{align*}
Reversing the steps leading to \eqref{int:multiplic intertwining
normalized before func 1} we obtain
\begin{align*}
&\int_{U_{\bigidx}}\int_{Z_{\bigidx_1,\bigidx_2}}\varphi_{1-s,s}'(\omega_{\bigidx_2,\bigidx_1}zd_{\bigidx}w_{\bigidx}u,I_{\bigidx_2},I_{2\bigidx_1+1},I_{\bigidx_1})\psi^{-1}(z)\psi^{-1}(u)dzdu.
\end{align*}
Using the definition of $\varphi_{1-s,s}'$ yields
\begin{align*}
&\int_{U_{\bigidx}}\int_{Z_{\bigidx_1,\bigidx_2}}(\nintertwiningfull{\tau_1\otimes\tau_2^*}{(s,1-s)}(d_{\bigidx}w_{\bigidx}u)\cdot
\varphi_{s,1-s})(\omega_{\bigidx_2,\bigidx_1}z,I_{\bigidx_2},I_{\bigidx_1})\psi^{-1}(z)\psi^{-1}(u)dzdu.
\end{align*}
The $dz$-integration comprises the \rhs\ of \eqref{eq:shahidi func
eq}. Applying this equation, we get
\begin{align*}
&\int_{U_{\bigidx}}\int_{Z_{\bigidx_2,\bigidx_1}}\varphi_{s,1-s}(\omega_{\bigidx_1,\bigidx_2}zd_{\bigidx}w_{\bigidx}u,I_{\bigidx_1},I_{2\bigidx_2+1},I_{\bigidx_2})\psi^{-1}(z)\psi^{-1}(u)dzdu.
\end{align*}
Repeat the steps leading from the \rhs\ of \eqref{int:func equation
intertwining operators start} to \eqref{int:multiplic intertwining
normalized before func 1} on the last integral with $\bigidx_1$ and
$\bigidx_2$ exchanged. Then apply \eqref{eq:Shahidi local
coefficient def} (with
$\psi^{-1}((-1)^{\bigidx_1}(u_2)_{\bigidx_2,\bigidx_2+1})$ instead
of $\psi^{-1}((u_2)_{\bigidx_2,\bigidx_2+1})$, $u_2\in
U_{\bigidx_2}$) and obtain
\begin{align*}
&\int_{U_{\bigidx}}\int_{Z_{\bigidx_2,\bigidx_1}}\varphi_s(\omega_{\bigidx_1,\bigidx_2}zd_{\bigidx}w_{\bigidx}u,I_{\bigidx_1},I_{2\bigidx_2+1},I_{\bigidx_2})\psi^{-1}(z)\psi^{-1}(u)dzdu.
\end{align*}
This is the \lhs\ of \eqref{int:func equation intertwining operators
start}.
\end{proof} 

We usually use the elements of $V'(\tau_1\otimes\tau_2,(s,s))$
instead of $V_{Q_{\bigidx}}^{H_{\bigidx}}(\tau,s)$. Let
$\varphi_s=\iota(f_s)$ where $f_s\in
V_{Q_{\bigidx}}^{H_{\bigidx}}(\tau,s)$. Then we denote by
$\nintertwiningfull{\tau}{s}\varphi_s$ the function obtained by
repeatedly applying the intertwining operators on the \rhs\ of
\eqref{eq:multiplicative property of intertwiners}. In other words,
$\nintertwiningfull{\tau}{s}\varphi_s=\varphi_{1-s}^*$. The last
claim shows
$\imath(\nintertwiningfull{\tau}{s}\imath^{-1}(\varphi_s))=\varphi_{1-s}^*$,
so writing $\nintertwiningfull{\tau}{s}\varphi_s$ is compliant with
the ``hidden" identifications. 

Let $s\in\C$. A function $\varphi_s\in
V'(\tau_1\otimes\tau_2,(s,s))$ defines an element
$\widehat{f}_{\varphi_s}\in
V_{Q_{\bigidx}}^{H_{\bigidx}}(\Whittaker{\tau}{\psi},s)$ via the
formula
\begin{align*}
&\widehat{f}_{\varphi_s}(h,b)=\absdet{b}^{-\half\bigidx-s+\half}\int_{Z_{\bigidx_2,\bigidx_1}}\varphi_s(\omega_{\bigidx_1,\bigidx_2}zbh,I_{\bigidx_1},I_{2\bigidx_2+1},I_{\bigidx_2})
\psi^{-1}(z)dz.
\end{align*}
Here $\omega_{\bigidx_1,\bigidx_2}$ is regarded as an element of
$M_{\bigidx}<Q_{\bigidx}$ and the Jacquet integral is defined as a
principal value.

The function $\widehat{f}_{\varphi_{1-s}^*}\in
V_{Q_{\bigidx}}^{H_{\bigidx}}(\Whittaker{\tau^*}{\psi},1-s)$ is
defined by
\begin{align*}
&\widehat{f}_{\varphi_{1-s}^*}(h,b)=\absdet{b}^{-\half\bigidx+s-\half}\int_{Z_{\bigidx_1,\bigidx_2}}\varphi_{1-s}^*(\omega_{\bigidx_2,\bigidx_1}zbh,I_{\bigidx_2},I_{2\bigidx_1+1},I_{\bigidx_1})
\psi^{-1}(z)dz.
\end{align*}
Because the application of the intertwining operator ``commutes"
with the Jacquet integral, 
\begin{claim}\label{claim:intertwining operator and Jacquet integral
commute} Let $\varphi_s\in V'(\tau_1\otimes\tau_2,(s,s))$. Then
$\nintertwiningfull{\tau}{s}\widehat{f}_{\varphi_s}=\widehat{f}_{\nintertwiningfull{\tau}{s}\varphi_s}$.
Here on the \lhs,
$\nintertwiningfull{\tau}{s}:V_{Q_{\bigidx}}^{H_{\bigidx}}(\Whittaker{\tau}{\psi},s)\rightarrow
V_{Q_{\bigidx}}^{H_{\bigidx}}(\Whittaker{\tau^*}{\psi},1-s)$ is
defined as explained in Section~\ref{subsection:the intertwining
operator for tau}. On the \rhs,
$\nintertwiningfull{\tau}{s}\varphi_s=\varphi_{1-s}^*$.
\end{claim}
\begin{proof}[Proof of Claim~\ref{claim:intertwining operator and Jacquet
integral commute}] Denote
\begin{align*}
\intertwiningfull{\tau}{s}\varphi_s=\imath(\intertwiningfull{\tau}{s}\imath^{-1}(\varphi_s))=
\intertwiningfull{\tau_1}{s}\intertwiningfull{\tau_1\otimes\tau_2^*}{(s,1-s)}\intertwiningfull{\tau_2}{s}\varphi_s
\end{align*}
(the second equality is \eqref{eq:non normalized multiplicative
property of intertwiners}). First we show
\begin{align}\label{intertwining non normalized commutes with an
extra factor} \intertwiningfull{\tau}{s}\widehat{f}_{\varphi_s}=
\omega_{\tau_1}(-1)^{\bigidx_2}\widehat{f}_{\intertwiningfull{\tau}{s}\varphi_s}.
\end{align}
For any $h\in H_{\bigidx}$,
\begin{align*}
\intertwiningfull{\tau}{s}\widehat{f}_{\varphi_s}(h,1)=
\int_{U_{\bigidx}}\int_{Z_{\bigidx_2,\bigidx_1}}\varphi_s(\omega_{\bigidx_1,\bigidx_2}zd_{\bigidx}w_{\bigidx}uh,I_{\bigidx_1},I_{2\bigidx_2+1},I_{\bigidx_2})\psi^{-1}(z)dzdu.
\end{align*}
Here we regard the $dzdu$-integration as a double integral (recall
that the $dz$-integration was defined by principal value). We
proceed formally with the argument and provide the justifications
below. Write
$d_{\bigidx}=diag(d_{\bigidx_2},(-1)^{\bigidx_2}d_{\bigidx_1})$. The
last integral becomes
\begin{align*}
\omega_{\tau_1}(-1)^{\bigidx_2}\int_{Z_{\bigidx_1,\bigidx_2}}\int_{U_{\bigidx}}\varphi_s(w_{\bigidx}u\omega_{\bigidx_2,\bigidx_1}zh,d_{\bigidx_1},I_{2\bigidx_2+1},d_{\bigidx_2})\psi^{-1}(z)dudz.
\end{align*}
As we have seen in the proof of Claim~\ref{claim:verification of
multiplicativity intertwiners}, this is equal to
\begin{align*}
&\omega_{\tau_1}(-1)^{\bigidx_2}\int_{Z_{\bigidx_1,\bigidx_2}}\intertwiningfull{\tau}{s}\varphi_s(\omega_{\bigidx_2,\bigidx_1}zh,I_{\bigidx_2},I_{2\bigidx_1+1},I_{\bigidx_1})\psi^{-1}(z)dz\\\notag
&=\omega_{\tau_1}(-1)^{\bigidx_2}\widehat{f}_{\intertwiningfull{\tau}{s}\varphi_s}(h,1).
\end{align*}

We sketch the justifications to the formal manipulations. As in
\cite{Soudry2} (Section~2), replace $\tau_1$ with
$\tau_1\absdet{}^{\zeta}$ and $\tau_2$ with
$\tau_2\absdet{}^{-\zeta}$, where $\zeta$ is a complex parameter. Fix
$h\in H_{\bigidx}$. There are $Q_1,Q_2\in\C(q^{-\zeta},q^{-s})$ such
that
\begin{align*}
Q_1(q^{-\zeta},q^{-s})=\intertwiningfull{\tau}{s}\widehat{f}_{\varphi_s}(h,1),\qquad
Q_2(q^{-\zeta},q^{-s})=\omega_{\tau_1}(-1)^{\bigidx_2}\widehat{f}_{\intertwiningfull{\tau}{s}\varphi_s}(h,1)
\end{align*}
for all $\zeta$ and $s$ with $\Re(s)>>0$. Here the $dz$-integration
is defined by principal value. For $\Re(\zeta)>>0$ and $\Re(s)>>0$,
\begin{align*}
&\int_{U_{\bigidx}}\int_{Z_{\bigidx_2,\bigidx_1}}|\varphi_s|(\omega_{\bigidx_1,\bigidx_2}zd_{\bigidx}w_{\bigidx}uh,I_{\bigidx_1},I_{2\bigidx_2+1},I_{\bigidx_2})dzdu<\infty.
\end{align*}
Therefore the manipulations above hold for $\Re(\zeta),\Re(s)>>0$
and hence $Q_1=Q_2$. This yields \eqref{intertwining non normalized
commutes with an extra factor} when we substitute $0$ for $\zeta$.

Let $\whittakerfunctional$ denote the Whittaker functional on
$V(\Whittaker{\tau}{\psi},s)$ given by the \lhs\ of
\eqref{eq:Shahidi local coefficient def}. Then the functional
$\varphi_s\mapsto\whittakerfunctional(\widehat{f}_{\varphi_s})$ is
the Whittaker functional on $V'(\tau_1\otimes\tau_2,(s,s))$ given by
the \lhs\ of \eqref{int:func equation intertwining operators start}.
Therefore
\begin{align*}
\nintertwiningfull{\tau}{s}\varphi_s=\frac{\whittakerfunctional(\widehat{f}_{\varphi_s})}{\whittakerfunctional(\widehat{f}_{\intertwiningfull{\tau}{s}\varphi_s})}
\intertwiningfull{\tau}{s}\varphi_s.
\end{align*}
Altogether,
\begin{align*}
\nintertwiningfull{\tau}{s}\widehat{f}_{\varphi_s} =
\frac{\whittakerfunctional(\widehat{f}_{\varphi_s})}{\whittakerfunctional(\intertwiningfull{\tau}{s}\widehat{f}_{\varphi_s})}
\intertwiningfull{\tau}{s}\widehat{f}_{\varphi_s}=
\frac{\whittakerfunctional(\widehat{f}_{\varphi_s})}{\whittakerfunctional(\widehat{f}_{\intertwiningfull{\tau}{s}\varphi_s})}
\widehat{f}_{\intertwiningfull{\tau}{s}\varphi_s}=\widehat{f}_{\nintertwiningfull{\tau}{s}\varphi_s},
\end{align*}
where we applied \eqref{intertwining non normalized commutes with an
extra factor} twice and used the fact that the mapping
$\varphi_s\mapsto \widehat{f}_{\varphi_s}$ is linear.
\end{proof} 

\subsection{The factors $M_{\tau}(s)$ and $M_{\tau_1\otimes\ldots\otimes\tau_k}(s)$}\label{subsection:the factors M tau s}
We define the factors $M_{\tau}(s)$ and
$M_{\tau_1\otimes\ldots\otimes\tau_k}(s)$ appearing in our theorems
(see Section~\ref{section:main results}). Let $\tau_i$ be an
irreducible representation of $\GL{\bigidx_i}$ for $1\leq i\leq k$.
Then
\begin{align*}
&M_{\tau_1\otimes\ldots\otimes\tau_k}(s)=\prod_{i=1}^k\ell_{\tau_i}(s)\ell_{\tau_i^*}(1-s)\prod_{1\leq
i<j\leq
k}\ell_{\tau_i\otimes\tau_j^*}(s)\ell_{\tau_j^*\otimes\tau_i}(1-s).
\end{align*}
In particular for $k=1$,
$M_{\tau_1}(s)=\ell_{\tau_1}(s)\ell_{\tau_1^*}(1-s)$ and we can
write
\begin{align*}
&M_{\tau_1\otimes\ldots\otimes\tau_k}(s)=\prod_{i=1}^kM_{\tau_i}(s)\prod_{1\leq
i<j\leq
k}\ell_{\tau_i\otimes\tau_j^*}(s)\ell_{\tau_j^*\otimes\tau_i}(1-s).
\end{align*}
These factors are inverses of polynomials in $\C[q^{-s},q^s]$. Note
that by definition,
$M_{\tau_1\otimes\ldots\otimes\tau_k}(s)=M_{\tau_k^*\otimes\ldots\otimes\tau_1^*}(1-s)$.
Theorems~\ref{theorem:poles of normalized intertwining for tau
cuspidal} and \ref{theorem:poles of normalized intertwining for
tau_1 tau_2 cuspidal} enable us to calculate or bound
$M_{\tau_1\otimes\ldots\otimes\tau_k}(s)$. For example if
$\tau_1,\tau_2$ are irreducible supercuspidal,
$\ell_{\tau_1\otimes\tau_2^*}(s)=L(\tau_1^*\times\tau_2^*,2-2s)$ and
\begin{align*}
M_{\tau_1\otimes\tau_2}(s)=
(\prod_{i=1}^2L(\tau_i^*,Sym^2,2-2s)L(\tau_i,Sym^2,2s))L(\tau_1^*\times\tau_2^*,2-2s)L(\tau_2\times\tau_1,2s).
\end{align*}

In the course of proving Theorem~\ref{theorem:gcd sub multiplicity
second var} we will encounter poles of
$\nintertwiningfull{\tau_i}{s}$ and
$\nintertwiningfull{\tau_i^*}{1-s}$. The crucial property of
$M_{\tau_i}(s)$ is that both
$M_{\tau_i}(s)^{-1}\nintertwiningfull{\tau_i}{s}$ and
$M_{\tau_i}(s)^{-1}\nintertwiningfull{\tau_i^*}{1-s}$ are
holomorphic. In order to derive the theorem we will also need an
upper bound for $\ell_{\tau}(s)$, described next.
\begin{claim}\label{claim:upper bound poles intertwiners l tau}
Assume that
$\tau=\cinduced{P_{\bigidx_1,\bigidx_2}}{\GL{\bigidx}}{\tau_1\otimes\tau_2}$
is irreducible and set\\
$L(s)=\ell_{\tau_1}(s)\ell_{\tau_2}(s)\ell_{\tau_1\otimes\tau_2^*}(s)$.
Then $\ell_{\tau}(s)\in L(s)\C[q^{-s},q^{s}]$.
\end{claim}
\begin{proof}[Proof of Claim~\ref{claim:upper bound poles intertwiners l
tau}] 
Let
$f_s\in\xi_{Q_{\bigidx}}^{H_{\bigidx}}(\Whittaker{\tau}{\psi},hol,s)$.
In Claim~\ref{claim:f_s is the image of varphi_s when zeta is fixed}
(Section~\ref{subsection:fixed zeta}) we will prove that there is
$\varphi_s\in\xi_{Q_{\bigidx_1,\bigidx_2}}^{H_{\bigidx}}(\tau_1\otimes\tau_2,hol,(s,s))$
such that $f_s=\widehat{f}_{\varphi_s}$, where
$\widehat{f}_{\varphi_s}$ is defined by a Jacquet integral as
explained Section~\ref{subsection:the multiplicativity of the
intertwining operator for tau induced}. Then by
Claim~\ref{claim:intertwining operator and Jacquet integral
commute},
\begin{align*}
L(s)^{-1}\nintertwiningfull{\tau}{s}f_s=L(s)^{-1}\nintertwiningfull{\tau}{s}\widehat{f}_{\varphi_s}=\widehat{f}_{L(s)^{-1}\nintertwiningfull{\tau}{s}\varphi_s}.
\end{align*}
By \eqref{eq:multiplicative property of intertwiners},
$L(s)^{-1}\nintertwiningfull{\tau}{s}\varphi_s\in\xi_{Q_{\bigidx_2,\bigidx_1}}^{H_{\bigidx}}(\tau_2^*\otimes\tau_1^*,hol,(1-s,1-s))$.
Claim~\ref{claim:f_s is the image of varphi_s when zeta is fixed}
also proves that for
$\varphi_s'\in\xi_{Q_{\bigidx_1,\bigidx_2}}^{H_{\bigidx}}(\tau_1\otimes\tau_2,hol,(s,s))$,
$\widehat{f}_{\varphi_s'}\in\xi_{Q_{\bigidx}}^{H_{\bigidx}}(\Whittaker{\tau}{\psi},hol,s)$.
Hence
$\widehat{f}_{L(s)^{-1}\nintertwiningfull{\tau}{s}\varphi_s}\in\xi_{Q_{\bigidx}}^{H_{\bigidx}}(\Whittaker{\tau^*}{\psi},hol,1-s)$
and thus $\ell_{\tau}(s)^{-1}$ divides $L(s)^{-1}$.
\end{proof} 
More generally, assume that
$\tau=\cinduced{P_{\bigidx_1,\ldots,\bigidx_k}}{\GL{\bigidx}}{\tau_1\otimes\ldots\otimes\tau_k}$
is irreducible. Then Claim~\ref{claim:upper bound poles intertwiners
l tau} implies 
\begin{align}\label{eq:multiplicative property of poles of intertwiners}
\ell_{\tau}(s)\in\prod_{i=1}^k\ell_{\tau_i}(s)\prod_{1\leq i<j\leq
k}\ell_{\tau_i\otimes\tau_j^*}(s)\C[q^{-s},q^s].
\end{align}
Hence $\ell_{\tau}(s)\in
M_{\tau_1\otimes\ldots\otimes\tau_k}(s)\C[q^{-s},q^s]$ and also
\begin{align*}
M_{\tau}(s)=\ell_{\tau}(s)\ell_{\tau^*}(1-s)\in
M_{\tau_1\otimes\ldots\otimes\tau_k}(s)\C[q^{-s},q^s].
\end{align*}
This means
that $M_{\tau_1\otimes\ldots\otimes\tau_k}(s)$ is an upper bound for
the poles of $M_{\tau}(s)$.

\section{Functional equation for $\GL{k}\times\GL{r}$}\label{subsection:func eq gln glm}
For $k\leq r$, let $\xi$ be a representation of $\GL{k}$ and $\phi$
be a representation of $\GL{r}$ (both representations are generic). We briefly recall the functional
equation of Jacquet, Piatetski-Shapiro and Shalika \cite{JPSS} for
$\GL{k}\times\GL{r}$ and $\xi\times\phi$.

According to Theorem~2.7 of \cite{JPSS} if $k<r$, for all $0\leq
j\leq r-k-1$, $W_{\xi}\in\Whittaker{\xi}{\psi^{-1}}$ and
$W_{\phi}\in\Whittaker{\phi}{\psi}$,
\begin{align}\label{eq:func eq glm gln}
&L(\phi\times\xi,s)^{-1}\epsilon(\phi\times\xi,\psi,s)\omega_{\xi}(-1)^{r-1}\Psi(W_{\xi},W_{\phi},j,s)\\\notag
&=L(\phi^*\times\xi^*,1-s)^{-1}\Psi(\widetilde{W_{\xi}},\left(\begin{array}{cc}I_k\\&J_{r-k}\\\end{array}\right) \cdot\widetilde{W_{\phi}},r-k-1-j,1-s).
\end{align}
Here
\begin{align*}
\Psi(W_{\xi},W_{\phi},j,s)&=\int_{\lmodulo{Z_{k}}{\GL{k}}}\int_{\Mat{j\times
k}} W_{\xi}(a)
W_{\phi}(\left(\begin{array}{ccc}a&0&0\\m&I_{j}&0\\0&0&I_{r-k-j}\end{array}\right))
\absdet{a}^{s-\frac{r-k}2}dmda
\end{align*}
and $\widetilde{W_{\xi}}(g)=W_{\xi}(J_k\transpose{g^{-1}})$. When
$k=r$, for all $\Phi\in\mathcal{S}(F^{r})$,
\begin{align}\label{eq:func eq glm gln n=m}
&L(\phi\times\xi,s)^{-1}\epsilon(\phi\times\xi,\psi,s)\omega_{\xi}(-1)^{r-1}\Psi(W_{\xi},W_{\phi},\Phi,s)\\\notag
&=L(\phi^*\times\xi^*,1-s)^{-1}\Psi(\widetilde{W_{\xi}},\widetilde{W_{\phi}},\widehat{\Phi},1-s),
\end{align}
where
\begin{align*}
\Psi(W_{\xi},W_{\phi},\Phi,s)=\int_{\lmodulo{Z_{r}}{\GL{r}}}W_{\xi}(a)W_{\phi}(a)\Phi(\eta_r
a)\absdet{a}^sda,\qquad \eta_r=(0,\ldots,0,1)
\end{align*}
and $\widehat{\Phi}$ denotes a Fourier transform of $\Phi$.

In the functional equations of \cite{JPSS}, instead of
$\phi^*$ appears $\phi^{\iota}$ (and $\xi^{\iota}$) where
$\phi^{\iota}(g)=\phi(\transpose{g^{-1}})$, but
$\Whittaker{\phi^{\iota}}{\psi}=\Whittaker{\phi^*}{\psi}$ (if
$W\in\Whittaker{\phi^{\iota}}{\psi}$, $J_{\bigidx}\cdot
W\in\Whittaker{\phi^*}{\psi}$). 

The integrals $\Psi(W_{\xi},W_{\phi},j,s)$ and $\Psi(W_{\xi},W_{\phi},\Phi,s)$ are absolutely convergent for $\Re(s)>>0$. That is, the integrals converge when we replace $W_{\xi}$, $W_{\phi}$ and $\Phi$ with $|W_{\xi}|$, $|W_{\phi}|$ and $|\Phi|$. Moreover, there is a constant $s_0>0$ which depends only on the representations $\sigma$ and $\tau$, such that the integrals are absolutely convergent for all $\Re(s)>s_0$.

Equality \eqref{eq:func eq glm gln} can be rewritten in the
following form:
\begin{align}\label{eq:modified func eq glm gln}
&\gamma(\phi\times\xi,\psi,s)\omega_{\xi}(-1)^{r-1}
\int_{\lmodulo{Z_{k}}{\GL{k}}}\int_{\Mat{j\times k}} W_{\xi}(a)
W_{\phi}(\left(\begin{array}{ccc}a&0&0\\m&I_{j}&0\\0&0&I_{r-k-j}\end{array}\right))
\absdet{a}^{s-\frac{r-k}2}dmda\\\notag
&=\int_{\lmodulo{Z_{k}}{\GL{k}}}\int_{\Mat{k\times r-k-j-1}}
W_{\xi}(a)
W_{\phi}(\left(\begin{array}{ccc}0&I_{j+1}&0\\0&0&I_{r-k-j-1}\\a&0&m\\\end{array}\right))
\absdet{a}^{s-\frac{r-k}2+j}dmda.
\end{align}
This is obtained using
$\widetilde{W_{\phi}}(g)=W_{\phi}(J_r\transpose{g^{-1}})$ (see \cite{Soudry} p. 70).

The family of integrals $\Psi(W_{\xi},W_{\phi},j,s)$ with
$W_{\xi},W_{\phi}$ varying, when $k<r$, or
$\Psi(W_{\xi},W_{\phi},\Phi,s)$ with $W_{\xi},W_{\phi},\Phi$ varying
if $k=r$, span a fractional ideal of $\C[q^{-s},q^s]$. The
$L$-factor $L(\phi\times\xi,s)$ was defined in \cite{JPSS} as the
unique generator of this ideal, in the form $P(q^{-s})^{-1}$ for
$P\in\C[X]$ satisfying $P(0)=1$. In other words,
$L(\phi\times\xi,s)$ was defined as a g.c.d. Note that in the case
$k<r$, $L(\phi\times\xi,s)$ is independent of $j$.

The $L$-factor was proved to be independent of the additive
(\nontrivial) character $\psi$ of the field. Moreover, if $a\in
A_{k}$ and $b\in A_{r-k}$, one can replace the character $\psi$ of
$Z_k$ with $\rconj{a}\psi$, where $\rconj{a}\psi(z)=\psi(aza^{-1})$,
and the character $\psi$ of $Z_r$ with $\rconj{diag(a,b)}\psi$. The
$L$-factor is invariant under such a change.

We will use the following property of the $L$-factor.
\begin{theorem}[\cite{JPSS} Proposition~8.1(i)]\label{theorem:tau cuspidal n large L factor of glm gln 1}
Let $\phi$ be an irreducible supercuspidal representation. If $r>k$ and $\xi$ is arbitrary,
$L(\phi\times\xi,s)=1$.
\end{theorem}
The $\gamma$-factor $\gamma(\phi\times\xi,\psi,s)$ was defined in \cite{JPSS} by
\begin{align}\label{eq:JPSS relation gamma and friends}
\gamma(\phi\times\xi,\psi,s)=\epsilon(\phi\times\xi,\psi,s)\frac{L(\phi^*\times\xi^*,1-s)}
{L(\phi\times\xi,s)}.
\end{align}
The $\epsilon$-factor $\epsilon(\phi\times\xi,\psi,s)$ belongs to $\C[q^{-s},q^s]^*$.

\section{Integration formula for $\lmodulo{H_{\bigidx}}{G_{\bigidx+1}}$}\label{subsection:integration formula}
We will need an integration formula for an integral on $\lmodulo{H_{\bigidx}}{G_{\bigidx+1}}$.
\begin{lemma}\label{lemma:integration formula for quotient space H_n G_n+1}
Let $f$ be a continuous complex-valued function on $\lmodulo{H_{\bigidx}}{G_{\bigidx+1}}$ such that
$\int_{\lmodulo{H_{\bigidx}}{G_{\bigidx+1}}}|f|(g)dg<\infty$. Then if $G_{\smallidx}$ is split,
\begin{align*}
\int_{\lmodulo{H_{\bigidx}}{G_{\bigidx+1}}}f(g)dg=\int_{\GL{1}}\int_{\overline{Z_{\bigidx,1}}}\int_{\Xi_{\bigidx}}f(vzx)dvdzdx.
\end{align*}
Here $\overline{Z_{\bigidx,1}}<L_{\bigidx+1}$, $x\in\GL{1}$ takes
the form $diag(I_{\bigidx},x,x^{-1},I_{\bigidx})$ and
\begin{align*}
\Xi_{\bigidx}=\setof{\left(\begin{array}{cccc}I_{\bigidx}&0&v&0\\&1&0&v'\\&&1&0\\&&&I_{\bigidx}\end{array}\right)}{v\in
F^{\bigidx}}<V_{\bigidx}.
\end{align*}
In the \quasisplit\ case,
\begin{align*}
\int_{\lmodulo{H_{\bigidx}}{G_{\bigidx+1}}}f(g)dg=\int_{\GL{1}}\int_{\overline{Z_{\bigidx-1,1}}}\int_{V_{\bigidx}\cap
G_2}\int_{\Xi_{\bigidx}}f(yv'v''zx)dv'dv''dzdx,
\end{align*}
where $\overline{Z_{\bigidx-1,1}}<L_{\bigidx}$,
$x=diag(I_{\bigidx-1},x,I_2,x^{-1},I_{\bigidx-1})$, the group $G_2$
in the intersection $V_{\bigidx}\cap G_2$ is considered as a
subgroup of $G_{\bigidx+1}$ (embedded in $L_{\bigidx-1}$),
\begin{align*}
y=\left(\begin{array}{cccccc}I_{\bigidx-1}\\&1\\&0&1\\&2\rho^{-2}&0&1\\&2\rho^{-3}&0&2\rho^{-1}&1\\&&&&&I_{\bigidx-1}\end{array}\right)\in
G_{\bigidx+1}
\end{align*} and $\Xi_{\bigidx}$ is a set of representatives from
$G_{\bigidx+1}$,
\begin{align*}
\Xi_{\bigidx}=\setof{\left(\begin{array}{cccccc}I_{\bigidx-1}&0&0&v_2&0&A(v_2)\\&1&&&&0\\&&1&&&0\\&&&1&&v_2'\\&&&&1&0\\&&&&&I_{\bigidx-1}\end{array}\right)}{v_2\in
F^{\bigidx-1}},
\end{align*}
whose exact definition is given in the proof.
\end{lemma}
\begin{proof}[Proof of Lemma~\ref{lemma:integration formula for quotient space H_n G_n+1}] 
Assume that $G_{\smallidx}$ is split. Then the double coset space
$\rmodulo{\lmodulo{H_{\bigidx}}{G_{\bigidx+1}}}{P_{\bigidx+1}}$ has
a single element (\cite{RGS} Proposition~4.4 case (3)) and we take
the identity $e\in G_{\bigidx+1}$ as its representative. Denote
$H_{\bigidx}^e=H_{\bigidx}\cap P_{\bigidx+1}$. We have,
\begin{align*}
H_{\bigidx}^e=\setof{\left(\begin{array}{cccc}x&m&\frac1{2\beta^2}m&u\\&1&0&\frac1{2\beta^2}m'\\&&1&m'\\&&&x^*\end{array}\right)}{x\in\GL{\bigidx}}.
\end{align*}
There is a right-invariant Haar measure $d_rp$ on
$\lmodulo{H_{\bigidx}^e}{P_{\bigidx+1}}$. Hence
\begin{align*}
\int_{\lmodulo{H_{\bigidx}}{G_{\bigidx+1}}}f(g)dg=\int_{\lmodulo{H_{\bigidx}^e}{P_{\bigidx+1}}}f(p)d_rp.
\end{align*}
Let $\GL{\bigidx+1}$ be embedded in $P_{\bigidx+1}$ by $c\mapsto
diag(c,c^*)$ ($c\in\GL{\bigidx+1}$) and let $\GL{1}$ be embedded in
$\GL{\bigidx+1}$ by $t\mapsto diag(I_{\bigidx},t)$ ($t\in F^*$).
Since the complement of $P_{\bigidx,1}\overline{Z_{\bigidx,1}}$ in
$\GL{\bigidx+1}$ is a set of zero measure, the complement of
$H_{\bigidx}^e\GL{1}\overline{Z_{\bigidx,1}}V_{\bigidx+1}=P_{\bigidx,1}\overline{Z_{\bigidx,1}}V_{\bigidx+1}$
in $P_{\bigidx+1}=\GL{\bigidx+1}V_{\bigidx+1}$ is a set of zero
measure (and
$H_{\bigidx}^e\GL{1}\overline{Z_{\bigidx,1}}V_{\bigidx+1}$ is open
and dense in $P_{\bigidx+1}$). It follows that
\begin{align*}
\int_{\lmodulo{H_{\bigidx}^e}{P_{\bigidx+1}}}f(p)d_rp=\int_{\lmodulo{H_{\bigidx}^e}{H_{\bigidx}^e}
\GL{1}\overline{Z_{\bigidx,1}}V_{\bigidx+1}}f(p)d_rp.
\end{align*}
Let
$R_{\bigidx}=\GL{1}\overline{Z_{\bigidx,1}}V_{\bigidx+1}\cap
H_{\bigidx}^e$,
\begin{align*}
&R_{\bigidx}=\setof{\left(\begin{array}{cccc}I_{\bigidx}&0&0&u\\&1&0&0\\&&1&0\\&&&I_{\bigidx}\end{array}\right)}{\transpose{u}J_{\bigidx}+J_{\bigidx}u=0}<V_{\bigidx+1}.
\end{align*}
Since $\GL{1}\overline{Z_{\bigidx,1}}V_{\bigidx+1}$ is a subgroup, we obtain
\begin{align*}
\int_{\lmodulo{H_{\bigidx}^e}{H_{\bigidx}^e\GL{1}\overline{Z_{\bigidx,1}}V_{\bigidx+1}}}f(p)d_rp&=
\int_{\lmodulo{R_{\bigidx}}{\GL{1}\overline{Z_{\bigidx,1}}V_{\bigidx+1}}}f(p)d_rp\\
&=\int_{\GL{1}}\int_{\overline{Z_{\bigidx,1}}}\int_{\lmodulo{R_{\bigidx}}{V_{\bigidx+1}}}f(vzx)dvdzdx.
\end{align*}
Then as groups $\lmodulo{R_{\bigidx}}{V_{\bigidx+1}}\isomorphic
\Xi_{\bigidx}$, whence we can replace the $dv$-integration with an
integration over $\Xi_{\bigidx}$ leading to the requested formula.

Now consider the \quasisplit\ case. In general, if $G$ is one of the
groups defined in Section~\ref{subsection:groups in study}, $H$ and
$P$ are closed subgroups of $G$ and $g\in G$, we have isomorphisms
of algebraic varieties
$\lmodulo{H}{HgP}\isomorphic\lmodulo{\rconj{g}H}{\rconj{g}HP}\isomorphic\lmodulo{(\rconj{g}H\cap
P)}{P}$. Then $dim(HgP)=dim(H)+dim(P)-dim (\rconj{g}H\cap P)$, where
$dim(\cdots)$ refers to the dimension as an algebraic variety over
the local field (see \cite{Spr} Corollary~5.5.6). Let
$G=G_{\bigidx+1}$ be \quasisplit, $H=H_{\bigidx}$ and
$P=P_{\bigidx}$. The double coset space
$\rmodulo{\lmodulo{H_{\bigidx}}{G_{\bigidx+1}}}{P_{\bigidx}}$ has
two elements (\cite{RGS} Proposition~4.4 case (2b)), we take
representatives $\{e,y\}$ where $e$ is the identity and $y$ was
given in the statement of the lemma.
Then $G_{\bigidx+1}=H_{\bigidx}P_{\bigidx}\dotcup
H_{\bigidx}yP_{\bigidx}$. 
A calculation shows that
\begin{align*}
H_{\bigidx}\cap P_{\bigidx}=
\setof{\left(\begin{array}{cccc}x&u&0&v\\&1&0&u'\\&&1&0\\&&&x^*\end{array}\right)}{x\in\GL{\bigidx}}\isomorphic
Q_{\bigidx}.
\end{align*}
Hence $dim (H_{\bigidx}\cap
P_{\bigidx})=\frac32\bigidx^2+\half\bigidx$. Also
\begin{align*}
&\rconj{y}H_{\bigidx}\cap P_{\bigidx}\\\notag&=
\setof{\left(\begin{array}{cccccc}x&u_1&u_2&-2\rho^{-1}
u_3&u_3&u_4\\0&1&\half b\rho^2&\half(a-1)\rho^2&\quarter(1-a)\rho^3&u_3'\\0&0&a&b\rho&-\half b\rho^2&u_2'\\0&0&b&a&\half(1-a)\rho&2\rho^{-2}u_3'\\0&0&0&0&1&u_1'
\\0&0&0&0&0&x^*\end{array}\right)}{x\in\GL{\bigidx-1},a^2-b^2\rho=1}.
\end{align*}
Then $dim (\rconj{y}H_{\bigidx}\cap
P_{\bigidx})=\frac32\bigidx^2-\half\bigidx$. It follows that for all
$\bigidx\geq1$, the complement of $H_{\bigidx}yP_{\bigidx}$ in
$G_{\bigidx+1}$ is of zero measure (and $H_{\bigidx}yP_{\bigidx}$ is
open and dense in $G_{\bigidx+1}$). Put
$H_{\bigidx}^y=\rconj{y}H_{\bigidx}\cap P_{\bigidx}$. There is a
right-invariant Haar measure $d_rp$ on
$\lmodulo{H_{\bigidx}^y}{P_{\bigidx}}$. Hence we have the formula
\begin{align*}
\int_{\lmodulo{H_{\bigidx}}{G_{\bigidx+1}}}f(g)dg
=\int_{\lmodulo{H_{\bigidx}^y}{P_{\bigidx}}}f(yp)d_rp.
\end{align*}
Denote
\begin{align*}
&G_1'=\setof{\left(\begin{array}{cccccc}I_{\bigidx-1}&\\&1&\half b\rho^2&\half(a-1)\rho^2&\quarter(1-a)\rho^3&\\&&a&b\rho&-\half b\rho^2&\\&&b&a&\half(1-a)\rho&\\
&&&&1&\\&&&&&I_{\bigidx-1}\end{array}\right)}{a^2-b^2\rho=1}\isomorphic G_1,\\
&R_{\bigidx}=\{\left(\begin{array}{cccccc}I_{\bigidx-1}&0&u_2&-2\rho^{-1}
u_3&u_3&u_4\\&1&&&&u_3'\\&&1&&&u_2'\\&&&1&&2\rho^{-2}u_3'\\&&&&1&0
\\&&&&&I_{\bigidx-1}\end{array}\right)\}<V_{\bigidx}.
\end{align*}
Let $\GL{\bigidx}$ be embedded in $P_{\bigidx}$ by $c\mapsto diag(c,I_2,c^*)$ ($c\in\GL{\bigidx}$). Then $H_{\bigidx}^y=(Y_{\bigidx}\ltimes R_{\bigidx})\rtimes G_1'$ where $Y_{\bigidx}<\GL{\bigidx}$ is the mirabolic subgroup. Since $P_{\bigidx}=(\GL{\bigidx}\ltimes V_{\bigidx})\rtimes G_1'$ and there is a right-invariant measure on $\lmodulo{(Y_{\bigidx}R_{\bigidx})}{(\GL{\bigidx}V_{\bigidx})}$,
\begin{align*}
\int_{\lmodulo{H_{\bigidx}^y}{P_{\bigidx}}}f(yp)d_rp=
\int_{\lmodulo{Y_{\bigidx}R_{\bigidx}}{\GL{\bigidx}V_{\bigidx}}}f(yp)d_rp.
\end{align*}
Let $\GL{1}$ be embedded in $\GL{\bigidx}$ by $t\mapsto
diag(I_{\bigidx-1},t)$. 
Since the complement of
$Y_{\bigidx}\GL{1}\overline{Z_{\bigidx-1,1}}=P_{\bigidx-1,1}\overline{Z_{\bigidx-1,1}}$
in $\GL{\bigidx}$ is of zero measure,
\begin{align*}
\int_{\lmodulo{Y_{\bigidx}R_{\bigidx}}{\GL{\bigidx}V_{\bigidx}}}f(yp)d_rp=
\int_{\lmodulo{Y_{\bigidx}R_{\bigidx}}{Y_{\bigidx}\GL{1}\overline{Z_{\bigidx-1,1}}V_{\bigidx}}}f(yp)d_rp.
\end{align*}
Note that $R_{\bigidx}<V_{\bigidx}<\GL{1}\overline{Z_{\bigidx-1,1}}V_{\bigidx}$ hence $Y_{\bigidx}\GL{1}\overline{Z_{\bigidx-1,1}}V_{\bigidx}=Y_{\bigidx}R_{\bigidx}\GL{1}\overline{Z_{\bigidx-1,1}}V_{\bigidx}$.  Also $Y_{\bigidx}R_{\bigidx}\cap \GL{1}\overline{Z_{\bigidx-1,1}}V_{\bigidx}=R_{\bigidx}$. Thus we obtain
\begin{align*}
\int_{\lmodulo{Y_{\bigidx}R_{\bigidx}}{Y_{\bigidx}\GL{1}\overline{Z_{\bigidx-1,1}}V_{\bigidx}}}f(yp)d_rp
=\int_{\lmodulo{R_{\bigidx}}{\GL{1}\overline{Z_{\bigidx-1,1}}V_{\bigidx}}}f(yp)d_rp.
\end{align*}
Let
\begin{align*}
V_{\bigidx}^{\star}=\{\left(\begin{array}{cccccc}I_{\bigidx-1}&0&v_1&v_2&v_3&v_4\\&1&&&&v_3'\\&&1&&&v_1'\\&&&1&&v_2'\\&&&&1&0
\\&&&&&I_{\bigidx-1}\end{array}\right)\}<V_{\bigidx}.
\end{align*}
Then $V_{\bigidx}=V_{\bigidx}^{\star}\rtimes (V_{\bigidx}\cap G_2)$ where
\begin{align*}
V_{\bigidx}\cap G_2=\{\left(\begin{array}{cccccc}I_{\bigidx-1}&&&&&\\&1&z_1&z_2&*&\\&&1&0&z_1'\\&&&1&z_2'\\&&&&1&
\\&&&&&I_{\bigidx-1}\end{array}\right)\}.
\end{align*}
Since also $R_{\bigidx}<V_{\bigidx}^{\star}$,
\begin{align*}
\int_{\lmodulo{R_{\bigidx}}{\GL{1}\overline{Z_{\bigidx-1,1}}V_{\bigidx}}}f(yp)d_rp=
\int_{\GL{1}}\int_{\overline{Z_{\bigidx-1,1}}}\int_{V_{\bigidx}\cap G_2}\int_{\lmodulo{R_{\bigidx}}{V_{\bigidx}^{\star}}}f(yv'v''zx)dv'dv''dzdx.
\end{align*}
Finally, $R_{\bigidx}$ is a normal subgroup of
$V_{\bigidx}^{\star}$,
$\lmodulo{R_{\bigidx}}{V_{\bigidx}^{\star}}\isomorphic
F^{\bigidx-1}$ (as groups) and we can choose a set of
representatives
\begin{align*}
\Xi_{\bigidx}=\setof{\theta(v_2)=\left(\begin{array}{cccccc}I_{\bigidx-1}&0&0&v_2&0&A(v_2)\\&1&&&&0\\&&1&&&0\\&&&1&&v_2'\\&&&&1&0\\&&&&&I_{\bigidx-1}\end{array}\right)}{v_2\in
F^{\bigidx-1}}\subset G_{\bigidx+1}.
\end{align*}
Here $A(v_2)=\frac1{2\rho}v_2(\transpose{v_2})J_{\bigidx-1}$ and
$v_2'=\frac1{\rho}(\transpose{v_2})J_{\bigidx-1}$. We define a group
structure on $\Xi_{\bigidx}$ by
$\theta(u_2)+\theta(v_2)=\theta(u_2+v_2)$. Then the mapping
$v_2\mapsto \theta(v_2)$ is a group isomorphism
$F^{\bigidx-1}\isomorphic \Xi_{\bigidx}$ and the quotient topology
on $\Xi_{\bigidx}$ is homeomorphic to the topology of
$F^{\bigidx-1}$. Hence we have a right-invariant Haar measure $dv'$
on $\Xi_{\bigidx}$ (defined using the measure on $F^{\bigidx-1}$),
and a continuous compactly supported function on
$\lmodulo{R_{\bigidx}}{V_{\bigidx}^{\star}}$ can be regarded as a
similar function on $F^{\bigidx-1}$. Moreover,
$R_{\bigidx}\theta(u_2)\cdot
R_{\bigidx}\theta(v_2)=R_{\bigidx}\theta(u_2+v_2)$. Therefore we can
write the $dv'$-integration on
$\lmodulo{R_{\bigidx}}{V_{\bigidx}^{\star}}$ using the measure on
$\Xi_{\bigidx}$ (for details, see Remark~\ref{remark:writing the
integration using representatives in general} below).
\end{proof} 

\begin{remark}\label{remark:writing the integration using representatives in general}
In general, if $G$ is a topological group and $H$ is a closed normal
subgroup, let $X\subset G$ be a set of representatives for
$\lmodulo{H}{G}$. For any function $f$ on $\lmodulo{H}{G}$, define a
function $\overline{f}$ on $X$ by $\overline{f}(x)=f(Hg)$ where $x\in Hg$. Assume
that there is a group structure on $X$ and we have a group
isomorphism $\theta:R\rightarrow X$ where $R$ is a topological group
with a right-invariant Haar measure $dr$. We define a similar
measure on $X$ by
\begin{align*}
\int_{X}\varphi(x)dx=\int_{R}\varphi(\theta(r))dr,
\end{align*}
where $\varphi$ is any (suitable) function on $X$. Assuming that the
quotient topology on $X$ is homeomorphic to the topology of $R$ via
$\theta$, a continuous compactly supported function $f$ on
$\lmodulo{H}{G}$ is mapped to a similar function on $R$ by $f\mapsto
\overline{f}\theta$. Further assume $H\theta(r_1)\cdot
H\theta(r_2)=H\theta(r_1+r_2)$ for any $r_1,r_2\in R$, where on the
\lhs\ the multiplication is in $\lmodulo{H}{G}$. Let $f$ be a
continuous compactly supported function on $\lmodulo{H}{G}$. Then
for any $g\in G$,
\begin{align*}
\int_{X}\overline{g\cdot f}(x)dx=\int_{R}f(H\theta(r)g)dr.
\end{align*}
We have $H\theta(r)g=H\theta(r)\cdot Hg$ and since $X$ is a set of
representatives, there is $x_g=\theta(r_g)\in X$ for some $r_g\in R$
such that $Hg=Hx_g$. Then $H\theta(r)\cdot Hg=H\theta(r)\cdot
H\theta(r_g)=H\theta(r+r_g)$. Hence the integral becomes
\begin{align*}
\int_{R}\overline{f}(\theta(r+r_g))dr=\int_{R}\overline{f}(\theta(r))dr=\int_X\overline{f}(x)dx.
\end{align*}
It follows that the linear functional $f\mapsto
\int_{X}\overline{f}(x)dx$ defines a right-invariant Haar measure on
$\lmodulo{H}{G}$.
\end{remark}

%% file: chapter_the_integrals.tex
\newtheorem{theorem}{Theorem}[section]
\newtheorem{proposition}{Proposition}[section]
\newtheorem{corollary}{Corollary}[section]
\newtheorem{lemma}{Lemma}[section]
\newtheorem{claim}{Claim}[section]
\theoremstyle{remark}
\newtheorem{remark}{Remark}[section]
\newtheorem{example}{Example}[section]
\theoremstyle{definition}
\newtheorem{definition}{Definition}[section]
\numberwithin{equation}{section}
\newcommand{\chapter}{\section} 
\input{thesis_notations}
\end{comment}

\chapter{The integrals}\label{chapter:the integrals}
We introduce the global Rankin-Selberg integral and relate it to
Langlands' (partial) $L$-function. The global integral appeared in \cite{GPS,G,Soudry4,RGS}. We bring the details of the decomposition into an Euler product of local integrals, for completeness and also to obtain the precise forms of the local integrals that we will work with. At almost all places, where all data are unramified, we prove that the local integrals are equal to quotients of local $L$-functions. The
local integrals, at arbitrary finite places, are the object of study of this thesis.

\section{Construction of the global
integral}\label{section:construction of the global integral}
In this section our notation is global. Let $F$ be a number field
with an Adele ring $\Adele$. Let $\rho\in F^*$ and define
$J_{2\smallidx,\rho}$ as in Section~\ref{subsection:groups in
study}. The group $G_{\smallidx}(\Adele)$ is the restricted direct
product of the groups $G_{\smallidx}(F_{\nu})$ where $\nu$ varies
over all places of $F$. Similarly, $H_{\bigidx}(\Adele)$ is the
restricted direct product of the groups $H_{\bigidx}(F_{\nu})$. If $\rho\in
F^2$, $G_{\smallidx}(F_{\nu})$ is defined using
$J_{2\smallidx,\rho_{\nu}}=J_{2\smallidx}$ and is split for all $\nu$. Otherwise
$G_{\smallidx}(F_{\nu})$ is defined using $J_{\nu}=diag(I_{\smallidx-1},
\bigl(\begin{smallmatrix}0&1\\-\rho_{\nu}&0\end{smallmatrix}\bigr),I_{\smallidx-1})\cdot
J_{2\smallidx}$ and it is \quasisplit\ at roughly half of the places and split at the
others (see e.g. \cite{Ser} p.~75). At the places where it is split,
$J_{\nu}$ represents the same bilinear form
as $J_{2\smallidx}$. 

The character $\psi$ is now
taken to be a \nontrivial\ unitary additive character of
$\lmodulo{F}{\Adele}$. Let $\psi_{\gamma}$ be the (generic)
character of $U_{G_{ \smallidx}}(\Adele)$ defined by
\begin{align}\label{character:psi gamma of U G}
\psi_{\gamma}(u)=
\begin{cases}\psi(\sum_{i=1}^{\smallidx-2}u_{i,i+1}+\quarter
u_{\smallidx-1,\smallidx}-\gamma u_{\smallidx-1,\smallidx+1})&\text{$G_{\smallidx}(F)$ is split},\\
\psi(\sum_{i=1}^{\smallidx-2}u_{i,i+1}+\half
u_{\smallidx-1,\smallidx+1})&\text{$G_{\smallidx}(F)$ is
\quasisplit}.\end{cases}
\end{align}

Let $\pi$ and $\tau$ be a pair of irreducible automorphic
cuspidal representations, $\pi$ of $G_{\smallidx}(\Adele)$ and $\tau$ of $\GLF{\bigidx}{\Adele}$. Assume that
$\pi$ is globally generic with respect to $\psi_{\gamma}^{-1}$. That is, $\pi$ has a global \nonzero\ Whittaker functional with respect to $\psi_{\gamma}^{-1}$, given by a Fourier coefficient.
Note that $\tau$ is necessarily globally generic (\cite{Sh} Section~5). 

The character $\psi_{\gamma}$ is uniquely determined according to the final form of the integral.
I.e., a different choice of $\psi_{\gamma}$ results in a slightly different integral.

To a smooth holomorphic section $f_s$ from the space
$\cinduced{Q_{\bigidx}(\Adele)}{H_{\bigidx}(\Adele)}{\tau\alpha^s}$, 
attach an Eisenstein series
\begin{align*}
E_{f_s}(h)=\sum_{y\in\lmodulo{Q_{\bigidx}(F)}{H_{\bigidx}(F)}}f_s(yh,1)\qquad(h\in
H_{\bigidx}(\Adele)).
\end{align*}
Note that $\tau$ is realized in its space of cusp forms on
$\GLF{\bigidx}{\Adele}$, so $f_s$ (for a fixed $s$) is a function in
two variables. The sum is absolutely convergent for $\Re(s)>>0$ and
has a meromorphic continuation to the whole plane (see e.g.
\cite{MW2} Chapter~4).

Let $\varphi$ be a cusp form belonging to the space of $\pi$. Fix
$s$ and let $f_s$ be as above. 
The form of the integral depends on the relative sizes of $\smallidx$ and $\bigidx$.
The construction follows the arguments of \cite{GPS,G,Soudry}.

\subsection{The case $\smallidx\leq\bigidx$}\label{subsection:The global integral for l<=n}
We integrate $\varphi$ against a Fourier coefficient of $E_{f_s}$. Consider the
subgroup $N_{\bigidx-\smallidx}=Z_{\bigidx-\smallidx}\ltimes
U_{\bigidx-\smallidx}$, where $Z_{\bigidx-\smallidx}<M_{\bigidx-\smallidx}$. In coordinates,
\begin{align*}
N_{\bigidx-\smallidx}=\setof{\left(\begin {array}{ccc}
  z & y_1 & y_2 \\
    & \idblk{2\smallidx+1} & y_1' \\
    &   & z^*
\end{array}\right)\in H_{\bigidx}}{z\in Z_{\bigidx-\smallidx}} \qquad(y_1'=-J_{2\smallidx+1}\transpose{y_1}J_{\bigidx-\smallidx}z^*).
\end{align*}
Define a character $\psi_{\gamma}$ of $N_{\bigidx-\smallidx}(\Adele)$ by
$\psi_{\gamma}(zy)=\psi(z)\psi(\transpose{e_{\bigidx-\smallidx}}ye_{\gamma})$ where $z\in Z_{\bigidx-\smallidx}(\Adele)$ and $y\in
U_{\bigidx-\smallidx}(\Adele)$
(with $e_{\bigidx-\smallidx},e_{\gamma}$ given in
Section~\ref{subsection:G_l in H_n}). Specifically,
\begin{align*}
\psi_{\gamma}(v)=\psi(\sum_{i=1}^{\bigidx-\smallidx-1}v_{i,i+1}+v_{\bigidx-\smallidx,\bigidx}+\gamma
v_{\bigidx-\smallidx,\bigidx+2})\qquad(v\in N_{\bigidx-\smallidx}(\Adele)).
\end{align*}
Note that $\psi_{\gamma}$ defines a character of
$\lmodulo{N_{\bigidx-\smallidx}(F)}{N_{\bigidx-\smallidx}(\Adele)}$ and if $\smallidx=\bigidx$, $N_{\bigidx-\smallidx}=\{1\}$ and $\psi_{\gamma}\equiv1$. According
to the embedding $G_{\smallidx}<H_{\bigidx}$ described in
Section~\ref{subsection:G_l in H_n}, $G_{\smallidx}$ normalizes
$N_{\bigidx-\smallidx}$ and stabilizes $\psi_{\gamma}$ (see also Remark~\ref{remark:why choose e gamma}). The Fourier
coefficient of $E_{f_s}$ with respect to the unipotent subgroup
$N_{\bigidx-\smallidx}(\Adele)$ and $\psi_{\gamma}$ is the function on $H_{\bigidx}(\Adele)$ given by
\begin{align*}
E_{f_s}^{\psi_{\gamma}}(h)=\int_{\lmodulo{N_{\bigidx-\smallidx}(F)}{N_{\bigidx-\smallidx}(\Adele)}}E_{f_s}(vh)\psi_{\gamma}(v)dv. 
\end{align*}
The global integral is
\begin{align}\label{integral:global l <= n}
I(\varphi,f_s,s)=\int\limits_{\lmodulo{G_{\smallidx}(F)}{G_{\smallidx}(\Adele)}}\varphi(g)E_{f_s}^{\psi_{\gamma}}(g)dg.
\end{align}
It is absolutely convergent in the whole plane except at the 
poles of the Eisenstein series. This follows from the
rapid decay of cusp forms and moderate growth of the Eisenstein
series. The integral $I(\varphi,f_s,s)$ defines a meromorphic
function of $s$. 

We may regard $I(\varphi,f_s,s)$
for any fixed $s$ which is not a pole of the Eisenstein series, as a bilinear form on
$\pi\times\cinduced{Q_{\bigidx}(\Adele)}{H_{\bigidx}(\Adele)}{\tau\alpha^s}$.
The definition implies
\begin{align}\label{eq:global integral in homspace l <= n}
I(g\cdot \varphi,(gv)\cdot f_s,s)=\psi_{\gamma}^{-1}(v)I(\varphi,
f_s,s)\qquad(\forall v\in N_{\bigidx-\smallidx}(\Adele), g\in
G_{\smallidx}(\Adele)).
\end{align}

For decomposable data, we will show that $I(\varphi,f_s,s)$
decomposes as an Euler product
\begin{align}\label{eq:the euler product}
I(\varphi,f_s,s)=\prod_{\nu}\Psi_{\nu}(W_{\varphi_{\nu}},f_{s,\nu},s).
\end{align}
The product is taken over all places $\nu$ of $F$, the local integrals $\Psi_{\nu}$
are given below 
and the local data $(W_{\varphi_{\nu}},f_{s,\nu})$ correspond to
$\varphi$ and $f_s$. This equality is in the sense of meromorphic
continuation. Proposition~\ref{propo:basic global identity l <= n}
proves \eqref{eq:the euler product} by showing that for
$\Re(s)>>0$, $I(\varphi,f_s,s)$ equals as an integral to an Eulerian
integral, which decomposes as the product on the \rhs. Then by
meromorphic continuation \eqref{eq:the euler product} holds for all
$s$.

The local integral $\Psi_{\nu}(W_{\varphi_{\nu}},f_{s,\nu},s)$ is
absolutely convergent in some right half-plane. As a bilinear form, it satisfies the local version of 
\eqref{eq:global integral in homspace l <= n}. The local space of bilinear forms satisfying \eqref{eq:global integral in homspace l <= n} 
enjoys a multiplicity-one property, which enables us to define the local factors. The details are given in Chapter~\ref{chapter:local factors}.
\begin{proposition}\label{propo:basic global identity l <= n}
Let $I(\varphi,f_s,s)$ be given by \eqref{integral:global l <= n}. For
$\Re(s)>>0$,
\begin{align*}
I(\varphi,f_s,s)=\int_{\lmodulo{U_{G_{\smallidx}}(\Adele)}{G_{\smallidx}(\Adele)}}
W_{\varphi}(g) \int_{R_{\smallidx,\bigidx}(\Adele)}
W_{f_s}^{\psi}(w_{\smallidx,\bigidx}rg,1)\psi_{\gamma}(r)drdg.
\end{align*}
Here $W_{\varphi}$ is a Fourier-Whittaker coefficient of $\varphi$,
\begin{align*}
W_{\varphi}(g)=\int_{\lmodulo{U_{G_{\smallidx}}(F)}{U_{G_{\smallidx}}(\Adele)}}\varphi(ug)\psi_{\gamma}(u)du,
\end{align*}
the function $u\mapsto\psi_{\gamma}(u)$ is the character defined by \eqref{character:psi gamma of U G}, 
\begin{align*}
&w_{\smallidx,\bigidx}=\left(\begin {array}{ccccc}
   &  \gamma I_{\smallidx} &  &   &  \\
    &   &  &   &   I_{\bigidx-\smallidx} \\
   &   & (-1)^{\bigidx-\smallidx} &   &  \\
 I_{\bigidx-\smallidx}  &   &  &  &  \\
  &   &  &  \gamma ^{-1} I_{\smallidx} &
\end{array}\right)\in H_{\bigidx},\\
&R_{\smallidx,\bigidx}=\{\left(\begin {array}{ccccc}
  I_{\bigidx-\smallidx} & x  & y & 0  & z \\
   & I_{\smallidx}  & 0 & 0  &  0 \\
   &   & 1 &  0 & y' \\
   &   &  &  I_{\smallidx} & x' \\
   &   &  &   & I_{\bigidx-\smallidx}
\end{array}\right)\}<H_{\bigidx},
\end{align*}
$\psi_{\gamma}(r)$ is the restriction of the character $\psi_{\gamma}$ of $N_{\bigidx-\smallidx}(\Adele)$ to
$R_{\smallidx,\bigidx}(\Adele)$
and for all $h\in H_{\bigidx}(\Adele)$,
\begin{align*}
W_{f_s}^{\psi}(h,1)=\int_{\lmodulo{Z_{\bigidx}(F)}{Z_{\bigidx}(\Adele)}}f_s(h,z)\psi^{-1}(z)dz.
\end{align*}
\end{proposition}

Before we turn to the proof of the proposition, let us show how it
is used to establish \eqref{eq:the euler product}. For $\varphi$ corresponding to a pure tensor,
$W_{\varphi}=\prod_{\nu}W_{\nu}$ where for all $\nu$,
$W_{\nu}\in\Whittaker{\pi_{\nu}}{\psi_{\nu}^{-1}}$ and for almost all $\nu$, $W_{\nu}$ equals the
normalized unramified Whittaker function.
Similarly for a decomposable vector $f_s$,
$W_{f_s}^{\psi}(h,1)=\prod_{\nu}f_{s,\nu}(h_{\nu},1)$ where
$f_{s,\nu}\in\cinduced{Q_{\bigidx}(F_{\nu})}{H_{\bigidx}(F_{\nu})}{\Whittaker{\tau_{\nu}}{\psi_{\nu}}
\absdet{}_{\nu}^{s-\half}}$ and for almost all $\nu$ this is the
unramified function, normalized so that $f_{s,\nu}(1,1)=1$. Note that $f_{s,\nu}$ is realized as a function
on $H_{\bigidx}(F_{\nu})\times\GLF{\bigidx}{F_{\nu}}$ and 
the function $m_{\nu}\mapsto
\delta_{Q_{\bigidx}(F_{\nu})}^{-1/2}(m_{\nu})\absdet{m_{\nu}}_{\nu}^{-s+\half}f_{s,\nu}(m_{\nu}h_{\nu},1)$ on $\GLF{\bigidx}{F_{\nu}}$
belongs to  $\Whittaker{\tau_{\nu}}{\psi_{\nu}}$.
By Proposition~\ref{propo:basic global
identity l <= n} with a cusp form $\varphi$ corresponding to a pure tensor and a decomposable $f_s$,
\begin{align*}
I(\varphi,f_s,s)=&\prod_{\nu}\Psi_{\nu}(W_{\nu},f_{s,\nu},s)\\
=&\prod_{\nu}\int_{\lmodulo{U_{G_{\smallidx}}(F_{\nu})}{G_{\smallidx}(F_{\nu})}}W_{\nu}(g)\int_{R_{\smallidx,\bigidx}(F_{\nu})}
f_{s,\nu}
((w_{\smallidx,\bigidx})_{\nu}rg,1)\psi_{\gamma,\nu}(r)drdg.
\end{align*}

\begin{proof}[Proof of Proposition~\ref{propo:basic global identity l <= n}] 
The first task is to unfold the Eisenstein series $E_{f_s}$ for
$\Re(s)>>0$ and show that the $\psi_{\gamma}$-coefficient is a sum
of either one (in the \quasisplit\ case) or three (in the split
case) summands. Throughout, we omit the reference to the field $F$ in
summations. Observe that
\begin{align*}
E_{f_s}(h)=\sum_{y\in
\lmodulo{Q_{\bigidx}}{H_{\bigidx}}}f_s(yh,1)=\sum_{w\in\rmodulo{\lmodulo{Q_{\bigidx}}{H_{\bigidx}}}{Q_{\bigidx-\smallidx}}}
\quad\sum_{q\in\lmodulo{Q_{\bigidx}^w}{Q_{\bigidx-\smallidx}}}f_s(wqh,1),
\end{align*}
where $Q_{\bigidx}^w=\rconj{w}Q_{\bigidx}\cap
Q_{\bigidx-\smallidx}$. The following is a set of representatives
for
$\rmodulo{\lmodulo{Q_{\bigidx}}{H_{\bigidx}}}{Q_{\bigidx-\smallidx}}$:
\begin{align*}
\setof{w_r=\left(\begin {array}{ccccccc}
  I_r &  &  &  & & & \\
   & &  &  &  & I_{\bigidx-\smallidx-r} & \\
   &  &  I_{\smallidx} &  &  & & \\
   &  &  & (-1)^{\bigidx-\smallidx-r} &  &  &\\
   &  &   &  & I_{\smallidx} & & \\
   & I_{\bigidx-\smallidx-r} &  &  & & & \\
    & & &  &  &  & I_r
\end{array}\right)}{r=\intrange{0}{\bigidx-\smallidx}}
\end{align*}
(see \cite{RGS} Chapter~4). We may change the order of the $dv$-integration and summation over $w$ in $E_{f_s}^{\psi_{\gamma}}(h)$, since $N_{\bigidx-\smallidx}(F)<Q_{\bigidx-\smallidx}(F)$ and
because of the nice convergence properties of the Eisenstein series. This yields
\begin{align*}
E_{f_s}^{\psi_{\gamma}}(h)=\sum_{r=0}^{\bigidx-\smallidx}
\int_{\lmodulo{N_{\bigidx-\smallidx}(F)}{N_{\bigidx-\smallidx}(\Adele)}}
\sum_{q\in\lmodulo{Q_{\bigidx}^{w_r}}{Q_{\bigidx-\smallidx}}}
f_s(w_rqvh,1)\psi_{\gamma}(v)dv.
\end{align*}
We will show that except for $r=0$, all other summands on the \rhs\
vanish. Considering the right-action of
$N_{\bigidx-\smallidx}G_{\smallidx}$ on
$\lmodulo{Q_{\bigidx}^{w_r}}{Q_{\bigidx-\smallidx}}$,
\begin{align}\label{eq:E_f psi second unfolding}
E_{f_s}^{\psi_{\gamma}}(h)=&\sum_{r=0}^{\bigidx-\smallidx}
\int_{\lmodulo{N_{\bigidx-\smallidx}(F)}{N_{\bigidx-\smallidx}(\Adele)}}
\quad\sum_{\eta\in\rmodulo{\lmodulo{Q_{\bigidx}^{w_r}}{Q_{\bigidx-\smallidx}}}{N_{\bigidx-\smallidx}G_{\smallidx}}}\\\nonumber&\times
\sum_{b\in\lmodulo{(Q_{\bigidx}^{w_r})^{\eta}}{N_{\bigidx-\smallidx}G_{\smallidx}}}f_s(w_r\eta
bvh,1)\psi_{\gamma}(v)dv.
\end{align}
Here $(Q_{\bigidx}^{w_r})^{\eta}=\rconj{\eta}(Q_{\bigidx}^{w_r})\cap
N_{\bigidx-\smallidx}G_{\smallidx}$. Writing the coordinates of
$Q_{\bigidx}^{w_r}$ explicitly,
\begin{align}\label{eq:Qn wr}
Q_{\bigidx}^{w_r}=\{\left(\begin {array}{cc|ccc|cc}
  a & x & y_1 & y_2 & y_3 & z_1 & z_2 \\
    & b & 0   & 0   & y_4 & 0   & z_1' \\ \hline
    &   &    &    &    & y_4'& y_3' \\
    &   &     & Q_{\smallidx}'  &   & 0   & y_2' \\
    &   &     &     &  & 0   & y_1' \\\hline
    &   &     &     &     & b^* & x' \\
    &   &     &     &     &     & a^*
\end{array}\right)\},
\end{align}
where
$a\in\GL{r},b\in\GL{\bigidx-\smallidx-r},y_1\in\Mat{r\times\smallidx},y_2\in\Mat{r\times1}$
and $Q_{\smallidx}'$ is the maximal parabolic subgroup of
$H_{\smallidx}<Q_{\bigidx-\smallidx}$ whose Levi part is isomorphic
to $\GL{\smallidx}$. Denote by $C_r$ the
subgroup of $Q_{\bigidx}^{w_r}$ consisting of the elements with $I_{2\smallidx+1}$ in the
middle block (i.e., $I_{2\smallidx+1}$ instead of $Q_{\smallidx}'$). According to \eqref{eq:Qn wr},
$Q_{\bigidx}^{w_r}=C_r\rtimes Q_{\smallidx}'$. Moreover,
\begin{align*}
\rconj{\eta}(Q_{\bigidx}^{w_r})\cap
N_{\bigidx-\smallidx}G_{\smallidx}=(\rconj{\eta}C_r\cap
N_{\bigidx-\smallidx})(\rconj{\eta}Q_{\smallidx}'\cap
G_{\smallidx})=C_r^{\eta}\rtimes Q_{\smallidx}'^{\eta}.
\end{align*}
Using this \eqref{eq:E_f psi second unfolding} takes the form
\begin{align}\label{eq:E_r unfolding}
E_{f_s}^{\psi_{\gamma}}(h)=\sum_{r=0}^{\bigidx-\smallidx}
\sum_{\eta\in\rmodulo{\lmodulo{Q_{\bigidx}^{w_r}}{Q_{\bigidx-\smallidx}}}{N_{\bigidx-\smallidx}G_{\smallidx}}}
\quad\sum_{g\in\lmodulo{Q_{\smallidx}'^{\eta}}{G_{\smallidx}}}
\quad\int_{\lmodulo{C_r^{\eta}(F)}{N_{\bigidx-\smallidx}(\Adele)}}
f_s(w_r\eta vgh,1)\psi_{\gamma}(v)dv.
\end{align}

For $k\geq0$, let $S_k$ be the group of permutations of
$\{\intrange{1}{k}\}$ ($S_0=\{1\}$). Denote the natural association
of $S_k$ with the permutation matrices in $\GL{k}$ by $\sigma\mapsto
b(\sigma)$ ($b(\sigma)_{i,j}=\delta_{i,\sigma(j)}$). Note that
we have
\begin{align*}
\rmodulo{\lmodulo{Q_{\bigidx}^{w_r}}{Q_{\bigidx-\smallidx}}}{N_{\bigidx-\smallidx}G_{\smallidx}}\isomorphic
\lmodulo{(S_r\times
S_{\bigidx-\smallidx-r})}{S_{\bigidx-\smallidx}}\times\rmodulo{\lmodulo{Q_{\smallidx}'}{H_{\smallidx}}}{G_{\smallidx}}.
\end{align*}
Here $S_r$ (resp. $S_{\bigidx-\smallidx-r}$) is embedded in
$S_{\bigidx-\smallidx}$ as the subgroup of permutations of
$\{\intrange{1}{r}\}$ (resp. of
$\{\intrange{r+1}{\bigidx-\smallidx}\}$). The space
$\rmodulo{\lmodulo{Q_{\smallidx}'}{H_{\smallidx}}}{G_{\smallidx}}$
is trivial containing a single element which we represent using the identity
element, when $G_{\smallidx}(F)$ is \quasisplit, otherwise it
contains two more elements (\cite{GPS} Section~3, \cite{RGS}
Proposition~4.4). Representatives for these elements
may be taken to be 
\begin{align*}
\xi_1=\left(\begin{array}{ccccc}I_{\smallidx-1}\\&0&0&-\gamma^{-1}\\&0&-1&-2\beta^{-1}\\&-\gamma&\beta&1\\&&&&I_{\smallidx-1}\\\end{array}\right),
\xi_2=\left(\begin{array}{ccccc}I_{\smallidx-1}\\&0&0&-\gamma^{-1}\\&0&-1&2\beta^{-1}\\&-\gamma&-\beta&1\\&&&&I_{\smallidx-1}\\\end{array}\right).
\end{align*}
Note that $\xi_1$ and $\xi_2$ are related to the matrix $M$ defined in Section~\ref{subsection:G_l in H_n}, specifically
$diag(I_{\smallidx},-\beta,I_{\smallidx})M\in Q_{\smallidx}'\xi_1$ and
\begin{align*}
diag(I_{\smallidx-1},\left(\begin{array}{ccc}&&1\\&\beta\\1\end{array}\right),I_{\smallidx-1})M\in Q_{\smallidx}'\xi_2.
\end{align*}
A set of representatives for
$\rmodulo{\lmodulo{Q_{\bigidx}^{w_r}}{Q_{\bigidx-\smallidx}}}{N_{\bigidx-\smallidx}G_{\smallidx}}$
is given by
\begin{align}\label{set:representatives for second filtration}
\mathcal{A}(r)=\setof{diag(b(\sigma),\xi,b(\sigma)^*)}{\sigma\in\lmodulo{(S_r\times
S_{\bigidx-\smallidx-r})}{S_{\bigidx-\smallidx}},\xi\in\rmodulo{\lmodulo{Q_{\smallidx}'}{H_{\smallidx}}}{G_{\smallidx}}}.
\end{align}
We claim the following:
\begin{claim}\label{claim:unfolding of E_f to short sum}
For all $r>0$, $\eta\in \mathcal{A}(r)$ and $h\in
H_{\bigidx}(\Adele)$,
\begin{align}\label{eq:claim unfolding of E_f to short sum}
\int_{\lmodulo{C_r^{\eta}(F)}{N_{\bigidx-\smallidx}(\Adele)}}
f_s(w_r\eta vh,1)\psi_{\gamma}(v)dv=0.
\end{align}
\end{claim}
The proof of this claim will be given below. It follows that only
$r=0$ contributes to $E_{f_s}^{\psi}(h)$,
\begin{align}\label{eq:E_r unfolding only r=0 counts}
E_{f_s}^{\psi_{\gamma}}(h)=\sum_{\eta\in\mathcal{A}(0)}
\quad\sum_{g\in\lmodulo{Q_{\smallidx}'^{\eta}}{G_{\smallidx}}}
\quad\int_{\lmodulo{C_0^{\eta}(F)}{N_{\bigidx-\smallidx}(\Adele)}}
f_s(w_0\eta vgh,1)\psi_{\gamma}(v)dv.
\end{align}
We return to
integral~\eqref{integral:global l <= n}. Set $\eta_0=I_{2\bigidx+1}$. In the \quasisplit\ case,
$\mathcal{A}(0)=\{\eta_0\}$. The next task is to show
that in the split case only $\eta_0$ contributes to
the global integral. Indeed, since
\begin{align*}
h\mapsto\sum_{g\in\lmodulo{Q_{\smallidx}'^{\eta}}{G_{\smallidx}}}
\quad\int_{\lmodulo{C_0^{\eta}(F)}{N_{\bigidx-\smallidx}(\Adele)}}
f_s(w_0\eta vgh,1)\psi_{\gamma}(v)dv
\end{align*}
is $G_{\smallidx}(F)$-invariant on the left, after plugging \eqref{eq:E_r unfolding only r=0 counts}
into \eqref{integral:global l <= n} we can change the order
of the $dg$-integration and summation over $\eta$ in \eqref{integral:global l <= n} and
obtain three integrals of the form
\begin{align}\label{integral:cusp againt E r,eta}
\int_{\lmodulo{Q_{\smallidx}'^{\eta}(F)}{G_{\smallidx}(\Adele)}}\varphi(g)
\int_{\lmodulo{C_0^{\eta}(F)}{N_{\bigidx-\smallidx}(\Adele)}}
f_s(w_0\eta vg,1)\psi_{\gamma}(v)dvdg.
\end{align}
When $\eta=diag(I_{\bigidx-\smallidx},\xi_i,I_{\bigidx-\smallidx})$,
$i=1,2$, $Q_{\smallidx}'^{\eta}$ is a parabolic subgroup of
$G_{\smallidx}$. To see this, consider for example $i=1$. In the definition of
$Q_{\smallidx}'^{\eta}$ we regarded $G_{\smallidx}$ as a subgroup of $H_{\smallidx}$ -
the subgroup $\setof{\rconj{M}[g]_{\mathcal{E}_{G_{\smallidx}}'}}{g\in G_{\smallidx}}$ (see Section~\ref{subsection:G_l in H_n}). Put
$G_{\smallidx}'=\setof{[g]_{\mathcal{E}_{G_{\smallidx}}'}}{g\in G_{\smallidx}}$.
The image of $G_{\smallidx}$ in $H_{\smallidx}$ is $\rconj{M}G_{\smallidx}'$. Also set
$t_{\beta}=diag(I_{\smallidx},-\beta,I_{\smallidx})$. Then
$t_{\beta}M=q\xi_1$ for some $q\in Q_{\smallidx}'$ and since $\rconj{q}Q_{\smallidx}'=
\rconj{t_{\beta}}Q_{\smallidx}'=Q_{\smallidx}'$,
\begin{align*}
Q_{\smallidx}'^{\eta}=\rconj{\xi_1}Q_{\smallidx}'\cap\rconj{M}G_{\smallidx}'
=\rconj{q\xi_1}Q_{\smallidx}'\cap\rconj{M}G_{\smallidx}'
=\rconj{M}(Q_{\smallidx}'\cap G_{\smallidx}'),
\end{align*}
where we regarded $\eta$ as an element of $H_{\smallidx}$. Because $Q_{\smallidx}'\cap G_{\smallidx}'$ is a parabolic subgroup of
$G_{\smallidx}'$, $Q_{\smallidx}'^{\eta}$ is a parabolic subgroup of
$G_{\smallidx}$. A similar calculation holds for $i=2$ (if $t_{\beta}$ is the matrix written above satisfying $t_{\beta}M\in Q_{\smallidx}'\xi_2$,
$\rconj{t_{\beta}}Q_{\smallidx}'\cap G_{\smallidx}'$ is a parabolic subgroup of $G_{\smallidx}'$).
As in \cite{GPS} (Section~1.2), 
using the fact that $\varphi$ is a cusp form it is shown that
when $\eta\ne\eta_0$, integral~\eqref{integral:cusp againt E r,eta} vanishes.
\begin{remark}\label{remark:reason for nonstandard embedding}
The embedding $G_{\smallidx}<H_{\bigidx}$ has the property that in the
split case $Q_{\smallidx}'^{\eta_0}$ is not a parabolic subgroup
of $G_{\smallidx}$. Otherwise a unique representative
$\eta_i=diag(I_{\bigidx-\smallidx},\xi_i,I_{\bigidx-\smallidx})$
would contribute to the global integral and $\rconj{\eta_i^{-1}}
T_{G_{\smallidx}}$ would not be a subgroup of $T_{H_{\bigidx}}$.
Hence the complication arising from the embedding, namely that $T_{G_{\smallidx}}$ is not a subgroup of $T_{H_{\bigidx}}$,
is unavoidable (this issue is also dealt with in \cite{GPS}).
\end{remark}

Therefore in both split and \quasisplit\ cases
integral~\eqref{integral:global l <= n} equals
\begin{align}\label{integral:after unfolding E_f completely}
\int_{\lmodulo{Q_{\smallidx}'^{\eta_0}(F)}{G_{\smallidx}(\Adele)}}\varphi(g)
\int_{\lmodulo{C_0^{\eta_0}(F)}{N_{\bigidx-\smallidx}(\Adele)}}
f_s(w_0vg,1)\psi_{\gamma}(v)dvdg.
\end{align}
Here
\begin{align*}
&Q_{\smallidx}'^{\eta_0}=\setof{q'(a,b,c,d)=\left(\begin
{array}{ccccc}
  a & b  & c & b'  & d \\
   & 1  & 0 & 0  &  b'' \\
   &   & 1 &  0 & c' \\
   &   &  &  1 & b''' \\
   &   &  &   & a^*
\end{array}\right)\in H_{\smallidx}}{a\in\GL{\smallidx-1}},\\
&C_0^{\eta_0}=\setof{\left(\begin {array}{ccccc}
  z & 0  & 0 & x  & 0 \\
   & I_{\smallidx}  & 0 & 0  &  x' \\
   &   & 1 &  0 & 0 \\
   &   &  &  I_{\smallidx} & 0 \\
   &   &  &   & z^*
\end{array}\right)\in H_{\bigidx}}{z\in Z_{\bigidx-\smallidx}}.
\end{align*}
Let $R_{\smallidx}'$ be the subgroup of $G_{\smallidx}$ whose image
in $H_{\smallidx}$ is $\{q'(I_{\smallidx-1},0,c,d)\}$. In
coordinates,
\begin{align*}
&R_{\smallidx}'=\begin{dcases}\{
\left(\begin{array}{cccc}I_{\smallidx-1}&4\gamma u_1&u_1&u_2\\
&1&0&u_1'\\
&&1&4\gamma u_1'\\
&&&I_{\smallidx-1}\end{array}\right)\}&\text{split $G_{\smallidx}(F),$}\\
\{\left(\begin{array}{cccc}I_{\smallidx-1}&u_1&0&u_2\\
&1&0&u_1'\\
&&1&0\\
&&&I_{\smallidx-1}\end{array}\right)\}&\text{\quasisplit\
$G_{\smallidx}(F)$.}
\end{dcases}
\end{align*}
Recall that $Y_{\smallidx}$ is the mirabolic subgroup of
$\GL{\smallidx}$. In the split case we regard $Y_{\smallidx}$ as a subgroup of $L_{\smallidx}$ and we have
$U_{G_{\smallidx}}=Z_{\smallidx}\ltimes R_{\smallidx}'$, where
$Z_{\smallidx}<Y_{\smallidx}$.
In the \quasisplit\ case let
\begin{align*}
P_{\smallidx}'=\setof{\left(\begin{array}{cccc}a&u_1&v_1&u_2\\
&1&0&u_1'\\
&&1&v_1'\\
&&&a^*\end{array}\right)}{a\in\GL{\smallidx-1}}<P_{\smallidx}.
\end{align*}
The group $Y_{\smallidx}$ is regarded as the quotient group
$\lmodulo{R_{\smallidx}'}{P_{\smallidx}'}$ and the subgroup $Z_{\smallidx}$ of $Y_{\smallidx}$ is
$\lmodulo{R_{\smallidx}'}{U_{G_{\smallidx}}}$. Then an element $z\in Z_{\smallidx}(\Adele)$ will actually be
a representative in $U_{G_{\smallidx}}(\Adele)$.

For a fixed $g\in
G_{\smallidx}(\Adele)$, define
\begin{align*}
\varphi_g^0(x)=\int_{\lmodulo{R_{\smallidx}'(F)}{R_{\smallidx}'(\Adele)}}\varphi(r'xg)dr'\qquad(x\in
Y_{\smallidx}(\Adele)).
\end{align*}
Then $\varphi_g^0$ is a cuspidal function on $Y_{\smallidx}(\Adele)$
and according to the Fourier transform for $\varphi_g^0(1)$ (\cite{Sh} Section~5, see
\cite{GPS} Section~1),
\begin{align*}
\varphi_g^0(1)=\sum_{\delta\in\lmodulo{Z_{\smallidx-1}}{\GL{\smallidx-1}}}\int_{\lmodulo{R_{\smallidx}'(F)}{R_{\smallidx}'(\Adele)}}
\int_{\lmodulo{Z_{\smallidx}(F)}{Z_{\smallidx}(\Adele)}}\varphi(r'z\delta
g)\psi^{\star}(z)dzdr',
\end{align*}
where $\GL{\smallidx-1}<Y_{\smallidx}$ and $\psi^{\star}$ is the
character of $\lmodulo{Z_{\smallidx}(F)}{Z_{\smallidx}(\Adele)}$
defined by
\begin{align*}
\psi^{\star}(z)=\begin{cases}
\psi(\sum_{i=1}^{\smallidx-2}z_{i,i+1}+\quarter z_{\smallidx-1,\smallidx})&\text{$G_{\smallidx}(F)$ is split,}\\
\psi(\sum_{i=1}^{\smallidx-2}z_{i,i+1}+\half
z_{\smallidx-1,\smallidx})&\text{$G_{\smallidx}(F)$ is
\quasisplit.}\end{cases}
\end{align*}
Note that in the \quasisplit\ case, $\psi^{\star}$ is well-defined
on $Z_{\smallidx}(\Adele)$ as a quotient group.
Write $u=r'z\in
U_{G_{\smallidx}}(\Adele)$ and change variables in $r$ and $z$ so that
\begin{align*}
u=\begin{dcases}
\left(\begin{array}{cccc}z_1&z_2&u_1&u_2\\
&1&0&u_1'\\
&&1&z_2'\\
&&&z_1^*\end{array}\right)&\text{split $G_{\smallidx}(F),$}\\
\left(\begin{array}{cccc}z_1&u_1&z_2&u_2\\
&1&0&u_1'\\
&&1&z_2'\\
&&&z_1^*\end{array}\right)&\text{\quasisplit\ $G_{\smallidx}(F)$.}
\end{dcases}
\end{align*}
We obtain
\begin{align*}
\varphi_g^0(1)=\sum_{\delta\in\lmodulo{Z_{\smallidx-1}}{\GL{\smallidx-1}}}\int_{\lmodulo{U_{G_{\smallidx}}(F)}{U_{G_{\smallidx}}(\Adele)}}
\varphi(u\delta
g)\psi_{\gamma}(u)du=\sum_{\delta\in\lmodulo{Z_{\smallidx-1}}{\GL{\smallidx-1}}}W_{\varphi}(\delta
g),
\end{align*}
where
$W_{\varphi}(g)=\int_{\lmodulo{U_{G_{\smallidx}}(F)}{U_{G_{\smallidx}}(\Adele)}}\varphi(ug)\psi_{\gamma}(u)du$. 

Next note that the function 
\begin{align*}
g\mapsto\int_{\lmodulo{C_0^{\eta_0}(F)}{N_{\bigidx-\smallidx}(\Adele)}}
f_s(w_0vg,1)\psi_{\gamma}(v)dv
\end{align*}
is invariant on the left for $r'\in R_{\smallidx}'(\Adele)$ and
$\delta\in\GL{\smallidx-1}(F)<Y_{\smallidx}(F)$. This follows
because $G_{\smallidx}$ stabilizes $\psi_{\gamma}$ and normalizes $N_{\bigidx-\smallidx}$, $r'$ commutes with any element of
$C_0^{\eta_0}$, $\rconj{w_0^{-1}}r'\in U_{\bigidx}$, $\GL{\smallidx-1}$ normalizes
$C_0^{\eta_0}$ and for $a\in\GLF{\smallidx-1}{F}$, $\rconj{w_0^{-1}}diag(I_{\bigidx-\smallidx},a,I_2,a^*,I_{\bigidx-\smallidx})\in M_{\bigidx}(F)$ (if $G_{\smallidx}(F)$ is \quasisplit, $\GL{\smallidx-1}$ is
regarded as a quotient group). Also
$\lmodulo{Z_{\smallidx-1}}{\GL{\smallidx-1}}\isomorphic\lmodulo{U_{G_{\smallidx}}}{Q_{\smallidx}'^{\eta_0}}$. Combining the formula for $\varphi_g^0(1)$ with these observations,
integral~\eqref{integral:after unfolding E_f completely} becomes
\begin{align*}
&\int_{\lmodulo{R_{\smallidx}'(\Adele)Q_{\smallidx}'^{\eta_0}(F)}{G_{\smallidx}(\Adele)}}\varphi_g^0(1)
\int_{\lmodulo{C_0^{\eta_0}(F)}{N_{\bigidx-\smallidx}(\Adele)}}
f_s(w_0vg,1)\psi_{\gamma}(v)dvdg\\\notag
&=\int_{\lmodulo{R_{\smallidx}'(\Adele)Q_{\smallidx}'^{\eta_0}(F)}{G_{\smallidx}(\Adele)}}
\sum_{\delta\in\lmodulo{Z_{\smallidx-1}}{\GL{\smallidx-1}}}W_{\varphi}(\delta
g)\int_{\lmodulo{C_0^{\eta_0}(F)}{N_{\bigidx-\smallidx}(\Adele)}}
f_s(w_0v\delta g,1)\psi_{\gamma}(v)dvdg\\\notag
&=\int_{\lmodulo{R_{\smallidx}'(\Adele)U_{G_{\smallidx}}(F)}{G_{\smallidx}(\Adele)}}
W_{\varphi}(g)
\int_{\lmodulo{C_0^{\eta_0}(F)}{N_{\bigidx-\smallidx}(\Adele)}}
f_s(w_0vg,1)\psi_{\gamma}(v)dvdg.
\end{align*}
For $z\in Z_{\smallidx}(\Adele)$,
$W_{\varphi}(zg)=\psi_{\gamma}^{-1}(z)W_{\varphi}(g)$. Now the integral becomes 
\begin{align*}
&\int_{\lmodulo{U_{G_{\smallidx}}(\Adele)}{G_{\smallidx}(\Adele)}}
W_{\varphi}(g)
\int_{\lmodulo{Z_{\smallidx}(F)}{Z_{\smallidx}(\Adele)}}
\int_{\lmodulo{C_0^{\eta_0}(F)}{N_{\bigidx-\smallidx}(\Adele)}}
f_s(w_0vzg,1)\psi_{\gamma}(v)\psi_{\gamma}^{-1}(z)dvdzdg\\\notag&=
\int_{\lmodulo{U_{G_{\smallidx}}(\Adele)}{G_{\smallidx}(\Adele)}}
W_{\varphi}(g)
\int_{\lmodulo{Z_{\smallidx}(F)}{Z_{\smallidx}(\Adele)}}
\int_{\lmodulo{C_0^{\eta_0}(\Adele)}{N_{\bigidx-\smallidx}(\Adele)}}
\int_{\lmodulo{C_0^{\eta_0}(F)}{C_0^{\eta_0}(\Adele)}}\\\notag
&\quad f_s(w_0yvzg,1)\psi_{\gamma}(yv)\psi_{\gamma}^{-1}(z)dydvdzdg.
\end{align*}
For fixed $v$ and $g$, consider the function on $Z_{\smallidx}(\Adele)$ given by
\begin{align*}
z\mapsto\int_{\lmodulo{C_0^{\eta_0}(F)}{C_0^{\eta_0}(\Adele)}}
f_s(w_0yzvg,1)\psi_{\gamma}(yv)\psi_{\gamma}^{-1}(z)dy.
\end{align*}
As above since the elements of $R_{\smallidx}'$ and $C_0^{\eta_0}$ commute (etc.), this is indeed well-defined on
$Z_{\smallidx}(\Adele)$ also in the \quasisplit\ case. In either case, this function
is left-invariant on $Z_{\smallidx}(F)$, because $z$ normalizes
$C_0^{\eta_0}(\Adele)$, stabilizes $\psi_{\gamma}(y)$ and
$f_s(w_0z,1)=f_s(w_0,1)$. Then the mapping on $N_{\bigidx-\smallidx}(\Adele)$ defined by
\begin{align*}
v\mapsto\int_{\lmodulo{Z_{\smallidx}(F)}{Z_{\smallidx}(\Adele)}}
\int_{\lmodulo{C_0^{\eta_0}(F)}{C_0^{\eta_0}(\Adele)}}
f_s(w_0yzvg,1)\psi_{\gamma}(yv)\psi_{\gamma}^{-1}(z)dydz
\end{align*}
is invariant on the left by $C_0^{\eta_0}(\Adele)$. In addition, $z$ normalizes $N_{\bigidx-\smallidx}(\Adele)$ and stabilizes $\psi_{\gamma}(v)$. It follows that
we can change $vz\mapsto zv$ in the integral and obtain
\begin{align*}
&\int_{\lmodulo{U_{G_{\smallidx}}(\Adele)}{G_{\smallidx}(\Adele)}}
W_{\varphi}(g)
\int_{\lmodulo{C_0^{\eta_0}(\Adele)}{N_{\bigidx-\smallidx}(\Adele)}}
\int_{\lmodulo{Z_{\smallidx}(F)}{Z_{\smallidx}(\Adele)}}
\int_{\lmodulo{C_0^{\eta_0}(F)}{C_0^{\eta_0}(\Adele)}}\\\notag
&\quad
f_s(\rconj{w_0^{-1}}(yz)w_0vg,1)\psi_{\gamma}(y)\psi_{\gamma}^{-1}(z)\psi_{\gamma}(v)dydzdvdg.
\end{align*}
If $G_{\smallidx}(F)$ is \quasisplit, it is now convenient to replace $Z_{\smallidx}$ with the following set of representatives:
\begin{align*}
\setof{\left(\begin{array}{cccc}z_1&0&z_2&A(z_1,z_2)\\
&1&0&0\\
&&1&z_2'\\
&&&z_1^*\end{array}\right)\in U_{G_{\smallidx}}}{z_1\in Z_{\smallidx-1}},
\end{align*}
where $A(z_1,z_2)$ is a function of $z_1$ and $z_2$.

The double integral
$\int_{\lmodulo{Z_{\smallidx}(F)}{Z_{\smallidx}(\Adele)}}
\int_{\lmodulo{C_0^{\eta_0}(F)}{C_0^{\eta_0}(\Adele)}}$ can be
written as
$\int_{\lmodulo{\widetilde{Z_{\bigidx}}(F)}{\widetilde{Z_{\bigidx}}(\Adele)}}$,
where $\widetilde{Z_{\bigidx}}<Q_{\bigidx}$ is isomorphic to
$Z_{\bigidx}$ which is considered as a subgroup of $M_{\bigidx}$.
To
pass to $Z_{\bigidx}$ we conjugate by
$\omega_{\smallidx,\bigidx-\smallidx}\in\GL{\smallidx}$ (defined in
Section~\ref{subsection:groups in study}). In addition use a
conjugation by $diag(\gamma I_{\smallidx},I_{\bigidx-\smallidx})$
to get the character $\psi$ of $Z_{\bigidx}(\Adele)$ (the standard
character, see Section~\ref{subsection:notation for reps}). Also
$R_{\smallidx,\bigidx}=\lmodulo{C_0^{\eta_0}}{N_{\bigidx-\smallidx}}$.
The final form of the integral is
\begin{equation*}\int_{\lmodulo{U_{G_{\smallidx}}(\Adele)}{G_{\smallidx}(\Adele)}}W_{\varphi}(g)
\int_{R_{\smallidx,\bigidx}(\Adele)}
W_{f_s}^{\psi}(w_{\smallidx,\bigidx}rg,1)\psi_{\gamma}(r)drdg. \qedhere
\end{equation*}
\end{proof} 

\begin{proof}[Proof of Claim~\ref{claim:unfolding of E_f to short sum}] 
Put $X=C_r^{\eta}\cap\rconj{w_r\eta}U_{\bigidx}$. Since $w_r\eta X\eta^{-1}w_r^{-1}<U_{\bigidx}$,
for any $c\in
C_r^{\eta}(\Adele)$ and $h\in H_{\bigidx}(\Adele)$,
\begin{align*}
\int_{\lmodulo{X(F)}{X(\Adele)}}f_s(w_r\eta
xch,1)\psi_{\gamma}(xc)dx=f_s(w_r\eta
ch,1)\psi_{\gamma}(c)\int_{\lmodulo{X(F)}{X(\Adele)}}\psi_{\gamma}(x)dx.
\end{align*}
Hence the mapping $c\mapsto \int_{\lmodulo{X(F)}{X(\Adele)}}f_s(w_r\eta
xch,1)\psi_{\gamma}(xc)dx$ is left-invariant on $C_r^{\eta}(F)$. Then \eqref{eq:claim unfolding
of E_f to short sum} equals
\begin{align*}
\int_{\lmodulo{C_r^{\eta}(\Adele)}{N_{\bigidx-\smallidx}(\Adele)}}
\int_{\lmodulo{(\lmodulo{X(F)}{X(\Adele)})}{(\lmodulo{C_r^{\eta}(F)}{C_r^{\eta}(\Adele)})}}
\int_{\lmodulo{X(F)}{X(\Adele)}}f_s(w_r\eta xcvh,1)\psi_{\gamma}(xcv)dxdcdv.
\end{align*}
Now it is enough to show $\frestrict{\psi_{\gamma}}{X(\Adele)}\nequiv 1$, because then the $dx$-integration vanishes.
The proof of the claim proceeds by finding $O<X$ such that $\frestrict{\psi_{\gamma}}{O(\Adele)}\nequiv 1$. The subgroup $O$
depends on $r$ and $\eta=diag(b(\sigma),\xi,b(\sigma)^*)$.

If there is some $1\leq j< \bigidx-\smallidx$ such that
$\sigma(j)\leq r$ and $\sigma(j+1)>r$, we let $O$ be the image in
$N_{\bigidx-\smallidx}$ of the subgroup of $Z_{\bigidx-\smallidx}$
consisting of matrices with $1$ on the diagonal, an arbitrary
element 
in the $(j,j+1)$-th coordinate and
zero elsewhere. 
Since for $z\in
Z_{\bigidx-\smallidx}$,
$(\rconj{b(\sigma)^{-1}}z)_{\sigma(j),\sigma(j+1)}=z_{j,j+1}$ one
sees that $w_r\eta O\eta^{-1}w_r^{-1}<U_{\bigidx}$ and $\eta
O\eta^{-1}<C_r$. Also $O<N_{\bigidx-\smallidx}$, whence $O<C_r^{\eta}=\rconj{\eta}C_r\cap
N_{\bigidx-\smallidx}$ and $O<X$. It is clear that $\psi_{\gamma}|_{O(\Adele)}\nequiv 1$.

If no such $j$ exists, we can assume
$\sigma(\bigidx-\smallidx-r+j)\leq r$ for all $1\leq j\leq r$. Hence
$\sigma(j)>r$ for $1\leq j\leq\bigidx-\smallidx-r$. By the
definition of $\mathcal{A}(r)$, we can assume
$\sigma^{-1}(1)<\ldots<\sigma^{-1}(r)$ (if $r>1$) and
$\sigma^{-1}(r+1)<\ldots<\sigma^{-1}(\bigidx-\smallidx)$ (if
$r<\bigidx-\smallidx$). Combining these observations, we may assume
\begin{align*}
\sigma^{-1}(j)=\begin{cases}\bigidx-\smallidx-r+j&1\leq j\leq r,\\
j-r&r<j\leq\bigidx-\smallidx.\end{cases}
\end{align*}
Hence
$b(\sigma)=\omega_{r,\bigidx-\smallidx-r}=\bigl(\begin{smallmatrix}&I_r\\I_{\bigidx-\smallidx-r}&\end{smallmatrix}\bigr)$.
In this case let $O'(\Adele)$ be the following subgroup of $N_{\bigidx-\smallidx}(\Adele)$.
If $\xi=I_{2\smallidx+1}$,
\begin{align*}
O'(\Adele)=\setof{diag(I_{\bigidx-\smallidx-1},\left(\begin{array}{ccccccc}1&0&0&a&a&0&-\half
a^2\\&I_{\smallidx-1}&&&&&0\\&&1&&&&-a\\&&&1&&&-a\\&&&&1&&0\\&&&&&I_{\smallidx-1}&0\\&&&&&&1\end{array}\right),I_{\bigidx-\smallidx-1})}{a\in\Adele
}.
\end{align*}
Otherwise
\begin{align*}
O'(\Adele)=\setof{diag(I_{\bigidx-\smallidx-1},\left(\begin{array}{ccccc}1&0&a&0&-\half
a^2\\&I_{\smallidx}&&&0\\&&1&&-a\\&&&I_{\smallidx}&0\\&&&&1\end{array}\right),I_{\bigidx-\smallidx-1})}{a\in\Adele
}.
\end{align*}
Define $O(\Adele)=\xi^{-1}O'(\Adele)\xi$ where $\xi$ is regarded as an element of $H_{\bigidx}$.
Then $O(\Adele)<N_{\bigidx-\smallidx}(\Adele)$ and $\eta
O(\Adele)\eta^{-1}$ is the subgroup of elements $u\in U_{\bigidx}(\Adele)$ with $1$
on the diagonal,
\begin{align*}
(u_{r,\bigidx+1},u_{r,\bigidx+2})=\begin{cases}
(a,a) &\xi=I_{2\smallidx+1},\\
(-a,-2a\beta^{-1}) &\xi=\xi_1,\\
(-a,2a\beta^{-1}) &\xi=\xi_2,\end{cases}
\end{align*}
where $a\in\Adele$, and $0$ in all other coordinates except
$u_{\bigidx+1,2\bigidx-r+2},u_{\bigidx,2\bigidx-r+2}$ and $u_{r,2\bigidx-r+2}$
(which are determined by $u_{r,\bigidx+1}$ and $u_{r,\bigidx+2}$).
\end{proof} 

\subsection{The case $\smallidx>\bigidx$}\label{subsection:The global integral for l>n}
We integrate a Fourier coefficient of $\varphi$ against $E_{f_s}$.
Let $N^{\smallidx-\bigidx}=Z_{\smallidx-\bigidx-1}\ltimes V_{\smallidx-\bigidx-1}$,
\begin{align*}
N^{\smallidx-\bigidx}=\setof{\left(\begin {array}{ccc}
  z & y_1 & y_2 \\
    & \idblk{2\bigidx+2} & y_1' \\
    &   & z^*
\end{array}\right)\in G_{\smallidx}}{z\in Z_{\smallidx-\bigidx-1}}.
\end{align*}
Define a character $\psi_{\gamma}$ of
$N^{\smallidx-\bigidx}(\Adele)$ by
$\psi_{\gamma}(zy)=\psi(z)\psi(\transpose{e_{\smallidx-\bigidx-1}}ye_{\gamma})$,
for $z\in Z_{\smallidx-\bigidx-1}(\Adele)$ and $y\in
V_{\smallidx-\bigidx-1}(\Adele)$ ($e_{\smallidx-\bigidx-1}$ and
$e_{\gamma}$ were given in Section~\ref{subsection:H_n in G_l}).
I.e., for $v\in N^{\smallidx-\bigidx}(\Adele)$,
\begin{align*}
\psi_{\gamma}(v)=
\begin{cases}\psi(\sum_{i=1}^{\smallidx-\bigidx-2}v_{i,i+1}+\quarter
v_{\smallidx-\bigidx-1,\smallidx}-\gamma v_{\smallidx-\bigidx-1,\smallidx+1})&\text{$G_{\smallidx}(F)$ is split},\\
\psi(\sum_{i=1}^{\smallidx-\bigidx-2}v_{i,i+1}+\half
v_{\smallidx-\bigidx-1,\smallidx+1})&\text{$G_{\smallidx}(F)$ is
\quasisplit}.\end{cases}
\end{align*}
Then $\psi_{\gamma}$ can be considered as a character of
$\lmodulo{N^{\smallidx-\bigidx}(F)}{N^{\smallidx-\bigidx}(\Adele)}$.
This character is trivial when $\smallidx=\bigidx+1$. Note that the
embedding $H_{\bigidx}<G_{\smallidx}$ (see
Section~\ref{subsection:H_n in G_l}) is such that $H_{\bigidx}$
normalizes $N^{\smallidx-\bigidx}$ and stabilizes $\psi_{\gamma}$. 
The $\psi_{\gamma}$-coefficient of $\varphi$ with respect to
$N^{\smallidx-\bigidx}(\Adele)$ is
\begin{align*}
\varphi^{\psi_{\gamma}}(g)=\int_{\lmodulo{N^{\smallidx-\bigidx}(F)}{N^{\smallidx-\bigidx}(\Adele)}}\varphi(vg)\psi_{\gamma}(v)dv\qquad(g\in G_{\smallidx}(\Adele)).
\end{align*}
The global integral is
\begin{align}\label{integral:global l > n}
I(\varphi,f_s,s)=\int_{\lmodulo{H_{\bigidx}(F)}{H_{\bigidx}(\Adele)}}\varphi^{\psi_{\gamma}}(h)E_{f_s}(h)dh.
\end{align}
For all $s$ (except at the poles of the Eisenstein series) it satisfies
\begin{align}\label{eq:global integral in homspace l > n}
I((hv)\cdot \varphi,h\cdot f_s,s)=\psi_{\gamma}^{-1}(v)I(\varphi,
f_s,s)\qquad(\forall v\in N^{\smallidx-\bigidx}(\Adele), h\in
H_{\bigidx}(\Adele)).
\end{align}
As in the previous section we prove that for decomposable data,
$I(\varphi,f_s,s)$ is representable as an Euler product.
\begin{proposition}\label{propo:basic global identity l > n}
Assume that $I(\varphi,f_s,s)$ is given by \eqref{integral:global l >
n}. For $\Re(s)>>0$,
\begin{align*}
I(\varphi,f_s,s)=\int_{\lmodulo{U_{H_{\bigidx}}(\Adele)}{H_{\bigidx}(\Adele)}}
(\int_{R^{\smallidx,\bigidx}(\Adele)}W_{\varphi}(rw^{\smallidx,\bigidx}h)
dr)W_{f_s}^{\psi}(h,1)dh,
\end{align*}
where $W_{\varphi}$ and $W_{f_s}^{\psi}$ are prescribed by
Proposition~\ref{propo:basic global identity l <= n},
\begin{align*}
&w^{\smallidx,\bigidx}=\left(\begin {array}{ccccc}
   &  I_{\bigidx} &  &   &  \\
  I_{\smallidx-\bigidx-1}  &   &  &   &    \\
   &   & I_2 &   &  \\
   &   &  &  & I_{\smallidx-\bigidx-1} \\
  &   &  & I_{\bigidx}  &
\end{array}\right)\in G_{\smallidx},\\
&R^{\smallidx,\bigidx}=\{\left(\begin {array}{ccccc}
  I_{\bigidx} &   &  &   &  \\
  x & I_{\smallidx-\bigidx-1}  &  &   &   \\
   &   & I_2 &   &  \\
   &   &  &  I_{\smallidx-\bigidx-1} &  \\
   &   &  &  x' & I_{\bigidx}
\end{array}\right)\}<G_{\smallidx}.
\end{align*}
\end{proposition}
\begin{proof}[Proof of Proposition~\ref{propo:basic global identity l > n}] 
Since $I(\varphi,f_s,s)$ is absolutely convergent and $\varphi^{\psi_{\gamma}}$ is invariant on the left for $H_{\bigidx}(F)$, for $\Re(s)>>0$,
\begin{align}\label{int:global l > n eulerian proof first step}
I(\varphi,f_s,s)&=\int_{\lmodulo{Q_{\bigidx}(F)}{H_{\bigidx}(\Adele)}}\varphi^{\psi_{\gamma}}(h)f_s(h,1)dh\\\notag
&=\int_{\lmodulo{M_{\bigidx}(F)U_{\bigidx}(\Adele)}{H_{\bigidx}(\Adele)}}(\int_{\lmodulo{U_{\bigidx}(F)}{U_{\bigidx}(\Adele)}}\varphi^{\psi_{\gamma}}(uh)du)f_s(h,1)dh.
\end{align}
Put $V'=N^{\smallidx-\bigidx}U_{\bigidx}$. Because $H_{\bigidx}$ stabilizes $\psi_{\gamma}$, we can
extend $\psi_{\gamma}$ to $V'$ trivially on $U_{\bigidx}$. For any $h\in H_{\bigidx}(\Adele)$, since
$\varphi$ is left-invariant by $G_{\smallidx}(F)$,
\begin{align}\label{int:global l > n creating V''}
\int_{\lmodulo{U_{\bigidx}(F)}{U_{\bigidx}(\Adele)}}\varphi^{\psi_{\gamma}}(uh)du
=\int_{\lmodulo{V'(F)}{V'(\Adele)}}\varphi(\rconj{(w^{\smallidx,\bigidx})^{-1}}v'w^{\smallidx,\bigidx}h)\psi_{\gamma}(v')dv'.
\end{align}
Set $V^{(0)}=\rconj{(w^{\smallidx,\bigidx})^{-1}}V'$. In coordinates
\begin{align*}
V^{(0)}=\setof{\left(\begin{array}{cccccc}I_{\bigidx}&0&\alpha_1u_1&\alpha_2u_1&e_4'&u_2\\e_1&z&e_2&e_3&e_5&e_4\\
&&1&0&*&*\\&&&1&*&*\\&&&&z^*&0\\&&&&e_1'&I_{\bigidx}\end{array}\right)}{z\in
Z_{\smallidx-\bigidx-1}},
\end{align*}
where
\begin{align*}(\alpha_1,\alpha_2)=\begin{cases}
(\beta,\frac1{2\beta})&\text{split $G_{\smallidx}$(F),}\\
(1,0)&\text{\quasisplit\ $G_{\smallidx}(F)$.}\end{cases}
\end{align*}
The \rhs\ of \eqref{int:global l > n
creating V''} equals
\begin{align}\label{int:global l > n creating V'' and letting V^0}
\int_{\lmodulo{V^{(0)}(F)}{V^{(0)}(\Adele)}}\varphi(v'w^{\smallidx,\bigidx}h)\psi_{\gamma}(v')dv',
\end{align}
where $\psi_{\gamma}$ is the character of $\lmodulo{V^{(0)}(F)}{V^{(0)}(\Adele)}$ defined by
\begin{align*}
\psi_{\gamma}(v')=
\begin{cases}\psi(z)\psi(\quarter(e_2)_{\smallidx-\bigidx-1}-\gamma (e_3)_{\smallidx-\bigidx-1})&\text{split $G_{\smallidx}(F)$,}\\
\psi(z)\psi(\half(e_3)_{\smallidx-\bigidx-1})&\text{\quasisplit\ $G_{\smallidx}(F)$.}\end{cases}
\end{align*}

Assume $\smallidx-\bigidx-1>0$. Decompose
$V^{(0)}=V^{(0)'}\rtimes R^{(1)}$ where $V^{(0)'}$ is the subgroup
of elements of $V^{(0)}$ for which the last row of $e_1$ is zero and
$R^{(1)}$ is the subgroup of matrices $r_1\in L_{\smallidx-1}$ that are images of
\begin{align*}
\left(\begin{array}{ccc}I_{\bigidx}\\&I_{\smallidx-\bigidx-2}\\r_1&&1\end{array}\right)\in
\GL{\smallidx-1}.
\end{align*}
Then \eqref{int:global l > n creating V'' and letting V^0} becomes
\begin{align}\label{int:global l > n creating V'' before 1 fourier}
\int_{\lmodulo{R^{(1)}(F)}{R^{(1)}(\Adele)}}
\int_{\lmodulo{V^{(0)'}(F)}{V^{(0)'}(\Adele)}}\varphi(v'r_1w^{\smallidx,\bigidx}h)\psi_{\gamma}(v')dv'dr_1.
\end{align}

For $x\in\Adele^{\bigidx}$ define $\jmath(x)\in G_{\smallidx}(\Adele)$
as follows. If $G_{\smallidx}(F)$ is split,
\begin{align*}
\jmath(x)=\left(\begin{array}{cccccc}I_{\bigidx}&&&x&&\\&I_{\smallidx-\bigidx-1}\\&&1&&&x'\\&&&1
\\&&&&I_{\smallidx-\bigidx-1}\\&&&&&I_{\bigidx}\end{array}\right),
\end{align*}
$x'=-\transpose{x}J_{\bigidx}$. In the \quasisplit\ case
\begin{align*}
\jmath(x)=\left(\begin{array}{cccccc}I_{\bigidx}&&&x&&A(x)\\&I_{\smallidx-\bigidx-1}\\&&1&&&\\&&&1
&&x'\\&&&&I_{\smallidx-\bigidx-1}\\&&&&&I_{\bigidx}\end{array}\right),
\end{align*}
where $A(x)=\frac1{2\rho}x(\transpose{x})J_{\bigidx}$ and $x'=\frac1{\rho}(\transpose{x})J_{\bigidx}$.

For any $h\in H_{\bigidx}(\Adele)$ and $r_1\in R^{(1)}(\Adele)$
define $\Phi_{r_1,h}(x):\Adele^{\bigidx}\rightarrow\C$ by
\begin{align*}
\Phi_{r_1,h}(x)=\int_{\lmodulo{V^{(0)'}(F)}{V^{(0)'}(\Adele)}}\varphi(v'\jmath(x)r_1w^{\smallidx,\bigidx}h)\psi_{\gamma}(v')dv'.
\end{align*}
The mapping $(r_1,h)\mapsto\Phi_{r_1,h}(0)$ is well-defined on
$\lmodulo{R^{(1)}(F)}{R^{(1)}(\Adele)}\times H_{\bigidx}(\Adele)$ and \eqref{int:global l > n creating V'' before 1 fourier}
equals
\begin{align*}
\int_{\lmodulo{R^{(1)}(F)}{R^{(1)}(\Adele)}}\Phi_{r_1,h}(0)dr_1
\end{align*}
(because $\jmath(0)=I_{2\smallidx}$).

We see that $\jmath(x)$ normalizes $V^{(0)'}$ and stabilizes
$\psi_{\gamma}$. Note that $\jmath(x)$ also normalizes $V^{(0)}$ but
does not stabilize $\psi_{\gamma}$ on $V^{(0)}$ since the last row
of $e_1$ interferes. Also $\varphi$ is left-invariant by
$G_{\smallidx}(F)$. Hence $\Phi_{r_1,h}$ defines a function on
$\lmodulo{F^{\bigidx}}{\Adele^{\bigidx}}$. According to the Fourier
inversion formula,
\begin{align*}
\Phi_{r_1,h}(0)&=\sum_{\lambda\in
\MatF{1\times\bigidx}{F}}\int_{\lmodulo{F^{\bigidx}}{\Adele^{\bigidx}}}\Phi_{r_1,h}(x)\psi^{\star\star}(\lambda
x)dx\\\notag &=\sum_{\lambda}
\int_{\lmodulo{F^{\bigidx}}{\Adele^{\bigidx}}}\int_{\lmodulo{V^{(0)'}(F)}{V^{(0)'}(\Adele)}}\varphi(v'\jmath(x)r_1w^{\smallidx,\bigidx}h)\psi_{\gamma}(v')\psi^{\star\star}(\lambda
x)dv'dx.
\end{align*}
Here $\psi^{\star\star}$ is the character of $\lmodulo{F}{\Adele}$ given
by
\begin{align*}\psi^{\star\star}(a)=\begin{cases}
\psi(-\gamma a)&\text{$G_{\smallidx}(F)$ is split,}\\
\psi(\half a)&\text{$G_{\smallidx}(F)$ is \quasisplit.}\end{cases}
\end{align*}

Let $\lambda\in R^{(1)}(F)$. Then $\lambda$ normalizes
$V^{(0)'}(\Adele)$, $\psi_{\gamma}(\rconj{\lambda}v')=\psi_{\gamma}(v')$ and
$\rconj{\lambda^{-1}}\jmath(x)=v_{\lambda,x}\jmath(x)$ with
$v_{\lambda,x}\in V^{(0)'}(\Adele)$ such that
$\psi_{\gamma}(v_{\lambda,x})=\psi^{\star\star}(\lambda x)$ (where on the \rhs\ we regard $\lambda$ as an element of $\MatF{1\times\bigidx}{F}$). Hence for
any $\lambda\in R^{(1)}(F)$,
\begin{align*}
&\int_{\lmodulo{V^{(0)'}(F)}{V^{(0)'}(\Adele)}}\varphi(v'\jmath(x)r_1w^{\smallidx,\bigidx}h)\psi_{\gamma}(v')\psi^{\star\star}(\lambda
x)dv'\\\notag
&=\int_{\lmodulo{V^{(0)'}(F)}{V^{(0)'}(\Adele)}}\varphi(\lambda
v'\jmath(x)r_1w^{\smallidx,\bigidx}h)\psi_{\gamma}(v')\psi^{\star\star}(\lambda
x)dv'\\\notag
&=\int_{\lmodulo{V^{(0)'}(F)}{V^{(0)'}(\Adele)}}\varphi(
v'\jmath(x)\lambda r_1w^{\smallidx,\bigidx}h)\psi_{\gamma}(v')dv'.
\end{align*}
Using this, integral~\eqref{int:global l > n creating V'' before 1
fourier} becomes
\begin{align*}
&\int_{\lmodulo{R^{(1)}(F)}{R^{(1)}(\Adele)}}\sum_{\lambda}
\int_{\lmodulo{F^{\bigidx}}{\Adele^{\bigidx}}}\int_{\lmodulo{V^{(0)'}(F)}{V^{(0)'}(\Adele)}}\varphi(v'\jmath(x)r_1w^{\smallidx,\bigidx}h)\psi_{\gamma}(v')\psi^{\star\star}(\lambda
x)dv'dxdr_1\\\notag &=
\int_{\lmodulo{R^{(1)}(F)}{R^{(1)}(\Adele)}}\sum_{\lambda}
\int_{\lmodulo{F^{\bigidx}}{\Adele^{\bigidx}}}\int_{\lmodulo{V^{(0)'}(F)}{V^{(0)'}(\Adele)}}\varphi(v'\jmath(x)\lambda
r_1w^{\smallidx,\bigidx}h)\psi_{\gamma}(v')dv'dxdr_1.
\end{align*}
The $dr_1$-integration and the summation over
$\lambda$ can be collapsed, we get
\begin{align*}
&\int_{R^{(1)}(\Adele)}
\int_{\lmodulo{F^{\bigidx}}{\Adele^{\bigidx}}}\int_{\lmodulo{V^{(0)'}(F)}{V^{(0)'}(\Adele)}}\varphi(v'\jmath(x)r_1w^{\smallidx,\bigidx}h)\psi_{\gamma}(v')dv'dxdr_1.
\end{align*}
Let $V^{(1)}$ be the subgroup of elements $v'\jmath(x)$. In coordinates,
\begin{align*}
V^{(1)}=\setof{\left(\begin{array}{cccccc}I_{\bigidx}&0&\alpha_1u_1&x&e_4'&u_2\\e_1&z&e_2&e_3&e_5&e_4\\
&&1&0&*&*\\&&&1&*&*\\&&&&z^*&0\\&&&&e_1'&I_{\bigidx}\end{array}\right)}{z\in
Z_{\smallidx-\bigidx-1}},
\end{align*}
where the last row of $e_1$ equals zero.
The integral becomes
\begin{align}\label{int:global l > n after one step}
&\int_{R^{(1)}(\Adele)}
\int_{\lmodulo{V^{(1)}(F)}{V^{(1)}(\Adele)}}\varphi(v'r_1w^{\smallidx,\bigidx}h)\psi_{\gamma}(v')dv'dr_1.
\end{align}

Assume $\smallidx-\bigidx-1>1$ (then the coordinates of $e_1$ are not all zero).
Decompose $V^{(1)}=V^{(1)'}\ltimes R^{(2)}$, $V^{(1)'}$ is the
subgroup of matrices whose last two rows of $e_1$ are zero and
$R^{(2)}$ consists of elements $r_2\in L_{\smallidx-2}$ that are images of
\begin{align*}
\left(\begin{array}{ccc}I_{\bigidx}\\&I_{\smallidx-\bigidx-3}\\r_2&&1\end{array}\right)\in
\GL{\smallidx-2}.
\end{align*}
For $x\in\Adele^{\bigidx}$, $\jmath(x)$ is redefined as the image in
$L_{\smallidx-1}$ of
\begin{align*}
\left(\begin{array}{cccc}I_{\bigidx}&&&x\\&I_{\smallidx-\bigidx-3}\\&&1\\&&&1\end{array}\right)\in\GL{\smallidx-1}.
\end{align*}
Define for $h\in H_{\bigidx}(\Adele)$, $r_1\in R^{(1)}(\Adele)$ and $r_2\in R^{(2)}(\Adele)$, $\Phi_{r_2,r_1,h}:\Adele^{\bigidx}\rightarrow\C$ by
\begin{align*}
\Phi_{r_2,r_1,h}(x)=\int_{\lmodulo{V^{(1)'}(F)}{V^{(1)'}(\Adele)}}\varphi(v'\jmath(x)r_2r_1w^{\smallidx,\bigidx}h)\psi_{\gamma}(v')dv'.
\end{align*}
Then $(r_2,r_1,h)\mapsto\Phi_{r_2,r_1,h}(0)$ is well-defined on $\lmodulo{R^{(2)}(F)}{R^{(2)}(\Adele)}\times R^{(1)}(\Adele)\times H_{\bigidx}(\Adele)$.
Here \eqref{int:global l > n after one step} is
\begin{align*}
&\int_{R^{(1)}(\Adele)}
\int_{\lmodulo{R^{(2)}(F)}{R^{(2)}(\Adele)}}\Phi_{r_2,r_1,h}(0)dr_2dr_1.
\end{align*}
Observe that $\jmath(x)$ normalizes $V^{(1)'}$ and stabilizes
$\psi_{\gamma}$ (since we zeroed out interfering rows of $e_1$).
Note that $\jmath(x)$ also normalizes $V^{(1)}$ but not $V^{(0)'}$. As
above $\Phi_{r_2,r_1,h}$ factors through the quotient map
$\Adele^{\bigidx}\rightarrow\lmodulo{F^{\bigidx}}{\Adele^{\bigidx}}$
and the Fourier inversion formula with respect to $\psi$ (instead of
$\psi^{\star\star}$) yields
\begin{align*}
\Phi_{r_2,r_1,h}(0)=
\sum_{\lambda}
\int_{\lmodulo{F^{\bigidx}}{\Adele^{\bigidx}}}
\int_{\lmodulo{V^{(1)'}(F)}{V^{(1)'}(\Adele)}}\varphi(v'\jmath(x)r_2r_1w^{\smallidx,\bigidx}h)\psi_{\gamma}(v')\psi(\lambda
x)dv'dx.
\end{align*}
Again using $\lambda\in R^{(2)}(F)$ to collapse the
$dr_2$-integration into an integration over $R^{(2)}(\Adele)$ and
letting $V^{(2)}$ be the subgroup of elements $v'\jmath(x)$ we reach
\begin{align*}
\int_{R^{(1)}(\Adele)}\int_{R^{(2)}(\Adele)}
\int_{\lmodulo{V^{(2)}(F)}{V^{(2)}(\Adele)}}\varphi(v'r_2r_1w^{\smallidx,\bigidx}h)\psi_{\gamma}(v')dv'dr_2dr_1.
\end{align*}
Proceeding likewise for $k=\intrange{1}{\smallidx-\bigidx-1}$, 
in the $k$-th step $V^{(k-1)'}<V^{(k-1)}$ is the subgroup of elements with the last $k$ rows of $e_1$ zeroed out, the subgroup $R^{(k)}<L_{\smallidx-k}$ is composed of elements
\begin{align*}
\left(\begin{array}{ccc}I_{\bigidx}\\&I_{\smallidx-\bigidx-1-k}\\r_1&&1\end{array}\right)\in
\GL{\smallidx-k}
\end{align*}
and for $k>1$, $\jmath(x)\in L_{\smallidx-k+1}$ is the image of
\begin{align*}
\left(\begin{array}{cccc}I_{\bigidx}&&&x\\&I_{\smallidx-\bigidx-1-k}\\&&1\\&&&1\end{array}\right)\in\GL{\smallidx-k+1}.
\end{align*}
After $k=\smallidx-\bigidx-1$ steps (still assuming $\smallidx-\bigidx-1>0$) integral~\eqref{int:global l > n creating V'' and letting V^0} becomes
\begin{align}\label{int:global l > n after l-n-1 steps}
\int_{R^{\smallidx,\bigidx}(\Adele)}
\int_{\lmodulo{V^{(k)}(F)}{V^{(k)}(\Adele)}}\varphi(v'rw^{\smallidx,\bigidx}h)\psi_{\gamma}(v')dv'dr,
\end{align}
where
\begin{align*}
V^{(k)}=
\setof{\left(\begin{array}{cccccc}I_{\bigidx}&x&u_1&t&e_4'&u_2\\&z&e_2&e_3&e_5&e_4\\
&&1&0&*&*\\&&&1&*&*\\&&&&z^*&x'\\&&&&&I_{\bigidx}\end{array}\right)}{x=(0|x_0),x_0\in\Mat{\bigidx\times(\smallidx-\bigidx-2)}, z\in
Z_{\smallidx-\bigidx-1}}.
\end{align*}
(The first column of $x$ is zero.) If $\smallidx-\bigidx-1=0$, integral~\eqref{int:global l > n creating V'' and letting V^0} trivially equals \eqref{int:global l > n after l-n-1 steps} with $k=0$ and
$V^{(0)}$ defined earlier ($R^{\bigidx+1,\bigidx}=\{1\}$),
\begin{align*}
V^{(0)}=\{\left(\begin{array}{cccccc}I_{\bigidx}&\alpha_1u_1&\alpha_2u_1&u_2\\
&1&0&*\\&&1&*\\&&&I_{\bigidx}\end{array}\right)\}.
\end{align*}
We continue with the arguments for any $\smallidx>\bigidx$ and set $k=\smallidx-\bigidx-1$.

Note that the constants $\alpha_1,\alpha_2\in F$ can be ignored when defining
the measure $dv'$ with respect to the correpsonding coordinates (see e.g. $V^{(1)}$) since
for any $\alpha\in F^*$, $|\alpha|_{\Adele}=1$.

Let $Y_{\bigidx+1}$ be regarded in the following manner. If $k>0$, $Y_{\bigidx+1}$ is embedded through the embedding of $\GL{\bigidx+1}$ in $L_{\bigidx+1}$. Also if $k=0$ and $G_{\smallidx}(F)$ is split,
$Y_{\bigidx+1}<L_{\bigidx+1}$. Otherwise, $Y_{\bigidx+1}$ is regarded as the quotient group
$\lmodulo{R_{\smallidx}'}{P_{\smallidx}'}$ defined in the proof of
Proposition~\ref{propo:basic global identity l <= n} ($V^{(0)}=R_{\smallidx}'$). Then observe that for
any $g\in G_{\smallidx}(\Adele)$, the function on
$Y_{\bigidx+1}(\Adele)$ defined by
\begin{align*}
x\mapsto\int_{\lmodulo{V^{(k)}(F)}{V^{(k)}(\Adele)}}\varphi(v'xg)\psi_{\gamma}(v')dv'
\end{align*}
is cuspidal. Hence as in
Proposition~\ref{propo:basic global identity l <= n} according to
the Fourier transform at the identity,
\begin{align*}
\int_{\lmodulo{V^{(k)}(F)}{V^{(k)}(\Adele)}}\varphi(v'g)\psi_{\gamma}(v')dv'=
\sum_{\delta\in\lmodulo{Z_{\bigidx}(F)}{\GLF{\bigidx}{F}}}W_{\varphi}(\delta g).
\end{align*}
Here $\delta=diag(\delta,I_{2(\smallidx-\bigidx)},\delta^*)\in L_{\smallidx-1}$ (in case $k=0$ and
\quasisplit\ $G_{\smallidx}(F)$, $\delta$ is a representative modulo $R_{\smallidx}'(F)$ of this form)
and the Fourier transform is applied with respect
to the character $\psi^{\star}$ of $\lmodulo{Z_{\bigidx+1}(F)}{Z_{\bigidx+1}(\Adele)}$
defined by
\begin{align*}
\psi^{\star}(z)=\begin{cases}
\psi(\sum_{i=1}^{\bigidx}z_{i,i+1})&k>0,\\
\psi(\sum_{i=1}^{\bigidx-1}z_{i,i+1}+\quarter z_{\bigidx,\bigidx+1})&\text{$k=0$ and $G_{\smallidx}(F)$ is split,}\\
\psi(\sum_{i=1}^{\bigidx-1}z_{i,i+1}+\half
z_{\bigidx,\bigidx+1})&\text{$k=0$ and $G_{\smallidx}(F)$ is
\quasisplit.}\end{cases}
\end{align*}
Put this into \eqref{int:global l > n after l-n-1 steps} and use the
fact that $\delta$ normalizes $R^{\smallidx,\bigidx}$ without changing
the measure (since the coordinates of $\delta$ belong to $F$). It follows that
\eqref{int:global l > n after l-n-1 steps} equals
\begin{align*}
\int_{R^{\smallidx,\bigidx}(\Adele)} \sum_{\delta\in\lmodulo{Z_{\bigidx}(F)}{\GLF{\bigidx}{F}}}
W_{\varphi}(rw^{\smallidx,\bigidx}(\rconj{w^{\smallidx,\bigidx}}\delta)h)dr.
\end{align*}

Returning to the global integral, plugging this into
\eqref{int:global l > n eulerian proof first step} yields
\begin{align*}
\int_{\lmodulo{Z_{\bigidx}(F)U_{\bigidx}(\Adele)}{H_{\bigidx}(\Adele)}}
\int_{R^{\smallidx,\bigidx}(\Adele)}W_{\varphi}(rw^{\smallidx,\bigidx}h)f_s(h,1)drdh.
\end{align*}
Here we used the fact that $\rconj{w^{\smallidx,\bigidx}}\delta$ lies in the image of
$M_{\bigidx}(F)$ under the embedding $H_{\bigidx}<G_{\smallidx}$ and the mapping $h\mapsto f_s(h,1)$ is left-invariant on $M_{\bigidx}(F)$. Finally we check that
for $z\in Z_{\bigidx}(\Adele)<M_{\bigidx}(\Adele)$,
\begin{align*}
\int_{R^{\smallidx,\bigidx}(\Adele)}W_{\varphi}(rw^{\smallidx,\bigidx}zh)dr=
\psi^{-1}(z)\int_{R^{\smallidx,\bigidx}(\Adele)}W_{\varphi}(rw^{\smallidx,\bigidx}h)dr.
\end{align*}
Since $U_{H_{\bigidx}}=Z_{\bigidx}U_{\bigidx}$, the requested form
of the integral is obtained,
\begin{equation*}
\int_{\lmodulo{U_{H_{\bigidx}}(\Adele)}{H_{\bigidx}(\Adele)}}
(\int_{R^{\smallidx,\bigidx}(\Adele)}W_{\varphi}(rw^{\smallidx,\bigidx}h)dr)W_{f_s}^{\psi}(h,1)dh.\qedhere
\end{equation*}
\end{proof} 

\section{Computation of $\Psi(W,f_s,s)$ with unramified data}\label{section:Computation of the local integral with unramified data}
In Section~\ref{section:construction of the global integral} we
showed how to write the global integral for decomposable data,
\begin{align*}
I(\varphi,f_s,s)=\prod_{\nu}\Psi_{\nu}(W_{\varphi_{\nu}},f_{s,\nu},s).
\end{align*}
Let $S$ be a finite set of places of $F$, containing all the \archimedean\ places, such that for $\nu\notin S$, all
data are unramified. The so-called unramified computation (at the finite places) is the
computation of
$\Psi_{\nu}(W_{\varphi_{\nu}},f_{s,\nu},s)$ for
$\nu\notin S$. As explained in the introduction (see Chapter~\ref{chapter:introduction}), $I(\varphi,f_s,s)$
is related to Langlands' partial $L$-function
$L^S(\pi\times\tau,s)=\prod_{\nu\notin
S}L(\pi_{\nu}\times\tau_{\nu},s)$ by showing that for $\nu\notin S$,
$L(\tau_{\nu},Sym^2,2s)\Psi_{\nu}(W_{\varphi_{\nu}},f_{s,\nu},s)$ is equal to Langlands' local $L$-function
$L(\pi_{\nu}\times\tau_{\nu},s)$ defined with respect to the Satake
parameters of $\pi_{\nu}$ and $\tau_{\nu}$. It follows that
\begin{align*}
(\prod_{\nu\notin S}L(\tau_{\nu},Sym^2,2s))I(\varphi,f_s,s)=L^S(\pi\times\tau,s)\prod_{\nu\in S}\Psi_{\nu}(W_{\varphi_{\nu}},f_{s,\nu},s).
\end{align*}

Fix a place $\nu\notin S$. We pass to a local notation and denote
$W=W_{\varphi_{\nu}}$, $f_s=f_{s,\nu}$, etc. Our field $F$ is again
local \nonarchimedean. Then $W\in\Whittaker{\pi}{\psi_{\gamma}^{-1}}$ and
$f_s\in V(\tau,s)$ are the normalized unramified elements. Specifically, $W$ is invariant on the
right by $K_{G_{\smallidx}}$, $W(1)=1$, $f_s$ is right-invariant by $K_{H_{\bigidx}}$ and $f_s(1,1)=1$
(see Section~\ref{subsection:preliminaries unramified} for the other ``unramified" parameters).
In this section we prove Theorem~\ref{theorem:unramified computation}. Namely, we prove
that for $\Re(s)>>0$,
\begin{align*}
\Psi(W,f_s,s)=\frac{L(\pi\times\tau,s)}{L(\tau,Sym^2,2s)}.
\end{align*}
This immediately implies that the function $s\mapsto \Psi(W,f_s,s)$ has a meromorphic continuation to a function in $\C(q^{-s})$ and
the equality between meromorphic continuations holds for all $s$.

In the split case we assume that $G_{\smallidx}$ is defined using $J_{2\smallidx}$. As noted in
Section~\ref{section:construction of the global integral}, when $\pi$ is a local component of a global representation, the
group $G_{\smallidx}$ might be defined using another matrix $J$ representing
the same bilinear form. Then there is some $\alpha\in\GL{2\smallidx}$ such that
$\transpose{\alpha}J_{2\smallidx}\alpha=J$ and if $\pi$ is a representation of $G_{\smallidx}$ defined using $J$,
$\pi^{\alpha}(g)=\pi(\rconj{\alpha}g)$ is a representation of $G_{\smallidx}$ defined using $J_{2\smallidx}$.
Our results for $\pi^{\alpha}$ imply the theorem for $\pi$.

The computation of
$\Psi(W,f_s,s)$ involves several techniques and is quite lengthy. In
the case $\smallidx=\bigidx+1$ it involves invariant theory
and the case $\smallidx>\bigidx+1$ is reduced to that
case. When $\smallidx<\bigidx$, the unipotent integration over
$R_{\smallidx,\bigidx}$ is difficult to handle directly. In this
case (and when $\smallidx=\bigidx$) we use the existence of the $\gamma$-factor and proof of its
multiplicativity in the first variable (see
Theorem~\ref{theorem:multiplicity first var}).

\subsection{Preliminaries for the unramified case}\label{subsection:preliminaries unramified}
Assume that all data are unramified. This includes the field (i.e., $F$ is an unramified extension of some $\mathbb{Q}_p$), the representations, the character $\psi$ ($\frestrict{\psi}{\mathcal{O}}\equiv1$, $\frestrict{\psi}{\mathcal{P}^{-1}}\nequiv1$) and all constants used in the
construction of the integral (e.g. $\rho$, $\gamma$), where a constant is unramified if it belongs to
$\mathcal{O}^*$. 
Throughout this section let $G=G(F)$ be one of the groups
defined in Section~\ref{subsection:groups in study}, i.e. $G$ is
either $G_{\smallidx},H_{\bigidx}$ or $\GL{k}$. As usual in the unramified
case, $K_{G}=G(\mathcal{O})$. Recall that we fixed a Borel subgroup $B_G=T_G\ltimes U_G$, where
$T_G$ is the torus and $U_G$ is the unipotent radical. The measure $dg$ of $G$ is normalized so that $vol(K_{G})=vol(T_{G}\cap K_{G})=1$.
\subsubsection{Torus elements and partitions}\label{torus elements
and partitions} We define the following subset of $T_G$:
\begin{align*}
T_G^-=\setof{t\in T_G}{\forall\alpha\in\Delta_{G},\alpha(a)\leq1}.
\end{align*}
For example $T_{\GL{\bigidx}}=A_{\bigidx}$, $A_{\bigidx}^-$ is the set of elements $diag(a_1,\ldots,a_{\bigidx})\in
A_{\bigidx}$ such that $|a_1|\leq\ldots\leq|a_{\bigidx}|$. If $\Delta_{G}=\emptyset$ (e.g., $G=\GL{1},G_1$),
$T_{G}^-=T_G$ by definition and in fact $T_{G}^-=G$.

Let $k$ be the rank of $G$ and assume $k>0$ (that is, $G$ is not the \quasisplit\ group $G_1$). We have a group homomorphism
$T_G\rightarrow\Integers^k$ defined by
$\delta(t)=(\vartheta(t_1),\ldots,\vartheta(t_k))\in\Integers^k$, where
$t_i$ is the $(i,i)$-th coordinate of $t$ and $\vartheta$ is the valuation of $F$.
The mapping $t\mapsto\delta(t)$ defines an isomorphism between
$\lmodulo{K_G\cap T_G}{T_G}$ and $\Integers^k$. A $k$-tuple
$\delta=(\delta_1,\ldots,\delta_k)\in\Integers^{k}$ which satisfies
$\delta_1\geq\ldots\geq\delta_k\geq0$ is called a $k$-parts
partition. It is a partition of $|\delta|=\sum_{i=1}^{k}\delta_i$.
Denote by $P_k$ the set of partitions of $k$ parts. Observe that if
$t\in T_G^-$ satisfies $|t_k|\leq1$, $\delta(t)\in P_k$. This fact
will often be used when writing the Iwasawa decomposition.

Any $\delta\in\Integers^{m}$ can be viewed as an element of
$\Integers^{m'}$ for $m'\geq m$, by padding it with zeros on the
right, i.e. $\delta\mapsto(\delta,0_{m'-m})\in\Integers^{m'}$. When clear from
the context we continue to denote it by $\delta$. This induces an
inclusion $P^m\subset P^{m'}$. The notation $P^{m'}\setminus P^m$ signifies set difference, i.e., $\delta\in P^{m'}\setminus P^m$ satisfies $\delta_k>0$ for some $k>m$.
\subsubsection{The local $L$-function for unramified
data}\label{subsection:local L functions def}
Let $\pi$ be an irreducible unramified (generic) representation of $G$. Then $\pi=\cinduced{B_{G}}{G}{\mu}$
for an unramified character $\mu$ of $T_G$ (see e.g. \cite{CSh}). Assume $\mu=\mu_1\otimes\ldots\otimes\mu_k$ where $k$ is the rank of $G$.

The Satake parameter $t_{\pi}$ of $\pi$ is regarded as a
representative of its conjugacy class in the $L$-group $\rconj{L}G$. It is given by
\begin{align*}
t_{\pi}=
\begin{cases}
diag(\mu_1(\varpi),\ldots,\mu_{\bigidx}(\varpi))\in\GLF{\bigidx}{\C}&G=\GL{\bigidx},\\
diag(\mu_1(\varpi),\ldots,\mu_{\smallidx}(\varpi),\mu_{\smallidx}(\varpi)^{-1},\ldots,\mu_1(\varpi)^{-1})\in
SO_{2\smallidx}(\C)&G=\text{split }G_{\smallidx},\\
diag(\mu_1(\varpi),\ldots,\mu_{\bigidx}(\varpi),\mu_{\bigidx}(\varpi)^{-1},\ldots,\mu_1(\varpi)^{-1})\in
Sp_{2\bigidx}(\C)&G=H_{\bigidx}.\end{cases}
\end{align*}
When $G_{\smallidx}$ is \nonsplit\ over the field $F$ and split
over a quadratic extension $K=F(\sqrt{\rho})$, define (with an
abuse of notation)
\begin{align*}
t_{\pi}=diag(\mu_1(\varpi),\ldots,\mu_{\smallidx-1}(\varpi),-1,1,\mu_{\smallidx-1}(\varpi)^{-1},\ldots,\mu_1(\varpi)^{-1})\in\GLF{2\smallidx}{\C}.
\end{align*}
\begin{remark}
In fact, by the definition of the Satake parameter in the
\quasisplit\ case it is an element of $SO_{2\smallidx}(\C)\rtimes
Gal(K/F)$, so it is of the form
\begin{align*}
diag(\mu_1(\varpi),\ldots,\mu_{\smallidx-1}(\varpi),1,1,\mu_{\smallidx-1}(\varpi)^{-1},\ldots,\mu_1(\varpi)^{-1})\cdot\epsilon\in
SO_{2\smallidx}(\C)\rtimes Gal(K/F),
\end{align*}
with the Galois conjugation $\epsilon\in Gal(K/F)$. The previous
definition will be more practical for us.
\end{remark}
In the \quasisplit\ case also denote
\begin{align*}
t'_{\pi}=diag(\mu_1(\varpi),\ldots,\mu_{\smallidx-1}(\varpi),\mu_{\smallidx-1}(\varpi)^{-1},\ldots,\mu_1(\varpi)^{-1})\in
Sp_{2(\smallidx-1)}(\C).
\end{align*}

Let $\tau$ be an irreducible unramified representation of $\GL{\bigidx}$. The
(local) $L$-function of $\pi\times\tau$ with a complex parameter $s$ is given by
\begin{align*}
L(\pi\times\tau,s)={\det(1-(t_{\pi}\otimes t_{\tau})q^{-s})}^{-1}.
\end{align*}

In addition $L(\tau,\varsigma,s)={\det(1-\varsigma(t_{\tau})q^{-s})}^{-1}$,
for a finite dimensional representation $\varsigma$ of
$\GLF{\bigidx}{\C}$. In particular we shall use
$\varsigma=Sym^2$ - the
symmetric square representation and $\wedge^2$ - the
antisymmetric square representation.

\subsubsection{The Casselman-Shalika
formula}\label{subsection:casselman shalika formula}
Let $\pi$ be an irreducible unramified representation of $G$ and let $\chi$ be an unramified character of $U_G$. The normalized unramified Whittaker function
$W\in\Whittaker{\pi}{\chi}$ is a function which is
right-invariant by $K_G$ and satisfies $W(1)=1$. It follows that $W$
is determined by its values on $T_G$ and it is simple to show that
$W(t)=0$ if $t\in T_G$ does not belong to $T_G^-$. An explicit formula for $W$
describes $W(t)$ for $t\in T_G^-$, in terms of the structure of the
group - the root system and Weyl subgroup, the Satake parameter
$t_{\pi}$ and $t$.

Shintani \cite{S} proved such a formula for a representation $\pi$
of $\GL{\bigidx}$. In the general case of a connected reductive group, an explicit formula was
proved by Kato \cite{K} and by Casselman and Shalika \cite{CS2}. Let
$\chi_{\delta}^{G}$ be the character of the irreducible
representation of $\rconj{L}G$ of highest weight $\delta$. For $t\in
T_G^-$ denote $\delta=\delta(t)$. Then assuming that $G$ is split,
\begin{align*}
W(t)=\delta_{B_{G}}^{\half}(t)\chi_{\delta}^{G}(t_{\pi}).
\end{align*}
In the case of $\GL{\bigidx}$ this is the original formulation of
Shintani \cite{S}. In the other cases this formula is the
interpretation of the result of Casselman and Shalika \cite{CS2}, using the
structure of the Weyl group (see e.g. \cite{Spr} p.~141) and
the Weyl character formulas (\cite{FH,OO}). 

When $G=G_{\smallidx}$ is \quasisplit,
\begin{align*}
W(t)=\delta_{B_{G_{\smallidx}}}^{\half}(t)\chi_{\delta}^{Sp_{2(\smallidx-1)}}(t'_{\pi}),
\end{align*}
where $\chi_{\delta}^{Sp_{2(\smallidx-1)}}$ is the character of the
irreducible representation of $Sp_{2(\smallidx-1)}(\C)$ of highest
weight $\delta$ (if $\smallidx=1$, $W(t)\equiv1$). To keep the notation uniform, set
$\chi_{\delta}^{G_{\smallidx}}(t_{\pi})=\chi_{\delta}^{Sp_{2(\smallidx-1)}}(t'_{\pi})$
for \quasisplit\ $G_{\smallidx}$.

\begin{remark}
Tamir \cite{T} provided an interpretation of the Casselman-Shalika formula
for a connected reductive \quasisplit\ group $G$. His formula involves a
character of a representation of $\rconj{L}G$ and the 
Satake parameter.
\end{remark}

\subsubsection{Branching rules}\label{subsection:branching rules} Let
$\tau$ be an irreducible unramified representation of $\GL{\bigidx}$ for $\bigidx>1$. We need
the following two formulas for
$\chi_{\delta}^{\GL{\bigidx}}(t_{\tau})$. First, for $\delta\in
P_{\bigidx}\setminus P_{\bigidx-1}$, i.e., for $\delta\in
P_{\bigidx}$ such that $\delta_{\bigidx}>0$,
\begin{align}\label{eq:chi GL for delta with non zero tail}
\chi_{\delta}^{\GL{\bigidx}}(t_{\tau})=(\prod_{i=1}^{\bigidx}t_{\tau}(i)^{\delta_{\bigidx}})\chi_{(\delta_1-\delta_{\bigidx},\ldots,\delta_{\bigidx-1}-\delta_{\bigidx},0)}^{\GL{\bigidx}}(t_{\tau}).
\end{align}
Here $t_{\tau}(i)$ denotes the $i$-th coordinate of the diagonal matrix $t_{\tau}$.
Second, the branching rule from $GL_{\bigidx}(\C)$ to
$GL_{\bigidx-1}(\C)\times GL_1(\C)$,
\begin{lemma}[\cite{GW} p.~350]\label{lemma:branching rule}
Let $\delta\in P_{\bigidx-1}$. Then
\begin{align*}
\chi_{\delta}^{\GL{\bigidx}}(t_{\tau})=\sum_{k=0}^{|\delta|}\sum_{\tilde{\delta}\in C(k,\delta)}\chi_{\tilde{\delta}}^{\GL{\bigidx-1}}(t_{\tau'})t_{\tau}(\bigidx)^k.
\end{align*}
Here $C(k,\delta)\subset P_{\bigidx-1}$ is a finite set (given explicitly in \cite{GW}) depending on $k$ and $\delta$, $C(0,\delta)=\{\delta\}$
and $t_{\tau'}=diag(t_{\tau}(1),\ldots,t_{\tau}(\bigidx-1))$.
\end{lemma}
Since
$(\delta_1-\delta_{\bigidx},\ldots,\delta_{\bigidx-1}-\delta_{\bigidx},0)\in
P_{\bigidx-1}$, when we combine these results we also see that
$t_{\tau}(\bigidx)$ appears in
$\chi_{\delta}^{\GL{\bigidx}}(t_{\tau})$ with only \nonnegative\
powers.

\subsection{The computation for $\smallidx>\bigidx$}\label{subsection:unramified l>n}
Here we prove Theorem~\ref{theorem:unramified computation} for the case $\smallidx>\bigidx$.
Let $W\in\Whittaker{\pi}{\psi_{\gamma}^{-1}}$ and
$f_s\in V(\tau,s)$ be the normalized unramified functions. We have the following claim, 
which follows using the Iwasawa decomposition of $H_{\bigidx}$.
\begin{claim}\label{claim:iwasawa l>=n+1}
Let $\smallidx>\bigidx$. For $\Re(s)>>0$, 
\begin{align*}
\Psi(W,f_s,s)=\sum_{\delta\in
P_{\bigidx}}\chi^{G_{\smallidx}}_{\delta}(t_{\pi})\chi^{\GL{\bigidx}}_{\delta}(t_{\tau})q^{-s|\delta|}.
\end{align*}
Here the measure $dr$ on $R^{\smallidx,\bigidx}$ is normalized by $vol(R^{\smallidx,\bigidx}(\mathcal{O}))=1$.
\end{claim}

According to this claim, the theorem is equivalent to the following
identity of power series in $q^{-s}$,
\begin{align}\label{eq:identity to prove for l=n+1}
L(\tau,Sym^2,2s)\sum_{\delta\in
P_{\bigidx}}\chi^{G_{\smallidx}}_{\delta}(t_{\pi})\chi^{\GL{\bigidx}}_{\delta}(t_{\tau})q^{-s|\delta|}=
L(\pi\times\tau,s).
\end{align}
We prove that if \eqref{eq:identity to prove for l=n+1} holds for
some $(\smallidx,\bigidx)$ with $\smallidx>\bigidx>1$, it also holds
for $(\smallidx,\bigidx-1)$. Indeed, let $\tau'$ be an irreducible unramified (generic, as always)
representation of $\GL{\bigidx-1}$. Construct a representation
$\tau=\cinduced{P_{\bigidx-1,1}}{\GL{\bigidx}}{\tau'\otimes\tau_{\bigidx}}$,
where $\tau_{\bigidx}$ is an unramified character of $\GL{1}$. Then $\tau$ is unramified (and generic, since $\tau'$ is generic) and we choose
$\tau_{\bigidx}$ so that $\tau$ would also be irreducible. Note that $t_{\tau}(\bigidx)=t_{\tau_{\bigidx}}(1)=\tau_{\bigidx}(\varpi)$. Consider
\eqref{eq:identity to prove for l=n+1} for $(\smallidx,\bigidx)$.
The summation over $\delta\in P_{\bigidx}$ can be written as two
sums: the first over $\delta\in P_{\bigidx-1}$, the second over
$\delta\in P_{\bigidx}\setminus P_{\bigidx-1}$. As explained in
Section~\ref{subsection:branching rules}, $t_{\tau}(\bigidx)$
appears in the summation with only \nonnegative\ powers.

According to the definitions,
\begin{align*}
&L(\tau,Sym^2,2s)=\prod_{1\leq
i,j\leq\bigidx}(1-t_{\tau}(i)t_{\tau}(j)q^{-2s})^{-1},\\\notag
&L(\pi\times\tau,s)=L(\pi\times\tau',s)L(\pi\times\tau_{\bigidx},s)=L(\pi\times\tau',s)\prod_{1\leq
i\leq2\smallidx}(1-t_{\pi}(i)t_{\tau}(\bigidx)q^{-s})^{-1}.
\end{align*}
It follows that $t_{\tau}(\bigidx)$ appears in \eqref{eq:identity to
prove for l=n+1} with only \nonnegative\ powers. Since
$\tau_{\bigidx}$ is any unramified character such that $\tau$ is
irreducible, $t_{\tau}(\bigidx)=q^{-z}$ where $z$ is an arbitrary
complex number so that $q^{-z}$ is outside a finite set of values. Therefore we may
regard \eqref{eq:identity to prove for l=n+1} as an identity of
power series in $t_{\tau}(\bigidx)$ and formally substitute $0$ for
$t_{\tau}(\bigidx)$. Equality~\eqref{eq:chi GL for delta with non
zero tail} and Lemma~\ref{lemma:branching rule} imply
\begin{align}\label{branching conclusion GLn}
\frestrict{\chi^{\GL{\bigidx}}_{\delta}(t_{\tau})}{t_{\tau}(\bigidx)=0}=
\begin{cases}\chi^{\GL{\bigidx-1}}_{\delta}(t_{\tau'})&\delta\in
P_{\bigidx-1},\\
0&\delta\in P_{\bigidx}\setminus P_{\bigidx-1}.\end{cases}
\end{align}
Thus we get
\begin{align*}
L(\tau',Sym^2,2s)\sum_{\delta\in
P_{\bigidx-1}}\chi^{G_{\smallidx}}_{\delta}(t_{\pi})\chi^{\GL{\bigidx-1}}_{\delta}(t_{\tau'})q^{-s|\delta|}=L(\pi\times\tau',s),
\end{align*}
which is \eqref{eq:identity to prove for l=n+1} for
$(\smallidx,\bigidx-1)$.

Now the next claim completes the proof of the theorem for all
$\smallidx>\bigidx$.
\begin{claim}\label{claim:unramified identity l=n+1}
Identity~\eqref{eq:identity to prove for l=n+1} holds for all $\bigidx\geq1$ and $\smallidx=\bigidx+1$.
\end{claim}

\begin{proof}[Proof of Claim~\ref{claim:iwasawa l>=n+1}]
The integral is
\begin{align*}
\Psi(W,f_s,s)=\int_{\lmodulo{U_{H_{\bigidx}}}{H_{\bigidx}}}
\int_{R^{\smallidx,\bigidx}} W(rw^{\smallidx,\bigidx}h)f_s(h,1)drdh.
\end{align*}
It is absolutely convergent for $\Re(s)>>0$, i.e., the integral with $|W|$ and $|f_s|$ replacing $W$ and $f_s$ is convergent. Using the Iwasawa decomposition of $H_{\bigidx}$ and according to
our normalization of the measures, the integral equals
\begin{align*}
\int_{A_{\bigidx}}\int_{R^{\smallidx,\bigidx}}
W(rw^{\smallidx,\bigidx}a)f_s(a,1)\delta_{B_{H_{\bigidx}}}^{-1}(a)drda.
\end{align*}
Since $(w^{\smallidx,\bigidx})^{-1}\in K_{H_{\bigidx}}$, $W$ is
right-invariant by $K_{H_{\bigidx}}$ and
$\rconj{(w^{\smallidx,\bigidx})^{-1}}a$ normalizes
$R^{\smallidx,\bigidx}$ changing $dr\mapsto
\absdet{a}^{\bigidx+1-\smallidx}dr$, this equals
\begin{align*}
\int_{A_{\bigidx}}\int_{R^{\smallidx,\bigidx}}
W(\rconj{(w^{\smallidx,\bigidx})^{-1}}ar)f_s(a,1)\absdet{a}^{\bigidx+1-\smallidx}\delta_{B_{H_{\bigidx}}}^{-1}(a)drda.
\end{align*}
Exactly as in
Ginzburg \cite{G} (p.~176, see also 
\cite{Soudry} p.~98) the
$K_{G_{\smallidx}}$-invariance of $W$ implies that for any $a$,
$W(\rconj{(w^{\smallidx,\bigidx})^{-1}}ar)=0$ unless $r\in
R^{\smallidx,\bigidx}(\mathcal{O})<K_{G_{\smallidx}}$. Also note that the function $W_{\tau}(a)=f_s(1,a)$ is the normalized unramified
Whittaker function of $\Whittaker{\tau}{\psi}$ and
$\delta_{B_{H_{\bigidx}}}=\delta_{Q_{\bigidx}}\delta_{B_{\GL{\bigidx}}}$.
Then we get
\begin{align*}
\int_{A_{\bigidx}}
W(\rconj{(w^{\smallidx,\bigidx})^{-1}}a)W_{\tau}(a)\absdet{a}^{\bigidx+\half-\smallidx+s}\delta_{Q_{\bigidx}}^{-\half}(a)\delta_{B_{\GL{\bigidx}}}^{-1}(a)da.
\end{align*}
Here we used the normalization $vol(R^{\smallidx,\bigidx}(\mathcal{O}))=1$.
Since $\smallidx>\bigidx$,
$W(\rconj{(w^{\smallidx,\bigidx})^{-1}}a)$ vanishes unless $a\in
A_{\bigidx}^-$ and $|a_{\bigidx}|\leq1$. Hence $\delta(a)\in
P_{\bigidx}$ and we can write the integral as a sum over $\delta\in
P_{\bigidx}$. Using the formulas for $W,W_{\tau}$ and the fact that
on $A_{\smallidx-1}$,
$\delta_{B_{G_{\smallidx}}}=\absdet{}^{2\smallidx-2\bigidx-1}\delta_{Q_{\bigidx}}\delta_{B_{\GL{\bigidx}}}$,
we get
\begin{equation*}
\sum_{\delta\in
P_{\bigidx}}\chi^{G_{\smallidx}}_{\delta}(t_{\pi})\chi^{\GL{\bigidx}}_{\delta}(t_{\tau})q^{-s|\delta|}.\qedhere
\end{equation*}
\end{proof} 

\begin{proof}[Proof of Claim~\ref{claim:unramified identity l=n+1}] 
First assume that $G_{\bigidx+1}$ is \quasisplit. In the case
$\bigidx=1$ one proves \eqref{eq:identity to prove for l=n+1}
directly - the summation is over the integers $\delta\geq0$,
$\tau$ is a character so
$\chi^{\GL{1}}_{\delta}(t_{\tau})=t_{\tau}(1)^{\delta}$ and
$\chi^{G_{2}}_{\delta}(t_{\pi})=\chi^{Sp_{2}}_{\delta}(t'_{\pi})$ is the $\delta$-th complete
symmetric polynomial in $t'_{\pi}(1),t'_{\pi}(2)$ (see \cite{FH} p.
$406$).

Let $\bigidx>1$. According to \cite{GPS} (in the App.),
\begin{align*}
L(\tau,{\wedge}^2,2s)\sum_{\delta\in
P_{\bigidx}}\chi^{Sp_{2\bigidx}}_{\delta}(t'_{\pi})\chi^{\GL{\bigidx}}_{\delta}(t_{\tau})q^{-s|\delta|}=
\det(1-(t'_{\pi}\otimes t_{\tau})q^{-s})^{-1}.
\end{align*}
Note that the argument in \textit{loc. cit.} is based on the results of
Ton-That (\cite{TT1}) which apply when $\bigidx\geq2$ (i.e., the
Symplectic group is of rank at least $2$). We immediately deduce
\eqref{eq:identity to prove for l=n+1} ($\chi_{\delta}^{G_{\bigidx+1}}(t_{\pi})=\chi_{\delta}^{Sp_{2\bigidx}}(t'_{\pi})$
for \quasisplit\ $G_{\bigidx+1}$, see Section~\ref{subsection:casselman shalika formula}).

The case of split $G_{\bigidx+1}$ follows from Ton-That
\cite{TT2}, see Ginzburg \cite{G}.
\end{proof} 
\subsection{The computation for $\smallidx\leq\bigidx$}\label{subsection:unramified l<=n}
We prove Theorem~\ref{theorem:unramified computation} when $\smallidx\leq\bigidx$.
In the case $\smallidx=1$ it follows from Bump, Friedberg and Furusawa
\cite{BFF}. Specifically, the integral with unramified data equals the value of
their unramified Bessel function at the identity, multiplied by $\prod_{1\leq i<j\leq \bigidx}(1-t_{\tau}(i)t_{\tau}(j)^{-1}q^{-1})^{-1}$ (they used a different normalization).

We now assume $\smallidx>1$. Then we may further assume $\pi=\cinduced{\overline{P_{\smallidx-1}}}{G_{\smallidx}}{\sigma\otimes\pi'}$,
where $\sigma$ is an irreducible unramified representation of $\GL{\smallidx-1}$
realized in $\Whittaker{\sigma}{\psi^{-1}}$ and $\pi'$ is an
unramified character of $G_1$. Let $W\in\Whittaker{\pi}{\psi_{\gamma}^{-1}}$ and
$f_s\in V(\tau,s)$ be the normalized unramified vectors.

Take an element $\varphi$ from the space of $\pi$ such that $\varphi$ is unramified and normalized by
$\varphi(I_{\smallidx},I_{\smallidx-1},I_2)=1$. The Whittaker function $W_{\varphi}\in\Whittaker{\pi}{\psi_{\gamma}^{-1}}$ corresponding to $\varphi$ is given by
\begin{align*}
W_{\varphi}(g)=\int_{V_{\smallidx-1}}\varphi(vg,I_{\smallidx-1},I_2)\psi_{\gamma}(v)dv,
\end{align*}
where the integral is defined as a principal value and $\psi_{\gamma}$ denotes the character of $U_{G_{\smallidx}}$.
Then $W_{\varphi}$ is unramified and the results of Casselman and Shalika \cite{CS2} imply in particular that
$W_{\varphi}(1)\ne0$ whence $W_{\varphi}(1)^{-1}W_{\varphi}=W$.
Let $\zeta\in\C$ and extend $\varphi$ to
$\varphi_{\zeta}\in V_{\overline{P_{\smallidx-1}}}^{G_{\smallidx}}(\sigma\otimes\pi',-\zeta+\half)$
using the Iwasawa decomposition (see Section~\ref{subsection:sections}). Then $\varphi_{\zeta}$ is unramified and
$\varphi_{\zeta}(1,1,1)=1$. We define
$W_{\varphi_{\zeta}}\in\Whittaker{\cinduced{\overline{P_{\smallidx-1}}}{G_{\smallidx}}{(\sigma\otimes\pi')\alpha^{-\zeta+\half}}}{\psi_{\gamma}^{-1}}$
by the integral above with $\varphi_{\zeta}$ replacing $\varphi$. The definition implies $W_{\varphi_{0}}=W_{\varphi}$.
According to the
Casselman-Shalika formula \cite{CS2},
\begin{align}\label{eq:cs result for W phi zeta}
W_{\varphi_{\zeta}}(1)=L(\pi'\times\sigma^*,\zeta+1)^{-1}L(\sigma^*,\wedge^2,2\zeta+1)^{-1}.
\end{align}
In particular $W_{\varphi_{\zeta}}(1)\ne0$ for $\Re(\zeta)>>0$ and for $\zeta=0$ (because $W_{\varphi}(1)\ne0$). 
Then $W_{\varphi_{\zeta}}'=W_{\varphi_{\zeta}}(1)^{-1}W_{\varphi_{\zeta}}$
is the normalized unramified Whittaker function, defined whenever $\zeta$ is not a zero of the \rhs\ of \eqref{eq:cs result for W phi zeta}. We will compute
$\Psi(W_{\varphi_{\zeta}}',f_s,s)$ for $\Re(\zeta)>>0$, then derive
the result by taking $\zeta=0$.

In Section~\ref{subsection:1st var l<n} we show the following
equality (see \eqref{eq:1st var l<n before eq for pi prime} with $k=\smallidx-1$),
\begin{align}\label{eq:unramified 1st var l<n after func eq sigma}
&\omega_{\sigma}(-1)^{\bigidx-1}\gamma(\sigma\times\tau,\psi,s-\zeta)\Psi(W_{\varphi_{\zeta}},f_s,s)\\\notag
&=\int_{\lmodulo{\overline{V_{\smallidx-1}}Z_{\smallidx-1}G_1}{G_{\smallidx}}}
\int_{\Mat{\smallidx-1\times \bigidx-\smallidx}}\int_{G_1}
\int_{R_{1,\bigidx}}\\\notag&\quad\varphi_{\zeta}(g,1,g')
f_s(w_{1,\bigidx}r'(\rconj{b_{\bigidx,\smallidx-1}}g')w(\rconj{w_{\smallidx,\bigidx}}m)g,1)\psi_{\gamma}(r')dr'dg'dmdg.
\end{align}
Here
\begin{align*}
&b_{c,{\smallidx-1}}=diag(I_c,(-1)^{\smallidx-1},I_c)\qquad (c\geq1),\\\notag
&w=\left(\begin{array}{ccccc}0&0&0&\gamma^{-1}I_{\smallidx-1}&0\\
I_{\bigidx-\smallidx}&0&0&0&0\\
0&0&b_{1,{\smallidx-1}}&0&0\\
0&0&0&0&I_{\bigidx-\smallidx}\\
0&\gamma I_{\smallidx-1}&0&0&0\end{array}\right),\\\notag
&\rconj{w_{\smallidx,\bigidx}}m=\left(\begin{array}{ccccc}
I_{\bigidx-\smallidx}&&&\gamma^{-1}m'\\
&I_{\smallidx-1}&&&\gamma^{-1}m\\
&&I_3\\
&&&I_{\smallidx-1}\\
&&&&I_{\bigidx-\smallidx}\\
\end{array}\right).
\end{align*}
The measure $dm$ is normalized by $vol(\MatF{\smallidx-1\times \bigidx-\smallidx}{\mathcal{O}})=1$.
Equality~\eqref{eq:unramified 1st var l<n after func eq sigma} is in
the sense of meromorphic continuation. Specifically, for any fixed $\zeta$ with $\Re(\zeta)>>0$ both sides have
meromorphic continuations to rational functions in $\C(q^{-s})$ and
these continuations are equal.
The integral $\Psi(W_{\varphi_{\zeta}},f_s,s)$ (resp. the \rhs\ of \eqref{eq:unramified 1st var l<n after func eq sigma}) is
absolutely convergent if it is convergent when we replace
$W_{\varphi_{\zeta}},f_s$ (resp. $\varphi_{\zeta},f_s$) with $|W_{\varphi_{\zeta}}|,|f_s|$ (resp. $|\varphi_{\zeta}|,|f_s|$) and drop the characters.
For brevity, set $\xi_{\tau,s}=\cinduced{Q_{\bigidx}}{H_{\bigidx}}{\tau\alpha^s}$.
\begin{claim}\label{claim:iwasawa l<=n}
The \rhs\ of \eqref{eq:unramified 1st var l<n after func eq sigma}
equals, in its domain of absolute convergence,
\begin{align*}
\frac{L(\pi'\times\tau,s)}{L(\tau,Sym^2,2s)}\sum_{\delta\in
P_{\smallidx-1}}\chi_{\delta}^{\GL{\smallidx-1}}(t_{\sigma^*})\Lambda_{\delta}(t_{\xi_{\tau,s}},t_{\pi'})q^{-(\zeta+\half)|\delta|}.
\end{align*}
Here $\Lambda_{\delta}(t_{\xi_{\tau,s}},t_{\pi'})$ is a rational function
in the Satake parameters of $\xi_{\tau,s}$ and $\pi'$. Its exact
form will not be needed.
\end{claim}
The unramified computation of Furusawa (App. of \cite{BFF}) will be used to prove,
\begin{claim}\label{claim:BFF unramified computation}
For each $1\leq k<\bigidx$ and irreducible unramified representation $\eta$ of
$\GL{k}$,
\begin{align*}
\sum_{\delta\in
P_{k}}\chi_{\delta}^{\GL{k}}(t_{\eta})\Lambda_{\delta}(t_{\xi_{\tau,s}},t_{\pi'})q^{-(\zeta+\half)|\delta|}
=\frac{L(\xi_{\tau,s}\times\eta,\zeta+\half)}{L(\pi'\times\eta,\zeta+1)L(\eta,\wedge^2,2\zeta+1)}.
\end{align*}
\end{claim}
We take the domain $D\subset\C\times\C$ of absolute convergence of $\Psi(W_{\varphi_{\zeta}}',f_s,s)$ of the form
$D=\setof{\zeta,s\in\C}{0<<\Re(\zeta)<A,\Re(s)>s_0}$, where $A$ and $s_0$ are constants depending only on
$\sigma,\pi'$ and $\tau$. 

Since $1<\smallidx\leq\bigidx$, $1\leq\smallidx-1<\bigidx$ and we may apply
Claim~\ref{claim:BFF unramified computation} with $k=\smallidx-1$
and $\eta=\sigma^*$. Applying Claims~\ref{claim:iwasawa l<=n},
\ref{claim:BFF unramified computation} 
and using \eqref{eq:cs result for W phi
zeta} we see that in $D$,
\begin{align*}
&\Psi(W_{\varphi_{\zeta}}',f_s,s)
=\omega_{\sigma}(-1)^{\bigidx-1}\gamma(\sigma\times\tau,\psi,s-\zeta)^{-1}
\frac{L(\pi'\times\tau,s)}{L(\tau,Sym^2,2s)}L(\xi_{\tau,s}\times\sigma^*,\zeta+\half).
\end{align*}
We have $\omega_{\sigma}(-1)^{\bigidx-1}=1$,
\begin{align*}
&L(\xi_{\tau,s}\times\sigma^*,\zeta+\half)=L(\tau\times\sigma^*,\zeta+\half+s-\half)L(\tau^*\times\sigma^*,\zeta+\half-(s-\half)),\\\notag
&\gamma(\sigma\times\tau,\psi^{-1},s-\zeta)=L(\sigma^*\times\tau^*,1-s+\zeta)L(\sigma\times\tau,s-\zeta)^{-1}
\end{align*}
and it follows that in $D$,
\begin{align*}
&\Psi(W_{\varphi_{\zeta}}',f_s,s)
=\frac{L(\sigma\times\tau,s-\zeta)L(\pi'\times\tau,s)L(\sigma^*\times\tau,s+\zeta)}{L(\tau,Sym^2,2s)}.
\end{align*}

Given $A$, there is a constant $s_1$ depending only on $\sigma,\pi',\tau$ and $A$ such that
the integral $\Psi(W_{\varphi_{\zeta}}',f_s,s)$ has a meromorphic continuation
$Q\in\C(q^{-\zeta},q^{-s})$ satisfying $\Psi(W_{\varphi_{\zeta}}',f_s,s)=Q(q^{-\zeta},q^{-s})$ in $D'=\setof{\zeta,s\in\C}{0\leq\Re(\zeta)<A,\Re(s)>s_1}$.
This can be proved using Bernstein's continuation principle (in \cite{Banks}),
see Section~\ref{subsection:meromorphic continuation} and \cite{BFF} (Section~5), or by a more direct approach (see Claim~\ref{claim:gamma is rational in s and zeta}).

Therefore in $D\cap D'$,
\begin{align*}
&Q(q^{-\zeta},q^{-s})
=\frac{L(\sigma\times\tau,s-\zeta)L(\pi'\times\tau,s)L(\sigma^*\times\tau,s+\zeta)}{L(\tau,Sym^2,2s)}.
\end{align*}
Both sides of this equality belong to $\C(q^{-\zeta},q^{-s})$ whence it holds for all $\zeta,s\in\C$.
Since the \rhs\ is defined for $\zeta=0$, so is $Q(q^{-\zeta},q^{-s})$. Moreover, for $\Re(s)>s_1$ 
we have $Q(q^{-0},q^{-s})=\Psi(W,f_s,s)$, because $W_{\varphi_0}'=W$. Thus for $\Re(s)>s_1$,
\begin{align*}
&\Psi(W,f_s,s)
=\frac{L(\sigma\times\tau,s)L(\pi'\times\tau,s)L(\sigma^*\times\tau,s)}{L(\tau,Sym^2,2s)}
=\frac{L(\pi\times\tau,s)}{L(\tau,Sym^2,2s)}.
\end{align*}
\begin{proof}[Proof of Claim~\ref{claim:iwasawa l<=n}] 
Since for $g\in G_{\smallidx}$ and $g'\in G_1$, $\varphi_{\zeta}(g,1,g')=\varphi_{\zeta}(g,1,1)\pi'(g')$, the
inner $dr'dg'$-integral on the \rhs\ of \eqref{eq:unramified 1st var
l<n after func eq sigma} constitutes a local integral for
$G_1\times\GL{\bigidx}$. Note that the conjugation of $g'$ by $b_{\bigidx,\smallidx-1}$ changes the embedding of $G_1$ in $H_{\bigidx}$, but $G_1$ still normalizes $N_{\bigidx-1}$ and stabilizes $\psi_{\gamma}$ (see Section~\ref{subsection:Twisting the embedding}). Then the \rhs\ of \eqref{eq:unramified 1st var l<n after
func eq sigma} equals
\begin{align*}
\int_{\lmodulo{\overline{V_{\smallidx-1}}Z_{\smallidx-1}G_1}{G_{\smallidx}}}
\int_{\Mat{\smallidx-1\times \bigidx-\smallidx}}
\varphi_{\zeta}(g,1,1)\Psi(\pi',w(\rconj{w_{\smallidx,\bigidx}}m)g\cdot
f_s,s)dmdg.
\end{align*}
Here $\pi'$ stands for the Whittaker function ($\pi'$ is a character). Apply the the Iwasawa
decomposition of $G_{\smallidx}$ with respect to the Borel subgroup
$\widehat{B}$ whose unipotent radical equals $\overline{V_{\smallidx-1}}\rtimes
Z_{\smallidx-1}$. We obtain
\begin{align*}
\int_{A_{\smallidx-1}}\int_{\Mat{\smallidx-1\times
\bigidx-\smallidx}}
\varphi_{\zeta}(a,1,1)\Psi(\pi',w(\rconj{w_{\smallidx,\bigidx}}m)a\cdot
f_s,s)\delta_{\widehat{B}}^{-1}(a)dmda.
\end{align*}
The function $W_{\sigma}(a)=\varphi_{\zeta}(1,a,1)$ is the normalized unramified
Whittaker function of $\Whittaker{\sigma}{\psi^{-1}}$.
Note that
$\varphi_{\zeta}(a,1,1)=\absdet{a}^{-\half\smallidx-\zeta}W_{\sigma}(a)$,
$a$ normalizes the subgroup of elements of the form $\rconj{w_{\smallidx,\bigidx}}m$ changing
$dm\mapsto\absdet{a}^{\bigidx-\smallidx}dm$ and
$\rconj{w^{-1}}a=a^*$. The integral equals
\begin{align*}
\int_{A_{\smallidx-1}}\int_{\Mat{\smallidx-1\times
\bigidx-\smallidx}}
W_{\sigma}(a)\Psi(\pi',a^*w(\rconj{w_{\smallidx,\bigidx}}m)\cdot
f_s,s)\absdet{a}^{\bigidx-\frac32\smallidx-\zeta}\delta_{\widehat{B}}^{-1}(a)dmda.
\end{align*}
Next we show that the $dm$-integration can be discarded.
Put $m_1=\rconj{w^{-1}}(\rconj{w_{\smallidx,\bigidx}}m)$. Then
$a^*m_1$ is the image in $M_{\bigidx}$ of
\begin{align*}
\left(\begin{array}{ccc}a^*\\m'&I_{\bigidx-\smallidx}\\&&1\end{array}\right)\in\GL{\bigidx}.
\end{align*}
Since $f_s$ is right-invariant by $K_{H_{\bigidx}}$, for all $k\in
K_{H_\bigidx}$ and $h\in H_{\bigidx}$,
\begin{align}\label{eq:invariancy by K}
&\Psi(\pi',hk\cdot f_s,s)=\Psi(\pi',h\cdot f_s,s).
\end{align}
In particular
\begin{align*}
\Psi(\pi',a^*w(\rconj{w_{\smallidx,\bigidx}}m)\cdot
f_s,s)=\Psi(\pi',a^*m_1\cdot f_s,s).
\end{align*}
Next observe that for all $z\in
Z_{\bigidx-1}<M_{\bigidx-1}$,
\begin{align}\label{eq:translation by Z}
&\Psi(\pi',z\cdot
f_s,s)=\psi(z)\Psi(\pi',f_s,s).
\end{align}
This follows from \eqref{eq:global
integral in homspace l <= n} or by a direct verification.
Now \eqref{eq:invariancy by K} and \eqref{eq:translation by Z} imply, as shown by Ginzburg \cite{G} (p.~176, see 
\cite{Soudry} p.~98), that for any
$a\in A_{\smallidx-1}$ the function $m\mapsto
\Psi(\pi',a^*m_1\cdot f_s,s)$
vanishes unless $m\in\MatF{\smallidx-1\times
\bigidx-\smallidx}{\mathcal{O}}$. Since $f_s$ is $K_{H_{\bigidx}}$-invariant and $vol(\MatF{\smallidx-1\times \bigidx-\smallidx}{\mathcal{O}})=1$, the integral becomes
\begin{align*}
\int_{A_{\smallidx-1}}W_{\sigma}(a)\Psi(\pi',a^*\cdot
f_s,s)\absdet{a}^{\bigidx-\frac32\smallidx-\zeta}\delta_{\widehat{B}}^{-1}(a)da.
\end{align*}
According to the definition of $\sigma^*$,
$W_{\sigma}(a)=W_{\sigma^*}(a^*)$, where $W_{\sigma^*}$ is the
normalized unramified element of $\Whittaker{\sigma^*}{\psi}$.
Using this we have
\begin{align*}
\int_{A_{\smallidx-1}}W_{\sigma^*}(a^*)\Psi(\pi',a^*\cdot
f_s,s)\absdet{a}^{\bigidx-\frac32\smallidx-\zeta}\delta_{\widehat{B}}^{-1}(a)da.
\end{align*}
The function on $H_{\bigidx}$ defined by $h\mapsto\Psi(\pi',h\cdot
f_s,s)$ is the (unramified) Bessel function of Bump, Friedberg and Furusawa
\cite{BFF}. According to their results, $\Psi(\pi',a^*\cdot
f_s,s)=0$ unless $|a^*_i|\leq1$ for all $1\leq
i\leq\smallidx-1$. In this case they proved a Casselman-Shalika type formula for
the integral, in the form
\begin{align*}
\Psi(\pi',a^*\cdot
f_s,s)=\frac{L(\pi'\times\tau,s)}{L(\tau,Sym^2,2s)}\delta_{B_{H_{\bigidx}}}^{\half}(a^*)\Lambda_{\delta(a^*)}(t_{\xi_{\tau,s}},t_{\pi'}).
\end{align*}
Also $W_{\sigma^*}$ vanishes unless $a^*\in A_{\smallidx-1}^-$.
Therefore $\delta^*=\delta(a^*)\in P_{\smallidx-1}$ ($|a^*_{\smallidx-1}|\leq1$). Also use the following calculations:
\begin{align*}
&\delta_{\widehat{B}}(a)=\delta_{B_{\GL{\smallidx-1}}}(a)\absdet{a}^{-\smallidx}=q^{\sum_{i=1}^{\smallidx-1}\vartheta(a_i)(-\smallidx+2i)}\absdet{a}^{-\smallidx},\\\notag
&\delta_{B_{\GL{j}}}^{\half}(a^*)=q^{\sum_{i=1}^{\smallidx-1}\vartheta(a_i)(\half j+\half-\smallidx+i)}\qquad\forall j\geq\smallidx-1.
\end{align*}
The integral becomes
\begin{align*}
\frac{L(\pi'\times\tau,s)}{L(\tau,Sym^2,2s)} \sum_{\delta^*\in
P_{\smallidx-1}}\chi_{\delta^*}^{\GL{\smallidx-1}}(t_{\sigma^*})\Lambda_{\delta^*}(t_{\xi_{\tau,s}},t_{\pi'})q^{-(\zeta+\half)|\delta^*|}.
\end{align*}
The claim follows when we replace $\delta^*$ with $\delta$.
\end{proof}  

\begin{proof}[Proof of Claim~\ref{claim:BFF unramified computation}] 
In the appendix of \cite{BFF}, Furusawa presented a new
Rankin-Selberg integral for $SO_{2\bigidx+1}\times\GL{\bigidx}$. Let
$\eta'$ be an irreducible unramified representation of $\GL{\bigidx}$. Then the
unramified computation of Furusawa yields the following identity,
\begin{align}\label{eq:local identity integral as sum BFF}
\sum_{\delta\in
P_{\bigidx-1}}\chi_{\delta}^{\GL{\bigidx}}(t_{\eta'})\Lambda_{\delta}(t_{\xi_{\tau,s}},t_{\pi'})q^{-(\zeta+\half)|\delta|}
&+ Q(q)\sum_{\delta\in P_{\bigidx}\setminus
P_{\bigidx-1}}\chi_{\delta}^{\GL{\bigidx}}(t_{\eta'})\Lambda_{\delta}(t_{\xi_{\tau,s}},t_{\pi'})q^{-(\zeta+\half)|\delta|}\\\notag
=\frac{L(\xi_{\tau,s}\times\eta',\zeta+\half)}{L(\pi'\times\eta',\zeta+1)L(\eta',\wedge^2,2\zeta+1)}.
\end{align}
Here $Q(q)=1-q^{-1}$ if $G_1$ is split ($\pi'$ is an unramified character of $G_1$)
and $Q(q)=1+q^{-1}$ otherwise.

For an irreducible unramified representation $\eta$ of $\GL{k}$, construct a
representation
$\eta'=\cinduced{P_{k,1,\ldots,1}}{\GL{\bigidx}}{\eta,\eta'_{k+1},\ldots,\eta'_{\bigidx}}$,
for unramified characters $\eta'_{k+1},\ldots,\eta'_{\bigidx}$ such
that $\eta'$ is irreducible and unramified. Then $t_{\eta'}(\bigidx)=t_{\eta'_{\bigidx}}(1)$. As in Section~\ref{subsection:unramified
l>n} 
we can put $t_{\eta'}(\bigidx)=0$ in \eqref{eq:local identity integral as sum BFF} and use
\eqref{branching conclusion GLn} to obtain the result for $k=\bigidx-1$. Setting $t_{\eta'}(\bigidx)=\ldots=t_{\eta'}(k+1)=0$, the result follows for all $1\leq k<\bigidx$.
\end{proof} 

%% file: chapter_uniqueness.tex
\newtheorem{theorem}{Theorem}[section]
\newtheorem{proposition}{Proposition}[section]
\newtheorem{corollary}{Corollary}[section]
\newtheorem{lemma}{Lemma}[section]
\newtheorem{claim}{Claim}[section]
\theoremstyle{remark}
\newtheorem{remark}{Remark}[section]
\newtheorem{example}{Example}[section]
\theoremstyle{definition}
\newtheorem{definition}{Definition}[section]
\numberwithin{equation}{section}
\newcommand{\chapter}{\section} 
\input{thesis_notations}
\end{comment}

\newcommand{\compinduced}[3]{ind_{#1}^{#2}({#3})} 

\chapter{Uniqueness properties}\label{chapter:uniqueness}
As we have seen in Chapter~\ref{chapter:the integrals}, the global
integrals can be regarded as bilinear forms satisfying certain
equivariance properties: \eqref{eq:global integral in homspace l <=
n} when $\smallidx\leq\bigidx$ and \eqref{eq:global integral in
homspace l
> n} if $\smallidx>\bigidx$. In Sections~\ref{subsection:meromorphic
continuation} and \ref{subsection:the gamma factor} we will explain
how to view the local integrals as bilinear forms satisfying similar
properties, first in a domain of absolute convergence, then by
meromorphic continuation in the whole plane except at the poles.

Here we prove Theorem~\ref{theorem:uniqueness}: the local versions
of the spaces of bilinear forms are at most one-dimensional, for
almost all values of $q^{-s}$. We split the proof into two
propositions - for $\smallidx\leq\bigidx$ and $\smallidx>\bigidx$.
Their proofs follow the line of arguments of Soudry \cite{Soudry}
(Section~8).

Note that when the representations are irreducible,
Theorem~\ref{theorem:uniqueness} follows from the results of Aizenbud et al. \cite{AGRS} (see 
\cite{GGP} p. 57) and of M{\oe}glin and Waldspurger \cite{MW}. In two
particular cases $\smallidx=\bigidx,\bigidx+1$ it was already proved by Gelbart and
Piatetski-Shapiro \cite{GPS} (again, for irreducible representations).

Our notation in this chapter is local (\nonarchimedean).
Let $\pi$ be a representation of $G_{\smallidx}$ and $\tau$ be a representation of $\GL{\bigidx}$.
We stress that here the only assumptions on the representations
are the ones stated in Section~\ref{subsection:notation for reps} - that they are smooth, admissible, finitely
generated and generic, and $\tau$ has a central character. In particular, we do not assume that the representations are realized in Whittaker models.

As described in the introduction (see Section~\ref{section:main
results}), the main application of the uniqueness properties is the
existence of a functional equation - equality~\eqref{eq:gamma def}.
Gelbart and Piatetski-Shapiro \cite{GPS} and Soudry \cite{Soudry}
used similar uniqueness results in order to derive meromorphic
continuation. Although we can also follow this approach here (see
Section~\ref{subsection:meromorphic continuation}), our main proof
of meromorphic continuation will not depend on
Theorem~\ref{theorem:uniqueness}.

\section{The case $\smallidx\leq\bigidx$
}\label{section:uniqneness l<=n}
The space of bilinear forms on $\pi\times\cinduced{Q_{\bigidx}}{H_{\bigidx}}{\tau\alpha^s}$ satisfying (the local version of) \eqref{eq:global integral in homspace l <= n} is equivalent to the space
$Bil_{G_{\smallidx}}(\pi,(\cinduced{Q_{\bigidx}}{H_{\bigidx}}{\tau\alpha^s})_{N_{\bigidx-\smallidx},\psi_{\gamma}^{-1}})$, where $Bil_{G_{\smallidx}}(\cdot,\cdot)$ denotes a space of $G_{\smallidx}$-equivariant bilinear forms and $(\cdots)_{N_{\bigidx-\smallidx},\psi_{\gamma}^{-1}}$
denotes the Jacquet module with respect to $N_{\bigidx-\smallidx}$ and $\psi_{\gamma}^{-1}$ (see \cite{BZ2} 1.8 (b)
for the definition).
\begin{proposition}\label{proposition:one dimensionality l <= n}
Except for a finite set of values of $q^{-s}$, the complex
vector-space
$Bil_{G_{\smallidx}}(\pi,(\cinduced{Q_{\bigidx}}{H_{\bigidx}}{\tau\alpha^s})_{N_{\bigidx-\smallidx},\psi_{\gamma}^{-1}})$
is at most one-dimensional.
\end{proposition}
\begin{proof}[Proof of Proposition~\ref{proposition:one dimensionality l <= n}] 
We have the following isomorphisms,
\begin{align*}
Bil_{G_{\smallidx}}(\pi,(\cinduced{Q_{\bigidx}}{H_{\bigidx}}{\tau\alpha^s})_{N_{\bigidx-\smallidx},\psi_{\gamma}^{-1}})
&\isomorphic\homspace{G_{\smallidx}}{(\cinduced{Q_{\bigidx}}{H_{\bigidx}}{\tau\alpha^s})_{N_{\bigidx-\smallidx},\psi_{\gamma}^{-1}}}{\dualrep{\pi}}\\\notag
&\isomorphic\homspace{H_{\bigidx}}{\cinduced{Q_{\bigidx}}{H_{\bigidx}}{\tau\alpha^s}}
{\cinduced{N_{\bigidx-\smallidx}G_{\smallidx}}{H_{\bigidx}}{\dualrep{\pi}\otimes\psi_{\gamma}^{-1}}}\\\notag
&\isomorphic
Bil_{H_{\bigidx}}(\compinduced{N_{\bigidx-\smallidx}G_{\smallidx}}{H_{\bigidx}}{\pi\otimes\psi_{\gamma}},\cinduced{Q_{\bigidx}}{H_{\bigidx}}{\tau\alpha^s}).
\end{align*}
Here $\compinduced{}{}{\cdots}$ refers to the compact normalized
induction, $\dualrep{\pi}\otimes\psi_{\gamma}^{-1}$ is the
representation of $N_{\bigidx-\smallidx}G_{\smallidx}$ on the space
of $\dualrep{\pi}$ defined by
$vg\mapsto\psi_{\gamma}^{-1}(v)\dualrep{\pi}(g)$ ($v\in
N_{\bigidx-\smallidx}$ and $g\in G_{\smallidx}$, recall that
$G_{\smallidx}$ normalizes $N_{\bigidx-\smallidx}$ and stabilizes
$\psi_{\gamma}$). The second and third isomorphisms follow from
\cite{BZ2} (1.9 (b), (c) and (d)).

For $w\in\rmodulo{\lmodulo{Q_{\bigidx}}{H_{\bigidx}}}{N_{\bigidx-\smallidx}G_{\smallidx}}$,
denote
\begin{align*}
Hom(w)=\homspace{(N_{\bigidx-\smallidx}G_{\smallidx})^{w^{-1}}}{\rconj{w}(\pi\otimes\psi_{\gamma})\otimes{\tau\alpha^s}}
{\delta
},
\end{align*}
where $(N_{\bigidx-\smallidx}G_{\smallidx})^{w^{-1}}=
\rconj{w^{-1}}(N_{\bigidx-\smallidx}G_{\smallidx})\cap Q_{\bigidx}$;
for a representation $\theta$ of
$N_{\bigidx-\smallidx}G_{\smallidx}$, $\rconj{w}\theta$ is the
representation of $(N_{\bigidx-\smallidx}G_{\smallidx})^{w^{-1}}$ on
the space of $\theta$ given by
$\rconj{w}\theta(x)=\theta(\rconj{w}x)$;
$\delta(x)=\delta_{\mathcal{C}(w)}(x,\rconj{w}x)\cdot\delta_{Q_{\bigidx}}^{-\half}(x)$
is a representation of
$(N_{\bigidx-\smallidx}G_{\smallidx})^{w^{-1}}$ on $\C$, where
$\mathcal{C}(w)=\setof{(x,\rconj{w}x)}{x\in
(N_{\bigidx-\smallidx}G_{\smallidx})^{w^{-1}}}<Q_{\bigidx}\times
N_{\bigidx-\smallidx}G_{\smallidx}$ and $\delta_{\mathcal{C}(w)}$ 
is the modulus character of $\mathcal{C}(w)$.


According to the Bruhat theory (see e.g. \cite{Silb} Theorems~1.9.4 and 1.9.5),
\begin{align*}
Bil_{H_{\bigidx}}(\compinduced{N_{\bigidx-\smallidx}G_{\smallidx}}{H_{\bigidx}}{\pi\otimes\psi_{\gamma}},\cinduced{Q_{\bigidx}}{H_{\bigidx}}{\tau\alpha^s})\subset \bigoplus_{w\in\rmodulo{\lmodulo{Q_{\bigidx}}{H_{\bigidx}}}{N_{\bigidx-\smallidx}G_{\smallidx}}}Hom(w),
\end{align*}
in the sense of semi-simplification.
In the proof of Proposition~\ref{propo:basic global identity l <= n}
we defined a set of representatives $\{w_r\}$ and for each $r$ we
described a set $\mathcal{A}(r)$ (see \eqref{set:representatives for
second filtration}). We take a set of representatives for
$\rmodulo{\lmodulo{Q_{\bigidx}}{H_{\bigidx}}}{N_{\bigidx-\smallidx}G_{\smallidx}}$
in the form $w=w_r\eta$ where $0\leq r\leq\bigidx-\smallidx$ and
$\eta\in\mathcal{A}(r)$. Set $\eta_0=I_{2\smallidx+1}$. Analogously
to Claim~\ref{claim:unfolding of E_f to short sum}, we will prove
that for $r>0$, $H(w)=0$, and for $r=0$, if $\eta\ne\eta_0$ then
except for finitely many values of $q^{-s}$, $H(w)=0$.

Indeed, if $r>0$, the proof of the claim implies
$\frestrict{\psi_{\gamma}}{\rconj{w}U_{\bigidx}\cap N_{\bigidx-\smallidx}}\nequiv1$. Let $y\in \rconj{w}U_{\bigidx}\cap N_{\bigidx-\smallidx}$ be such that $\psi_{\gamma}(y)\ne1$. Write $y=\rconj{w}u\in N_{\bigidx-\smallidx}$ with $u\in U_{\bigidx}$. Then
$x=\rconj{w^{-1}}y\in\rconj{w^{-1}}N_{\bigidx-\smallidx}\cap U_{\bigidx}<
(N_{\bigidx-\smallidx}G_{\smallidx})^{w^{-1}}\cap U_{\bigidx}$. Now for any $f\in H(w)$, $v_1$ in the space of $\pi$ and $v_2$ in the space of $\tau$,
\begin{align*}
f(v_1\otimes v_2)=
f((\rconj{w}(\pi\otimes\psi_{\gamma})\otimes\tau\alpha^s)(x)v_1\otimes v_2)=
f(\psi_{\gamma}(y)v_1\otimes v_2)=\psi_{\gamma}(y)f(v_1\otimes v_2).
\end{align*}
Hence $f\equiv0$. This shows $H(w)=0$ for $r>0$.

Assume $r=0$, $\eta\ne\eta_0$ (in which case $G_{\smallidx}$ is
necessarily split). Let $y=\rconj{\eta}q\in G_{\smallidx}$ be with
$q\in Q_{\smallidx}'$ ($Q_{\smallidx}'$ is given by \eqref{eq:Qn
wr}). Then $x=\rconj{w^{-1}}y\in
(N_{\bigidx-\smallidx}G_{\smallidx})^{w^{-1}}$ (because
$\rconj{w_0}Q_{\smallidx}'=Q_{\smallidx}'<Q_{\bigidx}$). In the above
notation,
\begin{align*}
\delta(x)
f(v_1\otimes v_2)=
f((\rconj{w}(\pi\otimes\psi_{\gamma})\otimes\tau\alpha^s)(x)v_1\otimes
v_2)= f(\pi(y)v_1\otimes \tau\alpha^s(x)v_2).
\end{align*}
Therefore if we put
$Q_{\smallidx}'^{\eta}=\rconj{\eta}Q_{\smallidx}'\cap
G_{\smallidx}$, 
\begin{align*}
f\in
Bil_{Q_{\smallidx}'^{\eta}}(
\pi,\rconj{w^{-1}}(\tau\alpha^{s
}\delta^{-1}))
&\isomorphic\homspace{Q_{\smallidx}'^{\eta}}{
\pi} {\rconj{w^{-1}}(\dualrep{\tau}\alpha^{1-s}
\delta)}.
\end{align*}
Since $\eta\ne\eta_0$, $Q_{\smallidx}'^{\eta}$ is a parabolic
subgroup of $G_{\smallidx}$ (see after \eqref{integral:cusp againt E
r,eta}). Then by Frobenius reciprocity (see e.g. \cite{BZ1} 2.28) we
get
\begin{align*}
\homspace{Q_{\smallidx}'^{\eta}}{
\pi}
{\rconj{w^{-1}}(\dualrep{\tau}\alpha^{1-s}
\delta)}&\isomorphic
\homspace{G_{\smallidx}}{
\pi}
{\induced{Q_{\smallidx}'^{\eta}}{G_{\smallidx}}{\delta_{Q_{\smallidx}'^{\eta}}^{-\half}\cdot\rconj{w^{-1}}(
\dualrep{\tau}\alpha^{1-s}
\delta)}}.
\end{align*}
(Recall that $Ind(\cdots)$ denotes normalized induction.) We claim
that this space vanishes, for all but a finite set of $q^{-s}$. In
fact, for any irreducible sub-quotients $\pi_0$ and $\tau_0$ of
$\pi$ and $\dualrep{\tau}$ (resp.), the proof of Lemma~10.1.2 of
\cite{GPS} immediately implies that outside finitely many values of
$q^{-s}$,
\begin{align}\label{space:bil space l<=n 0}
\homspace{G_{\smallidx}}{
\pi_0}
{\induced{Q_{\smallidx}'^{\eta}}{G_{\smallidx}}{\delta_{Q_{\smallidx}'^{\eta}}^{-\half}\cdot\rconj{w^{-1}}(\tau_0\alpha^{1-s}
\delta)}}=0.
\end{align}
Note that the lemma is applicable to any pair of irreducible
admissible (smooth) representations $\pi_0$ and $\tau_0$, not
necessarily generic. Its proof uses the fact that such
representations can be embedded in representations parabolically
induced from supercuspidal representations, then appeals to
\cite{BZ2} (Section~2). The fact that $Q_{\smallidx}'^{\eta}$ is a
parabolic subgroup is used, in order to replace
$\induced{Q_{\smallidx}'^{\eta}}{G_{\smallidx}}{\delta_{Q_{\smallidx}'^{\eta}}^{-\half}\cdot\rconj{w^{-1}}(\tau_0\alpha^{1-s}
\delta)}$ with a representation of $G_{\smallidx}$ parabolically
induced from a supercuspidal representation. Because $\pi$ and
$\dualrep{\tau}$ have finite Jordan-H\"{o}lder series and any
\nonzero\ element of
\begin{align*}
\homspace{G_{\smallidx}}{
\pi}
{\induced{Q_{\smallidx}'^{\eta}}{G_{\smallidx}}{\delta_{Q_{\smallidx}'^{\eta}}^{-\half}\cdot\rconj{w^{-1}}(\dualrep{\tau}\alpha^{1-s}
\delta)}}
\end{align*}
defines a \nonzero\ element of one of the spaces \eqref{space:bil
space l<=n 0} (see a detailed explanation of a similar argument below - in the paragraph before the definition
of \eqref{space:bil space l<=n 1 specific}, where we take $i,j\geq0$ with $i+j$ minimal, etc.), we
get that except a finite set of $q^{-s}$, $H(w_0\eta)=0$ for
$\eta\ne\eta_0$.

It is left to show that for $w=w_0\eta_0=w_0$, $Hom(w_0)$ is at most
one-dimensional, for almost all values of $q^{-s}$. Recall that
$\rconj{\eta}(Q_{\bigidx}^{w_r})\cap
N_{\bigidx-\smallidx}G_{\smallidx}=C_r^{\eta}\rtimes
Q_{\smallidx}'^{\eta}$, where
$Q_{\bigidx}^{w_r}=\rconj{w_r}Q_{\bigidx}\cap Q_{\bigidx-\smallidx}$
(see after \eqref{eq:Qn wr}). Then $\rconj{w_0}Q_{\bigidx}\cap
N_{\bigidx-\smallidx}G_{\smallidx}=\rconj{\eta_0}(Q_{\bigidx}^{w_0})\cap
N_{\bigidx-\smallidx}G_{\smallidx}=C_0^{\eta_0}\rtimes
Q_{\smallidx}'^{\eta_0}$. The subgroups $C_0^{\eta_0}$ and
$Q_{\smallidx}'^{\eta_0}$ were calculated explicitly after
\eqref{integral:after unfolding E_f completely}. We see that $w_0$
normalizes $Q_{\smallidx}'^{\eta_0}$ and then
$(N_{\bigidx-\smallidx}G_{\smallidx})^{w_0^{-1}}=\rconj{w_0^{-1}}(C_0^{\eta_0})\rtimes
Q_{\smallidx}'^{\eta_0}$, where $\rconj{w_0^{-1}}(C_0^{\eta_0})$ is
the image in $M_{\bigidx}$ of the subgroup
$U=\setof{\bigl(\begin{smallmatrix}z&0\\x&I_{\smallidx}\end{smallmatrix}\bigr)}{z\in
Z_{\bigidx-\smallidx}}<\GL{\bigidx}$ (if $\smallidx=\bigidx$,
$U=\{1\}$). Define a character $\psi'$ of $U$ by
$\psi'(\bigl(\begin{smallmatrix}z&0\\x&I_{\smallidx}\end{smallmatrix}\bigr))=\psi(z)\psi(\gamma
x_{\smallidx,1})$.

Let $R_{\smallidx}'<G_{\smallidx}$ be the unipotent subgroup defined
after \eqref{integral:after unfolding E_f completely}. Regard
$Y_{\smallidx}$  as explained in the proof of
Proposition~\ref{propo:basic global identity l <= n}: if
$G_{\smallidx}$ is split, $Y_{\smallidx}$ is a subgroup of
$L_{\smallidx}$ which normalizes $R_{\smallidx}'$. Hence
$\pi_{R_{\smallidx}',1}$ is a $Y_{\smallidx}$-module. If
$G_{\smallidx}$ is \quasisplit,
$Y_{\smallidx}=\lmodulo{R_{\smallidx}'}{P_{\smallidx}'}$ (see after
\eqref{integral:after unfolding E_f completely}) and because the
action of $R_{\smallidx}'$ on $\pi_{R_{\smallidx}',1}$ is trivial,
we again get that $\pi_{R_{\smallidx}',1}$ is a
$Y_{\smallidx}$-module.

Consider the embedding of $\GL{\smallidx}$ in $\GL{\bigidx}$ given
by $b\mapsto diag(I_{\bigidx-\smallidx},b)$. Denote by
$Y_{\smallidx}'$ the mirabolic subgroup of $\GL{\smallidx}$ (we
distinguish it from $Y_{\smallidx}$ defined above). The image of
$Y_{\smallidx}'$ under this embedding normalizes $U$ and stabilizes
$\psi'$, whence $(\tau\alpha^s)_{U,\psi'}$ is a
$Y_{\smallidx}'$-module. 
We recall that the action of $Y_{\smallidx}'$ on the Jacquet module
$(\tau\alpha^s)_{U,\psi'}$ is normalized by $\delta_U^{-\half}$ (see
\cite{BZ2} 1.8 (b)).

Let $y=\bigl(\begin{smallmatrix}a&b\\&1\end{smallmatrix}\bigr)\in
Y_{\smallidx}$. If $G_{\smallidx}$ is split, $y$ is an element of
$L_{\smallidx}$. In the \quasisplit\ case, we take a representative
in $P_{\smallidx}'$ to the coset of $y$. As an element of
$G_{\smallidx}$,
\begin{align*}
y=\begin{dcases}\left(\begin{array}{cccc}a&b&0&0\\&1&0&0\\&&1&b'\\&&&a^*\end{array}\right)&
\text{$G_{\smallidx}$ is split,}\\
\left(\begin{array}{cccc}a&0&b&*\\&1&0&0\\&&1&b'\\&&&a^*\end{array}\right)&\text{$G_{\smallidx}$
is \quasisplit.}
\end{dcases}
\end{align*}
Set
\begin{align*}
\mu=\begin{cases}\quarter &\text{split $G_{\smallidx}$,}\\
\half &\text{\quasisplit\ $G_{\smallidx}$.}\end{cases}
\end{align*}
The image of $y$ in $H_{\bigidx}$ ($G_{\smallidx}<H_{\bigidx}$) is
$diag(y',1,(y')^*)u$ where
$y'=diag(I_{\bigidx-\smallidx},\bigl(\begin{smallmatrix}a&\mu
b\\&1\end{smallmatrix}\bigr))\in Y_{\smallidx}'$ and $u\in
U_{\bigidx}$. In addition, this image lies in
$Q_{\smallidx}'^{\eta_0}$. Thus $(\tau\alpha^s)_{U,\psi'}$ is a
$Y_{\smallidx}$-module, the action of $y$ is defined through the
action of $y'$ ($\tau\alpha^s$ is trivial on $U_{\bigidx}$ by the
definition of $\cinduced{Q_{\bigidx}}{H_{\bigidx}}{\tau\alpha^s}$).
In the \quasisplit\ case, the choice of representative does not
matter, because $R_{\smallidx}'$ normalizes $U$ and stabilizes 
$\psi'$, and $\tau\alpha^s$ is trivial on $R_{\smallidx}'$ (the image of $R_{\smallidx}'$ in $H_{\bigidx}$ belongs to $U_{\bigidx}$).

Define
$\delta'=\rconj{w_0^{-1}}(\delta^{-1})\delta_{U}^{\half}$. We claim
\begin{align*}
H(w_0)\subset
Bil_{Y_{\smallidx}}(\pi_{R_{\smallidx}',1},(\tau\alpha^s\delta')_{U,\psi'}).
\end{align*}
The result of the proposition will follow once we establish a
similar assertion regarding
$Bil_{Y_{\smallidx}}(\pi_{R_{\smallidx}',1},(\tau\alpha^s\delta')_{U,\psi'})$.

Let $f\in Hom(w_0)$ and let $v_1$ and $v_2$ be as above. Since
$\rconj{w_0^{-1}}R_{\smallidx}'< \rconj{w_0^{-1}}G_{\smallidx}\cap
U_{\bigidx}<(N_{\bigidx-\smallidx}G_{\smallidx})^{w_0^{-1}}$, for
$r\in R_{\smallidx}'$,
\begin{align}\label{eq:uniqueness l <= n Rl' invariancy}
f(v_1\otimes v_2)=
f((\rconj{w_0}(\pi\otimes\psi_{\gamma})\otimes\tau\alpha^s)(\rconj{w_0^{-1}}r)v_1\otimes
v_2)=f(\pi(r)v_1\otimes v_2).
\end{align}
Hence $f$ is a bilinear form on $\pi_{R_{\smallidx}',1}\times\tau$. %
In addition if $x\in M_{\bigidx}$ is the image of $u\in U$,
\begin{align*}
f(v_1\otimes v_2)=
f((\rconj{w_0}(\pi\otimes\psi_{\gamma})\otimes\tau\alpha^s)(x)v_1\otimes
v_2)=(\psi')^{-1}(u)f(v_1\otimes (\tau\alpha^s)(u)v_2).
\end{align*}
Thus
\begin{align*}
f(v_1\otimes(\tau\alpha^s)(u)v_2)=\psi'(u)f(v_1\otimes v_2)
\end{align*}
and it follows that $f$ is a bilinear form on
$\pi_{R_{\smallidx}',1}\times(\tau\alpha^s)_{U,\psi'}$. Finally for
$y\in Y_{\smallidx}$,
\begin{align*}
\delta(\rconj{w_0^{-1}}y)f(v_1\otimes v_2)&=
f((\rconj{w_0}(\pi\otimes\psi_{\gamma})\otimes\tau\alpha^s)(\rconj{w_0^{-1}}y)v_1\otimes
v_2)\\\notag&=f(\pi(y)v_1\otimes\tau\alpha^s(y)v_2).
\end{align*}
In the \quasisplit\ case this is well-defined because of
\eqref{eq:uniqueness l <= n Rl' invariancy}. Also note that
$\tau\alpha^s(\rconj{w_0^{-1}}y)=\tau\alpha^s(y)$. We conclude $f\in
Bil_{Y_{\smallidx}}(\pi_{R_{\smallidx}',1},(\tau\alpha^s\delta')_{U,\psi'})$
($\delta'$ is trivial on $U$).

If $\smallidx=1$, $\pi$ is a character (because it is generic) and
$U=\rconj{\omega_{1,\bigidx-1}}Z_{\bigidx}$, whence the fact that
$\tau$ is generic implies that $(\tau\alpha^s\delta')_{U,\psi'}$ is
a one-dimensional space (for all $s$). Thus in this case
$Bil_{Y_{\smallidx}}(\pi_{R_{\smallidx}',1},(\tau\alpha^s\delta')_{U,\psi'})$
is at most one-dimensional and the proposition follows. Henceforth
we assume $\smallidx>1$.

In general if $0\leq m<\smallidx$, let
$B_m=\setof{\bigl(\begin{smallmatrix}c&x\\0&z\end{smallmatrix}\bigr)}{c\in\GL{m},z\in
Z_{\smallidx-m}}<Y_{\smallidx}$. For a representation $\sigma$ of
$\GL{m}$ and $\psi$ regarded as a character of $Z_{\smallidx-m}$,
$\sigma\otimes\psi$ naturally extends to a representation of $B_m$.
Then we have the induced representation
$\compinduced{B_{m}}{Y_{\smallidx}}{\sigma\otimes\psi}$.

Consider the space
\begin{align}\label{space:bil space l<=n 1}
Bil_{Y_{\smallidx}}(\compinduced{B_{m}}{Y_{\smallidx}}{\sigma\otimes\psi},\compinduced{B_{m'}}{Y_{\smallidx}}{\sigma'\alpha^{s+c'}\otimes\psi}).
\end{align}
Here $0\leq m,m'<\smallidx$, $\sigma$ and $\sigma'$ are irreducible
representations and $c'\in\C$. According to \cite{BZ2} (3.7) (or
\cite{JPSS} 1.4), it vanishes unless $m=m'$ and
$\sigma\isomorphic\dualrep{(\sigma')}\alpha^{1-s-c'}$, and when
$m=m'=0$, it is one-dimensional. If $m=m'>0$ this means that outside
a finite set of $q^{-s}$, the space vanishes.

The subgroup $R_{\smallidx}'$ is actually the image of
$U_{\smallidx-1}<H_{\smallidx-1}$ under the embedding
$H_{\smallidx-1}<G_{\smallidx}$. By Proposition~8.2 of \cite{GPS},
$\pi_{R_{\smallidx}',1}$ has a finite Jordan-H\"{o}lder series (as a $Y_{\smallidx}$-module), in
which the irreducible representation
$\compinduced{Z_{\smallidx}}{Y_{\smallidx}}{\psi}$ appears as a
sub-representation and does not appear as a sub-quotient of
$\lmodulo{\compinduced{Z_{\smallidx}}{Y_{\smallidx}}{\psi}}{\pi_{R_{\smallidx}',1}}$. 
Note that this proposition was stated for an irreducible admissible
representation of $SO_{2\smallidx+1}$, but the arguments are general
and apply to any smooth admissible finitely generated generic
representation of $G_{\smallidx}$ (split or \quasisplit). A similar
result holds for $(\tau\alpha^s\delta')_{U,\psi'}$ (\cite{BZ1} 5.22 and 5.15)
- $(\tau\alpha^s\delta')_{U,\psi'}$ has a finite Jordan-H\"{o}lder series,
$\compinduced{Z_{\smallidx}}{Y_{\smallidx}}{\psi}\subset(\tau\alpha^s\delta')_{U,\psi'}$
and $\compinduced{Z_{\smallidx}}{Y_{\smallidx}}{\psi}$ does not
appear as a sub-quotient of
$\lmodulo{\compinduced{Z_{\smallidx}}{Y_{\smallidx}}{\psi}}{(\tau\alpha^s\delta')_{U,\psi'}}$.

Now we argue as in \cite{JPSS} (2.10, after (6)). Let
\begin{align*}
0=\mathcal{V}_0\subset\mathcal{V}_1\subset\ldots\subset\mathcal{V}_k=\pi_{R_{\smallidx}',1},\qquad
0=\mathcal{W}_0\subset \mathcal{W}_1\subset\ldots\subset
\mathcal{W}_d=(\tau\alpha^s\delta')_{U,\psi'}
\end{align*}
be the corresponding Jordan-H\"{o}lder series of
$Y_{\smallidx}$-modules, where
$\mathcal{V}_1\isomorphic\mathcal{W}_1\isomorphic
\compinduced{Z_{\smallidx}}{Y_{\smallidx}}{\psi}$. According to the
structure theory of $Y_{\smallidx}$-modules (\cite{BZ2} Section~3),
for $i,j>1$,
\begin{align*}
\lmodulo{\mathcal{V}_{i-1}}{\mathcal{V}_{i}}\isomorphic
\compinduced{B_{m_i}}{Y_{\smallidx}}{\sigma_i\otimes\psi},\qquad
\lmodulo{\mathcal{W}_{j-1}}{\mathcal{W}_{j}}\isomorphic\compinduced{B_{m_j'}}{Y_{\smallidx}}{\sigma_j'\alpha^{s+c_j'}\otimes\psi},
\end{align*}
where $0<m_i,m_j'<\smallidx$, $\sigma_i$ (resp. $\sigma_j'$) is an
irreducible representation of $\GL{m_i}$ (resp. $\GL{m_j'}$) and
each $c_j'$ is a complex constant (independent of $s$). In fact,
$\sigma_i$ is the $(\smallidx-m_i)$-th Bernstein-Zelevinsky
derivative of $\pi_{R_{\smallidx}',1}$ (\cite{BZ2} 3.5).


Let $0\ne f\in
Bil_{Y_{\smallidx}}(\pi_{R_{\smallidx}',1},(\tau\alpha^s\delta')_{U,\psi'})$.
Take $i,j\geq0$ such that $\frestrict{f}{\mathcal{V}_{i}\times
\mathcal{W}_{j}}\nequiv0$ and $i+j$ is minimal. Then by definition
$i,j>0$,  $\frestrict{f}{\mathcal{V}_{i-1}\times
\mathcal{W}_{j}}\equiv0$ and $\frestrict{f}{\mathcal{V}_{i}\times
\mathcal{W}_{j-1}}\equiv0$. Hence $f$ is a \nontrivial\ element of
$Bil_{Y_{\smallidx}}(\lmodulo{\mathcal{V}_{i-1}}{\mathcal{V}_{i}},\lmodulo{\mathcal{W}_{j-1}}{\mathcal{W}_{j}})$,
which is of the form \eqref{space:bil space l<=n 1}. Now the arguments
above imply that except for a finite set of $q^{-s}$, $i=j=1$ and
$f\in Bil_{Y_{\smallidx}}(\mathcal{V}_1,\mathcal{W}_1)\isomorphic
Bil_{Y_{\smallidx}}(\compinduced{Z_{\smallidx}}{Y_{\smallidx}}{\psi},\compinduced{Z_{\smallidx}}{Y_{\smallidx}}{\psi})$,
which is one-dimensional.

Now assume that $s$ is such that $q^{-s}$ does not lie in the finite
set of $q^{-s}$, where any of the spaces
\begin{align}\label{space:bil space l<=n 1 specific}
Bil_{Y_{\smallidx}}(\compinduced{B_{m_i}}{Y_{\smallidx}}{\sigma_i\otimes\psi},\compinduced{B_{m_j'}}{Y_{\smallidx}}{\sigma_j'\alpha^{s+c_j'}\otimes\psi}),\quad
\text{$i>1$ or $j>1$},
\end{align}
might be \nontrivial. If
$Bil_{Y_{\smallidx}}(\pi_{R_{\smallidx}',1},(\tau\alpha^s\delta')_{U,\psi'})\ne0$,
fix $0\ne\varphi\in
Bil_{Y_{\smallidx}}(\pi_{R_{\smallidx}',1},(\tau\alpha^s\delta')_{U,\psi'})$.
Then
$\frestrict{\varphi}{Bil_{Y_{\smallidx}}(\mathcal{V}_1,\mathcal{W}_1)}\nequiv0$
and for any $f\in
Bil_{Y_{\smallidx}}(\pi_{R_{\smallidx}',1},(\tau\alpha^s\delta')_{U,\psi'})$
there is some $c_f\in\C$ satisfying
\begin{align*}
\frestrict{(f-c_f\cdot\varphi)}{Bil_{Y_{\smallidx}}(\mathcal{V}_1,\mathcal{W}_1)}\equiv0.
\end{align*}
Since the spaces \eqref{space:bil space l<=n 1 specific} vanish,
$f\equiv c_f\cdot\varphi$. We conclude that except for a finite set
of values of $q^{-s}$,
$Bil_{Y_{\smallidx}}(\pi_{R_{\smallidx}',1},(\tau\alpha^s\delta')_{U,\psi'})$
is at most one-dimensional, then the same holds for
$Bil_{G_{\smallidx}}(\pi,(\cinduced{Q_{\bigidx}}{H_{\bigidx}}{\tau\alpha^s})_{N_{\bigidx-\smallidx},\psi_{\gamma}^{-1}})$.
\end{proof} 

\section{The case $\smallidx>\bigidx$
}\label{section:uniqneness l>n}
Let $Bil_{H_{\bigidx}}(\pi_{N^{\smallidx-\bigidx},\psi_{\gamma}^{-1}},\cinduced{Q_{\bigidx}}{H_{\bigidx}}{\tau\alpha^s})$ denote the space of $H_{\bigidx}$-equivariant bilinear forms, where $(\cdots)_{N^{\smallidx-\bigidx},\psi_{\gamma}^{-1}}$ is a Jacquet module.
This space is equivalent to the space of bilinear forms on $\pi\times\cinduced{Q_{\bigidx}}{H_{\bigidx}}{\tau\alpha^s}$ satisfying \eqref{eq:global integral in
homspace l
> n}.
\begin{proposition}\label{proposition:one dimensionality l > n}
Except for a finite set of values of $q^{-s}$, the complex
vector-space
$Bil_{H_{\bigidx}}(\pi_{N^{\smallidx-\bigidx},\psi_{\gamma}^{-1}},\cinduced{Q_{\bigidx}}{H_{\bigidx}}{\tau\alpha^s})$
is at most one-dimensional.
\end{proposition}
\begin{proof}[Proof of Proposition~\ref{proposition:one dimensionality l > n}] 
Using \cite{BZ2} (1.9 (d), (e), (b) and (c)),
\begin{align*}
Bil_{H_{\bigidx}}(\pi_{N^{\smallidx-\bigidx},\psi_{\gamma}^{-1}},\cinduced{Q_{\bigidx}}{H_{\bigidx}}{\tau\alpha^s})
&\isomorphic\homspace{H_{\bigidx}}
{\pi_{N^{\smallidx-\bigidx},\psi_{\gamma}^{-1}}}{\cinduced{Q_{\bigidx}}{H_{\bigidx}}{\dualrep{\tau\alpha^s}}}\\\notag
&\isomorphic\homspace{\GL{\bigidx}}
{(\pi_{N^{\smallidx-\bigidx},\psi_{\gamma}^{-1}})_{U_{\bigidx},1}}{\dualrep{\tau\alpha^s}}\\\notag
&\isomorphic
Bil_{\GL{\bigidx}}(\pi_{V',\psi_{\gamma}^{-1}},\tau\alpha^s).
\end{align*}
Here $V'=N^{\smallidx-\bigidx}\rtimes U_{\bigidx}$ and $\psi_{\gamma}^{-1}$ is extended to $V'$ trivially on $U_{\bigidx}$. Observe that 
$\pi_{V',\psi_{\gamma}^{-1}}\isomorphic
(\pi_{R_{\smallidx}',1})_{\lmodulo{R_{\smallidx}'}{V'},\psi_{\gamma}^{-1}}$
as $\GL{\bigidx}$-modules, where $\GL{\bigidx}$ is embedded in
$G_{\smallidx}$ via $b\mapsto
diag(I_{\smallidx-\bigidx-1},b,I_2,b^*,I_{\smallidx-\bigidx-1})$
(this is the image of $M_{\bigidx}<H_{\bigidx}$ in $G_{\smallidx}$)
and note that $\GL{\bigidx}$ normalizes $V'$ and $R_{\smallidx}'$
and stabilizes $\psi_{\gamma}^{-1}$. The group
$\lmodulo{R_{\smallidx}'}{V'}$ is isomorphic to
\begin{align*}
\setof{\left(\begin{array}{ccc}z&u_1&u_2\\&I_{\bigidx}&0\\&&1\end{array}\right)}{z\in
Z_{\smallidx-\bigidx-1}}<\GL{\smallidx}.
\end{align*}
According to \cite{BZ2} (1.9 (b), (d) and (f)),
\begin{align*}
Bil_{\GL{\bigidx}}(\pi_{V',\psi_{\gamma}^{-1}},\tau\alpha^s)
&\isomorphic
\homspace{\GL{\bigidx}}{\pi_{V',\psi_{\gamma}^{-1}}}{\dualrep{\tau\alpha^s}}\\\notag
&\isomorphic
Bil_{Y_{\smallidx}}(\pi_{R_{\smallidx}',1},\compinduced{\GL{\bigidx}\ltimes(\lmodulo{R_{\smallidx}'}{V'})}{Y_{\smallidx}}{\tau\alpha^s\otimes\psi_{\gamma}}\delta_{Y_{\smallidx}}^{-1})\\\notag
&\isomorphic
Bil_{Y_{\smallidx}}(\pi_{R_{\smallidx}',1},\compinduced{\GL{\bigidx}\ltimes(\lmodulo{R_{\smallidx}'}{V'})}{Y_{\smallidx}}{\tau\alpha^{s-1}\otimes\psi_{\gamma}}).
\end{align*}
Here the group $Y_{\smallidx}$ is defined as explained after
\eqref{integral:after unfolding E_f completely} and
$\compinduced{}{}{\cdots}$ denotes the compact normalized induction.
The representation $\pi_{R_{\smallidx}',1}$ is a
$Y_{\smallidx}$-module, see
Section~\ref{section:uniqneness l<=n}. 

As in the proof of Proposition~\ref{proposition:one dimensionality l
<= n} and with a similar notation, we consider the Jordan-H\"{o}lder
series of $\pi_{R_{\smallidx}',1}$ and prove that the space
\begin{align}\label{space:bil space l>n 1}
Bil_{Y_{\smallidx}}(\compinduced{B_{m}}{Y_{\smallidx}}{\sigma\otimes\psi},
\compinduced{\GL{\bigidx}\ltimes(\lmodulo{R_{\smallidx}'}{V'})}{Y_{\smallidx}}{\tau\alpha^{s-1}\otimes\psi_{\gamma}}),
\end{align}
where $\sigma$ is an irreducible representation of $\GL{m}$ and
$0\leq m<\smallidx$, vanishes for all but a finite set of values of
$q^{-s}$, when $m>0$, and for $m=0$ it is at most one-dimensional
(for all $s$). This implies the result.

Let
\begin{align*}
X_{\bigidx}=\setof{\left(\begin{array}{ccc}a&u_1&u_2\\&b&u_3\\&&1\end{array}\right)}{a\in\GL{\smallidx-\bigidx-1},b\in\GL{\bigidx}}<\GL{\smallidx}.
\end{align*}
Then $\GL{\bigidx}\ltimes(\lmodulo{R_{\smallidx}'}{V'})$ can be
regarded as a subgroup of $X_{\bigidx}$. Also denote
\begin{align*}
Y_{\smallidx-\bigidx}=\setof{\left(\begin{array}{ccc}a&0&u_2\\&I_{\bigidx}&0\\&&1\end{array}\right)}{a\in\GL{\smallidx-\bigidx-1}}<X_{\bigidx}.
\end{align*}
We have the following isomorphism,
\begin{align*}
\compinduced{\GL{\bigidx}\ltimes(\lmodulo{R_{\smallidx}'}{V'})}{Y_{\smallidx}}{\tau\alpha^{s-1}\otimes\psi_{\gamma}}\isomorphic
\compinduced{X_{\bigidx}}{Y_{\smallidx}}{\compinduced{Z_{\smallidx-\bigidx}}{Y_{\smallidx-\bigidx}}
{\psi_{\gamma}}\otimes\tau\alpha^{s-1}}.
\end{align*}
For $w\in\rmodulo{\lmodulo{B_m}{Y_{\smallidx}}}{X_{\bigidx}}$ let
\begin{align*}
Hom(w)=\homspace{X_{\bigidx}^{w^{-1}}}{\rconj{w}(\compinduced{Z_{\smallidx-\bigidx}}{Y_{\smallidx-\bigidx}}
{\psi_{\gamma}}\otimes\tau\alpha^{s-1})\otimes{(\sigma\otimes\psi)}}{\delta},
\end{align*}
where $X_{\bigidx}^{w^{-1}}=\rconj{w^{-1}}X_{\bigidx}\cap B_m$ and
$\delta(x)=\delta_{\mathcal{C}(w)}(x,\rconj{w}x)\cdot\delta_{X_{\bigidx}}^{-\half}(\rconj{w}x)\cdot\delta_{B_{m}}^{-\half}(x)$, $\mathcal{C}(w)=\setof{(x,\rconj{w}x)}{x\in X_{\bigidx}^{w^{-1}}}<B_m\times X_{\bigidx}$.
According to the Bruhat theory (\cite{Silb} Theorems~1.9.4 and 1.9.5) the space \eqref{space:bil space l>n
1} is embedded in
$\bigoplus_{w\in\rmodulo{\lmodulo{B_m}{Y_{\smallidx}}}{X_{\bigidx}}}Hom(w)$ (in the sense of semi-simplification).
The first step is to show that only two elements $w$ may contribute
to this sum.

Using the Bruhat decomposition
$\GL{\smallidx-1}=Z_{\smallidx-1}W_{\GL{\smallidx-1}}P$, where
$W_{\GL{\smallidx-1}}$ is the Weyl group of $\GL{\smallidx-1}$ and
$P<\GL{\smallidx-1}$ is any parabolic subgroup, we can assume that
$w$ is of the form $w=diag(w^{(1)},1)$ with $w^{(1)}\in
W_{\GL{\smallidx-1}}$. 
Let $f\in H(w)$, $v_1$ be an element in the space of
$\compinduced{Z_{\smallidx-\bigidx}}{Y_{\smallidx-\bigidx}}
{\psi_{\gamma}}$, $v_2$ be in the space of $\tau$ and take $v_3$ in
the space of $\sigma$.

If $m=\smallidx-1$ or $\bigidx=\smallidx-1$, we can assume
$w^{(1)}=I_{\smallidx-1}$. Assume $m,\bigidx<\smallidx-1$. Let
$e_i\in F^{\smallidx-1}$ be the standard basis element (i.e., $e_i$
is the column with $1$ in the $i$-th row and $0$ elsewhere). Assume
that the last row of $w^{(1)}$ is $\transpose{e_i}$, with
$i>\smallidx-\bigidx-1$. Let
\begin{align*}
y=\left(\begin{array}{ccc}I_{\smallidx-\bigidx-1}&0&y_1\\&I_{\bigidx}&0\\&&1\end{array}\right)\in
X_{\bigidx}.
\end{align*}
Then $x=\rconj{w^{-1}}y\in X_{\bigidx}^{w^{-1}}$ and since
$(\sigma\otimes\psi)(x)\equiv1$,
\begin{align*}
f(v_1\otimes v_2\otimes v_3)&=
f((\rconj{w}(\compinduced{Z_{\smallidx-\bigidx}}{Y_{\smallidx-\bigidx}}
{\psi_{\gamma}}\otimes\tau\alpha^{s-1})\otimes(\sigma\otimes\psi))(x)v_1\otimes
v_2\otimes v_3)\\\notag&=
f(\compinduced{Z_{\smallidx-\bigidx}}{Y_{\smallidx-\bigidx}}
{\psi_{\gamma}}(y)v_1\otimes v_2\otimes v_3).
\end{align*}
Hence the functional $v_1\mapsto f(v_1\otimes v_2\otimes v_3)$ on
the space of
$\compinduced{Z_{\smallidx-\bigidx}}{Y_{\smallidx-\bigidx}}
{\psi_{\gamma}}$ factors through the Jacquet module
$(\compinduced{Z_{\smallidx-\bigidx}}{Y_{\smallidx-\bigidx}}
{\psi_{\gamma}})_{Z_{\smallidx-\bigidx-1,1},\psi^{(1)}}$, where
$\psi^{(1)}\equiv1$ is the trivial character. By \cite{BZ2} (3.2
(d)), this is the zero module whence $f\equiv0$ and $H(w)=0$.
Therefore we can assume $i\leq\smallidx-\bigidx-1$ and multiplying
$w$ on the right by an element of $X_{\bigidx}$, we may further
assume that the last row of $w^{(1)}$ is
$\transpose{e_{\smallidx-\bigidx-1}}$. Then $w$ takes the form
\begin{align*}
w=\left(\begin{array}{ccc}w^{(2)}&\\&1\\&&1\end{array}\right)
\left(\begin{array}{cccc}I_{\smallidx-\bigidx-2}\\&&I_{\bigidx}\\&1\\&&&1\end{array}\right),\qquad
w^{(2)}\in W_{\GL{\smallidx-2}}.
\end{align*}

If $m=\smallidx-2$, we can assume $w^{(2)}=I_{\smallidx-2}$ whence
$w=diag(I_{\smallidx-\bigidx-2},\omega_{\bigidx,1},1)$. If
$\bigidx=\smallidx-2$, we may assume $w=diag(\omega_{\bigidx,1},1)$.
Assume $m,\bigidx<\smallidx-2$ and that the last row of $w^{(2)}$ is
$\transpose{e_i}$ with $e_i\in F^{\smallidx-2}$ and
$i>\smallidx-\bigidx-2$. Then for any
\begin{align*}
y=\left(\begin{array}{cccc}I_{\smallidx-\bigidx-2}&y_1&0&y_2\\&1&0&y_3\\&&I_{\bigidx}&0\\&&&1\end{array}\right)\in
X_{\bigidx},
\end{align*}
$x=\rconj{w^{-1}}y\in X_{\bigidx}^{w^{-1}}$. Since
$(\sigma\otimes\psi)(x)=\psi(x_{\smallidx-1,\smallidx})=\psi(y_3)$,
\begin{align*}
f(v_1\otimes v_2\otimes v_3)&=
f((\rconj{w}(\compinduced{Z_{\smallidx-\bigidx}}{Y_{\smallidx-\bigidx}}
{\psi_{\gamma}}\otimes\tau\alpha^{s-1})\otimes(\sigma\otimes\psi))(x)v_1\otimes
v_2\otimes v_3)\\\notag&=
\psi(y_3)f(\compinduced{Z_{\smallidx-\bigidx}}{Y_{\smallidx-\bigidx}}
{\psi_{\gamma}}(y)v_1\otimes v_2\otimes v_3).
\end{align*}
Thus $v_1\mapsto f(v_1\otimes v_2\otimes v_3)$ factors through the
Jacquet module
$(\compinduced{Z_{\smallidx-\bigidx}}{Y_{\smallidx-\bigidx}}
{\psi_{\gamma}})_{Z_{\smallidx-\bigidx-2,1,1},\psi^{(2)}}$ where
$\psi^{(2)}(z)=\psi^{-1}(z_{\smallidx-\bigidx-1,\smallidx-\bigidx})$.
Since $\psi^{(2)}$ does not depend on the
$(\smallidx-\bigidx-2,\smallidx-\bigidx-1)$-th coordinate of $z$, it
is a degenerate character of $Z_{\smallidx-\bigidx-2,1,1}$, hence by
\cite{BZ2} (3.2 (e) and (d)) this is the zero module and $f\equiv0$.
Therefore we can assume $i=\smallidx-\bigidx-2$ and that $w$ takes
the form
\begin{align*}
w=\left(\begin{array}{ccc}w^{(3)}&\\&I_2\\&&1\end{array}\right)
\left(\begin{array}{cccc}I_{\smallidx-\bigidx-3}\\&&I_{\bigidx}\\&I_2\\&&&1\end{array}\right),\qquad
w^{(3)}\in W_{\GL{\smallidx-3}}.
\end{align*}

If $m=\smallidx-3$, we may take $w^{(3)}=I_{\smallidx-3}$ and
$w=diag(I_{\smallidx-\bigidx-3},\omega_{\bigidx,2},1)$. If
$\bigidx=\smallidx-3$, we can assume $w=diag(\omega_{\bigidx,2},1)$.

Proceeding as above we conclude that $H(w)$ vanishes, unless
\begin{align*}
w=\begin{cases}diag(I_{m-\bigidx},\omega_{\bigidx,\smallidx-m-1},1)&m>\bigidx,\\
diag(\omega_{\bigidx,\smallidx-\bigidx-1},1)&m\leq\bigidx.\end{cases}
\end{align*}
(This also holds when $w^{(1)}=I_{\smallidx-1}$, because
$\omega_{\bigidx,0}=I_{\bigidx}$.) The second step is to show that
in fact, only $diag(\omega_{\bigidx,\smallidx-\bigidx-1},1)$ may
contribute and in particular it follows that \eqref{space:bil space
l>n 1} vanishes unless $m\leq\bigidx$. Assume
$w=diag(I_{m-\bigidx},\omega_{\bigidx,\smallidx-m-1},1)$. Take
\begin{align*}
y=\left(\begin{array}{cccc}I_{m-\bigidx}&y_1&0&y_2\\&z&0&y_3\\&&I_{\bigidx}&0\\&&&1\end{array}\right)\in
X_{\bigidx}\quad(z\in Z_{\smallidx-m-1}),\qquad x=\rconj{w^{-1}}y\in
X_{\bigidx}^{w^{-1}}.
\end{align*}
Then
$\rconj{w}(\compinduced{Z_{\smallidx-\bigidx}}{Y_{\smallidx-\bigidx}}
{\psi_{\gamma}}\otimes\tau\alpha^{s-1})(x)=\compinduced{Z_{\smallidx-\bigidx}}{Y_{\smallidx-\bigidx}}
{\psi_{\gamma}}(y)$ and
$(\sigma\otimes\psi)(x)=\psi(\bigl(\begin{smallmatrix}z&y_3\\&1\end{smallmatrix}\bigr))$.
It follows that the mapping $v_1\mapsto f(v_1\otimes v_2\otimes
v_3)$ factors through
\begin{align*}
(\compinduced{Z_{\smallidx-\bigidx}}{Y_{\smallidx-\bigidx}}{\psi_{\gamma}})_{Z_{m-\bigidx,1,\ldots,1},\psi^{(\smallidx-m)}},
\end{align*}
where 
$\psi^{(\smallidx-m)}(z)=\psi^{-1}(\sum_{i=1}^{\smallidx-m-1}z_{m-\bigidx+i,m-\bigidx+i+1})$
(for $m=\smallidx-1$, $\psi^{(\smallidx-m)}\equiv1$). This space
vanishes according to \cite{BZ2} (3.2 (d) and (e),
$\psi^{(\smallidx-m)}$ is a degenerate character of
$Z_{m-\bigidx,1,\ldots,1}$ because it does not depend on the
$(m-\bigidx,m-\bigidx+1)$-th coordinate of $z$). Hence $H(w)=0$.

Now assume $w=diag(\omega_{\bigidx,\smallidx-\bigidx-1},1)$. The
third step is to show that when $m>0$, except for a finite set of
$q^{-s}$, $H(w)=0$, and for $m=0$, $H(w)$ is at most one-dimensional. We see
that
\begin{align*}
X_{\bigidx}^{w^{-1}}=\setof{\left(\begin{array}{cccc}c&u_1&0&u_2\\&z_1&0&u_3\\&&z_2&u_4\\&&&1\end{array}\right)}
{c\in\GL{m},z_1\in Z_{\bigidx-m},z_2\in Z_{\smallidx-\bigidx-1}}.
\end{align*}
For
\begin{align*}
y=\left(\begin{array}{ccc}z_2&0&u_4\\&I_{\bigidx}&0\\&&1\end{array}\right)\in
X_{\bigidx}\quad(z_2\in Z_{\smallidx-\bigidx-1}),\qquad
x=\rconj{w^{-1}}y\in X_{\bigidx}^{w^{-1}},
\end{align*}
we have
$\rconj{w}(\compinduced{Z_{\smallidx-\bigidx}}{Y_{\smallidx-\bigidx}}
{\psi_{\gamma}}\otimes\tau\alpha^{s-1})(x)=\compinduced{Z_{\smallidx-\bigidx}}{Y_{\smallidx-\bigidx}}
{\psi_{\gamma}}(y)$ and
$(\sigma\otimes\psi)(x)=\psi(\bigl(\begin{smallmatrix}z_2&u_4\\&1\end{smallmatrix}\bigr))$.
It follows that the mapping $v_1\mapsto f(v_1\otimes v_2\otimes
v_3)$ factors through
$(\compinduced{Z_{\smallidx-\bigidx}}{Y_{\smallidx-\bigidx}}{\psi_{\gamma}})_{Z_{\smallidx-\bigidx},\psi^{-1}}$,
which is one-dimensional according to \cite{BZ2} (3.2 (e)).

Let $\varphi$ be a \nonzero\ functional on
$(\compinduced{Z_{\smallidx-\bigidx}}{Y_{\smallidx-\bigidx}}{\psi_{\gamma}})_{Z_{\smallidx-\bigidx},\psi^{-1}}$.
Let $f\in H(w)$. There is a complex constant $c_f(v_2\otimes v_3)$
satisfying $f(v_1\otimes v_2\otimes v_3)=c_f(v_2\otimes
v_3)\varphi(v_1)$ for all $v_1$. The function $v_2\otimes v_3\mapsto
c_f(v_2\otimes v_3)$ is bilinear, hence extends to a functional on
the tensor product of the spaces of $\tau$ and $\sigma$. Set
\begin{align*}
y=\left(\begin{array}{cccc}I_{\smallidx-\bigidx-1}&0&0&0\\&c&u_1&0\\&&z_1&0\\&&&1\end{array}\right)\in
X_{\bigidx}\quad(c\in\GL{m},z_1\in Z_{\bigidx-m}).
\end{align*}
Then for any $v_1$,
\begin{align*}
\delta(\rconj{w^{-1}}y)f(v_1\otimes v_2\otimes v_3)&=
f(v_1\otimes\tau\alpha^{s-1}\left(\begin{array}{cc}c&u_1\\&z_1\end{array}\right)v_2\otimes
\psi(z_1)\sigma(c)v_3)
\end{align*}
and therefore
\begin{align*}
\delta(\rconj{w^{-1}}y)c_f(v_2\otimes v_3)&=
c_f(\tau\alpha^{s-1}\left(\begin{array}{cc}c&u_1\\&z_1\end{array}\right)v_2\otimes
\psi(z_1)\sigma(c)v_3).
\end{align*}
In particular taking $c=I_m$, we get that $c_f$ is a bilinear form
on $(\tau\alpha^{s-1})_{Z_{m,1,\ldots,1},\psi^{-1}}\times\sigma$.
Here $Z_{m,1,\ldots,1}<Z_{\bigidx}$ ($m+1+\ldots+1=\bigidx$) and
$\psi^{-1}(\bigl(\begin{smallmatrix}I_m&u_1\\&z_1\end{smallmatrix}\bigr))=\psi^{-1}(z_1)$
($z_1\in Z_{\bigidx-m}$). Put
$\delta'=\rconj{w^{-1}}(\delta^{-1})\delta_{Z_{m,1,\ldots,1}}^{\half}$.
Then
\begin{align*}
c_f\in
Bil_{\GL{m}}((\tau\alpha^{s-1}\delta')_{Z_{m,1,\ldots,1},\psi^{-1}},\sigma).
\end{align*}

Observe that 
\begin{align*}
Bil_{\GL{m}}((\tau\alpha^{s-1}\delta')_{Z_{m,1,\ldots,1},\psi^{-1}},\sigma)&\isomorphic
\homspace{\GL{m}}{(\tau\alpha^{s-1}\delta')_{Z_{m,1,\ldots,1},\psi^{-1}}}{\dualrep{\sigma}}\\\notag
&\isomorphic\homspace{\GL{\bigidx}}{\tau\alpha^{s-1}\delta'}{\induced{B_m}{\GL{\bigidx}}{\dualrep{\sigma}\otimes\psi^{-1}}}.
\end{align*}
When $m>0$, 
\begin{align*}
\homspace{\GL{\bigidx}}{\tau\alpha^{s-1}\delta'}{\induced{B_m}{\GL{\bigidx}}{\dualrep{\sigma}\otimes\psi^{-1}}}
&\isomorphic
\homspace{\GL{\bigidx}}{\compinduced{B_m}{\GL{\bigidx}}{\sigma\otimes\psi}}{\dualrep{\tau}\alpha^{2-s}(\delta')^{-1}}
\end{align*}
and according to \cite{GPS} (p.~107), outside of a finite set of
$q^{-s}$, this space vanishes. In fact \textit{loc. cit.} applies to
irreducible representations, in order to use the result here we
consider
$Bil_{\GL{\bigidx}}(\compinduced{B_m}{\GL{\bigidx}}{\sigma\otimes\psi},\tau\alpha^{s-1}\delta')$
and a finite Jordan-H\"{o}lder series of $\tau$ (as a
$GL_{\bigidx}$-module), see Section~\ref{section:uniqneness l<=n}
(before \eqref{space:bil space l<=n 0} and before \eqref{space:bil
space l<=n 1 specific}). Hence for all but a finite set of $q^{-s}$,
$c_f\equiv0$ for all $f\in H(w)$, forcing $H(w)=0$.

If $m=0$,
$\homspace{\GL{\bigidx}}{\tau\alpha^{s-1}\delta'}{\induced{B_m}{\GL{\bigidx}}{\dualrep{\sigma}\otimes\psi^{-1}}}$
is one-dimensional because $\tau$ is generic. Assuming $H(w)\ne0$,
take $0\ne f\in H(w)$, then $c_f\nequiv0$ and for any $f'\in H(w)$,
$c_{f'}=c'\cdot c_f$ for some $c'\in\C$. Then
\begin{align*}
f'(v_1\otimes v_2\otimes v_3)=c_{f'}(v_2\otimes
v_3)\varphi(v_1)=c'f(v_1\otimes v_2\otimes v_3).
\end{align*}
Hence in this case $H(w)$ is one-dimensional. This completes the proof of the assertions regarding
\eqref{space:bil space l>n 1}.

Now as in the proof of Proposition~\ref{proposition:one
dimensionality l <= n}, we obtain that except for a finite set of
values of $q^{-s}$,
$Bil_{Y_{\smallidx}}(\pi_{R_{\smallidx}',1},\compinduced{\GL{\bigidx}\ltimes(\lmodulo{R_{\smallidx}'}{V'})}{Y_{\smallidx}}{\tau\alpha^{s-1}\otimes\psi_{\gamma}})$
is at most one-dimensional and the same holds for
$Bil_{H_{\bigidx}}(\pi_{N^{\smallidx-\bigidx},\psi_{\gamma}^{-1}},\cinduced{Q_{\bigidx}}{H_{\bigidx}}{\tau\alpha^s})$.
\end{proof} 


%% file: chapter_basic_properties.tex
\newtheorem{theorem}{Theorem}[section]
\newtheorem{proposition}{Proposition}[section]
\newtheorem{corollary}{Corollary}[section]
\newtheorem{lemma}{Lemma}[section]
\newtheorem{claim}{Claim}[section]
\theoremstyle{remark}
\newtheorem{remark}{Remark}[section]
\newtheorem{example}{Example}[section]
\theoremstyle{definition}
\newtheorem{definition}{Definition}[section]
\numberwithin{equation}{section}
\newcommand{\chapter}{\section} 
\input{thesis_notations}
\end{comment}

\chapter{Basic properties of the integrals}\label{chapter:Properties of the integrals}
In 
Chapter~\ref{chapter:the integrals} we constructed the global
integrals and decomposed them into Euler products
of local integrals. The local integrals over finite places are the focus of this study.
Henceforth the terminology and notation are local (and \nonarchimedean, except for Chapter~\ref{chapter:archimedean results}).

In this chapter we establish a few key properties of the integrals,
to be used in subsequent chapters.
The main result is Proposition~\ref{proposition:iwasawa decomposition of the integral}, which
shows how to convert $\Psi(W,f_s,s)$ to a sum of integrals, each
over a torus, while removing any unipotent integrations (such integrations usually complicate calculations).
This proposition has the following applications:
holomorphicity of the integrals for supercuspidal representations (Corollary~\ref{corollary:integral is holomorphic for supercuspidal data}),
meromorphic continuation of the integrals (Section~\ref{subsection:meromorphic continuation}), and
sharp convergence results needed to prove Theorem~\ref{theorem:gcd for tempered reps} (Chapter~\ref{chapter:the gcd}). Furthermore, it forms the basis for
Chapter~\ref{chapter:upper boubnds on the gcd}, suggesting that the
integral can be associated with a Laurent series. 

Further results include the non-triviality of the integrals and a description of
Whittaker functionals for induced representations $\tau$ and $\pi$. These functionals will be used for integral manipulations in
Chapters~\ref{section:gamma_mult}-\ref{chapter:upper boubnds on the gcd}.

\section{Definition of the integrals}\label{subsection:the integrals}
We redefine, for clarity and quick reference, the integrals for
$G_{\smallidx}\times \GL{\bigidx}$ and a pair of representations
$\pi\times\tau$. Let $\pi$ be a smooth, admissible, finitely generated and generic representation of $G_{\smallidx}$
whose underlying space is $\Whittaker{\pi}{\psi_{\gamma}^{-1}}$,
where $\psi_{\gamma}$ is the generic character of
$U_{G_{\smallidx}}$ given by
\begin{align*}
\psi_{\gamma}(u)=
\begin{cases}\psi(\sum_{i=1}^{\smallidx-2}u_{i,i+1}+\quarter
u_{\smallidx-1,\smallidx}-\gamma u_{\smallidx-1,\smallidx+1})&\text{$G_{\smallidx}$ is split},\\
\psi(\sum_{i=1}^{\smallidx-2}u_{i,i+1}+\half u_{\smallidx-1,\smallidx+1})&\text{$G_{\smallidx}$ is \quasisplit}.\end{cases}
\end{align*}

Let $\tau$ be a smooth, admissible, finitely generated and generic representation of $\GL{\bigidx}$. We do not require $\pi$ or
$\tau$ to be irreducible, but $\tau$ needs to have a central character 
(see Section~\ref{subsection:the intertwining operator for tau}).
For $s\in\C$ form the representation
$\cinduced{Q_{\bigidx}}{H_{\bigidx}}{\tau\alpha^s}$ on the space $V(\tau,s)=V_{Q_{\bigidx}}^{H_{\bigidx}}(\tau,s)$
(see Section~\ref{subsection:sections}). An element
$f_s\in V(\tau,s)$ is regarded as a function on
$H_{\bigidx}\times\GL{\bigidx}$, where for any $h\in H_{\bigidx}$
the function $b\mapsto f_s(h,b)$ lies in $\Whittaker{\tau}{\psi}$.

The form of the integral depends on the 
size of $\smallidx$ relative to $\bigidx$.
\begin{definition}\label{definition:shape of local integral}
Fix $s\in\C$ and let $W\in\Whittaker{\pi}{\psi_{\gamma}^{-1}}$ and
$f_s\in V(\tau,s)$.
\begin{enumerate}
\item\label{definition:shap l <= n}
For $\smallidx\leq\bigidx$ the integral is
\begin{align*}
\Psi(W,f_s,s)=\int_{\lmodulo{U_{G_{\smallidx}}}{G_{\smallidx}}} W(g)
\int_{R_{\smallidx,\bigidx}}f_s(w_{\smallidx,\bigidx}rg,1)\psi_{\gamma}(r)drdg,
\end{align*}
where
\begin{align*}
&w_{\smallidx,\bigidx}=\left(\begin {array}{ccccc}
   &  \gamma I_{\smallidx} &  &   &  \\
    &   &  &   &   I_{\bigidx-\smallidx} \\
   &   & (-1)^{\bigidx-\smallidx} &   &  \\
 I_{\bigidx-\smallidx}  &   &  &  &  \\
  &   &  &  \gamma ^{-1} I_{\smallidx} &
\end{array}\right)\in H_{\bigidx},\\
&R_{\smallidx,\bigidx}=\{\left(\begin {array}{ccccc}
  I_{\bigidx-\smallidx} & x  & y & 0  & z \\
   & I_{\smallidx}  & 0 & 0  &  0 \\
   &   & 1 &  0 & y' \\
   &   &  &  I_{\smallidx} & x' \\
   &   &  &   & I_{\bigidx-\smallidx}
\end{array}\right)\}<H_{\bigidx}
\end{align*}
and $\psi_{\gamma}(r)=\psi(r_{\bigidx-\smallidx,\bigidx})$ is the
restriction of the character $\psi_{\gamma}$ of
$N_{\bigidx-\smallidx}$ to $R_{\smallidx,\bigidx}$ (see
Section~\ref{subsection:The global integral for l<=n}).
\item\label{definition:shap l > n}
For $\smallidx>\bigidx$,
\begin{align*}
\Psi(W,f_s,s)=\int_{\lmodulo{U_{H_{\bigidx}}}{H_{\bigidx}}}
(\int_{R^{\smallidx,\bigidx}}W(rw^{\smallidx,\bigidx}h)dr)f_s(h,1)dh.
\end{align*}
Here
\begin{align*}
&w^{\smallidx,\bigidx}=\left(\begin {array}{ccccc}
   &  I_{\bigidx} &  &   &  \\
  I_{\smallidx-\bigidx-1}  &   &  &   &    \\
   &   & I_2 &   &  \\
   &   &  &  & I_{\smallidx-\bigidx-1} \\
  &   &  & I_{\bigidx}  &
\end{array}\right)\in G_{\smallidx},\\
&R^{\smallidx,\bigidx}=\{\left(\begin {array}{ccccc}
  I_{\bigidx} &   &  &   &  \\
  x & I_{\smallidx-\bigidx-1}  &  &   &   \\
   &   & I_2 &   &  \\
   &   &  &  I_{\smallidx-\bigidx-1} &  \\
   &   &  &  x' & I_{\bigidx}
\end{array}\right)\}<G_{\smallidx}.
\end{align*}
\end{enumerate}
\end{definition}

For $f_s\in\xi(\tau,hol,s)=\xi_{Q_{\bigidx}}^{H_{\bigidx}}(\tau,hol,s)$ these integrals are absolutely
convergent for $\Re(s)>>0$, i.e.,
\begin{align*}
&\int_{\lmodulo{U_{G_{\smallidx}}}{G_{\smallidx}}} |W|(g)
\int_{R_{\smallidx,\bigidx}}|f_s|(w_{\smallidx,\bigidx}rg,1)drdg<\infty,\\
&\int_{\lmodulo{U_{H_{\bigidx}}}{H_{\bigidx}}}
(\int_{R^{\smallidx,\bigidx}}|W|(rw^{\smallidx,\bigidx}h)dr)|f_s|(h,1)dh<\infty.
\end{align*}
Moreover, there is some constant $s_0>0$ which depends only on the
representations $\pi$ and $\tau$, such that for all $\Re(s)>s_0$ the
integrals are absolutely convergent. The proofs of these facts
are similar to the proofs of convergence of Soudry
\cite{Soudry} (Section~4) and we provide a sketch below. The integrals have a meromorphic
continuation to
functions in $\C(q^{-s})$ - refer to 
Section~\ref{subsection:meromorphic continuation}. 
Additionally, the
integrals are \nontrivial, e.g. there exist data such that for all
$s$, $\Psi(W,f_s,s)\equiv1$ (see
Proposition~\ref{proposition:integral
can be made constant}). 

\begin{proposition}\label{proposition:convergence}
There is a constant $s_0>0$ depending only on the representations,
such that $\Psi(W,f_s,s)$ is absolutely convergent for all $s$ with $\Re(s)>s_0$,
$W\in\Whittaker{\pi}{\psi_{\gamma}^{-1}}$ and $f_s\in V(\tau,s)$.
\end{proposition}
\begin{proof}[Proof of Proposition~\ref{proposition:convergence}] 
We may already take $s\in\R$. First assume $\smallidx\leq\bigidx$. According to the Iwasawa decomposition of $G_{\smallidx}$,
and since $W$ and $f_s$ are smooth, we need to bound

\begin{align*}
\int_{A_{\smallidx-1}}\int_{G_1}|W|(ax)\int_{R_{\smallidx,\bigidx}}|f_s|(w_{\smallidx,\bigidx}rax,1)\delta_{B_{G_{\smallidx}}}^{-1}(a)drdxda.
\end{align*}

In the \quasisplit\ case, or if $[x]$ is bounded ($[x]$ was defined in Section~\ref{subsection:G_l in H_n}), the $dx$-integration
may be ignored. Thus we begin with the case $[x]>q^k$ where $k$ is chosen
according to Lemma~\ref{lemma:torus elements in double cosets}, applied with $k_0$ such that
$f_s$ is right-invariant by $\mathcal{N}_{H_{\bigidx},k_0}$. We can write the integral as a sum
of two integrals, the first over
$G_1^{0,k}=\setof{x=diag(b,b^{-1})\in G_1}{|b|^{-1}=[x]>q^k}$, the second over the set of $x$ with
$|b|=[x]>q^k$. Both are treated similarly, so we consider only the first.

Decompose $x=m_xu_xh_1n_x$ as in Lemma~\ref{lemma:torus elements in double cosets},
with $m_x=diag(b,1,b^{-1})$ and $n_x\in\mathcal{N}_{H_{\bigidx},k_0}$. As an element of $H_{\bigidx}$,
$a=diag(I_{\bigidx-\smallidx},a,I_3,a^*,I_{\bigidx-\smallidx})$. Then \\ $am_x=diag(I_{\bigidx-\smallidx},a,b,1,b^{-1},a^*,I_{\bigidx-\smallidx})$ normalizes $R_{\smallidx,\bigidx}$. We get
\begin{align*}
&\int_{A_{\smallidx-1}}\int_{G_1^{0,k}}|W|(ax)
(\absdet{a}\cdot[x]^{-1})^{\smallidx-\half\bigidx+s-\half}\\\notag&\int_{R_{\smallidx,\bigidx}}|f_s|(w_{\smallidx,\bigidx}rh_1,diag(a,\lfloor x \rfloor,I_{\bigidx-\smallidx}))drdxda.
\end{align*}
Since $\rconj{w_{\smallidx,\bigidx}^{-1}}R_{\smallidx,\bigidx}<\overline{U_{\bigidx}}$, it is enough to establish the convergence of
\begin{align*}
&\int_{A_{\smallidx-1}}\int_{G_1^{0,k}}|W|(ax)
(\absdet{a}\cdot[x]^{-1})^{\smallidx-\half\bigidx+s-\half}\int_{\overline{U_{\bigidx}}}|f_s|(u,diag(a,\lfloor x \rfloor,I_{\bigidx-\smallidx}))dudxda
\end{align*}
for any $f_s$. In fact, it is enough to consider $f_s=ch_{kN,W',s}$ where $k\in K_{H_{\bigidx}}$, $N<K$ is compact open and $W'\in\Whittaker{\tau}{\psi}$ (see Section~\ref{subsection:sections}). Note that if $G_{\smallidx}$ is \quasisplit\ or $[x]$ is bounded,
we reach a similar integral without the $dx$-integration.
In order to bound the $du$-integration, consider the Iwasawa decomposition $u=tvk$ for $u\in\overline{U_{\bigidx}}$, where
$t\in A_{\bigidx}\isomorphic T_{H_{\bigidx}}$, $v\in U_{H_{\bigidx}}$ and $k\in K_{H_{\bigidx}}$.
An analogue of Lemma~4.5 of \cite{Soudry} implies that $\absdet{t}$ and the simple roots of $\GL{\bigidx}$ evaluated at
$t$ are bounded from above and below.
Then the triple integral is bounded using the Whittaker expansion formula of Section~\ref{section:whittaker props}
for $W$ and $W'$.

In the case $\smallidx>\bigidx$ the argument follows exactly as in Section~4.2 of \cite{Soudry}: start with the Iwasawa decomposition for $H_{\bigidx}$,
then note that the $dr$-integration can be ignored, because if $a\in\GL{\bigidx}<L_{\bigidx}$, the mapping $r\mapsto W(ar)$
($r\in R^{\smallidx,\bigidx}$) has a compact support. Again one uses the asymptotic expansion of Whittaker functions
for $W$ and $W'$.
\end{proof} 

We note that a detailed proof of the absolute convergence can be obtained by collecting
the results we will present in this chapter. See Remark~\ref{remark:tate integrals for proving convergence}.

\section{Special $f_s$ and $W$}\label{subsection:special f and W}
The following two results describe a section and a Whittaker
function that will be used repeatedly to study the integrals (see e.g. Proposition~\ref{proposition:integral can be made constant}). 
\begin{lemma}\label{lemma:f with small support}
Assume $\smallidx\leq\bigidx$. There is a global constant $k_0\geq0$
depending on the embedding of $G_{\smallidx}$ in $H_{\bigidx}$ such
that the following holds. Let $W_{\tau}\in\Whittaker{\tau}{\psi}$
and $k>k_0$ be such that $W_{\tau}$ is right-invariant by
$\mathcal{N}_{\GL{\bigidx},k-k_0}$ and $\psi(\mathcal{P}^k)\equiv1$.
Then there exists $f_s\in\xi(\tau,std,s)$ such that for any
$b\in\GL{\bigidx}$ and
\begin{align*}
v=\left(\begin{array}{ccc}I_{\smallidx-1}\\y&x\\z&y'&I_{\smallidx-1}\\\end{array}\right)\in
(\overline{V_{\smallidx-1}}\rtimes G_1)< G_{\smallidx}\qquad
(y\in\Mat{2\times\smallidx-1},x\in G_1),
\end{align*}
\begin{align*}
&\int_{R_{\smallidx,\bigidx}}f_s(w_{\smallidx,\bigidx}rv,b)\psi_{\gamma}(r)dr=\begin{cases}
|\gamma|^{\smallidx(\half\bigidx+s-\half)}W_{\tau}(bt_{\gamma})&
v\in O_k,\\0&\text{otherwise}.
\end{cases}
\end{align*}
 Here
\begin{align*}
t_{\gamma}=\begin{cases}
diag(\gamma I_{\smallidx},I_{\bigidx-\smallidx})&|\gamma|\ne1,\\
I_{\bigidx}&\text{otherwise}\end{cases}
\end{align*}
and $O_k\subset \mathcal{N}_{G_{\smallidx},k-k_0}$ is a measurable
subset of positive measure ($vol(O_k)>0$) given in the proof.
\end{lemma}
\begin{proof}[Proof of Lemma~\ref{lemma:f with small support}] 
Put $w=t_{\gamma}^{-1}w_{\smallidx,\bigidx}$. The integral becomes
\begin{align*}
|\gamma|^{\smallidx(\half\bigidx+s-\half)}\int_{R_{\smallidx,\bigidx}}f_s(wrv,bt_{\gamma})\psi_{\gamma}(r)dr.
\end{align*}

Let $N=\mathcal{N}_{H_{\bigidx},k}$. Note that $w$ normalizes $N$.
Take $f_s=ch_{wN,cW_{\tau},s}$, where
$c>0$ is a volume constant depending on $k$ (to be chosen later).
Then $\support{f_s}= Q_{\bigidx}wN$ and for all
$a\in\GL{\bigidx}\isomorphic M_{\bigidx}$, $u\in U_{\bigidx}$ and
$y\in N$,
\begin{align*}
f_s(auwy,1)=c\delta_{Q_{\bigidx}}^{\half}(a)\absdet{a}^{s-\half}W_{\tau}(a).
\end{align*}
Now $f_s(wrv,bt_{\gamma})$ vanishes unless
$\rconj{w^{-1}}(rv)\in Q_{\bigidx}N$. If we denote
$[x]_{\mathcal{E}_{H_1}}=h=(h_{i,j})$ (see Section~\ref{subsection:G_l in H_n}), $\rconj{w}(rv)$ is of the
form
\begin{align}\label{shape:matrix rv conjugated}
\left(\begin {array}{ccccccc}
 & \text{\huge *}  &  &  &  & \text{\huge *} &  \\
 *  & * & * & *  & I_{\bigidx-\smallidx} & * & 0   \\
* & h_{3,1}  & * & (-1)^{\bigidx-\smallidx}h_{3,2} & 0 & h_{3,3} &0\\
 * & * & * & * & 0  & * & I_{\smallidx-1} \\
\end{array}\right).
\end{align}

The point is that in order to have $\rconj{w^{-1}}(rv)\in
Q_{\bigidx}N$, it is necessary that $h\in N$. Then the coordinates
of $x$ in $v$ may be ignored. After this it follows that the rest of
the non-constant coordinates of $v$ and $r$ are small.

Looking at \eqref{shape:matrix rv conjugated}, by our choice of $N$
one sees that if it belongs to $Q_{\bigidx}N$ then $h_{3,3}\ne0$ and
\begin{align*}
|h_{3,1}h_{3,3}^{-1}|,|h_{3,2}h_{3,3}^{-1}|\leq q^{-k}.
\end{align*}
To continue from here we analyze the specific coordinates of $h$.
Recall that we may assume $|\beta|=1$ and in the \quasisplit\ case
$|\rho|\geq1$ (but it is possible that $|2|<1$). We provide the
details for the split case first.

For
$x=\bigl(\begin{smallmatrix}\alpha&\\&\alpha^{-1}\end{smallmatrix}\bigr)$,
\begin{align*}
h=\left(\begin{array}{ccc} \half+\quarter (\alpha+ \alpha^{-1}) &
\frac 1 {2\beta}(\alpha-\alpha^{-1}) & \frac 2 {\beta^2}(\half-\quarter(\alpha+\alpha^{-1}))\\
\quarter\beta (\alpha-\alpha^{-1}) & \half (\alpha + \alpha^{-1}) & -\frac 1 {2\beta} (\alpha-\alpha^{-1})\\
\half\beta^2(\half-\quarter(\alpha+\alpha^{-1}))&-\quarter\beta
(\alpha-\alpha^{-1}) & \half+\quarter (\alpha + \alpha^{-1})
\\\end{array}\right)
\end{align*}
(see \eqref{eq:embedding of SO_2 split}).
Set $\xi=\quarter(\alpha+\alpha^{-1})$. Since
$|h_{3,1}h_{3,3}^{-1}|\leq q^{-k}$ we derive
$|\half-\xi|<|\half+\xi|$, whence
\begin{align*}
|h_{3,3}|=|\half+\xi|=|\half-\xi+\half+\xi|=1.
\end{align*}
Therefore $|h_{3,1}|,|h_{3,2}|\leq q^{-k}$ and looking at the other
coordinates we see that for $i\ne j$, $|h_{i,j}|\leq q^{-k}$. Also
from $|h_{3,1}|\leq q^{-k}$ we obtain
$\alpha+\alpha^{-1}\in2+8\mathcal{P}^k$ whence
$h_{i,i}\in1+\mathcal{P}^k$. This proves
$h\in\mathcal{N}_{H_1,k}$.

We show that $x\in\mathcal{N}_{G_1,k}$. First observe that $|\alpha|=1$. Otherwise
$[x]=\max(|\alpha|,|\alpha|^{-1})>1$ and 
we reach a contradiction to $|h_{3,2}h_{3,3}^{-1}|\leq q^{-k}$,
because $|h_{3,2}h_{3,3}^{-1}|=|h_{3,2}|=|\quarter|[x]>1$. Since
$|h_{3,2}|\leq q^{-k}$, $\alpha-\alpha^{-1}\in4\mathcal{P}^k$ and
because $|\alpha|=1$ we find that $(\alpha-1)(\alpha+1)=\alpha^2-1\in4\mathcal{P}^k$. If
$\alpha+1\in4\mathcal{P}$, using $\alpha^{-1}\in\alpha+4\mathcal{P}^k$ we get $|h_{3,1}|=|\gamma|$ -
contradiction (since $q^{-k}\geq|h_{3,1}|=|\gamma|\geq1$ and $k>0$).
Hence $|\alpha+1|>|4|q^{-1}$ and then $(\alpha-1)(\alpha+1)\in4\mathcal{P}^k$ implies
$|\alpha-1|\leq q^{-k}$, i.e. $\alpha\in1+\mathcal{P}^k$, therefore $x\in\mathcal{N}_{G_1,k}$.

Now we explain the \quasisplit\ case. Let
$x=\bigl(\begin{smallmatrix}a&b\rho\\b&a\end{smallmatrix}\bigr)$
(with
$a^2-b^2\rho=1$). Here $h_{3,1}=\half\gamma(1-a)$, $h_{3,2}=-\gamma
b$ and $h_{3,1}h_{3,3}^{-1}=\gamma\frac{1-a}{1+a}$ (see
Section~\ref{subsection:G_l in H_n}). If $|1-a|=|1+a|$,
$|\gamma|=|h_{3,1}h_{3,3}^{-1}|\leq q^{-k}$ -
contradiction ($k>0$). Hence $|a|=1$. 
Now $|h_{3,1}h_{3,3}^{-1}|\leq
q^{-k}$ implies
\begin{align*}
|1-a|\leq|\gamma|^{-1}|1+a|q^{-k}\leq|\gamma|^{-1}q^{-k}
\end{align*}
 ($|1+a|\leq1$ since $|a|=1$). Therefore
 $a\in1+\gamma^{-1}\mathcal{P}^k$ and $h_{3,3}=\half(1+a)\in1+\rho^{-1}\mathcal{P}^k$
 ($\gamma^{-1}=2\rho^{-1}$). In particular $|h_{3,3}|=1$ (since $|\rho|\geq1$), then $|h_{3,1}|\leq
 q^{-k}$ and
from $|h_{3,2}h_{3,3}^{-1}|\leq q^{-k}$ we get
$b\in\gamma^{-1}\mathcal{P}^k$. Again, inspecting the coordinates of
$h$ we find that for $i\ne j$, $|h_{i,j}|\leq q^{-k}$. On the diagonal we
have elements in $1+\rho^{-1}\mathcal{P}^k$ and
$1+\gamma^{-1}\mathcal{P}^k$. Because $|\rho|\geq1$, both subgroups
are contained in $1+\mathcal{P}^k$. Hence $h\in\mathcal{N}_{H_1,k}$ and we have also shown
$x\in\mathcal{N}_{G_1,k}$.

Since whether $G_{\smallidx}$ is split or not,
$h\in\mathcal{N}_{H_1,k}<N$, it is evident that
\begin{align*}
\rconj{w^{-1}}(rv)\in Q_{\bigidx}N\iff \rconj{w^{-1}}(rv_0)\in
Q_{\bigidx}N,\qquad
v_0=\left(\begin{array}{cccc}I_{\smallidx-1}\\y_1&1\\y_2&0&1\\z&*&*&I_{\smallidx-1}\end{array}\right).
\end{align*}
Write
\begin{align*}
r=\left(\begin{array}{ccccc}I_{\bigidx-\smallidx}&r_1&r_2&0&r_3\\&I_{\smallidx}&&&0\\&&1&&r_2'\\&&&I_{\smallidx}&r_1'\\&&&&I_{\bigidx-\smallidx}\end{array}\right),\qquad
[v_0]_{\mathcal{E}_{H_{\smallidx}}}=\left(\begin{array}{ccccc}I_{\smallidx-1}\\u_1&1\\u_2&&1\\u_3&&&1\\z&u_3'&u_2'&u_1'&I_{\smallidx-1}\end{array}\right).
\end{align*}
Also denote $r_1=(r_{1,1},r_{1,2})$ where
$r_{1,1}\in\Mat{\bigidx-\smallidx\times\smallidx-1}$. Then
$\rconj{w^{-1}}(rv_0)=mu$ for $m\in M_{\bigidx}$ and
$u\in\overline{U_{\bigidx}}$. Specifically, $m$ is the image of
\begin{align*}
\left(\begin{array}{ccc}I_{\smallidx-1}\\u_1&1\\&&I_{\bigidx-\smallidx}\end{array}\right)\in\GL{\bigidx}
\end{align*}
and
\begin{align*}
u=\left(\begin{array}{ccccccc}
I_{\smallidx-1}\\&1\\&&I_{\bigidx-\smallidx}\\\epsilon
u_2&0&\epsilon
r_2'&1\\r_{1,1}+r_{1,2}u_1+r_2u_2&r_{1,2}&r_3&\epsilon
r_2&I_{\bigidx-\smallidx}\\u_3&0&r_{1,2}'&0&0&1\\z-u_1'u_3&u_3'&r_{1,1}'-u_1'r_{1,2}'&
\epsilon u_2'&0&0&I_{\smallidx-1}\end{array}\right).
\end{align*}
Then $mu\in Q_{\bigidx}N$ if and only if $u\in N$. It follows that
the coordinates of the lower-left $\bigidx\times(\bigidx+1)$-block
of $u$ lie in $\mathcal{P}^k$. In particular
$u_2,u_3\in\mathcal{P}^k$, with a minor abuse of notation (e.g.
$u_2\in\mathcal{P}^k$ instead of
$u_2\in\MatF{1\times\smallidx-1}{\mathcal{P}^k}$). In the split case
\begin{align*}
u_1=y_1-\frac1{2\beta^2}y_2,\quad u_2=\beta
y_1+\frac1{2\beta}y_2,\quad u_3=-\gamma y_1+\quarter y_2.
\end{align*}
Hence $y_1\in\half\mathcal{P}^k$, $y_2\in\mathcal{P}^k$ and
$u_1\in\half\mathcal{P}^k$. Let $k_0\geq0$ be the valuation of $2$.
Then $u_1'u_3\in\half\mathcal{P}^{2k}\subset\mathcal{P}^k$ (because
$k>k_0$) whence $z\in\mathcal{P}^k$. 
It also follows that $r_{1,1}\in\mathcal{P}^k$ and
$r\in N$. Additionally 
$W_{\tau}(bt_{\gamma}m)=W_{\tau}(bt_{\gamma})$, because
$m\in\mathcal{N}_{\GL{\bigidx},k-k_0}$.

For the \quasisplit\ case the argument is similar. For instance
$u_1=y_2$, $u_2=y_1$ and $u_3=-\gamma y_2$ hence
$y_1\in\mathcal{P}^k$,
$u_1,y_2\in\gamma^{-1}\mathcal{P}^k\subset\mathcal{P}^k$.

Thus we conclude $f_s(wrv,bt_{\gamma})=0$ unless
$h\in\mathcal{N}_{H_1,k}$, $u_2,u_3,z\in\mathcal{P}^k$ and
$r\in N$. 
Conversely, if
$h\in\mathcal{N}_{H_1,k}$, $u_2,u_3,z\in\mathcal{P}^k$ and
$r\in N$, 
$f_s(wrv,bt_{\gamma})=W_{\tau}(bt_{\gamma})$. The subset
$O_k\subset\mathcal{N}_{G_{\smallidx},k-k_0}$ is defined as the
pre-image in $G_{\smallidx}$ of the set
\begin{align*}
\setof{\left(\begin{array}{c|ccc|c}I_{\smallidx-1}&&&&\\\hline u_1&&&&\\u_2&&h&&\\u_3&&&&\\\hline z&u_3'&u_2'&u_1'&I_{\smallidx-1}\end{array}\right)}{h\in\mathcal{N}_{H_1,k},u_2,u_3\in\MatF{1\times\smallidx-1}{\mathcal{P}^k}
,z\in\MatF{\smallidx-1\times\smallidx-1}{\mathcal{P}^k}}.
\end{align*}

This already shows that the integration over $v$ reduces to some
positive constant. Regarding the integration over $r$, the character
$\psi_{\gamma}$ is trivial since
$\psi_{\gamma}(r)=\psi((r_{1,2})_{\bigidx-\smallidx})$ (and
$\psi(\mathcal{P}^k)\equiv1$). Then the $drdv$-integration reduces
to a positive (\nonzero) constant $c^{-1}$.
\end{proof} 

The following lemma is a straightforward adaptation of a similar
claim of Soudry \cite{Soudry} (inside the proof of Proposition~6.1).
\begin{lemma}\label{lemma:W with small support}
Let $0\leq j<\smallidx$ and
$W_0\in\Whittaker{\pi}{\psi_{\gamma}^{-1}}$. For any $k$ large
enough (depending on $W_0$) there exists
$W\in\Whittaker{\pi}{\psi_{\gamma}^{-1}}$ such that for any
\begin{align*}
v=\left(\begin{array}{ccccc}a\\u&b\\&&I_2\\
&&&b^*\\&&&u'&a^*\end{array}\right) \qquad(a\in\GL{j},b\in
\overline{B_{\GL{\smallidx-j-1}}}),
\end{align*}
\begin{align*}
W(v)=
\begin{dcases}
W_0(\left(\begin{array}{ccc}a\\&I_{2(\smallidx-j)}\\&&a^*\end{array}\right))&
\left(\begin{array}{ccccc}I_j\\u&b\\&&I_2\\&&&b^*\\&&&u'&I_j\\\end{array}\right)
\in \mathcal{N}_{G_{\smallidx},k},\\0 &\text{otherwise}.
\end{dcases}
\end{align*}
In addition we can take $W$ as above which vanishes unless the last
row of $a$ lies in $\eta_j+\MatF{1\times j}{\mathcal{P}^{k}}$
($\eta_j=(0,\ldots,0,1)$).
\end{lemma}
The essence of the proof is the following. For a compact subgroup
$C<U_{G_{\smallidx}}$, define
$W\in\Whittaker{\pi}{\psi_{\gamma}^{-1}}$ by
$W(g)=\int_CW_0(gc)\psi_{\gamma}(c)dc$. Then if $x\in G_{\smallidx}$
satisfies $\rconj{x^{-1}}C<U_{G_{\smallidx}}$,
$W(x)=W_0(x)\int_C\psi_{\gamma}^{-1}(\rconj{x^{-1}}c)\psi_{\gamma}(c)dc$.
The proof follows by applying this argument inductively, starting with
$C$ such that the $dc$-integration vanishes unless the
last row of $(u|b)$ belongs to
$\eta_{\smallidx-1}+\MatF{1\times\smallidx-1}{\mathcal{P}^{k}}$. The
proof is omitted.

\section{The inner integration over
$R_{\smallidx,\bigidx}$}\label{subsection:the integral over
R_{\smallidx,\bigidx}} In the case $\smallidx<\bigidx$, the integral
$\Psi(W,f_s,s)$ contains an inner integration over the unipotent
subgroup
$R_{\smallidx,\bigidx}$. 
The properties of this integration resemble those of the Whittaker
functional, proved by Casselman and Shalika \cite{CS2} (see also
\cite{Sh2,Sh4}) and we follow their line of arguments.

Define a functional on $V(\tau,s)$ by
\begin{align*}
f_s\mapsto\Omega(f_s)=\int_{R_{\smallidx,\bigidx}}f_s(w_{\smallidx,\bigidx}r,1)\psi_{\gamma}(r)dr.
\end{align*}
This integral is absolutely convergent if it is convergent when we replace $f_s$ with $|f_s|$ and drop the character $\psi_{\gamma}$.
There exists a constant $s_1>0$ depending only on $\tau$ such that
if $\Re(s)>s_1$, this integral is absolutely convergent for all $f_s$ (see \cite{Soudry} Section~4.5). We confine ourselves to $s$
with $\Re(s)$ in this right half-plane. However, we will show in the
next few paragraphs that as a function of $s$, $\Omega(f_s)$ has an
analytic continuation to a polynomial in $\C[q^{-s},q^s]$, by which
it is defined for all $s$.

According to the (global) construction of the integral in
Section~\ref{subsection:The global integral for l<=n} or by a
direct verification,
\begin{align}\label{eq:omega in homspace}
\Omega(f_s)\in\homspace{N_{\bigidx-\smallidx}}{V(\tau,s)}{\psi_{\gamma}^{-1}}.
\end{align}

The following calculation will be used repeatedly. For $b\in\GL{k}$,
$k<\smallidx$, where $\GL{k}<L_{k}$ 
and $L_k$ is embedded in $H_{\bigidx}$ through $G_{\smallidx}$,
\begin{align}\label{eq:omega right translation by GL l-1}
\Omega(b\cdot
f_s)=\absdet{b}^{\smallidx-\bigidx+s-\half}\delta_{Q_{\bigidx}}^{\half}(b)\Omega(\tau(b)f_s)
=\absdet{b}^{\smallidx-\half\bigidx+s-\half}\Omega(\tau(b)f_s).
\end{align}
This follows because $L_{\smallidx-1}$ 
normalizes $R_{\smallidx,\bigidx}$, $G_{\smallidx}$ fixes $\psi_{\gamma}$ and
$\rconj{w_{\smallidx,\bigidx}^{-1}}L_{\smallidx-1}<M_{\bigidx}$.

Define for a compact open subgroup $N<N_{\bigidx-\smallidx}$ and
$f_s\in V(\tau,s)$ the function $f_s^{N,\psi_{\gamma}}\in V(\tau,s)$
by
\begin{align*}
f_s^{N,\psi_{\gamma}}=vol(N)^{-1}\int_{N}\psi_{\gamma}(n)n\cdot f_sdn.
\end{align*}
The function $f_s^{N,\psi_{\gamma}}$ is just a linear combination of
translations of $f_s$. The following claim shows that $\Omega$
(defined for $\Re(s)>s_1$) is invariant for such a twist of $f_s$.
\begin{claim}\label{claim:omega invariant for compact averaging}
For any $f_s\in V(\tau,s)$, compact open subgroup
$N<N_{\bigidx-\smallidx}$ and $g\in G_{\smallidx}$, $\Omega(g\cdot
f_s^{N,\psi_{\gamma}})=\Omega(g\cdot f_s)$.
\end{claim}
\begin{proof}[Proof of Claim~\ref{claim:omega invariant for compact averaging}] 
\begin{align*}
\Omega(g\cdot f_s^{N,\psi_{\gamma}})&=\int_{R_{\smallidx,\bigidx}}vol(N)^{-1}
\int_{N}f_s(w_{\smallidx,\bigidx}rgn,1)\psi_{\gamma}(n)\psi_{\gamma}(r)dndr\\
&=vol(N)^{-1}\int_{N}\int_{R_{\smallidx,\bigidx}}
f_s(w_{\smallidx,\bigidx}r(\rconj{g^{-1}}n)g,1)\psi_{\gamma}(n)\psi_{\gamma}(r)drdn\\
&=vol(N)^{-1}\int_{N}\int_{R_{\smallidx,\bigidx}}
f_s(w_{\smallidx,\bigidx}rg,1)\psi_{\gamma}^{-1}(\rconj{g^{-1}}n)\psi_{\gamma}(n)\psi_{\gamma}(r)drdn\\
&=vol(N)^{-1}\int_{N}\int_{R_{\smallidx,\bigidx}}
f_s(w_{\smallidx,\bigidx}rg,1)\psi_{\gamma}^{-1}(n)\psi_{\gamma}(n)\psi_{\gamma}(r)drdn\\
&=\Omega(g\cdot f_s).
\end{align*}
The order of integration may be changed because the integral is
absolutely  convergent. 
The third equality follows
from \eqref{eq:omega in homspace} and the forth from the fact that
$G_{\smallidx}$ stabilizes $\psi_{\gamma}$.
\end{proof}
Note that $N_{\bigidx-\smallidx}$ is exhausted by its compact open
subgroups, meaning that any compact subset $C\subset
N_{\bigidx-\smallidx}$ is contained  in some compact open subgroup
$N<N_{\bigidx-\smallidx}$.

Fix $s\in\C$ arbitrarily (i.e., not necessarily with $\Re(s)$ large).
We use the filtration of $V(\tau,s)$ according to the Geometrical
Lemma (\cite{BZ2}, 2.12). We follow the exposition of Mui\'{c}
\cite{Mu} (Section~3, see also \cite{BZ2,Cs}). Consider the
decomposition $H_{\bigidx}=\dotcup_{w\in\mathcal{A}}C(w)$ ($\dotcup$
- a disjoint union), where $\mathcal{A}$ is a set of representatives
for
$\rmodulo{\lmodulo{Q_{\bigidx}}{H_{\bigidx}}}{Q_{\bigidx-\smallidx}}$
and $C(w)=Q_{\bigidx}wQ_{\bigidx-\smallidx}$. The set $\mathcal{A}$
was given explicitly in the proof of Proposition~\ref{propo:basic
global identity l <= n},
\begin{align*}
\mathcal{A}=\setof{w_r=\left(\begin {array}{ccccccc}
  I_r &  &  &  & & & \\
   & &  &  &  & I_{\bigidx-\smallidx-r} & \\
   &  &  I_{\smallidx} &  &  & & \\
   &  &  & (-1)^{\bigidx-\smallidx-r} &  &  &\\
   &  &   &  & I_{\smallidx} & & \\
   & I_{\bigidx-\smallidx-r} &  &  & & & \\
    & & &  &  &  & I_r\end{array}\right)}{r=\intrange{0}{\bigidx-\smallidx}}.
\end{align*}
We also calculated $Q_{\bigidx}^{w_{r}}=\rconj{w_{r}}Q_{\bigidx}\cap
Q_{\bigidx-\smallidx}$ (see \eqref{eq:Qn wr}),
\begin{align*}
Q_{\bigidx}^{w_r}=\{\left(\begin {array}{cc|ccc|cc}
  a & x & y_1 & y_2 & y_3 & z_1 & z_2 \\
    & b & 0   & 0   & y_4 & 0   & z_1' \\ \hline
    &   &    &    &    & y_4'& y_3' \\
    &   &     & Q_{\smallidx}'  &   & 0   & y_2' \\
    &   &     &     &  & 0   & y_1' \\\hline
    &   &     &     &     & b^* & x' \\
    &   &     &     &     &     & a^*
\end{array}\right)\},
\end{align*}
where
$a\in\GL{r},b\in\GL{\bigidx-\smallidx-r},y_1\in\Mat{r\times\smallidx},y_2\in\Mat{r\times1}$.
For an algebraic variety $X$ defined over the field $F$, denote by $\dim(X)$ its dimension. We
have
\begin{align*}
\dim (Q_{\bigidx}^{w_r})=\half r^2+r\smallidx+\half
r+\bigidx^2-\bigidx\smallidx+\dim(Q_{\smallidx}').
\end{align*}
Since
$\dim{C(w_r)}=\dim(Q_{\bigidx})+\dim(Q_{\bigidx-\smallidx})-\dim(Q_{\bigidx}^{w_r})$ (see e.g. the proof of Lemma~\ref{lemma:integration formula for quotient space H_n G_n+1}),
we get $\dim{C(w_r)}>\dim{C(w_{r+1})}$ for all $0\leq
r<\bigidx-\smallidx$. Thus according to the special ordering defined
on the Bruhat cells, $w_0>\ldots>w_{\bigidx-\smallidx}$.

Let $C^{\geq w_r}=\dotcup_{w\geq w_r}C(w)$. In the following we
consider the elements of $V(\tau,s)$ as functions of one variable,
i.e., functions defined on $H_{\bigidx}$ taking values in $U$ - the
space of $\tau$. 
The space $V(\tau,s)$ as a representation of
$Q_{\bigidx-\smallidx}$ is filtered by the subspaces
\begin{align*}
F_{w_r}(s)=\setof{f_s\in V(\tau,s)}{\support{f_s}\subset C^{\geq
w_r}}
\end{align*}
($\support{f_s}$ denotes the support of $f_s$). For example, $F_{w_{\bigidx-\smallidx}}(s)=V(\tau,s)$.

Fix $r$ and consider also the decomposition
$Q_{\bigidx-\smallidx}=\dotcup_{\eta\in\mathcal{A}(r)}Q_n^{w_r}\eta
G_{\smallidx}N_{\bigidx-\smallidx}$, where $\mathcal{A}(r)$ is a
(finite) set of representatives for
$\rmodulo{\lmodulo{Q_{\bigidx}^{w_r}}{Q_{\bigidx-\smallidx}}}{G_{\smallidx}N_{\bigidx-\smallidx}}$.
The set $\mathcal{A}(r)$ was described in the proof of
Proposition~\ref{propo:basic global identity l <= n} (see
\eqref{set:representatives for second filtration}),
\begin{align*}
\mathcal{A}(r)=\setof{diag(b(\sigma),\xi,b(\sigma)^*)}{\sigma\in\lmodulo{(S_r\times
S_{\bigidx-\smallidx-r})}{S_{\bigidx-\smallidx}},\xi\in\rmodulo{\lmodulo{Q_{\smallidx}'}{H_{\smallidx}}}{G_{\smallidx}}}.
\end{align*}
We order the elements of $\mathcal{A}(r)$ according to the special
ordering
defined on the Bruhat cells. 

Order the set of pairs $\setof{(w_r,\eta)}{0\leq
r\leq\bigidx-\smallidx,\eta\in\mathcal{A}(r)}$ lexicographically,
i.e. $(w',\eta')>(w,\eta)$ if $w'>w$ or both $w'=w$, $\eta'>\eta$.
Let $C(w_r,\eta)=Q_{\bigidx}w_r\eta
G_{\smallidx}N_{\bigidx-\smallidx}$ and
$C^{\geq(w_r,\eta)}=\dotcup_{(w',\eta')\geq(w_r,\eta)}C(w',\eta')$.
The space $F_{w_r}(s)$ as a representation of
$G_{\smallidx}N_{\bigidx-\smallidx}$ is filtered by the subspaces
\begin{align*}
F_{w_r,\eta}(s)=\setof{f_s\in V(\tau,s)}{\support{f_s}\subset
C^{\geq(w_r,\eta)}},
\end{align*}
where $\eta$ varies over $\mathcal{A}(r)$.

The following claim demonstrates how to shrink the support of $f_s\in F_{w_r}(s)$
using a twist by a subgroup $N$.
\begin{claim}\label{claim:jacquet module vanishes for r > 0}
For any $r>0$, $\eta\in\mathcal{A}(r)$ and $f_s\in F_{w_r,\eta}(s)$
there exists a compact open subgroup $N<N_{\bigidx-\smallidx}$ such that
$f_s^{N,\psi_{\gamma}}\in F_{w',\eta'}(s)$ where
$(w',\eta')>(w_r,\eta)$. Furthermore, $N$ depends only on the
support of $f_s$ (and on $\psi_{\gamma}$).
\end{claim}
\begin{proof}[Proof of Claim~\ref{claim:jacquet module vanishes for r > 0}] 
Fix $r>0$ and $\eta$. It is enough to find $N$ such that for all
$x\in G_{\smallidx}N_{\bigidx-\smallidx}$,
\begin{align*}
f_s^{N,\psi_{\gamma}}(w_r\eta x)=vol(N)^{-1}\int_Nf_s(w_r\eta
xn)\psi_{\gamma}(n)dn=0.
\end{align*}
Regard the function
$f_s'=\lambda(\eta^{-1}w_r^{-1})(\frestrict{f_s}{C(w_r)})$ as a
function on $G_{\smallidx}N_{\bigidx-\smallidx}$. There is a compact
set $C\subset G_{\smallidx}N_{\bigidx-\smallidx}$ such that
$\support{f_s'}=(\rconj{\eta}(Q_{\bigidx}^{w_r})\cap
G_{\smallidx}N_{\bigidx-\smallidx})C$. Denote by
$C_{G_{\smallidx}}$, $C_{N_{\bigidx-\smallidx}}$ the projections of
$C$ on $G_{\smallidx}$, $N_{\bigidx-\smallidx}$ (resp.). These are
compact sets.

It is possible to show that there exists a compact subgroup
$O<\rconj{w_r\eta}U_{\bigidx}\cap N_{\bigidx-\smallidx}$ (which
depends on $r$ and $\eta$) on which
$\frestrict{\psi_{\gamma}}{O}\nequiv1$.

The subgroup $O$ was determined in the proof of
Claim~\ref{claim:unfolding of E_f to short sum}, in the global
setting. We adapt the definition to the local setting. Let
$\eta=diag(b(\sigma),\xi,b(\sigma)^*)$. If there is some $1\leq j<
\bigidx-\smallidx$ such that $\sigma(j)\leq r$ and $\sigma(j+1)>r$,
let $O$ be the image in $N_{\bigidx-\smallidx}$ of the subgroup
of $Z_{\bigidx-\smallidx}$ consisting of matrices with $1$ on the
diagonal, an arbitrary element of $\mathcal{P}^{-k}$ in the
$(j,j+1)$-th coordinate and zero elsewhere. Here $k\geq0$ is chosen
such that $\psi_{\gamma}$ is \nontrivial\ on $O$. If no such $j$
exists, we can assume
$b(\sigma)=\omega_{r,\bigidx-\smallidx-r}$ and we let
$O=\xi^{-1}O'\xi$ with $O'<N_{\bigidx-\smallidx}$ defined as follows. If $\xi=I_{2\smallidx+1}$,
\begin{align*}
O'=\setof{diag(I_{\bigidx-\smallidx-1},\left(\begin{array}{ccccccc}1&0&0&a&a&0&-\half
a^2\\&I_{\smallidx-1}&&&&&0\\&&1&&&&-a\\&&&1&&&-a\\&&&&1&&0\\&&&&&I_{\smallidx-1}&0\\&&&&&&1\end{array}\right),I_{\bigidx-\smallidx-1})}{a\in\mathcal{P}^{-k}
}.
\end{align*}
Otherwise
\begin{align*}
O'=\setof{diag(I_{\bigidx-\smallidx-1},\left(\begin{array}{ccccc}1&0&a&0&-\half
a^2\\&I_{\smallidx}&&&0\\&&1&&-a\\&&&I_{\smallidx}&0\\&&&&1\end{array}\right),I_{\bigidx-\smallidx-1})}{a\in\mathcal{P}^{-k}}.
\end{align*}

Since $N_{\bigidx-\smallidx}$ is exhausted by its compact open
subgroups, we can take a compact open subgroup
$N<N_{\bigidx-\smallidx}$ such that
$C_{N_{\bigidx-\smallidx}}\subset N$. Furthermore, using the fact that
$G_{\smallidx}$ normalizes $N_{\bigidx-\smallidx}$, we can take $N$
such that for all $g\in C_{G_{\smallidx}}$,
$O\subset\rconj{g^{-1}}N$. In more detail, let $N=N_{\bigidx-\smallidx}(\mathcal{P}^{-k})$ where $k>>0$. The set
$C_{G_{\smallidx}}$ is contained in a finite union
$\bigcup g_i\mathcal{N}_{G_{\smallidx},m}$
($g_i\in C_{G_{\smallidx}}$, $m>0$). Because $\mathcal{N}_{G_{\smallidx},m}$ normalizes $N$, it is enough to take a large
$k$ such that $O<\rconj{g_i^{-1}}N$ for each $g_i$.

We see that $N$ depends only on the
support of $f_s$ in $C(w_r,\eta)$ (and on $\psi_{\gamma}$).

Assume that for some $x\in G_{\smallidx}N_{\bigidx-\smallidx}$,
$f_s^{N,\psi_{\gamma}}(w_r\eta x)\ne0$. Then $x\in
(\rconj{\eta}(Q_{\bigidx}^{w_r})\cap
G_{\smallidx}N_{\bigidx-\smallidx})C_{G_{\smallidx}}N$ and hence it
is enough to prove that for all $g\in C_{G_{\smallidx}}$,
$f_s^{N,\psi_{\gamma}}(w_r\eta g)=0$. Since $G_{\smallidx}$
stabilizes $\psi_{\gamma}$, up to a volume constant (depending on
$g$) $f_s^{N,\psi_{\gamma}}(w_r\eta g)$ equals
\begin{align*}
\int_{\rconj{g^{-1}}N}f_s(w_r\eta ng)\psi_{\gamma}(n)dn.
\end{align*}
This vanishes because $O<\rconj{g^{-1}}N$ and we can factor the
integral through $O$ and obtain an inner integration of the
\nontrivial\ character $\frestrict{\psi_{\gamma}}{O}$ which
vanishes.
\end{proof} 

Again we treat $s$ as a parameter and study the behavior of the
functional on holomorphic sections. For $0\leq r\leq\bigidx-\smallidx$ let $F_{w_r}$ be the set of
$f_s\in\xi(\tau,hol,s)$ such that for all $s\in\C$, $f_s\in
F_{w_r}(s)$. Since for each $s$, $F_{w_r}(s)$ is a
$Q_{\bigidx-\smallidx}$-space, so is $F_{w_r}$. Also note that
$F_{w_r}$ is a $\C[q^{-s},q^s]$-module, i.e. if $f_s\in F_{w_r}$,
$P\cdot f_s\in F_{w_r}$ for $P\in\C[q^{-s},q^s]$. Similarly for
$\eta\in\mathcal{A}(r)$ we define $F_{w_r,\eta}$ as the set of
$f_s\in\xi(\tau,hol,s)$ such that for all $s\in\C$, $f_s\in
F_{w_r,\eta}(s)$. Then $F_{w_r,\eta}$ is a
$G_{\smallidx}N_{\bigidx-\smallidx}$-space and a
$\C[q^{-s},q^s]$-module. 

The above claim has the following corollary, showing how to
use twists by some $N$ in order to put holomorphic sections into
$F_{w_0}$. The key point is that $N$ will be independent of $s$.
\begin{corollary}\label{corollary:support of standard section made small}
For any $f_s\in\xi(\tau,hol,s)$ there exists a compact open subgroup
$N<N_{\bigidx-\smallidx}$, independent of $s$, such that
$f_s^{N,\psi_{\gamma}}\in F_{w_0}$.
\end{corollary}
\begin{proof}[Proof of Corollary~\ref{corollary:support of standard section made small}] 
We start with $f_s\in\xi(\tau,std,s)$. The support of standard
sections is independent of $s$, hence $f_s\in F_{w_r,\eta}$ for some
$r$ and $\eta\in\mathcal{A}(r)$. We show that there exists a subgroup $N$ independent of $s$ with
$f_s^{N,\psi_{\gamma}}\in F_{w_0}$, using induction on $r$ and $\eta$. For $r=0$ this is trivial. Let
$r>0$.

For every $s$, $f_s\in F_{w_r,\eta}(s)$. By Claim~\ref{claim:jacquet
module vanishes for r > 0} there is a subgroup $N'$, depending only
on the support of $f_s$ which is independent of $s$, such that
$f_s^{N',\psi_{\gamma}}\in F_{w',\eta'}(s)$ with
$(w',\eta')>(w_r,\eta)$. Thus $f_s^{N',\psi_{\gamma}}\in
F_{w',\eta'}$. The next step is to apply the induction hypothesis, but note that
$f_s^{N',\psi_{\gamma}}\in\xi(\tau,hol,s)$ might not be standard.

Write $f_s^{N',\psi_{\gamma}}$ as in \eqref{eq:holomorphic section
as combination of disjoint standard sections}, i.e.
$f_s^{N',\psi_{\gamma}}=\sum_{i=1}^mP_i\cdot f_s^{(i)}$ with
$0\ne P_i\in\C[q^{-s},q^s]$, $f_s^{(i)}=ch_{k_iN,v_i,s}\in\xi(\tau,std,s)$ and such
that if $k_{i_1}=\ldots=k_{i_c}$, $v_{i_1},\ldots,v_{i_c}$ are linearly
independent. Suppose that for some $i$, $f_s^{(i)}\notin
F_{w',\eta'}$. Since the support of $f_s^{(i)}$ is independent of
$s$, $f_s^{(i)}\notin F_{w',\eta'}(s)$ for all $s$. In fact, there is
some $x\in k_iN$ which does not belong to $C^{\geq (w',\eta')}$
($x\in\support{f_s^{(i)}}$). Let $\{i_1,\ldots,i_c\}$ be a maximal
set of indices such that $k_{i_1}=\ldots=k_{i_c}=k_i$. Choose $s_0$
such that $P_{i_j}(q^{-s_0},q^{s_0})=\alpha_j\ne0$ for some $1\leq
j\leq c$. Now on the one hand since $x\notin C^{\geq (w',\eta')}$,
$f_{s_0}^{N',\psi_{\gamma}}(x)=0$. On the other hand
$f_{s_0}^{N',\psi_{\gamma}}(x)=\alpha_1v_{i_1}+\ldots+\alpha_cv_{i_c}$
contradicting the fact that $v_{i_1},\ldots,v_{i_c}$ are linearly
independent. Hence for each $i$, $f_s^{(i)}\in F_{w',\eta'}$. Now
by the induction hypothesis $(f_s^{(i)})^{N_i,\psi_{\gamma}}\in F_{w_0}$,
for a subgroup $N_i$ independent of $s$.

In general
if $O_1<O_2<N_{\bigidx-\smallidx}$ are compact open,
$(f_s^{O_1,\psi_{\gamma}})^{O_2,\psi_{\gamma}}=f_s^{O_2,\psi_{\gamma}}$.
Then  if we take $N$ containing $N'$ and all of the subgroups $N_i$,
\begin{align*}
f_s^{N,\psi_{\gamma}}=
(f_s^{N',\psi_{\gamma}})^{N,\psi_{\gamma}}=(\sum_{i=1}^mP_i\cdot f_s^{(i)})^{N,\psi_{\gamma}}
=\sum_{i=1}^mP_i\cdot ((f_s^{(i)})^{N_i,\psi_{\gamma}})^{N,\psi_{\gamma}}\in F_{w_0}.
\end{align*}
This establishes the claim for standard sections.

Now if $f_s\in\xi(\tau,hol,s)$, write as above
$f_s=\sum_{i=1}^mP_i\cdot f_s^{(i)}$. For each $i$ we have $N_i$
independent of $s$ with $(f_s^{(i)})^{N_i,\psi_{\gamma}}\in F_{w_0}$,
whence for $N$ containing all of the subgroups $N_i$,
$f_s^{N,\psi_{\gamma}}\in F_{w_0}$.
\end{proof} 
The next claim implies
that for $f_s\in F_{w_0}(s)$, the integral $\Omega(f_s)$ is
absolutely convergent. Roughly, this is because
$R_{\smallidx,\bigidx}\cap
Q_{\bigidx}^{w_{0}}=\{1\}$. This claim will be used below to prove
Proposition~\ref{proposition:dr integral for torus of G_l}.
\begin{claim}\label{claim:omega f reduces to sum when support is compact}
Let $B\subset Q_{\bigidx-\smallidx}$ be a compact subset. Then
$R_{\smallidx,\bigidx}\cap Q_{\bigidx}^{w_0}B$ is contained in a compact
subset of $R_{\smallidx,\bigidx}$.
\end{claim}
\begin{proof}[Proof of Claim~\ref{claim:omega f reduces to sum when
support is compact}] 
Looking at the coordinates of $Q_{\bigidx}^{w_0}$ (see \eqref{eq:Qn
wr} or above), we may assume that any $b\in B$ is of the form
\begin{align*}
\left(\begin{array}{ccccc}I_{\bigidx-\smallidx}&y_1&y_2&0&y_3\\&I_{\smallidx}&&&0\\&&1&&y_2'\\&&&I_{\smallidx}&y_1'\\&&&&I_{\bigidx-\smallidx}\end{array}\right)
\cdot\left(\begin{array}{ccc}I_{\bigidx-\smallidx}\\&k\\&&I_{\bigidx-\smallidx}\end{array}\right)\qquad(k\in
K_{H_{\smallidx}}),
\end{align*}
where the coordinates of each $y_i$ are bounded from above.
Specifically, write $b=diag(a,qk,a^*)u\in M_{\bigidx-\smallidx}U_{\bigidx-\smallidx}$
for $a\in\GL{\bigidx-\smallidx}$, $q\in Q_{\smallidx}'$ ($Q_{\smallidx}'$ denotes the subgroup
$Q_{\smallidx}$ of $H_{\smallidx}$), $k\in K_{H_{\smallidx}}$ and $u\in U_{\bigidx-\smallidx}$. Since the
projection of $B$ on $U_{\bigidx-\smallidx}$ is compact, we can
assume that the non-constant coordinates of $u$ are bounded. Multiplying $b$ on the left by
$diag(a^{-1},q^{-1},(a^*)^{-1})\in
Q_{\bigidx}^{w_0}$ and by an element from $Q_{\bigidx}^{w_0}\cap U_{\bigidx-\smallidx}$ we see that it has the specified form.

Let
\begin{align*}
r=\left(\begin{array}{ccccc}I_{\bigidx-\smallidx}&r_1&r_2&0&r_3\\&I_{\smallidx}&&&0\\&&1&&r_2'\\&&&I_{\smallidx}&r_1'\\&&&&I_{\bigidx-\smallidx}\end{array}\right)\in
R_{\smallidx,\bigidx}\cap Q_{\bigidx}^{w_{0}}B.
\end{align*}
We can assume that there is an element $c\in
Q_{\bigidx}^{w_{0}}$ of the form
\begin{align*}
c=\left(\begin {array}{c|ccc|c}
  I_{\bigidx-\smallidx} & 0 & 0 & v & 0 \\ \hline
    &   &       &   & v' \\
    &   &     h  &   & 0 \\
    &   &       &   & 0 \\\hline
    &   &     &     &   I_{\bigidx-\smallidx}\end{array}\right)
    \qquad(v\in\Mat{\bigidx-\smallidx\times\smallidx},h\in
H_{\smallidx}),
\end{align*}
such that $cr=b\in B$. Writing $b$ as above we obtain $h=k\in
K_{H_{\smallidx}}$ and
\begin{align*}
\left(\begin{array}{ccccc}I_{\bigidx-\smallidx}&A_1&A_2&A_3&r_3+vr_1'\\&I_{\smallidx}&&&*\\&&1&&*\\&&&I_{\smallidx}&*\\&&&&I_{\bigidx-\smallidx}\end{array}\right)
=\left(\begin{array}{ccccc}I_{\bigidx-\smallidx}&y_1&y_2&0&y_3\\&I_{\smallidx}&&&0\\&&1&&y_2'\\&&&I_{\smallidx}&y_1'\\&&&&I_{\bigidx-\smallidx}\end{array}\right),
\end{align*}
with $(A_1|A_2|A_3)=(r_1|r_2|v)\cdot k^{-1}$. Now it follows that
the coordinates of $r_1,r_2$ and $v$ belong to a compact subset,
whence $r_3$ belongs to a compact subset of
$\MatF{\bigidx-\smallidx\times \bigidx-\smallidx}{F}$. This shows
that $R_{\smallidx,\bigidx}\cap
Q_{\bigidx}^{w_{0}}B$ is contained in a compact
subset of $R_{\smallidx,\bigidx}$.
\end{proof} 
Looking at the set $\mathcal{A}$ one sees that $w_0\in Q_{\bigidx}w_{\smallidx,\bigidx}$,
so in fact $C(w_0)=C(w_{\smallidx,\bigidx})$ and we may replace $w_0$
with $w_{\smallidx,\bigidx}$. Let $f_s\in\xi(\tau,std,s)$. The next
proposition shows that $\Omega(f_s)$, initially defined for
$\Re(s)>s_1$ to ensure the absolute convergence of the integral,
equals an element in $\C[q^{-s},q^s]$. Therefore $\Omega(f_s)$ has
an analytic continuation by which it can be defined for all $s$.
These results extend to $f_s\in\xi(\tau,hol,s)$. The proposition is
also the main tool in writing the Iwasawa decomposition in
Section~\ref{subsection:Iwasawa decomposition for Psi(W,f,s)}.

\begin{proposition}\label{proposition:dr integral for torus of G_l}
Let $f_s\in\xi(\tau,std,s)$. There exist
$P_i\in\C[q^{-s},q^s]$ and $W_i\in\Whittaker{\tau}{\psi}$,
$i=\intrange{1}{m}$, such
that for all 
$a\in A_{\smallidx-1}<T_{G_{\smallidx}}$ and $s$,
\begin{align}\label{eq:formula dr integral for torus of G_l no G_1}
\Omega(a\cdot
f_s)=\absdet{a}^{\smallidx-\half\bigidx+s-\half}\sum_{i=1}^mP_iW_i(
diag(a,I_{\bigidx-\smallidx+1})
).
\end{align}
In addition, for the split case there exists a constant $k>0$ such
that the following holds. For $t=ax\in
T_{G_{\smallidx}}$ with $x\in
G_1<T_{G_{\smallidx}}$ write $x=diag(b,b^{-1})$.
Then there exist $P_i\in\C[q^{-s},q^s]$ and
$W_i\in\Whittaker{\tau}{\psi}$, which depend on whether $\lfloor
x\rfloor=b$ or $\lfloor x\rfloor=b^{-1}$, such that for all $t$ satisfying $[x]>q^k$ and $s$,
\begin{align}\label{eq:formula dr integral for torus of G_l yes G_1}
\Omega(t\cdot
f_s)=(\absdet{a}[x]^{-1})^{\smallidx-\half\bigidx+s-\half}\sum_{i=1}^mP_iW_i(
diag(a,\lfloor x\rfloor,I_{\bigidx-\smallidx})
).
\end{align}
In both \eqref{eq:formula dr integral for torus of G_l no G_1} and
\eqref{eq:formula dr integral for torus of G_l yes G_1} the \lhs\ is
defined by the integral for $\Re(s)>s_1$ and the equality for all
$s$ is in the sense of analytic continuation.
\end{proposition}
\begin{proof}[Proof of Proposition~\ref{proposition:dr integral for torus of G_l}] 
Select $N$ for $f_s$ by
Corollary~\ref{corollary:support of standard section made small}.
Using Claim~\ref{claim:omega invariant for compact averaging} and
\eqref{eq:omega right translation by GL l-1}, for $\Re(s)>s_1$,
\begin{align*}
\Omega(a\cdot f_s)=\Omega(a\cdot f_s^{N,\psi_{\gamma}})
=\absdet{a}^{\smallidx-\half\bigidx+s-\half}
\int_{R_{\smallidx,\bigidx}}f_s^{N,\psi_{\gamma}}(w_{\smallidx,\bigidx}r,diag(a,I_{\bigidx-\smallidx+1}))\psi_{\gamma}(r)dr.
\end{align*}
Since $f_s^{N,\psi_{\gamma}}\in\xi(\tau,hol,s)\cap F_{w_0}$, we can write as in \eqref{eq:holomorphic section as combination of disjoint standard sections}
\begin{align}\label{eq:dr integral torus writing the compactly supported section}
f_s^{N,\psi_{\gamma}}=\sum_{i=1}^mP_if_s^{(i)}\qquad (
P_i\in\C[q^{-s},q^s],f_s^{(i)}=ch_{k_iN,v_i,s}\in\xi(\tau,std,s)).
\end{align}
As in the proof of Corollary~\ref{corollary:support of standard
section made small} we deduce that for each $i$, $f_s^{(i)}\in
F_{w_0}$. Hence each $\lambda(w_{\smallidx,\bigidx}^{-1})f_s^{(i)}$,
as a function on $Q_{\bigidx-\smallidx}$, is compactly supported
modulo $Q_{\bigidx}^{w_{\smallidx,\bigidx}}$ (similar to
$f_s^{N,\psi_{\gamma}}$) and this support is independent of $s$ (in
contrast with $f_s^{N,\psi_{\gamma}}$). I.e., there is a compact set
$B_i\subset Q_{\bigidx-\smallidx}$ such that for all $s$, the
support of $\lambda(w_{\smallidx,\bigidx}^{-1})f_s^{(i)}$ equals
$Q_{\bigidx}^{w_{\smallidx,\bigidx}}B_i$. According to
Claim~\ref{claim:omega f reduces to sum when support is compact},
$R_{\smallidx,\bigidx}\cap Q_{\bigidx}^{w_{\smallidx,\bigidx}}B_i$
is contained in a compact subset of $R_{\smallidx,\bigidx}$ ($Q_{\bigidx}^{w_{\smallidx,\bigidx}}=Q_{\bigidx}^{w_{0}}$). Let $N_i'<R_{\smallidx,\bigidx}$ be a compact open subgroup such
that $\lambda(w_{\smallidx,\bigidx}^{-1})f_s^{(i)}$ is
right-invariant by $N_i'$, $\frestrict{\psi_{\gamma}}{N_i'}\equiv1$
and $R_{\smallidx,\bigidx}\cap
Q_{\bigidx}^{w_{\smallidx,\bigidx}}B_i\subset\dotcup_{j=1}^{m_i}
r_{i,j}N_i'$ for some elements $r_{i,j}\in R_{\smallidx,\bigidx}$
and $m_i\geq1$. It follows that there exist constants $c_{i,j}\in\C$
such that for all $s$ and $a$,
\begin{align*}
\int_{R_{\smallidx,\bigidx}}f_s^{(i)}(w_{\smallidx,\bigidx}r,diag(a,I_{\bigidx-\smallidx+1}))\psi_{\gamma}(r)dr
=\sum_{j=1}^{m_i}c_{i,j}f_s^{(i)}(w_{\smallidx,\bigidx}r_{i,j},diag(a,I_{\bigidx-\smallidx+1})).
\end{align*}
In particular the integral $\Omega(a\cdot f_s^{N,\psi_{\gamma}})$ is
absolutely convergent for all $s$.
Next, there are $P_{i,j}\in\C[q^{-s},q^s]$ and
$W_{i,j}\in\Whittaker{\tau}{\psi}$ such that
\begin{align*}
c_{i,j}f_s^{(i)}(w_{\smallidx,\bigidx}r_{i,j},diag(a,I_{\bigidx-\smallidx+1}))=P_{i,j}W_{i,j}(diag(a,I_{\bigidx-\smallidx+1}))
\end{align*}
for all $s$ and $a$. Therefore we conclude
\begin{align*}
\int_{R_{\smallidx,\bigidx}}f_s^{N,\psi_{\gamma}}(w_{\smallidx,\bigidx}r,diag(a,I_{\bigidx-\smallidx+1}))\psi_{\gamma}(r)dr
=\sum_{i=1}^mP_i\sum_{j=1}^{m_i}P_{i,j}W_{i,j}(diag(a,I_{\bigidx-\smallidx+1})).
\end{align*}
This establishes \eqref{eq:formula dr integral for torus of G_l no G_1}.

Let $k_0>0$ be such that $f_s$ is right-invariant by
$\mathcal{N}_{H_{\bigidx},k_0}$ and let $k\geq k_0$ be as in
Lemma~\ref{lemma:torus elements in double cosets}. For $[x]>q^k$
write $x=m_xu_xhn_x$ as specified by the lemma, i.e.,
\begin{align*}
&m_x=\left(\begin{array}{ccc}\lfloor x\rfloor\\&1\\&&\lfloor
x\rfloor^{-1}\\\end{array}\right),\quad
u_x=\left(\begin{array}{ccc}1&c\lfloor x\rfloor^{-1}&-\half
c^2\lfloor x\rfloor^{-2}\\&1&-c \lfloor
x\rfloor^{-1}\\&&1\\\end{array}\right),\quad c^2=2\gamma^{-1}.
\end{align*}

The choice of $h$ depends on whether $|b|<|b|^{-1}$ or vice versa,
but is otherwise independent of $x$. We now assume
$|b|<|b|^{-1}$ and fix $h$. Since
$\psi_{\gamma}(r)=\psi(r_{\bigidx-\smallidx,\bigidx})$ and
$\rconj{w_{\smallidx,\bigidx}^{-1}}m_x\in M_{\bigidx}$,
\begin{align*}
\Omega(t\cdot f_s)
=&(\absdet{a}\cdot[x]^{-1})^{\smallidx-\half\bigidx+s-\half}
\\\notag&
\int_{R_{\smallidx,\bigidx}}f_s^{N,\psi_{\gamma}}(w_{\smallidx,\bigidx}ru_xh,
diag(a,\lfloor x\rfloor,I_{\bigidx-\smallidx})
)\psi(\lfloor
x\rfloor^{-1}r_{\bigidx-\smallidx,\bigidx})dr.
\end{align*}

Conjugating $r$ by $u_x$ and noting that
$\rconj{w_{\smallidx,\bigidx}^{-1}}u_x\in U_{\bigidx}$, the $dr$-integration equals
\begin{align*}
\int_{R_{\smallidx,\bigidx}}f_s^{N,\psi_{\gamma}}(w_{\smallidx,\bigidx}rh,\left(\begin{array}{ccc}a\\&
\lfloor x\rfloor\\&&I_{\bigidx-\smallidx}\\\end{array}\right)
\left(\begin{array}{ccc}I_{\smallidx-1}\\&
1&z_{r,x}\\&&I_{\bigidx-\smallidx}\\\end{array}\right))\psi(\lfloor
x\rfloor^{-1} r_{\bigidx-\smallidx,\bigidx})dr,
\end{align*}
where
\begin{align*}
(z_{r,x})_1=\gamma(c\lfloor
x\rfloor^{-1}r_{\bigidx-\smallidx,\bigidx+1}-\half c^2\lfloor
x\rfloor^{-2}r_{\bigidx-\smallidx,\bigidx}).
\end{align*}
This becomes
\begin{align*}
\int_{R_{\smallidx,\bigidx}}f_s^{N,\psi_{\gamma}}(w_{\smallidx,\bigidx}rh,
diag(a,\lfloor x\rfloor,I_{\bigidx-\smallidx})
)\psi^{\star}(r)dr.
\end{align*}
Here $\psi^{\star}$ is the character of $R_{\smallidx,\bigidx}$
defined by
\begin{align*}
\psi^{\star}(r)=\psi(\lfloor x\rfloor^{-1}(1-\gamma\half c^2)
r_{\bigidx-\smallidx,\bigidx}+ \gamma
cr_{\bigidx-\smallidx,\bigidx+1})=\psi(\gamma cr_{\bigidx-\smallidx,\bigidx+1}).
\end{align*}
By our fixing of $h$, $\psi^{\star}(r)$  no longer depends on $x$
(since $c$ is fixed). We can continue as above, write
$f_s^{N,\psi_{\gamma}}$ as in \eqref{eq:dr integral torus writing
the compactly supported section} and for each $i$,
$h\cdot(\lambda(w_{\smallidx,\bigidx}^{-1})f_s^{(i)})$ is still
compactly supported modulo
$Q_{\bigidx}^{w_{\smallidx,\bigidx}}$
(because $h\in Q_{\bigidx-\smallidx}$). Therefore for each
$f_s^{(i)}\in\xi(\tau,std,s)$,
\begin{align*}
\int_{R_{\smallidx,\bigidx}}f_s^{(i)}(w_{\smallidx,\bigidx}rh,
diag(a,\lfloor x\rfloor,I_{\bigidx-\smallidx})
)\psi^{\star}(r)dr
=\sum_{j=1}^{m_i}P_{i,j}'W_{i,j}'(
diag(a,\lfloor x\rfloor,I_{\bigidx-\smallidx})
).
\end{align*}
Here $P_{i,j}'\in\C[q^{-s},q^s]$, $W_{i,j}'\in\Whittaker{\tau}{\psi}$.
Now equality~\eqref{eq:formula dr integral for torus of G_l yes G_1} follows.

Note that the proof for $|b|>|b|^{-1}$ is identical but the actual polynomials $P_{i,j}'$ and Whittaker functions $W_{i,j}'$ may vary,
because they depend on $h$ and $\psi^{\star}$.
\end{proof} 

\begin{remark}\label{remark:l=n still writing dr integral for torus of G_l}
Observe that if $\smallidx=\bigidx$, in which case the $dr$-integration is
trivial, we can simply put $\Omega(h\cdot
f_s)=f_s(w_{\smallidx,\bigidx}h,1)$ ($h\in H_{\bigidx})$ and
Proposition~\ref{proposition:dr integral for torus of G_l} remains
valid.
\end{remark}

We rephrase the result of Claim~\ref{claim:jacquet module vanishes
for r > 0} in terms of the Jacquet module $V(\tau,s)_{N_{\bigidx-\smallidx},\psi_{\gamma}^{-1}}$.
The following description, until the end of this section, will not be used in the sequel.
Let $V(\tau,s)_{w_r}$ denote the space of locally constant functions
$f_s:C(w_r)\rightarrow U$ compactly supported modulo $Q_{\bigidx}$ and
such that
\begin{align*}
f_s(aux)=\delta_{Q_{\bigidx}}^{\half}(a)\absdet{a}^{s-\half}\tau(a)f_s(x)\qquad(a\in
\GL{\bigidx}\isomorphic M_{\bigidx}, u\in U_{\bigidx}, x\in C(w_r)).
\end{align*}
By definition $F_{w_0}(s)=V(\tau,s)_{w_0}$. For $r>0$, restriction
$f_s\mapsto\frestrict{f_s}{C(w_r)}$ induces an isomorphism
$\lmodulo{F_{w_{r-1}}(s)}{F_{w_{r}}(s)}\isomorphic V(\tau,s)_{w_r}$ as
representations of $Q_{\bigidx-\smallidx}$. Also
$V(\tau,s)_{w_r}\isomorphic ind_{Q_{\bigidx}^{w_r}}^{Q_{\bigidx-\smallidx}}(\rconj{w_r^{-1}}(\tau\alpha^s))$ (normalized compact induction, $\rconj{w_r^{-1}}(\tau\alpha^s)(x)=\tau\alpha^s(\rconj{w_r^{-1}}x)$),
where $f_s\in V(\tau,s)_{w_r}$ is mapped to $\lambda(w_r^{-1})f_s$.
Up to semi-simplification, $V(\tau,s)$ is isomorphic to
$\bigoplus_{r=0}^{\bigidx-\smallidx} V(\tau,s)_{w_r}$ as
$Q_{\bigidx-\smallidx}$-spaces. Claim~\ref{claim:jacquet module
vanishes for r > 0} implies
\begin{align*}
(V(\tau,s)_{w_r})_{N_{\bigidx-\smallidx},\psi_{\gamma}^{-1}}=0,\qquad\forall r>0.
\end{align*}
Consequently,
\begin{align}\label{eq:jacquet rephrasing}
V(\tau,s)_{N_{\bigidx-\smallidx},\psi_{\gamma}^{-1}}=(V(\tau,s)_{w_0})_{N_{\bigidx-\smallidx},\psi_{\gamma}^{-1}}.
\end{align}

Finally we mention that Claims~\ref{claim:jacquet module vanishes for
r > 0} and \ref{claim:omega f reduces to sum when support is compact} 
also imply that $\Omega(f_s)$, for any fixed $s$, has a sense
as a principal value. Indeed let
$\{R_{\smallidx,\bigidx}^v\}_{v=0}^{\infty}$ be an increasing
sequence of compact open subgroups, which exhausts
$R_{\smallidx,\bigidx}$. The limit
\begin{align*}
\Omega_{p.v.}(f_s)=\lim_{v\rightarrow\infty}\int_{R_{\smallidx,\bigidx}^v}f_s(w_{\smallidx,\bigidx}r,1)\psi_{\gamma}(r)dr
\end{align*}
exists and is finite. To explain this, first note that according to \eqref{eq:jacquet
rephrasing}, $f_s\in V(\tau,s)$ can be written as
$f_s=f_s^{(1)}+f_s^{(2)}$ with $f_s^{(1)}\in V(\tau,s)_{w_0}$ and
$f_s^{(2)}\in V(\tau,s)(N_{\bigidx-\smallidx},\psi_{\gamma}^{-1})$, where $V(\tau,s)(N_{\bigidx-\smallidx},\psi_{\gamma}^{-1})$ denotes
the subspace of $V(\tau,s)$ generated by finite sums $\sum_{i=1}^m\tau(n_i)v_i-\psi_{\gamma}^{-1}(n_i)v_i$, $n_i\in N_{\bigidx-\smallidx}$,
$v_i\in V(\tau,s)$. For $f_s^{(2)}$, the integral over $R_{\smallidx,\bigidx}^v$ vanishes when $v>>0$. For $f_s^{(1)}$,
the integral becomes a finite sum once $v$ is large enough, because $f_s^{(1)}$ is compactly supported modulo
$Q_{\bigidx}^{w_{\smallidx,\bigidx}}$ and by virtue of Claim~\ref{claim:omega f reduces to sum when support is compact}. Now \eqref{eq:omega in
homspace}, \eqref{eq:omega right translation by GL l-1} and
Claim~\ref{claim:omega invariant for compact averaging} hold for
$\Omega_{p.v.}(f_s)$. Then Proposition~\ref{proposition:dr integral
for torus of G_l} is valid for all $s$ with the $dr$-integration
given by principal value. The analytic continuations of $\Omega(f_s)$ and $\Omega_{p.v.}(f_s)$ for $f_s\in\xi(\tau,hol,s)$
are equal, since for $\Re(s)>>0$, by Lebesgue's Dominated Convergence Theorem, $\Omega(f_s)=\Omega_{p.v.}(f_s)$.
This description resembles that of the Whittaker functional in \cite{CS2} and the arguments of this paragraph
appeared in \cite{Sh2} (Section~3).

\section{Iwasawa decomposition for
$\Psi(W,f_s,s)$}\label{subsection:Iwasawa decomposition for
Psi(W,f,s)} Here we show how to write $\Psi(W,f_s,s)$ as a finite
sum of integrals over a torus. This is done for
$f_s\in\xi(\tau,std,s)$ and the definition of holomorphic sections implies a similar form for
$f_s\in\xi(\tau,hol,s)$. 

Write an element $x\in G_1$ as
$x=diag(b,b^{-1})$
and define for $k>0$,
\begin{align*}
&G_1^{0,k}=\setof{x\in G_1}{[x]>q^k,\lfloor x\rfloor=b},\qquad
&G_1^{\infty,k}=\setof{x\in G_1}{[x]>q^k,\lfloor x\rfloor=b^{-1}}.
\end{align*}
In addition for a set $\Lambda$ let $ch_{\Lambda}$ be the characteristic function of $\Lambda$. Denote by $D$ the domain of absolute convergence of the integrals $\Psi(W,f_s,s)$ with $f_s\in\xi(\tau,hol,s)$. This is a right half-plane depending only on the representations.
\begin{proposition}\label{proposition:iwasawa decomposition of the integral}
For each integral $\Psi(W,f_s,s)$, $f_s\in\xi(\tau,std,s)$ there exist integrals $I_s^{(1)},\ldots, I_s^{(m)}$ such that for all $s\in D$, $\Psi(W,f_s,s)=\sum_{i=1}^mI_s^{(i)}$.  If $\smallidx\leq\bigidx$, each $I_s^{(i)}$ is of the form
\begin{align}\label{int:iwasawa decomposition of the integral l<=n split}
\begin{dcases*}
\begin{array}{ll}\displaystyle
P\int_{A_{\smallidx-1}}\int_{G_1}ch_{\Lambda}(x)W^{\diamond}(ax)W'(
diag(a,\lfloor x\rfloor,I_{\bigidx-\smallidx}))\\\displaystyle
\qquad(\absdet{a}\cdot[x]^{-1})^{\smallidx-\half\bigidx+s-\half}\delta_{B_{G_{\smallidx}}}^{-1}(a)dxda
\end{array}&\text{split $G_{\smallidx}$,}\\\\
P\int_{A_{\smallidx-1}}W^{\diamond}(a)W'(diag(a,I_{\bigidx-\smallidx+1}))
\absdet{a}^{\smallidx-\half\bigidx+s-\half}\delta_{B_{G_{\smallidx}}}^{-1}(a)da
&\text{\quasisplit\ $G_{\smallidx}$.}
\end{dcases*}
\end{align}
Here $P\in\C[q^{-s},q^s]$, $W^{\diamond}\in\Whittaker{\pi}{\psi_{\gamma}^{-1}}$, $W'\in\Whittaker{\tau}{\psi}$. In the split case, $\Lambda$ is either $G_1^{0,k}$ or $G_1^{\infty,k}$ for a constant $k>0$, or a compact open subgroup of
$G_1$. In all cases $\Lambda$ is independent of $s$.

When $\smallidx>\bigidx$, $I_s^{(i)}$ takes the form
\begin{align}\label{int:iwasawa decomposition of the integral l>n}
P\int_{A_{\bigidx}}W^{\diamond}(diag(a,I_{2(\smallidx-\bigidx)},a^*))W'(a)\absdet{a}^{\frac32\bigidx-\smallidx+s+\half}\delta_{B_{H_{\bigidx}}}^{-1}(a)da.
\end{align}
\end{proposition}
\begin{remark}
Note that in the above integrals, except for $P$ and the exponents of $\absdet{a}$ and $[x]$, none of the terms depend on $s$.
\end{remark}

\begin{proof}[Proof of Proposition~\ref{proposition:iwasawa decomposition of the integral}] 
First assume $\smallidx\leq\bigidx$. According to the Iwasawa decomposition of $G_{\smallidx}$ and since $W$ and $f_s$ are $K_{G_{\smallidx}}$-finite, $\Psi(W,f_s,s)$ can be written as a sum of integrals of the form
\begin{align}\label{int:first iwasawa decomp}
\int_{T_{G_{\smallidx}}}W^{\diamond}(t)\Omega(t\cdot f_s')\delta_{B_{G_{\smallidx}}}^{-1}(t)dt
=\int_{A_{\smallidx-1}}\int_{G_1}W^{\diamond}(ax)\Omega(ax\cdot f_s')\delta_{B_{G_{\smallidx}}}^{-1}(a)dxda,
\end{align}
where $W^{\diamond}\in\Whittaker{\pi}{\psi_{\gamma}^{-1}}$ and $f_s'\in\xi(\tau,hol,s)$. This writing depends on the subgroups of $K_{G_{\smallidx}}$ for which $W$ and $f_s$ are right-invariant and since these are independent of $s$, this decomposition is independent of $s$. This means that there exists a finite number of integrals, each of the form \eqref{int:first iwasawa decomp}, such that for all $s$ their sum equals $\Psi(W,f_s,s)$. Henceforth whenever we write such a decomposition, it will be independent of $s$ in this sense.

If $G_{\smallidx}$ is \quasisplit\ \eqref{int:first iwasawa decomp} reduces to a finite sum of integrals of the form
\begin{align*}
\int_{A_{\smallidx-1}}W^{\diamond}(ax_i)\Omega(ax_i\cdot f_s')\delta_{B_{G_{\smallidx}}}^{-1}(a)da.
\end{align*}
Here the elements $x_i\in G_1$ belong to a finite set independent of
$s$. The section $x_i\cdot f_s'$ is still holomorphic so we can
decompose each integral once more to a sum of integrals of the form
\begin{align*}
P\int_{A_{\smallidx-1}}W^{\diamond}(a)\Omega(a\cdot f_s')\delta_{B_{G_{\smallidx}}}^{-1}(a)da,
\end{align*}
where $P\in\C[q^{-s},q^s]$ and the section $f_s'$ is standard. Here
we changed the order of integration over $A_{\smallidx-1}\times
R_{\smallidx,\bigidx}$ and (finite) summation. This is allowed since
in our domain of absolute convergence, $\Psi(W^{\diamond},f_s'',s)$
is absolutely convergent for all $W^{\diamond}$ and
$f_s''\in\xi(\tau,hol,s)$ whence
\begin{align}\label{int:absolute convergence f_s''}
\int_{A_{\smallidx-1}}\int_{R_{\smallidx,\bigidx}}|W^{\diamond}|(a) |f_s''|(w_{\smallidx,\bigidx}ra,1)\delta_{B_{G_{\smallidx}}}^{-1}(a)drda<\infty.
\end{align}
Now use Proposition~\ref{proposition:dr integral for torus of G_l} (regarding $\smallidx=\bigidx$, see Remark~\ref{remark:l=n still writing dr integral for torus of G_l}) to obtain a sum of integrals
\begin{align}\label{int:iwasawa decomposition final form quasisplit}
P\int_{A_{\smallidx-1}}W^{\diamond}(a)W'(diag(a,I_{\bigidx-\smallidx+1}))
\absdet{a}^{\smallidx-\half\bigidx+s-\half}\delta_{B_{G_{\smallidx}}}^{-1}(a)da.
\end{align}
Here $W'\in\Whittaker{\tau}{\psi}$. This is the required form in the \quasisplit\ case. Again, after plugging in the formula of Proposition~\ref{proposition:dr integral for torus of G_l} we changed the order of integration and summation. This is allowed because we may take $f_s''$ as in Lemma~\ref{lemma:f with small support}, defined with respect to an arbitrary $W'$, then \eqref{int:absolute convergence f_s''} implies
\begin{align*}
\int_{A_{\smallidx-1}}|W^{\diamond}|(a)|W'|(diag(a,I_{\bigidx-\smallidx+1}))
\absdet{a}^{\smallidx-\half\bigidx+\Re(s)-\half}\delta_{B_{G_{\smallidx}}}^{-1}(a)da<\infty.
\end{align*}

If $G_{\smallidx}$ is split we take $k>0$ as in
Proposition~\ref{proposition:dr integral for torus of G_l} and write
the $dx$-integration in \eqref{int:first iwasawa decomp} as a sum of
integrals, the first over $\setof{x\in G_1}{[x]>q^k}$ and the second
over $\setof{x\in G_1}{[x]\leq q^k}$. When $[x]\leq q^k$ we get as
in the \quasisplit\ case a finite sum of integrals of the form
\eqref{int:iwasawa decomposition final form quasisplit}. This form corresponds to
\eqref{int:iwasawa decomposition of the integral l<=n split} when we take
$\Lambda=\mathcal{N}_{G_{\smallidx},k'}$ with $k'>>0$ (depending on $W^{\diamond}$ and $W'$).
The
justification is similar to the above: analogously to
\eqref{int:absolute convergence f_s''} we have
\begin{align}\label{int:absolute convergence f_s'' split}
\int_{A_{\smallidx-1}}\int_{G_1}\int_{R_{\smallidx,\bigidx}}|W^{\diamond}|(ax)
|f_s''|(w_{\smallidx,\bigidx}rax,1)\delta_{B_{G_{\smallidx}}}^{-1}(a)drdxda<\infty
\end{align}
and we can use Lemma~\ref{lemma:f with small support}.

Assume $[x]>q^k$. We apply Proposition~\ref{proposition:dr integral
for torus of G_l} and get that \eqref{int:first iwasawa decomp} is a
sum of integrals of the form
\begin{align*}
P\int_{A_{\smallidx-1}}\int_{\Lambda}W^{\diamond}(ax)W'(diag(a,\lfloor x\rfloor,I_{\bigidx-\smallidx}))
(\absdet{a}\cdot[x]^{-1})^{\smallidx-\half\bigidx+s-\half}\delta_{B_{G_{\smallidx}}}^{-1}(a)dxda,
\end{align*}
with $\Lambda=G_1^{0,k},G_1^{\infty,k}$. Here the justification to
the last step - after plugging in Proposition~\ref{proposition:dr
integral for torus of G_l}, is a bit more complicated. First note
that for a fixed $W'$, it is enough to prove the absolute
convergence of the last integral for $k>>0$ (i.e., not necessarily
the $k$ chosen using Proposition~\ref{proposition:dr integral for
torus of G_l}). Let $W'$ be given and assume $\Lambda=G_1^{0,k}$. Consider the elements $k_0$ and $h_1$ of
Lemma~\ref{lemma:torus elements in double cosets}. According to the proof of
the lemma (see Remark~\ref{remark:torus elements in double cosets h_1 and k_0
small}) it is possible to take $k_0$ and $h_1$ such that $W'$ is
right-invariant by
$\rconj{(w_{\smallidx,\bigidx}h_1)^{-1}}\mathcal{N}_{H_{\bigidx},k_0}\cap
Q_{\bigidx}$. Therefore we may define $f_s''$ with support in
$Q_{\bigidx}w_{\smallidx,\bigidx}h_1\mathcal{N}_{H_{\bigidx},k_0}$,
such that $f_s''(w_{\smallidx,\bigidx}h_1n',b)=W'(b)$
($b\in\GL{\bigidx}$, $n'\in\mathcal{N}_{H_{\bigidx},k_0}$). Then for
all $a\in A_{\smallidx-1}$ and $x\in \Lambda$, where $k>k_0$ is
given by Lemma~\ref{lemma:torus elements in double cosets},
\begin{align*}
&\int_{R_{\smallidx,\bigidx}}|f_s''|(w_{\smallidx,\bigidx}rax,1)dr\\\notag&=(\absdet{a}\cdot[x]^{-1})^{\smallidx-\half\bigidx+\Re(s)-\half}
\int_{R_{\smallidx,\bigidx}}|f_s''|(w_{\smallidx,\bigidx}rh_1,diag(a,\lfloor x\rfloor,I_{\bigidx-\smallidx}))dr.
\end{align*}
Putting this into \eqref{int:absolute convergence f_s'' split} (with
$\Lambda$ instead of $G_1$) and using the smoothness of $f_s''$
yields
\begin{align*}
\int_{A_{\smallidx-1}}\int_{\Lambda}|W^{\diamond}|(ax)|W'|(diag(a,\lfloor x\rfloor,I_{\bigidx-\smallidx}))(\absdet{a}\cdot[x]^{-1})^{\smallidx-\half\bigidx+\Re(s)-\half}\delta_{B_{G_{\smallidx}}}^{-1}(a)dxda<\infty.
\end{align*}
The same argument applies to $\Lambda=G_1^{\infty,k}$, with $h_2$ of
Lemma~\ref{lemma:torus elements in double cosets}.

The result for $\smallidx>\bigidx$ follows again from the Iwasawa
decomposition and from the fact that for each $a\in\GL{\bigidx}$,
the support of the function on $R^{\smallidx,\bigidx}$ given by
\begin{align*}
r\mapsto W(ar)=W(\left(\begin{array}{ccccc}a\\r&I_{\smallidx-\bigidx-1}\\&&I_2\\
&&&I_{\smallidx-\bigidx-1}\\&&&r'&a^*\\\end{array}\right))
\end{align*} is contained in a compact set, independent of $a$ (see \cite{Soudry} Section~4). We mention that the
change to the measure of $R^{\smallidx,\bigidx}$ by conjugating with
$\rconj{(w^{\smallidx,\bigidx})^{-1}}a$ is
$\absdet{a}^{\bigidx-\smallidx+1}$.
\end{proof} 
\begin{remark}\label{remark:taking suitable half-plane for inner integrals}
The absolute convergence of all integrals in the proof is governed
by the parameters of the representations. Hence the justifications
to the formal manipulations concerning integrals with $W^\diamond$ and $W'$ (e.g. \eqref{int:iwasawa decomposition final form quasisplit})
can be obviated, simply by taking a suitable right half-plane $D$.
\end{remark}
This proposition has the following corollary, which is a direct
consequence of the properties of Whittaker functions for
supercuspidal representations (see \cite{CS2} Section~6).
\begin{corollary}\label{corollary:integral is holomorphic for supercuspidal data}
Let $\pi$ and $\tau$ be supercuspidal representations and assume that if
$\smallidx=\bigidx=1$, $G_1$ is \quasisplit. Then $\Psi(W,f_s,s)$ is
holomorphic for $f_s\in\xi(\tau,hol,s)$. 
\end{corollary}

\section{The integrals are \nontrivial}\label{subsection:the
integrals are non-trivial} The following proposition shows that the
integrals do not vanish identically. The actual result means more
than this, namely that the fractional ideal spanned by the integrals
contains the constant $1$ 
and consequently, the g.c.d. of the integrals can be taken in the form
$P(q^{-s})^{-1}$ where $P\in\C[X]$, see
Section~\ref{section:gcd and epsilon factor}. 
It is interesting to note that Gelbart and Piatetski-Shapiro
\cite{GPS} (Section~12) proved a similar result, foreseeing a
definition of the $L$-function of $\pi\times\tau$ as a g.c.d., as we
shall attempt to define.

\begin{proposition}\label{proposition:integral can be made constant}
There exist $W\in\Whittaker{\pi}{\psi_{\gamma}^{-1}}$ and $f_s\in\xi(\tau,hol,s)$ such that
$\Psi(W,f_s,s)$ is absolutely convergent and equals $1$, for all
$s$.
\end{proposition}
\begin{proof}[Proof of Proposition~\ref{proposition:integral can be made constant}] 
The proof follows the arguments of Soudry \cite{Soudry} (Section~6).
First assume $\smallidx>\bigidx$. Let
$W_0\in\Whittaker{\pi}{\psi_{\gamma}^{-1}}$,
$W'\in\Whittaker{\tau}{\psi}$ be such that $W_0(1)\ne0$, $W'(1)\ne0$
(such $W_0,W'$ exist). Define $W_1$ for $W_0$ using
Lemma~\ref{lemma:W with small support} with $j=0$ and $k_1>>0$
(depending on $W_0,W'$) and set $W=(w^{\smallidx,\bigidx})^{-1}\cdot
W_1$. Select
$f_s=ch_{\mathcal{N}_{H_{\bigidx},k},W',s}\in\xi(\tau,std,s)$
for $k>>k_1$ large. 
Assuming that $\Psi(W,f_s,s)$ is absolutely convergent at $s$, we replace the integration over
$\lmodulo{U_{H_{\bigidx}}}{H_{\bigidx}}$ with an integration over
$A_{\bigidx}\overline{Z_{\bigidx}}\overline{U_{\bigidx}}$ and obtain
\begin{align*}
\Psi(W,f_s,s)=\int_{A_{\bigidx}}\int_{\overline{Z_{\bigidx}}}\int_{\overline{U_{\bigidx}}}
\int_{R^{\smallidx,\bigidx}}W_1(rw^{\smallidx,\bigidx}azu(w^{\smallidx,\bigidx})^{-1})f_s(azu,1)\delta(a)drdudzda,
\end{align*}
where $\delta$ is an appropriate modulus character. Since $\overline{U_{\bigidx}}\cap Q_{\bigidx}\mathcal{N}_{H_{\bigidx},k}=\overline{U_{\bigidx}}\cap\mathcal{N}_{H_{\bigidx},k}$ and $k>>k_1$,
the integration over $\overline{U_{\bigidx}}$
reduces to a constant. Then by Lemma~\ref{lemma:W with small
support} the $drdzda$-integration is reduced and the integral equals
$cW_0(1)W'(1)$, where $c>0$ is a constant independent of $s$ ($c$
equals a product of volumes). Hence after suitably normalizing
$W_0$, $\Psi(W,f_s,s)=1$. The same arguments show
\begin{align*}
\int_{\lmodulo{U_{H_{\bigidx}}}{H_{\bigidx}}}\int_{R^{\smallidx,\bigidx}}|W|(rw^{\smallidx,\bigidx}h)|f_s|(h,1)drdh=c|W_0|(1)|W'|(1).
\end{align*}
Therefore $\Psi(W,f_s,s)$ is absolutely convergent for all $s$. This also
justifies our computation.

In the case $\smallidx\leq\bigidx$ take $W_0$ as above and $W'$ such
that $W'(t_{\gamma})\ne0$, where $t_{\gamma}$ is defined in Lemma~\ref{lemma:f with small support}. Then $W_1$ is obtained as before and we
put $W=W_1$. Now $f_s$ is selected by Lemma~\ref{lemma:f with small
support} for $W'$ with a large $k$ (depending on $W,W'$). We use
\eqref{eq:omega right translation by GL l-1} and get
\begin{align*}
\Psi(W,f_s,s)=&
\int_{A_{\smallidx-1}}\int_{\overline{Z_{\smallidx-1}}}\int_{\overline{V_{\smallidx-1}}}\int_{G_1}
W(azvx)\delta(a)\absdet{a}^{\smallidx-\half\bigidx+s-\half}\\\notag&\int_{R_{\smallidx,\bigidx}}f_s(w_{\smallidx,\bigidx}rvx,az)\psi_{\gamma}(r)drdxdvdzda.
\end{align*}
By the selection of $f_s$ this equals
\begin{align*}
c'|\gamma|^{\smallidx(\half\bigidx+s-\half)}\int_{A_{\smallidx-1}}\int_{\overline{Z_{\smallidx-1}}}
W(az)\absdet{a}^{\smallidx-\half\bigidx+s-\half}W'(azt_{\gamma})\delta(a)dzda.
\end{align*}
Here $c'\in\C$ is independent of $s$. As in the previous case we end
up with
$e(q^s)W_0(1)W'(t_{\gamma})$,
where
$e(q^s)=c'|\gamma|^{\smallidx(\half\bigidx+s-\half)}\in\C[q^{-s},q^s]^*$.
The result follows with $e(q^s)^{-1}f_s\in\xi(\tau,hol,s)$
instead of $f_s$. Also observe that if we replace $f_s$ with $|f_s|$
and drop $\psi_{\gamma}$, the arguments of Lemma~\ref{lemma:f with
small support} imply (for some constant $c''$)
\begin{align*}
\int_{R_{\smallidx,\bigidx}}|f_s|(w_{\smallidx,\bigidx}rvx,az)dr=c''|\gamma|^{\smallidx(\half\bigidx+\Re(s)-\half)}|W'|(azt_{\gamma}).
\end{align*}
This shows that $\Psi(W,f_s,s)$ is absolutely convergent for all
$s$.
\end{proof} 

\section{Meromorphic continuation}\label{subsection:meromorphic
continuation}
One of the fundamental requirements of the integrals is meromorphic
continuation to functions in $\C(q^{-s})$. Let $\Phi=\Phi(s)$ be a complex-valued function defined on some domain
$D_{\Phi}\subset\C$, which contains a non-empty open subset. We say that $\Phi$ has a meromorphic continuation to a function in
$\C(q^{-s})$, or extends to a function in $\C(q^{-s})$, if there is some $Q\in\C(q^{-s})$ such that for all $s\in D_{\Phi}$, $\Phi(s)=Q(q^{-s})$. Note that if $Q$ exists,
it is necessarily unique. Also by definition $Q$ has no poles in $D_{\Phi}$, i.e., the function $s\mapsto Q(q^{-s})$ is holomorphic on $D_{\Phi}$.

Let $W\in\Whittaker{\pi}{\psi_{\gamma}^{-1}}$ and $f_s\in\xi(\tau,hol,s)$. Recall that there is some $s_0>0$
depending only on $\pi$ and $\tau$, such that for all $\Re(s)>s_0$, $\Psi(W,f_s,s)$ is absolutely convergent
(see Section~\ref{subsection:the integrals}). 
We provide two proofs of the meromorphic continuation: directly with the aid of
Proposition~\ref{proposition:iwasawa decomposition of the integral}, or using Bernstein's continuation principle
(in \cite{Banks}, see also the formulation of \cite{Mu} Section~8) combined with the uniqueness results of Chapter~\ref{chapter:uniqueness}.
\begin{claim}\label{claim:meromorphic continuation 1}
There is a constant $s_1\geq s_0$ such that the following holds. For any
$W\in\Whittaker{\pi}{\psi_{\gamma}^{-1}}$ and $f_s\in\xi(\tau,hol,s)$ there is $Q\in\C(q^{-s})$ such that
$\Psi(W,f_s,s)=Q(q^{-s})$ for all $\Re(s)>s_1$. The poles of $Q$, in $q^{-s}$,
belong to a finite set depending only on $\pi$ and $\tau$.
\end{claim}
\begin{proof}[Proof of Claim~\ref{claim:meromorphic continuation 1}]
We argue as in \cite{JPSS} (Section~2.7, see also \cite{Soudry} Section~4.3).
According to Proposition~\ref{proposition:iwasawa decomposition of the integral}, in the right half-plane $\Re(s)>s_0$ we can write
$\Psi(W,f_s,s)$ as a sum of integrals $I_s^{(i)}$. Consider each $I_s^{(i)}$ separately. Put $I=I_s^{(i)}$.

We can plug the
asymptotic expansions of Whittaker functions from Section~\ref{section:whittaker props} into $I$.
Taking $s'\geq s_0$, depending only on $\pi$ and $\tau$, we see that $I$ equals a sum of products of Tate-type integrals.
For example, consider the case $1<\smallidx\leq\bigidx$ and split $G_{\smallidx}$. Then $I$ takes the form
\begin{align*}
\int_{A_{\smallidx-1}}\int_{G_1}ch_{\Lambda}(x)W^{\diamond}(ax)W'(
diag(a,\lfloor x\rfloor,I_{\bigidx-\smallidx}))(\absdet{a}\cdot[x]^{-1})^{\smallidx-\half\bigidx+s-\half}\delta_{B_{G_{\smallidx}}}^{-1}(a)dxda.
\end{align*}
Set $a=diag(a_1,\ldots,a_{\smallidx-1})\in A_{\smallidx-1}$. Using the asymptotic expansions for $W^{\diamond}$ and $W'$ and
changing $a_v\mapsto a_va_{v+1}$ for $1\leq v<\smallidx-1$, the integral is equal to a sum of products
$\prod_{j=1}^{\smallidx-1}I_j$ with each $I_j$ of the following form. First assume $j<\smallidx-1$.
\begin{align*}
I_j=\begin{dcases}
\int_{F^*}(\phi_j\cdot\phi_j'\cdot\eta_j\cdot\eta_j')(a_j)|a_j|^{-\smallidx-\half\bigidx+\frac32+s}d^*a_j&j=1,\\
\int_{F^*}(\phi_j\cdot\phi_j'\cdot\eta_j\cdot\eta_j')(a_j)|a_j|^{4j-2\smallidx-\bigidx-3+2s}d^*a_j&1<j<\smallidx-1.
\end{dcases}
\end{align*}
Here $\phi_j,\phi_j'\in\mathcal{S}(F)$, $\eta_j,\eta_j'$ are finite functions on $F^*$. The functions
$\phi_j,\eta_j$ correspond
to the expansion of $W^{\diamond}$ and $\phi_j',\eta_j'$ correspond to the expansion of $W'$. Since any finite function $\eta$ on $F^*$
can be written as a finite sum $\sum_{r}c_{r}\vartheta(a)^{k_{r}}\mu_{r}(a)$ ($a\in F^*$), where
$c_{r}\in\C$, $0\leq k_{r}\in\Integers$ and $\mu_{r}$ is a character, each $I_j$ for $j<\smallidx-1$
is a finite sum of integrals of the form
\begin{align}\label{int:tate example in the proof}
\int_{F^*}\phi(a)\vartheta^m(a)\theta(a)|a|^{A+Bs}d^*a,
\end{align}
where $\phi\in\mathcal{S}(F)$, $0\leq m\in\Integers$, $\theta$ is a unitary character, $A\in\C$ and $0<B\in\Integers$ (actually, here $B=1,2$).
An integral of the form \eqref{int:tate example in the proof} is called a Tate-type integral.
Clearly it has a rational function $Q\in\C(q^{-s})$ such that for $\Re(s)>s_1\geq s'$, the integral is equal to $Q(q^{-s})$. The
poles of $Q$, in $q^{-s}$, belong to a finite set. The value of $s_1$ and this set depend only on the representations.

Regarding $I_{\smallidx-1}$, it takes the form
\begin{align*}
&\int_{F^*}\int_{F^*}ch_{\Lambda}(b)(\phi\cdot\eta)(a_{\smallidx-1}b^{-1},a_{\smallidx-1}b)
(\phi'\cdot\eta')(a_{\smallidx-1}\lfloor b\rfloor^{-1})[(\phi''\cdot\eta'')(\lfloor b\rfloor)]\\&
|a_{\smallidx-1}|^{2\smallidx-\bigidx-7+2s}
|\lfloor b\rfloor|^{\smallidx-\half\bigidx-\half+s}d^*a_{\smallidx-1}d^*b.
\end{align*}
Here the integral over $G_1$ was written as an integral over $F^*$;
$\lfloor b\rfloor=b$ if $|b|\leq|b|^{-1}$ or equivalently if $|b|\leq1$, otherwise $\lfloor b\rfloor=b^{-1}$;
$\Lambda$ is either $\setof{b\in F^*}{|b|<q^{-k}}$ (replacing $G_1^{0,k}$),
$\setof{b\in F^*}{|b|>q^{k}}$ (instead of $G_1^{\infty,k}$), for a constant $k>0$, or a compact open subgroup of
$F^*$; $\phi\in\mathcal{S}(F^2)$, $\eta$ is a finite function on $(F^*)^2$, $\phi$ and $\eta$ correspond to
$W^{\diamond}$; $\phi',\phi''\in \mathcal{S}(F)$, $\eta'$ and $\eta''$ are finite functions on $F^*$, $\phi',\phi'',\eta'$ and $\eta''$ correspond to $W'$;
the factor $[(\phi''\cdot\eta'')(\lfloor b\rfloor)]$ appears only if $\smallidx<\bigidx$.

Assume that $\Lambda$ is a compact open subgroup. We claim that the integrand is a smooth function of $b$.

In general if $\varphi\in\mathcal{S}(F^2)$, the mapping
$(x,y)\mapsto \varphi(xy,xy^{-1})$ ($x,y\in F$, $y\ne0$) vanishes unless $x$ belongs to some compact subset of $F$ (that is, $x$ is bounded from above).
If $y$ belongs to a compact open subset $\Lambda_0$ of $F^*$,
the function $\varphi(xy,xy^{-1})$ is smooth in $y$, i.e., there is some $k_0$ such that
$\varphi(x(y(1+\mathcal{P}^{k_0})),x(y(1+\mathcal{P}^{k_0}))^{-1})=\varphi(xy,xy^{-1})$ for all $x\in F$ and $y\in\Lambda_0$
(because $y$ is bounded from above and below). Here $k_0$ depends on $\varphi$ and $\Lambda_0$.
Similarly if $\varphi\in\mathcal{S}(F)$, $x$ belongs to a compact subset of $F$ and $y\in\Lambda_0$, the mappings
$(x,y)\mapsto\varphi(xy),\varphi(xy^{-1})$ are smooth in $y$. In particular, $\varphi(y)$ and $\varphi(y^{-1})$ are smooth in $y$.

Additionally, for any $y\in F^*$, $k_0>0$ and $z\in F^*$ with $|z|=1$,
if $\lfloor y\rfloor = y$ then $\lfloor y(1+\varpi^{k_0}z)\rfloor = y(1+\varpi^{k_0}z)$
and if $\lfloor y\rfloor = y^{-1}$, $\lfloor y(1+\varpi^{k_0}z)\rfloor = y^{-1}(1+\varpi^{k_0}z)^{-1}$. Hence, for example,
the function $\varphi(\lfloor y\rfloor)$ is smooth in $y$ when $y\in\Lambda_0$.

Regarding the finite functions, any finite function $\eta_0$ on $(F^*)^c$ can be expressed as a finite sum
$\eta_0(y)=\sum_{r}\mu_{r}(y)P_{r}(y)$, where $\mu_{r}$ is a character and $P_{r}$ is a polynomial in
the valuation vector of $y$. Therefore $\eta_0(y)$ and $\eta_0(y^{-1})$ are smooth functions of $y$. Since $P_{r}(xy)$ is a polynomial in
$\vartheta(x_i),\ldots,\vartheta(x_c),\vartheta(y_1),\ldots,\vartheta(y_c)$, both $\eta_0(xy)$ and $\eta_0(xy^{-1})$
are smooth in $y$. Here smoothness holds even if $y$ is not restricted to a compact subset.

Applying these general remarks to the integrand at hand we get that it is a smooth function of $b$. Therefore the $db$-integration can be ignored and we
can again reduce to integrals of the form \eqref{int:tate example in the proof}.

Now assume $\Lambda=\setof{b\in F^*}{|b|<q^{-k}}$.
If $ch_{\Lambda}(b)\ne0$ then $b=\lfloor b\rfloor$ and $I_{\smallidx-1}$ becomes
\begin{align}\label{int:lambda is G_1 0 in meromorphic proof}
&\int_{F^*}\int_{F^*}ch_{\Lambda}(b)(\phi\cdot\eta)(a_{\smallidx-1}\lfloor b\rfloor^{-1},a_{\smallidx-1}\lfloor b\rfloor)
(\phi'\cdot\eta')(a_{\smallidx-1}\lfloor b\rfloor^{-1})[(\phi''\cdot\eta'')(\lfloor b\rfloor)]\\\notag&
|a_{\smallidx-1}|^{2\smallidx-\bigidx-7+2s}
|\lfloor b\rfloor|^{\smallidx-\half\bigidx-\half+s}d^*a_{\smallidx-1}d^*b.
\end{align}
Change $a_{\smallidx-1}\mapsto a_{\smallidx-1}\lfloor b\rfloor$. The integral equals
\begin{align}\label{int:lambda is G_1 0 in meromorphic proof 2}
&\int_{F^*}\int_{F^*}ch_{\Lambda}(b)(\phi\cdot\eta)(a_{\smallidx-1},a_{\smallidx-1}\lfloor b\rfloor^2)
(\phi'\cdot\eta')(a_{\smallidx-1})[(\phi''\cdot\eta'')(\lfloor b\rfloor)]\\\notag&
|a_{\smallidx-1}|^{2\smallidx-\bigidx-7+2s}
|\lfloor b\rfloor|^{3\smallidx-\frac32\bigidx-\frac{15}2+3s}d^*a_{\smallidx-1}d^*b.
\end{align}
If $a_{\smallidx-1}$ is bounded from below, the integrand is a smooth function of $a_{\smallidx-1}$ ($|\lfloor b\rfloor|\leq1$)
whence we get a sum of integrals of the form
\begin{align}\label{int:lambda is G_1 0 in meromorphic proof the db integral alone}
&\int_{\mathcal{P}^{k+1}}(\phi^{\star}\cdot\eta^{\star})(b^2)
[(\phi''\cdot\eta'')(b)]|b|^{3\smallidx-\frac32\bigidx-\frac{15}2+3s}d^*b,
\end{align}
where $\phi^{\star}\in\mathcal{S}(F)$ and $\eta^{\star}$ is a finite function on $F^*$. Each of these integrals can be written
as a sum of integrals of the form \eqref{int:tate example in the proof}.

Thus we can assume that $a_{\smallidx-1}$ is small with respect to the support of $\phi$, that is,
$\phi(a_{\smallidx-1},a_{\smallidx-1}\lfloor b\rfloor^2)$ is constant and can be ignored. Integral~\eqref{int:lambda is G_1 0 in meromorphic proof 2}
becomes a sum of integrals of the form
\begin{align*}
&\int_{F^*}\int_{F^*}ch_{\Lambda}(b)\varsigma_1(a_{\smallidx-1})\varsigma_2(a_{\smallidx-1}\lfloor b\rfloor^2)
(\phi'\cdot\eta')(a_{\smallidx-1})[(\phi''\cdot\eta'')(\lfloor b\rfloor)]\\\notag&
|a_{\smallidx-1}|^{2\smallidx-\bigidx-7+2s}
|\lfloor b\rfloor|^{3\smallidx-\frac32\bigidx-\frac{15}2+3s}d^*a_{\smallidx-1}d^*b.
\end{align*}
Here $\varsigma_1$ and $\varsigma_2$ are finite functions on $F^*$ (obtained from $\eta$).

Since $\varsigma_2(xy)$ can be written as a finite sum
$\sum_{r}c_r(\vartheta(x)+\vartheta(y))^{k_r}\mu_r(x)\mu_r(y)$ (where $c_r\in\C$, $0\leq k_r\in\Integers$ and $\mu_r$ denotes a character),
we can write this integral as a finite sum $\sum_uT_{u,1}T_{u,2}$, where $T_{u,1}$
takes the form \eqref{int:tate example in the proof} (obtained from a $d^*a_{\smallidx-1}$-integral)
and $T_{u,2}$ is of the form \eqref{int:lambda is G_1 0 in meromorphic proof the db integral alone}.
This completes the proof for $\Lambda=\setof{b\in F^*}{|b|<q^{-k}}$.

Finally if $\Lambda=\setof{b\in F^*}{|b|>q^{k}}$, we replace
$(\phi\cdot\eta)(a_{\smallidx-1}\lfloor b\rfloor^{-1},a_{\smallidx-1}\lfloor b\rfloor)$ with
$(\phi\cdot\eta)(a_{\smallidx-1}\lfloor b\rfloor,a_{\smallidx-1}\lfloor b\rfloor^{-1})$ in \eqref{int:lambda is G_1 0 in meromorphic proof},
change $a_{\smallidx-1}\mapsto a_{\smallidx-1}\lfloor b\rfloor$ and since if $|b|\ne1$, $\lfloor b\rfloor=\lfloor b^{-1}\rfloor$,
we can change $b\mapsto b^{-1}$ and again reach \eqref{int:lambda is G_1 0 in meromorphic proof 2} (with $(\phi\cdot\eta)(a_{\smallidx-1}\lfloor b\rfloor^2,a_{\smallidx-1})$).

The arguments when either $\smallidx=1$, $G_{\smallidx}$ is \quasisplit\ or $\smallidx>\bigidx$ are similar and simpler. E.g., if
$G_{\smallidx}$ is \quasisplit\ or $\smallidx>\bigidx$, there is no integration over $G_1$ to handle.

Altogether we have shown that each $I_s^{(i)}$ has a meromorphic continuation
$Q^{(i)}\in\C(q^{-s})$ such that $I_s^{(i)}=Q^{(i)}(q^{-s})$ for all $\Re(s)>s_1$, where $s_1\geq s'$.
The poles of $Q^{(i)}$, in $q^{-s}$, belong to a finite set $\Theta$. The constant $s_1$ and the set $\Theta$ depend only on the representations.
The result for $\Psi(W,f_s,s)$ immediately follows.
\end{proof} 
\begin{remark}
As explained in Remark~\ref{remark:taking suitable half-plane for inner integrals}, some of the justifications in the proof of
Proposition~\ref{proposition:iwasawa decomposition of the integral} can be ignored if we take some $s_0'\geq s_0$ and write the decomposition
to integrals $I_s^{(i)}$ in $\Re(s)>s_0'$. This does not interfere with the proof of the last claim, just take $s'\geq s_0'$.
\end{remark}

\begin{remark}\label{remark:tate integrals for proving convergence}
The proof of Claim~\ref{claim:meromorphic continuation 1} also implies the absolute convergence of the integrals.
In more detail, one can formally apply the arguments of
Proposition~\ref{proposition:iwasawa decomposition of the integral} to $\Psi(W,f_s,s)$ and write
$\Psi(W,f_s,s)$ as a sum of integrals $I_s^{(i)}$. Then proceed (again, formally) as in the proof above to
obtain sums of products of Tate-type integrals of the form \eqref{int:tate example in the proof}. Then the
absolute convergence of these last integrals implies the absolute convergence of $I_s^{(i)}$ and $\Psi(W,f_s,s)$.
This approach is similar to \cite{JPSS} (Section~2.7).
\end{remark}

Let $D'\subset\C$ be any domain containing a non-empty open set such that for all $s\in D'$,
$W\in\Whittaker{\pi}{\psi_{\gamma}^{-1}}$ and $f_s\in V(\tau,s)$, $\Psi(W,f_s,s)$ is absolutely convergent.
In $D'$, $\Psi(W,f_s,s)$ can be regarded as a bilinear form on
$\Whittaker{\pi}{\psi_{\gamma}^{-1}}\times V(\tau,s)$ and analogously to \eqref{eq:global integral in homspace l <= n} and \eqref{eq:global integral in homspace l > n}, it satisfies
\begin{align}\label{eq:bilinear special condition}
\begin{dcases}
\forall g\in G_{\smallidx},v\in N_{\bigidx-\smallidx},\quad\Psi(g\cdot W,(gv)\cdot f_s,s)=\psi_{\gamma}^{-1}(v)\Psi(W,f_s,s)\qquad&\smallidx\leq\bigidx,\\
\forall h\in H_{\bigidx},v\in
N^{\smallidx-\bigidx},\quad\Psi((hv)\cdot W,h\cdot
f_s,s)=\psi_{\gamma}^{-1}(v)\Psi(W,f_s,s)&\smallidx>\bigidx.
\end{dcases}
\end{align}
This (local) relation follows from the formal manipulations in the global construction but
can also be verified directly. Equivalently,
\begin{align*}
\Psi(W,f_s,s)\in\begin{dcases}Bil_{G_{\smallidx}}(\Whittaker{\pi}{\psi_{\gamma}^{-1}},V(\tau,s)_{N_{\bigidx-\smallidx},\psi_{\gamma}^{-1}})&\smallidx\leq\bigidx,\\
Bil_{H_{\bigidx}}((\Whittaker{\pi}{\psi_{\gamma}^{-1}})_{N^{\smallidx-\bigidx},\psi_{\gamma}^{-1}},V(\tau,s))&\smallidx>\bigidx.
\end{dcases}
\end{align*}
(See Sections~\ref{section:uniqneness l<=n} and \ref{section:uniqneness l>n}.)
Because $\Whittaker{\pi}{\psi_{\gamma}^{-1}}$ and $\Whittaker{\tau}{\psi}$ are quotients of $\pi$ and $\tau$ (recall that
in the construction of the integral we assume that $\tau$ is realized in $\Whittaker{\tau}{\psi}$),
according to Theorem~\ref{theorem:uniqueness}, outside of a finite number of values of $q^{-s}$ ($s\in\C$),
the space of bilinear forms
satisfying \eqref{eq:bilinear special condition} is at most
one-dimensional. We use this to conclude meromorphic continuation using Bernstein's continuation principle.
\begin{claim}\label{claim:meromorphic continuation 2}
There is a finite set $B$ of values of $q^{-s}$ (depending only on the representations) such that the following holds.
For any $W\in\Whittaker{\pi}{\psi_{\gamma}^{-1}}$ and $f_s\in\xi(\tau,hol,s)$ there is $Q\in\C(q^{-s})$ such that
$\Psi(W,f_s,s)=Q(q^{-s})$, for each $s\in D'$ where $q^{-s}\notin B$. The poles of $Q$ (in $q^{-s}$)
belong to $B$.
\end{claim}
\begin{proof}[Proof of Claim~\ref{claim:meromorphic continuation 2}]
As in \cite{GPS} (Section~12) and \cite{Soudry} (Section~8.4), the one-dimensionality of the
space of bilinear forms along with
Proposition~\ref{proposition:integral can be made constant} imply,
by virtue of Bernstein's continuation principle (\cite{Banks}), 
that $\Psi(W,f_s,s)$ with $f_s\in\xi(\tau,std,s)$ has a meromorphic continuation with the listed properties.
Now note that if $f_s=\sum_{i=1}^mP_if_s^{(i)}\in \xi(\tau,hol,s)$ ($P_i\in\C[q^{-s},q^s]$, $f_s^{(i)}\in\xi(\tau,std,s)$),
$\Psi(W,f_s,s)=\sum_{i=1}^mP_i\Psi(W,f_s^{(i)},s)$ for all $s\in D'$. Therefore the result extends to $f_s\in\xi(\tau,hol,s)$.
\end{proof} 
\begin{remark}
See also Bump, Friedberg and Furusawa \cite{BFF} (Section~5) who used
Bernstein's principle to conclude meromorphic continuation for
$\smallidx=1$ and unramified representations.
\end{remark}
\begin{remark}\label{remark:proving equality of W and f in any domain without Bernstein continuation principle}
There is another method for proving this claim, without appealing to Bernstein's continuation principle.
Denote by $\Psi(W,f_s,s)_{D'}$ the integral defined in $D'$. According to Claim~\ref{claim:meromorphic continuation 1},
the integral $\Psi(W,f_s,s)$ defined in $\Re(s)>>0$ extends to a
function in $\C(q^{-s})$. Thereby it is possible to regard $\Psi(W,f_s,s)$ as a bilinear form on
$\Whittaker{\pi}{\psi_{\gamma}^{-1}}\times
V(\tau,s)$ satisfying \eqref{eq:bilinear special condition}, for all $s\in D'$ such that $q^{-s}$ does not belong to some finite set
(see Section~\ref{subsection:the gamma factor}). By Theorem~\ref{theorem:uniqueness},
for all but a finite set of values of $q^{-s}$, the space of such forms is at most one-dimensional. Hence there is a finite set $B$ such that
for all $s\in D'$ with $q^{-s}\notin B$, $\Psi(W,f_s,s)$ is proportional to $\Psi(W,f_s,s)_{D'}$.
It is left to show that the proportionality factor is $1$,
then the meromorphic continuation of $\Psi(W,f_s,s)_{D'}$ is given by that of $\Psi(W,f_s,s)$.
Indeed, Proposition~\ref{proposition:integral can be made constant} implies that for a specific selection of
$W$ and $f_s$, $\Psi(W,f_s,s)$ is absolutely convergent
for all $s\in \C$ and equals $1$, hence the meromorphic continuation of $\Psi(W,f_s,s)$ is $1$ and it is equal to
$\Psi(W,f_s,s)_{D'}$ for all $s\in D'$.
\end{remark}


In particular apply Claim~\ref{claim:meromorphic continuation 2} in the domain
$\setof{s\in\C}{\Re(s)>s_0}$. Take $s_1\geq s_0$ large enough, so that the right half-plane $\Re(s)>s_1$ will not contain any $s$ for which $q^{-s}\in B$, where $B$ is given by
Claim~\ref{claim:meromorphic continuation 2}. Then for any $W$ and
$f_s\in\xi(\tau,hol,s)$ there is $Q\in\C(q^{-s})$ such that $\Psi(W,f_s,s)=Q(q^{-s})$, for all $\Re(s)>s_1$. This resembles
Claim~\ref{claim:meromorphic continuation 1}.

Both claims extend to rational sections. We restate the meromorphic continuation in a form
which is convenient for most of our arguments.
\begin{proposition}\label{proposition:meromorphic continuation}
Let $W\in\Whittaker{\pi}{\psi_{\gamma}^{-1}}$ and $f_s\in\xi(\tau,rat,s)$.
The integral $\Psi(W,f_s,s)$ defined in $\Re(s)>>0$ extends to a function in $\C(q^{-s})$.
Specifically, there are $s_1\in\R$ and $Q\in\C(q^{-s})$ such that $\Psi(W,f_s,s)=Q(q^{-s})$ for all $\Re(s)>s_1$.
If $f_s\in\xi(\tau,hol,s)$, the poles of $Q$ (in $q^{-s}$) belong to a finite set
and both this set and $s_1$ depend only on the representations.
If $f_s\notin\xi(\tau,hol,s)$, take $0\ne P\in\C[q^{-s}]$ satisfying $Pf_s\in\xi(\tau,hol,s)$.
Then the finite set of poles (containing the poles of $Q$) and $s_1$ depend only
on the representations and $P$.
\end{proposition}
\begin{proof}[Proof of Proposition~\ref{proposition:meromorphic continuation}]
The assertions concerning a holomorphic section follow from either of the above claims.
Regarding $f_s\in\xi(\tau,rat,s)$, let $s_1$ be such that the following holds for all $s$ with $\Re(s)>s_1$:
\begin{enumerate}
\item $\Psi(W,Pf_s,s)$ is absolutely convergent,
\item There is $R\in\C(X)$, independent of $s$, such that $\Psi(W,Pf_s,s)=R(q^{-s})$,
\item $P(q^{-s})\ne0$,
\item $\Psi(W,f_s,s)$ is absolutely convergent.
\end{enumerate}
Since $Pf_s\in\xi(\tau,hol,s)$, the first two properties can be satisfied by a constant $s_1$ depending
only on the representations. Regarding the last, observe that $f_s\in V(\tau,s)$ for each $s$ with $\Re(s)>s_1$, because
$Pf_s\in\xi(\tau,hol,s)$ and $P(q^{-s})\ne0$. Hence $\Psi(W,f_s,s)$ is absolutely convergent whenever
$\Re(s)>s_1$ and $s_1\geq s_0$.

We have an equality of integrals $\Psi(W,Pf_s,s)=P\Psi(W,f_s,s)$ for $s$ such that $\Re(s)>s_1$, whence
$\Psi(W,f_s,s)$ extends to $Q=P^{-1}R\in\C(q^{-s})$. The poles of $Q$ depend only on the representations (due to $R$)
and $P$.
\end{proof} 

We stress that we usually consider holomorphic sections, rational sections only appear as images of
specific intertwining operators. The following two particular cases essentially cover all integrals with rational sections that are needed.
Let $f_s\in\xi(\tau,hol,s)$ and apply Proposition~\ref{proposition:meromorphic continuation}
to $\nintertwiningfull{\tau}{s}f_s\in\xi(\tau^*,rat,1-s)$, with $P\in\C[q^{s-1}]$ such that $P\equalun\ell_{\tau}(s)^{-1}$
(see Section~\ref{subsection:the intertwining operator for tau}).
Then there is $Q\in\C(q^{-s})$ such that $\Psi(W,\nintertwiningfull{\tau}{s}f_s,1-s)=Q(q^{-s})$,
in a left half-plane ($\Re(1-s)>>0$) depending only on the representations. Moreover, the poles of $Q$ belong to a finite set depending only on the representations.
A similar result holds for $\Psi(W,\nintertwiningfull{\tau^*}{1-s}f_{1-s}',s)$,
where $f_{1-s}'\in\xi(\tau^*,hol,1-s)$. We use $P\in\C[q^{-s}]$ such that
$P\equalun\ell_{\tau^*}(1-s)^{-1}$.


Proposition~\ref{proposition:meromorphic continuation} and the proof of Claim~\ref{claim:meromorphic continuation 1}
imply the following corollary to
Proposition~\ref{proposition:iwasawa decomposition of the integral}.
\begin{corollary}\label{corollary:corro to iwasawa decomposition of the integral}
Let $W\in\Whittaker{\pi}{\psi_{\gamma}^{-1}}$ and $f_s\in\xi(\tau,hol,s)$.
The decomposition of Proposition~\ref{proposition:iwasawa decomposition of the integral},
namely $\Psi(W,f_s,s)=\sum_{i=1}^mI_s^{(i)}$, holds for all $s$ where both sides are defined by meromorphic continuation.
\end{corollary}

\section{Realization of $\tau$ induced from
$\tau_1\otimes\tau_2$}\label{subsection:Realization of tau for
induced representation} When
$\tau=\cinduced{P_{\bigidx_1,\bigidx_2}}{\GL{\bigidx}}{\tau_1\otimes\tau_2}$
it is convenient for the manipulations of $\Psi(W,f_s,s)$ to use an
explicit integral formula for the Whittaker functional on $\tau$. To
ensure convergence of this Jacquet integral, the representations
$\tau_1$ and $\tau_2$ are twisted using an auxiliary complex parameter $\zeta$ as
in \cite{Sh2,JPSS,Soudry,Soudry2}.

Let $\varepsilon_1=\tau_1\absdet{}^{\zeta}$ and
$\varepsilon_2=\tau_2\absdet{}^{-\zeta}$ be realized in their
Whittaker models with respect to $\psi$ and let
$\varepsilon=\cinduced{P_{\bigidx_1,\bigidx_2}}{\GL{\bigidx}}{\varepsilon_1\otimes\varepsilon_2}$ be realized in
the space of the induced representation (i.e., in
$V_{P_{\bigidx_1,\bigidx_2}}^{\GL{\bigidx}}(\varepsilon_1\otimes\varepsilon_2,(\half,\half))$).
Throughout this section we use the notation and results of Sections~\ref{subsection:realization of induced tau_1 and tau_2} and \ref{subsection:the multiplicativity of the intertwining operator for tau induced}. Consider the representation
\begin{align*}
\cinduced{Q_{\bigidx_1}}{H_{\bigidx}}{(\varepsilon_1\otimes
\cinduced{Q_{\bigidx_2}}{H_{\bigidx_2}}{\varepsilon_2\alpha^{s}})\alpha^{s}}
\end{align*}
and denote its space by $V'(\varepsilon_1\otimes\varepsilon_2,(s,s))$.
Any $\varphi_s\in V'(\varepsilon_1\otimes\varepsilon_2,(s,s))$ defines an element $\widehat{f}_{\varphi_s}\in V_{Q_{\bigidx}}^{H_{\bigidx}}(\Whittaker{\varepsilon}{\psi},s)$ by
\begin{align}\label{iso:iso 1}
&\widehat{f}_{\varphi_s}(h,b)=\absdet{b}^{-\half\bigidx-s+\half}\int_{Z_{\bigidx_2,\bigidx_1}}\varphi_s(\omega_{\bigidx_1,\bigidx_2}zbh,I_{\bigidx_1},I_{2\bigidx_2+1},I_{\bigidx_2})
\psi^{-1}(z)dz.
\end{align}
Here $h\in H_{\bigidx}$, $b\in\GL{\bigidx}$.
This (Jacquet) integral always has a sense as a
principal value, but there exists a $\zeta_0>0$ which depends only
on $\tau_1$ and $\tau_2$, such that for all $\zeta$ with
$\Re(\zeta)>\zeta_0$ it is absolutely convergent - $\int_{Z_{\bigidx_2,\bigidx_1}}|\varphi_s|(\omega_{\bigidx_1,\bigidx_2}zbh,1,1,1)dz<\infty$, for all $s$ and
$\varphi_s$.

The integral $\Psi(W,\varphi_s,s)$ is just
$\Psi(W,\widehat{f}_{\varphi_s},s)$ with formula \eqref{iso:iso 1}:
\begin{align}\label{int:n_1<l<=n_2 example starting form}
&\Psi(W,\varphi_s,s)\\\notag&=\begin{dcases}\int_{\lmodulo{U_{G_{\smallidx}}}{G_{\smallidx}}}W(g)
\int_{R_{\smallidx,\bigidx}}
\int_{Z_{\bigidx_2,\bigidx_1}}
\varphi_s(\omega_{\bigidx_1,\bigidx_2}zw_{\smallidx,\bigidx}rg,1,1,1)
\psi^{-1}(z)\psi_{\gamma}(r)dzdrdg&\smallidx\leq\bigidx,\\
\int_{\lmodulo{U_{H_{\bigidx}}}{H_{\bigidx}}}
\int_{R^{\smallidx,\bigidx}}\int_{Z_{\bigidx_2,\bigidx_1}}W(rw^{\smallidx,\bigidx}h)
\varphi_{s}(\omega_{\bigidx_1,\bigidx_2}zh,1,1,1)\psi^{-1}(z)dzdrdh&\smallidx>\bigidx.
\end{dcases}
\end{align}
We stress that $\Psi(W,\varphi_s,s)$ refers to the triple integral. It is absolutely convergent if it is convergent when we substitute $|W|,|\varphi_s|$ for $W,\varphi_s$ and drop the characters. For example when
$\smallidx\leq\bigidx$ and $\Re(s)>>\Re(\zeta)>>0$ (see Claim~\ref{claim:convergence n_1<l<n_2 starting integral} below),
\begin{align*}
\int_{\lmodulo{U_{G_{\smallidx}}}{G_{\smallidx}}}
\int_{R_{\smallidx,\bigidx}}
\int_{Z_{\bigidx_2,\bigidx_1}}|W|(g)|\varphi_{s}|(\omega_{\bigidx_1,\bigidx_2}zw_{\smallidx,\bigidx}rg,1,1,1)
dzdrdg<\infty.
\end{align*}
In contrast, when we write $\Psi(W,\widehat{f}_{\varphi_s},s)$ we
interpret the $dz$-integration by principal value.
Observe that if $\Psi(W,\varphi_s,s)$ is absolutely convergent at $s$ then so is
$\Psi(W,\widehat{f}_{\varphi_s},s)$.

Note that if $\tau$ is irreducible, one can also assume that
$\varepsilon$ is irreducible, this holds for all but finitely many values $q^{-\zeta}$ ($\zeta\in\C$).

We can regard $\zeta$ as fixed, e.g. take $\zeta=0$ and simply use
\eqref{iso:iso 1} to manipulate $\Psi(W,\widehat{f}_{\varphi_s},s)$.
Alternatively we may regard $\zeta$ as a parameter similar to $s$.
Another option is to allow $\zeta$ to vary inside a fixed compact
set. We describe these approaches below.

\subsection{Fixed $\zeta$}\label{subsection:fixed zeta}
Assume that $\zeta$ is fixed. 
The sections $\xi(\varepsilon_1\otimes\varepsilon_2,std,(s,s))$ are realized as images of functions from the space of \begin{align*}
\cinduced{Q_{\bigidx_1}\cap K_{H_{\bigidx}}}{K_{H_{\bigidx}}}{\varepsilon_1\otimes
\cinduced{Q_{\bigidx_2}\cap
K_{H_{\bigidx_2}}}{K_{H_{\bigidx_2}}}{\varepsilon_2}}
\end{align*}
Let $\varphi_s\in\xi(\varepsilon_1\otimes\varepsilon_2,hol,(s,s))$. For a fixed $s$, $\varphi_s$ defines an element $\widehat{f}_{\varphi_s}\in
V_{Q_{\bigidx}}^{H_{\bigidx}}(\Whittaker{\varepsilon}{\psi},s)$ by \eqref{iso:iso 1}. We claim that $\widehat{f}_{\varphi_s}$ is a holomorphic section.
\begin{claim}\label{claim:f_s is the image of varphi_s when zeta is
fixed} If $\varphi_s\in \xi(\varepsilon_1\otimes\varepsilon_2,std,(s,s))$ (resp.
$\varphi_s\in \xi(\varepsilon_1\otimes\varepsilon_2,hol,(s,s))$),
$\widehat{f}_{\varphi_s}\in \xi_{Q_{\bigidx}}^{H_{\bigidx}}(\Whittaker{\varepsilon}{\psi},std,s)$ (resp.
$\widehat{f}_{\varphi_s}\in \xi_{Q_{\bigidx}}^{H_{\bigidx}}(\Whittaker{\varepsilon}{\psi},hol,s)$).
The mapping $\varphi_s\mapsto \widehat{f}_{\varphi_s}$ takes
$\xi(\varepsilon_1\otimes\varepsilon_2,std,(s,s))$ onto $\xi_{Q_{\bigidx}}^{H_{\bigidx}}(\Whittaker{\varepsilon}{\psi},std,s)$ and
$\xi(\varepsilon_1\otimes\varepsilon_2,hol,(s,s))$ onto $\xi_{Q_{\bigidx}}^{H_{\bigidx}}(\Whittaker{\varepsilon}{\psi},hol,s)$.
\end{claim}
\begin{proof}[Proof of Claim~\ref{claim:f_s is the image of varphi_s when zeta is fixed}] 
Let $\varphi_s\in\xi(\varepsilon_1\otimes\varepsilon_2,std,(s,s))$. In order to prove
$\widehat{f}_{\varphi_s}\in\xi_{Q_{\bigidx}}^{H_{\bigidx}}(\Whittaker{\varepsilon}{\psi},std,s)$, it is enough to show that for any $k\in K_{H_{\bigidx}}$, the function $\widehat{f}_{\varphi_s}(k,1)$ is independent of $s$. This holds because when we write an Iwasawa decomposition
$\omega_{\bigidx_1,\bigidx_2}z=v_zt_zk_z$ for $z\in
Z_{\bigidx_2,\bigidx_1}$, with $v_z\in Z_{\bigidx}$, $t_z\in
A_{\bigidx}$ and $k_z\in K_{\GL{\bigidx}}$, we have
$\absdet{t_z}=1$. Now the definition implies that $\xi(\varepsilon_1\otimes\varepsilon_2,hol,(s,s))$ is mapped into
$\xi_{Q_{\bigidx}}^{H_{\bigidx}}(\Whittaker{\varepsilon}{\psi},hol,s)$.

Regarding the other direction, first observe that for any
$f\in V_{Q_{\bigidx}\cap K_{H_{\bigidx}}}^{K_{H_{\bigidx}}}(\Whittaker{\varepsilon}{\psi})$ there is
$f'\in V_{Q_{\bigidx}\cap K_{H_{\bigidx}}}^{K_{H_{\bigidx}}}(\varepsilon)$ such that
\begin{align*}
f(k,1)=\int_{Z_{\bigidx_2,\bigidx_1}}f'(k,\omega_{\bigidx_1,\bigidx_2}z,I_{\bigidx_1},I_{\bigidx_2})\psi^{-1}(z)dz\qquad(k\in K_{H_{\bigidx}}).
\end{align*}
It follows that when $f$ (resp. $f'$) is extended to $f_s\in V_{Q_{\bigidx}}^{H_{\bigidx}}(\Whittaker{\varepsilon}{\psi},s)$
(resp. $f_s'\in V_{Q_{\bigidx}}^{H_{\bigidx}}(\varepsilon,s)$) using the Iwasawa decomposition,
\begin{align}\label{int:functional is onto inside claim proof}
f_{s}(h,b)=\int_{Z_{\bigidx_2,\bigidx_1}}f_{s}'(h,\omega_{\bigidx_1,\bigidx_2}zb,I_{\bigidx_1},I_{\bigidx_2})\psi^{-1}(z)dz\qquad(h\in H_{\bigidx},b\in\GL{\bigidx}).
\end{align}
Since the Iwasawa decomposition defines the onto mapping
\begin{align*}
V_{Q_{\bigidx}\cap K_{H_{\bigidx}}}^{K_{H_{\bigidx}}}(\Whittaker{\varepsilon}{\psi})\rightarrow \xi_{Q_{\bigidx}}^{H_{\bigidx}}(\Whittaker{\varepsilon}{\psi},std,s)
\end{align*}
(see Section~\ref{subsection:sections}),
we get that any $f_{s}\in\xi_{Q_{\bigidx}}^{H_{\bigidx}}(\Whittaker{\varepsilon}{\psi},std,s)$
can be defined by \eqref{int:functional is onto inside claim proof} using $f_{s}'\in\xi_{Q_{\bigidx}}^{H_{\bigidx}}(\varepsilon,std,s)$ and then $f_{s}\in\xi_{Q_{\bigidx}}^{H_{\bigidx}}(\Whittaker{\varepsilon}{\psi},hol,s)$ is defined using $f_{s}'\in\xi_{Q_{\bigidx}}^{H_{\bigidx}}(\varepsilon,hol,s)$.

The representations
$\cinduced{Q_{\bigidx}}{H_{\bigidx}}{\varepsilon\alpha^s}$ and
$\cinduced{Q_{\bigidx_1}}{H_{\bigidx}}{(\varepsilon_1\otimes\cinduced{Q_{\bigidx_2}}{H_{\bigidx_2}}{\varepsilon_2\alpha^{s}})\alpha^{s}}$ are isomorphic according to
\begin{align*}
(h,b_1,h_2,b_2)\mapsto f'(h_2h,I_{\bigidx},b_1,b_2),\quad (h,b,b_1,b_2)\mapsto\absdet{b}^{-\half\bigidx-s+\half}\varphi(bh,b_1,I_{2\bigidx_2+1},b_2).
\end{align*}
Here $f'\in V_{Q_{\bigidx}}^{H_{\bigidx}}(\varepsilon,s)$,
$\varphi$ belongs to the space of $\cinduced{Q_{\bigidx_1}}{H_{\bigidx}}{(\varepsilon_1\otimes\cinduced{Q_{\bigidx_2}}{H_{\bigidx_2}}{\varepsilon_2\alpha^{s}})\alpha^{s}}$, $h_2\in H_{\bigidx_2}$ and 
$b_i\in\GL{\bigidx_i}$. These isomorphisms also define isomorphisms
between
\begin{align*}
\cinduced{Q_{\bigidx}\cap
K_{H_{\bigidx}}}{K_{H_{\bigidx}}}{\varepsilon}, \quad
\cinduced{Q_{\bigidx_1}\cap K_{H_{\bigidx}}}{K_{H_{\bigidx}}}{\varepsilon_1\otimes
\cinduced{Q_{\bigidx_2}\cap
K_{H_{\bigidx_2}}}{K_{H_{\bigidx_2}}}{\varepsilon_2}}.
\end{align*}
Thus for any $f_{s}'\in\xi_{Q_{\bigidx}}^{H_{\bigidx}}(\varepsilon,std,s)$ (resp. $f_{s}'\in\xi_{Q_{\bigidx}}^{H_{\bigidx}}(\varepsilon,hol,s)$) we can find
$\varphi_s\in\xi(\varepsilon_1\otimes\varepsilon_2,std,(s,s))$ (resp. $\varphi_s\in\xi(\varepsilon_1\otimes\varepsilon_2,hol,(s,s))$) such that
\begin{align*}
f_{s}'(h,b,b_1,b_2)=\absdet{b}^{-\half\bigidx-s+\half}\varphi_s(bh,b_1,I_{2\bigidx_2+1},b_2).
\end{align*}
Plugging this into \eqref{int:functional is onto inside claim proof} we obtain $f_{s}=\widehat{f}_{\varphi_s}$ whence the mapping $\varphi_s\mapsto \widehat{f}_{\varphi_s}$ is
onto as required.
\end{proof} 
The following observations relate $\Psi(W,\varphi_s,s)$ to
$\Psi(W,\widehat{f}_{\varphi_s},s)$.
\begin{claim}\label{claim:deducing varphi and f_s are interchangeable in psi}
Assume that for $W\in\Whittaker{\pi}{\psi_{\gamma}^{-1}}$ and
$\varphi_s\in\xi(\varepsilon_1\otimes\varepsilon_2,hol,s)$,
$\Psi(W,\varphi_s,s)$ is absolutely convergent at $s$. 
Then $\Psi(W,\varphi_s,s)=\Psi(W,\widehat{f}_{\varphi_s},s)$ at $s$, as integrals. 
\end{claim}
\begin{proof}[Proof of Claim~\ref{claim:deducing varphi and f_s are interchangeable in psi}] 
Assume for instance $\smallidx\leq\bigidx$, the other case
being similar. We use the idea of Jacquet, Piatetski-Shapiro and
Shalika \cite{JPSS} (p.~424). Let $s$ be such that
$\Psi(W,\varphi_s,s)$ is absolutely convergent. Thus
\begin{align*}
\int_{R_{\smallidx,\bigidx}}\int_{Z_{\bigidx_2,\bigidx_1}}
|\varphi_s|(\omega_{\bigidx_1,\bigidx_2}zw_{\smallidx,\bigidx}rg,1,1,1)dzdr<\infty
\end{align*}
for all $g\in\support{W}\setminus E_1$ where $E_1$ is a set of zero
measure (here $\setminus$ denotes set difference). Since $W$ and $\varphi_s$ are smooth it follows that the
last integral is finite for all $g\in\support{W}$. Hence
$\widehat{f}_{\varphi_s}(w_{\smallidx,\bigidx}rg,1)$ is absolutely
convergent for all $g\in\support{W}$ and $r\in
R_{\smallidx,\bigidx}\setminus E_2$, for a set $E_2$ of zero
measure. Again by smoothness, this is true for all $r\in
R_{\smallidx,\bigidx}$. According to Fubini's Theorem, $\Psi(W,\varphi_s,s)$ can be computed using iterated
integrals and by
Lebesgue's Dominated Convergence Theorem the
$dz$-integration in $\Psi(W,\varphi_s,s)$ can be replaced with
$\widehat{f}_{\varphi_s}(w_{\smallidx,\bigidx}rg,1)$ (defined by
principal value). Therefore the integrals
$\Psi(W,\varphi_s,s)$ and $\Psi(W,\widehat{f}_{\varphi_s},s)$ are equal at $s$.
\end{proof} 
Claim~\ref{claim:convergence n_1<l<n_2 starting integral} below
implies that Claim~\ref{claim:deducing varphi and f_s are
interchangeable in psi} is \nontrivial. We state the result here.
\begin{corollary}\label{corollary:deducing varphi and f_s are interchangeable in psi for all varphi}
There exists a constant $\zeta_0$ depending only on the representations
$\pi$, $\tau_1$ and $\tau_2$ such that the following holds. Fix $\zeta$ with $\Re(\zeta)>\zeta_0$. Then there is a
constant $s_1$ depending only on $\zeta$ and the representations, satisfying
$\Psi(W,\varphi_s,s)=\Psi(W,\widehat{f}_{\varphi_s},s)$ for all $\Re(s)>s_1$, $W\in\Whittaker{\pi}{\psi_{\gamma}^{-1}}$
and $\varphi_s\in\xi(\varepsilon_1\otimes\varepsilon_2,hol,s)$.
Therefore, the integral $\Psi(W,\varphi_s,s)$ also has a meromorphic continuation to a function in $\C(q^{-s})$,
given by the continuation of $\Psi(W,\widehat{f}_{\varphi_s},s)$. Consequently,
$\Psi(W,\varphi_s,s)=\Psi(W,\widehat{f}_{\varphi_s},s)$ in $\C(q^{-s})$.
\end{corollary}
\begin{proof}[Proof of Corollary~\ref{corollary:deducing varphi and f_s are interchangeable in psi for all varphi}] 
Claim~\ref{claim:convergence n_1<l<n_2 starting integral} shows the
existence of such constants $\zeta_0,s_1$ (with $s_1$ depending on $\zeta$, $\Re(\zeta)>\zeta_0$), for which
$\Psi(W,\varphi_s,s)$ is absolutely
convergent for all $\Re(s)>s_1$, $W$ and $\varphi_s$. Hence the assumption of
Claim~~\ref{claim:deducing varphi and f_s are interchangeable in
psi} is satisfied for all such $s$.
\end{proof} 
Finally we have the following result, to be used in Chapter~\ref{chapter:upper boubnds on the gcd} in the proof of Lemma~\ref{lemma:n_1<l<=n_2 basic identity}.
\begin{lemma}\label{lemma:replacing a section with another}
Assume that $\tau$ is realized in $\Whittaker{\tau}{\psi}$. For $\varphi_{s}\in\xi(\varepsilon_1\otimes\varepsilon_2,hol,(s,s))$
let $\widehat{f}_{\varphi_s}$ be
defined by \eqref{iso:iso 1} and write
$\widehat{f}_{\varphi_s}=\sum_{i=1}^mP_ich_{k_iN,W_{\zeta}^{(i)},s}$
with $P_i\in\C[q^{-s},q^{s}]$,
$W_{\zeta}^{(i)}\in\Whittaker{\varepsilon}{\psi}$ as in
\eqref{eq:holomorphic section as combination of disjoint standard
sections}. Let
$f_s\in\xi_{Q_{\bigidx}}^{H_{\bigidx}}(\tau,hol,s)$. Then there exists $\varphi_{s}\in\xi(\varepsilon_1\otimes\varepsilon_2,hol,(s,s))$
such that this expression for $\widehat{f}_{\varphi_s}$ has the following properties:
\begin{enumerate}
\item If we let $\zeta$ vary, the data $(m,k_i,N,P_i)$ is
independent of $\zeta$.
\item For all $\zeta$, $W_{\zeta}^{(i)}$ is right-invariant by
$(\rconj{k_i^{-1}}N)\cap Q_{\bigidx}$.
\item For each $b\in
\GL{\bigidx}$ there is $Q_i\in\C[q^{-\zeta},q^{\zeta}]$ such that for
all $\zeta$, $W_{\zeta}^{(i)}(b)=Q_i$.
\item If we put $\zeta=0$ in the expression for $\widehat{f}_{\varphi_s}$ we get $\widehat{f}_{\varphi_s}=f_s$.
\end{enumerate}
\end{lemma}
\begin{proof}[Proof of Lemma~\ref{lemma:replacing a section with another}]
Denote by $U$ the space of $\cinduced{P_{\bigidx_1,\bigidx_2}}{\GL{\bigidx}}{\tau_1\otimes\tau_2}$.
We make the following general remark. For $v\in U$, $k\in K_{H_{\bigidx}}$ and a compact open subgroup $N<K_{H_{\bigidx}}$ such that $v$ is right-invariant by
$(\rconj{k^{-1}}N)\cap Q_{\bigidx}$,
the function $ch_{kN,v,s}\in\xi_{Q_{\bigidx}}^{H_{\bigidx}}(\cinduced{P_{\bigidx_1,\bigidx_2}}{\GL{\bigidx}}{\tau_1\otimes\tau_2},std,s)$
is defined (with a minor abuse of notation) by
\begin{align*}
ch_{kN,v,s}(h,b,b_1,b_2)=
\begin{dcases}
\absdet{a}^{\half\bigidx+s-\half}v(ba,b_1,b_2)& h\in Q_{\bigidx}kN,\\
0& \text{otherwise}.
\end{dcases}
\end{align*}
Here $h\in H_{\bigidx}$, $b\in\GL{\bigidx}$, $b_i\in\GL{\bigidx_i}$ and for the case $h\in Q_{\bigidx}kN$ we wrote
$h=diag(a,1,a^*)ukn'$ with $a\in \GL{\bigidx}$, $u\in U_{\bigidx}$ and $n'\in N$. Let $W_v\in\Whittaker{\tau}{\psi}$ be the Whittaker function
obtained by applying the Whittaker functional, defined by the $dz$-integration of \eqref{int:functional is onto inside claim proof}, to $v$. A calculation implies
\begin{align*}
ch_{kN,W_v,s}(h,b)=\int_{Z_{\bigidx_2,\bigidx_1}}ch_{kN,v,s}(h,\omega_{\bigidx_1,\bigidx_2}zb,I_{\bigidx_1},I_{\bigidx_2})\psi^{-1}(z)dz.
\end{align*}
In other words, $ch_{kN,W_v,s}\in\xi_{Q_{\bigidx}}^{H_{\bigidx}}(\tau,std,s)$ is the image of
$ch_{kN,v,s}$ under \eqref{int:functional is onto inside claim proof}.

As we explained in the proof of Claim~\ref{claim:f_s is the image of varphi_s when zeta is
fixed}, for any $f_s\in\xi_{Q_{\bigidx}}^{H_{\bigidx}}(\tau,hol,s)$ we can find
$f_s'\in\xi_{Q_{\bigidx}}^{H_{\bigidx}}(\cinduced{P_{\bigidx_1,\bigidx_2}}{\GL{\bigidx}}{\tau_1\otimes\tau_2},hol,s)$
such that $f_s$ is the image of $f_s'$ under \eqref{int:functional is onto inside claim proof}.
Write $f_s'=\sum_{i=1}^mP_ich_{k_iN,v^{(i)},s}$ with
$P_i\in\C[q^{-s},q^{s}]$, $v^{(i)}\in U$ as in
\eqref{eq:holomorphic section as combination of disjoint standard
sections}. According to the definition, $v^{(i)}$ is right-invariant by $(\rconj{k_i^{-1}}N)\cap Q_{\bigidx}$.
Then $f_s=\sum_{i=1}^mP_ich_{k_iN,W^{(i)},s}$ where
$W^{(i)}=W_{v^{(i)}}\in\Whittaker{\tau}{\psi}$ is obtained from $v^{(i)}$. 

Any $v^{(i)}$ can be extended to an element $v_{\zeta}^{(i)}$ in the space $U_{\zeta}$ of
$\cinduced{P_{\bigidx_1,\bigidx_2}}{\GL{\bigidx}}{\varepsilon_1\otimes\varepsilon_2}$,
such that $v_{0}^{(i)}=v^{(i)}$ and for all $\zeta$, $v_{\zeta}^{(i)}$ is right-invariant
by $(\rconj{k_i^{-1}}N)\cap Q_{\bigidx}$. Then for each $\zeta$,
\begin{align*}
f_{\zeta,s}'=\sum_{i=1}^mP_ich_{k_iN,v_{\zeta}^{(i)},s}\in\xi_{Q_{\bigidx}}^{H_{\bigidx}}(\cinduced{P_{\bigidx_1,\bigidx_2}}{\GL{\bigidx}}{\varepsilon_1\otimes\varepsilon_2},hol,s)
\end{align*}
is defined and $f_{0,s}'=f_s'$. Let $W_{\zeta}^{(i)}\in\Whittaker{\varepsilon}{\psi}$ be obtained from $v_{\zeta}^{(i)}$ by applying the Whittaker functional to $v_{\zeta}^{(i)}$.
According to \cite{CS2} (Section~2, see also \cite{Sh2} Section~3), for any fixed $b$,
$W_{\zeta}^{(i)}(b)\in \C[q^{-\zeta},q^{\zeta}]$. Also $W_{\zeta}^{(i)}$ is right-invariant by $(\rconj{k_i^{-1}}N)\cap Q_{\bigidx}$ and $W_{0}^{(i)}=W^{(i)}$.


We define $\varphi_{s}\in\xi(\varepsilon_1\otimes\varepsilon_2,hol,(s,s))$ using the isomorphism described in the proof of Claim~\ref{claim:f_s is the image of varphi_s when zeta is fixed}. Namely,
$\varphi_{s}(h,b_1,h_2,b_2)=f_{\zeta,s}'(h_2h,I_{\bigidx},b_1,b_2)$.
Then $\widehat{f}_{\varphi_s}=\sum_{i=1}^mP_ich_{k_iN,W_{\zeta}^{(i)},s}$ has the required properties.
\end{proof}

\subsection{Variable $\zeta$}\label{subsection zeta vairable}
Here we consider $\zeta$ as a parameter, similar to $s$. Assume that $\tau_i$ is realized in its Whittaker model
$\Whittaker{\tau_i}{\psi}$, $i=1,2$. Let
\begin{align*}
&\Pi_{s+\zeta,s-\zeta}=\cinduced{Q_{\bigidx_1,\bigidx_2}}{H_{\bigidx}}{(\tau_1\otimes\tau_2)\alpha^{(s+\zeta,s-\zeta)}},\\\notag &\Pi_{s+\zeta,s-\zeta}'=\cinduced{Q_{\bigidx_1}}{H_{\bigidx}}{(\tau_1\otimes
\cinduced{Q_{\bigidx_2}}{H_{\bigidx_2}}{\tau_2\alpha^{s-\zeta}})\alpha^{s+\zeta}}.
\end{align*}
These are just $\Pi_{s_1,s_2}$ and $\Pi_{s_1,s_2}'$ of Section~\ref{subsection:realization of induced tau_1 and tau_2} with $s_1=s+\zeta$, $s_2=s-\zeta$.
Since in general if $\phi$ is a representation of $\GL{\bigidx}$, $\Whittaker{\phi\absdet{}^{\zeta}}{\psi}\isomorphic\Whittaker{\phi}{\psi}\absdet{}^{\zeta}$,
\begin{align*}
\Pi_{s+\zeta,s-\zeta}'\isomorphic \cinduced{Q_{\bigidx_1}}{H_{\bigidx}}{(\varepsilon_1\otimes
\cinduced{Q_{\bigidx_2}}{H_{\bigidx_2}}{\varepsilon_2\alpha^{s}})\alpha^{s}}.
\end{align*}
Denote the space of $\Pi_{s+\zeta,s-\zeta}'$ by
$V'(\tau_1\otimes\tau_2,(s+\zeta,s-\zeta))$.
The integral $\Psi(W,\varphi_{\zeta,s},s)$ with
$\varphi_{\zeta,s}\in V'(\tau_1\otimes\tau_2,(s+\zeta,s-\zeta))$, is
defined by \eqref{int:n_1<l<=n_2 example starting form} (for brevity, we write $\varphi_{\zeta,s}$ instead of $\varphi_{s+\zeta,s-\zeta}$). 
The following claim proves that there is a domain in $\zeta$ and $s$
depending only on the representations, for which
$\Psi(W,\varphi_{\zeta,s},s)$ is absolutely convergent. This will imply that the integrals have a meromorphic continuation to functions in $\C(q^{-\zeta},q^{-s})$.
\begin{claim}\label{claim:convergence n_1<l<n_2 starting integral}
There exist constants $C_1,C_2,C_3>0$ which depend only on the
representations $\pi$, $\tau_1$ and $\tau_2$, such that
for all $\zeta$ and $s$ in a domain $D$ of the form
\begin{align*}
\qquad C_1<C_2\Re(\zeta)+C_3<\Re(s),
\end{align*}
$\Psi(W,\varphi_{\zeta,s},s)$ is absolutely convergent for all
$W\in\Whittaker{\pi}{\psi_{\gamma}^{-1}}$ and
$\varphi_{\zeta,s}\in V'(\tau_1\otimes\tau_2,(s+\zeta,s-\zeta))$.
\end{claim}
\begin{proof}[Proof of Claim~\ref{claim:convergence n_1<l<n_2 starting integral}] 
We settle for a proof when $\smallidx\leq\bigidx$. We follow the
line of reasoning of Soudry \cite{Soudry2} (Lemma~3.1).
Using the Iwasawa decomposition of $G_{\smallidx}$ we are reduced to
considering
\begin{align*}
\int_{A_{\smallidx-1}}\int_{G_1}|W|(ax)\delta_{B_{G_{\smallidx}}}^{-1}(a)
\int_{R_{\smallidx,\bigidx}}
\int_{Z_{\bigidx_2,\bigidx_1}}|\varphi_{\zeta,s}|(\omega_{\bigidx_1,\bigidx_2}zw_{\smallidx,\bigidx}rax,1,1,1)
dzdrdxda.
\end{align*}
We may assume that $s$ and $\zeta$ are real. Put $diag(a,I_{\bigidx-\smallidx+1})
=diag(\ddot{a},\dot{a})$ where $\ddot{a}\in A_{\bigidx_2}$ and $\dot{a}\in A_{\bigidx_1}$.
Shifting $a$ to the left in $\varphi_{\zeta,s}$ 
we obtain
\begin{align*}
&\int_{A_{\smallidx-1}}\int_{G_1}|W|(ax)\delta_{B_{G_{\smallidx}}}^{-1}(a)
\absdet{\dot{a}}^{\smallidx-\half\bigidx_1-\bigidx_2+s+\zeta-\half}
\absdet{\ddot{a}}^{\smallidx-\half\bigidx_2+s-\zeta-\half}\\\notag
&\int_{R_{\smallidx,\bigidx}}
\int_{Z_{\bigidx_2,\bigidx_1}}|\varphi_{\zeta,s}|(\omega_{\bigidx_1,\bigidx_2}zw_{\smallidx,\bigidx}rx,\dot{a},1,\ddot{a})
dzdrdxda.
\end{align*}
Further assume $\smallidx>\bigidx_2$. Assume that
$\varphi_{\zeta,s}$ is invariant on the right for
$\mathcal{N}_{H_{\bigidx},k_0}$, $k_0>>0$. Since when $G_1$ is
\quasisplit\ or $[x]$ is bounded, the $dx$-integration
may be ignored, we first treat the case $[x]>q^k$ where $k$ is chosen
according to Lemma~\ref{lemma:torus elements in double cosets},
applied with $k_0$. Arguing as in the proof of
Proposition~\ref{proposition:dr integral for torus of G_l} we get a
sum of two integrals of the form
\begin{align}\label{int:convergence n_1<l<n_2 starting integral Iwasawa 2}
&\int_{A_{\smallidx-1}}\int_{G_1}ch_{\Lambda}(x)|W|(ax)\delta_{B_{G_{\smallidx}}}^{-1}(a)
(\absdet{\dot{a}}[x]^{-1})^{\smallidx-\half\bigidx_1-\bigidx_2+s+\zeta-\half}\\\notag
&\absdet{\ddot{a}}^{\smallidx-\half\bigidx_2+s-\zeta-\half}
\int_{R_{\smallidx,\bigidx}}
\int_{Z_{\bigidx_2,\bigidx_1}}|\varphi_{\zeta,s}|(\omega_{\bigidx_1,\bigidx_2}zw_{\smallidx,\bigidx}rh,\dot{a}x^{\wedge},1,\ddot{a})
dzdrdxda,
\end{align}
where $\Lambda=G_1^{0,k},G_1^{\infty,k}$ and $h=h_1,h_2$ (resp.),
see Proposition~\ref{proposition:iwasawa decomposition of the
integral} and Lemma~\ref{lemma:torus elements in double cosets}, and
$x^{\wedge}=diag(I_{\smallidx-\bigidx_2-1},\lfloor
x\rfloor,I_{\bigidx-\smallidx})\in A_{\bigidx_1}$.

Therefore it is enough to consider the convergence of
\begin{align*}
&\int_{A_{\smallidx-1}}\int_{G_1}ch_{\Lambda}(x)|W|(ax)\delta_{B_{G_{\smallidx}}}^{-1}(a)
(\absdet{\dot{a}}[x]^{-1})^{\smallidx-\half\bigidx_1-\bigidx_2+s+\zeta-\half}\\\notag
&\absdet{\ddot{a}}^{\smallidx-\half\bigidx_2+s-\zeta-\half}
\int_{R_{\smallidx,\bigidx}}
\int_{Z_{\bigidx_2,\bigidx_1}}|\varphi_{\zeta,s}|(\omega_{\bigidx_1,\bigidx_2}zw_{\smallidx,\bigidx}r,\dot{a}x^{\wedge},1,\ddot{a})
dzdrdxda
\end{align*}
for an arbitrary $\varphi_{\zeta,s}$. 
Replace $\varphi_{\zeta,s}$ with
$w^{-1}\cdot\varphi_{\zeta,s}$ for
$w=\omega_{\bigidx_1,\bigidx_2}w_{\smallidx,\bigidx}$. Since
$\rconj{w^{-1}}R_{\smallidx,\bigidx}<\overline{U_{\bigidx}}$ and
$\rconj{\omega_{\bigidx_2,\bigidx_1}}Z_{\bigidx_2,\bigidx_1}=\overline{Z_{\bigidx_1,\bigidx_2}}$, it is enough to
bound
\begin{align}\label{int:convergence n_1<l<n_2 starting integral Iwasawa 3}
&\int_{A_{\smallidx-1}}\int_{G_1}ch_{\Lambda}(x)|W|(ax)\delta_{B_{G_{\smallidx}}}^{-1}(a)
(\absdet{\dot{a}}[x]^{-1})^{\smallidx-\half\bigidx_1-\bigidx_2+s+\zeta-\half}\\\notag
&\absdet{\ddot{a}}^{\smallidx-\half\bigidx_2+s-\zeta-\half}
\int_{\overline{U_{\bigidx}}}
\int_{\overline{Z_{\bigidx_1,\bigidx_2}}}|\varphi_{\zeta,s}|(zu,\dot{a}x^{\wedge},1,\ddot{a})
dzdudxda.
\end{align}
Decompose $z=v_zt_zk_z\in
Z_{\bigidx}A_{\bigidx}K_{\GL{\bigidx}}$, $u=v_ut_uk_u\in
U_{H_{\bigidx}}A_{\bigidx}K_{H_{\bigidx}}$ according to the Iwasawa
decomposition. Write $t_z=diag(\dot{t_z},\ddot{t_z})$,
$t_u=diag(\dot{t_u},\ddot{t_u})$ for $\dot{t_z},\dot{t_u}\in
A_{\bigidx_1}$ and $\ddot{t_z},\ddot{t_u}\in A_{\bigidx_2}$. Note
that $\absdet{t_z}=1$ whence $\absdet{\dot{t_z}}=\absdet{\ddot{t_z}}^{-1}$.
Also $z$ normalizes $U_{H_{\bigidx}}$ and $t_u$ normalizes
$\overline{Z_{\bigidx_1,\bigidx_2}}$ changing
$dz\mapsto\absdet{\dot{t_u}}^{-\bigidx_2}\absdet{\ddot{t_u}}^{\bigidx_1}dz$,
so the $dzdu$-integral equals
\begin{align*}
&\int_{\overline{U_{\bigidx}}}
\int_{\overline{Z_{\bigidx_1,\bigidx_2}}}\absdet{t_u}^{\half\bigidx_1+s-\half}\absdet{\dot{t_u}}^{\zeta}\absdet{\ddot{t_u}}^{\half\bigidx-\zeta}\\\notag&\absdet{\ddot{t_z}}^{-\half\bigidx-2\zeta}
|\varphi_{\zeta,s}|(k_zk_u,\dot{a}x^{\wedge}\dot{t_u}\dot{t_z},1,\ddot{a}\ddot{t_u}\ddot{t_z})
dzdu.
\end{align*}

For any $A\in\Mat{m\times m'}$ denote by $\matrow{A}{j}$ the $j$-th
row of $A$ and let $\norm{A}$ be the sup-norm of $A$. For
$1\leq j\leq2\bigidx+1$ put $\widetilde{j}=2\bigidx+2-j$. By the results
of Soudry \cite{Soudry} (Section~11.15), if
$t_z=diag((t_z)_1,\ldots,(t_z)_{\bigidx})$ and
$t_u=diag((t_u)_1,\ldots,(t_u)_{\bigidx})$,
\begin{align*}
&\max_{k\leq j\leq
\bigidx}{\norm{\matrow{z}{j}}}\leq|(t_z)_k\cdot\ldots\cdot(t_z)_{\bigidx}|\leq\prod_{j=k}^{\bigidx}\norm{\matrow{z}{j}}\qquad
(\forall \bigidx_1<k\leq\bigidx),\\\notag &\max_{1\leq j\leq
k}{\norm{\matrow{u}{\widetilde{j}}}}\leq|(t_u)_1\cdot\ldots\cdot(t_u)_k|^{-1}\leq\prod_{j=1}^k\norm{\matrow{u}{\widetilde{j}}}\qquad(\forall
1\leq k\leq\bigidx).
\end{align*}
This implies
\begin{align*}
&\absdet{\ddot{t_z}}^{-1}\leq(\max_{\bigidx_1<j\leq
\bigidx}{\norm{\matrow{z}{j}}})^{-1},\\
&\absdet{t_u}\leq\norm{u}^{-1},\\
&\absdet{\dot{t_u}}\leq(\max_{1\leq j\leq
\bigidx_1}{\norm{\matrow{u}{\widetilde{j}}}})^{-1},\\
&\absdet{\ddot{t_u}}^{-1}\leq\norm{u}^{\bigidx}\absdet{\dot{t_u}}.
\end{align*}
The last inequality follows using $\absdet{\ddot{t_u}}^{-1}=\absdet{t_u}^{-1}\absdet{\dot{t_u}}$. 
Furthermore, by \cite{Soudry} (\textit{loc. cit.}) there exists a constant
$c>0$ such that $\norm{z}^{-c}\leq\alpha(t_z)\leq\norm{z}^c$ for all
$\alpha\in\Delta_{\GL{\bigidx}}$, and also
$\norm{u}^{-c}\leq\alpha(t_u)\leq\norm{u}^c$ where
$\alpha\in\Delta_{H_{\bigidx}}$ is of the form
$\alpha=\epsilon_i-\epsilon_{i+1}$ with $1\leq i<\bigidx$
(i.e., $\alpha(t_u)=(t_u)_i(t_u)_{i+1}^{-1}$). The constant $c$ depends only on $\bigidx$.

Now $|\varphi_{\zeta,s}|(k_zk_u,\dot{a}x^{\wedge}\dot{t_u}\dot{t_z},1,\ddot{a}\ddot{t_u}\ddot{t_z})$
can be bounded first by replacing $\varphi_{\zeta,s}$ with a finite sum of functions from $V'(\tau_1\otimes\tau_2,(s+\zeta,s-\zeta))$,
each supported on a double coset $Q_{\bigidx_1}kN$ ($k\in K_{H_{\bigidx}}$, $N<K_{H_{\bigidx}}$), then using the asymptotic expansion
of Section~\ref{section:whittaker props} to bound Whittaker functions from $\Whittaker{\tau_1}{\psi}$ and
$\Whittaker{\tau_2}{\psi}$. See \cite{Soudry} (Section~4.4) and \cite{JPSS2} (Section~2.3).

Assume $\zeta>\half\bigidx$ and $s>\half$. Collecting the
arguments above, the $dzdu$-integral is bounded by a finite sum of
integrals
\begin{align}\label{int:convergence n_1<l<n_2 dzdu 2}
&\mu(\dot{a}x^{\wedge})\mu'(\ddot{a})\int_{\overline{U_{\bigidx}}}
\int_{\overline{Z_{\bigidx_1,\bigidx_2}}}\norm{u}^{-\half\bigidx_1-s+\half}
(\max_{1\leq j\leq
\bigidx_1}{\norm{\matrow{u}{\widetilde{j}}}})^{-2\zeta+\half\bigidx}\norm{u}^{\bigidx\zeta-\half\bigidx^2}\\\notag&
(\max_{\bigidx_1<j\leq
\bigidx}{\norm{\matrow{z}{j}}})^{-\half\bigidx-2\zeta}
(\norm{u}\cdot\norm{z})^{cM}dzdu.
\end{align}
Here $\mu$ and $\mu'$ are positive characters and $M>0$ is a
constant depending only on $\tau_1$ and $\tau_2$. Note that
$\max_{\bigidx_1<j\leq \bigidx}{\norm{\matrow{z}{j}}}=\norm{z}$.
Then for the $dz$-integration in \eqref{int:convergence n_1<l<n_2
dzdu 2} to converge we need $-\half\bigidx-2\zeta+cM<<0$, i.e.
$\zeta>C_1$ for a constant $C_1>0$ depending only on the
representations. Also $\max_{1\leq j\leq
\bigidx_1}{\norm{\matrow{u}{\widetilde{j}}}}\geq1$ and
$-2\zeta+\half\bigidx<0$ (because $-\zeta+\half\bigidx<0$) so this
factor only decreases the integrand and can be ignored. Then to bound
\eqref{int:convergence n_1<l<n_2 dzdu 2} we need to bound
\begin{align*}
&\int_{\overline{U_{\bigidx}}}
\norm{u}^{-\half\bigidx_1-s+\half+\bigidx\zeta-\half\bigidx^2+cM}du.
\end{align*}
For $s>>\bigidx\zeta+cM$ this integral is absolutely convergent. Thus $s$ belongs to a domain of the form $s>C_2\zeta+C_3$ with constants $C_2=\bigidx$ and $C_3>0$. Returning to integral~\eqref{int:convergence
n_1<l<n_2 starting integral Iwasawa 3} we have to bound
\begin{align*}
&\int_{A_{\smallidx-1}}\int_{G_1}ch_{\Lambda}(x)|W|(ax)\delta_{B_{G_{\smallidx}}}^{-1}(a)
(\absdet{\dot{a}}[x]^{-1})^{\smallidx-\half\bigidx_1-\bigidx_2+s+\zeta-\half}\\\notag
&\absdet{\ddot{a}}^{\smallidx-\half\bigidx_2+s-\zeta-\half}
\mu(\dot{a}x^{\wedge})\mu'(\ddot{a})dxda.
\end{align*}
This is possible once we take $C_3$ which is large enough, depending on
$\pi$, $\tau_1$ and $\tau_2$. To prove this one can use the asymptotic expansion of $W$ (see Section~\ref{section:whittaker props}),
argue as in \cite{JPSS} (Section~2.7) and eventually bound Tate-type integrals. It is also possible to use a bound of $W$ by a gauge
(see \cite{Soudry} Section~4). 

As mentioned earlier, if $G_{\smallidx}$ is \quasisplit\ or $[x]$ is bounded, one
drops the $dx$-integration and proceeds similarly. 

Note that for the case $\smallidx\leq\bigidx_2$ integral
\eqref{int:convergence n_1<l<n_2 starting integral Iwasawa 2} is
replaced with
\begin{align*}
&\int_{A_{\smallidx-1}}\int_{G_1}ch_{\Lambda}(x)|W|(ax)\delta_{B_{G_{\smallidx}}}^{-1}(a)
(\absdet{\ddot{a}}[x]^{-1})^{\smallidx-\half\bigidx_2+s-\zeta-\half}\\\notag
&\int_{R_{\smallidx,\bigidx}}
\int_{Z_{\bigidx_2,\bigidx_1}}|\varphi_{\zeta,s}|(\omega_{\bigidx_1,\bigidx_2}zw_{\smallidx,\bigidx}rh,1,1,\ddot{a}\cdot
diag(I_{\smallidx-1},\lfloor x\rfloor,I_{\bigidx_2-\smallidx}))
dzdrdxda.
\end{align*}
The proof is similar to the above.
\end{proof} 

Repeating the argument in the proof of Claim~\ref{claim:deducing
varphi and f_s are interchangeable in psi}, for $\varphi_{\zeta,s}\in\xi(\tau_1\otimes\tau_2,hol,(s+\zeta,s-\zeta))$,
$\Psi(W,\varphi_{\zeta,s},s)=\Psi(W,\widehat{f}_{\varphi_{\zeta,s}},s)$
as integrals defined in $D$.

Similarly to Proposition~\ref{proposition:integral can be made
constant}, we show that the integral $\Psi(W,\varphi_{\zeta,s},s)$
can be made constant. In fact, the following result is slightly
stronger than that proposition.
\begin{proposition}\label{proposition:integral can be made constant zeta phi version}
There exist $W\in\Whittaker{\pi}{\psi_{\gamma}^{-1}}$ and
$\varphi_{\zeta,s}\in\xi(\tau_1\otimes\tau_2,hol,(s+\zeta,s-\zeta))$
such that $\Psi(W,\varphi_{\zeta,s},s)$ is absolutely convergent and
equals $1$, for all $\zeta$ and $s$. Moreover, there is a constant
$C$ (independent of $\zeta$ and $s$) such that for all $\zeta$ and $s$,
\begin{align*}
\begin{dcases}
\int_{\lmodulo{U_{H_{\bigidx}}}{H_{\bigidx}}}\int_{R^{\smallidx,\bigidx}}\int_{Z_{\bigidx_2,\bigidx_1}}|W|(rw^{\smallidx,\bigidx}h)|\varphi_{\zeta,s}|(\omega_{\bigidx_1,\bigidx_2}zh,1,1,1)dzdrdh<C&\smallidx>\bigidx,\\
\int_{\lmodulo{U_{G_{\smallidx}}}{G_{\smallidx}}}\int_{R_{\smallidx,\bigidx}}\int_{Z_{\bigidx_2,\bigidx_1}}
|W|(g)|\varphi_{\zeta,s}|(\omega_{\bigidx_1,\bigidx_2}zw_{\smallidx,\bigidx}rg,1,1,1)dzdrdg<C&\smallidx\leq\bigidx.
\end{dcases}
\end{align*}
\end{proposition}

\begin{proof}[Proof of Proposition~\ref{proposition:integral can be made constant zeta phi version}] 
We use the same arguments of Proposition~\ref{proposition:integral
can be made constant} except that we choose the element
$W'\in\Whittaker{\tau}{\psi}$ more carefully.

Assume $\smallidx>\bigidx$. Fix $\zeta$ and $s$. We realize $\cinduced{P_{\bigidx_1,\bigidx_2}}{\GL{\bigidx}}{\tau_1\absdet{}^{\zeta}\otimes\tau_2\absdet{}^{-\zeta}}$ in the space
$V_{P_{\bigidx_1,\bigidx_2}}^{\GL{\bigidx}}(\tau_1\otimes\tau_2,(\zeta+\half,-\zeta+\half))$. Select
$\theta_{\zeta}\in V_{P_{\bigidx_1,\bigidx_2}}^{\GL{\bigidx}}(\tau_1\otimes\tau_2,(\zeta+\half,-\zeta+\half))$ such that
$\support{\theta_{\zeta}}=P_{\bigidx_1,\bigidx_2}\omega_{\bigidx_1,\bigidx_2}\mathcal{N}_{\GL{\bigidx},k'}$,
$k'>>0$, $\theta_{\zeta}(1,1,1)=1$ and for $a=diag(a_1,a_2)\in
A_{\bigidx_1,\bigidx_2}$, $z\in Z_{\bigidx_1,\bigidx_2}$ and
$u\in\mathcal{N}_{\GL{\bigidx},k'}$,
\begin{align*}
\theta_{\zeta}(az\omega_{\bigidx_1,\bigidx_2}u,1,1)=\delta_{P_{\bigidx_1,\bigidx_2}}^{\half}(a)\absdet{a_1}^{\zeta}\absdet{a_2}^{-\zeta}\theta_{\zeta}(1,a_1,a_2).
\end{align*}
Let $W_0,W_1,W=(w^{\smallidx,\bigidx})^{-1}\cdot W_1$ be as in that
proposition. I.e., $W_1$ is selected using Lemma~\ref{lemma:W with
small support} with $W_0$, $j=0$ and $k_1>>k'$. Let
\begin{align*}
f_{\zeta,s}'\in
V_{Q_{\bigidx}}^{H_{\bigidx}}(\cinduced{P_{\bigidx_1,\bigidx_2}}{\GL{\bigidx}}{\tau_1\absdet{}^{\zeta}\otimes\tau_2\absdet{}^{-\zeta}},s) \end{align*}
be with
$\support{f_{\zeta,s}'}=Q_{\bigidx}\mathcal{N}_{H_{\bigidx},k}$,
$k>>k_1$, such that
\begin{align*}
f_{\zeta,s}'(au,1,1,1)=\delta_{Q_{\bigidx}}^{\half}(a)\absdet{a}^{s-\half}\theta_{\zeta}(a,1,1),
\end{align*}
where $a\in\GL{\bigidx}\isomorphic M_{\bigidx}$ and $u\in\mathcal{N}_{H_{\bigidx},k}$. Define a function $\varphi_{\zeta,s}$ in the space of
\begin{align*}
\cinduced{Q_{\bigidx_1}}{H_{\bigidx}}{(\tau_1\otimes\cinduced{Q_{\bigidx_2}}{H_{\bigidx_2}}{\tau_2\alpha^{s-\zeta}})\alpha^{s+\zeta}}
\end{align*}
by $\varphi_{\zeta,s}(h,b_1,h_2,b_2)=f_{\zeta,s}'(h_2h,I_{\bigidx},b_1,b_2)$  (see the proof of Claim~\ref{claim:f_s is the image of varphi_s when zeta is fixed}). If we let $\zeta$ and $s$ vary,
we see that the restriction of the mapping
$h\mapsto f_{\zeta,s}'(h,1,1,1)$ to $K_{H_{\bigidx}}$ is independent of $\zeta$ and $s$. Therefore
\begin{align*}
\varphi_{\zeta,s}\in\xi(\tau_1\otimes\tau_2,std,(s+\zeta,s-\zeta)).
\end{align*}
For $\zeta$ and $s$ such that $\Psi(W,\varphi_{\zeta,s},s)$ is
absolutely convergent,
\begin{align*}
\Psi(W,\varphi_{\zeta,s},s)&=\int_{\overline{B_{\GL{\bigidx}}}}\int_{\overline{U_{\bigidx}}}
\int_{R^{\smallidx,\bigidx}}W(rw^{\smallidx,\bigidx}au)\delta_{Q_{\bigidx}}^{\half}(a)\absdet{a}^{s-\half}\\\notag&
\quad\times\int_{Z_{\bigidx_2,\bigidx_1}}f_{\zeta,s}'(u,\omega_{\bigidx_1,\bigidx_2}za,1,1)\psi^{-1}(z)\delta(a)dzdrduda\notag\\
&=cW_0(1)\int_{Z_{\bigidx_2,\bigidx_1}}\theta_{\zeta}(\omega_{\bigidx_1,\bigidx_2}z,1,1)\psi^{-1}(z)dz.
\end{align*}
Here $c>0$ is a volume constant independent of $\zeta$ and $s$.
Now $\theta_{\zeta}(\omega_{\bigidx_1,\bigidx_2}z,1,1)$ vanishes
unless $\rconj{\omega_{\bigidx_2,\bigidx_1}}z\in
P_{\bigidx_1,\bigidx_2}\mathcal{N}_{\GL{\bigidx},k'}$ and since
$\rconj{\omega_{\bigidx_2,\bigidx_1}}z\in\overline{Z_{\bigidx_1,\bigidx_2}}$,
the $dz$-integration can be disregarded and we get
$c'W_0(1)\theta_{\zeta}(1,1,1)$ for some volume constant $c'>0$
(independent of $\zeta,s$). After suitably normalizing $W_0$,
$\Psi(W,\varphi_{\zeta,s},s)=1$.

The same arguments show that there is a constant $C$ such that for
any $\zeta,s\in\C$,
\begin{align*}
&\int_{\lmodulo{U_{H_{\bigidx}}}{H_{\bigidx}}}\int_{R^{\smallidx,\bigidx}}|W|(rw^{\smallidx,\bigidx}h)\int_{Z_{\bigidx_2,\bigidx_1}}|\varphi_{\zeta,s}|(\omega_{\bigidx_1,\bigidx_2}zh,1,1,1)dzdrdh\\\notag
&=c|W_0|(1)\int_{Z_{\bigidx_2,\bigidx_1}}|\theta_{\zeta}|(\omega_{\bigidx_1,\bigidx_2}z,1,1)dz=c'|W_0|(1)|\theta_{\zeta}|(1,1,1)<C.
\end{align*}

In the case $\smallidx\leq\bigidx$ replace $\theta_{\zeta}$ with
$t_{\gamma}\cdot\theta_{\zeta}$, e.g.
$\support{t_{\gamma}\cdot\theta_{\zeta}}=P_{\bigidx_1,\bigidx_2}\omega_{\bigidx_1,\bigidx_2}\mathcal{N}_{\GL{\bigidx},k'}$ ($t_{\gamma}$ is defined in Lemma~\ref{lemma:f with small support}).
Take $W_0,W_1$ as above and $W=W_1$. The same arguments of
Lemma~\ref{lemma:f with small support} prove that there is a function
$f_{\zeta,s}'$ (in the previous space) such that for any
$b\in\GL{\bigidx}$ and $v\in V=(\overline{V_{\smallidx-1}}\rtimes
G_1)<G_{\smallidx}$,
\begin{align}\label{eq:props of f_s'}
&\int_{R_{\smallidx,\bigidx}}f_{\zeta,s}'(w_{\smallidx,\bigidx}rv,b,1,1)\psi_{\gamma}(r)dr=\begin{dcases}
P_s(\gamma)t_{\gamma}\cdot\theta_{\zeta}(b,1,1)& v\in O_{k}\subset
\mathcal{N}_{G_{\smallidx},k-k_0},\\0 &\text{otherwise}.
\end{dcases}
\end{align}
Here $P_s(\gamma)=|\gamma|^{\smallidx(\half\bigidx+s-\half)}$, $k_0\geq0$ is some constant and $O_{k}$ is a measurable subset.
Let
$\varphi_{\zeta,s}$ be defined as above and normalized by $P_s(\gamma)^{-1}\in\C[q^{-s},q^s]^*$. Then
\begin{align*}
\varphi_{\zeta,s}\in\xi(\tau_1\otimes\tau_2,hol,(s+\zeta,s-\zeta)).
\end{align*}
Now for some volume constants $c'',c'''>0$,
\begin{align*}
\Psi(W,\varphi_{\zeta,s},s)&=c''W(1)\int_{Z_{\bigidx_2,\bigidx_1}}t_{\gamma}\cdot\theta_{\zeta}(\omega_{\bigidx_1,\bigidx_2}z,1,1)\psi^{-1}(z)dz=c'''W(1)t_{\gamma}\cdot\theta_{\zeta}(1,1,1).
\end{align*}
The same arguments show that for any $\zeta$ and $s$,
\begin{align}\label{int:shahidi n_1<l<n_2 first integral subst absolute value}
&\int_{\lmodulo{U_{G_{\smallidx}}}{G_{\smallidx}}}\int_{R_{\smallidx,\bigidx}}
\int_{Z_{\bigidx_2,\bigidx_1}}|W|(g)|\varphi_{\zeta,s}|(\omega_{\bigidx_1,\bigidx_2}zw_{\smallidx,\bigidx}rg,1,1,1)dzdrdg\\\notag
&=c'''|W|(1)|t_{\gamma}\cdot\theta_{\zeta}|(1,1,1)<C.
\end{align}
This completes the proof.
\end{proof} 
The integral $\Psi(W,\varphi_{\zeta,s},s)$, where
$\varphi_{\zeta,s}\in\xi(\tau_1\otimes\tau_2,rat,(s+\zeta,s-\zeta))$,
extends to a function in $\C(q^{-\zeta},q^{-s})$. Similarly to the
description in Section~\ref{subsection:meromorphic continuation}
this is deduced either from the Iwasawa decomposition or from
Bernstein's continuation principle. To use the Iwasawa decomposition one can apply
Proposition~\ref{proposition:iwasawa decomposition of the integral} to $\Psi(W,\widehat{f}_{\varphi_{\zeta,s}},s)$ in $D$, since there
$\Psi(W,\varphi_{\zeta,s},s)=\Psi(W,\widehat{f}_{\varphi_{\zeta,s}},s)$. Then as we prove in Claim~\ref{claim:iwasawa decomposition of the integral with zeta in tau} below,
$W'\in\Whittaker{\tau}{\psi}$ is replaced by $W_{\zeta}'\in\Whittaker{\cinduced{P_{\bigidx_1,\bigidx_2}}{\GL{\bigidx}}{\tau_1\absdet{}^{\zeta}\otimes\tau_2\absdet{}^{-\zeta}}}{\psi}$
(see also the proof of Claim~\ref{claim:gamma is rational in s and zeta}).
To apply Bernstein's principle use
Proposition~\ref{proposition:integral can be made constant zeta phi
version} instead of Proposition~\ref{proposition:integral can be
made constant}.

We also conclude that
$\Psi(W,\widehat{f}_{\varphi_{\zeta,s}},s)$, defined in $D$, extends to $\C(q^{-\zeta},q^{-s})$
and the equality
$\Psi(W,\varphi_{\zeta,s},s)=\Psi(W,\widehat{f}_{\varphi_{\zeta,s}},s)$ 
holds in $\C(q^{-\zeta},q^{-s})$.

\subsection{$\zeta$ in a fixed compact}\label{subsection:zeta in a fixed compact} For any compact subset
$C\subset\C$ there is a constant $s_C>0$ depending only on $C$,
$\pi$, $\tau_1$ and $\tau_2$, such that for all $\zeta\in C$ and
$\Re(s)>s_C$, $\Psi(W,f_s,s)$ is absolutely convergent for all $W$
and $f_s\in V_{Q_{\bigidx}}^{H_{\bigidx}}(\Whittaker{\varepsilon}{\psi},s)$. E.g., for
$\smallidx\leq\bigidx$,
\begin{align*}
\int_{\lmodulo{U_{G_{\smallidx}}}{G_{\smallidx}}} |W|(g)
\int_{R_{\smallidx,\bigidx}}|f_s|(w_{\smallidx,\bigidx}rg,1)drdg<\infty.
\end{align*}
Moreover, this convergence is uniform in $\zeta$. To see this one follows the
proofs of convergence for these integrals and notices that convergence is governed by
the exponents of the representations. The twist by $\zeta$ has the
effect of twisting the exponents by $\absdet{}^{\mp\zeta}$. See
Remark~\ref{remark:tate integrals for proving convergence} and \cite{JPSS,Soudry,Soudry2}.


In Sections~\ref{subsection:The basic identity tau} and \ref{subsection:sub mult second var proof body} we
use an analogue of Proposition~\ref{proposition:iwasawa decomposition of the integral},
when $f_s\in\xi(\Whittaker{\varepsilon}{\psi},hol,s)$ and $\zeta$ is allowed to vary in some fixed compact set. The proposition provides an Iwasawa decomposition
separately for each $\zeta$, but we will need to describe this decomposition when $\zeta$ varies.
\begin{claim}\label{claim:iwasawa decomposition of the integral with zeta in tau}
Let $W\in\Whittaker{\pi}{\psi_{\gamma}^{-1}}$, $\varphi_{\zeta,s}\in\xi(\tau_1\otimes\tau_2,std,(s+\zeta,s-\zeta))$ and
set $f_{\zeta,s}=\widehat{f}_{\varphi_{\zeta,s}}\in\xi(\Whittaker{\varepsilon}{\psi},std,s)$.
Let $C\subset\C$ be a compact subset and $s_C$ be a constant as above.
Then there is a decomposition $\Psi(W,f_{\zeta,s},s)=\sum_{i=1}^mI_s^{(i)}$ for all $\zeta\in C$ and $\Re(s)>s_C$, and the form
of $I_s^{(i)}$ is \eqref{int:iwasawa decomposition of the integral l<=n split} when $\smallidx\leq\bigidx$ or
\eqref{int:iwasawa decomposition of the integral l>n} for $\smallidx>\bigidx$. 
The function $W'\in\Whittaker{\tau}{\psi}$
appearing in $I_s^{(i)}$ is replaced by $W_{\zeta}'\in\Whittaker{\varepsilon}{\psi}$ with the following properties:
\begin{enumerate}
\item There is a compact open subgroup $N<\GL{\bigidx}$ such that $W_{\zeta}'$ is right-invariant by $N$ for all $\zeta$.
\item For any $b\in\GL{\bigidx}$, $W_{\zeta}'(b)\in\C[q^{-\zeta},q^{\zeta}]$.
\end{enumerate}
The function $W_{\zeta}'$ is the only term in 
$I_s^{(i)}$, which depends on $\zeta$.
\end{claim}
\begin{proof}[Proof of Claim~\ref{claim:iwasawa decomposition of the integral with zeta in tau}] 
We review the steps in the proof of Proposition~\ref{proposition:iwasawa decomposition of the integral} and describe the dependency on $\zeta$. Consider the case $\smallidx\leq\bigidx$. The subgroup of $K_{H_{\bigidx}}$ under which
$f_{\zeta,s}$ is right-invariant, is independent of $\zeta$. Also for any $k\in K_{H_{\bigidx}}$,
$k\cdot f_{\zeta,s}=\widehat{f}_{k\cdot\varphi_{\zeta,s}}$ and since $k\cdot\varphi_{\zeta,s}\in\xi(\tau_1\otimes\tau_2,std,(s+\zeta,s-\zeta))$,
$k\cdot f_{\zeta,s}$ is an image of a standard section. Hence for all $\zeta\in C$ and $\Re(s)>s_C$,
$\Psi(W,f_{\zeta,s},s)$ can be written as a sum of integrals of the form \eqref{int:first iwasawa decomp}, i.e.,
\begin{align*}
\int_{T_{G_{\smallidx}}}W^{\diamond}(t)\Omega(t\cdot f_{\zeta,s}')\delta_{B_{G_{\smallidx}}}^{-1}(t)dt.
\end{align*}
Here $W^{\diamond}\in\Whittaker{\pi}{\psi_{\gamma}^{-1}}$ and $f_{\zeta,s}'=\widehat{f}_{\varphi_{\zeta,s}'}$
with $\varphi_{\zeta,s}'\in\xi(\tau_1\otimes\tau_2,std,(s+\zeta,s-\zeta))$. The only dependence of this writing on $\zeta$ and $s$ is in the section $f_{\zeta,s}'$, that is,
there are $m$ integrals of this form such that their sum is equal to $\Psi(W,f_{\zeta,s},s)$, for all
$\zeta\in C$ and $\Re(s)>s_C$.

Since $\varphi_{\zeta,s}$ is a standard section, we can write
\begin{align*}
f_{\zeta,s}'=\sum_{i=1}^vch_{y_iU,W_{\zeta}^{(i)},s},
\end{align*}
where $y_i\in K_{H_{\bigidx}}$, $U<K_{H_{\bigidx}}$ is a compact open subgroup, $W_{\zeta}^{(i)}\in\Whittaker{\varepsilon}{\psi}$
is right-invariant by $\rconj{y_i^{-1}}U\cap Q_{\bigidx}$ for all $\zeta$ and $W_{\zeta}^{(i)}(b)\in\C[q^{-\zeta},q^{\zeta}]$ ($b\in\GL{\bigidx}$). The data $(v,y_i,U)$ do not depend on $\zeta$ and $s$ and
the functions $W_{\zeta}^{(i)}$ do not depend on $s$.

Assume that $G_{\smallidx}$ is \quasisplit. Then we can replace the $dt$-integration with an integration over
$A_{\smallidx-1}$. In this process we replace $f_{\zeta,s}'$ with sections $x\cdot f_{\zeta,s}'$
where $x\in G_1$. In general if $W_{\zeta}\in\Whittaker{\varepsilon}{\psi}$ satisfies the properties stated in the claim,
so does $b\cdot W_{\zeta}$ for any $b\in\GL{\bigidx}$. According to the proof of Claim~\ref{claim:translation of standard section is nearly standard},
for any $h\in H_{\bigidx}$,
\begin{align}\label{eq:f s zeta first writing iwasawa for zeta 2}
h\cdot f_{\zeta,s}'=\sum_{i=1}^{v'}Q_i'ch_{y_i'U',(W_{\zeta}')^{(i)},s},
\end{align}
where $Q_i'\in\C[q^{-s},q^{s}]$ is independent of $\zeta$ and this writing has properties similar to the above
(e.g., $(W_{\zeta}')^{(i)}\in\Whittaker{\varepsilon}{\psi}$
is right-invariant by $\rconj{(y_i')^{-1}}U'\cap Q_{\bigidx}$ for all $\zeta$).
Therefore we need to consider integrals of the form
\begin{align*}
\int_{A_{\smallidx-1}}W^{\diamond}(a)\Omega(a\cdot ch_{y_i'U',(W_{\zeta}')^{(i)},s})\delta_{B_{G_{\smallidx}}}^{-1}(a)da.
\end{align*}

The next step is to apply Proposition~\ref{proposition:dr integral for torus of G_l} to $\Omega(a\cdot ch_{y_i'U',(W_{\zeta}')^{(i)},s})$. The element $a$ will be shifted to the
second argument of $ch_{y_i'U',(W_{\zeta}')^{(i)},s}$, hence can be ignored for now and we consider $\Omega(ch_{kN,W_{\zeta},s})$
(i.e., $k=y_i'$, $N=U'$, $W_{\zeta}=(W_{\zeta}')^{(i)}$).

In order to evaluate $\Omega(ch_{kN,W_{\zeta},s})$, in the proof of Proposition~\ref{proposition:dr integral for torus of G_l}
we replaced $ch_{kN,W_{\zeta},s}$ by $(ch_{kN,W_{\zeta},s})^{O,\psi_{\gamma}}$ for a certain compact open subgroup $O$.
The subgroup $O$ depended on $\support{ch_{kN,W_{\zeta},s}}$ (and on $\psi_{\gamma}$). We can take $O$ so that it would be independent of $\zeta$. This is because
for a fixed $\zeta$, either $W_{\zeta}\nequiv0$ whence $\support{ch_{kN,W_{\zeta},s}}=Q_{\bigidx}kN$, or $W_{\zeta}\equiv0$ and then any $O$ is suitable
since $\support{ch_{kN,W_{\zeta},s}}\equiv0$. Then for all $\zeta$, $(ch_{kN,W_{\zeta},s})^{O,\psi_{\gamma}}\in F_{w_0}$ (in the notation of
Section~\ref{subsection:the integral over R_{\smallidx,\bigidx}}).
Next write
\begin{align*}
(ch_{kN,W_{\zeta},s})^{O,\psi_{\gamma}}=\sum_{i=1}^mP_i
ch_{k_iN',W_{\zeta}^{(i)},s}
\end{align*}
with $P_i\in\C[q^{-s},q^{s}]$, $k_i\in K_{H_{\bigidx}}$, $N'<K_{H_{\bigidx}}$ is a compact open subgroup,
$W_{\zeta}^{(i)}\in\Whittaker{\varepsilon}{\psi}$
is right-invariant by $\rconj{k_i^{-1}}N'\cap Q_{\bigidx}$ for all $\zeta$ and $W_{\zeta}^{(i)}(b)\in\C[q^{-\zeta},q^{\zeta}]$.
The data $(m,k_i,N')$ is independent of $\zeta$ and $s$, $P_i$ is independent of $\zeta$ and
the functions $W_{\zeta}^{(i)}$ are independent of $s$. Note that here (as opposed to the above),
the double cosets $Q_{\bigidx}k_1N',\ldots,Q_{\bigidx}k_mN'$ need not be disjoint.

Since we can assume
that for each $i$ there is some $\zeta$ such that $W_{\zeta}^{(i)}\nequiv0$, we get that for each $i$ there is $b_i\in\GL{\bigidx}$ such that
$W_{\zeta}^{(i)}(b_i)\in\C[q^{-\zeta},q^{\zeta}]$ is not the zero polynomial. Hence except for a finite set $B$ of values of $q^{-\zeta}$,
$W_{\zeta}^{(i)}\nequiv0$ for all $\zeta$ and $1\leq i\leq m$. Fix $\zeta$ such that $q^{-\zeta}\notin B$. We can rewrite
\begin{align*}
(ch_{kN,W_{\zeta},s})^{O,\psi_{\gamma}}=\sum_{i=1}^{m'}P_i'ch_{k_i'N',(W_{\zeta}')^{(i)},s}
\end{align*}
so that if $k_{i_1}'=\ldots=k_{i_c}'$, $(W_{\zeta}')^{(i_1)},\ldots,(W_{\zeta}')^{(i_c)}$ are linearly
independent (see after the proof of Claim~\ref{claim:translation of standard section is nearly standard}).
For each $k_i$, $1\leq i\leq m$, there is $k_i'$ with $k_i=k_i'$ (also each $k_i'$ is equal
to some $k_i$).
Then since $(ch_{kN,W_{\zeta},s})^{O,\psi_{\gamma}}\in F_{w_0}$, for each $i$ we get
$ch_{k_i'N',(W_{\zeta}')^{(i)},s}\in F_{w_0}$ (see the proofs of
Corollary~\ref{corollary:support of standard
section made small} and Proposition~\ref{proposition:dr integral for torus of G_l}). Because $(W_{\zeta}')^{(i)}\nequiv0$,
$Q_{\bigidx}k_i'N\subset Q_{\bigidx}w_{0}Q_{\bigidx-\smallidx}$ for all $1\leq i\leq m'$. Thus
$Q_{\bigidx}k_iN\subset Q_{\bigidx}w_{0}Q_{\bigidx-\smallidx}$ for all $1\leq i\leq m$ whence
$ch_{k_iN',W_{\zeta}^{(i)},s}\in F_{w_0}$ for all $\zeta\in\C$.

Now we can proceed as in Proposition~\ref{proposition:dr integral for torus of G_l}. Let
$B_i\subset Q_{\bigidx-\smallidx}$ be such that for all $s$, and all $\zeta$ with $q^{-\zeta}\notin B$, the
support of $\lambda(w_{\smallidx,\bigidx}^{-1})ch_{k_iN',W_{\zeta}^{(i)},s}$ equals
$Q_{\bigidx}^{w_{\smallidx,\bigidx}}B_i$. Then for all $\zeta$ and $s$,
$\support{\lambda(w_{\smallidx,\bigidx}^{-1})ch_{k_iN',W_{\zeta}^{(i)},s}}\subset Q_{\bigidx}^{w_{\smallidx,\bigidx}}B_i$. By virtue of
Claim~\ref{claim:omega f reduces to sum when support is compact} we conclude that there exist $r_{i,j}\in R_{\smallidx,\bigidx}$, $m_i\geq1$ and
$c_{i,j}\in\C$ such that for all $\zeta$ and $s$,
\begin{align*}
\Omega(ch_{k_iN',W_{\zeta}^{(i)},s})
=\sum_{j=1}^{m_i}c_{i,j}ch_{k_iN',W_{\zeta}^{(i)},s}(w_{\smallidx,\bigidx}r_{i,j},1).
\end{align*}
Next note that
\begin{align*}
c_{i,j}ch_{k_iN',W_{\zeta}^{(i)},s}(w_{\smallidx,\bigidx}r_{i,j},diag(a,I_{\bigidx-\smallidx+1}))=
P_{i,j}\cdot(b_{i,j}\cdot W_{\zeta}^{(i)})(diag(a,I_{\bigidx-\smallidx+1})),
\end{align*}
for some polynomials $P_{i,j}\in\C[q^{-s},q^s]$ and $b_{i,j}\in\GL{\bigidx}$. The function $W_{\zeta}^{(i,j)}=b_{i,j}\cdot W_{\zeta}^{(i)}$ belongs to
$\Whittaker{\varepsilon}{\psi}$ and satisfies the properties stated in the claim (because $W_{\zeta}^{(i)}$ does). Therefore we get for all $\zeta$ and $s$,
\begin{align*}
\Omega(a\cdot ch_{kN,W_{\zeta},s})
=\absdet{a}^{\smallidx-\half\bigidx+s-\half}\sum_{i=1}^mP_i\sum_{j=1}^{m_i}P_{i,j}W_{\zeta}^{(i,j)}(diag(a,I_{\bigidx-\smallidx+1})).
\end{align*}
The only dependence on $s$ in this expression is in the factors $q^{-s}$ appearing in $P_i,P_{i,j}$
and in the exponent of $\absdet{a}$. The only dependence on $\zeta$ is in the
Whittaker functions.

From here we continue as in Proposition~\ref{proposition:iwasawa decomposition of the integral}. Regarding the split case,
when the integration on $G_1$ is over a non-compact subset and we evaluate $\Omega(ax\cdot ch_{kN,W_{\zeta},s})$
with $x\in G_1$, we need to multiply $ch_{kN,W_{\zeta},s}$
by an element $h$ obtained via Lemma~\ref{lemma:torus elements in double cosets}. Then
$h\cdot ch_{kN,W_{\zeta},s}$ can be written as a finite sum as in \eqref{eq:f s zeta first writing iwasawa for zeta 2}.

In case $\smallidx>\bigidx$ note that 
the argument regarding the $dr$-integration depends only on $W\in\Whittaker{\pi}{\psi_{\gamma}^{-1}}$ (and in particular, not on $\zeta$).

For the justifications of the formal steps we can proceed as in the proof of
Proposition~\ref{proposition:iwasawa decomposition of the integral} or simply take $s_C$ large enough
(see Remark~\ref{remark:taking suitable half-plane for inner integrals}).
\end{proof} 

\section{Realization of $\pi$ induced from
$\sigma\otimes\pi'$}\label{subsection:Realization of pi for induced
representation} When $\pi$ is an induced representation, as in
Section~\ref{subsection:Realization of tau for induced
representation} we use an explicit integral for the Whittaker
functional on $\pi$. Assume
$\pi=\cinduced{\overline{P_k}}{G_{\smallidx}}{\sigma\otimes\pi'}$,
$0<k<\smallidx$, $\sigma$ is realized in
$\Whittaker{\sigma}{\psi^{-1}}$ and $\pi'$ is realized in
$\Whittaker{\pi'}{\psi_{\gamma}^{-1}}$. Then a Whittaker functional
is defined by
\begin{align}\label{eq:whittaker for pi, k<l}
\whittakerfunctional(\varphi)=\int_{V_k}\varphi(v,I_k,I_{2(\smallidx-k)})\psi_{\gamma}(v)dv.
\end{align}
Here $\varphi$ belongs to the space of $\pi$, $\psi_{\gamma}=\frestrict{\psi_{\gamma}}{V_k}$ where
$\psi_{\gamma}$ is the character of $U_{G_{\smallidx}}$ (e.g. if
$k<\smallidx-1$, $\psi_{\gamma}(v)=\psi(v_{k,k+1})$). This integral
has a sense as a principal value. The corresponding Whittaker
function $W_{\varphi}\in\Whittaker{\pi}{\psi_{\gamma}^{-1}}$ is $W_{\varphi}(g)=\whittakerfunctional(g\cdot \varphi)$.

\begin{remark}
Alternatively we can take $\pi$ induced from $P_k$ but then
$\whittakerfunctional(\varphi)$ needs to be adapted accordingly, for
instance we need to translate $\varphi$ on the left by a Weyl element and consider separately the
cases of odd and even $k$.
\end{remark}

If $k=\smallidx$ (which is only possible when $G_{\smallidx}$ is
split), we need to consider $\pi$ induced from either
$\overline{P_{\smallidx}}$ or
$\rconj{\kappa}(\overline{P_{\smallidx}})$, where
$\kappa=diag(I_{\smallidx-1},J_2,I_{\smallidx-1})$ (see
Section~\ref{subsection:groups in study}). Let $\sigma$ be a representation of $\GL{\smallidx}$ realized in $\Whittaker{\sigma}{\psi^{-1}}$.
In the former case assume $\pi=\cinduced{\overline{P_{\smallidx}}}{G_{\smallidx}}{\sigma}$, 
\begin{align}\label{eq:whittaker for pi, k=l}
\whittakerfunctional(\varphi)=\int_{V_{\smallidx}}\varphi(v,d_{\gamma})\psi_{\gamma}(v)dv,\qquad
d_{\gamma}=diag(I_{\smallidx-1},4)\in\GL{\smallidx},
\end{align}
where $\psi_{\gamma}(v)=\psi(-\gamma v_{\smallidx-1,\smallidx+1})$.
In the latter case
$\pi=\cinduced{\rconj{\kappa}(\overline{P_{\smallidx}})}{G_{\smallidx}}{\sigma}$,
define
\begin{align}\label{eq:whittaker for pi, k=l 2}
\whittakerfunctional(\varphi)=\int_{\rconj{\kappa}V_{\smallidx}}\varphi(
v,d_{\gamma})\psi_{\gamma}(v)dv,\qquad
d_{\gamma}=diag(I_{\smallidx-1},-\gamma^{-1})\in\GL{\smallidx}.
\end{align}
Here $\psi_{\gamma}(v)=\psi(\quarter v_{\smallidx-1,\smallidx})$ (note that $\rconj{\kappa}\overline{V_{\smallidx}}=\overline{\rconj{\kappa}V_{\smallidx}}$).

Similarly to Section~\ref{subsection:Realization of tau for induced
representation}, introduce an auxiliary parameter $\zeta\in\C$ in
order to ensure the absolute convergence of the integral defining $\whittakerfunctional(\varphi)$.
Again $\zeta$ can be treated as fixed, variable or fixed in some
compact subset. The situation here is simpler, since we do not need
to consider sections in two variables simultaneously. We briefly describe
the results that will be used. 

Let $P<G_{\smallidx}$ be either $\overline{P_k}$ or $\rconj{\kappa}(\overline{P_{\smallidx}})$ according to $\pi$ ($0<k\leq\smallidx$).
Consider the representation $\cinduced{P}{G_{\smallidx}}{(\sigma\otimes\pi')\alpha^{-\zeta+\half}}$ on the space $V_{P}^{G_{\smallidx}}(\sigma\otimes\pi',-\zeta+\half)$,
where $\alpha^{-\zeta+\half}$ is the character of the Levi part of $P$ defined by
$(b,g')\mapsto\alpha^{-\zeta+\half}(b)$ for $(b,g')\in\GL{k}\times G_{\smallidx-k}$ (the Levi part of $P$ is isomorphic to $\GL{k}\times G_{\smallidx-k}$).
The space of standard sections of this representation is
$\xi(\sigma\otimes\pi',std,-\zeta+\half)=\xi_{P}^{G_{\smallidx}}(\sigma\otimes\pi',std,-\zeta+\half)$. Let $\varphi_{\zeta}\in\xi(\sigma\otimes\pi',std,-\zeta+\half)$. The integral
$\Psi(\varphi_{\zeta},f_s,s)$ is just $\Psi(W_{\varphi_{\zeta}},f_s,s)$ with $W_{\varphi_{\zeta}}$ replaced by one of the explicit formulas described above. For instance if $k<\smallidx\leq\bigidx$,
\begin{align}\label{int:k<l<n example starting form}
&\Psi(\varphi_{\zeta},f_s,s)=\int_{\lmodulo{U_{G_{\smallidx}}}{G_{\smallidx}}}
(\int_{V_k}\varphi_{\zeta}(vg,1,1)\psi_{\gamma}(v)dv)\int_{R_{\smallidx,\bigidx}}
f_s(w_{\smallidx,\bigidx}rg,1)\psi_{\gamma}(r)drdg.
\end{align}
This is a triple integral. We say that it is absolutely convergent if it converges when we
replace $\varphi_{\zeta},f_s$ with $|\varphi_{\zeta}|,|f_s|$ and drop the characters.
There is a domain $D$ of the form
\begin{align*}
\Re(\zeta)>C_0,\qquad \Re(\zeta)+C_2<\Re(s)<(1+C_1)\Re(\zeta)+C_3,
\end{align*}
such that $\Psi(\varphi_{\zeta},f_s,s)$ is absolutely convergent. Here
$C_0,\ldots, C_3$ are constants and $C_1>0$. The constants depend only
on $\sigma$, $\pi'$ and $\tau$ (see \cite{Soudry} Proposition~11.15).

When we write $\Psi(W_{\varphi_{\zeta}},f_s,s)$ we refer to \eqref{int:k<l<n example starting form} as a $drdg$-integral, i.e., the $dv$-integration is interpreted as principal value. Then for any compact subset
$C\subset\C$ there is a constant $s_C$ depending only on $C$, $\sigma$,
$\pi'$ and $\tau$ such that for all $\zeta\in C$ and $\Re(s)>s_C$,
$\Psi(W_{\varphi_{\zeta}},f_s,s)$ is absolutely convergent. E.g.,
\begin{align*}
&\int_{\lmodulo{U_{G_{\smallidx}}}{G_{\smallidx}}}\int_{R_{\smallidx,\bigidx}}
|\int_{V_k}\varphi_{\zeta}(vg,1,1)\psi_{\gamma}(v)dv|
|f_s|(w_{\smallidx,\bigidx}rg,1)drdg<\infty.
\end{align*}
Similarly to Claim~\ref{claim:deducing varphi and f_s are interchangeable in psi}, if for a fixed $\zeta$,
$\Psi(\varphi_{\zeta},f_s,s)$ is absolutely convergent at $s$,
$\Psi(\varphi_{\zeta},f_s,s)=\Psi(W_{\varphi_{\zeta}},f_s,s)$ at $s$. However, the domain of absolute convergence
of $\Psi(\varphi_{\zeta},f_s,s)$ does not contain a right half-plane in the parameter $s$. Therefore
an analogue of Corollary~\ref{corollary:deducing varphi and f_s are interchangeable in psi for all varphi}
does not immediately follow. The next claim establishes a similar assertion.
\begin{claim}\label{claim:meromorphic continuation for varphi instead of W in psi}
Fix $\zeta$ such that $\Re(\zeta)>C_0$ and put
\begin{align*}
D'=\setof{s\in\C}{\Re(\zeta)+C_2<\Re(s)<(1+C_1)\Re(\zeta)+C_3}.
\end{align*}
There is a finite set $B$ of values of $q^{-s}$ such that the following holds.
For any $\varphi_{\zeta}\in V_{P}^{G_{\smallidx}}(\sigma\otimes\pi',-\zeta+\half)$
and $f_s\in \xi(\tau,hol,s)$ there is
$Q_1\in\C(q^{-s})$ such that $\Psi(\varphi_{\zeta},f_s,s)=Q_1(q^{-s})$ for each $s\in D'$ such that
$q^{-s}\notin B$. Furthermore, let $Q_2\in\C(q^{-s})$ be such that
$\Psi(W_{\varphi_{\zeta}},f_s,s)=Q_2(q^{-s})$ for $\Re(s)>>0$.
Then $Q_1=Q_2$.
\end{claim}
\begin{remark}
The equality $\Psi(\varphi_{\zeta},f_s,s)=\Psi(W_{\varphi_{\zeta}},f_s,s)$ between integrals holds in $D'$ while
the ``default" domain of definition of $\Psi(W_{\varphi_{\zeta}},f_s,s)$ is a right half-plane.
\end{remark}
\begin{proof}[Proof of Claim~\ref{claim:meromorphic continuation for varphi instead of W in psi}] 
The integral $\Psi(W_{\varphi_{\zeta}},f_s,s)$ is absolutely convergent in both $D'$ and some right half-plane depending only on
$\sigma,\pi',\tau$ and $\zeta$. Let $D''$ be the union of these domains. Then by
Claim~\ref{claim:meromorphic continuation 2} (see also Remark~\ref{remark:proving equality of W and f in any domain without Bernstein continuation principle}), $\Psi(W_{\varphi_{\zeta}},f_s,s)=Q_1(q^{-s})$ for some $Q_1\in\C(q^{-s})$
for each $s\in D''$ such that $q^{-s}$ does not belong to some finite set $B$.
In particular for $s\in D'$ with $q^{-s}\notin B$,
\begin{align*}
\Psi(\varphi_{\zeta},f_s,s)=\Psi(W_{\varphi_{\zeta}},f_s,s)=Q_1(q^{-s}).
\end{align*}
Hence $\Psi(\varphi_{\zeta},f_s,s)$ extends to an element of $\C(q^{-s})$. Now taking
$\Re(s)>>0$ we get $Q_1(q^{-s})=\Psi(W_{\varphi_{\zeta}},f_s,s)=Q_2(q^{-s})$.
\end{proof} 

For any $W\in\Whittaker{\pi}{\psi_{\gamma}^{-1}}$ we can select
$\varphi_{\zeta}\in\xi(\sigma\otimes\pi',std,-\zeta+\half)$
such that
$W_{\varphi_0}=W$. This is because we can take $\varphi$ in the space of $\pi$ satisfying $W=W_{\varphi}$ and extend it to $\varphi_{\zeta}\in\xi(\sigma\otimes\pi',std,-\zeta+\half)$.

In addition, the integral $\Psi(\varphi_{\zeta},f_s,s)$ for
$\varphi_{\zeta}\in\xi(\sigma\otimes\pi',rat,-\zeta+\half)$
and $f_s\in\xi(\tau,rat,s)$ extends to a rational function
in $\C(q^{-\zeta},q^{-s})$.

%% file: chapter_local_factors.tex
\newtheorem{theorem}{Theorem}[section]
\newtheorem{proposition}{Proposition}[section]
\newtheorem{corollary}{Corollary}[section]
\newtheorem{lemma}{Lemma}[section]
\newtheorem{claim}{Claim}[section]
\theoremstyle{remark}
\newtheorem{remark}{Remark}[section]
\newtheorem{example}{Example}[section]
\theoremstyle{definition}
\newtheorem{definition}{Definition}[section]
\numberwithin{equation}{section}
\newcommand{\chapter}{\section} 
\input{thesis_notations}
\end{comment}

\chapter{Local factors}\label{chapter:local factors}
We define the local factors: the $\gamma$-factor, the g.c.d. and the
$\epsilon$-factor, and discuss the definitions. The study of these factors covers most of the remaining chapters.

\section{Definitions of local factors}\label{subsection:definitions of local factors}
\subsection{The $\gamma$-factor}\label{subsection:the gamma factor}
Let $\pi$ be a representation of $G_{\smallidx}$ and let $\tau$ be a
representation of $\GL{\bigidx}$. 
The observation underlying the definition is that $\Psi(W,f_s,s)$
and $\Psi(W,\nintertwiningfull{\tau}{s}f_s,1-s)$ satisfy certain
equivariance properties, which place them in a single
one-dimensional space of bilinear forms.

As explained in Section~\ref{subsection:meromorphic continuation},
for $\Re(s)>>0$ the integral $\Psi(W,f_s,s)$ can be regarded as a
bilinear form on $\Whittaker{\pi}{\psi_{\gamma}^{-1}}\times
V(\tau,s)$ satisfying \eqref{eq:bilinear special condition}. We show how to extend this form to the whole plane
except a finite set of values of $q^{-s}$.

Let $W\in\Whittaker{\pi}{\psi_{\gamma}^{-1}}$, $s_0\in\C$ and
$y_{s_0}\in V(\tau,s_0)$ be given. Take any $f_s\in\xi(\tau,hol,s)$
satisfying $f_{s_0}=y_{s_0}$. According to
Proposition~\ref{proposition:meromorphic continuation} there is
$Q\in\C(q^{-s})$ such that $\Psi(W,f_s,s)=Q(q^{-s})$ for
$\Re(s)>>0$. Assuming that $q^{-s_0}$ is not a pole of $Q(q^{-s})$,
we define $\Psi(W,y_{s_0},s_0)=Q(q^{-s_0})$. Recall that the set of
poles, in $q^{-s}$, of $Q$ belong to a finite set depending only on
the representations.

To show that this is well-defined, assume $g_s\in\xi(\tau,hol,s)$
also satisfies $g_{s_0}=y_{s_0}$. Then $\Psi(W,g_s,s)=R(q^{-s})$ for
$R\in\C(q^{-s})$ and we claim $Q(q^{-s_0})=R(q^{-s_0})$. Indeed, let
$h_s=f_s-g_s$. Then $h_{s_0}=0$. Write $h_s$ as in
\eqref{eq:holomorphic section as combination of disjoint standard
sections}, i.e. $h_s=\sum_{i=1}^mP_i\cdot h_s^{(i)}$ with $0\ne
P_i\in\C[q^{-s},q^s]$, $h_s^{(i)}=ch_{k_iN,v_i,s}\in\xi(\tau,std,s)$
and such that if $k_{i_1}=\ldots=k_{i_c}$, $v_{i_1},\ldots,v_{i_c}$
are linearly independent. As in the proof of
Corollary~\ref{corollary:support of standard section made small}, we
see that $q^{-s_0}$ is a common zero of all polynomials
$P_1,\ldots,P_m$. In detail, let $\{i_1,\ldots,i_c\}$ be a maximal
set of indices such that $k_{i_1}=\ldots=k_{i_c}=k_i$ and put
$P_{i_j}(q^{-s_0},q^{s_0})=\alpha_j$. Then
$0=h_{s_0}(k_i)=\alpha_1v_{i_1}+\ldots+\alpha_cv_{i_c}$. Since
$v_{i_1},\ldots,v_{i_c}$ are linearly independent, $\alpha_j=0$ for
each $j$. It follows that $P_i=(1-q^{-s+s_0})A_i$ with
$A_i\in\C[q^{-s},q^s]$ for each $i$, whence $h_s=(1-q^{-s+s_0})h_s'$
with $h_s'\in\xi(\tau,hol,s)$. Now
\begin{align*}
Q(q^{-s})-R(q^{-s})=\Psi(W,h_s,s)=(1-q^{-s+s_0})\Psi(W,h_s',s)
\end{align*}
and since the poles
of $\Psi(W,h_s',s)$ (in $q^{-s}$) belong to a finite set depending only on the representations, we can assume
that $q^{-s_0}$ is not a pole of $\Psi(W,h_s',s)$. Thus $Q(q^{-s_0})=R(q^{-s_0})$.

In this manner we may regard $\Psi(W,f_s,s)$ as a bilinear form on
$\Whittaker{\pi}{\psi_{\gamma}^{-1}}\times V(\tau,s)$, for all
$s\in\C$ outside finitely many values of $q^{-s}$, and it still
satisfies \eqref{eq:bilinear special condition}. The same applies to
$\Psi(W,\nintertwiningfull{\tau}{s}f_s,1-s)$,
defined first for $\Re(s)<<0$. 

Thus for all but a finite
number of values of $q^{-s}$,
\begin{align*}
\Psi(W,\nintertwiningfull{\tau}{s}f_s,1-s),\Psi(W,f_s,s)\in\begin{dcases}Bil_{G_{\smallidx}}(\Whittaker{\pi}{\psi_{\gamma}^{-1}},V(\tau,s)_{N_{\bigidx-\smallidx},\psi_{\gamma}^{-1}})&\smallidx\leq\bigidx,\\
Bil_{H_{\bigidx}}((\Whittaker{\pi}{\psi_{\gamma}^{-1}})_{N^{\smallidx-\bigidx},\psi_{\gamma}^{-1}},V(\tau,s))&\smallidx>\bigidx.
\end{dcases}
\end{align*}
By Theorem~\ref{theorem:uniqueness} these spaces,
outside of a finite set of $q^{-s}$, are at most one-dimensional (in fact one dimensional, because the meromorphic continuation
of $\Psi(W,f_s,s)$ is \nontrivial, by Proposition~\ref{proposition:integral can be made constant}).

Therefore $\Psi(W,f_s,s)$ and
$\Psi(W,\nintertwiningfull{\tau}{s}f_s,1-s)$, as functions in
$\C(q^{-s})$, are proportional and can be related by a functional
equation. Namely, there exists a proportionality factor
$\gamma(\pi\times\tau,\psi,s)$ such that for all
$W\in\Whittaker{\pi}{\psi_{\gamma}^{-1}}$ and
$f_s\in\xi(\tau,hol,s)=\xi_{Q_{\bigidx}}^{H_{\bigidx}}(\tau,hol,s)$,
\begin{align}\label{eq:gamma def}
\gamma(\pi\times\tau,\psi,s)\Psi(W,f_s,s)=c(\smallidx,\tau,\gamma,s)\Psi(W,\nintertwiningfull{\tau}{s}f_s,1-s).
\end{align}
Here
$c(\smallidx,\tau,\gamma,s)=\omega_{\tau}(\gamma)^{-2}|\gamma|^{-2\bigidx(s-\half)}$
if $\bigidx<\smallidx$ and $c(\smallidx,\tau,\gamma,s)=1$ otherwise.

Note that the $\gamma$-factor $\gamma(\pi\times\tau,\psi,s)$ is
well-defined by Proposition~\ref{proposition:integral can be made
constant}. We immediately have
$\gamma(\pi\times\tau,\psi,s)\in\C(q^{-s})$. Also define
$\gamma(\pi\times\tau,\psi,s)=1$ if $\pi$ is the (trivial)
representation of $G_0=\{1\}$. Additionally, once \eqref{eq:gamma
def} was established, it also holds for any $f_s\in\xi(\tau,rat,s)$.

For an irreducible representation $\tau$, the $\gamma$-factor does
not vanish identically. This follows from \eqref{eq:multiplication
of normalized intertwiners} and
Proposition~\ref{proposition:integral can be made constant}.
Specifically, take $W$ and $f_{1-s}'\in\xi(\tau^*,hol,1-s)$ such
that $\Psi(W,f_{1-s}',1-s)\equiv1$ and apply \eqref{eq:gamma def} to
$W$ and
$f_s=\nintertwiningfull{\tau^*}{1-s}f_{1-s}'\in\xi(\tau,rat,s)$.

The factor $c(\smallidx,\tau,\gamma,s)$ is included in order to get
a compact form for the multiplicativity properties of
$\gamma(\pi\times\tau,\psi,s)$ (see Section~\ref{subsection:Twisting
the embedding}). Similarly to the $\epsilon$-factor of \cite{JPSS},
$c(\smallidx,\tau,\gamma,s)$ is related to the character
$\psi_{\gamma}$ in the integrals.

The central characters in Theorem~\ref{theorem:multiplicity first
var} are introduced by the applications of the $\GL{k}\times
\GL{\bigidx}$-functional equations and \eqref{eq:Shahidi local coefficient def}. In order
to get ``clean" multiplicativity formulas, similar to those
satisfied by Shahidi's $\gamma$-factors, we define a normalized
$\gamma$-factor $\Gamma(\pi\times\tau,\psi,s)$. This factor is
defined only for irreducible representations.

Let $\pi$ and $\tau$ be irreducible representations of
$G_{\smallidx}$ and $\GL{\bigidx}$ (resp.). We define
\begin{align*}
\Gamma(\pi\times\tau,\psi,s)=(\omega_{\pi}(-1)|\rho|^{s-\half}(-1,\rho)\gamma_{\psi}(\rho))^{\bigidx}\omega_{\tau}(-1)^{\smallidx}\omega_{\tau}(2\gamma)
\gamma(\pi\times\tau,\psi,s).
\end{align*}
Here $\omega_{\pi}$ is the central character of $\pi$ (the center of $G_{\smallidx}$ is finite), $(,)$ is the (quadratic) Hilbert symbol of the field $F$ and $\gamma_{\psi}$ is the normalized Weil factor associated to $\psi$, $\gamma_{\psi}(\cdot)^4=1$ (\cite{Weil} Section~14, $\gamma_{\psi}(a)$ is $\gamma_F(a,\psi)$ in the notation of \cite{Rao}).
If $\smallidx=0$, define $\Gamma(\pi\times\tau,\psi,s)=1$.

Note that in the split case we assumed $|\beta|=1$ (see Section~\ref{subsection:groups in study}), hence $|\rho|=|\beta^2|=1$. Also $(-1,x^2)=1$ and $\gamma_{\psi}(x^2)=1$ for any $x\in F^*$. Hence if $G_{\smallidx}$ is split,
\begin{align*}
\Gamma(\pi\times\tau,\psi,s)=\omega_{\pi}(-1)^{\bigidx}\omega_{\tau}(-1)^{\smallidx}\omega_{\tau}(2\gamma)
\gamma(\pi\times\tau,\psi,s).
\end{align*}

In order to apply the global arguments and deduce the equality between $\Gamma(\pi\times\tau,\psi,s)$ and
the $\gamma$-factor of Shahidi in the \quasisplit\ case (see Section~\ref{section:main results}), we must consider the following two issues. First, it is necessary to prove Theorem~\ref{theorem:main multiplicative properties} also when the split $G_{\smallidx}$ is defined with respect to $J=diag(I_{\smallidx-1},\bigl(\begin{smallmatrix}0&1\\-\rho&0\end{smallmatrix}\bigr),I_{\smallidx-1})\cdot
J_{2\smallidx}$, when $\rho\in F^2$ and $|\rho|\ne0$ (in fact, $|\rho|\geq1$ is enough). This is because, as already 
observed in Section~\ref{section:construction of the global integral}, this group appears as a local component of the global \quasisplit\ group $G_{\smallidx}(\Adele)$. Therefore, only in Section~\ref{subsection:the gamma factor}, the split $G_{\smallidx}$ can be defined with respect to either $J_{2\smallidx}$ or $J$.

Second, it is necessary to show that when $\smallidx=\bigidx=1$ and $G_{\smallidx}$ is \quasisplit, $\Gamma(\pi\times\tau,\psi,s)$ is equal to $\gamma^{Artin}(\pi\times\tau,\psi,s)$ - the Artin $\gamma$-factor under the local Langlands correspondence. As proved in \cite{me5} (Section~6.1), in this case
\begin{align*}
\gamma(\pi\times\tau,\psi,s)=\omega_{\tau}(\rho)^{-1}\omega_{\tau}(-1)\gamma_{\psi}(\rho)^{-1}(-1,\rho)|\rho|^{\half-s}\omega_{\pi}(-1)
\gamma^{Artin}(\pi\times\tau,\psi,s).
\end{align*}
This explains the appearance of the Hilbert symbol and Weil factor in the definition of $\Gamma(\pi\times\tau,\psi,s)$. Indeed, it immediately follows that $\Gamma(\pi\times\tau,\psi,s)=\gamma^{Artin}(\pi\times\tau,\psi,s)$. In the split case such a computation is not required.

Theorems~\ref{theorem:multiplicity second var} and \ref{theorem:multiplicity first var}, proved in Chapter~\ref{section:gamma_mult}, readily
imply Theorem~\ref{theorem:main multiplicative properties}:
\begin{proof}[Proof of Theorem~\ref{theorem:main multiplicative properties}]
Applying
Theorem~\ref{theorem:multiplicity second var} inductively shows
$\gamma(\pi\times\tau,\psi,s)=\prod_{i=1}^a\gamma(\pi\times\tau_i,\psi,s)$. 
The multiplicativity in $\tau$ follows 
since
\begin{align*}
\Gamma(\pi\times\tau,\psi,s)&=\prod_{i=1}^a(\omega_{\pi}(-1)|\rho|^{s-\half}(-1,\rho)\gamma_{\psi}(\rho))^{\bigidx_i}\omega_{\tau_i}(-1)^{\smallidx}\omega_{\tau_i}(2\gamma)\gamma(\pi\times\tau_i,\psi,s)\\
&=\prod_{i=1}^a\Gamma(\pi\times\tau_i,\psi,s).
\end{align*}
Consider the multiplicativity in $\pi$. If $k=\smallidx$, $\Gamma(\pi'\times\tau,\psi,s)=1$, otherwise
\begin{align*}
\Gamma(\pi'\times\tau,\psi,s)=
(\omega_{\pi'}(-1)|\rho|^{s-\half}(-1,\rho)\gamma_{\psi}(\rho))^{\bigidx}\omega_{\tau}(-1)^{\smallidx-k}\omega_{\tau}(2\gamma)\gamma(\pi'\times\tau,\psi,s).
\end{align*}
According to Theorem~\ref{theorem:multiplicity first var},
\begin{align}\label{eq:multiplicativity helper results first var even orthogonal}
\gamma(\pi\times\tau,\psi,s)=&\omega_{\sigma}(-1)^{\bigidx}\omega_{\tau}(-1)^k[|\rho|^{-\bigidx(s-\half)}\omega_{\tau}(2\gamma)^{-1}]\\\nonumber&\times\gamma(\pi'\times\tau,\psi,s)\prod_{i=1}^m\gamma(\sigma_i\times\tau,\psi,s)\gamma(\sigma_i^*\times\tau,\psi,s).
\end{align}
Here $\sigma$ denotes a representation of $\GL{k}$ parabolically induced from $\sigma_1\otimes\ldots\otimes\sigma_m$ and the factors in square brackets appear only if $k=\smallidx$ (only possible in the split case). The factor $|\rho|^{\cdots}$ was omitted in the statement of Theorem~\ref{theorem:multiplicity first var} because we assumed $|\beta|=1$ (see Chapter~\ref{section:gamma_mult} after \eqref{int:mult 1st var n<k<=l final form}, the constant $c_{\tau,\beta}$ actually equals $c(\smallidx,\tau,\gamma,s)\omega_{\tau}(2\gamma)|\beta|^{2\bigidx(s-\half)}$). The computations in Chapter~\ref{section:gamma_mult} are performed for $G_{\smallidx}$ defined with respect to $J_{2\smallidx,\rho}$, which is equal to $J_{2\smallidx}$ when $\rho\in F^2$. They remain valid for $G_{\smallidx}$ defined with respect to $J$ and $\rho\in F^2$ ($J$ was given before the proof of the theorem), without any changes except for one case - the proof of \eqref{eq:multiplicativity helper results first var even orthogonal} for $k=\smallidx>\bigidx$ (Section~\ref{subsection:1st var k=l>n+1}), where only minor modifications to the coordinates are needed.
Applying \eqref{eq:multiplicativity helper results first var even orthogonal} and because $\omega_{\pi}(-1)=\omega_{\sigma}(-1)\omega_{\pi'}(-1)$,
\begin{align*}
\Gamma(\pi\times\tau,\psi,s)=&(\omega_{\pi'}(-1)|\rho|^{s-\half}(-1,\rho)\gamma_{\psi}(\rho))^{\bigidx}\omega_{\tau}(-1)^{\smallidx-k}
\omega_{\tau}(2\gamma)[|\rho|^{-\bigidx(s-\half)}\omega_{\tau}(2\gamma)^{-1}]\\&\times\gamma(\pi'\times\tau,\psi,s)\prod_{i=1}^m\gamma(\sigma_i\times\tau,\psi,s)\gamma(\sigma_i^*\times\tau,\psi,s)
\\=&\Gamma(\pi'\times\tau,\psi,s)\prod_{i=1}^m\gamma(\sigma_i\times\tau,\psi,s)\gamma(\sigma_i^*\times\tau,\psi,s).
\end{align*}
Here we used the fact that if $k=\smallidx$, evidently $\rho\in F^2$ whence $(-1,\rho)=\gamma_{\psi}(\rho)=1$.
\end{proof} 

We also have the following claim, proving that
$\Gamma(\pi\times\tau,\psi,s)$ is identical with Shahidi's
$\gamma$-factor in the unramified case.
\begin{claim}\label{claim:gamma factor for unramified representations}
Let $\pi$ and $\tau$ be irreducible unramified representations.
Assume that all data are unramified (see
Section~\ref{subsection:preliminaries unramified}). Then
\begin{align*}
\Gamma(\pi\times\tau,\psi,s)=\gamma(\pi\times\tau,\psi,s)=\frac{L(\pi\times\tau^*,1-s)}{L(\pi\times\tau,s)}.
\end{align*}
\end{claim}
\begin{proof}[Proof of Claim~\ref{claim:gamma factor for unramified representations}] 
When $\pi$ and $\tau$ are irreducible and unramified,
$\Gamma(\pi\times\tau,\psi,s)=\gamma(\pi\times\tau,\psi,s)$. Let
$W\in\Whittaker{\pi}{\psi_{\gamma}^{-1}}$, $f_s\in\xi(\tau,std,s)$
and $f_{1-s}'\in\xi(\tau^*,std,1-s)$ be the normalized unramified
elements. That is, $W$ is right-invariant by $K_{G_{\smallidx}}$,
$W(1)=1$, $f_s$ and $f_{1-s}'$ are right-invariant by
$K_{H_{\bigidx}}$, $f_s(1,1)=f_{1-s}'(1,1)=1$. Then
\begin{align*}
\nintertwiningfull{\tau}{s}f_s&=\frac{L(\tau^*,Sym^2,2-2s)}{L(\tau,Sym^2,2s-1)}\intertwiningfull{\tau}{s}f_s\\\notag&=
\frac{L(\tau^*,Sym^2,2-2s)}{L(\tau,Sym^2,2s-1)}\frac{L(\tau,Sym^2,2s-1)}{L(\tau,Sym^2,2s)}f_{1-s}'.
\end{align*}
The first equality follows from \eqref{eq:gamma up to units} and
Shahidi \cite{Sh3} (Theorem~3.5), for the second equality see e.g.
Gelbart and Piatetski-Shapiro \cite{GPS} (Section~13).  Now the
result follows from Theorem~\ref{theorem:unramified computation}
when we plug $W$ and $f_s$ into \eqref{eq:gamma def}. Note that
$c(\smallidx,\tau,\gamma,s)=1$ since $|\gamma|=1$
($\gamma\in\mathcal{O}^*$ by our assumption) and 
$\frestrict{\omega_{\tau}}{\mathcal{O}^*}\equiv1$. Also $\gamma_{\psi}|_{\mathcal{O}^*}\equiv1$ and 
$(,)|_{\mathcal{O}^*\times\mathcal{O}^*}\equiv1$.
\end{proof} 

\begin{remark}
In the split case Claim~\ref{claim:gamma factor for unramified
representations} follows immediately from Theorem~\ref{theorem:main
multiplicative properties}, \eqref{eq:JPSS relation gamma and
friends} and the computation of the $\GL{k}\times\GL{\bigidx}$
integrals with unramified data in \cite{JS1}.
\end{remark}

\begin{remark}
One can define $\Gamma(\pi\times\tau,\psi,s)$ in a different, more cumbersome way. Let $\pi$ and $\tau$ be irreducible representations of
$G_{\smallidx}$ and $\GL{\bigidx}$ (resp.). There is an irreducible
supercuspidal representation
$\eta_1\otimes\ldots\otimes\eta_b\otimes\pi_0'$ of
$\GL{j_1}\times\ldots\times\GL{j_b}\times G_{\smallidx-k_0}$, where
$k_0=j_1+\ldots+j_b$ ($0\leq k_0\leq\smallidx$), such that $\pi$ is
a quotient of a representation parabolically induced from
$\eta_1\otimes\ldots\otimes\eta_b\otimes\pi_0'$. Put
$\sigma_0=\cinduced{P_{j_1,\ldots,j_b}}{\GL{k_0}}{\eta_1\otimes\ldots\otimes\eta_b}$.
One can define
\begin{align*}
\Gamma'(\pi\times\tau,\psi,s)=\omega_{\sigma_0}(-1)^{\bigidx}\omega_{\tau}(-1)^{k_0}[\omega_{\tau}(2\gamma)]\gamma(\pi\times\tau,\psi,s).
\end{align*}
Here $\omega_{\tau}(2\gamma)$ appears only if $k_0=\smallidx$. In
the trivial case of a representation $\pi$ of $G_0$ put
$\Gamma'(\pi\times\tau,\psi,s)=1$. Note that if $k_0=0$ (i.e.
$\pi=\pi_0'$),
$\Gamma'(\pi\times\tau,\psi,s)=\gamma(\pi\times\tau,\psi,s)$.

We show that $\Gamma'(\pi\times\tau,\psi,s)$ is well-defined. Indeed,
the representation $\eta_1\otimes\ldots\otimes\eta_b\otimes\pi_0'$ exists and is unique
up to conjugation by some Weyl element (\cite{BZ2} Section~2). Let
$\theta_1\otimes\ldots\otimes\theta_c\otimes\theta'$ be an
irreducible supercuspidal representation, where each $\theta_i$ is a
representation of $\GL{m_i}$ and $\theta'$ is a representation of
$G_{\smallidx-j}$ ($0\leq j\leq\smallidx$). Assume that $\pi$ is a
quotient of a representation parabolically induced from
$\theta_1\otimes\ldots\otimes\theta_c\otimes\theta'$. Denote by
$\eta$ a representation parabolically induced from
$\theta_1\otimes\ldots\otimes\theta_c$ such that $\pi$ is a quotient
of a representation parabolically induced from $\eta\otimes\theta'$.
Then $j=k_0$, $c=b$ and up to a change of indices, for each $i$ the
representation $\theta_i$ is isomorphic to $\eta_i$ or
$\widetilde{\eta_{i}}$, whence
$\omega_{\eta}(-1)=\omega_{\sigma_0}(-1)$.

The proof of Theorem~\ref{theorem:main multiplicative properties} for $\Gamma'(\pi\times\tau,\psi,s)$ now proceeds as follows. Assume that $\tau$ is an
irreducible quotient of a representation, parabolically induced from
an irreducible representation $\tau_1\otimes\ldots\otimes\tau_a$, as
in the statement of the theorem. For brevity, we drop $\psi$ and $s$
from the notation of the $\gamma$-factors. Applying
Theorem~\ref{theorem:multiplicity second var} inductively shows
$\gamma(\pi\times\tau)=\prod_{i=1}^a\gamma(\pi\times\tau_i)$ and we
get
\begin{align*}
\Gamma'(\pi\times\tau)=\prod_{i=1}^a\omega_{\sigma_0}(-1)^{\bigidx_i}\omega_{\tau_i}(-1)^{k_0}[\omega_{\tau_i}(2\gamma)]\gamma(\pi\times\tau_i)=\prod_{i=1}^a\Gamma'(\pi\times\tau_i).
\end{align*}

Now assume that $\pi$ is an irreducible quotient of a
representation, parabolically induced from an irreducible
representation $\sigma_1\otimes\ldots\otimes\sigma_m\otimes\pi'$,
with $k<\smallidx$ ($\pi'$ is an irreducible representation of
$G_{\smallidx-k}$). We can assume that $\pi$ is a quotient of
$\cinduced{P_k}{G_{\smallidx}}{\sigma\otimes\pi'}$, where $\sigma$
is a representation of $\GL{k}$ parabolically induced from
$\sigma_1\otimes\ldots\otimes\sigma_m$.

In general denote $S[\varepsilon]=\theta$ if $\varepsilon$ is a
quotient of a representation parabolically induced from an
irreducible supercuspidal representation $\theta$. If
$S[\varepsilon_i]=\theta_i$ for $i=1,2$ and $\varepsilon$ is a
quotient of a representation parabolically induced from
$\varepsilon_1\otimes\varepsilon_2$,
$S[\varepsilon]=\theta_1\otimes\theta_2$ (use \cite{BZ2} 1.9 (a),
(g) and (c)). 

Assume
\begin{align*}
S[\sigma]=\theta_1\otimes\ldots\otimes\theta_c,\quad
S[\pi']=\theta_{c+1}\otimes\ldots\otimes\theta_{c+d}\otimes\theta', 
\end{align*}
where each $\theta_i$ is a representation of $\GL{m_i}$ and
$\theta'$ is a representation of $G_{\smallidx-k-j}$, $0\leq
j\leq\smallidx-k$. Then
\begin{align*}
S[\pi]=\theta_{1}\otimes\ldots\otimes\theta_{c+d}\otimes\theta'.
\end{align*}
Since also $S[\pi]=\eta_1\otimes\ldots\otimes\eta_b\otimes\pi_0'$,
as above we get
$k=k_0-j$ (because $\smallidx-k-j=\smallidx-k_0$), $c+d=b$ and up to
a change of order, each $\theta_i$ is isomorphic to $\eta_i$ or
$\widetilde{\eta_i}$. Hence $k\leq k_0$ and if $\eta$ is a
representation of $\GL{k_0-k}$ parabolically induced from
$\theta_{c+1}\otimes\ldots\otimes\theta_{c+d}$,
$\omega_{\sigma_0}(-1)=\omega_{\sigma}(-1)\omega_{\eta}(-1)$.

According to the definition,
\begin{align*}
&\Gamma'(\pi'\times\tau)=\omega_{\eta}(-1)^{\bigidx}\omega_{\tau}(-1)^{k_0-k}[\omega_{\tau}(2\gamma)]\gamma(\pi'\times\tau),
\end{align*}
where $\omega_{\tau}(2\gamma)$ appears only if $k_0-k=\smallidx-k$,
i.e. if $k_0=\smallidx$. Next, since $\pi$ is a quotient of
$\cinduced{P_k}{G_{\smallidx}}{\sigma\otimes\pi'}$,
Theorem~\ref{theorem:multiplicity first var} yields
\begin{align*}
\gamma(\pi\times\tau)=\omega_{\sigma}(-1)^{\bigidx}\omega_{\tau}(-1)^{k}\gamma(\sigma\times\tau)\gamma(\pi'\times\tau)\gamma(\sigma^*\times\tau).
\end{align*}
Thus
\begin{align*}
\Gamma'(\pi\times\tau)&=\omega_{\sigma_0\otimes\sigma}(-1)^{\bigidx}\omega_{\tau}(-1)^{k_0+k}[\omega_{\tau}(2\gamma)]\gamma(\sigma\times\tau)\gamma(\pi'\times\tau)\gamma(\sigma^*\times\tau)\\
&=\gamma(\sigma\times\tau)\Gamma'(\pi'\times\tau)\gamma(\sigma^*\times\tau).
\end{align*}

If $k=\smallidx$, by looking at $S[\sigma]$ and $S[\pi]$ we get
$k_0=\smallidx$ and $\omega_{\sigma_0}(-1)=\omega_{\sigma}(-1)$,
therefore
\begin{align*}
&\gamma(\pi\times\tau)=\omega_{\sigma}(-1)^{\bigidx}\omega_{\tau}(-1)^{\smallidx}\omega_{\tau}(2\gamma)^{-1}\gamma(\sigma\times\tau)\gamma(\sigma^*\times\tau),\\\notag
&\Gamma'(\pi\times\tau)=\omega_{\sigma}(-1)^{2\bigidx}\omega_{\tau}(-1)^{2\smallidx}\gamma(\sigma\times\tau)\gamma(\sigma^*\times\tau)=\gamma(\sigma\times\tau)\Gamma'(\pi'\times\tau)\gamma(\sigma^*\times\tau).
\end{align*}

Since according to \cite{JPSS} (Theorem~3.1),
$\gamma(\sigma\times\tau)=\prod_{i=1}^m\gamma(\sigma_i\times\tau)$
and
$\gamma(\sigma^*\times\tau)=\prod_{i=1}^m\gamma(\sigma_i^*\times\tau)$,
this completes the proof of Theorem~\ref{theorem:main multiplicative properties} for $\Gamma'(\pi\times\tau,\psi,s)$.

One minor error in the definition of $\Gamma'(\pi\times\tau,\psi,s)$ is that if
$\smallidx=\bigidx=1$ and $G_{\smallidx}$ is \quasisplit, it is not equal to $\gamma^{Artin}(\pi\times\tau,\psi,s)$. Hence in the \quasisplit\ case, the global arguments do not give an equality between
$\Gamma'(\pi\times\tau,\psi,s)$ and the $\gamma$-factor of Shahidi. Also in the split case if one drops the assumption
$|\beta|=1$, the definition of $\Gamma'(\pi\times\tau,\psi,s)$ is imprecise. The definition of
$\Gamma'(\pi\times\tau,\psi,s)$ will not be used here.
\end{remark}

\subsection{The g.c.d. and $\epsilon$-factor}\label{section:gcd and epsilon factor}
In this section we assume that $\tau$ is
irreducible, the g.c.d. is defined only for an irreducible $\tau$
(see Remark~\ref{remark:take irreducible for gcd} below).

As explained in Section~\ref{subsection:func eq gln glm}, Jacquet,
Piatetski-Shapiro and Shalika \cite{JPSS} defined the $L$-factor as
a g.c.d. of the integrals. We adapt their approach to our setting,
using the idea of Piatetski-Shapiro and Rallis \cite{PR2,PS} (see
also \cite{Ik,HKS}). Namely, we define a g.c.d. for integrals with
good sections. Let
\begin{align*}
\xi_{Q_{\bigidx}}^{H_{\bigidx}}(\tau,good,s)=\xi_{Q_{\bigidx}}^{H_{\bigidx}}(\tau,hol,s)\cup
\nintertwiningfull{\tau^*}{1-s}\xi_{Q_{\bigidx}}^{H_{\bigidx}}(\tau^*,hol,1-s)
\end{align*}
be the set of good sections, i.e., either holomorphic sections or
the images of such under the normalized intertwining operator. Set
$\xi(\tau,good,s)=\xi_{Q_{\bigidx}}^{H_{\bigidx}}(\tau,good,s)$.
According to equality~\eqref{eq:multiplication of normalized
intertwiners} (which is valid for an irreducible $\tau$),
\begin{align*}
\nintertwiningfull{\tau}{s}\xi(\tau,good,s)=\xi(\tau^*,good,1-s).
\end{align*}
That is, the intertwining operator $\nintertwiningfull{\tau}{s}$ is
a bijection of good sections.

According to Proposition~\ref{proposition:meromorphic continuation},
the integral $\Psi(W,f_s,s)$ where
$W\in\Whittaker{\pi}{\psi_{\gamma}^{-1}}$ and
$f_s\in\xi(\tau,good,s)$, extends to $Q\in\C(q^{-s})$. The right
half-plane where $\Psi(W,f_s,s)=Q(q^{-s})$ and the finite set of
possible poles of $Q$ depend only on the representations (see the
discussion after the proposition).

The integrals $\Psi(W,f_s,s)$ where
$W\in\Whittaker{\pi}{\psi_{\gamma}^{-1}}$ and
$f_s\in\xi(\tau,good,s)$, regarded by meromorphic continuation as
elements of $\C(q^{-s})$, span a fractional ideal
$\mathcal{I}_{\pi\times\tau}(s)$ of $\C[q^{-s},q^s]$ which contains
the constant $1$ (see Proposition~\ref{proposition:integral can be
made constant}). Hence, it admits a unique generator in the form
$P(q^{-s})^{-1}$, with $P\in\C[X]$ such that $P(0)=1$. This
generator is what
we call the g.c.d. of the integrals $\Psi(W,f_s,s)$. 
Denote $\gcd(\pi\times\tau,s)=P(q^{-s})^{-1}$. Here the character $\psi$ of the field,
used to define the Whittaker models 
and the character of $R_{\smallidx,\bigidx}$, is absent. In Corollary~\ref{corollary:ideal independent of character of the field} below
we show that the g.c.d. is indeed independent of $\psi$.

In case $\pi$ is a representation of $G_0$ we define $\gcd(\pi\times\tau,s)=1$.

By the one-dimensionality results described in
Section~\ref{subsection:the gamma factor}, the quotients
\begin{align*}
\frac{\Psi(W,f_s,s)}{\gcd(\pi\times\tau,s)},\frac{\Psi(W,\nintertwiningfull{\tau}{s}f_s,1-s)}{\gcd(\pi\times\tau^*,1-s)}\qquad(f_s\in\xi(\tau,good,s))
\end{align*}
which by definition belong to $\C[q^{-s},q^s]$ are proportional.
There exists a proportionality factor
$\epsilon(\pi\times\tau,\psi,s)$ satisfying for all
$W\in\Whittaker{\pi}{\psi_{\gamma}^{-1}}$ and
$f_s\in\xi(\tau,good,s)$,
\begin{align}\label{eq:epsilon def}
\frac{\Psi(W,\nintertwiningfull{\tau}{s}f_s,1-s)}{\gcd(\pi\times\tau^*,1-s)}
=c(\smallidx,\tau,\gamma,s)^{-1}\epsilon(\pi\times\tau,\psi,s)\frac{\Psi(W,f_s,s)}{\gcd(\pi\times\tau,s)}.
\end{align}
This is our functional equation for $\pi\times\tau$.

Combining \eqref{eq:gamma def} and \eqref{eq:epsilon def} we derive a relation similar to
\eqref{eq:JPSS relation gamma and friends},
\begin{align}\label{eq:gamma and epsilon}
\gamma(\pi\times\tau,\psi,s)
=\epsilon(\pi\times\tau,\psi,s)\frac{\gcd(\pi\times\tau^*,1-s)}{\gcd(\pi\times\tau,s)}.
\end{align}
Equality~\eqref{eq:gamma and epsilon} resembles the relation between
the $\gamma$, $L$ and $\epsilon$-factors of Shahidi \cite{Sh3},
where the g.c.d. is replaced by the $L$-function. By Shahidi's
definitions, when $\pi$ and $\tau$ are standard modules, the
$\gamma$-factor and $L$-function are multiplicative in their
inducing data, hence so is the $\epsilon$-factor. Here
$\gamma(\pi\times\tau,\psi,s)$ is multiplicative, but in order to
deduce this for $\epsilon(\pi\times\tau,\psi,s)$ we still need to
establish proper multiplicative properties for the g.c.d. Currently
this task seems difficult.

Another fundamental property of the $\epsilon$-factor, is that it is
invertible.
\begin{claim}\label{claim:epsilon is a unit}
$\epsilon(\pi\times\tau,\psi,s)\in\C[q^{-s},q^s]^*$.
\end{claim}
\begin{proof}[Proof of Claim~\ref{claim:epsilon is a unit}] 
A priori $\epsilon(\pi\times\tau,\psi,s)\in\C(q^{-s})$. Take $W_i$
and $f_s^{(i)}\in\xi(\tau,good,s)$, $i=\intrange{1}{k}$, such that
\begin{align*}
\gcd(\pi\times\tau,s)=\sum_{i=1}^k\Psi(W_i,f_s^{(i)},s).
\end{align*}
Plugging $W_i,f_s^{(i)}$ into \eqref{eq:epsilon def} and summing,
we get $\epsilon(\pi\times\tau,\psi,s)\in\C[q^{-s},q^s]$. According
to the definition of $\epsilon(\pi\times\tau^*,\psi,1-s)$,
\begin{align*}
&\frac{\Psi(W,\nintertwiningfull{\tau^*}{1-s}\nintertwiningfull{\tau}{s}f_s,s)}{\gcd(\pi\times\tau,s)}\\\notag
&=c(\smallidx,\tau^*,\gamma,1-s)^{-1}\epsilon(\pi\times\tau^*,\psi,1-s)
\frac{\Psi(W,\nintertwiningfull{\tau}{s}f_s,1-s)}{\gcd(\pi\times\tau^*,1-s)}.
\end{align*}
Using \eqref{eq:multiplication of normalized intertwiners} and
\eqref{eq:epsilon def} on the \lhs, and since $c(\smallidx,\tau,\gamma,s)c(\smallidx,\tau^*,\gamma,1-s)=1$,
\begin{align*}
\frac{\Psi(W,\nintertwiningfull{\tau}{s}f_s,1-s)}{\gcd(\pi\times\tau^*,1-s)}
=\epsilon(\pi\times\tau,\psi,s)\epsilon(\pi\times\tau^*,\psi,1-s)\frac{\Psi(W,\nintertwiningfull{\tau}{s}f_s,1-s)}{\gcd(\pi\times\tau^*,1-s)}.
\end{align*}
Since the \lhs\ is not identically zero (because $\tau$ is irreducible),
$\epsilon(\pi\times\tau,\psi,s)\epsilon(\pi\times\tau^*,\psi,1-s)=1$,
whence $\epsilon(\pi\times\tau,\psi,s)$ is a unit.
\end{proof} 
As we have seen in this section, the g.c.d. definition using good
sections is natural in the sense that it immediately implies three
expected fundamental results: a standard functional equation
\eqref{eq:epsilon def}, a standard identity for the $\gamma$-factor
\eqref{eq:gamma and epsilon} and an invertible $\epsilon$-factor. In
Section~\ref{subsection:motivation underlying the good sections} we
discuss a preliminary (unsuccessful) attempt to define the g.c.d.
using only holomorphic sections and provide the background for the
method of good sections.
\begin{remark}\label{remark:take irreducible for gcd}
The fact that $\tau$ is irreducible is needed for
\eqref{eq:multiplication of normalized intertwiners}. In the absence
of \eqref{eq:multiplication of normalized intertwiners} the \lhs\ of
\eqref{eq:epsilon def} might not be a polynomial. In turn, the
$\epsilon$-factor might not be invertible.
\end{remark}

\section{The integral for a good section}\label{subsection:integral
with respect to sections} When describing the integral
$\Psi(W,f_s,s)$ for a good section $f_s$, it is often convenient to
assume that $f_s$ is holomorphic or even standard.
\begin{proposition}\label{proposition:integral for good section}
For any $f_s\in\xi(\tau,good,s)$ there exist $P_i\in\C[q^{-s},q^s]$,
$f_s^{(i)}\in\xi(\tau,std,s)$ such that
\begin{align*}
\Psi(W,f_s,s)=\ell_{\tau^*}(1-s)\sum_{i=1}^kP_i\Psi(W,f_s^{(i)},s).
\end{align*}
If $f_s\in\xi(\tau,hol,s)$, the factor $\ell_{\tau^*}(1-s)$ may be dropped.
This is an equality of integrals defined for $\Re(s)>>0$, hence also an equality in $\C(q^{-s})$ between meromorphic continuations.
\end{proposition}
\begin{proof}[Proof of Proposition~\ref{proposition:integral for good section}]
Let $Q\in\C(q^{-s})$. Taking $s$ with $\Re(s)$ large enough, depending only on the representations and the poles of $Q$, the integral defined at $s$ is a bilinear form and also $\Psi(W,Qf_s,s)=Q\Psi(W,f_s,s)$. Now the proposition follows from the rationality of the intertwining operator
(see Section~\ref{subsection:the intertwining operator for tau}).
\end{proof}

\section{The g.c.d. for supercuspidal representations}\label{section:gcd for supercuspidal reps}
Assume that $\pi$ and $\tau$ are supercuspidal representations. In
Corollary~\ref{corollary:integral is holomorphic for supercuspidal
data} we proved that $\Psi(W,f_s,s)$ is holomorphic for
$f_s\in\xi(\tau,hol,s)$. This corollary together with
Proposition~\ref{proposition:integral for good section} immediately imply the
following result for the g.c.d.
\begin{corollary}\label{corollary:gcd for supercuspidal data}
Let $\pi$ be a supercuspidal representation and $\tau$ be an
irreducible supercuspidal representation. Assume that if
$\smallidx=\bigidx=1$, $G_1$ is \quasisplit. Then
$\gcd(\pi\times\tau,s)\in \ell_{\tau^*}(1-s)\C[q^{-s},q^s]$.
\end{corollary}
If $\pi$ and $\tau$ are irreducible unitary supercuspidal
representations (in particular, tempered), 
Theorem~\ref{theorem:gcd for tempered reps} will imply 
$\gcd(\pi\times\tau,s)=L(\pi\times\tau,s)$, then
Corollary~\ref{corollary:gcd for supercuspidal data} shows that
$L(\pi\times\tau,s)^{-1}$ divides $L(\tau,Sym^2,2s)^{-1}$, since
$\ell_{\tau^*}(1-s)=L(\tau,Sym^2,2s)$ (see
Section~\ref{subsection:the intertwining operator for tau}).

\section{The motivation underlying good sections}\label{subsection:motivation underlying the good sections}
One of the basic properties required of the integrals is meromorphic
continuation to functions in $\C(q^{-s})$. Therefore the largest
reasonable set of sections to consider would be rational sections,
i.e., elements of $\xi(\tau,rat,s)$. We are looking for poles of the
integrals rather than superfluous poles introduced by sections,
hence we try to be conservative in our usage of rational sections.
Indeed, it will have been best to use just standard sections.
However, the space of standard sections is not closed for
right-translations by $H_{\bigidx}$, so a minimal set of sections to
start with would be holomorphic sections.

The first attempt to define a g.c.d. in our setting is to consider
the fractional ideal spanned by $\Psi(W,f_s,s)$ with
$f_s\in\xi(\tau,hol,s)$. This definition does not fit well in the
functional equation. If only holomorphic sections were used, the
ratio
$\gcd(\pi\times\tau^*,1-s)^{-1}\Psi(W,\nintertwiningfull{\tau}{s}f_s,1-s)$
might not be a polynomial, because $\nintertwiningfull{\tau}{s}f_s$
is not necessarily a holomorphic section. In turn the
$\epsilon$-factor might not be an exponential.

Moreover, there is an intrinsic problem with considering only
holomorphic sections. Recall that one of our goals is to relate the
g.c.d. to the $L$-function (see Chapter~\ref{chapter:introduction}).
Assume that $\pi$ and $\tau$ are irreducible supercuspidal. Then as
we have seen in Corollary~\ref{corollary:integral is holomorphic for
supercuspidal data} (granted that either $\smallidx>1$, $\bigidx>1$
or $G_{\smallidx}$ is \quasisplit),
for $f_s\in\xi(\tau,hol,s)$, $\Psi(W,f_s,s)$ is holomorphic so
the g.c.d. of the integrals is $1$. However, 
the Langlands $L$-function $L(\pi\times\tau,s)$ may have a pole. To
see this consider the local functorial lift of $\pi$ to
$\GL{2\smallidx}$ (see e.g. \cite{Co}). This gives an irreducible
self-dual representation $\Pi$ of $\GL{2\smallidx}$ such that
$L(\pi\times\tau,s)=L(\Pi\times\tau,s)$. If $\Pi$ is also
supercuspidal, taking $\bigidx=2\smallidx$ and $\tau=\Pi$ we get
$L(\pi\times\tau,s)=L(\Pi\times\Pi,s)$, which has a pole at $s=0$.

Piatetski-Shapiro and Rallis \cite{PR2,PS} studied a (global) Rankin-Selberg construction 
for an automorphic cuspidal representation $\pi$ of
$\GLF{2}{\Adele_{\mathbb{K}}}^{\circ}=\setof{g\in\GLF{2}{\Adele_{\mathbb{K}}}}{\det(g)\in
K^*}$, where $\mathbb{K}$ is a semi-simple commutative algebra of
degree $3$ over a number field $K$ and $\Adele_{\mathbb{K}}$ is the
Adele ring of $\mathbb{K}$. Their global integral represented the
Langlands global $L$-function $L(\pi,\sigma,s)$, where $\sigma$ is a
certain representation of the Langlands $L$-group of
$\GLF{2}{\Adele_{\mathbb{K}}}^{\circ}$ in
$\C_2\otimes\C_2\otimes\C_2$. Let $S_{\infty}$ be the set of
infinite primes and let $S\supset S_{\infty}$ be a finite set of
primes such that for $\nu\notin S$ all data are unramified. They
defined the partial $L$-function $L^S(\pi,\sigma,s)=\prod_{\nu\notin
S}L_{\nu}(\pi_{\nu},\sigma_{\nu},s)$. Each factor
$L_{\nu}(\pi_{\nu},\sigma_{\nu},s)$ ($\nu\notin S$) was defined
using the Satake Isomorphism Theorem.
Piatetski-Shapiro and Rallis proved that $L^S(\pi,\sigma,s)$ has a meromorphic continuation with a finite number of poles and found the possible locations of the poles. 
In order to extend their results to $L^{S_{\infty}}(\pi,\sigma,s)$,
a definition of $L_{\nu}(\pi_{\nu},\sigma_{\nu},s)$ at the ramified
primes was needed. To accomplish this they introduced the notion of
good sections mentioned above, and defined
$L_{\nu}(\pi_{\nu},\sigma_{\nu},s)$ as the unique (normalized)
generator of the fractional ideal spanned by the local integrals
with good sections (i.e., the g.c.d.). This led to a functional
equation in the form \eqref{eq:epsilon def}, with an
$\epsilon$-factor. They also concluded a global functional equation
for $L^{S_{\infty}}(\pi,\sigma,s)$.

Ikeda \cite{Ik} extended the work of Piatetski-Shapiro and Rallis
\cite{PS} and analyzed the poles of the global $L$-function
$L(\pi,\sigma,s)$. He also studied the poles of the local
$L$-functions. For the finite unramified case he showed that the
g.c.d. definition coincides with the definition based on the Satake
parameterization. We note that the definition of good sections used
by Ikeda is somewhat different from \cite{PS}. In \cite{Ik2}, Ikeda
studied the g.c.d. for the \archimedean\ case and, for a specific
class of representations, related it to the standard 
$L$-function.

The method of good sections has also been applied by Harris, Kudla
and Sweet \cite{HKS}, in the context of the local theta
correspondence between unitary groups, in order to define
$\epsilon$-factors using the doubling method.

Although the actual definition of good sections varies to some
extent among the studies mentioned above, there are essential
properties of such a set that are common to all. Firstly, it is
stable under the normalized intertwining operator, i.e., the
operator is a bijection of good sections. Secondly, it contains the
holomorphic sections, required in order for the integrals to span a
fractional ideal.

The idea of using good sections instead of just holomorphic sections addresses the
lack of symmetry in the functional equation. Since
$\nintertwiningfull{\tau}{s}\xi(\tau,good,s)=\xi(\tau^*,good,1-s)$,
both sides of the equation are polynomials.
The g.c.d. may contain poles originating from the intertwining operator.
Then $\epsilon(\pi\times\tau,\psi,s)$ is seen to be 
invertible. We are also in a better position to confront the
scenario of supercuspidal representations $\pi$ and $\tau=\Pi$
illustrated above, since now $\Psi(W,f_s,s)$ for a good section
$f_s$ may not be holomorphic, so $\gcd(\pi\times\Pi,s)$ may contain
poles. In fact as explained in Section~\ref{section:gcd for supercuspidal reps} (see Corollary~\ref{corollary:gcd for supercuspidal data}), 
$\gcd(\pi\times\Pi,s)^{-1}$ divides $L(\Pi,Sym^2,2s)^{-1}$. Also
$L(\Pi\times\Pi,s)=L(\Pi,Sym^2,s)L(\Pi,\Lambda^2,s)$ ($\Lambda^2$ -
the antisymmetric square representation) and the pole at $s=0$ of
$L(\Pi\times\Pi,s)$ is expected to appear in $L(\Pi,Sym^2,s)$, so
$\gcd(\pi\times\Pi,s)$ and $L(\pi\times\Pi,s)=L(\Pi\times\Pi,s)$ may
coincide. Of course, Corollary~\ref{corollary:gcd for pi tempered
tau supercuspidal split case} will imply
$\gcd(\pi\times\Pi,s)=L(\pi\times\Pi,s)$ in this setting. 

\section{The g.c.d. is independent of $\psi$}\label{subsection:properties of the ideal}
We discuss the dependence of $\mathcal{I}_{\pi\times\tau}(s)$ on the
character used to define the Whittaker model $\Whittaker{\pi}{\psi_{\gamma}^{-1}}$. 
The character $\psi_{\gamma}$ of $U_{G_{\smallidx}}$ was given in Section~\ref{subsection:the integrals}.
Any character $\psi'$ of $U_{G_{\smallidx}}$ in the
$T_{G_{\smallidx}}$-orbit of $\psi_{\gamma}$ is in the form
$\rconj{t}\psi_{\gamma}$ for some $t\in T_{G_{\smallidx}}$, where
$\rconj{t}\psi_{\gamma}(u)=\psi_{\gamma}(tut^{-1})$. Replacing
$\psi_{\gamma}$ with $\rconj{t}\psi_{\gamma}$ entails replacing the
character in $\Whittaker{\tau}{\psi}$, the embedding of
$G_{\smallidx}$ in $H_{\bigidx}$ for $\smallidx\leq\bigidx$ (or the
embedding $H_{\bigidx}<G_{\smallidx}$ when $\smallidx>\bigidx$) and
the character of $R_{\smallidx,\bigidx}$ (for $\smallidx<\bigidx$).
The following propositions detail the invariancy of
$\mathcal{I}_{\pi\times\tau}(s)$ under such a change.

We start with the case $\smallidx\leq\bigidx$. Write $t=ax\in
T_{G_{\smallidx}}$ for $a\in A_{\smallidx-1}$, $x\in G_1$. Decompose
the image of $x$ in
$H_1$, $x=diag(m_x,1,m_x^{-1})z_x$ where 
$z_x\in U_{H_1}K_{H_1}$.
Also let $b\in A_{\bigidx-\smallidx}$ (if $\smallidx<\bigidx$).
Set $d=diag(a,m_x,b),d'=diag(b^*,a,m_x)\in A_{\bigidx}$.
We have,
\begin{proposition}\label{proposition:ideal independent of characters l<=n}
For $\smallidx\leq\bigidx$, let
$\mathcal{I}_{\pi\times\tau}(s)^{t,z_x,b}$ be the fractional ideal
spanned by $\Psi(W,f_s,s)$ with
$W\in\Whittaker{\pi}{\rconj{t}\psi_{\gamma}^{-1}}$, $\tau$ realized
in $\Whittaker{\tau}{\rconj{d}\psi}$ and
\begin{align*}
[g]_{\mathcal{E}_{H_{\bigidx}}}=diag(I_{\bigidx-\smallidx},z_xM^{-1}[g]_{\mathcal{E}_{G_{\smallidx}}'}Mz_x^{-1},I_{\bigidx-\smallidx}),
\end{align*}
where $M$ is defined in Section~\ref{subsection:G_l in H_n}. If
$\smallidx<\bigidx$, the character of $R_{\smallidx,\bigidx}$ is
given by $\rconj{d'}\psi_{\gamma}$. Then
$\mathcal{I}_{\pi\times\tau}(s)^{t,z_x,b}=\mathcal{I}_{\pi\times\tau}(s)$.
\end{proposition}
\begin{proof}[Proof of Proposition~\ref{proposition:ideal independent of characters l<=n}] 
The proof follows the arguments of Soudry \cite{Soudry}
(Section~10). If $W\in\Whittaker{\pi}{\psi_{\gamma}^{-1}}$, the function $g\mapsto W(tg)$
belongs to $\Whittaker{\pi}{\rconj{t}\psi_{\gamma}^{-1}}$. Similarly for
$\tau$ realized in $\Whittaker{\tau}{\psi}$ and $f_s\in V(\tau,s)$, the mapping $y\mapsto f_s(h,dy)$
lies in $\Whittaker{\tau}{\rconj{d}\psi}$. Then
$\mathcal{I}_{\pi\times\tau}'(s)$ is spanned by integrals of the form
\begin{align*}
\int_{\lmodulo{U_{G_{\smallidx}}}{G_{\smallidx}}}\int_{R_{\smallidx,\bigidx}}
W(tg)f_s(w_{\smallidx,\bigidx}r(\rconj{z_x^{-1}}g),d)\rconj{d'}\psi_{\gamma}(r)drdg,
\end{align*}
with $W\in\Whittaker{\pi}{\psi_{\gamma}^{-1}}$ and
$f_s\in\xi(\tau,good,s)$. Note that as input to $f_s$, $g$ is
written according to our usual embedding (i.e.,
$g=diag(I_{\bigidx-\smallidx},M^{-1}[g]_{\mathcal{E}_{G_{\smallidx}}'}M,I_{\bigidx-\smallidx})$).
We may move $d$ to the first argument of $f_s$ and conjugate it by
$w_{\smallidx,\bigidx}$, it normalizes $R_{\smallidx,\bigidx}$ and
the integral becomes
\begin{align*}
C(d)\int_{\lmodulo{U_{G_{\smallidx}}}{G_{\smallidx}}}\int_{R_{\smallidx,\bigidx}}
W(tg)f_s(w_{\smallidx,\bigidx}r(\rconj{w_{\smallidx,\bigidx}}d)(\rconj{z_x^{-1}}g),1)\psi_{\gamma}(r)drdg.
\end{align*}
Here $C(d)\in\C[q^{-s},q^s]^*$ is a function of $d$. Changing
$g\mapsto t^{-1}gt$ in the integral yields, for some $C'(t)\in\C^*$,
\begin{align*}
C(d)C'(t)\int_{\lmodulo{U_{G_{\smallidx}}}{G_{\smallidx}}}\int_{R_{\smallidx,\bigidx}}
t\cdot W(g)(\rconj{w_{\smallidx,\bigidx}}d)\cdot
f_s(w_{\smallidx,\bigidx}rg,1)\psi_{\gamma}(r)drdg,
\end{align*}
which is just $C(d)C'(t)\Psi(t\cdot
W,\rconj{w_{\smallidx,\bigidx}}d\cdot
f_s,s)\in\mathcal{I}_{\pi\times\tau}(s)$.
\end{proof} 

Now consider the case $\smallidx>\bigidx$. Write
$t=diag(b,a,x,a^*,b^*)\in T_{G_{\smallidx}}$ where $b\in
A_{\bigidx}$, $a\in A_{\smallidx-\bigidx-1}$ and $x\in G_1$.
Then we show,
\begin{proposition}\label{proposition:ideal independent of characters l>n}
For $\smallidx>\bigidx$, let
$\mathcal{I}_{\pi\times\tau}(s)^{t,x,b}$ be the fractional ideal
spanned by $\Psi(W,f_s,s)$ with
$W\in\Whittaker{\pi}{\rconj{t}\psi_{\gamma}^{-1}}$, $\tau$ realized
in $\Whittaker{\tau}{\rconj{b}\psi}$ and
\begin{align*}
[h]_{\mathcal{E}_{G_{\smallidx}}}=diag(I_{\smallidx-\bigidx-1},x^{-1}M^{-1}[h]_{\mathcal{E}_{H_{\bigidx}}'}Mx,I_{\smallidx-\bigidx-1}).
\end{align*}
Here $M$ is the matrix defined in Section~\ref{subsection:H_n in
G_l}. Then
$\mathcal{I}_{\pi\times\tau}(s)^{t,x,b}=\mathcal{I}_{\pi\times\tau}(s)$.
\end{proposition}
\begin{proof}[Proof of Proposition~\ref{proposition:ideal independent of characters l>n}] 
Let $W\in\Whittaker{\pi}{\psi_{\gamma}^{-1}}$,
$f_s\in\xi(\tau,good,s)$. As in the case $\smallidx\leq\bigidx$, for
suitable $C(t),C'(t)\in\C[q^{-s},q^s]^*$,
\begin{align*}
&\int_{\lmodulo{U_{H_{\bigidx}}}{H_{\bigidx}}}
\int_{R^{\smallidx,\bigidx}}W(trw^{\smallidx,\bigidx}(\rconj{x}h))f_s(h,b)drdh\\\notag
&=C(t)\int_{\lmodulo{U_{H_{\bigidx}}}{H_{\bigidx}}}
\int_{R^{\smallidx,\bigidx}}W(rw^{\smallidx,\bigidx}(\rconj{w^{\smallidx,\bigidx}}t)(\rconj{x}h))f_s(bh,1)drdh\\\notag
&=C'(t)\int_{\lmodulo{U_{H_{\bigidx}}}{H_{\bigidx}}}
\int_{R^{\smallidx,\bigidx}}\rconj{w^{\smallidx,\bigidx}}t\cdot W(rw^{\smallidx,\bigidx}h)b\cdot f_s(h,1)drdh.
\end{align*}
Here as input to $W$,
$h=diag(I_{\smallidx-\bigidx-1},M^{-1}[h]_{\mathcal{E}_{H_{\bigidx}}'}M,I_{\smallidx-\bigidx-1})$.
The result follows from this.
\end{proof} 

Recall that we have a fixed \nontrivial\ unitary additive character
$\psi$ of the field, used to define the characters of the Whittaker
models and $R_{\smallidx,\bigidx}$. As a corollary to the
propositions above, we show that the g.c.d. is independent of
$\psi$. Any other such character $\psi'$ takes the form
$\psi'(x)=\psi(cx)$ for some $c\ne0$. Changing $\psi$ effectively
changes $\Whittaker{\pi}{\psi_{\gamma}^{-1}}$,
$\Whittaker{\tau}{\psi}$ and the character $\psi_{\gamma}$ of
$R_{\smallidx,\bigidx}$ (for $\smallidx<\bigidx$).
\begin{corollary}\label{corollary:ideal independent of character of the field}
Let $\mathcal{I}_{\pi\times\tau}'(s)$ be the fractional ideal
spanned by the integrals $\Psi(W,f_s,s)$ with the character $\psi$
in the construction replaced by $\psi'$. Then
$\mathcal{I}_{\pi\times\tau}'(s)=\mathcal{I}_{\pi\times\tau}(s)$.
\end{corollary}
\begin{proof}[Proof of Corollary~\ref{corollary:ideal independent of character of the field}] 
Assume $\psi'(x)=\psi(cx)$. If $\smallidx\leq\bigidx$ we apply
Proposition~\ref{proposition:ideal independent of characters l<=n}
with $a=diag(c^{\smallidx-1},\ldots,c^2,c)$, $t=diag(a,I_2,a^*)$,
$b=diag(c^{-1},c^{-2},\ldots,c^{-(\bigidx-\smallidx)})$,
$d=diag(a,1,b)$ and $d'=diag(b^*,a,1)$.

In the case $\smallidx>\bigidx$, use
Proposition~\ref{proposition:ideal independent of characters l>n}
with $a=diag(c^{\smallidx-\bigidx-1},\ldots,c^2,c)$,\\
$b=diag(c^{\smallidx-1},c^{\smallidx-2},\ldots,c^{\smallidx-\bigidx})$
and $t=diag(b,a,I_2,a^*,b^*)$.
\end{proof} 


%% file: chapter_gamma_mult.tex
\newtheorem{theorem}{Theorem}[section]
\newtheorem{proposition}{Proposition}[section]
\newtheorem{corollary}{Corollary}[section]
\newtheorem{lemma}{Lemma}[section]
\newtheorem{claim}{Claim}[section]
\theoremstyle{remark}
\newtheorem{remark}{Remark}[section]
\newtheorem{example}{Example}[section]
\theoremstyle{definition}
\newtheorem{definition}{Definition}[section]
\numberwithin{equation}{section}
\newcommand{\chapter}{\section} 
\input{thesis_notations}
\end{comment}

\chapter{Multiplicativity of the $\gamma$-factor}\label{section:gamma_mult}
In this chapter we prove that the $\gamma$-factor
$\gamma(\pi\times\tau,\psi,s)$ is multiplicative in both variables $\pi$ and $\tau$.

\section{Multiplicativity in the second variable}\label{section:2nd variable}

\subsection{Outline of the proof}
Let $\pi$ be a representation of $G_{\smallidx}$. We prove
Theorem~\ref{theorem:multiplicity second var}, namely if $\tau$ is a
quotient of
$\cinduced{P_{\bigidx_1,\bigidx_2}}{\GL{\bigidx}}{\tau_1\otimes\tau_2}$,
\begin{align}\label{eq:gamma second var}
\gamma(\pi\times\tau,\psi,s)=\gamma(\pi\times\tau_1,\psi,s)\gamma(\pi\times\tau_2,\psi,s).
\end{align}
We prove the theorem for
$\tau=\cinduced{P_{\bigidx_1,\bigidx_2}}{\GL{\bigidx}}{\tau_1\otimes\tau_2}$.
The case where $\tau$ is a quotient of this representation follows
immediately, because we can replace the Whittaker model of $\tau$
with the Whittaker model of
$\cinduced{P_{\bigidx_1,\bigidx_2}}{\GL{\bigidx}}{\tau_1\otimes\tau_2}$.

The proof is divided into five cases, in each we apply three
functional equations: for $\pi\times\tau_2$, $\tau_1\times\tau_2$
and $\pi\times\tau_1$, in this order. The form of the first equation
depends on whether $\smallidx\leq\bigidx_2$ or
$\smallidx>\bigidx_2$, since this determines the form of the integral for $\pi\times\tau_2$. 
Then we apply
Shahidi's functional equation \eqref{eq:shahidi func eq}. Next we
have two cases again, according to the form of the integral 
for $\pi\times\tau_1$. Altogether these are four cases which we
handle when $\smallidx\leq\bigidx$, the last
case is when $\smallidx>\bigidx$. 
The final step is to use \eqref{eq:multiplicative property of
intertwiners} in order to combine the applications of three
intertwining operators, each emerging from the respective equation,
into one operator $\nintertwiningfull{\tau}{s}$.

The arguments involve two essentially different types of passages
between integrals. One occurs between integrals which have a common
domain of absolute convergence, the other when the integrals
converge in possibly disjoint domains. 

\subsection{Twisting the embedding and the mysterious $c(\smallidx,\tau,\gamma,s)$}\label{subsection:Twisting the embedding}
Throughout the proof we will obtain integrals $\Psi(W,f_s,s)$ 
where the embedding $G_{\smallidx}<H_{\bigidx}$ (or
$H_{\bigidx}<G_{\smallidx}$) is ``twisted" by some element. The twisted
embedding still preserves the character $\psi_{\gamma}$ of $N_{\bigidx-\smallidx}$ (or $N^{\smallidx-\bigidx}$, when $\smallidx>\bigidx$) used to define the original embedding but
the actual basis changes. We describe these twists here for later
reference.

In general let $A$ be a group and $B<A$ a subgroup. For a function
$f$ defined on $B$ and $a\in A$ which normalizes $B$, define
$f^a$ by $f^a(b)= f(\rconj{a}b)$, i.e.
$f^a=a\cdot\lambda(a) f$.

For $k_1,k_2\geq1$ set
\begin{align*}
b_{k_1,k_2}=diag(I_{k_1},(-1)^{k_2},I_{k_1})\in\GL{2k_1+1}.
\end{align*}
Let $\tau$ be a representation of $\GL{\bigidx}$ and let $k\geq1$.
Let $f_s\in\xi(\tau,std,s)$. Since $b_{\bigidx,k}$ commutes with the
elements of $M_{\bigidx}$ and normalizes $U_{\bigidx}$ and
$K_{H_{\bigidx}}$, $f_s^{b_{\bigidx,k}}\in\xi(\tau,std,s)$. It
follows that the mapping $f_s\mapsto f_s^{b_{\bigidx,k}}$ takes
holomorphic sections to holomorphic sections. If $g\in
G_{\smallidx}$ is regarded as an element of $H_{\bigidx}$ (assuming
$\smallidx\leq\bigidx$), $\rconj{b_{\bigidx,k}}g$ still stabilizes
$e_{\gamma}$ and in fact this conjugation corresponds to multiplying
the matrix $M$ of Section~\ref{subsection:G_l in H_n} on the right
by $b_{\smallidx,k}$. Because
$\rconj{b_{\bigidx,k}}R_{\smallidx,\bigidx}=R_{\smallidx,\bigidx}$
and $b_{\bigidx,k}$ stabilizes $\psi_{\gamma}$,
\begin{align}\label{eq:integral with b n k}
\Psi(W,f_s^{b_{\bigidx,k}},s)
&=\int_{\lmodulo{U_{G_{\smallidx}}}{G_{\smallidx}}}W(g)\int_{R_{\smallidx,\bigidx}}f_s(w_{\smallidx,\bigidx}r(\rconj{b_{\bigidx,k}}g),1)\psi_{\gamma}(r)drdg.
\end{align}
In addition
$\nintertwiningfull{\tau}{s}f_s^{b_{\bigidx,k}}=(\nintertwiningfull{\tau}{s}f_s)^{b_{\bigidx,k}}$.

Let $\epsilon=(-1)^{\bigidx-\smallidx}$. For $k\geq1$ put
\begin{align*}
t_k=diag(\epsilon\gamma^{-1}I_k,1,\epsilon\gamma I_k)\in H_k.
\end{align*}
Then for $f_s\in\xi(\tau,hol,s)$,
$f_s^{t_{\bigidx}}=\omega_{\tau}(\epsilon\gamma)|\gamma|^{\bigidx(\half\bigidx+s-\half)}t_{\bigidx}\cdot f_s\in\xi(\tau,hol,s)$. If $|\gamma|\ne1$, it is possible that for
a standard section $f_s$, $f_s^{t_{\bigidx}}$ is not a standard section.
For $f_s\in\xi(\tau,hol,s)$,
\begin{align*}
\Psi(W,f_s^{t_{\bigidx}},s)
=\omega_{\tau}(\epsilon\gamma)|\gamma|^{\bigidx(\half\bigidx+s-\half)}\Psi(W,t_{\bigidx}\cdot f_s,s).
\end{align*}
Using the definition of
$\nintertwiningfull{\tau}{s}$ (see Section~\ref{subsection:the
intertwining operator for tau}),
\begin{align*}
\nintertwiningfull{\tau}{s}\lambda(t_{\bigidx}) f_s=
\omega_{\tau}(\gamma)^{2}|\gamma|^{2\bigidx(s-\half)}\lambda(t_{\bigidx})\nintertwiningfull{\tau}{s}f_s.
\end{align*}
Therefore
\begin{align}\label{eq:intertwiner and left and right translation}
\nintertwiningfull{\tau}{s} f_s^{t_{\bigidx}}=
\omega_{\tau}(\gamma)^{2}|\gamma|^{2\bigidx(s-\half)}(\nintertwiningfull{\tau}{s}f_s)^{t_{\bigidx}}.
\end{align}
Note that when $\bigidx<\smallidx$, the normalization factor on the
\rhs\ is just $c(\smallidx,\tau,\gamma,s)^{-1}$. In the course of
the proof whenever $\bigidx_i<\smallidx\leq\bigidx$ and we form a
$G_{\smallidx}\times\GL{\bigidx_i}$ integral for $\pi\times\tau_i$,
the argument $h\in \lmodulo{U_{H_{\bigidx_i}}}{H_{\bigidx_i}}$ of
the section corresponding to the representation induced from
$\tau_i$ will be twisted by $t_{\bigidx_i}$, introducing the factor
$c(\smallidx,\tau_i,\gamma,s)$ (see the proof of
Claim~\ref{claim:n_1 < l < n_2 functional eq pi and tau_1}).

\subsection{The basic identity}\label{subsection:The basic identity tau}
We use the realization of $\tau$ described in
Section~\ref{subsection:fixed zeta}, 
i.e., replace $\tau$ with $\varepsilon$, fix $\zeta$ with
$\Re(\zeta)>>0$ and realize the space
$\xi(\varepsilon_1\otimes\varepsilon_2,std,(s,s))$ by extending
functions of the space of $\cinduced{Q_{\bigidx_1}\cap
K_{H_{\bigidx}}}{K_{H_{\bigidx}}}{\varepsilon_1\otimes
\cinduced{Q_{\bigidx_2}\cap
K_{H_{\bigidx_2}}}{K_{H_{\bigidx_2}}}{\varepsilon_2}}$.

For $W\in\Whittaker{\pi}{\psi_{\gamma}^{-1}}$ and
$\varphi_s\in\xi(\varepsilon_1\otimes\varepsilon_2,hol,(s,s))$, by
Corollary~\ref{corollary:deducing varphi and f_s are interchangeable
in psi for all varphi} we have
$\Psi(W,\varphi_{s},s)=\Psi(W,\widehat{f}_{\varphi_s},s)$ in
$\C(q^{-s})$. Additionally, Claim~\ref{claim:deducing varphi and f_s
are interchangeable in psi} implies that for $\Re(s)<<0$,
$\Psi(W,\nintertwiningfull{\varepsilon}{s}\varphi_{s},1-s)=\Psi(W,\widehat{f}_{\nintertwiningfull{\varepsilon}{s}\varphi_{s}},1-s)$.
Since by Claim~\ref{claim:intertwining operator and Jacquet integral
commute},
$\widehat{f}_{\nintertwiningfull{\varepsilon}{s}\varphi_{s}}=\nintertwiningfull{\varepsilon}{s}\widehat{f}_{\varphi_{s}}$,
$\Psi(W,\nintertwiningfull{\varepsilon}{s}\varphi_{s},1-s)=\Psi(W,\nintertwiningfull{\varepsilon}{s}\widehat{f}_{\varphi_{s}},1-s)$.
Therefore
$\Psi(W,\nintertwiningfull{\varepsilon}{s}\varphi_{s},1-s)$ also
extends to a function in $\C(q^{-s})$ and the last equality holds in
$\C(q^{-s})$.

In Sections~\ref{subsection:n_1<l<=n_2}-\ref{subsection:n<l} we
prove that for all $W\in\Whittaker{\pi}{\psi_{\gamma}^{-1}}$ and
$\varphi_s\in\xi(\varepsilon_1\otimes\varepsilon_2,hol,(s,s))$, in
$\C(q^{-s})$,
\begin{align}\label{eq:gamma second var for zeta large with integrals}
&\gamma(\pi\times\varepsilon_1,\psi,s)\gamma(\pi\times\varepsilon_2,\psi,s)\Psi(W,\varphi_{s},s)=c(\smallidx,\varepsilon,\gamma,s)\Psi(W,\nintertwiningfull{\varepsilon}{s}\varphi_{s},1-s).
\end{align}
Let $f_s=\widehat{f}_{\varphi_{s}}\in\xi(\Whittaker{\varepsilon}{\psi},hol,s)$ be defined by \eqref{iso:iso 1}.
It follows that
\begin{align*}
&\gamma(\pi\times\varepsilon_1,\psi,s)\gamma(\pi\times\varepsilon_2,\psi,s)\Psi(W,f_{s},s)
=c(\smallidx,\varepsilon,\gamma,s)\Psi(W,\nintertwiningfull{\varepsilon}{s}f_{s},1-s)
\end{align*}
and according to the definitions,
\begin{align*}
&\gamma(\pi\times\tau_1,\psi,s+\zeta)\gamma(\pi\times\tau_2,\psi,s-\zeta)\Psi(W,f_{s},s)
=c(\smallidx,\varepsilon,\gamma,s)\Psi(W,\nintertwiningfull{\varepsilon}{s}f_{s},1-s).
\end{align*}
Hence for all $\Re(\zeta)>>0$,
\begin{align}\label{eq:gamma second var for zeta large 2}
&\gamma(\pi\times\tau_1,\psi,s+\zeta)\gamma(\pi\times\tau_2,\psi,s-\zeta)=\gamma(\pi\times\varepsilon,\psi,s).
\end{align}
Finally we claim,
\begin{claim}\label{claim:gamma is rational in s and zeta}
$\gamma(\pi\times\varepsilon,\psi,s)\in\C(q^{-\zeta},q^{-s})$.
\end{claim}
This means that \eqref{eq:gamma second var for zeta large 2} holds
for all $\zeta$ and 
\eqref{eq:gamma second var} follows immediately by putting $\zeta=0$.

\begin{proof}[Proof of Claim~\ref{claim:gamma is rational in s and
zeta}]  
Let $W\in\Whittaker{\pi}{\psi_{\gamma}^{-1}}$ and
$\varphi_{\zeta,s}\in\xi(\tau_1\otimes\tau_2,hol,(s+\zeta,s-\zeta))$.
For any fixed $\zeta$,
$\varphi_{\zeta,s}\in\xi(\varepsilon_1\otimes\varepsilon_2,hol,(s,s))$.
Since $\nintertwiningfull{\varepsilon}{s}$ is the composition of
\begin{align*}
&\nintertwiningfull{\varepsilon_2}{s}=\nintertwiningfull{\tau_2}{s-\zeta},\\\notag
&\nintertwiningfull{\varepsilon_1\otimes\varepsilon_2^*}{(s,1-s)}=\nintertwiningfull{\tau_1\otimes\tau_2^*}{(s+\zeta,1-s+\zeta)},\\\notag
&\nintertwiningfull{\varepsilon_1}{s}=\nintertwiningfull{\tau_1}{s+\zeta}
\end{align*}
(see \eqref{eq:multiplicative property of intertwiners}), we have
$\nintertwiningfull{\varepsilon}{s}\varphi_{\zeta,s}\in
\xi(\tau_2^*\otimes\tau_1^*,rat,(1-s+\zeta,1-s-\zeta))$. In
particular, $\nintertwiningfull{\varepsilon}{s}\varphi_{\zeta,s}$ is
defined for any $\zeta$ and if we fix $\zeta$,
$\nintertwiningfull{\varepsilon}{s}\widehat{f}_{\varphi_{\zeta,s}}
=\widehat{f}_{\nintertwiningfull{\varepsilon}{s}\varphi_{\zeta,s}}
\in\xi(\Whittaker{\varepsilon^*}{\psi},rat,1-s)$ is also defined.
Now from \eqref{eq:gamma def}, for any fixed $\zeta$,
\begin{align*}
\gamma(\pi\times\varepsilon,\psi,s)\Psi(W,\widehat{f}_{\varphi_{\zeta,s}},s)=
c(\smallidx,\varepsilon,\gamma,s)\Psi(W,\nintertwiningfull{\varepsilon}{s}\widehat{f}_{\varphi_{\zeta,s}},1-s).
\end{align*}

Take $W$ and $\varphi_{\zeta,s}$ according to
Proposition~\ref{proposition:integral can be made constant zeta phi
version}. Then for all $\zeta$ and $s$,
$\Psi(W,\widehat{f}_{\varphi_{\zeta,s}},s)=1$. Hence for any fixed
$\zeta$,
\begin{align*}
\gamma(\pi\times\varepsilon,\psi,s)=c(\smallidx,\varepsilon,\gamma,s)\Psi(W,\nintertwiningfull{\varepsilon}{s}\widehat{f}_{\varphi_{\zeta,s}},1-s).
\end{align*}

Let $C\subset\C$ be a compact set. We will show that there is
$Q_C\in\C(q^{-\zeta},q^{-s})$ such that
$\Psi(W,\nintertwiningfull{\varepsilon}{s}\widehat{f}_{\varphi_{\zeta,s}},1-s)=Q_C(q^{-\zeta},q^{-s})$
for all $\zeta\in C$ and $s$ in a left half-plane depending on $C$.
Since $C$ is arbitrary, it follows that there is
$Q\in\C(q^{-\zeta},q^{-s})$ such that for all complex $\zeta$, and
$s$ in a left half-plane depending on $\zeta$,
$\Psi(W,\nintertwiningfull{\varepsilon}{s}\widehat{f}_{\varphi_{\zeta,s}},1-s)=Q(q^{-\zeta},q^{-s})$.
Thus for all $\zeta\in \C$,
$\gamma(\pi\times\varepsilon,\psi,s)=c(\smallidx,\varepsilon,\gamma,s)Q$.
In addition
$c(\smallidx,\varepsilon,\gamma,s)\in\C[q^{\mp\zeta},q^{\mp s}]^*$,
because if $\smallidx>\bigidx$,
$c(\smallidx,\varepsilon,\gamma,s)=|\gamma|^{2\zeta(\bigidx_2-\bigidx_1)}c(\smallidx,\tau,\gamma,s)$.
Therefore $\gamma(\pi\times\tau,\psi,s)\in\C(q^{-\zeta},q^{-s})$.

As explained in Section~\ref{subsection:zeta in a fixed compact},
there is a constant $s_C$ depending only on $C$, $\pi$ and $\tau$,
such that in a domain $D_1=\setof{\zeta,s}{\zeta\in C,\Re(s)<S_C}$,
$\Psi(W,\theta_{1-s},1-s)$ is absolutely convergent for any $W$ and
$\theta_{1-s}\in
V_{Q_{\bigidx}}^{H_{\bigidx}}(\Whittaker{\varepsilon^*}{\psi},1-s)$.

Write
$\nintertwiningfull{\varepsilon}{s}\varphi_{\zeta,s}=\ell_{\varepsilon}(s)\sum_{i=1}^mP_i\varphi_{\zeta,s}^{(i)}$
where $P_i\in\C[q^{\mp\zeta},q^{\mp s}]$ and
$\varphi_{\zeta,s}^{(i)}\in
\xi(\tau_2^*\otimes\tau_1^*,std,(1-s+\zeta,1-s-\zeta))$. For any
fixed $\zeta\in C$ the meromorphic continuation of
$\Psi(W,\nintertwiningfull{\varepsilon}{s}\widehat{f}_{\varphi_{\zeta,s}},1-s)$
is equal to the meromorphic continuation of
$\sum_{i=1}^m\ell_{\varepsilon}(s)P_i\Psi(W,\widehat{f}_{\varphi_{\zeta,s}^{(i)}},1-s)$.

According to Claim~\ref{claim:iwasawa decomposition of the integral
with zeta in tau}, we can write
$\Psi(W,\widehat{f}_{\varphi_{\zeta,s}^{(i)}},1-s)$ as a sum of
integrals $I_{1-s}^{(j)}$ in the form prescribed by
Proposition~\ref{proposition:iwasawa decomposition of the integral},
where $W'\in\Whittaker{\tau^*}{\psi}$ is replaced by
$W_{\zeta}'\in\Whittaker{\varepsilon^*}{\psi}$ with the properties
stated in the claim. Taking the constant $s_C$ small enough, the
integrals $I_{1-s}^{(j)}$ are also absolutely convergent in $D_1$.
Therefore we should prove that there is
$Q_j\in\C(q^{-\zeta},q^{-s})$ such that
$I_{1-s}^{(j)}=Q_j(q^{-\zeta},q^{-s})$ in $D_1$.

We can use the asymptotic expansion of the Whittaker functions
$W^{\diamond}\in\Whittaker{\pi}{\psi_{\gamma}^{-1}}$ and
$W_{\zeta}'$ (see Section~\ref{section:whittaker props}) in order to
write $I_{1-s}^{(j)}$ in $D_1$, as a sum of products of Tate-type
integrals, as in the proof of Claim~\ref{claim:meromorphic
continuation 1}. 

In the notation of Section~\ref{section:whittaker props}, let
$\mathcal{A}_{\tau^*}$ (resp. $\mathcal{A}_{\varepsilon^*}$) be the
finite set of functions for $\tau^*$ (resp. $\varepsilon^*$). The
twist of $\tau_2^*\otimes\tau_1^*$ by
$\absdet{}^{\zeta}\otimes\absdet{}^{-\zeta}$ has the effect of
twisting the exponents of $\tau^*$ by $\absdet{}^{\mp\zeta}$. These
exponents are related to $\mathcal{A}_{\varepsilon^*}$. When we
write an expansion for $W_{\zeta}'$, we can use the explicit formula
for Whittaker functions given by Lapid and Mao \cite{LM}, detailed
in the proof of Claim~\ref{claim:convergence of absolute value l<=n
split SO part for s>0}. The expansion takes the same form described
in Section~\ref{section:whittaker props} (it is a refinement), but
the characters from $\mathcal{A}_{\varepsilon^*}$ are exponents.
Then each character $\eta\in\mathcal{A}_{\varepsilon^*}$ appearing
in the expansion of $W_{\zeta}'$ is equal to
$\absdet{}^{\mp\zeta}\eta_i'$ for some character $\eta_i'\in
\mathcal{A}_{\tau^*}$. The other terms in the expansion are
independent of $\zeta$. Therefore each of the Tate-type integrals
belongs to $\C(q^{-\zeta},q^{-s})$.
\end{proof} 
For any pair of representations $\xi_1\times\xi_2$ of
$\GL{k_1}\times\GL{k_2}$ and $s_1,s_2\in\C$, denote by
$V'(\xi_1\otimes\xi_2,(s_1,s_2))$ the space of the representation
$\cinduced{Q_{k_1}}{H_{k_1+k_2}}{(\xi_1\otimes
\cinduced{Q_{k_2}}{H_{k_2}}{\xi_2\alpha^{s_2}})\alpha^{s_1}}$. Let
$W\in\Whittaker{\pi}{\psi_{\gamma}^{-1}}$ and
$\varphi_s\in\xi(\varepsilon_1\otimes\varepsilon_2,hol,(s,s))$.
\subsection{The case $\bigidx_1<\smallidx\leq\bigidx_2$}\label{subsection:n_1<l<=n_2}
Recall that $\Psi(W,\varphi_s,s)$ equals
\begin{align}\label{int:mult l<n_2 starting form}
\int_{\lmodulo{U_{G_{\smallidx}}}{G_{\smallidx}}}W(g)
\int_{R_{\smallidx,\bigidx}}
\int_{Z_{\bigidx_2,\bigidx_1}}\varphi_s(\omega_{\bigidx_1,\bigidx_2}zw_{\smallidx,\bigidx}rg,1,1,1)
\psi^{-1}(z)\psi_{\gamma}(r)dzdrdg.
\end{align}
As proved in Claim~\ref{claim:convergence n_1<l<n_2 starting
integral} this integral is absolutely convergent in some domain in
$\zeta$ and $s$ depending only on $\pi$, $\tau_1$ and $\tau_2$, and
we can take $\zeta$ with $\Re(\zeta)>>0$. The domains of convergence
will change in the course of the proof, from
$\{\Re(s)>>\Re(\zeta)\}$ to $\{A_1\leq\Re(s)<<\Re(\zeta)\}$, then
$\{-\Re(\zeta)<<\Re(s)<<\Re(\zeta),\Re(s)\leq A_2\}$ ($A_1,A_2$ are
constants) and finally $\{\Re(1-s)>>\Re(\zeta)\}$. The specific
parameters, which depend only on the representations, are similar to
the ones computed by Soudry \cite{Soudry2}. The important thing to
note is, that there is a constant $C>0$ depending only on the
representations, such that any $\zeta$ with $\Re(\zeta)>C$ is
simultaneously suitable for all the domains. Thus a change of
domains effectively changes just $s$.

In the domain of absolute convergence of \eqref{int:mult l<n_2
starting form}, the forthcoming manipulations are justified.
If we write
\begin{align*}
Z_{\bigidx_2,\bigidx_1}=\{\left(\begin{array}{ccc}
I_{\smallidx}&0&z_1\\&I_{\bigidx_2-\smallidx}&z_2\\&&I_{\bigidx_1}\\\end{array}\right)\},
\end{align*}
\begin{align*}
\rconj{w_{\smallidx,\bigidx}}Z_{\bigidx_2,\bigidx_1}=\{\left(\begin{array}{ccccccc}
I_{\bigidx_1}&z_2'&&&\gamma^{-1}z_1'\\&I_{\bigidx_2-\smallidx}\\&&I_{\smallidx}&&&&\gamma^{-1}z_1\\&&&1\\
&&&&I_{\smallidx}\\&&&&&I_{\bigidx_2-\smallidx}&z_2\\&&&&&&I_{\bigidx_1}\end{array}\right)\}.
\end{align*}
We see that $\rconj{w_{\smallidx,\bigidx}}Z_{\bigidx_2,\bigidx_1}$
is a subgroup of $N_{\bigidx-\smallidx}$ which normalizes
$R_{\smallidx,\bigidx}$. Therefore after conjugating $z$ by
$w_{\smallidx,\bigidx}$, we may change the order of
$\rconj{w_{\smallidx,\bigidx}}z$ and $r$ to get
\begin{align*}
\int_{\lmodulo{U_{G_{\smallidx}}}{G_{\smallidx}}}W(g)
\int_{R_{\smallidx,\bigidx}}
\int_{Z_{\bigidx_2,\bigidx_1}}\varphi_s(\omega_{\bigidx_1,\bigidx_2}w_{\smallidx,\bigidx}r(\rconj{w_{\smallidx,\bigidx}}z)g,1,1,1)
\psi^{-1}(z)\psi_{\gamma}(r)dzdrdg.
\end{align*}
If $\smallidx<\bigidx_2$,
$z_2'=-J_{\bigidx_1}\transpose{z_2}J_{\bigidx_2-\smallidx}$, whence
\begin{align*}
\psi^{-1}(z)=\psi^{-1}((z_2)_{\bigidx_2-\smallidx,1})=
\psi((z_2')_{\bigidx_1,1})=\psi_{\gamma}(\rconj{w_{\smallidx,\bigidx}}z).
\end{align*}
When $\smallidx=\bigidx_2$ also
$\psi^{-1}(z)=\psi_{\gamma}(\rconj{w_{\smallidx,\bigidx}}z)$.

Decompose $R_{\smallidx,\bigidx}=R_{\smallidx,\bigidx_2}\ltimes
R_{\bigidx_1}$, where $R_{\smallidx,\bigidx_2}<H_{\bigidx_2}$ and
\begin{align*}
R_{\bigidx_1}=R_{\smallidx,\bigidx}\cap
U_{\bigidx_1}=\{\left(\begin{array}{ccccccc}
I_{\bigidx_1}&&x_1&y_1&&m_1&m_2\\&I_{\bigidx_2-\smallidx}&&&&&m_1'\\&&I_{\smallidx}&&&&\\&&&1&&&y_1'\\
&&&&I_{\smallidx}&&x_1'\\&&&&&I_{\bigidx_2-\smallidx}&\\&&&&&&I_{\bigidx_1}\end{array}\right)\}.
\end{align*}
The integral equals
\begin{align*}
&\int_{\lmodulo{U_{G_{\smallidx}}}{G_{\smallidx}}}W(g)
\int_{R_{\smallidx,\bigidx_2}} \int_{R_{\bigidx_1}}
\int_{Z_{\bigidx_2,\bigidx_1}}\\&\varphi_s(\omega_{\bigidx_1,\bigidx_2}w_{\smallidx,\bigidx}rr_1(\rconj{w_{\smallidx,\bigidx}}z)g,1,1,1)
\psi_{\gamma}(rr_1)\psi_{\gamma}(\rconj{w_{\smallidx,\bigidx}}z)dzdr_1drdg.
\end{align*}
Note that if $\smallidx<\bigidx_2$,
$\frestrict{\psi_{\gamma}}{R_{\bigidx_1}}\equiv1$. Now
$R_{\bigidx_1}\rtimes
\rconj{w_{\smallidx,\bigidx}}Z_{\bigidx_2,\bigidx_1}=U_{\bigidx_1}$
and we define $du=dzdr_1$.
Therefore we have 
\begin{align*}
\int_{\lmodulo{U_{G_{\smallidx}}}{G_{\smallidx}}}W(g)
\int_{R_{\smallidx,\bigidx_2}}
\int_{U_{\bigidx_1}}\varphi_s(\omega_{\bigidx_1,\bigidx_2}w_{\smallidx,\bigidx}rug,1,1,1)
\psi_{\gamma}(r)\psi_{\gamma}(u)dudrdg.
\end{align*}

Since $G_{\smallidx}<M_{\bigidx_1}$, $G_{\smallidx}$ normalizes
$U_{\bigidx_1}$, the measure $du$ and the character $\psi_{\gamma}$
remain unchanged and we get
\begin{align*}
\int_{\lmodulo{U_{G_{\smallidx}}}{G_{\smallidx}}}W(g)
\int_{R_{\smallidx,\bigidx_2}}\int_{U_{\bigidx_1}}
\varphi_s(\omega_{\bigidx_1,\bigidx_2}w_{\smallidx,\bigidx}rgu,1,1,1)
\psi_{\gamma}(r)\psi_{\gamma}(u)dudrdg.
\end{align*}
Observe that
$\omega_{\bigidx_1,\bigidx_2}w_{\smallidx,\bigidx}=w_{\smallidx,\bigidx_2}w'$
with
\begin{align*}
w'=
\left(\begin{array}{ccc}
&&I_{\bigidx_1}\\
&b_{\bigidx_2,\bigidx_1}\\
I_{\bigidx_1}\end{array}\right)\in H_{\bigidx}.
\end{align*}
The element $w'$ normalizes $R_{\smallidx,\bigidx_2}$ preserving
$dr$ and $\psi_{\gamma}$. Since
$w_{\smallidx,\bigidx_2},r,\rconj{w'}g\in H_{\bigidx_2}$, we may
shift these elements to the third argument of $\varphi_s$. After
changing integration order we arrive at
\begin{align}\label{int:mult l<n_2 before applying eq for inner tau}
\int_{U_{\bigidx_1}}
\int_{\lmodulo{U_{G_{\smallidx}}}{G_{\smallidx}}}W(g)
\int_{R_{\smallidx,\bigidx_2}}
\varphi_s(w'u,1,w_{\smallidx,\bigidx_2}r(\rconj{b_{\bigidx_2,\bigidx_1}}g),1)
\psi_{\gamma}(r)\psi_{\gamma}(u)drdgdu.
\end{align}
The inner $drdg$-integration in \eqref{int:mult l<n_2 before
applying eq for inner tau} is a Rankin-Selberg integral for
$G_{\smallidx}\times\GL{\bigidx_2}$ and $\pi\times\varepsilon_2$
(regarding the conjugation by $b_{\bigidx_2,\bigidx_1}$, see
Section~\ref{subsection:Twisting the embedding}). Integral~\eqref{int:mult l<n_2 before applying eq for inner tau} is absolutely convergent if it is convergent when we replace $W,\varphi_s$ with $|W|,|\varphi_s|$ and remove $\psi_{\gamma}$. If we repeat the manipulations
\eqref{int:mult l<n_2 starting form}-\eqref{int:mult l<n_2 before applying eq for
inner tau} with $W,\varphi_s$ replaced with $|W|,|\varphi_s|$ and without $\psi_{\gamma},\psi^{-1}$, we reach
\eqref{int:mult l<n_2 before applying eq for inner tau} with $|W|,|\varphi_s|$ and without $\psi_{\gamma}$. This shows that \eqref{int:mult l<n_2 starting form} and \eqref{int:mult l<n_2 before applying eq for
inner tau} are absolutely convergent in the same domain. Let
\begin{align*}
\varphi_{s,1-s}=\nintertwiningfull{\varepsilon_2}{s}\varphi_s\in\xi(\varepsilon_1\otimes\varepsilon_2^*,rat,(s,1-s))
\end{align*}
be defined as in Section~\ref{subsection:the multiplicativity of the
intertwining operator for tau induced}. Formally, after applying the
functional equation \eqref{eq:gamma def} we expect to have
\begin{align}\label{int:mult l<n_2 after applying eq for inner tau}
\int_{U_{\bigidx_1}}
\int_{\lmodulo{U_{G_{\smallidx}}}{G_{\smallidx}}}W(g)
\int_{R_{\smallidx,\bigidx_2}}
\varphi_{s,1-s}(w'u,1,w_{\smallidx,\bigidx_2}r(\rconj{b_{\bigidx_2,\bigidx_1}}g),1)
\psi_{\gamma}(r)\psi_{\gamma}(u)drdgdu.
\end{align}
Integral~\eqref{int:mult l<n_2 after applying eq for inner tau}
converges absolutely in some domain in $\zeta$ and $s$ depending
only on the representations, see Soudry \cite{Soudry2} (Lemma~3.3).

As observed in \cite{Soudry} (p.70), there is a problem with the
passage from integral~\eqref{int:mult l<n_2 before applying eq for
inner tau} to integral~\eqref{int:mult l<n_2 after applying eq for
inner tau}, namely that they might have disjoint domains of absolute
convergence. Following the reasoning of \cite{Soudry,Soudry2}, this
passage is justified by considering each integral as a meromorphic
function in $\C(q^{-s})$.

Integral~\eqref{int:mult l<n_2 after applying eq for inner tau}
extends to a function in $\C(q^{-s})$. To see this, note that we can
repeat the steps \eqref{int:mult l<n_2 starting
form}-\eqref{int:mult l<n_2 before applying eq for inner tau} with
$\varphi_{s,1-s}$ replacing $\varphi_s$ and obtain \eqref{int:mult
l<n_2 after applying eq for inner tau}. This proves that in its
domain of absolute convergence \eqref{int:mult l<n_2 after applying
eq for inner tau} satisfies the same equivariance properties as
$\Psi(W,\varphi_s,s)$. Also analogously to
Proposition~\ref{proposition:integral can be made constant zeta phi
version} (for a fixed $\zeta$), we can select data
$W,\theta_{s,1-s}$
($\theta_{s,1-s}\in\xi(\varepsilon_1\otimes\varepsilon_2^*,hol,(s,1-s))$)
such that \eqref{int:mult l<n_2 starting form} and \eqref{int:mult
l<n_2 after applying eq for inner tau} (with $\theta_{s,1-s}$) equal
$1$ for all $s$. Now the meromorphic continuation follows as
explained in the proof of Claim~\ref{claim:meromorphic continuation
2}.

Therefore except for finitely many values of
$q^{-s}$, both integrals~\eqref{int:mult l<n_2 before applying eq for
inner tau} and \eqref{int:mult l<n_2 after applying eq for
inner tau} can be considered by meromorphic continuation as bilinear forms on
$\Whittaker{\pi}{\psi_{\gamma}^{-1}}\times
V'(\varepsilon_1\otimes\varepsilon_2,(s,s))$ satisfying
\eqref{eq:bilinear special condition}, hence they are proportional
(see Section~\ref{subsection:the gamma factor}). The
proportionality factor is calculated using specific substitutions of
$W$ and $\varphi_s$, as we show in the next claim.
\begin{claim}\label{claim:n_1 < l < n_2 functional eq pi and tau_2}
Integral~\eqref{int:mult l<n_2 before applying eq for inner tau}
multiplied by $\gamma(\pi\times\varepsilon_2,\psi,s)$ equals
integral~\eqref{int:mult l<n_2 after applying eq for inner tau}, as
meromorphic functions. 
\end{claim}
The proofs of this and forthcoming claims are postponed to
Section~\ref{subsection:proofs claims 2nd var}.

In fact, one does not need to prove meromorphic continuation for
\eqref{int:mult l<n_2 after applying eq for inner tau}
independently, repeating the arguments of
Claim~\ref{claim:meromorphic continuation 2}. If
$D^*=\{A_1\leq\Re(s)<<\Re(\zeta)\}$ is the domain of absolute
convergence of \eqref{int:mult l<n_2 after applying eq for inner
tau}, we can consider the integral \eqref{int:mult l<n_2 after
applying eq for inner tau} and the meromorphic continuation of
\eqref{int:mult l<n_2 before applying eq for inner tau} as bilinear
forms on $\Whittaker{\pi}{\psi_{\gamma}^{-1}}\times
V'(\varepsilon_1\otimes\varepsilon_2,(s,s))$ only for $s\in D^*$.
Then the above claim shows that they are proportional by a function
in $\C(q^{-s})$ and the meromorphic continuation of \eqref{int:mult
l<n_2 after applying eq for inner tau} follows from that of
\eqref{int:mult l<n_2 before applying eq for inner tau}. A similar
argument was used in Remark~\ref{remark:proving equality of W and f
in any domain without Bernstein continuation principle}. 

Hereby we consider \eqref{int:mult l<n_2 after applying eq for inner
tau} in its domain of absolute convergence, where our forthcoming
manipulations are allowed. Apply the steps from \eqref{int:mult
l<n_2 starting form} to \eqref{int:mult l<n_2 before applying eq for
inner tau} in reverse order to integral~\eqref{int:mult l<n_2
after applying eq for inner tau} and arrive at
\begin{align}\label{int:mult l<n_2 after eq inner tau starting form}
\int_{\lmodulo{U_{G_{\smallidx}}}{G_{\smallidx}}}W(g)
\int_{R_{\smallidx,\bigidx}}
\int_{Z_{\bigidx_2,\bigidx_1}}\varphi_{s,1-s}(\omega_{\bigidx_1,\bigidx_2}zw_{\smallidx,\bigidx}rg,1,1,1)
\psi^{-1}(z)\psi_{\gamma}(r)dzdrdg.
\end{align}
Denote
\begin{align*}
\varphi_{1-s,s}'=\nintertwiningfull{\varepsilon_1\otimes\varepsilon_2^*}{(s,1-s)}\varphi_{s,1-s}\in\xi(\varepsilon_2^*\otimes\varepsilon_1,rat,(1-s,s))
\end{align*}
(see Section~\ref{subsection:the multiplicativity of the
intertwining operator for tau induced}). Formally apply Shahidi's
functional equation \eqref{eq:shahidi func eq} to \eqref{int:mult
l<n_2 after eq inner tau starting form} and reach
\begin{align}\label{int:mult l<n_2 after shahidi equation}
&\int_{\lmodulo{U_{G_{\smallidx}}}{G_{\smallidx}}}W(g)
\int_{R_{\smallidx,\bigidx}} \int_{Z_{\bigidx_1,\bigidx_2}}
\varphi_{1-s,s}'(\omega_{\bigidx_2,\bigidx_1}mw_{\smallidx,\bigidx}rg,1,1,1)
\psi^{-1}(m)\psi_{\gamma}(r)dmdrdg.
\end{align}
There is a domain in $\zeta$ and $s$ at which
integral~\eqref{int:mult l<n_2 after shahidi equation} converges
absolutely (see \cite{Soudry2} Lemma~3.6). This domain and
the domain of \eqref{int:mult l<n_2 after eq inner tau starting
form} might be disjoint. We tackle this as above, considering
\eqref{int:mult l<n_2 after eq inner tau starting form} and
\eqref{int:mult l<n_2 after shahidi equation} by meromorphic continuation as bilinear forms on
$\Whittaker{\pi}{\psi_{\gamma}^{-1}}\times
V'(\varepsilon_1\otimes\varepsilon_2^*,(s,1-s))$
and using a specific
substitution.
\begin{claim}\label{claim:n_1 < l < n_2 shahidi equation}
Integral~\eqref{int:mult l<n_2 after eq inner tau starting form}
equals integral~\eqref{int:mult l<n_2 after shahidi equation}, as
meromorphic functions.
\end{claim}

We continue with integral~\eqref{int:mult l<n_2 after shahidi
equation} in its domain of absolute convergence. Next we show how a part of the
$dm$-integration may be absorbed into $U_{G_{\smallidx}}$. Write
\begin{align*}
m=\left(\begin{array}{ccc}I_{\bigidx_1}&&m_3\\&I_{\smallidx-\bigidx_1}\\&&I_{\bigidx-\smallidx}\\\end{array}\right)
\left(\begin{array}{cccc}I_{\bigidx_1}&m_1&m_2&0\\&I_{\smallidx-\bigidx_1-1}\\&&1\\&&&I_{\bigidx-\smallidx}\\\end{array}\right)
=\eta\varsigma.
\end{align*}
Then the $dm$-integration in \eqref{int:mult l<n_2 after shahidi
equation} equals
\begin{align*}
&\int_{\Mat{\bigidx_1\times
\bigidx-\smallidx}}\int_{\Mat{\bigidx_1\times \smallidx-\bigidx_1}}
\varphi_{1-s,s}'(\omega_{\bigidx_2,\bigidx_1}\eta\varsigma
w_{\smallidx,\bigidx}rg,1,1,1) \psi^{-1}(\varsigma)d\varsigma d\eta.
\end{align*}

Let $Q_{\bigidx_1}<H_{\bigidx_1}$ be the standard parabolic subgroup
whose Levi part is isomorphic to $\GL{\bigidx_1}$ and let
$U_{\bigidx_1}<Q_{\bigidx_1}$ be its unipotent radical. We can
regard $U_{\bigidx_1}$ as a subgroup of
$V_{\bigidx_1}<G_{\smallidx}$. The subgroup
$Z_{\bigidx_1,\smallidx-\bigidx_1-1}<\GL{\smallidx-1}$, embedded in
$G_{\smallidx}$ through $L_{\smallidx-1}$, is also a subgroup of
$V_{\bigidx_1}$ (if $\smallidx=\bigidx_1+1$,
$Z_{\bigidx_1,\smallidx-\bigidx_1-1}=\{1\}$). Put
\begin{align*}
w''=\left(\begin{array}{ccccc}I_{\bigidx_1}\\&&&I_{\smallidx-\bigidx_1-1}&\\
&&I_2\\
&I_{\smallidx-\bigidx_1-1}\\
&&&&I_{\bigidx_1}\\
\end{array}\right)\in\GL{2\smallidx}.
\end{align*}
Now let
$V_{\bigidx_1}'=U_{\bigidx_1}\cdot\rconj{w''}Z_{\bigidx_1,\smallidx-\bigidx_1-1}$
(a direct product). Then $V_{\bigidx_1}'$ is a normal subgroup of
$V_{\bigidx_1}$ and we fix a set of representatives
$V_{\bigidx_1}''$ for $\lmodulo{V_{\bigidx_1}'}{V_{\bigidx_1}}$. In
the split case,
\begin{align*}
&V_{\bigidx_1}'=\setof{v\in V_{\bigidx_1}}
{[v]_{\mathcal{E}_{G_{\smallidx}}}=\left(\begin{array}{cccccc}I_{\bigidx_1}&0&v_3&\frac1{4\gamma}v_3&v_4&v_5
\\&I_{\smallidx-\bigidx_1-1}&&&&v_4'\\
&&1&&&\frac1{4\gamma}v_3'\\
&&&1&&v_3'\\
&&&&I_{\smallidx-\bigidx_1-1}&0\\
&&&&&I_{\bigidx_1}\\
\end{array}\right)},\\
&V_{\bigidx_1}''=\setof{v\in V_{\bigidx_1}}
{[v]_{\mathcal{E}_{G_{\smallidx}}}=\left(\begin{array}{cccccc}I_{\bigidx_1}&v_1&0&v_2&0&0
\\&I_{\smallidx-\bigidx_1-1}&&&&0\\
&&1&&&v_2'\\
&&&1&&0\\
&&&&I_{\smallidx-\bigidx_1-1}&v_1'\\
&&&&&I_{\bigidx_1}\\
\end{array}\right)}.
\end{align*}
In the \quasisplit\ case,
\begin{align*}
&V_{\bigidx_1}'=\setof{v\in V_{\bigidx_1}}
{[v]_{\mathcal{E}_{G_{\smallidx}}}=\left(\begin{array}{cccccc}I_{\bigidx_1}&0&v_3&0&v_4&v_5
\\&I_{\smallidx-\bigidx_1-1}&&&&-J_{\smallidx-\bigidx_1-1}\transpose{v_4}J_{\bigidx_1}\\
&&1&&&-\transpose{v_3}J_{\bigidx_1}\\
&&&1&&0\\
&&&&I_{\smallidx-\bigidx_1-1}&0\\
&&&&&I_{\bigidx_1}\\
\end{array}\right)},\\
&V_{\bigidx_1}''=\setof{v\in V_{\bigidx_1}}
{[v]_{\mathcal{E}_{G_{\smallidx}}}=\left(\begin{array}{cccccc}I_{\bigidx_1}&v_1&0&v_2&0&A(v_2)
\\&I_{\smallidx-\bigidx_1-1}&&&&0\\
&&1&&&0\\
&&&1&&\frac1{\rho}(\transpose{v_2})J_{\bigidx_1}\\
&&&&I_{\smallidx-\bigidx_1-1}&-J_{\smallidx-\bigidx_1-1}\transpose{v_1}J_{\bigidx_1}\\
&&&&&I_{\bigidx_1}\\
\end{array}\right)},
\end{align*}
where $A(v_2)=\frac1{2\rho}v_2(\transpose{v_2})J_{\bigidx_1}$. For $v\in V_{\bigidx_1}''$, the first $\bigidx_1$ rows of $\rconj{w_{\smallidx,\bigidx}^{-1}}[v]_{\mathcal{E}_{H_{\bigidx}}}$
are of the form
\begin{align*}
\begin{dcases}
\left(\begin{array}{cccccccc} I_{\bigidx_1}&v_1&
-\gamma v_2&0_{\bigidx-\smallidx}&(-1)^{\bigidx-\smallidx}\frac{\beta^3}2 v_2&0_{\bigidx-\smallidx}&\gamma^2 v_2&0_{\smallidx-1}\\
\end{array}\right)&\text{split $G_{\smallidx}$,}\\
\left(\begin{array}{ccccccc} I_{\bigidx_1}&v_1&
\half v_2&0_{2(\bigidx-\smallidx)+1}&-\half\gamma v_2&0_{\smallidx-\bigidx_1-1}&\gamma^2 A(v_2)\\
\end{array}\right)&\text{\quasisplit\ $G_{\smallidx}$}.
\end{dcases}
\end{align*}
Here $0_k$ denotes the zero matrix of $\Mat{\bigidx_1\times k}$.
Note that while in the split case we in fact have
$V_{\bigidx_1}=V_{\bigidx_1}''\ltimes V_{\bigidx_1}'$, in the
\quasisplit\ case $V_{\bigidx_1}''$ is not a subgroup of
$V_{\bigidx_1}$, but the quotient topology on $V_{\bigidx_1}''$ is
homeomorphic to the natural topology $V_{\bigidx_1}''$ inherits from
$\Mat{\bigidx_1\times\smallidx-\bigidx_1}$.

We see that $\varsigma$ is a general element from the projection of
$\rconj{w_{\smallidx,\bigidx}^{-1}}V_{\bigidx_1}''$ into
$M_{\bigidx}$ and $\eta$ commutes with any element from the
projection of $\rconj{w_{\smallidx,\bigidx}^{-1}}V_{\bigidx_1}''$
into $U_{\bigidx}$. Since
$\varphi_{1-s,s}'\in\xi(\varepsilon_2^*\otimes\varepsilon_1,rat,(1-s,s))$
and $\omega_{\bigidx_2,\bigidx_1}$ normalizes $U_{\bigidx}$, the
function on $H_{\bigidx}$ given by
$h\mapsto\varphi_{1-s,s}'(\omega_{\bigidx_2,\bigidx_1}h,1,1,1)$ is
left $U_{\bigidx}$-invariant. Therefore integral~\eqref{int:mult
l<n_2 after shahidi equation} equals
\begin{align*}
&\int_{\lmodulo{U_{G_{\smallidx}}}{G_{\smallidx}}}W(g)
\int_{R_{\smallidx,\bigidx}} \int_{\Mat{\bigidx_1\times
\bigidx-\smallidx}}\int_{V_{\bigidx_1}''}
\\\notag&\varphi_{1-s,s}'(\omega_{\bigidx_2,\bigidx_1}\eta
w_{\smallidx,\bigidx}v''rg,1,1,1)
\psi_{\gamma}^{-1}(v'')\psi_{\gamma}(r)dv''d\eta drdg.
\end{align*}
Here the measure $dv''$ on $V_{\bigidx_1}''$ is normalized according
to $d\varsigma$ and $\psi_{\gamma}^{-1}(v'')$ refers to the
character $\psi_{\gamma}$ of $U_{G_{\smallidx}}$. Note that the
$dv''$-integration makes sense, i.e. for $v'\in V_{\bigidx_1}'$,
\begin{align*}
\varphi_{1-s,s}'(\omega_{\bigidx_2,\bigidx_1}\eta
w_{\smallidx,\bigidx}v'v'',1,1,1)
\psi_{\gamma}^{-1}(v'v'')=\varphi_{1-s,s}'(\omega_{\bigidx_2,\bigidx_1}\eta
w_{\smallidx,\bigidx}v'',1,1,1) \psi_{\gamma}^{-1}(v'').
\end{align*}
Hence we may also change the order of integration, $dv''d\eta
dr\mapsto d\eta drdv''$.

In general for $v\in V_{\bigidx_1}$ and $r\in
R_{\smallidx,\bigidx}$, $\rconj{v}r$ is a member of
$N_{\bigidx-\smallidx}$ not necessarily in $R_{\smallidx,\bigidx}$.
However, we may write $\rconj{v}r=n_vr_v$ for $n_v\in
N_{\bigidx-\smallidx}$ such that $n_v\notin R_{\smallidx,\bigidx}$
and $r_v\in R_{\smallidx,\bigidx}$. By a direct verification,
\begin{align*}
&\int_{R_{\smallidx,\bigidx}}\int_{\Mat{\bigidx_1\times
\bigidx-\smallidx}}
\varphi_{1-s,s}'(\omega_{\bigidx_2,\bigidx_1}\eta
w_{\smallidx,\bigidx}n_vr_v,1,1,1) \psi_{\gamma}(r)d\eta
dr\\\notag&=\int_{R_{\smallidx,\bigidx}}\int_{\Mat{\bigidx_1\times
\bigidx-\smallidx}}
\varphi_{1-s,s}'(\omega_{\bigidx_2,\bigidx_1}\eta
w_{\smallidx,\bigidx}r,1,1,1) \psi_{\gamma}(r)d\eta dr.
\end{align*}
In fact, $n_v$ contributes a value of the form $\psi(\cdots)$ that
is canceled by a change of variables $r_v\mapsto r$ which affects
$\psi_{\gamma}(r)$.
Therefore it is allowed to change the order of $v''$ and $r$.
The integral becomes
\begin{align*}
&\int_{\lmodulo{U_{G_{\smallidx}}}{G_{\smallidx}}}W(v''g)
\int_{V_{\bigidx_1}''}\int_{R_{\smallidx,\bigidx}}\int_{\Mat{\bigidx_1\times
\bigidx-\smallidx}}
\varphi_{1-s,s}'(\omega_{\bigidx_2,\bigidx_1}\eta
w_{\smallidx,\bigidx}rv''g,1,1,1) \psi_{\gamma}(r)d\eta drdv''dg.
\end{align*}

Since $U_{G_{\smallidx}}=(Z_{\bigidx_1}\ltimes V_{\bigidx_1})\rtimes
U_{G_{\smallidx-\bigidx_1}}
$ and $V_{\bigidx_1}'$ is normalized by $U_{G_{\smallidx-\bigidx_1}}$, we get
\begin{align}\label{integral:mult l<n_2 integral before defining factoring R}
&\int_{\lmodulo{Z_{\bigidx_1}V_{\bigidx_1}'U_{G_{\smallidx-\bigidx_1}}}{G_{\smallidx}}}W(g)
\int_{R_{\smallidx,\bigidx}}\int_{\Mat{\bigidx_1\times
\bigidx-\smallidx}}
\varphi_{1-s,s}'(\omega_{\bigidx_2,\bigidx_1}\eta
w_{\smallidx,\bigidx}rg,1,1,1)\psi_{\gamma}(r)d\eta drdg.
\end{align}

Observe that
$R^{\smallidx,\bigidx_1}(V_{\bigidx_1}'U_{G_{\smallidx-\bigidx_1}})=(V_{\bigidx_1}'U_{G_{\smallidx-\bigidx_1}})R^{\smallidx,\bigidx_1}$,
whence it is a subgroup. In coordinates, for split $G_{\smallidx}$,
\begin{align*}
R^{\smallidx,\bigidx_1}V_{\bigidx_1}'U_{G_{\smallidx-\bigidx_1}}=
\Big\{\left(\begin{array}{cccccc}I_{\bigidx_1}&0&v_3&\frac1{4\gamma}v_3&v_4&v_5\\
x&z&u_1&u_2&u_3&v_4'\\
&&1&0&u_2'&\frac1{4\gamma}v_3'\\
&&&1&u_1'&v_3'\\
&&&&z^*&0\\
&&&&x'&I_{\bigidx_1}
\end{array}\right):z\in Z_{\smallidx-\bigidx_1-1}\Big\}
\end{align*}
(in the \quasisplit\ case instead of $\frac1{4\gamma}v_3$ we have
$0$, see the form of $V_{\bigidx_1}'$ above). Then we see that
$R^{\smallidx,\bigidx_1}V_{\bigidx_1}'U_{G_{\smallidx-\bigidx_1}}$
is normalized by $Z_{\bigidx_1}$. Therefore
integral~\eqref{integral:mult l<n_2 integral before defining
factoring R} factors through $R^{\smallidx,\bigidx_1}$. Furthermore,
the function on $G_{\smallidx}$ defined by
\begin{align*}
g\mapsto\int_{R_{\smallidx,\bigidx}}\int_{\Mat{\bigidx_1\times
\bigidx-\smallidx}}
\varphi_{1-s,s}'(\omega_{\bigidx_2,\bigidx_1}\eta
w_{\smallidx,\bigidx}rg,1,1,1)\psi_{\gamma}(r) d\eta dr
\end{align*}
is $R^{\smallidx,\bigidx_1}$-invariant on the left. To see this,
first note that $R^{\smallidx,\bigidx_1}$ normalizes
$R_{\smallidx,\bigidx}$ keeping $dr$ and $\psi_{\gamma}$ unchanged.
Second, if $r'\in R^{\smallidx,\bigidx_1}$,
$\rconj{w_{\smallidx,\bigidx}^{-1}}r'$ is the image of
\begin{align*}
\left(\begin{array}{cccc}
I_{\bigidx_1}&&&\\
x&I_{\smallidx-\bigidx_1-1}&&\\
&&1&\\
&&&I_{\bigidx-\smallidx}\end{array}\right)\in\overline{Z_{\bigidx_1,\bigidx_2}}
\end{align*}
in $M_{\bigidx}$ and
$\rconj{(\rconj{w_{\smallidx,\bigidx}^{-1}}r')}\eta=\upsilon\eta$
where $\upsilon$ is the image in $M_{\bigidx}$ of
\begin{align*}
\left(\begin{array}{cccc}
I_{\bigidx_1}&&&\\
&I_{\smallidx-\bigidx_1-1}&&-xm_3\\
&&1&\\
&&&I_{\bigidx-\smallidx}\\\end{array}\right)\in
Z_{\bigidx_1,\bigidx_2}.
\end{align*}
Now for any $h\in H_{\bigidx}$,
$\varphi_{1-s,s}'(\omega_{\bigidx_2,\bigidx_1}(\rconj{w_{\smallidx,\bigidx}^{-1}}r')h,1,1,1)
=\varphi_{1-s,s}'(\omega_{\bigidx_2,\bigidx_1}h,1,1,1)$
and
\begin{align*}
\varphi_{1-s,s}'(\omega_{\bigidx_2,\bigidx_1}\upsilon
h,1,1,1)=\psi(\rconj{\omega_{\bigidx_2,\bigidx_1}^{-1}}\upsilon)\varphi_{1-s,s}'(\omega_{\bigidx_2,\bigidx_1}h,1,1,1)
=\varphi_{1-s,s}'(\omega_{\bigidx_2,\bigidx_1}h,1,1,1).
\end{align*}
Thus we obtain
\begin{align*}
&\int_{\lmodulo{R^{\smallidx,\bigidx_1}Z_{\bigidx_1}V_{\bigidx_1}'U_{G_{\smallidx-\bigidx_1}}}{G_{\smallidx}}}
(\int_{R^{\smallidx,\bigidx_1}}W(r'g)dr')
\int_{R_{\smallidx,\bigidx}}\int_{\Mat{\bigidx_1\times
\bigidx-\smallidx}}\\\notag
&\varphi_{1-s,s}'(\omega_{\bigidx_2,\bigidx_1}\eta
w_{\smallidx,\bigidx}rg,1,1,1)\psi_{\gamma}(r)d\eta drdg.
\end{align*}

Next (inspecting the coordinates above) we see that
\begin{align*}
\rconj{w^{\smallidx,\bigidx_1}}(R^{\smallidx,\bigidx_1}Z_{\bigidx_1}V_{\bigidx_1}'U_{G_{\smallidx-\bigidx_1}})
=U_{H_{\bigidx_1}}\ltimes(Z_{\smallidx-\bigidx_1-1}\ltimes
V_{\smallidx-\bigidx_1-1}),
\end{align*}
where $U_{H_{\bigidx_1}}$ is embedded in $G_{\smallidx}$ through
$G_{\bigidx_1+1}$.
Hence
$H_{\bigidx_1}\rconj{w^{\smallidx,\bigidx_1}}(R^{\smallidx,\bigidx_1}Z_{\bigidx_1}V_{\bigidx_1}'U_{G_{\smallidx-\bigidx_1}})$
is a subgroup of $G_{\smallidx}$
($H_{\bigidx_1}<G_{\bigidx_1+1}<G_{\smallidx}$) and we may factor
the integral through $H_{\bigidx_1}$ to get
\begin{align}\label{integral:mult l<n_2 integral before defining F}
&\int_{\lmodulo{H_{\bigidx_1}Z_{\smallidx-\bigidx_1-1}V_{\smallidx-\bigidx_1-1}}{G_{\smallidx}}}\int_{\lmodulo{U_{H_{\bigidx_1}}}{H_{\bigidx_1}}}
(\int_{R^{\smallidx,\bigidx_1}}W(r'w^{\smallidx,\bigidx_1}h'g)dr')\\\notag
&\int_{R_{\smallidx,\bigidx}}\int_{\Mat{\bigidx_1\times
\bigidx-\smallidx}}
\varphi_{1-s,s}'(\omega_{\bigidx_2,\bigidx_1}\eta
w_{\smallidx,\bigidx}rw^{\smallidx,\bigidx_1}h'g,1,1,1)\psi_{\gamma}(r)d\eta
drdh'dg.
\end{align}
Note that here $H_{\bigidx_1}$ is considered as a subgroup of
$G_{\smallidx}$, so as input to $W$, $h'$ is written relative to
$\mathcal{E}_{G_{\smallidx}}$. In the first argument of
$\varphi_{1-s,s}'$ we write $h'$ by passing from
$\mathcal{E}_{G_{\smallidx}}$ to $\mathcal{E}_{H_{\bigidx}}$. Now
the following technical claim reveals the inner integral we seek.
\begin{claim}\label{claim:n_1 < l < n_2 properties of F}
In its domain of absolute convergence, integral~\eqref{integral:mult
l<n_2 integral before defining F} equals
\begin{align}\label{int:mult var 2 before applying eq for outer tau}
&\int_{\lmodulo{H_{\bigidx_1}Z_{\smallidx-\bigidx_1-1}V_{\smallidx-\bigidx_1-1}}{G_{\smallidx}}}
\int_{R_{\smallidx,\bigidx}}\int_{\Mat{\bigidx_1\times\bigidx-\smallidx}}
\int_{\lmodulo{U_{H_{\bigidx_1}}}{H_{\bigidx_1}}}\\\notag
&(\int_{R^{\smallidx,\bigidx_1}}W(r'w^{\smallidx,\bigidx_1}h'g)dr')\varphi_{1-s,s}'(\omega_{\bigidx_2,\bigidx_1}\eta
w_{\smallidx,\bigidx}rw^{\smallidx,\bigidx_1}g,1,\rconj{t_{\bigidx_1}}h',1)\psi_{\gamma}(r)dh'd\eta
drdg.
\end{align}
Here as input to $\varphi_{1-s,s}'$, $h'$ is written in coordinates
relative to $\mathcal{E}_{H_{\bigidx_1}}$.
\end{claim}

Integral~\eqref{int:mult var 2 before applying eq for outer tau} is
absolutely convergent if it is convergent when $W,\varphi_{1-s,s}'$
are replaced with $|W|,|\varphi_{1-s,s}'|$ and $\psi_{\gamma}$ is
dropped. The manipulations \eqref{int:mult l<n_2 after shahidi
equation}-\eqref{int:mult var 2 before applying eq for outer tau}
remain valid if we take $|W|,|\varphi_{1-s,s}'|$ and drop the
characters, then we reach \eqref{int:mult var 2 before applying eq
for outer tau} with $|W|,|\varphi_{1-s,s}'|$ and without
$\psi_{\gamma}$. This shows that both \eqref{int:mult l<n_2 after
shahidi equation} and \eqref{int:mult var 2 before applying eq for
outer tau} are absolutely convergent in the same domain.

One recognizes the inner $dr'dh'$-integration in \eqref{int:mult var
2 before applying eq for outer tau} as a Rankin-Selberg integral for
$G_{\smallidx}\times\GL{\bigidx_1}$ and $\pi\times\varepsilon_1$
(see Section~\ref{subsection:Twisting the embedding} for the
conjugation by $t_{\bigidx_1}$). Let
\begin{align*}
\varphi_{1-s}^*=\nintertwiningfull{\varepsilon_1}{s}\varphi_{1-s,s}'\in\xi(\varepsilon_2^*\otimes\varepsilon_1^*,rat,(1-s,1-s)).
\end{align*}
After (formally) applying the functional equation \eqref{eq:gamma
def} we should get
\begin{align}\label{int:mult var 2 after applying eq for outer tau}
&\int_{\lmodulo{H_{\bigidx_1}Z_{\smallidx-\bigidx_1-1}V_{\smallidx-\bigidx_1-1}}{G_{\smallidx}}}
\int_{R_{\smallidx,\bigidx}}\int_{\Mat{\bigidx_1\times\bigidx-\smallidx}}
\int_{\lmodulo{U_{H_{\bigidx_1}}}{H_{\bigidx_1}}}\\\notag
&(\int_{R^{\smallidx,\bigidx_1}}W(r'w^{\smallidx,\bigidx_1}h'g)dr')\varphi_{1-s}^*(\omega_{\bigidx_2,\bigidx_1}\eta
w_{\smallidx,\bigidx}rw^{\smallidx,\bigidx_1}g,1,\rconj{t_{\bigidx_1}}h',1)\psi_{\gamma}(r)dh'd\eta
drdg.
\end{align}
Integral~\eqref{int:mult var 2 after applying eq for outer tau}
converges absolutely in a domain in $\zeta$ and $s$ (see
\cite{Soudry2} Lemma~3.10) and can be regarded by meromorphic continuation as
a bilinear form on
$\Whittaker{\pi}{\psi_{\gamma}^{-1}}\times
V'(\varepsilon_2^*\otimes\varepsilon_1,(1-s,s))$. The passage is justified by the next claim.
\begin{claim}\label{claim:n_1 < l < n_2 functional eq pi and tau_1}
Integral~\eqref{int:mult var 2 before applying eq for outer tau}
multiplied by $\gamma(\pi\times\varepsilon_1,\psi,s)$ equals
integral~\eqref{int:mult var 2 after applying eq for outer tau}, as
meromorphic functions.
\end{claim}
Reversing the passage \eqref{int:mult l<n_2 after shahidi
equation}-\eqref{int:mult var 2 before applying eq for outer tau}
and using the fact that by \eqref{eq:multiplicative property of
intertwiners},
$\varphi_{1-s}^*=\nintertwiningfull{\varepsilon}{s}\varphi_s$,
yields $\Psi(W,\nintertwiningfull{\varepsilon}{s}\varphi_s,1-s)$,
i.e.
\begin{align*}
&\int_{\lmodulo{U_{G_{\smallidx}}}{G_{\smallidx}}}W(g)
\int_{R_{\smallidx,\bigidx}} \int_{Z_{\bigidx_1,\bigidx_2}}
\nintertwiningfull{\varepsilon}{s}\varphi_s(\omega_{\bigidx_2,\bigidx_1}mw_{\smallidx,\bigidx}rg,1,1,1)
\psi^{-1}(m)\psi_{\gamma}(r)dmdrdg.
\end{align*}
This also shows that the domain of absolute convergence of
\eqref{int:mult var 2 after applying eq for outer tau} is of the
form prescribed by Claim~\ref{claim:convergence n_1<l<n_2 starting
integral} with $s$ replaced by $1-s$, i.e.
$\{\Re(1-s)>>\Re(\zeta)\}$.

Altogether collecting Claims~\ref{claim:n_1 < l < n_2 functional eq
pi and tau_2}, \ref{claim:n_1 < l < n_2 shahidi equation} and
\ref{claim:n_1 < l < n_2 functional eq pi and tau_1} gives
\eqref{eq:gamma second var for zeta large with integrals} (note that
$\smallidx<\bigidx$ and whenever $\smallidx\leq\bigidx$,
$c(\smallidx,\varepsilon,\gamma,s)=1$).

\begin{remark}\label{remark:technical claims are the same for gamma second var}
In Sections~\ref{subsection:l<=n_1,n_2}-\ref{subsection:n<l},
justifications to integral manipulations are implicit. Passages
involving the application of a functional equation are always
between meromorphic continuations. Other manipulations take place in
the same domain of absolute convergence.
\end{remark}

\subsection{The case $\smallidx\leq\bigidx_1,\bigidx_2$}\label{subsection:l<=n_1,n_2}
We may repeat the steps in Section~\ref{subsection:n_1<l<=n_2} from
integral~\eqref{int:mult l<n_2 starting form} to
integral~\eqref{int:mult l<n_2 after shahidi equation}. Our starting
point is thus
\begin{align}\label{int:mult l<=n_1,n_2 after shahidi equation}
&\int_{\lmodulo{U_{G_{\smallidx}}}{G_{\smallidx}}}W(g)
\int_{R_{\smallidx,\bigidx}}\int_{Z_{\bigidx_1,\bigidx_2}}
\varphi_{1-s,s}'(\omega_{\bigidx_2,\bigidx_1}mw_{\smallidx,\bigidx}rg,1,1,1)
\psi^{-1}(m)\psi_{\gamma}(r)dmdrdg.
\end{align}

We continue by repeating the steps from \eqref{int:mult l<n_2
starting form} to \eqref{int:mult l<n_2 before applying eq for inner
tau}, with the roles of $\bigidx_1$ and $\bigidx_2$ exchanged. We look at
$\rconj{w_{\smallidx,\bigidx}}Z_{\bigidx_1,\bigidx_2}<N_{\bigidx-\smallidx}$
and decompose $R_{\smallidx,\bigidx}=R_{\smallidx,\bigidx_1}\ltimes
R_{\bigidx_2}$, with $R_{\smallidx,\bigidx_1}<H_{\bigidx_1}$. Here
$R_{\bigidx_2}\rtimes\rconj{w_{\smallidx,\bigidx}}Z_{\bigidx_1,\bigidx_2}=U_{\bigidx_2}$
and $\psi^{-1}(m)=\psi_{\gamma}(\rconj{w_{\smallidx,\bigidx}}m)$.

Now the integral takes the form
\begin{align}\label{int:mult l<=n_1,n_2 before applying eq for outer tau}
&\int_{U_{\bigidx_2}}\int_{\lmodulo{U_{G_{\smallidx}}}{G_{\smallidx}}}W(g)
\int_{R_{\smallidx,\bigidx_1}}
\varphi_{1-s,s}'(w'u,1,w_{\smallidx,\bigidx_1}r(\rconj{b_{\bigidx_1,\bigidx_2}}g),1)
\psi_{\gamma}(r)\psi_{\gamma}(u)drdgdu.
\end{align}
The inner $drdg$-integration in \eqref{int:mult l<=n_1,n_2 before
applying eq for outer tau} comprises an integral for
$G_{\smallidx}\times\GL{\bigidx_1}$ and $\pi\times\varepsilon_1$.
Let
\begin{align}\label{int:mult l<=n_1,n_2 after applying eq for outer tau}
&\int_{U_{\bigidx_2}}\int_{\lmodulo{U_{G_{\smallidx}}}{G_{\smallidx}}}W(g)
\int_{R_{\smallidx,\bigidx_1}}
\varphi_{1-s}^*(w'u,1,w_{\smallidx,\bigidx_1}r(\rconj{b_{\bigidx_1,\bigidx_2}}g),1)
\psi_{\gamma}(r)\psi_{\gamma}(u)drdgdu.
\end{align}
Then integral~\eqref{int:mult l<=n_1,n_2 before applying eq for
outer tau} multiplied by $\gamma(\pi\times\varepsilon_1,\psi,s)$
equals integral~\eqref{int:mult l<=n_1,n_2 after applying eq for
outer tau}, as meromorphic functions. Reproducing the steps from
\eqref{int:mult l<=n_1,n_2 after shahidi equation} to
\eqref{int:mult l<=n_1,n_2 before applying eq for outer tau} in an
opposite direction one arrives at
$\Psi(W,\nintertwiningfull{\varepsilon}{s}\varphi_s,1-s)$. As in
Section~\ref{subsection:n_1<l<=n_2} this establishes \eqref{eq:gamma
second var for zeta large with integrals}.

\subsection{The case $\bigidx_2<\smallidx\leq\bigidx_1$}\label{subsection:n_2<l<=n_1}
Essentially, this case is obtained by exchanging
$\varepsilon_i,\bigidx_{i}\leftrightarrow
\varepsilon_{3-i},\bigidx_{3-i}$ in
Section~\ref{subsection:n_1<l<=n_2}. However we still start with
constructing an integral for $\pi\times\varepsilon_2$. Begin with
$\Psi(W,\varphi_s,s)$ and argue as in the passage \eqref{int:mult
l<n_2 after shahidi equation}-\eqref{int:mult var 2 before applying
eq for outer tau}, with the roles of $\bigidx_1$ and $\bigidx_2$
reversed. There we showed how to collapse a part of the
$dm$-integration into $U_{G_{\smallidx}}$. Here this is done with
the $dz$-integration over $Z_{\bigidx_2,\bigidx_1}$.
Then $\Psi(W,\varphi_s,s)$ equals
\begin{align}\label{integral:mult n_2<l<=n_1 integral before defining F}
\int_{\lmodulo{Z_{\bigidx_2}V_{\bigidx_2}'U_{G_{\smallidx-\bigidx_2}}}{G_{\smallidx}}}W(g)
\int_{R_{\smallidx,\bigidx}}
\int_{\Mat{\bigidx_2\times\bigidx-\smallidx}}
\varphi_s(\omega_{\bigidx_1,\bigidx_2}\eta
w_{\smallidx,\bigidx}rg,1,1,1) \psi_{\gamma}(r)d\eta drdg.
\end{align}
This integral resembles integral~\eqref{integral:mult l<n_2 integral
before defining factoring R}. 
Similarly to the passage \eqref{integral:mult l<n_2 integral before
defining factoring R}-\eqref{integral:mult l<n_2 integral before
defining F}, the $d\eta dr$-integration is
$R^{\smallidx,\bigidx_2}$-invariant on the left (as a function of
$g$) and we may also factor through $H_{\bigidx_2}$. Analogously to
Claim~\ref{claim:n_1 < l < n_2 properties of F},
\begin{claim}\label{claim:n_2<l<=n_1 properties of F}
In its domain of absolute convergence, integral~\eqref{integral:mult
n_2<l<=n_1 integral before defining F} equals
\begin{align}\label{int:mult var 2 n_2<l<=n_1 before applying eq for inner tau}
&\int_{\lmodulo{H_{\bigidx_2}Z_{\smallidx-\bigidx_2-1}V_{\smallidx-\bigidx_2-1}}{G_{\smallidx}}}
\int_{R_{\smallidx,\bigidx}}\int_{\Mat{\bigidx_2\times\bigidx-\smallidx}}
\int_{\lmodulo{U_{H_{\bigidx_2}}}{H_{\bigidx_2}}}\\\notag
&(\int_{R^{\smallidx,\bigidx_2}}W(r'w^{\smallidx,\bigidx_2}h'g)dr')\varphi_s(\omega_{\bigidx_1,\bigidx_2}\eta
w_{\smallidx,\bigidx}rw^{\smallidx,\bigidx_2}g,1,\rconj{t_{\bigidx_2}}h',1)\psi_{\gamma}(r)dh'd\eta
drdg.
\end{align}
Here as input to $\varphi_s$, $h'$ is written in coordinates
relative to $\mathcal{E}_{H_{\bigidx_2}}$.
\end{claim}
The proof of this claim is omitted. 
Applying the functional equation \eqref{eq:gamma def} for
$G_{\smallidx}\times\GL{\bigidx_2}$ and $\pi\times\varepsilon_2$ to
the $dr'dh'$-integration in \eqref{int:mult var 2 n_2<l<=n_1 before
applying eq for inner tau} gives
\begin{align}\label{int:mult var 2 n_2<l<=n_1 after applying eq for inner tau}
&\int_{\lmodulo{H_{\bigidx_2}Z_{\smallidx-\bigidx_2-1}V_{\smallidx-\bigidx_2-1}}{G_{\smallidx}}}
\int_{R_{\smallidx,\bigidx}}\int_{\Mat{\bigidx_2\times\bigidx-\smallidx}}
\int_{\lmodulo{U_{H_{\bigidx_2}}}{H_{\bigidx_2}}}\\\notag
&(\int_{R^{\smallidx,\bigidx_2}}W(r'w^{\smallidx,\bigidx_2}h'g)dr')\varphi_{s,1-s}(\omega_{\bigidx_1,\bigidx_2}\eta
w_{\smallidx,\bigidx}rw^{\smallidx,\bigidx_2}g,1,\rconj{t_{\bigidx_2}}h',1)\psi_{\gamma}(r)dh'd\eta
drdg.
\end{align}
Integral~\eqref{int:mult var 2 n_2<l<=n_1 before applying eq for
inner tau} multiplied by $\gamma(\pi\times\varepsilon_2,\psi,s)$
equals integral~\eqref{int:mult var 2 n_2<l<=n_1 after applying eq
for inner tau}, as meromorphic functions.

Reversing the passage leading from $\Psi(W,\varphi_s,s)$ to
\eqref{int:mult var 2 n_2<l<=n_1 before applying eq for inner tau}
we obtain
\begin{align*}
\int_{\lmodulo{U_{G_{\smallidx}}}{G_{\smallidx}}}W(g)
\int_{R_{\smallidx,\bigidx}}
\int_{Z_{\bigidx_2,\bigidx_1}}\varphi_{s,1-s}(\omega_{\bigidx_1,\bigidx_2}zw_{\smallidx,\bigidx}rg,1,1,1)
\psi^{-1}(z)\psi_{\gamma}(r)dzdrdg,
\end{align*}
which is exactly integral~\eqref{int:mult l<n_2 after eq inner tau
starting form}. Applying \eqref{eq:shahidi func eq} 
we get
\begin{align}\label{int:mult n_2<l<=n_1 after shahidi equation}
&\int_{\lmodulo{U_{G_{\smallidx}}}{G_{\smallidx}}}W(g)
\int_{R_{\smallidx,\bigidx}} \int_{Z_{\bigidx_1,\bigidx_2}}
\varphi_{1-s,s}'(\omega_{\bigidx_2,\bigidx_1}mw_{\smallidx,\bigidx}rg,1,1,1)
\psi^{-1}(m)\psi_{\gamma}(r)dmdrdg,
\end{align}
which is actually integral~\eqref{int:mult l<=n_1,n_2 after shahidi equation}
from Section~\ref{subsection:l<=n_1,n_2}, where also $\smallidx\leq\bigidx_1$.
Repeating the argument there yields \eqref{eq:gamma second var for zeta large with integrals}.

\subsection{The case $\bigidx_1,\bigidx_2<\smallidx\leq\bigidx$}\label{subsection:n_1,n_2<l<n}
We start by repeating the steps of
Section~\ref{subsection:n_2<l<=n_1} from $\Psi(W,\varphi_s,s)$
to 
integral~\eqref{int:mult n_2<l<=n_1 after shahidi equation},
which is the same as integral~\eqref{int:mult l<n_2 after shahidi equation} of Section~\ref{subsection:n_1<l<=n_2},
where also $\bigidx_1<\smallidx$ and we proceed as prescribed there.

\subsection{The case $\bigidx<\smallidx$}\label{subsection:n<l}
Now $\Psi(W,\varphi_s,s)$ takes the form
\begin{align}\label{int:mult n<l starting form}
\int_{\lmodulo{U_{H_{\bigidx}}}{H_{\bigidx}}}
(\int_{R^{\smallidx,\bigidx}}W(rw^{\smallidx,\bigidx}h)dr)
\int_{Z_{\bigidx_2,\bigidx_1}}\varphi_s(\omega_{\bigidx_1,\bigidx_2}zh,1,1,1)
\psi^{-1}(z)dzdh.
\end{align}
Recall that in the Jacquet integral for $\varphi_s$,
$Z_{\bigidx_2,\bigidx_1}$ is a subgroup of $H_{\bigidx}$ embedded
through $M_{\bigidx}$.
For
$z=\left(\begin{array}{cc}I_{\bigidx_2}&z\\&I_{\bigidx_1}\end{array}\right)\in
Z_{\bigidx_2,\bigidx_1}$,
\begin{align*}
\rconj{(w^{\smallidx,\bigidx})^{-1}}z=\left(\begin{array}{ccccc}
I_{\bigidx_2}&z\\&I_{\bigidx_1}\\&&I_{2(\smallidx-\bigidx)}\\
&&&I_{\bigidx_1}&z'\\&&&&I_{\bigidx_2}\\\end{array}\right),
\end{align*}
so $\rconj{(w^{\smallidx,\bigidx})^{-1}}z$ normalizes
$R^{\smallidx,\bigidx}$. Because
$\bigidx_2\leq\bigidx-1\leq\smallidx-2$,
$W((\rconj{(w^{\smallidx,\bigidx})^{-1}}z)g)=\psi^{-1}(z)W(g)$ for
any $g\in G_{\smallidx}$.
It follows that integral~\eqref{int:mult n<l starting form} equals
\begin{align*}
\int_{\lmodulo{U_{H_{\bigidx}}}{H_{\bigidx}}}
\int_{Z_{\bigidx_2,\bigidx_1}} \int_{R^{\smallidx,\bigidx}}
W(rw^{\smallidx,\bigidx}zh)
\varphi_s(\omega_{\bigidx_1,\bigidx_2}zh,1,1,1)drdzdh.
\end{align*}

Write $U_{H_{\bigidx}}=((N_{\bigidx_1}\cdot N_{\bigidx_2})\ltimes Z_{\bigidx_2,\bigidx_1})\ltimes U_{\bigidx}$ with
\begin{align*}
N_{\bigidx_1}=\setof{\left(\begin{array}{cc}I_{\bigidx_2}\\&z_1\\\end{array}\right)}{z_1\in Z_{\bigidx_1}},\qquad
N_{\bigidx_2}=\setof{\left(\begin{array}{cc}z_2\\&I_{\bigidx_1}\\\end{array}\right)}{z_2\in Z_{\bigidx_2}}.
\end{align*}
Here $((N_{\bigidx_1}\cdot N_{\bigidx_2})\ltimes
Z_{\bigidx_2,\bigidx_1})<\GL{\bigidx}$ is considered as a subgroup
of $M_{\bigidx}$. We readily collapse the integration over
$Z_{\bigidx_2,\bigidx_1}$ to obtain
\begin{align*}
\int_{\lmodulo{N_{\bigidx_1}N_{\bigidx_2}U_{\bigidx}}{H_{\bigidx}}}
\int_{R^{\smallidx,\bigidx}} W(rw^{\smallidx,\bigidx}h)
\varphi_s(\omega_{\bigidx_1,\bigidx_2}h,1,1,1)drdh.
\end{align*}

Decompose $R^{\smallidx,\bigidx}=R^{\smallidx,\bigidx}_1\cdot R^{\smallidx,\bigidx}_2$, a direct product where
\begin{align*}
R^{\smallidx,\bigidx}_1=\{\left(\begin{array}{cccc}I_{\bigidx_2}\\&I_{\bigidx_1}\\&x_1&I_{\smallidx-\bigidx-1}\\&&&1\\\end{array}\right)\},\qquad
R^{\smallidx,\bigidx}_2=\{\left(\begin{array}{cccc}I_{\bigidx_2}\\&I_{\bigidx_1}\\x_2&&I_{\smallidx-\bigidx-1}\\&&&1\\\end{array}\right)\}.
\end{align*}
The mapping $h\mapsto\varphi_s(h,1,1,1)$ is
$Z_{\bigidx_1,\bigidx_2}$-invariant on the left, whence
$h\mapsto\varphi_s(\omega_{\bigidx_1,\bigidx_2}h,1,1,1)$ is left
$\overline{Z_{\bigidx_2,\bigidx_1}}=\rconj{\omega_{\bigidx_1,\bigidx_2}}Z_{\bigidx_1,\bigidx_2}$-invariant.
Let
$N_{\bigidx_1,\bigidx_2}=\overline{Z_{\bigidx_2,\bigidx_1}}N_{\bigidx_1}N_{\bigidx_2}$.
It is a subgroup of $M_{\bigidx}$ and thus normalizes $U_{\bigidx}$.
Therefore we may integrate over $\overline{Z_{\bigidx_2,\bigidx_1}}$
and our integral equals
\begin{align*}
\int_{\lmodulo{N_{\bigidx_1,\bigidx_2}U_{\bigidx}}{H_{\bigidx}}}
\int_{\overline{Z_{\bigidx_2,\bigidx_1}}}
\int_{R^{\smallidx,\bigidx}_1}\int_{R^{\smallidx,\bigidx}_2}
W(r_2r_1(\rconj{(w^{\smallidx,\bigidx})^{-1}}z)w^{\smallidx,\bigidx}h)
\varphi_s(\omega_{\bigidx_1,\bigidx_2}h,1,1,1)dr_2dr_1dzdh.
\end{align*}
For $r_1\in R^{\smallidx,\bigidx}_1$ and
$z\in\overline{Z_{\bigidx_2,\bigidx_1}}$,
$\rconj{(w^{\smallidx,\bigidx})^{-1}}z$ satisfies
$\rconj{(\rconj{(w^{\smallidx,\bigidx})^{-1}}z)}r_1=r_{2,z}r_1$ for
some $r_{2,z}\in R^{\smallidx,\bigidx}_2$. Also, as a direct product
$\rconj{(w^{\smallidx,\bigidx})^{-1}}\overline{Z_{\bigidx_2,\bigidx_1}}\cdot
R^{\smallidx,\bigidx}_2=R^{\smallidx,\bigidx_2}$. Thus we have
\begin{align*}
&\int_{\lmodulo{N_{\bigidx_1,\bigidx_2}U_{\bigidx}}{H_{\bigidx}}}
\int_{R^{\smallidx,\bigidx}_1} \int_{R^{\smallidx,\bigidx_2}}
W(rr_1w^{\smallidx,\bigidx}h)
\varphi_s(\omega_{\bigidx_1,\bigidx_2}h,1,1,1)drdr_1dh.
\end{align*}

For $h_2\in H_{\bigidx_2}$, where $H_{\bigidx_2}$ is regarded as a
subgroup of $H_{\bigidx}$, denote
$h_2'=\rconj{\omega_{\bigidx_1,\bigidx_2}}h_2$ and
$H_{\bigidx_2}'=\rconj{\omega_{\bigidx_1,\bigidx_2}}H_{\bigidx_2}$.
We factor the integration through $H_{\bigidx_2}'$. Explicitly
writing $[h_2']_{\mathcal{E}_{G_{\smallidx}}}$, we find that
$\rconj{(w^{\smallidx,\bigidx})^{-1}}h_2'$ commutes with any element
of $R^{\smallidx,\bigidx}_1$ and
$\rconj{(w^{\smallidx,\bigidx})^{-1}}h_2'=\rconj{(w^{\smallidx,\bigidx_2})^{-1}}h_2$.
Hence the integral is equal to
\begin{align}\label{int:mult n<l before applying eq for inner tau}
&\int_{\lmodulo{H_{\bigidx_2}'N_{\bigidx_1,\bigidx_2}U_{\bigidx}}{H_{\bigidx}}}
\int_{R^{\smallidx,\bigidx}_1}
\int_{\lmodulo{U_{H_{\bigidx_2}'}}{H_{\bigidx_2}'}}
\int_{R^{\smallidx,\bigidx_2}}\\\notag
&W(rw^{\smallidx,\bigidx_2}h_2(w^{\smallidx,\bigidx_2})^{-1}r_1w^{\smallidx,\bigidx}h)
\varphi_s(\omega_{\bigidx_1,\bigidx_2}h,1,h_2,1)drdh_2'dr_1dh.
\end{align}
Note that
$N_{\bigidx_1,\bigidx_2}U_{\bigidx}=\rconj{\omega_{\bigidx_1,\bigidx_2}}U_{H_{\bigidx}}$
and $H_{\bigidx_2}U_{H_{\bigidx}}<H_{\bigidx}$, whence $H_{\bigidx_2}'N_{\bigidx_1,\bigidx_2}U_{\bigidx}<H_{\bigidx}$.

The inner $drdh_2'$-integration in \eqref{int:mult n<l before
applying eq for inner tau} is an integral for
$G_{\smallidx}\times\GL{\bigidx_2}$ and $\pi\times\varepsilon_2$.
Let
\begin{align}\label{int:mult n<l after applying eq for inner tau}
&\int_{\lmodulo{H_{\bigidx_2}'N_{\bigidx_1,\bigidx_2}U_{\bigidx}}{H_{\bigidx}}}
\int_{R^{\smallidx,\bigidx}_1}
\int_{\lmodulo{U_{H_{\bigidx_2}'}}{H_{\bigidx_2}'}}\int_{R^{\smallidx,\bigidx_2}}\\\notag
&W(rw^{\smallidx,\bigidx_2}h_2(w^{\smallidx,\bigidx_2})^{-1}r_1w^{\smallidx,\bigidx}h)
\varphi_{s,1-s}(\omega_{\bigidx_1,\bigidx_2}h,1,h_2,1)drdh_2'dr_1dh.
\end{align}
By \eqref{eq:gamma def}, integral~\eqref{int:mult n<l before
applying eq for inner tau} multiplied by
$c(\smallidx,\varepsilon_2,\gamma,s)^{-1}\gamma(\pi\times\varepsilon_2,\psi,s)$
equals \eqref{int:mult n<l after applying eq for inner tau}.

Reversing the steps from \eqref{int:mult n<l starting form} to
\eqref{int:mult n<l before applying eq for inner tau},
integral~\eqref{int:mult n<l after applying eq for inner tau}
becomes
\begin{align}\label{int:mult n<l before applying shahidi equation}
\int_{\lmodulo{U_{H_{\bigidx}}}{H_{\bigidx}}}
(\int_{R^{\smallidx,\bigidx}}W(rw^{\smallidx,\bigidx}h)dr)
\int_{Z_{\bigidx_2,\bigidx_1}}\varphi_{s,1-s}(\omega_{\bigidx_1,\bigidx_2}zh,1,1,1)
\psi^{-1}(z)dzdh.
\end{align}
Applying 
\eqref{eq:shahidi func eq} to \eqref{int:mult n<l before applying shahidi equation} we get
\begin{align}\label{int:mult l<n after shahidi equation}
&\int_{\lmodulo{U_{H_{\bigidx}}}{H_{\bigidx}}}
(\int_{R^{\smallidx,\bigidx}}W(rw^{\smallidx,\bigidx}h)dr)
\int_{Z_{\bigidx_1,\bigidx_2}}
\varphi_{1-s,s}'(\omega_{\bigidx_2,\bigidx_1}mh,1,1,1)
\psi^{-1}(m)dmdh.
\end{align}

Repeat the steps from \eqref{int:mult n<l starting form} to
\eqref{int:mult n<l before applying eq for inner tau}, with the
roles of $\varepsilon_1$, $\bigidx_1$ and $\varepsilon_2$,
$\bigidx_2$ reversed. The first step is to collapse the
$dm$-integration. The relevant modifications of the definitions are
\begin{align*}
\begin{array}{ll}
N_{\bigidx_1}=\setof{\left(\begin{array}{cc}z_1\\&I_{\bigidx_2}\\\end{array}\right)}{z_1\in
Z_{\bigidx_1}},&
N_{\bigidx_2}=\setof{\left(\begin{array}{cc}I_{\bigidx_1}\\&z_2\\\end{array}\right)}{z_2\in
Z_{\bigidx_2}},\\
R^{\smallidx,\bigidx}_1=\{\left(\begin{array}{cccc}I_{\bigidx_1}\\&I_{\bigidx_2}\\x_1&&I_{\smallidx-\bigidx-1}\\&&&1\\\end{array}\right)\},&
R^{\smallidx,\bigidx}_2=\{\left(\begin{array}{cccc}I_{\bigidx_1}\\&I_{\bigidx_2}\\&x_2&I_{\smallidx-\bigidx-1}\\&&&1\\\end{array}\right)\},\\
N_{\bigidx_2,\bigidx_1}=\overline{Z_{\bigidx_1,\bigidx_2}}N_{\bigidx_1}N_{\bigidx_2},&
H_{\bigidx_1}'=\rconj{\omega_{\bigidx_2,\bigidx_1}}H_{\bigidx_1}.
\end{array}
\end{align*}
\begin{remark}
Note that $N_{\bigidx_i}$ and $R^{\smallidx,\bigidx}_i$ are obtained
from the previous definitions by conjugating with
$\omega_{\bigidx_2,\bigidx_1}$.
\end{remark}
Also use the fact that the function
$h\mapsto\varphi_{1-s,s}'(\omega_{\bigidx_2,\bigidx_1}h,1,1,1)$
is left $\overline{Z_{\bigidx_1,\bigidx_2}}$-invariant.
We reach the following integral
\begin{align}\label{int:mult n<l before applying eq for outer tau}
&\int_{\lmodulo{H_{\bigidx_1}'N_{\bigidx_2,\bigidx_1}U_{\bigidx}}{H_{\bigidx}}}
\int_{R^{\smallidx,\bigidx}_2}
\int_{\lmodulo{U_{H_{\bigidx_1}'}}{H_{\bigidx_1}'}}\int_{R^{\smallidx,\bigidx_1}}\\\notag
&W(rw^{\smallidx,\bigidx_1}h_1(w^{\smallidx,\bigidx_1})^{-1}r_2w^{\smallidx,\bigidx}h)
\varphi_{1-s,s}'(\omega_{\bigidx_2,\bigidx_1}h,1,h_1,1)drh_1'dr_2dh.
\end{align}
The inner $drdh_1'$-integration in \eqref{int:mult n<l before
applying eq for outer tau} is an integral for
$G_{\smallidx}\times\GL{\bigidx_1}$ and $\pi\times\varepsilon_1$.
Integral~\eqref{int:mult n<l before applying eq for outer tau}
multiplied by
$c(\smallidx,\varepsilon_1,\gamma,s)^{-1}\gamma(\pi\times\varepsilon_1,\psi,s)$
equals 
\begin{align}\label{int:mult n<l after applying eq for outer tau}
&\int_{\lmodulo{H_{\bigidx_1}'N_{\bigidx_2,\bigidx_1}U_{\bigidx}}{H_{\bigidx}}}
\int_{R^{\smallidx,\bigidx}_2}
\int_{\lmodulo{U_{H_{\bigidx_1}'}}{H_{\bigidx_1}'}}\int_{R^{\smallidx,\bigidx_1}}
\\\notag&W(rw^{\smallidx,\bigidx_1}h_1(w^{\smallidx,\bigidx_1})^{-1}r_2w^{\smallidx,\bigidx}h)
\varphi_{1-s}^*(\omega_{\bigidx_2,\bigidx_1}h,1,h_1,1)drh_1'dr_2dh.
\end{align}
Now \eqref{int:mult n<l after applying eq for outer tau} resembles
\eqref{int:mult n<l before applying eq for inner tau}. Reversing the
transitions \eqref{int:mult n<l starting form}-\eqref{int:mult n<l
before applying eq for inner tau} (with the new definitions of
$N_{\bigidx_i}$, $R^{\smallidx,\bigidx}_i$ etc.) leads to
$\Psi(W,\nintertwiningfull{\varepsilon}{s}\varphi_s,1-s)$. Finally
note that when $\bigidx<\smallidx$, in particular
$\bigidx_1,\bigidx_2<\smallidx$ hence
$c(\smallidx,\varepsilon_1,\gamma,s)c(\smallidx,\varepsilon_2,\gamma,s)=c(\smallidx,\varepsilon,\gamma,s)$.

\subsection{Proofs of claims
}\label{subsection:proofs claims 2nd var}
The following lemma is
used to justify all passages between integrals with disjoint domains
of absolute convergence.

\begin{lemma}\label{lemma:uniformly bounded integral is constant}
Let $\pi$ be a representation of $G_{\smallidx}$ and let $\rho$ be a
representation of $\GL{\bigidx}$ realized in
$\Whittaker{\rho}{\psi}$. 
Let $W\in\Whittaker{\pi}{\psi_{\gamma}^{-1}}$,
$f_s\in\xi(\rho,hol,s)=\xi_{Q_{\bigidx}}^{H_{\bigidx}}(\rho,hol,s)$.
Assume that there is a constant $C$ (independent $s$) such that for
all $s\in\C$,
\begin{align*}
\begin{dcases}
\int_{\lmodulo{U_{H_{\bigidx}}}{H_{\bigidx}}}\int_{R^{\smallidx,\bigidx}}|W|(rw^{\smallidx,\bigidx}h)|f_s|(h,1)drdh<C&\smallidx>\bigidx,\\
\int_{\lmodulo{U_{G_{\smallidx}}}{G_{\smallidx}}}\int_{R_{\smallidx,\bigidx}}
|W|(g)|f_s|(w_{\smallidx,\bigidx}rg,1)drdg<C&\smallidx\leq\bigidx.
\end{dcases}
\end{align*}
Then for all $s$, $\Psi(W,f_s,s)=C'$ where $C'$ is a constant
independent of $s$.
\end{lemma}
\begin{proof}[Proof of Lemma~\ref{lemma:uniformly bounded integral is constant}] 
First assume $\smallidx\leq\bigidx$. According to the Iwasawa
decomposition of $G_{\smallidx}$ there are elements $k_i\in
K_{G_{\smallidx}}$, $i=\intrange{1}{m}$, such that for all $s\in\C$,
\begin{align*}
C&>\sum_{i}\int_{A_{\smallidx-1}}\int_{G_1}|k_{i}\cdot
W|(ax)\int_{R_{\smallidx,\bigidx}}|k_{i}\cdot
f_s|(w_{\smallidx,\bigidx}rax,1)\delta(a)drdxda\\\notag
&\geq\sum_{i}\int_{A_{\smallidx-1}}\int_{G_1}|k_{i}\cdot
W|(ax)\Big|\int_{R_{\smallidx,\bigidx}}k_{i}\cdot
f_s(w_{\smallidx,\bigidx}rax,1)\psi_{\gamma}^{-1}(r)dr\Big|\delta(a)dxda
\end{align*}
($\delta=\delta_{B_{G_{\smallidx}}}^{-1}$). Assume that
$G_{\smallidx}$ is split. The next step is to plug in the formula of
Proposition~\ref{proposition:dr integral for torus of G_l}. This
entails writing the integration over $G_1$ using the subsets
$G_1^{0,k},G_1^{\infty,k}$ (defined in
Section~\ref{subsection:Iwasawa decomposition for Psi(W,f,s)}) and
$\setof{x\in G_1}{[x]\leq q^k}$, where $k$ is a constant independent
of $s$. We get a finite sum bounded by $C$,
\begin{align*}
&\sum_{i,j}\int_{A_{\smallidx-1}}\int_{\Lambda_j}|W_{i,j}|(ax)\Big|\sum_{\alpha}P_{\alpha}(q^{-s},q^s)W_{\alpha}(diag(a,\lfloor
x\rfloor,I_{\bigidx-\smallidx}))\Big|\\\notag&(\absdet{a}[x]^{-1})^{M+\Re(s)}\delta(a)dxda.
\end{align*}
Here $\Lambda_j$ varies over the aforementioned subsets. The
elements $W_{\alpha},P_{\alpha}(X,X^{-1})$ depend on $i$ and $j$ but
not on $s$. The summation over $\alpha$ is finite. Also $M$ is some
real constant independent of $s$. In particular, the only dependence
on $s$ is in the factor $q^{-s}$ replacing $X$ and in the exponent
of $\absdet{a}[x]^{-1}$.

Each $dxda$-integral can be written as a series of the form
\begin{align}\label{lemma:bounded is constant sum absvalue}
\sum_{t=N}^{\infty}\left(\int\int|W_{i,j}|(ax)\Big|\sum_{\alpha}P_{\alpha}(q^{-s},q^s)W_{\alpha}(diag(a,\lfloor
x\rfloor,I_{\bigidx-\smallidx}))\Big|\delta(a)dxda\right)q^{-t(M+\Re(s))}.
\end{align}
Here the double integration is over all $a$ and $x$ such that
$\absdet{a}[x]^{-1}=q^{-t}$. The fact that $W_{\alpha}$ vanishes
away from zero (see Section~\ref{section:whittaker props}) implies
that the coordinates of $a$ are all bounded from above. Hence the
integration is over a compact set independent of $s$, and
$N\in\Integers$ is also independent of $s$.

The same manipulations apply without the absolute value:
\begin{align*}
\Psi(W,f_s,s)=&\sum_{i,j}\int_{A_{\smallidx-1}}\int_{\Lambda_j}W_{i,j}(ax)\sum_{\alpha}P_{\alpha}(q^{-s},q^s)W_{\alpha}(diag(a,\lfloor
x\rfloor,I_{\bigidx-\smallidx}))\\\notag&(\absdet{a}[x]^{-1})^{M+s}\delta(a)dxda
\end{align*}
and each $dxda$-integral takes the form $\sum_{t=N}^{\infty}B_t(s)$,
where
\begin{align}\label{lemma:bounded is constant sum}
B_t(s)=\left(\int\int
W_{i,j}(ax)\sum_{\alpha}P_{\alpha}(q^{-s},q^s)W_{\alpha}(diag(a,\lfloor
x\rfloor,I_{\bigidx-\smallidx}))\delta(a)dxda\right)q^{-t(M+s)}.
\end{align}
Observe that $B_t(s)\in\C[q^{-s},q^s]$. This follows because
$W_{i,j}$ and $W_{\alpha}$ are smooth and the integration in
\eqref{lemma:bounded is constant sum} is over a compact set
independent of $s$. Hence $B_t$ is an entire function of $s$.
Moreover because \eqref{lemma:bounded is constant sum absvalue} is
bounded by $C$, in particular $|B_t(s)|\leq C$ for all $s$. Hence
$B_t(s)=c_t$ is a constant. The series $\sum_{t=N}^{\infty}c_t$ is
absolutely convergent because it is majorized by
\eqref{lemma:bounded is constant sum absvalue}. We conclude that
$\sum_{t=N}^{\infty}B_t(s)$ is a constant, whence $\Psi(W,f_s,s)$ is
a constant, independent of $s$.

The proof in the \quasisplit\ case is almost identical. There is no
need to partition $G_1$ and we replace $diag(a,\lfloor
x\rfloor,I_{\bigidx-\smallidx})$ with
$diag(a,I_{\bigidx-\smallidx+1})$.

The case $\smallidx>\bigidx$ is similar, we describe it briefly. Put
$W_1=w^{\smallidx,\bigidx}\cdot W$. Using the Iwasawa decomposition
of $H_{\bigidx}$ and the argument regarding the support of $W(ar)$
for $r\in R^{\smallidx,\bigidx}$ (see
Proposition~\ref{proposition:iwasawa decomposition of the
integral}), we have for some $r_1,\ldots,r_{v}\in
R^{\smallidx,\bigidx}$ and $k_1,\ldots,k_m\in K_{H_{\bigidx}}$,
\begin{align*}
C&>\sum_{i,j}\int_{A_{\bigidx}}|r_jk_i \cdot W_1
|(\rconj{(w^{\smallidx,\bigidx})^{-1}}a)|k_i\cdot
f_s|(a,1)\delta(a)da
\end{align*}
($\delta=\delta_{B_{H_{\bigidx}}}^{-1}$). For each $i$, $k_i\cdot
f_s=\sum_{\alpha}P_{\alpha}(q^{-s},q^s)f_s^{(\alpha)}$ where
$f_s^{(\alpha)}\in\xi(\rho,std,s)$ ($P_{\alpha},f_s^{(\alpha)}$
depend on $i$). Then we get
\begin{align*}
\sum_{i,j}\int_{A_{\bigidx}}|r_jk_i\cdot
W_1|(\rconj{(w^{\smallidx,\bigidx})^{-1}}a)\Big|
\sum_{\alpha}P_{\alpha}(q^{-s},q^s)W_{\alpha}(a)\Big|\delta(a)\absdet{a}^{M'+\Re(s)}da.
\end{align*}
Here $M'$ is some real constant. In addition,
\begin{align*}
\Psi(W,f_s,s)=\sum_{i,j}\int_{A_{\bigidx}}r_jk_i\cdot
W_1(\rconj{(w^{\smallidx,\bigidx})^{-1}}a)
\sum_{\alpha}P_{\alpha}(q^{-s},q^s)W_{\alpha}(a)\delta(a)\absdet{a}^{M'+s}da.
\end{align*}
The proof follows from this, using the fact that $r_jk_i\cdot W_1$
vanishes away from zero.
\end{proof} 

\begin{proof}[Proof of Claim~\ref{claim:n_1 < l < n_2 functional eq pi and tau_2}] 
Recall that \eqref{int:mult l<n_2 before applying eq for inner tau}
and \eqref{int:mult l<n_2 after applying eq for inner tau} are
\begin{align*}
&\int_{U_{\bigidx_1}}
\int_{\lmodulo{U_{G_{\smallidx}}}{G_{\smallidx}}}W(g)
\int_{R_{\smallidx,\bigidx_2}}
\varphi_s(w'u,1,w_{\smallidx,\bigidx_2}r(\rconj{b_{\bigidx_2,\bigidx_1}}g),1)
\psi_{\gamma}(r)\psi_{\gamma}(u)drdgdu,\\\notag
&\int_{U_{\bigidx_1}}
\int_{\lmodulo{U_{G_{\smallidx}}}{G_{\smallidx}}}W(g)
\int_{R_{\smallidx,\bigidx_2}}
\varphi_{s,1-s}(w'u,1,w_{\smallidx,\bigidx_2}r(\rconj{b_{\bigidx_2,\bigidx_1}}g),1)
\psi_{\gamma}(r)\psi_{\gamma}(u)drdgdu.
\end{align*}
We use the idea of \cite{Soudry} (p.68) and prove that
integrals~\eqref{int:mult l<n_2 before applying eq for inner tau}
and \eqref{int:mult l<n_2 after applying eq for inner tau} are
proportional using a specific substitution. Let $D$ (resp. $D^*$) be
the domain of absolute convergence of \eqref{int:mult l<n_2 before
applying eq for inner tau} (resp. \eqref{int:mult l<n_2 after
applying eq for inner tau}). These domains do not depend on the
actual functions $W,\varphi_s$ and $\varphi_{s,1-s}$. We recall that
$D^*$ takes the form $\{A_1\leq\Re(s)<<\Re(\zeta)\}$ (see
Section~\ref{subsection:n_1<l<=n_2}).

Replace $W,\varphi_s$ with $W,\varphi_s'$ given by
Proposition~\ref{proposition:integral can be made constant zeta phi
version} (for a fixed $\zeta$). The manipulations \eqref{int:mult
l<n_2 starting form}-\eqref{int:mult l<n_2 before applying eq for
inner tau} still apply and 
since integral~\eqref{int:mult l<n_2 starting
form} is absolutely convergent and equals $1$ for all $s$, the same
holds for \eqref{int:mult l<n_2 before applying eq for inner tau}.

Denote the inner $drdg$-integration in \eqref{int:mult l<n_2 before
applying eq for inner tau} (resp. \eqref{int:mult l<n_2 after
applying eq for inner tau}), defined for each $u\in U_{\bigidx_1}$,
by $I_u(s)$ (resp. $\widetilde{I}_u(s)$). Integral $I_u(s)$ is
actually
$\Psi(W,\varphi(w'u,1,\cdot,\cdot)^{b_{\bigidx_2,\bigidx_1}},s-\zeta)$,
where
$\varphi(w'u,1,\cdot,\cdot)\in\xi_{Q_{\bigidx_2}}^{H_{\bigidx_2}}(\tau_2,hol,s-\zeta)$
(defined in Section~\ref{subsection:the multiplicativity of the
intertwining operator for tau induced}) and
\begin{align*}
\varphi(w'u,1,h_2,b_2)^{b_{\bigidx_2,\bigidx_1}}=\varphi(w'u,1,\rconj{b_{\bigidx_2,\bigidx_1}}h_2,b_2).
\end{align*}
Integral $\widetilde{I}_u(s)$ is
\begin{align*}
&\Psi(W,(\nintertwiningfull{\varepsilon_2}{s}\varphi(w'u,1,\cdot,\cdot))^{b_{\bigidx_2,\bigidx_1}},1-s+\zeta)\\\notag
&=\Psi(W,\nintertwiningfull{\varepsilon_2}{s}\varphi(w'u,1,\cdot,\cdot)^{b_{\bigidx_2,\bigidx_1}},1-s+\zeta),
\end{align*}
where the equality follows from Section~\ref{subsection:Twisting the
embedding}. These are double integrals over $G_{\smallidx}\times
R_{\smallidx,\bigidx_2}$ and since $\zeta$ is fixed,
$\xi_{Q_{\bigidx_2}}^{H_{\bigidx_2}}(\tau_2,hol,s-\zeta)=\C[q^{\mp
s}]\otimes\xi_{Q_{\bigidx_2}}^{H_{\bigidx_2}}(\tau_2,std,s-\zeta)$.
The integral $I_u(s)$ is absolutely convergent for all $s\in\C$ and
$\widetilde{I}_u(s)$ is absolutely convergent in $D^*$, because
$\varphi_s,\varphi_{s,1-s}$ and $\psi_{\gamma}$ are smooth.

Denote by $Q_u\in\C(q^{-s})$ the meromorphic continuation of
$I_u(s)$, obtained as described in
Section~\ref{subsection:meromorphic continuation}. It satisfies
$I_u(s)=Q_u(q^{-s})$ whenever $\Re(s)>s_0$, for some $s_0>0$.
Similarly let $\widetilde{Q}_u\in\C(q^{-s})$ be the meromorphic
continuation of $\widetilde{I}_u(s)$. Then
$\widetilde{I}_u(s)=\widetilde{Q}_u(q^{-s})$ for all
$s\in\Omega=\setof{s\in\C}{\Re(s-\zeta)<s_1}$, for some $s_1<0$. We
may assume (by a priori taking $\zeta$ large enough so that
$\Re(s-\zeta)<<0$) that $D^*\subset\Omega$, hence
$\widetilde{I}_u(s)=\widetilde{Q}_u(q^{-s})$ in $D^*$.

Next we claim $I_u(s)=Q_u(q^{-s})$ also in $D^*$. Indeed, $I_u(s)$
satisfies the condition of Lemma~\ref{lemma:uniformly bounded
integral is constant}, because the constant $C$ of
Proposition~\ref{proposition:integral can be made constant zeta phi
version} is independent of $s$ and moreover, the functions
$\varphi_s,\varphi_{s,1-s}$ and $\psi_{\gamma}$ are smooth. Then the
lemma implies that $I_u(s)$ equals a constant $C'$ independent of
$s$. Hence also $Q_u(q^{-s})=C'$ and $I_u(s)=Q_u(q^{-s})$ for all
complex $s$.

According to \eqref{eq:gamma def}, in $\C(q^{-s})$ we have
$\gamma(\pi\times\varepsilon_2,\psi,s)Q_u=\widetilde{Q}_u$ (here
$c(\smallidx,\varepsilon_2,\gamma,s)=1$ because
$\smallidx\leq\bigidx_2$). Since the poles of
$\gamma(\pi\times\varepsilon_2,\psi,s)$ (in $q^{-s}$) belong to a
finite set, we can further assume that
$\gamma(\pi\times\varepsilon_2,\psi,s)$ does not have a pole in
$D^*$.

Thus for all $s\in D^*$,
\begin{align}\label{eq:special chain}
\gamma(\pi\times\varepsilon_2,\psi,s)I_u(s)=\gamma(\pi\times\varepsilon_2,\psi,s)Q_u(q^{-s})=
\widetilde{Q}_u(q^{-s})=\widetilde{I}_u(s).
\end{align}
As a meromorphic function \eqref{int:mult l<n_2 before applying eq
for inner tau} equals $1$ also for $s\in D^*$ hence we can consider
integral~\eqref{int:mult l<n_2 before applying eq for inner tau} in
$D^*$ (instead of $D$). 
For $s\in D^*$, by Fubini's Theorem and \eqref{eq:special chain} we
can replace the inner $drdg$-integration in \eqref{int:mult l<n_2
before applying eq for inner tau}, multiplied by
$\gamma(\pi\times\varepsilon_2,\psi,s)$, with the inner
$drdg$-integration of \eqref{int:mult l<n_2 after applying eq for
inner tau} and obtain (again using Fubini's Theorem) \eqref{int:mult
l<n_2 after applying eq for inner tau}. We conclude that the
proportionality factor between \eqref{int:mult l<n_2 before applying
eq for inner tau} and \eqref{int:mult l<n_2 after applying eq for
inner tau} is $\gamma(\pi\times\varepsilon_2,\psi,s)$.
\end{proof} 

\begin{proof}[Proof of Claim~\ref{claim:n_1 < l < n_2 shahidi equation}] 
Recall that \eqref{int:mult l<n_2 after eq inner tau starting form}
and \eqref{int:mult l<n_2 after shahidi equation} are
\begin{align*}
&\int_{\lmodulo{U_{G_{\smallidx}}}{G_{\smallidx}}}W(g)
\int_{R_{\smallidx,\bigidx}}
\int_{Z_{\bigidx_2,\bigidx_1}}\varphi_{s,1-s}(\omega_{\bigidx_1,\bigidx_2}zw_{\smallidx,\bigidx}rg,1,1,1)
\psi^{-1}(z)\psi_{\gamma}(r)dzdrdg,\\\notag
&\int_{\lmodulo{U_{G_{\smallidx}}}{G_{\smallidx}}}W(g)
\int_{R_{\smallidx,\bigidx}} \int_{Z_{\bigidx_1,\bigidx_2}}
\varphi_{1-s,s}'(\omega_{\bigidx_2,\bigidx_1}mw_{\smallidx,\bigidx}rg,1,1,1)
\psi^{-1}(m)\psi_{\gamma}(r)dmdrdg.
\end{align*}
In order to find the proportionality factor we may replace
$\varphi_{s,1-s}$, which depends on
$\nintertwiningfull{\varepsilon_2}{s}$, in \eqref{int:mult l<n_2
after eq inner tau starting form} with an arbitrary element
$\theta_{s,1-s}\in\xi(\varepsilon_1\otimes\varepsilon_2^*,rat,(s,1-s))$
and then $\varphi_{1-s,s}'$ is replaced with
$\nintertwiningfull{\varepsilon_1\otimes\varepsilon_2^*}{(s,1-s)}\theta_{s,1-s}$.
This is because both integrals are considered (by meromorphic
continuation) as bilinear forms on
$\Whittaker{\pi}{\psi_{\gamma}^{-1}}\times
V'(\varepsilon_1\otimes\varepsilon_2^*,(s,1-s))$, hence the
proportionality factor can be calculated using any element of
$\Whittaker{\pi}{\psi_{\gamma}^{-1}}\times
V'(\varepsilon_1\otimes\varepsilon_2^*,(s,1-s))$. 
We can also apply 
Proposition~\ref{proposition:integral can be made constant zeta phi
version} to the representation
$\cinduced{Q_{\bigidx_1}}{H_{\bigidx}}{(\varepsilon_1\otimes\cinduced{Q_{\bigidx_2}}{H_{\bigidx_2}}{\varepsilon_2^*\alpha^{1-s}})\alpha^s}$ 
and obtain $W$ and
$\theta_{s,1-s}\in\xi(\varepsilon_1\otimes\varepsilon_2^*,hol,(s,1-s))$
for which \eqref{int:mult l<n_2 after eq inner tau starting form} is
absolutely convergent and equals $1$, for all $s$.

Now proceed as in the proof of Claim~\ref{claim:n_1 < l < n_2
functional eq pi and tau_2}. Note that the inner $dz$ and
$dm$-integrations above form both sides of Shahidi's functional
equation \eqref{eq:shahidi func eq} and the proportionality factor
between them is $1$. The smoothness of $W,\theta_{s,1-s}$ and
$\nintertwiningfull{\varepsilon_1\otimes\varepsilon_2^*}{(s,1-s)}\theta_{s,1-s}$
implies that these inner integrations are absolutely convergent for
all $g\in G_{\smallidx}$ such that $W(g)\ne0$ and $r\in
R_{\smallidx,\bigidx}$.
\end{proof} 

\begin{proof}[Proof of Claim~\ref{claim:n_1 < l < n_2 properties of F}] 
Let $R'=\rconj{w_{\smallidx,\bigidx}}\Mat{\bigidx_1\times
\bigidx-\smallidx}\ltimes R_{\smallidx,\bigidx}$. It is the subgroup
of $U_{\bigidx-\smallidx}$ of elements whose first
$\bigidx-\smallidx$ rows take the form
\begin{align*}
\left(\begin{array}{ccccccc}I_{\bigidx-\smallidx}&x&y&0_{\smallidx-\bigidx_1}&m&z\\\end{array}\right),
\end{align*}
where $x\in\Mat{\bigidx-\smallidx\times\smallidx}$,
$y\in\Mat{\bigidx-\smallidx\times1}$, $0_{\smallidx-\bigidx_1}$ is
the zero matrix of
$\Mat{(\bigidx-\smallidx)\times(\smallidx-\bigidx_1)}$ and
$m\in\Mat{\bigidx-\smallidx\times\bigidx_1}$. Define $dr'=d\eta dr$.
First we show that
$\rconj{(w^{\smallidx,\bigidx_1})^{-1}}H_{\bigidx_1}$ normalizes
$R'$ and fixes $dr'$ and $\psi_{\gamma}$, by using the Bruhat
decomposition
$H_{\bigidx_1}=\overline{Q_{\bigidx_1}}W\overline{Q_{\bigidx_1}}$,
where $W\isomorphic\Integers_2^{\bigidx_1}$ is a subgroup of the
Weyl group of $H_{\bigidx_1}$. Let $x\in\GL{\bigidx_1}\isomorphic
M_{\bigidx_1}$. Since $H_{\bigidx_1}$ is considered as a subgroup of
$G_{\smallidx}$, to write
$x'=\rconj{(w^{\smallidx,\bigidx_1})^{-1}}x$ as an element of
$H_{\bigidx}$ we pass through the embedding
$H_{\bigidx_1}<G_{\smallidx}$. If
\begin{align*}
&[x]_{\mathcal{E}_{G_{\smallidx}}}=\left(\begin{array}{ccccc}
I_{\smallidx-\bigidx_1-1}&&&\\
&x\\
&&I_2&\\
&&&x^*\\
&&&&I_{\smallidx-\bigidx_1-1}\\\end{array}\right),\\
&[x']_{\mathcal{E}_{H_{\bigidx}}}=\left(\begin{array}{ccccccc}
I_{\bigidx-\smallidx}\\
&x\\
&&I_{\smallidx-\bigidx_1}\\
&&&1\\
&&&&I_{\smallidx-\bigidx_1}\\
&&&&&x^*\\
&&&&&&I_{\bigidx-\smallidx}\\\end{array}\right).
\end{align*}
(See Sections~\ref{subsection:H_n in G_l} and \ref{subsection:G_l in H_n}.) Evidently $x'$ normalizes $R'$, fixes $dr'$ and $\psi_{\gamma}$. For
$u\in \overline{U_{\bigidx_1}}$ the computation varies if
$G_{\smallidx}$ is split or \quasisplit, but the coordinates
relative to $\mathcal{E}_{H_{\bigidx}}$ are the same. Set
$u'=\rconj{(w^{\smallidx,\bigidx_1})^{-1}}u$. If
\begin{align*}
[u]_{\mathcal{E}_{H_{\bigidx_1}}}=\left(\begin{array}{ccc}I_{\bigidx_1}&&\\u_1&1&\\u_2&u_1'&I_{\bigidx_1}\\\end{array}\right),
\end{align*}
\begin{align*}
[u']_{\mathcal{E}_{H_{\bigidx}}}=\left(\begin{array}{ccccccc}I_{\bigidx-\smallidx}\\&I_{\bigidx_1}&&&&\\
&&I_{\smallidx-\bigidx_1}\\&u_1&&1&&\\&&&&I_{\smallidx-\bigidx_1}\\&u_2&&u_1'&&I_{\bigidx_1}\\&&&&&&I_{\bigidx-\smallidx}\\\end{array}\right).
\end{align*}
Hence $u'$ preserves $R'$, $dr'$ and $\psi_{\gamma}$ as claimed.
Regarding $w\in W$, note that $w$ is a product of transpositions
$w_i\in W$ satisfying 
\begin{align*}
&[\rconj{(w^{\smallidx,\bigidx_1})^{-1}}w_i]_{\mathcal{E}_{H_{\bigidx}}}=\left(\begin{array}{ccccccc}
I_{\bigidx-\smallidx+r_1}\\
&0&&&&1\\
&&I_{r_2+\smallidx-\bigidx_1}\\
&&&-1\\
&&&&I_{r_2+\smallidx-\bigidx_1}\\
&1&&&&0\\
&&&&&&I_{\bigidx-\smallidx+r_1}\end{array}\right),
\end{align*}
with $r_1,r_2\geq0$ such that $r_1+r_2=\bigidx_1-1$. Such elements have the required properties. 
These computations show that integral~\eqref{integral:mult l<n_2
integral before defining F} equals
\begin{align*}
&\int_{\lmodulo{H_{\bigidx_1}Z_{\smallidx-\bigidx_1-1}V_{\smallidx-\bigidx_1-1}}{G_{\smallidx}}}\int_{\lmodulo{U_{H_{\bigidx_1}}}{H_{\bigidx_1}}}
(\int_{R^{\smallidx,\bigidx_1}}W(r'w^{\smallidx,\bigidx_1}h'g)dr')\\\notag
&\int_{R_{\smallidx,\bigidx}}\int_{\Mat{\bigidx_1\times
\bigidx-\smallidx}}
\varphi_{1-s,s}'(\omega_{\bigidx_2,\bigidx_1}w_{\smallidx,\bigidx}w^{\smallidx,\bigidx_1}h'
(\rconj{\rconj{w^{\smallidx,\bigidx_1}}}(\rconj{w_{\smallidx,\bigidx}}\eta
r))g,1,1,1)\psi_{\gamma}(r)d\eta drdh'dg.
\end{align*}
In order to get the required form of the integral we need to slide
$h'$ to the third argument of $\varphi_{1-s,s}'$.
Set
$w=\omega_{\bigidx_2,\bigidx_1}w_{\smallidx,\bigidx}w^{\smallidx,\bigidx_1}$.
Since we can write the $dh'$-integration over
$\overline{B_{H_{\bigidx_1}}}$, we can assume
$h'\in\overline{B_{H_{\bigidx_1}}}$. Then it is left to show that
for all $h\in H_{\bigidx}$, $x\in M_{\bigidx_1}<H_{\bigidx_1}$ and
$u\in\overline{U_{\bigidx_1}}<H_{\bigidx_1}$,
\begin{align*}
&\varphi_{1-s,s}'(wxh,1,1,1)=\varphi_{1-s,s}'(wh,1,x,1),\\\notag
&\varphi_{1-s,s}'(wuh,1,x,1)=\varphi_{1-s,s}'(wh,1,x(\rconj{t_{\bigidx_1}}u),1).
\end{align*}
Since $t_{\bigidx_1}$ commutes with $x$,
$x(\rconj{t_{\bigidx_1}}u)=\rconj{t_{\bigidx_1}}(xu)$. Writing
$h'=xu$ yields the requested form of the integral.

We compute $\rconj{w^{-1}}x$. With the same notation as above,
\begin{align*}
\rconj{w^{-1}}x=\rconj{\omega_{\bigidx_2,\bigidx_1}^{-1}}(\rconj{w_{\smallidx,\bigidx}^{-1}}x')=
\left(\begin{array}{ccccc}I_{\bigidx_2}\\&x\\&&1\\&&&x^*\\&&&&I_{\bigidx_2}\\\end{array}\right).
\end{align*}
We see that $\rconj{w^{-1}}x$ belongs to $H_{\bigidx_1}$ directly
embedded in $H_{\bigidx}$ (without going through $G_{\smallidx}$),
i.e., embedded in the middle block of $Q_{\bigidx_2}$ - the standard
parabolic subgroup of $H_{\bigidx}$ whose Levi part is isomorphic to
$\GL{\bigidx_2}\times H_{\bigidx_1}$ . Hence
$\varphi_{1-s,s}'(wxh,1,1,1)=\varphi_{1-s,s}'(wh,1,x,1)$.

For $u$,
\begin{align*}
\rconj{w^{-1}}u=\rconj{\omega_{\bigidx_2,\bigidx_1}^{-1}}(\rconj{w_{\smallidx,\bigidx}^{-1}}u')=
\left(\begin{array}{ccccc}I_{\bigidx_2}\\&I_{\bigidx_1}&&\\&\gamma^{-1}(-1)^{\bigidx-\smallidx}
u_1&1&\\&\gamma^{-2}u_2&\gamma^{-1}(-1)^{\bigidx-\smallidx}
u_1'&I_{\bigidx_1}\\&&&&I_{\bigidx_2}\\\end{array}\right).
\end{align*}
Again $\rconj{w^{-1}}u\in H_{\bigidx_1}<H_{\bigidx}$,
$\varphi_{1-s,s}'(wuh,1,x,1)=\varphi_{1-s,s}'(wh,1,x(\rconj{t_{\bigidx_1}}u),1)$.
\end{proof} 

\begin{proof}[Proof of Claim~\ref{claim:n_1 < l < n_2 functional eq pi and tau_1}] 
As in Claim~\ref{claim:n_1 < l < n_2 shahidi equation} we can
replace $\varphi_{1-s,s}'$, which depends on
$\nintertwiningfull{\varepsilon_2}{s}$ and
$\nintertwiningfull{\varepsilon_1\otimes\varepsilon_2^*}{(s,1-s)}$,
with an arbitrary element of
$\xi(\varepsilon_2^*\otimes\varepsilon_1,rat,(1-s,s))$. The
arguments of Claim~\ref{claim:n_1 < l < n_2 functional eq pi and
tau_2} apply here as well. Proposition~\ref{proposition:integral can
be made constant zeta phi version} is used to obtain $W$ and
$\theta_{1-s,s}\in\xi(\varepsilon_2^*\otimes\varepsilon_1,hol,(1-s,s))$.

The inner $dr'dh'$-integrations will be absolutely convergent for
all $g$, $r$ and $\eta$. The proportionality factor between these
integrations needs to be recalculated, because of the conjugation by
$t_{\bigidx_1}$. In general for
$W\in\Whittaker{\pi}{\psi_{\gamma}^{-1}}$ and
$f_s'\in\xi(\varepsilon_1,rat,s)$, according to \eqref{eq:gamma def}
and \eqref{eq:intertwiner and left and right translation},
\begin{align*}
\gamma(\pi\times\varepsilon_1,\psi,s)\Psi(W,(f_s')^{t_{\bigidx_1}},s)&=c(\smallidx,\varepsilon_1,\gamma,s)\Psi(W,\nintertwiningfull{\varepsilon_1}{s}(f_s')^{t_{\bigidx_1}},1-s)\\\notag
&=\Psi(W,(\nintertwiningfull{\varepsilon_1}{s}f_s')^{t_{\bigidx_1}},1-s).
\end{align*}
Thus the inner $dr'dh'$-integrations are proportional by
$\gamma(\pi\times\varepsilon_1,\psi,s)$.
\end{proof}  

\section{Multiplicativity in the first variable}\label{section:1st variable}

\subsection{Reduction of the proof}
Let $\tau$ be a representation of $\GL{\bigidx}$. In this section we
prove Theorem~\ref{theorem:multiplicity first var}: if $\sigma$ is a
representation of $\GL{k}$, $0<k\leq\smallidx$, and $\pi$ is a
quotient of $\cinduced{P_k}{G_{\smallidx}}{\sigma\otimes\pi'}$ or
$\cinduced{\rconj{\kappa}P_{\smallidx}}{G_{\smallidx}}{\sigma}$,
\begin{align}\label{eq:gamma first var}
\gamma(\pi\times\tau,\psi,s)=\omega_{\sigma}(-1)^{\bigidx}\omega_{\tau}(-1)^k[\omega_{\tau}(2\gamma)^{-1}]\gamma(\sigma\times\tau,\psi,s)\gamma(\pi'\times\tau,\psi,s)\gamma(\sigma^*\times\tau,\psi,s).
\end{align}
As in Section~\ref{section:2nd variable} we may
assume that $\pi$ is equal to the induced representation. 

The proof is reduced to two particular cases:
$k<\smallidx\leq\bigidx$ and $k=\smallidx>\bigidx$. First we explain
how these cases together with Theorem~\ref{theorem:multiplicity
second var} imply the general case.

Assume $k<\smallidx$ and $\smallidx>\bigidx$. Let $\tau_1$ be a
representation of $\GL{m}$ for $m>\smallidx$ such that
$\gamma(\pi\times\tau_1,\psi,s)\ne0$ (e.g. take an irreducible
$\tau_1$) and consider
$\rho=\cinduced{P_{\bigidx,m}}{\GL{\bigidx+m}}{\tau\otimes\tau_1}$.
For brevity, we drop $\psi$ and $s$ from the notation of the
$\gamma$-factors. According to Theorem~\ref{theorem:multiplicity
second var},
\begin{align}\label{eq:multi first var second var}
\gamma(\pi\times\rho)= \gamma(\pi\times\tau)\gamma(\pi\times\tau_1).
\end{align}
Since $k<\smallidx<m+\bigidx$, $\gamma(\pi\times\rho)$ is
multiplicative in the first variable. In addition the
$\gamma$-factors of $\GL{k}\times\GL{\bigidx+m}$ are multiplicative
(\cite{JPSS} Theorem~3.1). Hence
\begin{align}\label{eq:multi first var second var opening}
\gamma(\pi\times\rho)
=&\omega_{\sigma}(-1)^{\bigidx+m}\omega_{\tau}(-1)^k\omega_{\tau_1}(-1)^k\\\notag&\gamma(\sigma\times\tau)\gamma(\sigma\times\tau_1)
\gamma(\pi'\times\rho)
\gamma(\sigma^*\times\tau)\gamma(\sigma^*\times\tau_1).
\end{align}
Also $\gamma(\pi'\times\rho)$ is multiplicative in $\rho$  and
because $k<\smallidx<m$, $\gamma(\pi\times\tau_1)$ is multiplicative
in $\pi$. Therefore
\begin{align*}
&\gamma(\pi'\times\rho)=
\gamma(\pi'\times\tau)\gamma(\pi'\times\tau_1),\\
&\gamma(\pi\times\tau_1)=
\omega_{\sigma}(-1)^m\omega_{\tau_1}(-1)^k\gamma(\sigma\times\tau_1)\gamma(\pi'\times\tau_1)\gamma(\sigma^*\times\tau_1).
\end{align*}
Applying these equalities to \eqref{eq:multi first var second var
opening} we get
\begin{align}\label{eq:multi first var second var lhs}
\gamma(\pi\times\rho)
=\omega_{\sigma}(-1)^{\bigidx}\omega_{\tau}(-1)^k\gamma(\sigma\times\tau)
\gamma(\pi'\times\tau)\gamma(\pi\times\tau_1)
\gamma(\sigma^*\times\tau).
\end{align}
Comparing \eqref{eq:multi first var second var} and \eqref{eq:multi
first var second var lhs},
\begin{align*}
\gamma(\pi\times\tau)=
\omega_{\sigma}(-1)^{\bigidx}\omega_{\tau}(-1)^k\gamma(\sigma\times\tau)\gamma(\pi'\times\tau)
\gamma(\sigma^*\times\tau).
\end{align*}

The last case is when $k=\smallidx\leq\bigidx$ and $\pi$ is induced
from either $P_{\smallidx}$ or $\rconj{\kappa}P_{\smallidx}$. Since
the argument is identical for either parabolic, assume
$\pi=\cinduced{P_{\smallidx}}{G_{\smallidx}}{\sigma}$. Let
$\sigma_1$ be a representation of $\GL{m}$ with $m>\bigidx$. Define
$\rho=\cinduced{P_{m+\smallidx}}{G_{m+\smallidx}}{\cinduced{P_{m,\smallidx}}{\GL{m+\smallidx}}{\sigma_1\otimes\sigma}}$.
Since
$\rho\isomorphic\cinduced{P_m}{G_{m+\smallidx}}{\sigma_1\otimes\pi}$
and $m<m+\smallidx$,
\begin{align}\label{eq:multi first var k=l}
\gamma(\rho\times\tau)=&
\omega_{\sigma_1}(-1)^{\bigidx}\omega_{\tau}(-1)^{m}
\gamma(\sigma_1\times\tau)\gamma(\pi\times\tau)\gamma(\sigma_1^*\times\tau).
\end{align}
Because $m+\smallidx>\bigidx$,
\begin{align*}
\gamma(\rho\times\tau)=&
\omega_{\sigma_1}(-1)^{\bigidx}
\omega_{\sigma}(-1)^{\bigidx}\omega_{\tau}(-1)^{m+\smallidx}\omega_{\tau}(2\gamma)^{-1}\\\notag&
\gamma(\cinduced{P_{m,\smallidx}}{\GL{m+\smallidx}}{\sigma_1\otimes\sigma}\times\tau)\gamma((\cinduced{P_{m,\smallidx}}{\GL{m+\smallidx}}{\sigma_1\otimes\sigma})^*\times\tau).
\end{align*}
Combining this with \eqref{eq:multi first var k=l} yields
\begin{align*}
\gamma(\pi\times\tau)=
\omega_{\sigma}(-1)^{\bigidx}\omega_{\tau}(-1)^{\smallidx}\omega_{\tau}(2\gamma)^{-1}\gamma(\sigma\times\tau)\gamma(\sigma^*\times\tau).
\end{align*}

\subsection{The basic identity}\label{subsection:The basic identity pi}
Let $W\in\Whittaker{\pi}{\psi_{\gamma}^{-1}}$ and
$f_s\in\xi(\tau,hol,s)$. We use the realization of $\pi$ described
in Section~\ref{subsection:Realization of pi for induced
representation}: assume
$\pi=\cinduced{P}{G_{\smallidx}}{\sigma\otimes\pi'}$ where $P$ is
either $\overline{P_k}$ or
$\rconj{\kappa}(\overline{P_{\smallidx}})$, $\sigma$ is realized in
$\Whittaker{\sigma}{\psi^{-1}}$ and $\pi'$ (if $k<\smallidx$) is
realized in $\Whittaker{\pi'}{\psi_{\gamma}^{-1}}$, and select
$\varphi_{\zeta}\in\xi_{P}^{G_{\smallidx}}(\sigma\otimes\pi',std,-\zeta+\half)$
for $W$ (i.e., $W_{\varphi_0}=W$).

According to Claim~\ref{claim:meromorphic continuation for varphi
instead of W in psi},
$\Psi(\varphi_{\zeta},f_s,s)=\Psi(W_{\varphi_{\zeta}},f_s,s)$ in
$\C(q^{-s})$. This claim can be easily adapted to showing
$\Psi(\varphi_{\zeta},\nintertwiningfull{\tau}{s}f_s,1-s)=\Psi(W_{\varphi_{\zeta}},\nintertwiningfull{\tau}{s}f_s,1-s)$,
in $\C(q^{-s})$. 

In Sections~\ref{subsection:1st
var l<n}-\ref{subsection:1st var k=l>n+1} we prove for a fixed
$\zeta$ with $\Re(\zeta)>>0$ that in $\C(q^{-s})$,
\begin{align}\label{eq:gamma first var for zeta large with integrals}
&\omega_{\sigma}(-1)^{\bigidx}\omega_{\tau}(-1)^k\gamma(\sigma\times\tau,\psi,s-\zeta)\gamma(\pi'\times\tau,\psi,s)\gamma(\sigma^*\times\tau,\psi,s+\zeta)\\\notag&
c(\smallidx,\tau,\gamma,s)^{-1}[\omega_{\tau}(2\gamma)^{-1}]
\Psi(\varphi_{\zeta},f_s,s)=\Psi(\varphi_{\zeta},\nintertwiningfull{\tau}{s}f_s,1-s),
\end{align}
where the factor $\omega_{\tau}(2\gamma)^{-1}$ appears only if
$k=\smallidx$. 
Analogously to Claim~\ref{claim:gamma is rational in s and zeta},
\begin{align*}
\gamma(\cinduced{P}{G_{\smallidx}}{(\sigma\otimes\pi')\alpha^{-\zeta+\half}}\times\tau,\psi,s)\in\C(q^{-\zeta},q^{-s})
\end{align*}
(the proof is actually simpler, since the twist by $\zeta$ does not
involve the intertwining operator and the sections related to
$\tau$). Arguing as in Section~\ref{subsection:The basic identity
tau}, equality~\eqref{eq:gamma first var} follows by letting
$\zeta\rightarrow0$. 
\subsection{The case $k<\smallidx\leq\bigidx$}\label{subsection:1st var l<n}
The proof follows the arguments of Soudry \cite{Soudry}
(Theorem~11.4). We apply the functional equations for
$\sigma\times\tau$, $\pi'\times\tau$ and $\sigma\times\tau^*$.
Manipulations of integrals with possibly disjoint domains of
absolute convergence are explained using meromorphic continuations
and uniqueness properties as in Section~\ref{subsection:n_1<l<=n_2}.
Proofs of convergence can be derived using the techniques
of \cite{Soudry} (Section~11). 
The domains of convergence change from
$\{\Re(\zeta)<<\Re(s)<<(1+C_1)\Re(\zeta)\}$, where $C_1>0$, to
$\{(1-C_2)\Re(\zeta)<<\Re(s)<<\Re(\zeta)\}$ with $1>C_2>0$, then
$\{(1-C_2)\Re(\zeta)<<\Re(1-s)<<\Re(\zeta)\}$ and finally
$\{\Re(\zeta)<<\Re(1-s)<<(1+C_1)\Re(\zeta)\}$. The constants $C_1$
and $C_2$ depend only on the representations.

Our starting point is
\begin{align}\label{eq:1st var l<n start form}
&\Psi(\varphi_{\zeta},f_s,s)=\int_{\lmodulo{U_{G_{\smallidx}}}{G_{\smallidx}}}
(\int_{V_k}\varphi_{\zeta}(vg,1,1)\psi_{\gamma}(v)dv)\int_{R_{\smallidx,\bigidx}}
f_s(w_{\smallidx,\bigidx}rg,1)\psi_{\gamma}(r)drdg.
\end{align}
Here we used \eqref{eq:whittaker for pi, k<l} for
$W_{\varphi_{\zeta}}$. Since $U_{G_{\smallidx}}=(V_k\rtimes
Z_k)\rtimes U_{G_{\smallidx-k}}$, the integral becomes
\begin{align*}
\int_{\lmodulo{Z_k U_{G_{\smallidx-k}}}{G_{\smallidx}}}
\varphi_{\zeta}(g,1,1)
\int_{R_{\smallidx,\bigidx}}f_s(w_{\smallidx,\bigidx}rg,1)\psi_{\gamma}(r)drdg.
\end{align*}

We have the following integration formula, derived using
$G_{\smallidx}=K_{G_{\smallidx}}P^{\star}$ where
$P^{\star}=(\GL{k}\times B_{G_{\smallidx-k}})\ltimes\overline{V_k}$
is a parabolic subgroup.
\begin{align}\label{int:formula multiple integrations}
\int_{\lmodulo{Z_k U_{G_{\smallidx-k}}}{G_{\smallidx}}}F(g)dg=
\int_{K_{G_{\smallidx}}}\int_{T_{G_{\smallidx-k}}}\int_{\overline{V_k''}}
\int_{\lmodulo{Z_k}{\GL{k}}}\int_{\overline{V_k'}}F(amvtk)dmdadvdtdk,
\end{align}
where
\begin{align*}
\overline{V_k'}&=\{\left(\begin{array}{ccccc}I_k\\m&I_{\smallidx-1-k}\\&&I_2\\&&&I_{\smallidx-1-k}\\&&&m'&I_k\end{array}\right)\},\\
\overline{V_k''}&=\{\left(\begin{array}{cccccc}I_k&\\0&I_{\smallidx-1-k}\\v_1&0&1&&&\\v_2&0&0&1&&\\v&0&0&0&I_{\smallidx-1-k}&\\z&v'&*&*&0&I_k\\\end{array}\right)\}.
\end{align*}
Using this formula, 
the integral equals
\begin{align}\label{int:1st var l<n before taking a out}
&\int_{K_{G_{\smallidx}}}\int_{T_{G_{\smallidx-k}}}\int_{\overline{V_k''}}
\int_{\lmodulo{Z_k}{\GL{k}}}\int_{\overline{V_k'}}
\varphi_{\zeta}(amvtk,1,1)\\\notag
&\int_{R_{\smallidx,\bigidx}}f_s(w_{\smallidx,\bigidx}ramvtk,1)\psi_{\gamma}(r)dr
dmdadvdtdk.
\end{align}

For $a\in\GL{k}$, $\delta_{P_k}(a)=|\det{a}|^{2\smallidx-k-1}$,
hence
\begin{align*}
\varphi_{\zeta}(a,1,1)=\absdet{a}^{-\smallidx+\frac{k+1}2-\zeta}\varphi_{\zeta}(1,a,1).
\end{align*}
Also according to \eqref{eq:omega right translation by GL l-1}, 
\begin{align*}
\int_{R_{\smallidx,\bigidx}}f_s(w_{\smallidx,\bigidx}ram,1)\psi_{\gamma}(r)dr=
\delta_{Q_{\bigidx}}^{\half}(a)\absdet{a}^{\smallidx-\bigidx+s-\half}\int_{R_{\smallidx,\bigidx}}f_s(w_{\smallidx,\bigidx}r,am)\psi_{\gamma}(r)dr.
\end{align*}
\begin{remark}
The main obstacle in addressing the case $k=\smallidx\leq\bigidx$
here is that $\rconj{w_{\smallidx,\bigidx}^{-1}}L_{\smallidx}$ is
not contained in $Q_{\bigidx}$, so \eqref{eq:omega right translation
by GL l-1} is not applicable.
\end{remark}
Combining these observations \eqref{int:1st var l<n before taking a
out} becomes
\begin{align}\label{eq:1st var l<n before func eq sigma}
&\int_{K_{G_{\smallidx}}}\int_{T_{G_{\smallidx-k}}}\int_{\overline{V_k''}}
\int_{R_{\smallidx,\bigidx}} \int_{\lmodulo{Z_k}{\GL{k}}}
\int_{\Mat{\smallidx-1-k\times k}} \varphi_{\zeta}(vtk,a,1)\\\notag
&f_s(w_{\smallidx,\bigidx}rvtk,am)
\absdet{a}^{s-\zeta-\frac{\bigidx-k}2}\psi_{\gamma}(r)dmdadrdvdtdk.
\end{align}
Here we identified $\overline{V_k'}$ with $\Mat{\smallidx-1-k\times
k}$. The inner $dmda$-integration is a Rankin-Selberg integral for
$\GL{k}\times\GL{\bigidx}$ and $\sigma\times\tau$. By
\eqref{eq:modified func eq glm gln} with $j=\smallidx-1-k$ we
(formally) obtain
\begin{align}\label{eq:1st var l<n after func eq sigma}
&\int_{K_{G_{\smallidx}}}\int_{T_{G_{\smallidx-k}}}\int_{\overline{V_k''}}
\int_{R_{\smallidx,\bigidx}}
\int_{\lmodulo{Z_k}{\GL{k}}}\int_{\Mat{k\times \bigidx-\smallidx}}
\varphi_{\zeta}(vtk,a,1)\\\notag &f_s(w_{\smallidx,\bigidx}rvtk,
\left(\begin{array}{ccc}0&I_{\smallidx-k}&0\\0&0&I_{\bigidx-\smallidx}\\I_k&0&m\end{array}\right)
a)\absdet{a}^{s-\zeta-\frac{\bigidx+k}2+\smallidx-1}\psi_{\gamma}(r)dmdadrdvdtdk.\notag
\end{align}
This integral is absolutely convergent in the sense that it is convergent if we replace
$\varphi_{\zeta},f_s$ with $|\varphi_{\zeta}|,|f_s|$ and drop $\psi_{\gamma}$.
In its domain of absolute convergence \eqref{eq:1st var l<n after
func eq sigma} satisfies the same equivariance properties
\eqref{eq:bilinear special condition} as
$\Psi(\varphi_{\zeta},f_s,s)$. This is evident from \eqref{int:after
func sigma and tau gamma first var showing invariancy} below, which
is equal as an integral to \eqref{eq:1st var l<n after func eq
sigma}. It is also possible to show that there are functions
$\varphi_{\zeta}$ and $f_s$ for which \eqref{eq:1st var l<n after
func eq sigma} is absolutely convergent and equals $1$, for all $s$
(see the proof of Claim~\ref{claim:l<n functional eq pi prime}).
Then from Section~\ref{subsection:meromorphic continuation} it
follows that \eqref{eq:1st var l<n after func eq sigma} has a
meromorphic continuation to a function in $\C(q^{-s})$. Thus for all
but a finite set of $q^{-s}$, it is a bilinear form on
$V_{\overline{P_k}}^{G_{\smallidx}}(\sigma\otimes\pi',-\zeta+\half)\times
V(\tau,s)$ satisfying \eqref{eq:bilinear special condition}
(recall that $\zeta$ is fixed). This also holds for \eqref{eq:1st
var l<n before func eq sigma}, since it is equal as an integral to
$\Psi(\varphi_{\zeta},f_s,s)$ and because of
Claim~\ref{claim:meromorphic continuation for varphi instead of W in
psi}. Thus \eqref{eq:1st var l<n before func eq sigma} and
\eqref{eq:1st var l<n after func eq sigma} are proportional. In
Section~\ref{subsection:proofs claims 1st var} we prove,
\begin{claim}\label{claim:l<n functional eq sigma}
Integral~\eqref{eq:1st var l<n before func eq sigma} multiplied by
$\omega_{\sigma}(-1)^{\bigidx-1}\gamma(\sigma\times\tau,\psi,s-\zeta)$
equals integral~\eqref{eq:1st var l<n after func eq sigma}, as
meromorphic functions.
\end{claim}
As in Section~\ref{subsection:n_1<l<=n_2} (see the paragraph after
Claim~\ref{claim:n_1 < l < n_2 functional eq pi and tau_2}), we can
skip the independent proof of meromorphic continuation for
\eqref{eq:1st var l<n after func eq sigma} and regard the integral
\eqref{eq:1st var l<n after func eq sigma} and the meromorphic
continuation of \eqref{eq:1st var l<n before func eq sigma} as
bilinear forms on
$V_{\overline{P_k}}^{G_{\smallidx}}(\sigma\otimes\pi',-\zeta+\half)\times
V(\tau,s)$ only for $s$ in the domain of absolute convergence of
\eqref{eq:1st var l<n after func eq sigma}.

Reverse the steps \eqref{int:1st var l<n before taking a
out}-\eqref{eq:1st var l<n before func eq sigma} to get that
\eqref{eq:1st var l<n after func eq sigma} equals
\begin{align*}
&\int_{K_{G_{\smallidx}}}\int_{T_{G_{\smallidx-k}}}\int_{\overline{V_k''}}\int_{\lmodulo{Z_k}{\GL{k}}}
\int_{\Mat{k\times \bigidx-\smallidx}}\int_{R_{\smallidx,\bigidx}}\varphi_{\zeta}(avtk,1,1)\\
&f_s(w_{\smallidx,\bigidx}ravtk,\left(\begin{array}{ccc}0&I_{\smallidx-k}&0\\0&0&I_{\bigidx-\smallidx}\\I_k&0&m\end{array}\right))
\absdet{a}^{\smallidx-1-k}\psi_{\gamma}(r)drdmdadvdtdk.
\end{align*}

Changing the order $am\mapsto ma$ in \eqref{int:formula multiple
integrations} multiplies $dm$ by $\absdet{a}^{\smallidx-1-k}$, then
equality~\eqref{int:formula multiple integrations} implies the
formula
\begin{align*}
\int_{\lmodulo{\overline{V_k'}Z_k
U_{G_{\smallidx-k}}}{G_{\smallidx}}}F(g)dg=
\int_{K_{G_{\smallidx}}}\int_{T_{G_{\smallidx-k}}}\int_{\overline{V_k''}}
\int_{\lmodulo{Z_k}{\GL{k}}}F(avtk)\absdet{a}^{\smallidx-1-k}dadvdtdk.
\end{align*}
Here $\overline{V_k'}Z_k U_{G_{\smallidx-k}}<G_{\smallidx}$ because
$\overline{V_k'}(Z_kU_{G_{\smallidx-k}})=(Z_kU_{G_{\smallidx-k}})\overline{V_k'}$.

Let
\begin{align*}
\left(\begin{array}{ccc}0&I_{\smallidx-k}&0\\0&0&I_{\bigidx-\smallidx}\\I_k&0&m\end{array}\right)
=\left(\begin{array}{ccc}0&I_{\smallidx-k}&0\\0&0&I_{\bigidx-\smallidx}\\I_k&0&0\end{array}\right)
\left(\begin{array}{ccc}I_k&0&m\\0&I_{\smallidx-k}&0\\0&0&I_{\bigidx-\smallidx}\\\end{array}\right)
=\omega_{\bigidx-k,k}m.
\end{align*}
Since the function on $G_{\smallidx}$ given by
\begin{align*}
g\mapsto\int_{\Mat{k\times \bigidx-\smallidx}}
\int_{R_{\smallidx,\bigidx}}\varphi_{\zeta}(g,1,1)f_s(w_{\smallidx,\bigidx}rg,\omega_{\bigidx-k,k}m)\psi_{\gamma}(r)drdm
\end{align*}
is $\overline{V_k'}Z_kU_{G_{\smallidx-k}}$-invariant on the left
(invariancy by $\overline{V_k'}U_{G_{\smallidx-k}}$ holds even for
each fixed $m$),
we may apply the latter formula to the integral and derive
\begin{align}\label{int:after func sigma and tau gamma first var
showing invariancy}
&\int_{\lmodulo{\overline{V_k'}Z_kU_{G_{\smallidx-k}}}{G_{\smallidx}}}
\int_{\Mat{k\times \bigidx-\smallidx}} \int_{R_{\smallidx,\bigidx}}
\varphi_{\zeta}(g,1,1)
f_s(w_{\smallidx,\bigidx}rg,\omega_{\bigidx-k,k}m)\psi_{\gamma}(r)drdmdg.
\end{align}

Now our task is to prepare an inner integration for
$\pi'\times\tau$. Essentially it is shown how to mix
$\overline{V_k''}$ and $R_{\smallidx,\bigidx}$ to form
$R_{\smallidx-k,\bigidx}$.
Start with factoring the integral through $\overline{V_k''}$, noting
that $\overline{V_k'}\ltimes\overline{V_k''}=\overline{V_k}$ and
$\varphi_{\zeta}(vg,1,1)=\varphi_{\zeta}(g,1,1)$ for
$v\in\overline{V_k}$. This leads to
\begin{align}\label{eq:1st var l<n after eq for sigma and after int formula}
&\int_{\lmodulo{\overline{V_k}Z_kU_{G_{\smallidx-k}}}{G_{\smallidx}}}
\varphi_{\zeta}(g,1,1)
\int_{\overline{V_k''}} \int_{\Mat{k\times \bigidx-\smallidx}}
\int_{R_{\smallidx,\bigidx}}
f_s(w_{\smallidx,\bigidx}rvg,\omega_{\bigidx-k,k}m)
\psi_{\gamma}(r)drdmdvdg.
\end{align}
By the following claim it is possible to change the order of $m$ and
$rv$.
\begin{claim}\label{claim:l<n first var inner integral m r and v}
\begin{align*}
&\int_{\overline{V_k''}}
\int_{\Mat{k\times \bigidx-\smallidx}}
\int_{R_{\smallidx,\bigidx}}f_s(\omega_{\bigidx-k,k}mw_{\smallidx,\bigidx}rvg,1)\psi_{\gamma}(r)drdmdv\\
&=\int_{\overline{V_k''}} \int_{\Mat{k\times \bigidx-\smallidx}}
\int_{R_{\smallidx,\bigidx}}f_s(\omega_{\bigidx-k,k}w_{\smallidx,\bigidx}rv(\rconj{w_{\smallidx,\bigidx}}m)g,1)\psi_{\gamma}(r)drdmdv.
\end{align*}
\end{claim}

Write
$\omega_{\bigidx-k,k}w_{\smallidx,\bigidx}=w_{\smallidx-k,\bigidx}w$
with
\begin{align*}
w=\left(\begin{array}{ccccc}0&0&0&\gamma^{-1}I_k&0\\
I_{\bigidx-\smallidx}&0&0&0&0\\
0&0&b_{\smallidx-k,k}&0&0\\
0&0&0&0&I_{\bigidx-\smallidx}\\
0&\gamma I_k&0&0&0\end{array}\right)\in H_{\bigidx},
b_{\smallidx-k,k}=diag(I_{\smallidx-k},(-1)^k,I_{\smallidx-k}).
\end{align*}
The next claim shows how to combine the coordinates of $\overline{V_k''}$ and $R_{\smallidx,\bigidx}$.
\begin{claim}\label{claim:l<n first var inner integral r and v}
\begin{align*}
&\int_{\overline{V_k''}}
\int_{\Mat{k\times \bigidx-\smallidx}}
\int_{R_{\smallidx,\bigidx}}f_s(\omega_{\bigidx-k,k}w_{\smallidx,\bigidx}rv(\rconj{w_{\smallidx,\bigidx}}m)g,1)\psi_{\gamma}(r)drdmdv\\
&=\int_{\Mat{k\times \bigidx-\smallidx}}
\int_{R_{\smallidx-k,\bigidx}}f_s(w_{\smallidx-k,\bigidx}r'w(\rconj{w_{\smallidx,\bigidx}}m)g,1)\psi_{\gamma}(r')dr'dm.
\end{align*}
On the \rhs\ $\psi_{\gamma}$ denotes the character of
$N_{\bigidx-(\smallidx-k)}$ restricted to $R_{\smallidx-k,\bigidx}$.
\end{claim}

Using Claims~\ref{claim:l<n first var inner integral m r and v} and
\ref{claim:l<n first var inner integral r and v},
integral~\eqref{eq:1st var l<n after eq for sigma and after int
formula} becomes
\begin{align*}
&\int_{\lmodulo{\overline{V_k}Z_kU_{G_{\smallidx-k}}}{G_{\smallidx}}}
\int_{\Mat{k\times \bigidx-\smallidx}}
\int_{R_{\smallidx-k,\bigidx}}
\varphi_{\zeta}(g,1,1)
f_s(w_{\smallidx-k,\bigidx}r'w(\rconj{w_{\smallidx,\bigidx}}m)g,1)\psi_{\gamma}(r')dr'dmdg.
\end{align*}
This integral may be factored through $G_{\smallidx-k}$. For $g'\in
G_{\smallidx-k}$, $(\rconj{w_{\smallidx,\bigidx}}m)$ commutes with
$g'$ and $\rconj{w^{-1}}g'=\rconj{b_{\bigidx,k}}g'$ 
(where $g'$ is considered as an element of $H_{\bigidx}$, if we
consider $g'$ as an element of $H_{\smallidx-k}$ we get
$\rconj{b_{\smallidx-k,k}}g'$).
The integral equals
\begin{align}\label{eq:1st var l<n before eq for pi prime}
&\int_{\lmodulo{\overline{V_k}Z_kG_{\smallidx-k}}{G_{\smallidx}}}
\int_{\Mat{k\times \bigidx-\smallidx}}
\int_{\lmodulo{U_{G_{\smallidx-k}}}{G_{\smallidx-k}}}
\int_{R_{\smallidx-k,\bigidx}}
\\\notag
&\varphi_{\zeta}(g,1,g')
f_s(w_{\smallidx-k,\bigidx}r'(\rconj{b_{\bigidx,k}}g')w(\rconj{w_{\smallidx,\bigidx}}m)g,1)\psi_{\gamma}(r')dr'dg'dmdg.
\end{align}
The manipulations \eqref{eq:1st var l<n after func eq sigma}-\eqref{eq:1st var l<n before eq for pi prime} still hold when we replace $\varphi_{\zeta},f_s$ with $|\varphi_{\zeta}|,|f_s|$ and drop the character. This means that
\eqref{eq:1st var l<n after func eq sigma} and \eqref{eq:1st var l<n before eq for pi prime} are absolutely convergent in the same domain. The inner $dr'dg'$-integration may be regarded as an integral for
$G_{\smallidx-k}\times\GL{\bigidx}$ and $\pi'\times\tau$ (see
Section~\ref{subsection:Twisting the embedding} for the conjugation
$\rconj{b_{\bigidx,k}}g'$). Applying the functional equation to integral~\eqref{eq:1st var l<n before eq for pi prime} we arrive at
\begin{align}\label{eq:1st var l<n after eq for pi prime}
&\int_{\lmodulo{\overline{V_k}Z_kG_{\smallidx-k}}{G_{\smallidx}}}
\int_{\Mat{k\times \bigidx-\smallidx}}
\int_{\lmodulo{U_{G_{\smallidx-k}}}{G_{\smallidx-k}}}
\int_{R_{\smallidx-k,\bigidx}}\\\notag
&\varphi_{\zeta}(g,1,g')\nintertwiningfull{\tau}{s}
f_s(w_{\smallidx-k,\bigidx}r'(\rconj{b_{\bigidx,k}}g')w(\rconj{w_{\smallidx,\bigidx}}m)g,1)\psi_{\gamma}(r')dr'dg'dmdg.
\end{align}
We have already shown that \eqref{eq:1st var l<n before eq for pi
prime} (as a meromorphic function) is proportional to
$\Psi(\varphi_{\zeta},f_s,s)$. Below we show that \eqref{eq:1st var
l<n after eq for pi prime} resembles \eqref{eq:1st var l<n after
func eq sigma} and it will follow that it has a meromorphic
continuation which is proportional to
$\Psi(\varphi_{\zeta},\nintertwiningfull{\tau}{s}f_s,1-s)$. The
integrals $\Psi(\varphi_{\zeta},f_s,s)$ and
$\Psi(\varphi_{\zeta},\nintertwiningfull{\tau}{s}f_s,1-s)$ are
proportional, since they are equal as rational functions to
$\Psi(W_{\varphi_{\zeta}},f_s,s)$ and
$\Psi(W_{\varphi_{\zeta}},\nintertwiningfull{\tau}{s}f_s,1-s)$
(resp.), see Section~\ref{subsection:The basic identity pi}. 
Therefore \eqref{eq:1st var l<n before eq for pi prime} and
\eqref{eq:1st var l<n after eq for pi prime} are also proportional
and it remains to use a specific substitution to find the
proportionality factor.
\begin{claim}\label{claim:l<n functional eq pi prime}
Integral~\eqref{eq:1st var l<n before eq for pi prime} multiplied by
$\gamma(\pi'\times\tau,\psi,s)$ equals integral~\eqref{eq:1st var
l<n after eq for pi prime}, as meromorphic functions.
\end{claim}

Reproduce the steps \eqref{eq:1st var l<n after func eq
sigma}-\eqref{eq:1st var l<n before eq for pi prime} backwards to
land at

\begin{align}\label{eq:1st var l<n before func eq dual sigma}
&\int_{K_{G_{\smallidx}}}\int_{T_{G_{\smallidx-k}}}\int_{\overline{V_k''}}\int_{R_{\smallidx,\bigidx}}
\int_{\lmodulo{Z_k}{\GL{k}}}\int_{\Mat{k\times \bigidx-\smallidx}}
\varphi_{\zeta}(vtk,a,1)\\&\nintertwiningfull{\tau}{s}f_s(w_{\smallidx,\bigidx}rvtk,
\left(\begin{array}{ccc}0&I_{\smallidx-k}&0\\0&0&I_{\bigidx-\smallidx}\\I_k&0&m\end{array}\right)
a)\absdet{a}^{-s-\zeta-\frac{\bigidx+k}2+\smallidx}\psi_{\gamma}(r)\notag\\&dmdadrdvdtdk.\notag
\end{align}
This is just \eqref{eq:1st var l<n after func eq sigma} with
$s\mapsto1-s$ and $f_s\mapsto\nintertwiningfull{\tau}{s}f_s$. Apply
the functional equation for $\GL{k}\times\GL{\bigidx}$ and
$\sigma\times\tau^*$ to get (compare to \eqref{eq:1st var l<n before
func eq sigma})

\begin{align}\label{eq:1st var l<n after func eq dual sigma}
&\int_{K_{G_{\smallidx}}}\int_{T_{G_{\smallidx-k}}}\int_{\overline{V_k''}}\int_{R_{\smallidx,\bigidx}}
\int_{\lmodulo{Z_k}{\GL{k}}}\int_{\Mat{\smallidx-1-k\times k}}
\varphi_{\zeta}(vtk,a,1)\\\notag&\nintertwiningfull{\tau}{s}f_s(w_{\smallidx,\bigidx}rvtk,
am)\absdet{a}^{1-s-\zeta-\frac{\bigidx-k}2}\psi_{\gamma}(r)dmdadrdvdtdk.
\end{align}
Observe that by Claim~\ref{claim:l<n functional eq sigma} it
immediately follows that integral~\eqref{eq:1st var l<n before func
eq dual sigma} multiplied by
$\omega_{\sigma}(-1)^{1-\bigidx}\gamma(\sigma\times\tau^*,\psi,1-s-\zeta)^{-1}$
equals integral~\eqref{eq:1st var l<n after func eq dual sigma}, as
meromorphic functions. This is because we established the
proportionality factor between \eqref{eq:1st var l<n before func eq
sigma} and \eqref{eq:1st var l<n after func eq sigma} for
arbitrary $(\sigma,\tau,\varphi_{\zeta},f_s,s)$. 

We return from \eqref{eq:1st var l<n after func eq dual sigma}
reversing the passage \eqref{eq:1st var l<n start
form}-\eqref{eq:1st var l<n before func eq sigma} and reach
\begin{align*}
&\int_{\lmodulo{U_{G_{\smallidx}}}{G_{\smallidx}}}
(\int_{V_k}\varphi_{\zeta}(vg,1,1)\psi_{\gamma}(v)dv)
\int_{R_{\smallidx,\bigidx}}\nintertwiningfull{\tau}{s}f_s(w_{\smallidx,\bigidx}rg,1)\psi_{\gamma}(r)drdg\\
&=\Psi(\varphi_{\zeta},\nintertwiningfull{\tau}{s}f_s,1-s).
\end{align*}
Altogether it was shown that
\begin{align*}
&\gamma(\sigma\times\tau^*,\psi,1-s-\zeta)^{-1}\gamma(\pi'\times\tau,\psi,s)\gamma(\sigma\times\tau,\psi,s-\zeta)\Psi(\varphi_{\zeta},f_s,s)\\
&=\Psi(\varphi_{\zeta},\nintertwiningfull{\tau}{s}f_s,1-s).
\end{align*}
Since
$\gamma(\sigma\times\tau^*,\psi,1-s-\zeta)^{-1}=\omega_{\sigma}(-1)^{\bigidx}\omega_{\tau}(-1)^k\gamma(\sigma^*\times\tau,\psi,s+\zeta)$,
equality~\eqref{eq:gamma first var for zeta large with integrals}
follows. 

\subsection{The case $k=\smallidx>\bigidx$}\label{subsection:1st var k=l>n+1}
Here $G_{\smallidx}$ is necessarily split. Assume first that $\pi$
is induced from $\overline{P_{\smallidx}}$. Start with
\begin{align}\label{int:mult 1st var l=k>n+1 starting form}
\Psi(\varphi_{\zeta},f_s,s)=\int_{\lmodulo{U_{H_{\bigidx}}}{H_{\bigidx}}}
\int_{R^{\smallidx,\bigidx}}
(\int_{V_{\smallidx}}\varphi_{\zeta}(vrw^{\smallidx,\bigidx}h,d_{\gamma})\psi_{\gamma}(v)dv)
f_s(h,1)drdh,
\end{align}
where $W_{\varphi_{\zeta}}$ was replaced by \eqref{eq:whittaker
for pi, k=l} ($d_{\gamma}=diag(I_{\smallidx-1},4)$). Decompose
$V_{\smallidx}=V_{\smallidx}'\cdot V_{\smallidx}''$ with
\begin{align*}
&V_{\smallidx}'=\{\left(\begin{array}{cccccc}
I_{\bigidx}&0&0&\frac1{2\beta}v_1&0&v_2\\
&I_{\smallidx-\bigidx-1}&0&0&0&0\\
&&1&0&0&\frac1{2\beta}v_1'\\
&&&1&0&0\\
&&&&I_{\smallidx-\bigidx-1}&0\\
&&&&&I_{\bigidx}\\
\end{array}\right)\},\\
&V_{\smallidx}''=\{\left(\begin{array}{cccccc}
I_{\bigidx}&0&0&0&v_4&0\\
&I_{\smallidx-\bigidx-1}&0&v_3&v_5&v_4'\\
&&1&0&v_3'&0\\
&&&1&0&0\\
&&&&I_{\smallidx-\bigidx-1}&0\\
&&&&&I_{\bigidx}\\
\end{array}\right)\}.
\end{align*}
Here if $\smallidx=\bigidx+1$, $V_{\smallidx}''=\{1\}$. We write the integral as
\begin{align}\label{int:mult 1st var l=k>n+1 after starting form}
\int_{\lmodulo{U_{H_{\bigidx}}}{H_{\bigidx}}}
\int_{R^{\smallidx,\bigidx}}\int_{V_{\smallidx}'}\int_{V_{\smallidx}''}
\varphi_{\zeta}(v''v'rw^{\smallidx,\bigidx}h,d_{\gamma})\psi_{\gamma}(v''v')
f_s(h,1)dv''dv'drdh.
\end{align}

It is first shown how to replace the $dv'$-integration with a
$du_0$-integration over $U_{\bigidx}$. Let
\begin{align*}
u_0=\left(\begin{array}{ccc}I_{\bigidx}&v_1&v_2+\half
v_1v_1'\\&1&v_1'\\&&I_{\bigidx}\\\end{array}\right)\in U_{\bigidx},
\end{align*}
with $\transpose{v_2}J_{\bigidx}+J_{\bigidx}v_2=0$ and
$v_1'=-\transpose{v_1}J_{\bigidx}$. If we put
$u_0^{\star}=\rconj{(w^{\smallidx,\bigidx})^{-1}}u_0$,
$u_0^{\star}=\eta(v')v'$ where $\eta(v')$ is the image in
$L_{\smallidx}$ of
\begin{align*}
&\left(\begin{array}{ccc}I_{\bigidx}&&\beta v_1\\&I_{\smallidx-\bigidx-1}\\&&1\\\end{array}\right)\in\GL{\smallidx}
\end{align*}
and
\begin{align*}
&v'=\left(\begin{array}{cccccc}
I_{\bigidx}&0&0&\frac1{2\beta}v_1&0&v_2\\
&I_{\smallidx-\bigidx-1}&0&0&0&0\\
&&1&0&0&\frac1{2\beta}v_1'\\
&&&1&0&0\\
&&&&I_{\smallidx-\bigidx-1}&0\\
&&&&&I_{\bigidx}\\
\end{array}\right).
\end{align*}

Assume $\smallidx>\bigidx+1$. In this case $\varphi_{\zeta}$ is
left-invariant by $\eta(v')$,
$\psi_{\gamma}(v''v')=\psi_{\gamma}(v'')$ and $\eta(v')$ normalizes
$V_{\smallidx}''$ without changing $\psi_{\gamma}$. Hence
\eqref{int:mult 1st var l=k>n+1 after starting form} equals
\begin{align*}
&\int_{\lmodulo{U_{H_{\bigidx}}}{H_{\bigidx}}}
\int_{R^{\smallidx,\bigidx}}\int_{V_{\smallidx}'}\int_{V_{\smallidx}''}
\varphi_{\zeta}(v''\eta(v') v'rw^{\smallidx,\bigidx}h,d_{\gamma})
\psi_{\gamma}(v'')f_s(h,1)dv''dv'drdh\\
&=\int_{\lmodulo{U_{H_{\bigidx}}}{H_{\bigidx}}}
\int_{R^{\smallidx,\bigidx}}
\int_{U_{\bigidx}}\int_{V_{\smallidx}''}
\varphi_{\zeta}(v''u_0^{\star}rw^{\smallidx,\bigidx}h,d_{\gamma})
\psi_{\gamma}(v'')f_s(h,1)dv''du_0drdh.
\end{align*}
Here the measure $du_0$ is normalized according to $dv'$. For a
fixed
\begin{align*}
r=\left(\begin{array}{ccccc}I_{\bigidx}&&\\x&I_{\smallidx-\bigidx-1}\\&&I_2\\&&&I_{\smallidx-\bigidx-1}\\&&&x'&I_{\bigidx}\end{array}\right)\in
R^{\smallidx,\bigidx},
\end{align*}
$\rconj{(u_0^{\star})^{-1}}r=b_{r,u_0}z_{r,u_0}r$ where $b_{r,u_0}$
is the image in $L_{\smallidx}$ of
\begin{align*}
\left(\begin{array}{ccc}I_{\bigidx}&&\\&I_{\smallidx-\bigidx-1}&-\beta
xv_1\\&&1\\\end{array}\right)\in\GL{\smallidx}
\end{align*}
and
\begin{align*}
&z_{r,u_0}=\left(\begin{array}{cccccc}
I_{\bigidx}&0&0&0&*&0\\
&I_{\smallidx-\bigidx-1}&0&-\frac1{2\beta}xv_1&*&*\\
&&1&0&\frac1{2\beta}v_1'x'&0\\
&&&1&0&0\\
&&&&I_{\smallidx-\bigidx-1}&0\\
&&&&&I_{\bigidx}\\
\end{array}\right)\in V_{\smallidx}''.
\end{align*}
Then the integral is
\begin{align*}
\int_{\lmodulo{U_{H_{\bigidx}}}{H_{\bigidx}}}
\int_{R^{\smallidx,\bigidx}}\int_{U_{\bigidx}}\int_{V_{\smallidx}''}
\varphi_{\zeta}(v''b_{r,u_0}z_{r,u_0}rw^{\smallidx,\bigidx}u_0h,d_{\gamma})
\psi_{\gamma}(v'')f_s(h,1)dv''du_0drdh.
\end{align*}
Now $b_{r,u_0}$ normalizes $V_{\smallidx}''$ with no change to
$\psi_{\gamma}$ and the change $v''\mapsto v''z_{r,u_0}^{-1}$
changes
$\psi_{\gamma}(v'')\mapsto\psi_{\gamma}(v'')\psi^{-1}(\quarter\beta
(xv_1)_{\smallidx-\bigidx-1})$. Since
$\varphi_{\zeta}(b_{r,u_0}g,d_{\gamma})=\psi(\quarter\beta
(xv_1)_{\smallidx-\bigidx-1})\varphi_{\zeta}(g,d_{\gamma})$ we get
\begin{align}\label{int:int before formula multiple integrations in H}
\int_{\lmodulo{U_{H_{\bigidx}}}{H_{\bigidx}}}
\int_{R^{\smallidx,\bigidx}}
\int_{U_{\bigidx}}\int_{V_{\smallidx}''}
\varphi_{\zeta}(v''rw^{\smallidx,\bigidx}u_0h,d_{\gamma})
\psi_{\gamma}(v'')f_s(h,1)dv''du_0drdh.
\end{align}

If $\smallidx=\bigidx+1$, $\eta(v')$ satisfies
$\varphi_{\zeta}(1,d_{\gamma})\psi_{\gamma}(v')=\varphi_{\zeta}(\eta(v'),d_{\gamma})$
(note that $\psi_{\gamma}(v')=\psi(-\quarter\beta(v_1)_{\bigidx})$)
whence
\begin{align*}
\int_{V_{\smallidx}'}\varphi_{\zeta}(v'h,d_{\gamma})\psi_{\gamma}(v')dv'=\int_{V_{\smallidx}'}\varphi_{\zeta}(\eta(v')v'h,d_{\gamma})dv'=
\int_{U_{\bigidx}}\varphi_{\zeta}(u_0h,d_{\gamma})du_0,
\end{align*}
where $du_0$ is normalized as above, and \eqref{int:mult 1st var
l=k>n+1 after starting form} still equals \eqref{int:int before
formula multiple integrations in H}. In this case
$\psi_{\gamma}(v'')\equiv1$ since $V_{\smallidx}''=\{1\}$.

We proceed with \eqref{int:int before formula multiple integrations
in H}, for any $\smallidx>\bigidx$. Utilize the integration formula
\begin{align*}
\int_{\lmodulo{U_{H_{\bigidx}}}{H_{\bigidx}}}F(h)dh=\int_{\overline{U_{\bigidx}}}\int_{\lmodulo{Z_{\bigidx}}{\GL{\bigidx}}}F(au)\absdet{a}^{-\bigidx}dadu.
\end{align*}
The subgroup $U_{\bigidx}<H_{\bigidx}$ is normalized by $a\in
\GL{\bigidx}<M_{\bigidx}$ with a change of measure
$du_0\mapsto\absdet{a}^{\bigidx}du_0$. The element
$\rconj{(w^{\smallidx,\bigidx})^{-1}}a\in L_{\smallidx}$ normalizes
$R^{\smallidx,\bigidx}$ and $V_{\smallidx}''$, multiplies $dr$ by
$\absdet{a}^{\bigidx-\smallidx+1}$ and $dv''$ by
$\absdet{a}^{\smallidx-\bigidx-1}$, preserves $\psi_{\gamma}$ and
its image in $\GL{\smallidx}$ commutes with $d_{\gamma}$. Also
$\varphi_{\zeta}(a,1)=\absdet{a}^{-\frac{\smallidx-1}2-\zeta}\varphi_{\zeta}(1,a)$.
The integral becomes 
\begin{align}\label{int:int before formula multiple integrations in H after a}
&\int_{\overline{U_{\bigidx}}}
\int_{\lmodulo{Z_{\bigidx}}{\GL{\bigidx}}}
\int_{R^{\smallidx,\bigidx}}
\int_{U_{\bigidx}}\int_{V_{\smallidx}''}
\varphi_{\zeta}(v''rw^{\smallidx,\bigidx}u_0u,ad_{\gamma})\\\notag&f_s(u,a)\absdet{a}^{s-\zeta+\frac{\bigidx-\smallidx}2}\psi_{\gamma}(v'')
dv''du_0drdadu.
\end{align}
The $da$-integration defines an integral for
$\GL{\bigidx}\times\GL{\smallidx}$. However, applying
\eqref{eq:modified func eq glm gln} now (with $j=0$) introduces a
unipotent integration which interferes one step ahead (see
Remark~\ref{remark:problem with j=0 and func eq} below). This
problem is overcome by using the coordinates of $r$. Since
$R^{\smallidx,\bigidx}<L_{\smallidx}$ normalizes $V_{\smallidx}''$
without affecting $\psi_{\gamma}$ and its projection on
$GL_{\smallidx}$ commutes with $d_{\gamma}$, we get
\begin{align}\label{int:mult 1st var l=k>n+1 before applying func eq sigma and tau}
&\int_{\overline{U_{\bigidx}}}
\int_{U_{\bigidx}}\int_{V_{\smallidx}''}
\int_{\lmodulo{Z_{\bigidx}}{\GL{\bigidx}}}
\int_{\Mat{\smallidx-\bigidx-1\times\bigidx}}
\varphi_{\zeta}(v''w^{\smallidx,\bigidx}u_0u,ard_{\gamma})\\\notag&f_s(u,a)\absdet{a}^{s-\zeta+\frac{\bigidx-\smallidx}2}\psi_{\gamma}(v'')
drdadv''du_0du.
\end{align}

Now the $drda$-integration forms an integral for
$\GL{\bigidx}\times\GL{\smallidx}$ and $\tau\times\sigma$. Using
\eqref{eq:modified func eq glm gln} with $j=\smallidx-\bigidx-1$ we
obtain
\begin{align}\label{int:mult 1st var k=l>n+1 after applying func eq sigma and tau}
&\int_{\overline{U_{\bigidx}}}
\int_{U_{\bigidx}}\int_{V_{\smallidx}''}
\int_{\lmodulo{Z_{\bigidx}}{\GL{\bigidx}}}
\varphi_{\zeta}(v''w^{\smallidx,\bigidx}u_0u,
\left(\begin{array}{cc}0&I_{\smallidx-\bigidx}\\a&0\\\end{array}\right)d_{\gamma})\\\notag&f_s(u,a)\absdet{a}^{s-\zeta-\frac{\bigidx-\smallidx}2-1}\psi_{\gamma}(v'')
dadv''du_0du.
\end{align}
As in Section~\ref{subsection:1st var l<n} (\eqref{eq:1st var l<n
before func eq sigma}-\eqref{eq:1st var l<n after func eq sigma})
integrals~\eqref{int:mult 1st var l=k>n+1 before applying func eq
sigma and tau} and \eqref{int:mult 1st var k=l>n+1 after applying
func eq sigma and tau} have proportional meromorphic continuations.
Note that \eqref{int:mult 1st var k=l>n+1 after applying func eq
sigma and tau} is equal as an integral to \eqref{int:k=l>n+1
reversing steps without dr} below and by a direct verification,
integral \eqref{int:k=l>n+1 reversing steps without dr} satisfies
\eqref{eq:bilinear special condition}.
\begin{claim}\label{claim:k=l>n+1 functional eq sigma}
Integral~\eqref{int:mult 1st var l=k>n+1 before applying func eq
sigma and tau} multiplied by
$\omega_{\tau}(-1)^{\smallidx-1}\gamma(\sigma\times\tau,\psi^{-1},s-\zeta)$
equals integral~\eqref{int:mult 1st var k=l>n+1 after applying func
eq sigma and tau}, as meromorphic functions.
\end{claim}

Reversing steps \eqref{int:int before formula multiple integrations
in H}-\eqref{int:int before formula multiple integrations in H after
a} (without $dr$) we get
\begin{align}\label{int:k=l>n+1 reversing steps without dr}
&\int_{\lmodulo{U_{H_{\bigidx}}}{H_{\bigidx}}}
\int_{U_{\bigidx}}\int_{V_{\smallidx}''}\varphi_{\zeta}(v''w^{\smallidx,\bigidx}u_0h,
\omega_{\smallidx-\bigidx,\bigidx}d_{\gamma})\psi_{\gamma}(v'')f_s(h,1)dv''du_0dh.
\end{align}
Since $f_s(u_0h,1)=f_s(h,1)$ for $u_0\in U_{\bigidx}$ and
$U_{H_{\bigidx}}=Z_{\bigidx}U_{\bigidx}$, the $du_0dh$-integration
collapses into an integral over
$\lmodulo{Z_{\bigidx}}{H_{\bigidx}}$.
Factoring through $\overline{U_{\bigidx}}$ leads to
\begin{align*}
&\int_{\lmodulo{\overline{U_{\bigidx}}Z_{\bigidx}}{H_{\bigidx}}}
\int_{\overline{U_{\bigidx}}} \int_{V_{\smallidx}''}
\varphi_{\zeta}(v''w^{\smallidx,\bigidx}uh,
\omega_{\smallidx-\bigidx,\bigidx}d_{\gamma})\psi_{\gamma}(v'')f_s(uh,1)dv''dudh.
\end{align*}

Recall that in Section~\ref{subsection:1st var l<n} the next step
was to apply a functional equation for
$G_{\smallidx-k}\times\GL{\bigidx}$. Here we construct the \lhs\ of
Shahidi's functional equation \eqref{eq:Shahidi local coefficient
def}. 
Define a function
$F(h)$ on $H_{\bigidx}$ by
\begin{align*}
F(h)=\int_{V_{\smallidx}''}\varphi_{\zeta}(v''w^{\smallidx,\bigidx}h,
\omega_{\smallidx-\bigidx,\bigidx}d_{\gamma})\psi_{\gamma}(v'')dv''.
\end{align*}

\begin{claim}\label{claim:k=l>n+1 W part invariant for u bar}
Let $u\in\overline{U_{\bigidx}}$. For any $h\in H_{\bigidx}$,
$F(uh)=\psi^{-1}(\frac2{\beta}u_{\bigidx+1,1})F(h)$, where $u$ is
written relative to $\mathcal{E}_{H_{\bigidx}}$.
\end{claim}

\begin{remark}\label{remark:problem with j=0 and func eq}
If we applied the functional equation for $\tau\times\sigma$ with
$j=0$, a similar claim would fail.
\end{remark}

Then the integral equals
\begin{align*}
&\int_{\lmodulo{\overline{U_{\bigidx}}Z_{\bigidx}}{H_{\bigidx}}}
F(h)\int_{\overline{U_{\bigidx}}}
f_s(uh,1)\psi^{-1}(\frac2{\beta}u_{\bigidx+1,1})dudh.
\end{align*}
\begin{remark}\label{remark:l>n gamma proof uncovers H_n times GL_l integral}
This resembles the Rankin-Selberg integral for
$SO_{2\bigidx+1}\times\GL{\smallidx}$ studied by Soudry
\cite{Soudry}. The mapping $f_s\mapsto\int_{\overline{U_{\bigidx}}}
f_s(u,1)\psi^{-1}(\frac2{\beta}u_{\bigidx+1,1})du$ defines a
Whittaker functional on $V(\tau,s)$.
\end{remark}
Rewriting the $du$-integration over $U_{\bigidx}$ we obtain
\begin{align}\label{int:1var k=l>n+1 before applying Shahidi local coefficient}
&\int_{\lmodulo{\overline{U_{\bigidx}}Z_{\bigidx}}{H_{\bigidx}}}F(h)
(\int_{U_{\bigidx}}
f_s(w_{\bigidx}uw_{\bigidx}^{-1}h,1)\psi^{-1}(\frac2{\beta}(-1)^{\bigidx+1}u_{\bigidx,\bigidx+1})du)dh.
\end{align}
According to equality~\eqref{eq:Shahidi local coefficient def}, as meromorphic continuations
integral~\eqref{int:1var k=l>n+1 before applying Shahidi local
coefficient} equals
\begin{align}\label{int:1var k=l>n+1 after applying Shahidi local coefficient}
&c_{\tau,\beta}\int_{\lmodulo{\overline{U_{\bigidx}}Z_{\bigidx}}{H_{\bigidx}}}F(h)
(\int_{U_{\bigidx}}
\nintertwiningfull{\tau}{s}f_s(w_{\bigidx}uw_{\bigidx}^{-1}h,1)\psi^{-1}(\frac2{\beta}(-1)^{\bigidx+1}u_{\bigidx,\bigidx+1})du)dh.
\end{align}
Here
$c_{\tau,\beta}=|\frac{\beta}2|^{-2\bigidx(s-\half)}\omega_{\tau}(\frac{\beta}2)^{-2}$
is calculated by substituting $y\cdot f_s$ for $f_s$ in
\eqref{eq:Shahidi local coefficient def} where
\begin{align*}
y=diag((-1)^{\bigidx}\frac2{\beta}I_{\bigidx},1,(-1)^{\bigidx}\frac{\beta}{2}I_{\bigidx})
diag((d_{\bigidx}^*)^{-1},1,d_{\bigidx}^{-1})w_{\bigidx}^{-1}h.
\end{align*}

The justification to the last passage is the same as in
Section~\ref{subsection:1st var l<n} (Claim~\ref{claim:l<n
functional eq pi prime}).
Now we backtrack the transitions \eqref{int:mult 1st var
k=l>n+1 after applying func eq sigma and tau}-\eqref{int:1var
k=l>n+1 before applying Shahidi local coefficient} and arrive at
\begin{align}\label{int:mult 1st var k=l>n+1 before applying func eq sigma and dual tau}
&c_{\tau,\beta}\int_{\overline{U_{\bigidx}}}
\int_{U_{\bigidx}}\int_{V_{\smallidx}''}
\int_{\lmodulo{Z_{\bigidx}}{\GL{\bigidx}}}
\varphi_{\zeta}(v''w^{\smallidx,\bigidx}u_0u,
\left(\begin{array}{cc}0&I_{\smallidx-\bigidx}\\a&0\\\end{array}\right)d_{\gamma})\\\notag&
\nintertwiningfull{\tau}{s}f_s(u,a)\absdet{a}^{-s-\zeta-\frac{\bigidx-\smallidx}2}\psi_{\gamma}(v'')
dadv''du_0du.
\end{align}
This is just \eqref{int:mult 1st var k=l>n+1 after applying func eq
sigma and tau} with $s\mapsto1-s$ and
$f_s\mapsto\nintertwiningfull{\tau}{s}f_s$. The functional equation
for $\GL{\bigidx}\times\GL{\smallidx}$ and $\tau^*\times\sigma$
implies that integral~\eqref{int:mult 1st var k=l>n+1 before
applying func eq sigma and dual tau} multiplied by
$\omega_{\tau}(-1)^{\smallidx-1}\gamma(\sigma\times\tau^*,\psi^{-1},1-s-\zeta)^{-1}$
equals (compare to \eqref{int:mult 1st var l=k>n+1 before applying
func eq sigma and tau})
\begin{align}\label{int:mult 1st var k=l>n+1 after applying func eq sigma and dual tau}
&c_{\tau,\beta}\int_{\overline{U_{\bigidx}}}
\int_{U_{\bigidx}}\int_{V_{\smallidx}''}
\int_{\lmodulo{Z_{\bigidx}}{\GL{\bigidx}}}
\int_{\Mat{\smallidx-\bigidx-1\times\bigidx}}
\varphi_{\zeta}(v''w^{\smallidx,\bigidx}u_0u,ard_{\gamma})\\\notag
&\nintertwiningfull{\tau}{s}f_s(u,a)\absdet{a}^{1-s-\zeta+\frac{\bigidx-\smallidx}2}\psi_{\gamma}(v'')
drdadv''du_0du.
\end{align}

As in Section~\ref{subsection:1st var l<n} (see \eqref{eq:1st var
l<n before func eq dual sigma}-\eqref{eq:1st var l<n after func eq
dual sigma}) the last passage, which is between meromorphic
continuations, is already justified by Claim~\ref{claim:k=l>n+1
functional eq sigma}.
Finally, reversing the steps \eqref{int:mult 1st var l=k>n+1
starting form}-\eqref{int:mult 1st var l=k>n+1 before applying func
eq sigma and tau} yields
\begin{align}\label{int:mult 1st var n<k<=l final form}
c_{\tau,\beta}\int_{\lmodulo{U_{H_{\bigidx}}}{H_{\bigidx}}}
\int_{R^{\smallidx,\bigidx}}
(\int_{V_{\smallidx}}\varphi_{\zeta}(vrw^{\smallidx,\bigidx}h,d_{\gamma})\psi_{\gamma}(v)dv)
\nintertwiningfull{\tau}{s}f_s(h,1)drdh.
\end{align}

Since
\begin{align*}
&\gamma(\sigma\times\tau,\psi^{-1},s-\zeta)=\omega_{\sigma}(-1)^{\bigidx}\omega_{\tau}(-1)^{\smallidx}\gamma(\sigma\times\tau,\psi,s-\zeta),\\ &\gamma(\sigma\times\tau^*,\psi^{-1},1-s-\zeta)^{-1}=\gamma(\sigma^*\times\tau,\psi,s+\zeta),\\
&c_{\tau,\beta}=c(\smallidx,\tau,\gamma,s)\omega_{\tau}(\beta)^2|\beta|^{2\bigidx(s-\half)}=c(\smallidx,\tau,\gamma,s)\omega_{\tau}(2\gamma)
\end{align*}
(recall that we assume $|\beta|=1$), we conclude
\begin{align*}
&c(\smallidx,\tau,\gamma,s)^{-1}\omega_{\sigma}(-1)^{\bigidx}\omega_{\tau}(-1)^{\smallidx}\omega_{\tau}(2\gamma)^{-1}\gamma(\sigma\times\tau,\psi,s-\zeta)
\gamma(\sigma^*\times\tau,\psi,s+\zeta)\Psi(\varphi_{\zeta},f_s,s)\\
&=\Psi(\varphi_{\zeta},\nintertwiningfull{\tau}{s}f_s,1-s).
\end{align*}
This completes the proof for $\pi$ induced from
$\overline{P_{\smallidx}}$.

We explain the necessary adjustments in case
$\pi=\cinduced{\rconj{\kappa}(\overline{P_{\smallidx}})}{G_{\smallidx}}{\sigma}$.
We use the Whittaker functional for $\pi$ given by
\eqref{eq:whittaker for pi, k=l 2}.
Decompose
$\rconj{\kappa}V=\rconj{\kappa}(V_{\smallidx}')\rconj{\kappa}(V_{\smallidx}'')$
and write $u_0^{\star}=\eta(v') v'$, where now $\eta(v')$ equals the
matrix denoted above by $v'$ (without the coordinates of $v_2$) 
and normalizes $\rconj{\kappa}V_{\smallidx}''$. In the decomposition
$b_{r,u_0}z_{r,u_0}r$ we exchange the matrices of $b_{r,u_0}$ and
$z_{r,u_0}$, whence $z_{r,u_0}\in\rconj{\kappa}V_{\smallidx}''$ and
$\varphi_{\zeta}(b_{r,u_0}g,d_{\gamma})=\psi^{-1}(\quarter \beta
(xv_1)_{\smallidx-\bigidx-1})\varphi_{\zeta}(g,d_{\gamma})$. Since
$\smallidx>\bigidx$, the passage to the integration over
$\lmodulo{Z_\bigidx}{\GL{\bigidx}}$ is the same as above. Also
$R^{\smallidx,\bigidx}=\rconj{\kappa}R^{\smallidx,\bigidx}$ so this
subgroup normalizes $\rconj{\kappa}V_{\smallidx}''$. The function
$F(h)$ is defined in the same way but now satisfies
$F(uh)=\psi(\frac2{\beta}u_{\bigidx+1,1})F(h)$, which does not
change the constant $c_{\tau,\beta}$ in \eqref{int:1var k=l>n+1
after applying Shahidi local coefficient}.
There are no further differences.

\subsection{Proofs of claims
}\label{subsection:proofs claims 1st var}
\begin{proof}[Proof of claim~\ref{claim:l<n functional eq sigma}] 
Recall that integral~\eqref{eq:1st var l<n before func eq sigma} is
\begin{align*}
&\int_{K_{G_{\smallidx}}}\int_{T_{G_{\smallidx-k}}}\int_{\overline{V_k''}}
\int_{R_{\smallidx,\bigidx}} \int_{\lmodulo{Z_k}{\GL{k}}}
\int_{\Mat{\smallidx-1-k\times k}} \varphi_{\zeta}(vtk,a,1)\\\notag
&f_s(w_{\smallidx,\bigidx}rvtk,am)
\absdet{a}^{s-\zeta-\frac{\bigidx-k}2}\psi_{\gamma}(r)dmdadrdvdtdk.
\end{align*}
Let $W_{\sigma}\in\Whittaker{\sigma}{\psi^{-1}}$ and
$W_{\tau}\in\Whittaker{\tau}{\psi}$ be arbitrary. Let $k_1>0$ be
such that $W_{\tau}$ is right-invariant by
$\mathcal{N}_{\GL{\bigidx},k_1}$.

If $k<\smallidx-1$, select
$W_{\pi'}\in\Whittaker{\pi'}{\psi_{\gamma}^{-1}}$ by
Lemma~\ref{lemma:W with small support} with $j=0$ and $k_2>>k_1$, for some $W_0\in\Whittaker{\pi'}{\psi_{\gamma}^{-1}}$ such that
$W_0(1)\ne0$. 
In case $k=\smallidx-1$, $\pi'$ is a character of $G_1$ and we put
$W_{\pi'}=\pi'$.

Choose $f_s\in\xi(\tau,hol,s)$ according to Lemma~\ref{lemma:f with
small support} with $W_{\tau}$ and $k_3>>k_2$. Finally we take
$\varphi_{\zeta}$ with support in
$\overline{P_k}\mathcal{N}_{G_{\smallidx},k_4}$ such that for all
$a\in\GL{k}$, $h\in G_{\smallidx-k}$, $v\in\overline{V_k}$ and $u\in
\mathcal{N}_{G_{\smallidx},k_4}$,
$\varphi(ahvu,1,1)=\absdet{a}^{-\smallidx+\frac{k+1}2-\zeta}W_{\sigma}(a)W_{\pi'}(h)$.
Taking $k_4>>k_3$ we may assume that 
$f_s$ is right-invariant on the image of
$\mathcal{N}_{G_{\smallidx},k_4}$ in $H_{\bigidx}$.

With the above selection \eqref{eq:1st var l<n before func eq sigma} becomes
\begin{align*}
&c\int_{\lmodulo{U_{G_{\smallidx-k}}}{G_{\smallidx-k}}}\int_{\overline{V_k''}}
\int_{R_{\smallidx,\bigidx}} \int_{\lmodulo{Z_k}{\GL{k}}}
\int_{\Mat{\smallidx-1-k\times k}} W_{\sigma}(a)W_{\pi'}(h)\\\notag
&f_s(w_{\smallidx,\bigidx}rvh,am)
\absdet{a}^{s-\zeta-\frac{\bigidx-k}2}\psi_{\gamma}(r)dmdadrdvdh.
\end{align*}
Here $c>0$ is a volume constant depending on $k_4$. Let $V'=\overline{V_{\smallidx-k-1}}\rtimes G_1<G_{\smallidx-k}$.
Using the formula
\begin{align*}
\int_{\lmodulo{U_{G_{\smallidx-k}}}{G_{\smallidx-k}}}F(h)dh=
\int_{\overline{B_{\GL{\smallidx-k-1}}}}\int_{V'}F(bv')\delta(b)dv'db,
\end{align*}
where $\delta$ is a suitable modulus character, we get
\begin{align}\label{eq:substitution integral first var K<l}
&c\int_{\overline{B_{\GL{\smallidx-k-1}}}}\int_{V'}\int_{\overline{V_k''}}
\int_{R_{\smallidx,\bigidx}} \int_{\lmodulo{Z_k}{\GL{k}}}
\int_{\Mat{\smallidx-1-k\times k}}
W_{\sigma}(a)W_{\pi'}(bv')\\\notag
&f_s(w_{\smallidx,\bigidx}rvv',amb)
\absdet{a}^{s-\zeta-\frac{\bigidx-k}2}\psi_{\gamma}(r)\absdet{b}^{-k+\smallidx-\bigidx+s-\half}\delta(b)dmdadrdvdv'db.
\end{align}

For fixed $a,m$ and $b$, since $vv'\in
\overline{V_{\smallidx-1}}\rtimes G_1$, we deduce that the
$dr$-integration vanishes unless
$v,v'\in\mathcal{N}_{G_{\smallidx},k_3-k_0}$ (with $k_0$ - the
constant of Lemma~\ref{lemma:f with small support}). Thereby since
$k_3>>k_2$, $W_{\pi'}(bv')=W_{\pi'}(b)$. It follows that up to a
measure constant,
\begin{align*}
&\int_{V'}\int_{\overline{V_k''}}\int_{R_{\smallidx,\bigidx}}
W_{\pi'}(bv')f_s(w_{\smallidx,\bigidx}rvv',amb)\psi_{\gamma}(r)drdvdv'
=P_s(\gamma)W_{\pi'}(b)W_{\tau}(ambt_{\gamma}),
\end{align*}
where $P_s(\gamma)=|\gamma|^{\smallidx(\half\bigidx+s-\half)}$ and
$t_{\gamma}$ is given by Lemma~\ref{lemma:f with small support}.

Now $W_{\pi'}(b)=0$ unless $b\in\mathcal{N}_{G_{\smallidx-k},k_2}$
and because $k_2>>k_1$,
$W_{\tau}(ambt_{\gamma})=W_{\tau}(amt_{\gamma})$ (if
$k=\smallidx-1$, there is no $db$-integration at all). Hence
\eqref{eq:substitution integral first var K<l} equals
\begin{align}\label{int:substitution claim k<l<=n lhs of equation}
&c'P_s(\gamma)\int_{\lmodulo{Z_k}{\GL{k}}}\int_{\Mat{\smallidx-1-k\times
k}}W_{\sigma}(a)W_{\tau}(amt_{\gamma})
\absdet{a}^{s-\zeta-\frac{\bigidx-k}2}dmda.
\end{align}
The constant $c'>0$ equals a product of volumes depending on $k_i$,
$1\leq i\leq 4$ (but not on $s$).

Integral~\eqref{int:substitution claim k<l<=n lhs of equation} forms
the \lhs\ of the $\GL{k}\times\GL{\bigidx}$ functional equation
\eqref{eq:modified func eq glm gln}, with the following exception:
on the one hand, our computation is valid in the domain $D$ of
absolute convergence of \eqref{eq:1st var l<n before func eq sigma},
which takes the form $\{\Re(\zeta)<<\Re(s)<<(1+C_1)\Re(\zeta)\}$. On
the other hand, the integral on the \lhs\ of \eqref{eq:modified func
eq glm gln} is defined in a right half-plane $\Re(s-\zeta)>s_0$,
where $s_0$ is a constant depending only on $\sigma$ and $\tau$.
However since $C_1>0$,
$lim_{\Re(\zeta)\rightarrow\infty}C_1\Re(\zeta)=\infty$ hence we can
take $\Re(\zeta)>>0$ such that $D$ contains a non-empty domain $D_0$
of the form $\setof{s\in\C}{s_0<<\Re(s-\zeta)<<C_1\Re(\zeta)}$. Then
in $D_0$ integral~\eqref{int:substitution claim k<l<=n lhs of
equation} is equal to the integral on the \lhs\ of
\eqref{eq:modified func eq glm gln}.

When we apply the same
substitution to integral~\eqref{eq:1st var l<n after func eq sigma}
we obtain
\begin{align}\label{int:substitution claim k<l<=n rhs of equation}
&c'P_s(\gamma)\int_{\lmodulo{Z_k}{\GL{k}}}\int_{\Mat{k\times
\bigidx-\smallidx}}W_{\sigma}(a)W_{\tau}
(\left(\begin{array}{ccc}0&I_{\smallidx-k}&0\\0&0&I_{\bigidx-\smallidx}\\I_k&0&m\end{array}\right)at_{\gamma})
\absdet{a}^{s-\zeta-\frac{\bigidx+k}2+\smallidx-1}dmda.
\end{align}
This integral comprises the \rhs\ of \eqref{eq:modified func eq glm
gln}, but it is defined in a domain $D^*$ of the form
$\{(1-C_2)\Re(\zeta)<<\Re(s)<<\Re(\zeta)\}$, while the \rhs\ of
\eqref{eq:modified func eq glm gln} is defined in a left half-plane
$\Re(s-\zeta)<s_1$. Because $C_2>0$, we have
$lim_{\Re(\zeta)\rightarrow\infty}-C_2\Re(\zeta)=-\infty$ whence
$D^*$ contains a non-empty domain $D_1$ of the form
$\setof{s\in\C}{-C_2\Re(\zeta)<<\Re(s-\zeta)<<s_1}$. In $D_1$
integral~\eqref{int:substitution claim k<l<=n rhs of equation} is
equal to the integral on the \rhs\ of \eqref{eq:modified func eq glm
gln}.

Now it follows from \eqref{eq:modified func eq glm gln} that as functions in $\C(q^{-s})$,
\eqref{int:substitution claim k<l<=n lhs of equation} multiplied by $\omega_{\sigma}(-1)^{\bigidx-1}\gamma(\sigma\times\tau,\psi,s-\zeta)$ equals \eqref{int:substitution claim k<l<=n rhs of equation}.
\end{proof} 

\begin{proof}[Proof of Claim~\ref{claim:l<n first var inner integral m r and v}] 
We start with
\begin{align*}
\int_{\overline{V_k''}} \int_{\Mat{k\times \bigidx-\smallidx}}
\int_{R_{\smallidx,\bigidx}}f_s(\omega_{\bigidx-k,k}mw_{\smallidx,\bigidx}rvg,1)\psi_{\gamma}(r)drdmdv.
\end{align*}
Let $r\in R_{\smallidx,\bigidx}$ be with the first
$\bigidx-\smallidx$ rows
\begin{align*}
\left(\begin{array}{ccccccc}
I_{\bigidx-\smallidx}&x_1&x_2&x_3&x_4&0_{\smallidx}&z\\
\end{array}\right),
\end{align*}
where $x_1\in\Mat{\bigidx-\smallidx\times k}$,
$x_2\in\Mat{\bigidx-\smallidx\times\smallidx-1-k}$,
$x_3,x_4\in\Mat{\bigidx-\smallidx\times1}$. 
Let $y\in
\overline{V_k''}$ and put
\begin{align*}
v=[y]_{\mathcal{E}_{H_{\bigidx}}}=\left(\begin{array}{ccccccccc}
I_{\bigidx-\smallidx}\\
&I_k\\
&0&I_{\smallidx-1-k}\\
&v_1&0&1\\
&v_2&0&0&1\\
&v_3&0&0&0&1\\
&v_4&0&0&0&0&I_{\smallidx-1-k}&\\
&v_5&v_4'&v_3'&v_2'&v_1'&0&I_k\\
&&&&&&&&I_{\bigidx-\smallidx}\\
\end{array}\right).
\end{align*}
Also let
\begin{align*}
&\rconj{w_{\smallidx,\bigidx}}m=\left(\begin{array}{ccccc}
I_{\bigidx-\smallidx}&&&\gamma^{-1}m'\\
&I_k&&&\gamma^{-1}m\\
&&I_{2(\smallidx-k)+1}\\
&&&I_k\\
&&&&I_{\bigidx-\smallidx}\\
\end{array}\right).
\end{align*}

By an explicit calculation
$\rconj{(\rconj{w_{\smallidx,\bigidx}}m)^{-1}}(rv)=ur_mv$, where
$u\in U_{\bigidx-\smallidx}$ is such that
$\rconj{(w_{\smallidx,\bigidx}^{-1}\omega_{\bigidx-k,k}^{-1})}u$ is
the image in $M_{\bigidx}$ of
\begin{align*}
\left(\begin{array}{cccc}
I_{\smallidx-1-k}\\&1&-v_1m\\&&I_{\bigidx-\smallidx}\\&&&I_k\\\end{array}\right)\in
Z_{\bigidx},
\end{align*}
$r_m\in R_{\smallidx,\bigidx}$ and its first $\bigidx-\smallidx$ rows are
\begin{align*}
\left(\begin{array}{ccccccc}
I_{\bigidx-\smallidx}&x_1+\ldots&x_2+\ldots&x_3+\gamma^{-1}m'v_3'&x_4+\ldots&0_{\smallidx}&z+\ldots\\\end{array}\right).
\end{align*}
By a change of variables $r_m\mapsto r$. This changes the character
$\psi_{\gamma}(r)=\psi((x_3)_{\bigidx-\smallidx})$ to
$\psi_{\gamma}(r)\psi(-(\gamma^{-1}m'v_3')_{\bigidx-\smallidx})$.
However, as we show next this change to $\psi_{\gamma}$ is canceled
by $u$. Indeed,
$f_s(\rconj{(w_{\smallidx,\bigidx}^{-1}\omega_{\bigidx-k,k}^{-1})}u,1)=\psi(-(v_1m)_1)f_s(1,1)$
and
\begin{align*}
\psi(-(v_1m)_1)\psi(-(\gamma^{-1}m'v_3')_{\bigidx-\smallidx})=1.
\end{align*}
To
see this one needs to inspect the coordinates $v_1,v_3'$ which
correspond to certain coordinates of $y$. Let
\begin{align*}
&y=\left(\begin{array}{cccccc}I_k&\\0&I_{\smallidx-1-k}\\y_1&0&1&&&\\y_2&0&0&1&&\\y_4&0&0&0&I_{\smallidx-1-k}&\\y_5&y_4'&*&*&0&I_k\\\end{array}\right).
\end{align*}
Then
\begin{align*}
&(v_1,v_2,v_3)=\begin{cases}(y_1-\frac1{2\beta^2}y_2,\beta y_1+\frac1{2\beta}y_2,-\gamma y_1+\quarter y_2)&\text{split
$G_{\smallidx}$,}\\
(y_2,y_1,- \gamma y_2)&\text{\quasisplit\ $G_{\smallidx}$}
\end{cases}
\end{align*}
and $v_i'=-J_{k}\transpose{v_i}$ for $1\leq i\leq3$. In the split
case,
\begin{align*}
-\gamma^{-1}m'v_3'=
-\gamma^{-1}(-J_{\bigidx-\smallidx}\transpose{m}J_k)(\gamma J_k\transpose{y_1}-\quarter J_k\transpose{y_2})=J_{\bigidx-\smallidx}\transpose{(y_1m)}-\frac1{2\beta^2}J_{\bigidx-\smallidx}\transpose{(y_2m)},
\end{align*}
so
\begin{align*}
-(\gamma^{-1}m'v_3')_{\bigidx-\smallidx}=(y_1m)_1-\frac1{2\beta^2}(y_2m)_1=(v_1m)_1.
\end{align*}
In the \quasisplit\ case,
\begin{align*}
-(\gamma^{-1}m'v_3')_{\bigidx-\smallidx}=(m'(-J_k\transpose{y_2}))_{\bigidx-\smallidx}=(y_2m)_1=(v_1m)_1.
\end{align*}

Combining the above manifests the requested form of the integral
\begin{equation*}
\int_{\overline{V_k''}} \int_{\Mat{k\times \bigidx-\smallidx}}
\int_{R_{\smallidx,\bigidx}}f_s(\omega_{\bigidx-k,k}w_{\smallidx,\bigidx}rv(\rconj{w_{\smallidx,\bigidx}}m)g,1)\psi_{\gamma}(r)drdmdv.\qedhere
\end{equation*}
\end{proof} 

\begin{proof}[Proof of Claim~\ref{claim:l<n first var inner integral r and v}] 
Begin with
\begin{align*}
&\int_{\overline{V_k''}} \int_{\Mat{k\times \bigidx-\smallidx}}
\int_{R_{\smallidx,\bigidx}}f_s(\omega_{\bigidx-k,k}w_{\smallidx,\bigidx}rv(\rconj{w_{\smallidx,\bigidx}}m)g,1)\psi_{\gamma}(r)drdmdv.
\end{align*}
We use the forms of $r$ and $v$ from the proof of
Claim~\ref{claim:l<n first var inner integral m r and v}. Recall
that
$\omega_{\bigidx-k,k}w_{\smallidx,\bigidx}=w_{\smallidx-k,\bigidx}w$.
We see that $\rconj{w^{-1}}(rv)=ur'_{v_1}$, where $u\in U_k$,
$\rconj{w_{\smallidx-k,\bigidx}^{-1}}u$ is the image in $
M_{\bigidx}$ of
\begin{align*}
\left(\begin{array}{cccc}I_{\smallidx-1-k}&&&\\
&1&&v_1\\
&&I_{\bigidx-\smallidx}&\\
&&&I_k\\\end{array}\right)\in Z_{\bigidx},
\end{align*}
$r'_{v_1}\in R_{\smallidx-k,\bigidx}$ and its first $\bigidx-\smallidx+k$ rows are
\begin{align*}
\left(\begin{array}{ccccccccc}
I_k&0&\gamma^{-1}v_4'&\gamma^{-1}v_3'&(-1)^k\gamma^{-1}v_2'&0_{\smallidx-k}&\gamma^{-1}(x_1'-v_1'x_3')&\gamma^{-2}(v_5-v_1'v_3)\\
0&I_{\bigidx-\smallidx}&x_2&x_3&(-1)^kx_4&0_{\smallidx-k}&z&\gamma^{-1}x_1+\ldots\\\end{array}\right).
\end{align*}
Using a variables change in $x_1$ and $v_5$ (which does not change
$\psi_{\gamma}(r)$), $r'_{v_1}\mapsto r'$, i.e., the dependence on
$v_1'$ is removed. In addition, by inspecting the coordinates of $v$
as in the proof of Claim~\ref{claim:l<n first var inner integral m r
and v} one sees that for general $r$ and $v$, $r'$ so obtained is a
general element of $R_{\smallidx-k,\bigidx}$. The measure $dv$ is
given by $dv=dy_1dy_2dy_4dy_5$. We define the measure $dr'=drdv$.

In addition, if $\smallidx<\bigidx$,
$f_s(\rconj{w_{\smallidx-k,\bigidx}^{-1}}u,1)=f_s(1,1)$ and
$\psi_{\gamma}(r)=\psi((x_3)_{\bigidx-\smallidx})=\psi_{\gamma}(r')$,
where $\psi_{\gamma}$ on the \rhs\ designates the character of
$N_{\bigidx-(\smallidx-k)}$ (restricted to
$R_{\smallidx-k,\bigidx}$). For $\smallidx=\bigidx$, in which case
the $dr$-integration is missing ($R_{\smallidx,\bigidx}=\{1\}$),
$f_s(\rconj{w_{\smallidx-k,\bigidx}^{-1}}u,1)=\psi((v_1)_1)f_s(1,1)$
and one sees that
$\psi((v_1)_1)=\psi((\gamma^{-1}v_3')_k)=\psi_{\gamma}(r')$.

Therefore the integral equals
\begin{equation*}
\int_{\Mat{k\times \bigidx-\smallidx}}
\int_{R_{\smallidx-k,\bigidx}}f_s(w_{\smallidx-k,\bigidx}r'w(\rconj{w_{\smallidx,\bigidx}}m)g,1)\psi_{\gamma}(r')dr'dm.\qedhere
\end{equation*}
\end{proof} 

\begin{proof}[Proof of Claim~\ref{claim:l<n functional eq pi prime}] 
The argument used for proving Claim~\ref{claim:n_1 < l < n_2
functional eq pi and tau_2} applies here as well, except that
Proposition~\ref{proposition:integral can be made constant zeta phi
version} needs to be replaced and the domains of convergence are different.

Exactly as in \cite{Soudry} (p.68) we let
$W_{\sigma}\in\Whittaker{\sigma}{\psi^{-1}}$,
$W_{\tau}\in\Whittaker{\tau}{\psi}$ be such that the
$\GL{k}\times\GL{\bigidx}$ integral on the right hand side of
\eqref{eq:func eq glm gln} converges absolutely for all $s$ and
equals $1$. Such functions exist by \cite{JPSS} (Section~2.7), in
fact one simply takes $W_{\tau}$ such that the support of
$\frestrict{(diag(I_k,J_{\bigidx-k})\cdot
\widetilde{W_{\tau}})}{Y_{\bigidx}}$ belongs to a small compact
neighborhood of the identity. We may use $W_{\sigma}$ and $W_{\tau}$
to construct $\varphi_{\zeta}$ and $f_s$ (resp.), as in the proof of
Claim~\ref{claim:l<n functional eq sigma}. The resulting integral
\eqref{eq:1st var l<n after func eq sigma} converges absolutely and
equals $1$, for all $s$ (we need to multiply $f_s$ by
$(c')^{-1}P_s(\gamma)^{-1}$ to get $1$, recall that $c'$ is
independent of $s$ and
$P_s(\gamma)=|\gamma|^{\smallidx(\half\bigidx+s-\half)}$, so $f_s$
is still a holomorphic section). In addition if we replace
$\varphi_{\zeta},f_s$ with $|\varphi_{\zeta}|,|f_s|$ and remove
$\psi_{\gamma}$, the integral is bounded
by a constant $C$ independent of $s$. The same properties hold for \eqref{eq:1st var l<n before eq for pi prime}. 

Regarding the domains of convergence, let $D^*$ be the domain of
absolute convergence of \eqref{eq:1st var l<n after eq for pi
prime}. It takes the form
$\{(1-C_2)\Re(\zeta)<<\Re(1-s)<<\Re(\zeta)\}$. As in the proof of
Claim~\ref{claim:n_1 < l < n_2 functional eq pi and tau_2}, let
$\widetilde{I}_{m,g}(s)$ denote the inner $dr'dg'$-integration of
\eqref{eq:1st var l<n after eq for pi prime} ($m\in \Mat{k\times
\bigidx-\smallidx}$, $g\in
\lmodulo{\overline{V_k}Z_kG_{\smallidx-k}}{G_{\smallidx}}$). Its
meromorphic continuation $\widetilde{Q}_{m,g}\in\C(q^{-s})$
satisfies $\widetilde{I}_{m,g}(s)=\widetilde{Q}_{m,g}(q^{-s})$ for
all $s\in\Omega=\setof{s\in\C}{\Re(s)<s_1}$. Taking $\Re(\zeta)$
large enough, we may assume $D^*\subset \Omega$ and also that
$\gamma(\pi'\times\tau,\psi,s)$ does not have a pole in $D^*$.

Now we continue as in Claim~\ref{claim:n_1 < l < n_2
functional eq pi and tau_2}.
\end{proof} 

\begin{proof}[Proof of Claim~\ref{claim:k=l>n+1 functional eq sigma}] 
Recall that integral~\eqref{int:mult 1st var l=k>n+1 before applying
func eq sigma and tau} is
\begin{align*}
&\int_{\overline{U_{\bigidx}}}\int_{U_{\bigidx}}\int_{V_{\smallidx}''}
\int_{\lmodulo{Z_{\bigidx}}{\GL{\bigidx}}}
\int_{\Mat{\smallidx-\bigidx-1\times\bigidx}}
\varphi_{\zeta}(v''w^{\smallidx,\bigidx}u_0u,ard_{\gamma})f_s(u,a)\\\notag&\absdet{a}^{s-\zeta+\frac{\bigidx-\smallidx}2}\psi_{\gamma}(v'')
drdadv''du_0du.
\end{align*}
Let $W_{\sigma}\in\Whittaker{\sigma}{\psi^{-1}}$,
$W_{\tau}\in\Whittaker{\tau}{\psi}$ be arbitrary. Define
$\varphi_{\zeta}$ using $W_{\sigma}$ as in the proof of
Claim~\ref{claim:l<n functional eq sigma} with
$\support{\varphi_{\zeta}}=\overline{P_{\smallidx}}w^{\smallidx,\bigidx}\mathcal{N}_{G_{\smallidx},k_1}$,
$k_1>>0$. Similarly define $f_s=ch_{\mathcal{N}_{H_{\bigidx},k_2},W_{\tau},s}$, 
$k_2>k_1$. Then $f_s(u,a)=0$ unless $u\in
\mathcal{N}_{H_{\bigidx},k_2}$ and the $du$-integration can be
disregarded. Now
$\varphi_{\zeta}(v''w^{\smallidx,\bigidx}u_0,ard_{\gamma})=0$ unless
$v''\rconj{(w^{\smallidx,\bigidx})^{-1}}u_0\in
\overline{P_{\smallidx}}\mathcal{N}_{G_{\smallidx},k_1}$. 
Write
\begin{align*}
u_0=\left(\begin{array}{ccc}I_{\bigidx}&c_1&c_2\\&1&c_1'\\&&I_{\bigidx}\end{array}\right)\in
U_{\bigidx}.
\end{align*}
Then $v''\rconj{(w^{\smallidx,\bigidx})^{-1}}u_0=b_1y$ where $b_1\in
L_{\smallidx}$ and $y\in V_{\smallidx}$ is of the form
\begin{align*}
\left(\begin{array}{cccccc}
I_{\bigidx}&0&0&\frac1{2\beta}c_1&*&c_2-\half c_1c_1'\\
&I_{\smallidx-\bigidx-1}&0&*&*&*\\
&&1&0&*&\frac1{2\beta}c_1'\\
&&&1&0&0\\
&&&&I_{\smallidx-\bigidx-1}&0\\
&&&&&I_{\bigidx}\\
\end{array}\right).
\end{align*}
Hence
$v''\rconj{(w^{\smallidx,\bigidx})^{-1}}u_0\in\overline{P_{\smallidx}}\mathcal{N}_{G_{\smallidx},k_1}$
if and only if
$u_0,v''\in\mathcal{N}_{G_{\smallidx},k_1}$.
Also for $u_0,v''\in\mathcal{N}_{G_{\smallidx},k_1}$, $\varphi_{\zeta}(v''w^{\smallidx,\bigidx}u_0,ard_{\gamma})=W_{\sigma}(ard_{\gamma})$. Thus the $dv''du_0$-integration can be ignored. 
The integral becomes
\begin{align*}
c\int_{\lmodulo{Z_{\bigidx}}{\GL{\bigidx}}}
\int_{\Mat{\smallidx-\bigidx-1\times\bigidx}}
W_{\sigma}(ard_{\gamma})W_{\tau}(a)\absdet{a}^{s-\zeta+\frac{\bigidx-\smallidx}2}drda.
\end{align*}
Here $c>0$ depends on $k_1$ and $k_2$. Applying the same arguments
to integral~\eqref{int:mult 1st var k=l>n+1 after applying func eq
sigma and tau} we find that it equals
\begin{align*}
c\int_{\lmodulo{Z_{\bigidx}}{\GL{\bigidx}}} W_{\sigma}(
\left(\begin{array}{cc}0&I_{\smallidx-\bigidx}\\a&0\\\end{array}\right)d_{\gamma})
W_{\tau}(a)\absdet{a}^{s-\zeta-\frac{\bigidx-\smallidx}2-1}da.
\end{align*}
By \eqref{eq:modified func eq glm gln} the integrals are
proportional by
$\omega_{\tau}(-1)^{\smallidx-1}\gamma(\sigma\times\tau,\psi^{-1},s-\zeta)$.
\end{proof} 

\begin{proof}[Proof of Claim~\ref{claim:k=l>n+1 W part invariant for u bar}] 
Recall the definition
\begin{align*}
F(h)=\int_{V_{\smallidx}''}\varphi_{\zeta}(v''w^{\smallidx,\bigidx}h,
\omega_{\smallidx-\bigidx,\bigidx}d_{\gamma})\psi_{\gamma}(v'')dv''.
\end{align*}
For $u\in\overline{U_{\bigidx}}$ with
\begin{align*}
[u]_{\mathcal{E}_{H_{\bigidx}}}=\left(\begin{array}{ccc}I_{\bigidx}\\u_1&1\\u_2&u_1'&I_{\bigidx}\\\end{array}\right),
\end{align*}
\begin{align*}
\rconj{(w^{\smallidx,\bigidx})^{-1}}u=
\left(\begin{array}{cccccc}I_{\bigidx}\\0&I_{\smallidx-\bigidx-1}\\\frac1{2\beta}u_1&0&1\\\beta
u_1&0&0&1\\0&0&0&0&I_{\smallidx-\bigidx-1}\\u_2&0&\beta
u_1'&\frac1{2\beta}u_1'&0&I_{\bigidx}\\\end{array}\right).
\end{align*}
Let
\begin{align*}
v''=\left(\begin{array}{cccccc}
I_{\bigidx}&0&0&0&v_4&0\\
&I_{\smallidx-\bigidx-1}&0&v_3&v_5&v_4'\\
&&1&0&v_3'&0\\
&&&1&0&0\\
&&&&I_{\smallidx-\bigidx-1}&0\\
&&&&&I_{\bigidx}\\
\end{array}\right)\in V_{\smallidx}''.
\end{align*}
Then $\rconj{(\rconj{(w^{\smallidx,\bigidx})^{-1}}u)}v''=b_uv_u$,
where $b_u$ is the image in $L_{\smallidx}$ of
\begin{align*}
\left(\begin{array}{ccc} I_{\bigidx}&&\\\beta
v_3u_1+v_4'u_2&I_{\smallidx-\bigidx-1}&\beta
v_4'u_1'\\&&1\\\end{array}\right)\in\GL{\smallidx}
\end{align*}
and
\begin{align*}
v_u=\left(\begin{array}{cccccc}
I_{\bigidx}&0&0&0&v_4&0\\
&I_{\smallidx-\bigidx-1}&0&v_3+\frac1{2\beta}v_4'u_1'&v_5+\ldots&v_4'\\
&&1&0&v_3'-\frac1{2\beta}u_1v_4&0\\
&&&1&0&0\\
&&&&I_{\smallidx-\bigidx-1}&0\\
&&&&&I_{\bigidx}\\
\end{array}\right)\in V_{\smallidx}''.
\end{align*}
It follows that
\begin{align*}
F(uh)=\int_{V_{\smallidx}''}\varphi_{\zeta}((\rconj{(w^{\smallidx,\bigidx})^{-1}}u)b_uv_uw^{\smallidx,\bigidx}h,
\omega_{\smallidx-\bigidx,\bigidx}d_{\gamma})\psi_{\gamma}(v'')dv''.
\end{align*}
Now on the one hand, changing variables in $v_u$ removes the
dependence on $u$ and changes $\psi_{\gamma}(v'')=\psi(-\gamma
(v_3)_{\smallidx-\bigidx-1})$ to
$\psi_{\gamma}(v'')\psi(\frac{\beta}4(v_4'u_1')_{\smallidx-\bigidx-1})$.
On the other hand,
\begin{align*}
\varphi_{\zeta}((\rconj{(w^{\smallidx,\bigidx})^{-1}}u)b_u,
\omega_{\smallidx-\bigidx,\bigidx}d_{\gamma})
=\psi^{-1}(\frac2{\beta}(u_1)_1+\frac{\beta}4(v_4'u_1')_{\smallidx-\bigidx-1})\varphi_{\zeta}(1,
\omega_{\smallidx-\bigidx,\bigidx}d_{\gamma}).
\end{align*}
We conclude $F(uh)=\psi^{-1}(\frac2{\beta}(u_1)_1)F(h)$, or when
considering $u$ in coordinates relative to
$\mathcal{E}_{H_{\bigidx}}$,
$F(uh)=\psi^{-1}(\frac2{\beta}u_{\bigidx+1,1})F(h)$ as claimed.
\end{proof} 

%% file: chapter_archimedean_results.tex
\newtheorem{theorem}{Theorem}[section]
\newtheorem{proposition}{Proposition}[section]
\newtheorem{corollary}{Corollary}[section]
\newtheorem{lemma}{Lemma}[section]
\newtheorem{claim}{Claim}[section]
\theoremstyle{remark}
\newtheorem{remark}{Remark}[section]
\newtheorem{example}{Example}[section]
\theoremstyle{definition}
\newtheorem{definition}{Definition}[section]
\numberwithin{equation}{section}
\newcommand{\chapter}{\section} 
\input{thesis_notations}
\end{comment}

\chapter{Archimedean places}\label{chapter:archimedean results}
In this chapter the field $F$ is either $\R$ or $\C$.

\section{The \archimedean\ integrals}\label{subsection:the archimendea integrals}
The groups and embeddings are defined as described in
Section~\ref{subsection:groups in study}.
If $F=\C$, $G_{\smallidx}$ is
always split.

Let $\pi$ and $\tau$ be a pair of continuous representations of
$G_{\smallidx}$ and $\GL{\bigidx}$ (resp.), on Fr\'{e}chet spaces,
which are smooth, of moderate growth, admissible, finitely generated
and generic. Assume that $\pi$ is realized in
$\Whittaker{\pi}{\psi_{\gamma}^{-1}}$, where $\psi_{\gamma}$ is
defined as in Section~\ref{subsection:the integrals}, and $\tau$ is
realized in $\Whittaker{\tau}{\psi}$.

For $s\in\C$ form the representation
$\cinduced{Q_{\bigidx}}{H_{\bigidx}}{\tau\alpha^s}$ on the space
$V(\tau,s)=V_{Q_{\bigidx}}^{H_{\bigidx}}(\tau,s)$, consisting of the
smooth functions $f_s:H_{\bigidx}\times \GL{\bigidx}\rightarrow\C$
such that for $q=au\in Q_{\bigidx}$ ($a\in\GL{\bigidx}\isomorphic
M_{\bigidx}$, $u\in U_{\bigidx}$), $h\in H_{\bigidx}$ and
$b\in\GL{\bigidx}$,
$f_s(qh,b)=\delta_{Q_{\bigidx}}^{\half}(q)\absdet{a}^{s-\half}f_s(h,ba)$,
and the mapping $b\mapsto f_s(h,b)$ lies in
$\Whittaker{\tau}{\psi}$.

Fix $s\in\C$ and let $W\in\Whittaker{\pi}{\psi_{\gamma}^{-1}}$,
$f_s\in V(\tau,s)$. The Rankin-Selberg integral $\Psi(W,f_s,s)$ is
given by Definition~\ref{definition:shape of local integral}. It is
absolutely convergent for $\Re(s)>>0$.

\section{The functional equation and $\gamma$-factor}
Until the end of this chapter, assume that $G_{\smallidx}$ is split.
Let $\pi$ and $\tau$ be irreducible representations. We have the
following proposition, essentially proved by Soudry \cite{Soudry3}.
\begin{proposition}\label{proposition:archimedean split results}
\begin{enumerate}[leftmargin=*]
\item The integrals $\Psi(W,f_s,s)$ admit a meromorphic continuation
to $\C$ and this continuation defines a continuous bilinear form on
$\Whittaker{\pi}{\psi_{\gamma}^{-1}}\times V(\tau,s)$ satisfying
\eqref{eq:bilinear special condition}.
\item There is a
meromorphic function $\gamma(\pi\times\tau,\psi,s)$ satisfying for
all $W$ and $f_s$,
\begin{align*}
\gamma(\pi\times\tau,\psi,s)\Psi(W,f_s,s)=c(\smallidx,\tau,\gamma,s)\Psi(W,\nintertwiningfull{\tau}{s}f_s,1-s).
\end{align*}
Here
$c(\smallidx,\tau,\gamma,s)=\omega_{\tau}(\gamma)^{-2}|\gamma|^{-2\bigidx(s-\half)}$
if $\bigidx<\smallidx$ and $c(\smallidx,\tau,\gamma,s)=1$ otherwise
(this is \eqref{eq:gamma def} with the same factor
$c(\smallidx,\tau,\gamma,s)$).
\item Write $\pi=\cinduced{P}{G_{\smallidx}}{\sigma}$
with $\sigma$ realized in $\Whittaker{\sigma}{\psi^{-1}}$, where $P$
is either $\overline{P_{\smallidx}}$ or
$\rconj{\kappa}\overline{P_{\smallidx}}$. Then
\begin{align*}
\gamma(\pi\times\tau,\psi,s)=
\omega_{\sigma}(-1)^{\bigidx}\omega_{\tau}(-1)^{\smallidx}\omega_{\tau}(2\gamma)^{-1}\gamma(\sigma\times\tau,\psi,s)\gamma(\sigma^*\times\tau,\psi,s).
\end{align*}
\end{enumerate}
\end{proposition}

Write $\pi=\cinduced{P}{G_{\smallidx}}{\sigma}$ as in the
proposition and define
\begin{align*}
\Gamma(\pi\times\tau,\psi,s)=\omega_{\sigma}(-1)^{\bigidx}\omega_{\tau}(-1)^{\smallidx}\omega_{\tau}(2\gamma)\gamma(\pi\times\tau,\psi,s).
\end{align*}
Then
\begin{align*}
\Gamma(\pi\times\tau,\psi,s)=\gamma(\sigma\times\tau,\psi,s)\gamma(\sigma^*\times\tau,\psi,s).
\end{align*}
We conclude that in the \archimedean\ places (where $G_{\smallidx}$
is split), $\Gamma(\pi\times\tau,\psi,s)$ is identical with the corresponding $\gamma$-factor of Shahidi 
on $SO_{2\smallidx}\times\GL{\bigidx}$ defined in \cite{Sh3}.

As stated above, the proposition was proved in \cite{Soudry3}, but
the integrals are not exactly the ones we use. We provide a short
sketch of the proof, focusing on the modifications to the integrals
and formal identities.
\begin{proof}[Proof of Proposition~\ref{proposition:archimedean
split results}] 
Let $\pi=\cinduced{\overline{P_{\smallidx}}}{G_{\smallidx}}{\sigma}$
(the proof for
$\pi=\cinduced{\rconj{\kappa}\overline{P_{\smallidx}}}{G_{\smallidx}}{\sigma}$
is similar and skipped). We introduce an auxiliary complex parameter
$\zeta$ and consider the space
$V_{\overline{P_{\smallidx}}}^{G_{\smallidx}}(\sigma,-\zeta+\half)$.
A Whittaker functional on this space is given by \eqref{eq:whittaker
for pi, k=l}, i.e., for $\varphi_{\zeta}\in
V_{\overline{P_{\smallidx}}}^{G_{\smallidx}}(\sigma,-\zeta+\half)$,
\begin{align*}
\whittakerfunctional(\varphi_{\zeta})=\int_{V_{\smallidx}}\varphi_{\zeta}(v,diag(I_{\smallidx-1},4))\psi_{\gamma}(v)dv.
\end{align*}
The Whittaker function
$W_{\varphi_{\zeta}}\in\Whittaker{\cinduced{\overline{P_{\smallidx}}}{G_{\smallidx}}{\sigma\alpha^{-\zeta+\half}}}{\psi_{\gamma}^{-1}}$
is defined by $W_{\varphi_{\zeta}}(g)=\whittakerfunctional(g\cdot
\varphi_{\zeta})$.

Let $s\in\C$. For any $\mu\in F^*$, define a generic character
$\psi_{\mu}$ of $U_{H_{\bigidx}}$ by
$\psi_{\mu}(u)=\psi(-\sum_{i=1}^{\bigidx-1}u_{i,i+1}+\mu
u_{\bigidx,\bigidx+1})$. A Whittaker functional on $V(\tau,s)$ with
respect to $\psi_{\mu}$ is given by
\begin{align*}
\whittakerfunctional'(f_s)=\int_{U_{\bigidx}}f_s(w_{\bigidx}u,I_{\bigidx})\psi^{-1}(\mu
u_{\bigidx,\bigidx+1})du,
\end{align*}
where $f_s\in V(\tau,s)$. Then
$W_{f_s}(h)=\whittakerfunctional'(h\cdot f_s)$. Our normalization of
$\nintertwiningfull{\tau}{s}$ was defined by the functional equation
\eqref{eq:Shahidi local coefficient def}, where the Whittaker
functional on $V(\tau,s)$ corresponded to the character
$u\mapsto\psi(\sum_{i=1}^{\bigidx}u_{i,i+1})$. Keeping the same
normalization we have
\begin{align}\label{eq:archimedean functional equation for Whittaker
on indued from tau}
W_{f_s}=|\mu|^{2\bigidx(s-\half)}\omega_{\tau}(\mu)^{2}W_{\nintertwiningfull{\tau}{s}f_s}.
\end{align}
This is obtained from \eqref{eq:Shahidi local coefficient def} by
substituting $y\cdot f_s$ for $f_s$ where
\begin{align*}
y=diag(d_{\bigidx}^*,1,d_{\bigidx})diag(\mu(-1)^{\bigidx}I_{\bigidx},1,\mu^{-1}(-1)^{\bigidx}I_{\bigidx}).
\end{align*}

First assume $\smallidx\leq\bigidx$. The following equality is a
minor modification of equality~(7.8) of \cite{Soudry3} (Section~7).
\begin{align}\label{eq:arhcimedean l<=n, even l}
\Psi(w\cdot W_{\varphi_{\zeta}},m\cdot
f_s,s)=C\gamma(\sigma^*\times\tau^*,\psi,1+\zeta-s)A(W_{f_s},\varphi_{\zeta},\zeta).
\end{align}
Here
\begin{align*}
&\mu=(-1)^{\smallidx}\beta,\\
&w=\begin{cases}w_{\smallidx}&\text{$\smallidx$ is even,}\\
I_{2\smallidx}&\text{$\smallidx$ is odd},\end{cases}\\\notag
&\widetilde{M}=diag(1,-\beta,1)M\in H_1\qquad(\text{$M$ - the matrix
defined in Section~\ref{subsection:G_l in H_n}}),\\
&m=\widetilde{M}^{-1}\cdot
diag(\gamma^{-1}I_{\bigidx-\smallidx},I_{2\smallidx+1},\gamma
I_{\bigidx-\smallidx})w_{\bigidx-\smallidx,\bigidx},\\
&C=|2|^{3\smallidx+2\zeta-3}
\omega_{\sigma}(\gamma)^{-1}|\gamma|^{-\smallidx(\half\smallidx-\bigidx+\half-\zeta)}\omega_{\sigma^*}(-1)^{\bigidx-1};
\end{align*}
$A(W_{f_s},\varphi_{\zeta},\zeta)$ is the Rankin-Selberg integral
for $SO_{2\bigidx+1}\times\GL{\smallidx}$ and
$\cinduced{Q_{\bigidx}}{H_{\bigidx}}{\tau\alpha^s}\times
\sigma\alpha^{-\zeta+\half}$ of \cite{Soudry,Soudry3} defined by
\begin{align*}
A(W_{f_s},\varphi_{\zeta},\zeta)=\begin{dcases}
\int_{\lmodulo{U_{G_{\smallidx}}}{G_{\smallidx}}}\int_{X_{(\smallidx,\bigidx)}}
W_{f_s}(xj_{\smallidx,\bigidx}(g))\varphi_{\zeta}(w_{\smallidx}g,I_{\smallidx})dxdg&\text{even
$\smallidx$,}\\
\int_{\lmodulo{\overline{U_{G_{\smallidx}}}}{G_{\smallidx}}}\int_{X_{(\smallidx,\bigidx)}}
W_{f_s}(xj_{\smallidx,\bigidx}(\rconj{w_{\smallidx}}g))\varphi_{\zeta}(g,I_{\smallidx})dxdg
&\text{odd $\smallidx$,}\end{dcases}
\end{align*}
where
\begin{align*}
X_{(\smallidx,\bigidx)}=\{\left(\begin{array}{ccccc}I_{\smallidx}\\x&I_{\bigidx-\smallidx}\\&&1\\&&&I_{\bigidx-\smallidx}\\&&&x'&I_{\smallidx}\\\end{array}\right)\}<H_{\bigidx}
\end{align*}
and 
\begin{align*}
j_{\smallidx,\bigidx}(\left(\begin{array}{cc}A&B\\C&D\end{array}\right))=\left(\begin{array}{ccc}A&&B\\&I_{2(\bigidx-\smallidx)+1}\\C&&D\end{array}\right)\in
H_{\bigidx}\qquad
(\left(\begin{array}{cc}A&B\\C&D\end{array}\right)\in
G_{\smallidx}).
\end{align*}

Equality~\eqref{eq:arhcimedean l<=n, even l} holds for
$\Re(\zeta),\Re(s)>>0$, where $\Psi(w\cdot
W_{\varphi_{\zeta}},m\cdot f_s,s)$ is absolutely convergent, and
$A(W_{f_s},\varphi_{\zeta},\zeta)$ is defined by meromorphic
continuation.

Let $B(\zeta,s)$ denote the space of continuous bilinear forms on
$\Whittaker{\cinduced{\overline{P_{\smallidx}}}{G_{\smallidx}}{\sigma\alpha^{-\zeta+\half}}}{\psi_{\gamma}^{-1}}\times
V(\tau,s)$ satisfying \eqref{eq:bilinear special condition}. Soudry
\cite{Soudry3} (Section~3) proved that except for a discrete set of
$\zeta$ and $s$, this space is at most one dimensional.

Soudry \cite{Soudry3} (Section~5) proved that
$A(W_{f_s},\varphi_{\zeta},\zeta)$, as a function of $\zeta$ and
$s$, has a meromorphic continuation to $\C^2$ and this continuation,
in its domain of definition, defines a continuous bilinear form in
$B(\zeta,s)$. In addition, the integral $\Psi(w\cdot
W_{\varphi_{\zeta}},m\cdot f_s,s)$ in its domain of absolute
convergence, i.e., for $\Re(\zeta),\Re(s)>>0$, also belongs to
$B(\zeta,s)$. The continuity of the integral is proved similarly to
the proof of Lemma~1 of \cite{Soudry3} (Section~6). Hence for
$\Re(\zeta),\Re(s)>>0$, the meromorphic continuation of
$A(W_{f_s},\varphi_{\zeta},\zeta)$ and the integral $\Psi(w\cdot
W_{\varphi_{\zeta}},m\cdot f_s,s)$ belong to $B(\zeta,s)$, thus they
are proportional. The proportionality factor is
$C\gamma(\sigma^*\times\tau^*,\psi,1+\zeta-s)$.

Now it immediately follows that $\Psi(w\cdot
W_{\varphi_{\zeta}},m\cdot f_s,s)$ has a meromorphic continuation to
$\C^2$ which, when defined, belongs to $B(\zeta,s)$. Then the first
assertion of the proposition follows by substituting $0$ for
$\zeta$.

When we replace $f_s$ with $\nintertwiningfull{\tau}{s}f_s$ in
\eqref{eq:arhcimedean l<=n, even l} we obtain
\begin{align}\label{eq:arhcimedean l<=n, even l with intertwiner}
&\Psi(w\cdot W_{\varphi_{\zeta}},m\cdot
\nintertwiningfull{\tau}{s}f_s,1-s)\\\notag&=C\gamma(\sigma^*\times\tau,\psi,1+\zeta-(1-s))A(W_{\nintertwiningfull{\tau}{s}f_s},\varphi_{\zeta},\zeta).
\end{align}
This implies that $\Psi(w\cdot W_{\varphi_{\zeta}},m\cdot
\nintertwiningfull{\tau}{s}f_s,1-s)$ also extends to a meromorphic
function on $\C^2$, belonging to $B(\zeta,s)$ for all $\zeta$ and
$s$ where it is defined. Now the uniqueness properties of
$B(\zeta,s)$ imply the existence of the functional equation
\eqref{eq:gamma def} in the \archimedean\ case. 

The function $W_{\nintertwiningfull{\tau}{s}f_s}$ appearing in
$A(W_{\nintertwiningfull{\tau}{s}f_s},\varphi_{\zeta},\zeta)$ is
defined by meromorphic continuation. Therefore we can use
\eqref{eq:archimedean functional equation for Whittaker on indued
from tau} and obtain, noting that
$\omega_{\tau}(\beta)^2=\omega_{\tau}(2\gamma)$ and by our
assumptions $|\beta|=1$,
\begin{align*}
A(W_{\nintertwiningfull{\tau}{s}f_s},\varphi_{\zeta},\zeta)=\omega_{\tau}(2\gamma)^{-1}A(W_{f_s},\varphi_{\zeta},\zeta).
\end{align*}
Now dividing \eqref{eq:arhcimedean l<=n, even l with intertwiner} by
\eqref{eq:arhcimedean l<=n, even l} and letting $\zeta\rightarrow0$
we conclude
\begin{align*}
\gamma(\pi\times\tau,\psi,s)=\omega_{\sigma}(-1)^{\bigidx}
\omega_{\tau}(-1)^{\smallidx}\omega_{\tau}(2\gamma)^{-1}\gamma(\sigma\times\tau,\psi,s)
\gamma(\sigma^*\times\tau,\psi,s).
\end{align*}

Assume $\smallidx>\bigidx$. In this case we can use the formal
manipulations from \eqref{int:mult 1st var l=k>n+1 starting form} to
\eqref{int:1var k=l>n+1 before applying Shahidi local coefficient}
(in Section~\ref{subsection:1st var k=l>n+1}, see
Remark~\ref{remark:l>n gamma proof uncovers H_n times GL_l
integral}) and deduce the following identity, which is a minor
modification of identity (6.10) of \cite{Soudry3} (Section~6).
\begin{align}\label{eq:arhcimedean l>n}
\Psi(W_{\varphi_{\zeta}},f_s,s)=\omega_{\tau}(-1)^{\smallidx-1}\gamma(\sigma\times\tau,\psi^{-1},s-\zeta)^{-1}
A(W_{f_s},\varphi_{\zeta},\zeta).
\end{align}
Here $A(W_{f_s},\varphi_{\zeta},\zeta)$ is the Rankin-Selberg
integral for $SO_{2\bigidx+1}\times\GL{\smallidx}$ and
$\cinduced{Q_{\bigidx}}{H_{\bigidx}}{\tau\alpha^s}\times
\sigma\alpha^{-\zeta+\half}$ of \cite{Soudry,Soudry3} defined by
\begin{align*}
&\int_{\lmodulo{\overline{U_{\bigidx}}Z_{\bigidx}}{H_{\bigidx}}}
W_{f_s}(w_{\bigidx}^{-1}h)
\int_{V_{\smallidx}''}\varphi_{\zeta}(v''w^{\smallidx,\bigidx}h,
\omega_{\smallidx-\bigidx,\bigidx}\cdot
diag(I_{\smallidx-1},4))\psi_{\gamma}(v'')dv''dh,
\end{align*}
where $\mu=\frac2{\beta}(-1)^{\bigidx+1}$,
\begin{align*}
&V_{\smallidx}''=\{\left(\begin{array}{cccccc}
I_{\bigidx}&0&0&0&v_4&0\\
&I_{\smallidx-\bigidx-1}&0&v_3&v_5&v_4'\\
&&1&0&v_3'&0\\
&&&1&0&0\\
&&&&I_{\smallidx-\bigidx-1}&0\\
&&&&&I_{\bigidx}\\
\end{array}\right)\}<G_{\smallidx}
\end{align*}
and $\psi_{\gamma}(v'')=\psi(-\gamma(v_3)_{\smallidx-\bigidx-1})$.

Equality~\eqref{eq:arhcimedean l>n} holds for
$\Re(\zeta),\Re(s)>>0$, where $A(W_{f_s},\varphi_{\zeta},\zeta)$ is
absolutely convergent and $\Psi(W_{\varphi_{\zeta}},f_s,s)$ is
defined by meromorphic continuation. In fact, the arguments of
Soudry \cite{Soudry3} (Section~5) show that
$\Psi(W_{\varphi_{\zeta}},f_s,s)$, which resembles
$A(W_{f_s},\varphi_{\zeta},\zeta)$ in the case
$\smallidx\leq\bigidx$, has a meromorphic continuation to $\C^2$
which belongs (when defined) to $B(\zeta,s)$. When
$\smallidx>\bigidx$, $A(W_{f_s},\varphi_{\zeta},\zeta)$ was proved
to be continuous in its domain of absolute convergence (\textit{loc.
cit.} Section~6, Lemma~1), hence in this domain it belongs to
$B(\zeta,s)$.

As above the result follows by substituting
$\nintertwiningfull{\tau}{s}f_s$ for $f_s$ in \eqref{eq:arhcimedean
l>n}. 
\end{proof} 

%% file: chapter_gcd.tex
\newtheorem{theorem}{Theorem}[section]
\newtheorem{proposition}{Proposition}[section]
\newtheorem{corollary}{Corollary}[section]
\newtheorem{lemma}{Lemma}[section]
\newtheorem{claim}{Claim}[section]
\theoremstyle{remark}
\newtheorem{remark}{Remark}[section]
\newtheorem{example}{Example}[section]
\theoremstyle{definition}
\newtheorem{definition}{Definition}[section]
\numberwithin{equation}{section}
\newcommand{\chapter}{\section} 
\input{thesis_notations}
\end{comment}

\chapter{The g.c.d.}\label{chapter:the gcd}
In this chapter we relate the g.c.d. to Shahidi's 
$L$-function in the tempered case and provide a lower bound for the
more general case, where $\pi$ is an arbitrary irreducible
representation and $\tau$ is tempered (as always, $\pi$ and $\tau$ are smooth, admissible and generic). This bound is sharp in the
sense that
it coincides with the upper bound (Theorem~\ref{theorem:gcd sub multiplicity first var}) up to the poles of $M_{\tau}(s)$. 
Several other properties proved here, for example the results of
Section~\ref{subsection:gcd for ramified twist of tau} concerning
highly ramified representations, will be used in
Chapter~\ref{chapter:upper boubnds on the gcd}.

Throughout this chapter $\tau$ is always irreducible.

\section{Computation of the g.c.d. for tempered representations}\label{subsection:tempered}
Assume that $\pi$ and $\tau$ are tempered representations (in particular, irreducible). The following proposition is the key ingredient in the proof of 
Theorem~\ref{theorem:gcd for tempered reps}. 
\begin{proposition}\label{proposition:convergence for tempered}
The integral $\Psi(W,f_s,s)$ with $f_s\in\xi(\tau,std,s)$ has no
poles for $\Re(s)>0$.
\end{proposition}
We derive the theorem first. Namely, we prove that $L(\pi\times\tau,s)^{-1}$ divides
$\gcd(\pi\times\tau,s)^{-1}$, $\gcd(\pi\times\tau,s)\in
L(\pi\times\tau,s)M_{\tau}(s)\C[q^{-s},q^s]$ and under a certain assumption on the intertwining operators, $\gcd(\pi\times\tau,s)=L(\pi\times\tau,s)$.

Recall that by
Corollary~\ref{corollary:gamma is identical to local coeff in split case},
$\gamma(\pi\times\tau,\psi,s)$ given by \eqref{eq:gamma def} equals
(up to an invertible factor in $\C[q^{-s},q^s]$) the corresponding
$\gamma$-factor of Shahidi. By Shahidi's definition of the
$\gamma$-factor \cite{Sh3} and \eqref{eq:gamma and epsilon},
\begin{align}\label{eq:back to back quotients gamm aL and gcd}
\frac{\gcd(\pi\times\tau^*,1-s)}{\gcd(\pi\times\tau,s)}\equalun\gamma(\pi\times\tau,\psi,s)\equalun\frac{L(\pi\times\tau^*,1-s)}{L(\pi\times\tau,s)}.
\end{align}
(Recall that $\equalun$ means up to invertible factors in $\C[q^{-s},q^s]$.)
Here the $L$-functions on the \rhs\ are the ones defined by
Shahidi. Casselman and Shahidi \cite{CSh} (Section 4) proved that
when $\pi$ and $\tau$ are (generic) tempered, the $L$-function
$L(\pi\times\tau,s)$ is holomorphic for $\Re(s)>0$. Similarly,
$L(\pi\times\tau^*,1-s)$ is holomorphic for $\Re(s)<1$ 
(since $\tau$ is tempered, $\dualrep{\tau}$ is also tempered so $\tau^*\isomorphic\dualrep{\tau}$ is tempered). Therefore the quotient on the \rhs\ is reduced. It follows immediately that
$L(\pi\times\tau,s)^{-1}$ divides $\gcd(\pi\times\tau,s)^{-1}$.

The integrals $\Psi(W,f_s,s)$ with $f_s\in\xi(\tau,hol,s)$
span a fractional ideal of $\C[q^{-s},q^s]$ which contains $1$,
according to Proposition~\ref{proposition:integral can be made
constant}. 
Thus there is a polynomial $P_0\in\C[X]$ with $P_0(0)=1$, of minimal degree
such that $P_0(q^{-s})\Psi(W,f_s,s)\in\C[q^{-s},q^s]$ for all $W$
and $f_s\in\xi(\tau,hol,s)$. Put
${\gcd}_0(\pi\times\tau,s)=P_0(q^{-s})^{-1}$. Then
${\gcd}_0(\pi\times\tau,s)^{-1}$ divides
$\gcd(\pi\times\tau,s)^{-1}$ and according to
Proposition~\ref{proposition:integral for good section} we can write
\begin{align*}
\gcd(\pi\times\tau,s)={\gcd}_0(\pi\times\tau,s)\ell_{\tau^*}(1-s)P(q^{-s},q^s),
\end{align*}
where $P\in\C[q^{-s},q^s]$ divides
$\ell_{\tau^*}(1-s)^{-1}$. Similarly
\begin{align*}
\gcd(\pi\times\tau^*,1-s)={\gcd}_0(\pi\times\tau^*,1-s)\ell_{\tau}(s)\widetilde{P}(q^{-s},q^s),
\end{align*}
here $\widetilde{P}\in\C[q^{-s},q^s]$ divides
$\ell_{\tau}(s)^{-1}$. Consider the
quotient
\begin{align*}
\frac{\gcd(\pi\times\tau^*,1-s)}{\gcd(\pi\times\tau,s)}=\frac{{\gcd}_0(\pi\times\tau^*,1-s)\ell_{\tau}(s)\widetilde{P}(q^{-s},q^s)}{{\gcd}_0(\pi\times\tau,s)\ell_{\tau^*}(1-s)P(q^{-s},q^s)}.
\end{align*}
Proposition~\ref{proposition:convergence for tempered} implies that
${\gcd}_0(\pi\times\tau^*,1-s)^{-1}$ and
${\gcd}_0(\pi\times\tau,s)^{-1}$ are relatively prime. However,
${\gcd}_0(\pi\times\tau,s)$ and
$\ell_{\tau}(s)\widetilde{P}(q^{-s},q^{s})$ may have common factors
and factors of $\ell_{\tau^*}(1-s)P(q^{-s},q^{s})$ may also appear
in the numerator. Canceling common factors and using \eqref{eq:back
to back quotients gamm aL and gcd} we see that there are
$Q_1,Q_2\in\C[q^{-s},q^s]$ satisfying
\begin{align*}
L(\pi\times\tau,s)=Q_1(q^{-s},q^s){\gcd}_0(\pi\times\tau,s)\ell_{\tau^*}(1-s)P(q^{-s},q^s)Q_2(q^{-s},q^s).
\end{align*}
Here $Q_1$ consists of the common factors of
${\gcd}_0(\pi\times\tau,s)^{-1}$ and
$(\ell_{\tau}(s)\widetilde{P})^{-1}$. The polynomial $Q_2$ divides
$(\ell_{\tau^*}(1-s)P)^{-1}$ (note that
$(\ell_{\tau}(s)\widetilde{P})^{-1},(\ell_{\tau^*}(1-s)P)^{-1}\in\C[q^{-s},q^s]$).
Hence
\begin{align*}
{\gcd}_0(\pi\times\tau,s)\in
Q_1(q^{-s},q^s)^{-1}L(\pi\times\tau,s)\C[q^{-s},q^s]\subset
\ell_{\tau}(s)L(\pi\times\tau,s)\C[q^{-s},q^s].
\end{align*}
Therefore
\begin{align*}
\gcd(\pi\times\tau,s)\in
L(\pi\times\tau,s)M_{\tau}(s)\C[q^{-s},q^s].
\end{align*}

Now assume that the intertwining operators
\begin{align*}
L(\tau,Sym^2,2s-1)^{-1}\intertwiningfull{\tau}{s},\qquad
L(\tau^*,Sym^2,1-2s)^{-1}\intertwiningfull{\tau^*}{1-s}
\end{align*}
are holomorphic. According to equality~\eqref{eq:gamma up to units}, there is some $e(q^{-s},q^{s})\in\C[q^{-s},q^s]^*$ such that
\begin{align*}
\nintertwiningfull{\tau^*}{1-s}&=\gamma(\tau^*,Sym^2,\psi,1-2s)\intertwiningfull{\tau^*}{1-s}\\&=e(q^{-s},q^{s})\frac{L(\tau,Sym^2,2s)}{L(\tau^*,Sym^2,1-2s)}\intertwiningfull{\tau^*}{1-s}.
\end{align*}
Since we assume that $L(\tau^*,Sym^2,1-2s)^{-1}\intertwiningfull{\tau^*}{1-s}$ is holomorphic,
the poles of $\nintertwiningfull{\tau^*}{1-s}$
are contained in the poles of $L(\tau,Sym^2,2s)$, which by
Theorem~\ref{theorem:tempered L
function holomorphic for half plane} lie in $\Re(s)\leq0$.
Thus by Propositions~\ref{proposition:integral for good section} and
\ref{proposition:convergence for tempered} the poles of
$\Psi(W,f_s,s)$ for $f_s\in\xi(\tau,good,s)$ are in $\Re(s)\leq0$
and the same is true for $\gcd(\pi\times\tau,s)$. Similarly we see
that the poles of $\gcd(\pi\times\tau^*,1-s)$ are in $\Re(s)\geq1$.
Thus the quotient $\gcd(\pi\times\tau^*,1-s)\gcd(\pi\times\tau,s)^{-1}$
is reduced and then \eqref{eq:back to back quotients gamm aL and
gcd} implies
\begin{align*}
\gcd(\pi\times\tau,s)=L(\pi\times\tau,s).
\end{align*}
This is an exact equality (i.e., not just up to invertible factors)
since the normalization of these factors is identical.

\begin{proof}[Proof of Proposition~\ref{proposition:convergence for tempered}] 
Consider first the case $\smallidx\leq\bigidx$ and split
$G_{\smallidx}$. By Proposition~\ref{proposition:iwasawa
decomposition of the integral} and Corollary~\ref{corollary:corro to
iwasawa decomposition of the integral}, as meromorphic continuations
$\Psi(W,f_s,s)=\sum_{i=1}^mI_s^{(i)}$, with $I_s^{(i)}$ of the form
\eqref{int:iwasawa decomposition of the integral l<=n split}. Hence
showing that the meromorphic continuation of any integral of the
form \eqref{int:iwasawa decomposition of the integral l<=n split} is
holomorphic for $\Re(s)>0$ implies this also for $\Psi(W,f_s,s)$.

The meromorphic continuation of integral~\eqref{int:iwasawa
decomposition of the integral l<=n split} is a function
$Q\in\C(q^{-s})$ such that for $\Re(s)>>0$, $Q(q^{-s})$ equals the
value of the integral at $s$. We will
prove that the integral (see \eqref{int:iwasawa
decomposition of the integral l<=n split})
\begin{align*}
\int_{A_{\smallidx-1}}\int_{G_1}ch_{\Lambda}(x)|W^{\diamond}|(ax)|W'|(diag(a,\lfloor
x\rfloor,I_{\bigidx-\smallidx})
)
(\absdet{a}[x]^{-1})^{\smallidx-\half\bigidx+\Re(s)-\half}\delta_{B_{G_{\smallidx}}}^{-1}(a)dxda
\end{align*}
is convergent for $\Re(s)>0$ and this convergence is uniform for
$\Re(s)$ in a compact set. This shows that integral~\eqref{int:iwasawa
decomposition of the integral l<=n split} is defined
for all $s$ with $\Re(s)>0$ and is a holomorphic function of $s$ in
this domain. Since it coincides with $Q(q^{-s})$ for $\Re(s)>>0$, it
follows that $Q(q^{-s})$ (as a function of $s$) is holomorphic for
$\Re(s)>0$.

We may already take $s\in\R$. Because on
$A_{\smallidx-1}$,
$\delta_{B_{G_{\smallidx}}}=\absdet{}^{2\smallidx-2\bigidx-1}\delta_{Q_{\bigidx}}\delta_{B_{\GL{\bigidx}}}$
and also
\begin{align*}
\delta_{B_{\GL{\bigidx}}}(diag(I_{\smallidx-1},\lfloor
x\rfloor,I_{\bigidx-\smallidx})
)=|\lfloor
x\rfloor|^{-2\smallidx+\bigidx+1}
=([x]^{-1})^{-2\smallidx+\bigidx+1},
\end{align*}
this integral is bounded from above by
\begin{align}\label{int:absolute value of iwasawa decompoed integral l<=n split}
\int_{A_{\smallidx-1}}\int_{G_1}ch_{\Lambda}(x)|W^{\diamond}|(ax)\delta_{B_{G_{\smallidx}}}^{-\half}(a)
(\delta_{B_{\GL{\bigidx}}}^{-\half}\cdot\absdet{}^s\cdot |W'|)(
diag(a,\lfloor x\rfloor,I_{\bigidx-\smallidx})
)dxda.
\end{align}
Since $|\lfloor x\rfloor|\leq1$ and $W'$ vanishes away from zero
(see Section~\ref{section:whittaker props}), the coordinates of $a$
are all bounded from above. Hence we may simply replace
$ch_{\Lambda}(x)$ in the last integral with a non-negative Schwartz
function $\Phi\in\mathcal{S}(F^{\smallidx})$, which is a function of
$a$ and $\lfloor x\rfloor$. By the Cauchy-Schwarz Inequality
integral~\eqref{int:absolute value of iwasawa decompoed integral
l<=n split} is bounded by the product of square roots of the
following two integrals:
\begin{align}\label{int:absolute value l<=n split SO part}
&\int_{A_{\smallidx-1}}\int_{G_1}|W^{\diamond}|^2(ax)\absdet{a}^s[x]^{-s}\delta_{B_{G_{\smallidx}}}^{-1}(a)dxda,\\
\label{int:absolute value l<=n split GL part with x floor}
&\int_{A_{\smallidx-1}}\int_{G_1}\Phi^2(a,\lfloor
x\rfloor)(\delta_{B_{\GL{\bigidx}}}^{-1}\cdot\absdet{}^s\cdot
|W'|^2)(diag(a,\lfloor x\rfloor,I_{\bigidx-\smallidx})
)dxda.
\end{align}
We can replace the $dx$-integration in \eqref{int:absolute value
l<=n split GL part with x floor} with an integration over
$a_{\smallidx}\in F^*$. Then \eqref{int:absolute value l<=n split GL
part with x floor} is a sum of two integrals - the first with
$|a_{\smallidx}|\leq1$, the second with $|a_{\smallidx}|>1$, both
bounded by an integral of the form
\begin{align}
\label{int:absolute value l<=n split GL part}
&\int_{A_{\smallidx}}|W'|^2(a)\Phi^2(a)\absdet{a}^s\delta_{B_{\GL{\bigidx}}}^{-1}(a)da.
\end{align}

Integrals~\eqref{int:absolute value l<=n split SO part} and
\eqref{int:absolute value l<=n split GL part} converge for $s>0$,
uniformly for $s$ in a compact set, since $\pi$ and $\tau$ are
tempered. This is proved using the asymptotic expansion of
Whittaker functions (\cite{CS2} Section~6, see also \cite{LM}) and
the fact that the exponents of a tempered representation are
non-negative (\cite{W} Proposition~III.2.2).

Note that the convergence of integral~\eqref{int:absolute value l<=n
split GL part} was proved by Jacquet, Piatetski-Shapiro and Shalika
\cite{JPSS} (Section~8) for a square-integrable $\tau$, this implies
the convergence in the tempered case, see Jacquet and Shalika
\cite{JS}. 

Below, we provide a proof for the convergence of \eqref{int:absolute
value l<=n split SO part}.
\begin{claim}\label{claim:convergence of absolute value l<=n split SO
part for s>0} Integral~\eqref{int:absolute value l<=n split SO part}
converges for $s>0$, uniformly for $s$ in a compact set.
\end{claim}


For $\smallidx\leq\bigidx$ and \quasisplit\ $G_{\smallidx}$ one just
repeats the above arguments, ignoring the integration over $G_1$.

Assume $\smallidx>\bigidx$. As above, it is enough to show that an integral $I_s^{(i)}$, now of the form
\eqref{int:iwasawa decomposition of the integral l>n}, is absolutely convergent for $s>0$ and this convergence is uniform for $s$ in a compact set (taking $s\in\R$). Thus we need to bound
\begin{align*}
\int_{A_{\bigidx}}|W^{\diamond}|(a)|W'|(a)\absdet{a}^{\frac32\bigidx-\smallidx+s+\half}\delta_{B_{H_{\bigidx}}}^{-1}(a)da.
\end{align*}
Since $W^{\diamond}$ vanishes away from zero and $\smallidx>\bigidx$, there is
a constant $c$ (depending on $W^{\diamond}$) such that $|a_i|<c$ for all $1\leq
i\leq\bigidx$. 
Therefore we can find a non-negative $\Phi\in
\mathcal{S}(F^{\bigidx})$ such that the last integral is bounded by
\begin{align*}
\int_{A_{\bigidx}}|W^{\diamond}|(a)|W'|(a)\Phi(a)\absdet{a}^{\frac32\bigidx-\smallidx+s+\half}\delta_{B_{H_{\bigidx}}}^{-1}(a)da.
\end{align*}
Using $\delta_{B_{H_{\bigidx}}}=\delta_{Q_{\bigidx}}\delta_{B_{\GL{\bigidx}}}$ this becomes
\begin{align*}
\int_{A_{\bigidx}}|W^{\diamond}|(a)\delta_{B_{H_{\bigidx}}}^{-\half}(a)|W'|(a)\Phi(a)\absdet{a}^{\bigidx-\smallidx+s+\half}\delta_{B_{\GL{\bigidx}}}^{-\half}(a)da.
\end{align*}
As above the Cauchy-Schwarz Inequality 
implies that we should bound
\begin{align*}
\int_{A_{\bigidx}}|W^{\diamond}|^2(a)\absdet{a}^s\delta_{B_{G_{\smallidx}}}^{-1}(a)da,
\end{align*}
which is bounded by \eqref{int:absolute value l<=n split SO part}
multiplied by a measure constant (the integrand of
\eqref{int:absolute value l<=n split SO part} is smooth), and an
integral of the form
\begin{align*}
\int_{A_{\bigidx}}|W'|^2(a)\Phi^2(a)\absdet{a}^s\delta_{B_{\GL{\bigidx}}}^{-1}(a)da,
\end{align*}
bounded in a similar manner as \eqref{int:absolute value l<=n split
GL part}.
\end{proof}

\begin{proof}[Proof of Claim~\ref{claim:convergence of absolute value l<=n split SO part for s>0}]
If $\smallidx=1$, we may assume $W^{\diamond}=\pi$, then the fact
that $\pi$ is tempered implies $|W^{\diamond}|=1$ and the assertion
of the claim clearly holds. Now assume $\smallidx>1$. We need some
notation. Recall that $\Delta_{G_{\smallidx}}$ denotes the set of
simple roots of $G_{\smallidx}$ 
and let $R^+$ be the set of positive roots. Any
$\Delta\subset\Delta_{G_{\smallidx}}$ determines a parabolic
subgroup $P_{\Delta}$, e.g. $P_{\emptyset}=B_{G_{\smallidx}}$, $P_{\Delta_{G_{\smallidx}}}=G_{\smallidx}$. For a
parabolic subgroup $P=P_{\Delta}$ let $R_P^+$ be the set of positive
roots which belong to the Levi part of $P$, for instance
$R^+=R_{G_{\smallidx}}^+$. Also put $T=T_{G_{\smallidx}}$. With this
notation $\delta_{P}(t)=\prod_{r\in R^+\setminus R_P^+}|r(t)|$ for
$t\in T$.

We formulate the results of Lapid and Mao \cite{LM} for
$W^{\diamond}\in\Whittaker{\pi}{\psi_{\gamma}^{-1}}$ (here $G_{\smallidx}$ is
split). 
For $t\in T$ define the valuation vector of $t$,
$H(t)\in\Integers^{|\Delta_{G_{\smallidx}}|}$, by
$H(t)_{\alpha}=\vartheta(\alpha(t))$ where
$\alpha\in\Delta_{G_{\smallidx}}$ (recall that
$|x|=q^{-\vartheta(x)}$ for $x\in F^*$). For $P=P_{\Delta}$ let
$M_P\subset F^{|\Delta_{G_{\smallidx}}|}$ be given by
\begin{align*}
M_P=\prod_{\alpha\in\Delta_{G_{\smallidx}}}\begin{cases}F^*&\alpha\in\Delta,\\F&\alpha\notin\Delta.\end{cases}
\end{align*}
Let $\theta$ belong to $\mathcal{S}(M_P)$ - the space of Schwartz
functions on $M_P$. Then $\theta$ defines a function on $T$ by
$\theta(t)=\theta(\eta(t))$, where $\eta(t)\in
(F^*)^{|\Delta_{G_{\smallidx}}|}$ is given by
$\eta(t)_{\alpha}=\alpha(t)$ for $\alpha\in\Delta_{G_{\smallidx}}$.
For $x\in F$, if there is some constant $c>0$ such that $|x|<c$,
write $x\prec0$. If moreover $c^{-1}<|x|<c$, write $x\prec1$. If
$t\in\support{\theta}$, for all $\alpha\in\Delta_{G_{\smallidx}}$ we
have $\alpha(t)\prec0$, and if $\alpha\in\Delta$, also
$\alpha(t)\prec1$.

Let $Z_M$ be the center of the Levi part $M$ of $P$. To a character
$\chi$ of $Z_M$ and an integer $m\geq1$, let $\mathcal{F}_{\chi,m}$
be the space of functions $t\mapsto\chi(t)Q(H(t))$ where $Q$ is a
polynomial in the coordinates of $H(t)$ of degree less than $m$.
Then define
\begin{align*}
\mathcal{F}_{P,\chi,m}=\Span{\C}\setof{\xi(t)\theta(t)}{\xi\in\mathcal{F}_{\chi,m},\theta\in\mathcal{S}(M_P)}.
\end{align*}
Also let $\mathcal{E}_{P}(\pi)$ denote the set of cuspidal exponents
of $\pi$ with respect to $P$. We are ready to state the asymptotic
expansion. 
According to \cite{LM} (Theorem~3.1) there exists
$m\geq1$ such that for any $W\in\Whittaker{\pi}{\psi_{\gamma}^{-1}}$
there are functions
$\Phi_{P_{\Delta},\chi}\in\mathcal{F}_{P_{\Delta},\chi,m}$
satisfying
\begin{align*}
W(t)=\sum_{\Delta\subset\Delta_{G_{\smallidx}}}\delta_{P_{\Delta}}^{\half}(t)\sum_{\chi\in\mathcal{E}_{P_{\Delta}}(\pi)}\Phi_{P_{\Delta},\chi}(t)\qquad(\forall
t\in T).
\end{align*}

Putting this formula into \eqref{int:absolute value l<=n split SO
part} we see that the integral equals
\begin{align*}
\int_{T}\sum_{\Delta_1,\Delta_2\subset\Delta_{G_{\smallidx}}}
\sum_{\chi_1\in\mathcal{E}_{P_1}(\pi)}
\sum_{\chi_2\in\mathcal{E}_{P_2}(\pi)}
(\delta_{B_{G_{\smallidx}}}^{-1}\delta_{P_1}^{\half}\delta_{P_2}^{\half}\Phi_{P_1,\chi_1}\overline{\Phi_{P_2,\chi_2}})(t)
(\prod_{1\leq i\leq\smallidx-1}|t_i|)^{s}[t_{\smallidx}]^{-s}dt.
\end{align*}
Here $P_i=P_{\Delta_i}$,
$[t_{\smallidx}]=\max(|t_{\smallidx}|,|t_{\smallidx}|^{-1})$ and
recall that $s$ is real. In order to bound this it is enough to show
that for any $\Delta_i$ and $\chi_i$, $i=1,2$,
\begin{align*}
\int_{T}
(\delta_{B_{G_{\smallidx}}}^{-1}\delta_{P_1}^{\half}\delta_{P_2}^{\half}|\Phi_{P_1,\chi_1}\overline{\Phi_{P_2,\chi_2}}|)(t)
(\prod_{1\leq
i\leq\smallidx-1}|t_i|)^{s}[t_{\smallidx}]^{-s}dt<\infty.
\end{align*}
We may assume $\Phi_{P_i,\chi_i}(t)=\chi_i(t)Q_i(H(t))\theta_i(t)$
with $\theta_i\in\mathcal{S}(M_{P_i})$.

Let $t\in T$. First we show that the modulus characters may be
ignored in the computation. Let $r\in R^+$ and write
$r=\sum_{\alpha\in\Delta_{G_{\smallidx}}}n_{\alpha}\alpha$ with
integer coefficients $n_{\alpha}\geq0$ (then
$r(t)=\prod_{\alpha\in\Delta_{G_{\smallidx}}}\alpha(t)^{n_{\alpha}}$).
If $\alpha'\in \Delta_1\cup\Delta_2$, without loss of generality
$\alpha'\in\Delta_1$. Since we may assume $t\in\support{\theta_1}$
(otherwise $\theta_1(t)=0$), $\alpha'(t)\prec1$. If
$\alpha'\notin\Delta_1\cup\Delta_2$ and $n_{\alpha'}\ne0$, we get
$r\notin R_{P_1}^+\cup R_{P_2}^+$ (because $R_{P_i}^+$ is spanned by
$\Delta_i$, see \cite{Spr} Section~8.4). Hence $r(t)$ appears in
both products $\delta_{P_1}^{\half}(t)$ and
$\delta_{P_2}^{\half}(t)$ and contributes $\alpha'(t)^{\half
n_{\alpha'}}$ to each, canceling the factor
$\alpha'(t)^{-n_{\alpha'}}$ appearing in
$\delta_{B_{G_{\smallidx}}}^{-1}(t)$ due to $r(t)$. This shows that
any factor $\alpha'(t)$ appearing in the product
$(\delta_{B_{G_{\smallidx}}}^{-1}\delta_{P_1}^{\half}\delta_{P_2}^{\half})(t)$
(in the expression of any $r\in R^+$) is bounded from above and
below. Thus we can ignore
$\delta_{B_{G_{\smallidx}}}^{-1}\delta_{P_1}^{\half}\delta_{P_2}^{\half}$
altogether (if the exponent of $\alpha'(t)$ in
$(\delta_{B_{G_{\smallidx}}}^{-1}\delta_{P_1}^{\half}\delta_{P_2}^{\half})(t)$
is positive, we use the upper bound, otherwise we use the lower
bound).

Second, assuming $t\in\support{\theta_1}$, we get $\alpha(t)\prec0$
for all $\alpha\in\Delta_{G_{\smallidx}}$. In particular
$t_{\smallidx-1}t_{\smallidx}^{-1}\prec0$ and
$t_{\smallidx-1}t_{\smallidx}\prec0$, hence $t_{\smallidx-1}\prec0$.
Then $t_{\smallidx-2}t_{\smallidx-1}^{-1}\prec0$ implies
$t_{\smallidx-2}\prec0$ and it follows that $t_i\prec0$ for all
$1\leq i\leq\smallidx-1$. Since the functions $Q_i$ are polynomials
in the valuations of the simple roots of $t$, the convergence of the
integral depends only on the factor
$|\chi_1|(t)|\chi_2|(t)(\prod_{1\leq
i\leq\smallidx-1}|t_i|)^{s}[t_{\smallidx}]^{-s}$. Because $\chi_1$
and $\chi_2$ are \nonnegative\ exponents ($\pi$ is tempered),
whenever $s>0$ the integral converges and it is clear that this
convergence is uniform, when $s$ belongs to a compact set.
\end{proof} 

\begin{proof}[Proof of Corollary~\ref{corollary:tempered theorem holds for essentially
tempered}] 
Let $\pi$ be tempered and let $\tau$ be essentially tempered. Take
$v\in\C$ such that $\tau_v=\absdet{}^v\tau$ is tempered. According
to the definitions,
\begin{align*}
\begin{array}{ll}
L(\pi\times\tau_v,s)=L(\pi\times\tau,s+v),&
\gcd(\pi\times\tau_v,s)=\gcd(\pi\times\tau,s+v),\\
\intertwiningfull{\tau_v}{s}=\intertwiningfull{\tau}{s+v}&
\intertwiningfull{(\tau_v)^*}{1-s}=\intertwiningfull{\tau^*}{1-(s+v)}
\end{array}
\end{align*}
and $M_{\tau_v}(s)=M_{\tau}(s+v)$. Applying Theorem~\ref{theorem:gcd
for tempered reps} to $\pi$ and $\tau_v$ yields
\begin{align*}
&L(\pi\times\tau,s+v)\in
\gcd(\pi\times\tau,s+v)\C[q^{-s},q^s],\\\notag
&\gcd(\pi\times\tau,s+v)\in
L(\pi\times\tau,s+v)M_{\tau}(s+v)\C[q^{-s},q^s].
\end{align*}
Also $L(\tau_v,Sym^2,2s-1)=L(\tau,Sym^2,2(s+v)-1)$, hence if
$L(\tau,Sym^2,2s-1)^{-1}\intertwiningfull{\tau}{s}$ is holomorphic,
so is
\begin{align*}
L(\tau_v,Sym^2,2s-1)^{-1}\intertwiningfull{\tau_v}{s}=L(\tau,Sym^2,2(s+v)-1)^{-1}\intertwiningfull{\tau}{s+v}.
\end{align*}
Thus if $L(\tau,Sym^2,2s-1)^{-1}\intertwiningfull{\tau}{s}$ and
$L(\tau^*,Sym^2,1-2s)^{-1}\intertwiningfull{\tau^*}{1-s}$ are
holomorphic,
\begin{align*}
\gcd(\pi\times\tau,s+v)=\gcd(\pi\times\tau_v,s)=L(\pi\times\tau_v,s)=L(\pi\times\tau,s+v).
\end{align*}
Changing $s\mapsto s-v$, Theorem~\ref{theorem:gcd for tempered reps}
also holds when $\tau$ is essentially tempered.
\end{proof} 

\section{Weak lower bound on the g.c.d.}\label{embedding poles in pi}
Let $\pi$ be parabolically induced from a representation $\sigma\otimes\pi'$ of $L_k$, where $\sigma$ is a representation of $\GL{k}$, and let $\tau$ be an irreducible representation of $\GL{\bigidx}$. We consider cases where the poles of the Rankin-Selberg $\GL{k}\times\GL{\bigidx}$ integrals for $\sigma\times\tau$ are contained in
$\gcd(\pi\times\tau,s)$. This provides a weak lower bound on
$\gcd(\pi\times\tau,s)$, which will be strengthened in the process
of proving Theorem~\ref{theorem:lower bound for gcd 2}. 
\begin{lemma}\label{lemma:sub multiplicativity embed some sections in first second var k<l or k=l>n}
Let
$\pi=\cinduced{\overline{P_k}}{G_{\smallidx}}{\sigma\otimes\pi'}$
with $0<k\leq\smallidx$, where $\sigma$ is a representation of
$\GL{k}$ and $\pi'$ is a representation of $G_{\smallidx-k}$. If
$k=\smallidx$, assume $\smallidx>\bigidx$ and that $\pi$ is induced
from either $\overline{P_{\smallidx}}$ or
$\rconj{\kappa}(\overline{P_{\smallidx}})$. Then
$L(\sigma\times\tau,s)\in\gcd(\pi\times\tau,s)\C[q^{-s},q^s]$.
\end{lemma}

\begin{proof}[Proof of Lemma~\ref{lemma:sub multiplicativity embed some sections in first second var k<l or k=l>n}] 
Let $W_{\sigma}\in\Whittaker{\sigma}{\psi^{-1}}$ and
$W_{\tau}\in\Whittaker{\tau}{\psi}$ be arbitrary. Assume that
$W_{\sigma}$ (resp. $W_{\tau}$) is right-invariant by
$\mathcal{N}_{\GL{k},k_0}$ (resp. $\mathcal{N}_{\GL{\bigidx},k_0}$).
Select $\varphi$ in the space of $\pi$ with support in
$\overline{P_k}\mathcal{N}_{G_{\smallidx},k_1}$, $k_1>>k_0$, which
is right-invariant by $\mathcal{N}_{G_{\smallidx},k_1}$ and such
that $\varphi(a,1,1)=\delta_{P_k}^{-\half}(a)W_{\sigma}(a)$ for
$a\in\GL{k}$ ($\delta_{P_k}(a)=\absdet{a}^{2\smallidx-k-1}$). Then
$\varphi$ defines a Whittaker function
$W_{\varphi}\in\Whittaker{\pi}{\psi_{\gamma}^{-1}}$ by virtue of
\eqref{eq:whittaker for pi, k<l}. Let $W_1$ be as in
Lemma~\ref{lemma:W with small support}, with $j$ to be specified
below, defined for $W_{\varphi}$, with $k_2>>k_1$ (i.e.,
$W_{\varphi}$ is $W_0$ of the lemma). The additional option in the
lemma for $W_1$ concerning the last row of $a\in\GL{j}$ is not used
yet. If $\bigidx<\smallidx$, set
$W=(w^{\smallidx,\bigidx})^{-1}\cdot W_1$ and otherwise $W=W_1$.
Finally let $f_s=ch_{\mathcal{N}_{H_{\bigidx},k_3},W_{\tau},s}\in\xi(\tau,std,s)$, 
$k_3>>k_2$. 
The proof varies according to the relative sizes
of $k$, $\smallidx$ and $\bigidx$.
\begin{enumerate}[leftmargin=*]
\item\label{case:embedding pi n<k} $\bigidx<k<\smallidx$. In the selection of $W$ above, Lemma~\ref{lemma:W with small
support} is applied with $j=\bigidx$. The starting point is
\begin{align*}
\Psi(W,f_s,s)=\int_{\lmodulo{U_{H_{\bigidx}}}{H_{\bigidx}}}
(\int_{R^{\smallidx,\bigidx}}W(rw^{\smallidx,\bigidx}h)dr)f_s(h,1)dh.
\end{align*}
Using the formula
\begin{align*}
\int_{\lmodulo{U_{H_{\bigidx}}}{H_{\bigidx}}}F(h)dh=
\int_{\lmodulo{Z_{\bigidx}}{\GL{\bigidx}}}\int_{\overline{U_{\bigidx}}}F(au)\absdet{a}^{-\bigidx}duda,
\end{align*}
the integral becomes
\begin{align*}
\int_{\lmodulo{Z_{\bigidx}}{\GL{\bigidx}}}\int_{\overline{U_{\bigidx}}}
(\int_{R^{\smallidx,\bigidx}}W(rw^{\smallidx,\bigidx}au)dr)f_s(u,a)\absdet{a}^{-\half\bigidx+s-\half}duda.
\end{align*}
By our choice of $f_s$, $f_s(u,a)=0$ unless $u\in\mathcal{N}_{H_{\bigidx},k_3}$ in which case $f_s(u,a)=f_s(1,a)$. Taking $k_3$ large enough, $W$ is right-invariant for such elements $u$. Therefore the $du$-integration can be ignored. Also
conjugating $R^{\smallidx,\bigidx}$ by $\rconj{(w^{\smallidx,\bigidx})^{-1}}a$ changes $dr$ by
$\absdet{a}^{\bigidx-\smallidx+1}$. Thus we get, up to some constant
depending on $k_3$,
\begin{align*}
\int_{\lmodulo{Z_{\bigidx}}{\GL{\bigidx}}}
(\int_{R^{\smallidx,\bigidx}}W(\left(\begin{array}{ccccc}a\\r&I_{\smallidx-\bigidx-1}\\&&I_2\\
&&&I_{\smallidx-\bigidx-1}\\&&&r'&a^*\end{array}\right)w^{\smallidx,\bigidx}
)dr)W_{\tau}(a)\absdet{a}^{\half\bigidx-\smallidx+s+\half}da.
\end{align*}
According to the selection of $W$ and $W_1$ the $dr$-integration is ignored and (up
to a constant depending on $k_2$) we have
\begin{align}\label{int:before shifting a across V_k}
\int_{\lmodulo{Z_{\bigidx}}{\GL{\bigidx}}}W_{\varphi}(diag(a,I_{2(\smallidx-\bigidx)},a^*))W_{\tau}(a)\absdet{a}^{\half\bigidx-\smallidx+s+\half}da.
\end{align}
Now substitute \eqref{eq:whittaker for pi, k<l} for $W_{\varphi}$. Integral~\eqref{int:before
shifting a across V_k} equals
\begin{align*}
\int_{\lmodulo{Z_{\bigidx}}{\GL{\bigidx}}}
(\int_{V_k}\varphi(v\cdot diag(a,I_{2(\smallidx-\bigidx)},a^*),I_k,I_{2(\smallidx-k)})\psi_{\gamma}(v)dv)
W_{\tau}(a)\absdet{a}^{\half\bigidx-\smallidx+s+\half}da.
\end{align*}
Since $\bigidx<k$, $diag(a,I_{2(\smallidx-\bigidx)},a^*)$ stabilizes $\psi_{\gamma}$ (and normalizes $V_k$) whence the last
integral equals
\begin{align*}
\int_{\lmodulo{Z_{\bigidx}}{\GL{\bigidx}}}
(\int_{V_k}\varphi(v,diag(a,I_{k-\bigidx}),1)\psi_{\gamma}(v)dv)W_{\tau}(a)\absdet{a}^{s-\half(k-\bigidx)}da.
\end{align*}
By our choice of $\varphi$ the $dv$-integration is discarded (up to a constant depending on $k_1$) whence
we have
\begin{align*}
\int_{\lmodulo{Z_{\bigidx}}{\GL{\bigidx}}}
W_{\sigma}(diag(a,I_{k-\bigidx}))W_{\tau}(a)\absdet{a}^{s-\half(k-\bigidx)}da.
\end{align*}
Our domain of absolute convergence of $\Psi(W,f_s,s)$, where this
computation is justified, is a right half-plane. Hence for
$\Re(s)>>0$, the last integral is equal to
$\Psi(W_{\tau},W_{\sigma},0,s)$ defined in
Section~\ref{subsection:func eq gln glm}. Because $W_{\sigma}$ and
$W_{\tau}$ are arbitrary and according to the definition of
$L(\sigma\times\tau,s)$ (see Section~\ref{subsection:func eq gln
glm}), this shows
\begin{align*}
L(\sigma\times\tau,s)\in\gcd(\pi\times\tau,s)\C[q^{-s},q^s].
\end{align*}
\item\label{case:embedding pi k<n<l} $k<\bigidx<\smallidx$. In choosing $W_1$ we take $j=k$. The starting form of the integral is
the same as in case~\eqref{case:embedding pi n<k} but now we utilize
the formula
\begin{align*}
\int_{\lmodulo{U_{H_{\bigidx}}}{H_{\bigidx}}}F(h)dh=
\int_{\lmodulo{Z_k}{\GL{k}}}\int_{\overline{Z_{k,\bigidx-k}}}\int_{\overline{B_{\GL{\bigidx-k}}}}\int_{\overline{U_\bigidx}}F(azbu)
\absdet{a}^{-2\bigidx+k}\delta(b)dudbdzda.
\end{align*}
Here $\GL{k}<M_k<H_{\bigidx}$,
\begin{align*}
\overline{Z_{k,\bigidx-k}}=\{\left(\begin{array}{ccccc}I_k\\z&I_{\bigidx-k}\\&&1\\&&&I_{\bigidx-k}\\&&&z'&I_k\\\end{array}\right)\},
\end{align*}
$\overline{B_{\GL{\bigidx-k}}}<M_{\bigidx-k}<H_{\bigidx-k}$ and
$\delta$ is a modulus character.
 Note that
$f_s(azbu,1)=\absdet{ab}^{\half\bigidx+s-\half}f_s(u,azb)$. As above
it follows that $u$ can be ignored. This leads to
\begin{align*}
&\int_{\lmodulo{Z_k}{\GL{k}}}\int_{\overline{Z_{k,\bigidx-k}}}\int_{\overline{B_{\GL{\bigidx-k}}}}
(\int_{R^{\smallidx,\bigidx}}W(\left(\begin{array}{cccc}a\\z&b\\r_1&r_2&I_{\smallidx-\bigidx-1}\\&&&1\\\end{array}\right)w^{\smallidx,\bigidx}
)d(r_1,r_2))\\\notag&W_{\tau}(azb)\delta(b)\absdet{b}^{\frac32\bigidx-\smallidx+s+\half}\absdet{a}^{-\half\bigidx+k-\smallidx+s+\half}dbdzda,
\end{align*}
where the matrix in $\GL{\smallidx}$ in the argument of $W$ is regarded as an element of $L_{\smallidx}$.
By our choice of $W$ and because $k_2>k_0$, the $d(r_1,r_2)dbdz$-integration reduces to a
constant and we have
\begin{align}\label{int:case embedding pi k<n<l before stabilizing explanation}
& \int_{\lmodulo{Z_k}{\GL{k}}}
(\int_{V_k}\varphi(v\cdot diag(a,I_{2(\smallidx-k)},a^*)
,1,1)\psi_{\gamma}(v)dv)\\\notag&W_{\tau}(diag(a,I_{\bigidx-k}))\absdet{a}^{-\half\bigidx+k-\smallidx+s+\half}da.
\end{align}
Note that $a\in\GL{k}$ does not stabilize $\psi_{\gamma}$. 
Write $a=x_ay_ab_a$ with $x_a\in F^*$ in the center of $\GL{k}$,
$y_a\in Y_k$ ($Y_k$ - the mirabolic subgroup) and $b_a\in
K_{\GL{k}}(=\GLF{k}{\mathcal{O}})$. Since $W_{\tau}$ vanishes away from zero and $k<\bigidx$, $W_{\tau}(diag(a,I_{\bigidx-k}))$ vanishes unless $|x_a|<c$ where
$c$ is some positive constant depending on $W_{\tau}$. If
$v^{(k+1)}\in F^k$ is the column containing the first $k$ rows of
the $(k+1)$-th column of $v$, $\psi_{\gamma}(v)=\psi((v^{(k+1)})_k)$
($k<\smallidx-1$) and
\begin{align*}
&\int_{V_k}\varphi(v\cdot diag(a,I_{2(\smallidx-k)},a^*)
,1,1)\psi_{\gamma}(v)dv\\\notag&=
\delta_{P_k}^{\half}(a)
\int_{V_k}\varphi(v,a,1)\psi(((x_ay_ab_a)v^{(k+1)})_k)dv.
\end{align*}
According to our choice of $\varphi$, the integrand vanishes unless
$v\in\mathcal{N}_{G_{\smallidx},k_1}$ and if we choose $k_1$ large
enough, depending on the constant $c$,
$\psi(((x_ay_ab_a)v^{(k+1)})_k)$ will be identically $1$ for all
$x_a$, $y_a$ and $b_a$ with $|x_a|<c$. Hence as above we get
\begin{align*}
\int_{V_k}\varphi(v,a,1)\psi(((x_ay_ab_a)v^{(k+1)})_k)dv=W_{\sigma}(a).
\end{align*}
Therefore the integral equals
\begin{align*}
\int_{\lmodulo{Z_k}{\GL{k}}}W_{\sigma}(a)W_{\tau}(diag(a,I_{\bigidx-k}))\absdet{a}^{s-\half(\bigidx-k)}da
=\Psi(W_{\sigma},W_{\tau},0,s).
\end{align*}

\item\label{case:embedding pi l>n} $\bigidx<k=\smallidx$. Assume that $\pi$ is induced from
$\overline{P_{\smallidx}}$. The same substitutions and arguments of
case~\eqref{case:embedding pi n<k} apply here as well, except that
we plug \eqref{eq:whittaker for pi, k=l} instead of
\eqref{eq:whittaker for pi, k<l} into \eqref{int:before shifting a
across V_k}. Now if $\bigidx<\smallidx-1$,
$diag(a,I_{2(\smallidx-\bigidx)},a^*)$ with $a\in\GL{\bigidx}$
stabilizes $\psi_{\gamma}(v)=\psi(-\gamma
v_{\smallidx-1,\smallidx+1})$ and the result follows. Otherwise
$\bigidx=\smallidx-1$, continue as in case~\eqref{case:embedding pi
k<n<l} using the fact that $W_{\sigma}$ vanishes away from zero
(instead of $W_{\tau}$). We obtain
\begin{align*}
\int_{V_{\smallidx}}\varphi(v\cdot diag(a,I_{2(\smallidx-\bigidx)},a^*),d_{\gamma})\psi_{\gamma}(v)dv=
\delta_{P_{\smallidx}}^{\half}(a)W_{\sigma}(diag(a,I_{k-\bigidx})).
\end{align*}
The result follows when this is put into \eqref{int:before shifting
a across V_k}. When $\pi$ is induced from
$\rconj{\kappa}(\overline{P_{\smallidx}})$ we utilize
\eqref{eq:whittaker for pi, k=l 2}.

\item\label{case:embedding pi n=k} $\bigidx=k<\smallidx$. Contrary to
case~\eqref{case:embedding pi n<k}, $a$ does not stabilize
$\psi_{\gamma}$, while the method of case~\eqref{case:embedding pi
k<n<l} falls short because the properties of $W_{\tau}$ do not allow
to bound $x_a$. Following Cogdell and Piatetski-Shapiro \cite{CPS2}
(Section~3), it is enough to obtain two types of integrals:
$\Psi(W_{\sigma},W_{\tau},\Phi,s)$ where $\Phi\in\mathcal{S}(F^{k})$
satisfies $\Phi(0)\ne0$ (any such $\Phi$ is sufficient) and
\begin{align}
\label{int:glm gln non exceptional poles integrals}
&\int_{\lmodulo{Z_k}{Y_k}}W_{\sigma}(y)W_{\tau}(y)\absdet{y}^{s-1}d_ry.
\end{align}
In fact, assume that $L(\sigma\times\tau,s)$ has a pole at $s_0$. Then the
residue of $\Psi(W_{\sigma},W_{\tau},\Phi,s_0)$ defines a
\nontrivial\ trilinear form on
$\Whittaker{\sigma}{\psi^{-1}}\times\Whittaker{\tau}{\psi}\times\mathcal{S}(F^{k})$.
If this form vanishes identically on the subspace
$\mathcal{S}_0=\setof{\Phi\in \mathcal{S}(F^{k})}{\Phi(0)=0}$, it
represents a \nontrivial\ trilinear form on
$\Whittaker{\sigma}{\psi^{-1}}\times\Whittaker{\tau}{\psi}\times\lmodulo{\mathcal{S}_0}{\mathcal{S}(F^{k})}$.
Then the pole is obtained by $\Psi(W_{\sigma},W_{\tau},\Phi,s)$ with any $\Phi$ such that $\Phi(0)\ne0$. Otherwise, $s_0$
is a pole of some $\Psi(W_{\sigma},W_{\tau},\Phi,s)$ with
$\Phi\in\mathcal{S}_0$, in which case using the Iwasawa
decomposition the integral is seen to be equal to a sum of
integrals, each of the form \eqref{int:glm gln non exceptional poles
integrals}.

We show that both types of integrals can be obtained.
As in case~\eqref{case:embedding pi n<k} starting from
$\Psi(W,f_s,s)$, where $W_1$ is selected with $j=k$ $(=\bigidx)$ we
reach \eqref{int:before shifting a across V_k}. In general our
selection of $\varphi$ implies
\begin{align*}
\int_{V_k}\varphi(va,1,1)\psi_{\gamma}(v)dv=
\delta_{P_k}^{\half}(a)W_{\sigma}(a)\int_{V_k\cap\mathcal{N}_{G_{\smallidx},k_1}}\psi_{\gamma}(\rconj{a^{-1}}v)dv.
\end{align*}
Writing $a=x_ay_ab_a$ with the above notation we see that there is a
constant $c>0$ depending on $\psi_{\gamma}$ and on $k_1$, such that
for $|x_a|\geq c$ the integral on the \rhs\ vanishes and for
$|x_a|<c$, we get $\delta_{P_k}^{\half}(a)W_{\sigma}(a)$ multiplied by a
measure constant depending on $k_1$.
Note that here (as opposed to case \eqref{case:embedding pi k<n<l})
it is possible that $k=\smallidx-1$, then
$\psi_{\gamma}(v)=\psi(\alpha_1 v_{\smallidx-1,\smallidx}+\alpha_2
v_{\smallidx-1,\smallidx+1})$ where $\alpha_1,\alpha_2\in F$ depend
on whether $G_{\smallidx}$ is split or not. 

Let $\Lambda_c=\setof{a\in\GL{k}}{|x_a|<c}$. The integral becomes
\begin{align*}
\int_{\lmodulo{Z_k}{\GL{k}}}W_{\sigma}(a)W_{\tau}(a)ch_{\Lambda_c}(a)\absdet{a}^sda.
\end{align*}
We can replace $ch_{\Lambda_c}(a)$ with $\Phi(\eta_k a)$,
$\eta_k=(0,\ldots,0,1)$, for some $\Phi\in\mathcal{S}(F^k)$ such
that $\Phi(0)\ne0$. Hence we have $\Psi(W_{\sigma},W_{\tau},\Phi,s)$
for arbitrary $W_{\sigma}$ and $W_{\tau}$. 

Now we turn to integrals of the form \eqref{int:glm gln non
exceptional poles integrals}, where we can assume $k>1$. Here $W_1$
is chosen using Lemma~\ref{lemma:W with small support} with $j=k$,
defined for $W_{\varphi}$ with $k_2>>k_1$, and so that $W_1$
vanishes on $a\in\GL{k}$ unless $\eta_ka$ lies in
$\eta_k+\MatF{1\times k}{\mathcal{P}^{k_2}}$. The constant $k_2$ is
large enough so that $W_{\varphi}$ is right-invariant on
$\mathcal{N}_{G_{\smallidx},k_2}$. As in case~\eqref{case:embedding
pi n<k}, $\Psi(W,f_s,s)$ becomes
\begin{align*}
\int_{\lmodulo{Z_k}{\GL{k}}}
(\int_{R^{\smallidx,k}}W(\left(\begin{array}{ccc}a\\r&I_{\smallidx-k-1}\\&&1\\\end{array}\right)w^{\smallidx,k}
)dr)W_{\tau}(a)\absdet{a}^{\half k-\smallidx+s+\half}da.
\end{align*}
As above the $dr$-integration may be ignored. We use the following
formula, derived using $\GL{k}=P_{k-1,1}K_{\GL{k}}$:
\begin{align*}
\int_{\lmodulo{Z_k}{\GL{k}}}F(g)dg=\int_{K_{\GL{k}}}\int_{\lmodulo{Z_k}{Y_k}}\int_{F^*}F(yxb)\absdet{y}^{-1}|x|^{k-1}d^*xd_rydb.
\end{align*}
Here $x=\left(\begin{array}{cc}I_{k-1}\\&x\end{array}\right)$. The
integral is now
\begin{align*}
&\int_{K_{\GL{k}}}\int_{\lmodulo{Z_k}{Y_k}}\int_{F^*} W_1(
diag(yxb,I_{2(\smallidx-k)},(yxb)^*))W_{\tau}(yxb)\\\notag&\absdet{y}^{\half
k-\smallidx+s-\half}|x|^{\frac32k-\smallidx+s-\half}d^*xd_rydb.
\end{align*}
Denote
\begin{align*}
b=\left(\begin{array}{cc}b_1&b_2\\b_3&b_4\end{array}\right)\in
K_{\GL{k}},\qquad b_1\in\Mat{k-1\times k-1}.
\end{align*}
The condition imposed by $W_1$ on $\eta_kyxb=\eta_kxb$ is that
$xb_3\in\MatF{1\times k-1}{\mathcal{P}^{k_2}}$ and
$xb_4\in1+\mathcal{P}^{k_2}$. When this is fulfilled, $|x|=1$,
$W_1(yxb)=W_{\varphi}(yxb)$ and we have a decomposition of $xb$ as
\begin{align*}
&\left(\begin{array}{cc}b_1&b_2\\xb_3&xb_4\end{array}\right) \\\notag
&=
\left(\begin{array}{cc}b_1-b_2b_4^{-1}b_3&b_2(xb_4)^{-1}\\0&1\end{array}\right)
\left(\begin{array}{cc}I_{k-1}&0\\xb_3&xb_4\end{array}\right)\in
(Y_k\cap K_{\GL{k}})\mathcal{N}_{\GL{k},k_2}.
\end{align*}
Since $W_{\varphi}$ (as well as $W_{\tau}$) is right-invariant on
$\mathcal{N}_{\GL{k},k_2}$ and the measure $d_ry$ is invariant for
translations on the right, we get that up to a constant the integral
equals
\begin{align*}
&\int_{\lmodulo{Z_k}{Y_k}}W_{\varphi}(diag(y,I_{2(\smallidx-k)},y^*))
W_{\tau}(y)\absdet{y}^{\half k-\smallidx+s-\half}d_ry.
\end{align*}
Here the hidden constant equals the volume of the subsets of
$K_{\GL{k}}$ and $F^*$ for which $\eta_kxb$ lies in $\eta_k+\MatF{1\times
k}{\mathcal{P}^{k_2}}$. We plug in \eqref{eq:whittaker for pi, k<l}
and get
\begin{align*}
&\int_{\lmodulo{Z_k}{Y_k}}
(\int_{V_k}\varphi(v\cdot diag(y,I_{2(\smallidx-k)},y^*)
,1,1)\psi_{\gamma}(v)dv)
W_{\tau}(y)\absdet{y}^{\half k-\smallidx+s-\half}d_ry.
\end{align*}
Since $Y_k$ stabilizes $\psi_{\gamma}(v)$ we can shift $y$ to the
second argument of $\varphi$ and obtain \eqref{int:glm gln non
exceptional poles integrals}.

\item\label{case:embedding pi l<=n} $\bigidx\geq\smallidx>k$. 
Again $W_1$ is chosen with $j=k$.
Here, as opposed to the previous cases, the section $f_s$ is
selected according to Lemma~\ref{lemma:f with small support} for
$t_{\gamma}^{-1}\cdot W_{\tau}$, with $k_3>>k_2$ large so that $W$
is right-invariant by $\mathcal{N}_{G_{\smallidx},k_3-c_0}$
($t_{\gamma}$ - given in Lemma~\ref{lemma:f with small support},
$c_0$ is the constant $k_0$ of Lemma~\ref{lemma:f with small
support}). We use the formula
\begin{align*}
&\int_{\lmodulo{U_{G_{\smallidx}}}{G_{\smallidx}}}F(g)dg\\\notag&=
\int_{\overline{B_{\GL{\smallidx-1-k}}}}\int_{G_1}\int_{\overline{V_{\smallidx-1}}}\int_{\lmodulo{Z_k}{\GL{k}}}
\int_{\overline{Z_{k,\smallidx-1-k}}}F(ambvx)\delta_{P_k}^{-1}(a)\delta(b)dmdadvdxdb,
\end{align*}
where $\overline{B_{\GL{\smallidx-1-k}}}$ is a subgroup of $G_{\smallidx-k}$, $\GL{k}<L_k$
and
\begin{align*}
\overline{Z_{k,\smallidx-1-k}}=\{\left(\begin{array}{ccccc}I_k\\m&I_{\smallidx-1-k}\\&&I_2\\
&&&I_{\smallidx-1-k}\\&&&m'&I_k\end{array}\right)\}.
\end{align*}
Note that $a$, $m$ and $b$ belong to
$\GL{\smallidx-1}<L_{\smallidx-1}$ and using \eqref{eq:omega right
translation by GL l-1} we obtain
\begin{align*}
\Psi(W,f_s,s)=&
\int_{\overline{B_{\GL{\smallidx-1-k}}}}\int_{G_1}\int_{\overline{V_{\smallidx-1}}}\int_{\lmodulo{Z_k}{\GL{k}}}
\int_{\overline{Z_{k,\smallidx-1-k}}}W(ambvx)\int_{R_{\smallidx,\bigidx}}
\\\notag&f_s(w_{\smallidx,\bigidx}rvx,diag(amb,I_{\bigidx-\smallidx+1}))\psi_{\gamma}(r)\absdet{a}^
{-\half\bigidx+k-\smallidx+s+\half}\\\notag&\absdet{b}^{s+M}drdmdadvdxdb.
\end{align*}
Here $M$ is some constant. By our choice of $f_s$ this equals a
constant in $\C[q^{-s},q^s]^*$ times
\begin{align*}
&\int_{\overline{B_{\GL{\smallidx-1-k}}}}\int_{\lmodulo{Z_k}{\GL{k}}}\int_{\overline{Z_{k,\smallidx-1-k}}}W(amb)
W_{\tau}(diag(amb,I_{\bigidx-\smallidx+1}))\\\notag&\absdet{a}^{-\half\bigidx+k-\smallidx+s+\half}\absdet{b}^{s+M}dmdadb.
\end{align*}
Using the properties of $W$, the $dmdb$-integration may be ignored
and after plugging in \eqref{eq:whittaker for pi, k<l} this equals
\begin{align*}
\int_{\lmodulo{Z_k}{\GL{k}}}(\int_{V_k}\varphi(va,1,1)\psi_{\gamma}(v)dv)W_{\tau}(diag(a,I_{\bigidx-k}))\absdet{a}^{-\half\bigidx+k-\smallidx+s+\half}da.
\end{align*}
Actually this is \eqref{int:case embedding pi k<n<l before
stabilizing explanation}. Because $k<\bigidx$, proceeding as in case~\eqref{case:embedding pi k<n<l} the
fact that $W_{\tau}$ vanishes away from zero enables us to obtain
$\Psi(W_{\sigma},W_{\tau},0,s)$.\qedhere
\end{enumerate}
\end{proof} 

\section{Lower bound for an irreducible $\pi$ and a tempered $\tau$}\label{subsection:Lower bound for pi induced tau tempered}
Let $\pi$ be an irreducible (generic) representation of
$G_{\smallidx}$. According to the standard module conjecture of
Casselman and Shahidi \cite{CSh} proved by Mui\'{c} \cite{Mu3} in
the context of classical groups, we may assume that $\pi$ is a
standard module. Namely,
$\pi=\cinduced{P_k}{G_{\smallidx}}{\sigma\otimes\pi'}$ where
$\sigma=\cinduced{P_{k_1,\ldots,k_m}}{\GL{k}}{\sigma_1\otimes\ldots\otimes\sigma_m}$
(if $m=1$, $\sigma=\sigma_1$), $\sigma_i=\absdet{}^{e_i}\sigma'_i$,
$\sigma'_i$ is a square-integrable representation of $\GL{k_i}$,
$e_i\in\R$, $0>e_1\geq\ldots\geq e_m$ and $\pi'$ is a tempered
representation of $G_{\smallidx-k}$ (see also \cite{Mu2}). Note that
$0\leq k\leq\smallidx$ and if $k=\smallidx$, it is also possible
that
$\pi=\cinduced{\rconj{\kappa}P_{\smallidx}}{G_{\smallidx}}{\sigma}$.

Let $\tau$ be tempered such that the intertwining operators
\begin{align*}
L(\tau,Sym^2,2s-1)^{-1}\intertwiningfull{\tau}{s},\qquad
L(\tau^*,Sym^2,1-2s)^{-1}\intertwiningfull{\tau^*}{1-s}
\end{align*}
are holomorphic. 
We prove Theorem~\ref{theorem:lower bound for gcd 2}, namely if
$k<\smallidx$ or $k=\smallidx>\bigidx$,
\begin{align*}
L(\sigma\times\tau,s)\gcd(\pi'\times\tau,s)L(\sigma^*\times\tau,s)\in\gcd(\pi\times\tau,s)\C[q^{-s},q^s].
\end{align*}
Recall that if $k=\smallidx$, by definition
$\gcd(\pi'\times\tau,s)=1$. Since the theorem is trivial when $k=0$,
assume $k>0$.

Applying Lemma~\ref{lemma:sub multiplicativity embed some sections
in first second var k<l or k=l>n} (valid for our range of
$k,\smallidx$ and $\bigidx$), for some $P_1,P_2\in\C[q^{-s},q^s]$,
\begin{align}\label{eq:general pi tempered tau substitution}
L(\sigma\times\tau,s)=\gcd(\pi\times\tau,s)P_1,\quad
L(\sigma\times\tau^*,1-s)=\gcd(\pi\times\tau^*,1-s)P_2.
\end{align}
According to Theorem~\ref{theorem:multiplicity first var},
\begin{align*}
\gamma(\pi\times\tau,\psi,s)\equalun\gamma(\sigma\times\tau,\psi,s)\gamma(\pi'\times\tau,\psi,s)\gamma(\sigma^*\times\tau,\psi,s).
\end{align*}
Then by \eqref{eq:gamma and epsilon} and \eqref{eq:JPSS relation
gamma and friends},
\begin{align*}
\frac{\gcd(\pi\times\tau^*,1-s)}{\gcd(\pi\times\tau,s)}\equalun
\frac{L(\sigma^*\times\tau^*,1-s)\gcd(\pi'\times\tau^*,1-s)L(\sigma\times\tau^*,1-s)}
{L(\sigma\times\tau,s)\gcd(\pi'\times\tau,s)L(\sigma^*\times\tau,s)}.
\end{align*}
Now using \eqref{eq:general pi tempered tau substitution} we obtain
\begin{align}\label{eq:quotient}
\frac{P_2L(\sigma^*\times\tau^*,1-s)\gcd(\pi'\times\tau^*,1-s)}{P_1L(\sigma^*\times\tau,s)\gcd(\pi'\times\tau,s)}\equalun1.
\end{align}

As proved in Section~\ref{subsection:tempered} the poles of
$\gcd(\pi'\times\tau,s)$ lie in $\Re(s)\leq0$, because $\pi'$ and
$\tau$ are tempered. 
According to \cite{JPSS} (Theorem~3.1),
\begin{align*}
L(\sigma^*\times\tau,s)\in \prod_{i=1}^m
L(\sigma_i^*\times\tau,s)\C[q^{-s},q^s].
\end{align*}
Hence any pole of $L(\sigma^*\times\tau,s)$ is contained in some
$L(\sigma_i^*\times\tau,s)=L((\sigma'_i)^*\times\tau,s-e_i)$, and
since $\sigma'_i$ is square-integrable we get from 
\cite{JPSS} (Section~8) that the poles of
$L((\sigma'_i)^*\times\tau,s-e_i)$ lie in $\Re(s)\leq e_i<0$. We
conclude that the poles of the denominator of \eqref{eq:quotient}
are in $\Re(s)\leq0$.

In a similar manner we prove that the poles of the numerator of
\eqref{eq:quotient} are in $\Re(s)\geq1$. Specifically, the poles of
$\gcd(\pi'\times\tau^*,1-s)$ are in $\Re(1-s)\leq0$ because $\pi'$
and $\tau^*$ are tempered, and any pole of
$L(\sigma^*\times\tau^*,1-s)$ appears in some
$L((\sigma'_i)^*\times\tau^*,1-s-e_i)$.

Thus any pole appearing in the denominator must be canceled by
$P_1$, i.e.,
\begin{align*}
P_1L(\sigma^*\times\tau,s)\gcd(\pi'\times\tau,s)\in\C[q^{-s},q^s].
\end{align*}
Together with \eqref{eq:general pi tempered tau substitution} this gives 
the result.

\section{The g.c.d. of ramified twists}\label{subsection:gcd for ramified twist of tau}
Let $\tau$ be an irreducible representation of $\GL{\bigidx}$ and
let $\mu$ be a character of $F^*$ extended to $\GL{\bigidx}$ by $b\mapsto\mu(\det(b))$. 

Recall that a character of $F^*$ is ramified if it is \nontrivial\
on $\mathcal{O}^*$ and any character of $F^*$ is trivial on
$1+\mathcal{P}^{k}$ for some $k>0$. We call a character highly
ramified, if it is \nontrivial\ on $1+\mathcal{P}^{k}$ for some
large $k$. We say that a certain property is satisfied for all
sufficiently highly ramified characters, if there is some $k>0$ such
that for any character which is \nontrivial\ on $1+\mathcal{P}^{k}$,
this property holds. For a character $\mu$ and $a\in\Integers$, the character $\mu^a$ is
defined by $\mu^a(x)=\mu(x)^a$. Clearly if $\mu^a$ is highly
ramified, so is $\mu$. The next claim establishes a few facts about
ramified characters that we will need.
\begin{claim}\label{claim:properties of sufficiently ramified characters}
The following statements are valid.
\begin{enumerate}
\item For all sufficiently highly ramified $\mu^{2\bigidx}$, no unramified twist of $\tau\mu$ is
self-dual.
\item
Assume that no unramified twist of $\tau$ is self-dual. Then also no
unramified twist of $\tau^*$ is self-dual.
\item For any $0\ne a\in\Integers$ there is $M\geq 0$ such that for each
$k>M$, there is a unitary character $\mu$ of $F^*$ such that $\mu^a$
is \nontrivial\ on $1+\mathcal{P}^k$. In particular for any $k>>0$
there is $\mu$ such that $\mu^{2\bigidx}$ is \nontrivial\ on
$1+\mathcal{P}^k$.
\end{enumerate}
\end{claim}
\begin{proof}[Proof of Claim~\ref{claim:properties of sufficiently ramified
characters}]
\begin{enumerate}[leftmargin=*]
\item Suppose $\tau\mu\absdet{}^u\isomorphic\tau^*\mu^{-1}\absdet{}^{-u}$ for
some $u\in\C$ ($\tau$ is irreducible whence
$\dualrep{\tau}\isomorphic\tau^*$). In particular the central
characters are equal, which is impossible if $k$ is large so that
$\frestrict{\omega_{\tau}}{1+\mathcal{P}^k}\equiv1$ (regarding $F^*$
as the center of $\GL{\bigidx}$),
$\frestrict{\omega_{\tau^*}}{1+\mathcal{P}^k}\equiv1$ and
$\frestrict{\mu^{2\bigidx}}{1+\mathcal{P}^k}\nequiv1$.
\item Otherwise for some $u\in\C$,
$\tau^*\absdet{}^u\isomorphic(\tau^*\absdet{}^u)^*$ whence
$\tau\absdet{}^{-u}\isomorphic(\tau\absdet{}^{-u})^*$,
contradiction.
\item It is enough to construct a character of $\mathcal{O}^*$ with
the required properties. Furthermore, it is enough to construct a
character $\mu$ of $\rmodulo{1+\mathcal{P}^k}{1+\mathcal{P}^m}$, for
some $m>k$, so that $\mu^a$ is \nontrivial, since any such character
lifts to a character of $\rmodulo{\mathcal{O}^*}{1+\mathcal{P}^m}$
(because $\rmodulo{\mathcal{O}^*}{1+\mathcal{P}^m}$ is a finite
abelian group), which lifts to a character of $\mathcal{O}^*$.

Assume $|a|=q^{-v}$ where $v\in\Integers$. Let $k>0$ be given and
let $m$ satisfy $k<m\leq 2k$ and $k+v<m$. Since $m\leq 2k$,
$\rmodulo{1+\mathcal{P}^k}{1+\mathcal{P}^m}\isomorphic
\rmodulo{\mathcal{P}^k}{\mathcal{P}^m}$ as groups. Let $\psi'$ be a
unitary additive character of $F$ such that
$\frestrict{\psi'}{\mathcal{P}^{m-1}}\nequiv1$ and
$\frestrict{\psi'}{\mathcal{P}^m}\equiv1$. Then $\mu(1+b)=\psi'(b)$
($b\in\mathcal{P}^k$) is a character of $1+\mathcal{P}^k$, trivial
on $1+\mathcal{P}^m$. Also $\mu^a(1+b)=\psi'(ab)$ and because
$k+v<m$, $a\mathcal{P}^k$ contains $\mathcal{P}^{m-1}$ whence
$\mu^a$ is \nontrivial. Put $m=2k$, $M=\max(v,0)$. Then for any
$k>M$, $k+v<m$.\qedhere
\end{enumerate}
\end{proof} 
We show that for all sufficiently highly ramified characters
$\mu^2$, a twist of $\tau$ by $\mu$ removes all poles except perhaps
those of the intertwining operator. This scenario is considered,
first because it resembles a result proved by Jacquet,
Piatetski-Shapiro and Shalika \cite{JPSS} (Proposition~2.13) and
this property is expected from the g.c.d. Second, this result will
be useful in the course of proving Theorem~\ref{theorem:gcd sub
multiplicity first var} (similarly to its use in \cite{JPSS}
Section~7).
\begin{proposition}\label{proposition:gcd for ramified twist of tau}
For any character $\mu$ such that $\mu^2$ is sufficiently highly
ramified,
\begin{align*}
\gcd(\pi\times(\tau\mu),s)\in \ell_{(\tau\mu)^*}(1-s)\C[q^{-s},q^s].
\end{align*}
\end{proposition}
\begin{proof}[Proof of Proposition~\ref{proposition:gcd for ramified twist of tau}] 
We prove the proposition for the case $1<\smallidx\leq\bigidx$ and
split $G_{\smallidx}$. The other cases are similar (and simpler). We
follow the method of \textit{loc. cit.} According to
Proposition~\ref{proposition:integral for good section} it is enough
to show that $\Psi(W,f_s,s)\in\C[q^{-s},q^s]$ for
$f_s\in\xi(\tau\mu,std,s)$. By virtue of
Proposition~\ref{proposition:iwasawa decomposition of the integral}
(and using the same notation) we should show this for integrals of
the form \eqref{int:iwasawa decomposition of the integral l<=n
split},
\begin{align*}
&\int_{A_{\smallidx-1}}\int_{G_1}ch_{\Lambda}(x)W^{\diamond}(ax)(\mu\cdot
W')(diag(a,\lfloor x\rfloor,I_{\bigidx-\smallidx})
)\\\notag
&(\absdet{a}\cdot[x]^{-1})^{\smallidx-\half\bigidx+s-\half}\delta_{B_{G_{\smallidx}}}^{-1}(a)dxda.
\end{align*}

According to the asymptotic expansion described in
Section~\ref{section:whittaker props},
\begin{align*}
&W^{\diamond}(t)=\sum_{\eta_1,\ldots,\eta_{\smallidx}\in\mathcal{A}_{\pi}}\phi_{\eta_1,\ldots,\eta_{\smallidx}}(\alpha_1(t),\ldots,\alpha_{\smallidx}(t))\prod_{i=1}^{\smallidx}\eta_i(\alpha_i(t)),\\\notag
&W'(t')=\sum_{\eta_1',\ldots,\eta_{\bigidx-1}'\in
\mathcal{A}_{\tau}}\phi'_{\eta_1',\ldots,\eta_{\bigidx-1}'}(\alpha_1'(t'),\ldots,\alpha_{\bigidx-1}'(t'))\prod_{i=1}^{\bigidx-1}\eta_i'(\alpha_i'(t')).
\end{align*}
Here $t\in T_{G_{\smallidx}}$,
$\phi_{(\ldots)}\in\mathcal{S}(F^{\smallidx})$,
$\alpha_i\in\Delta_{G_{\smallidx}}$, $t'\in A_{\bigidx}$,
$\phi'_{(\ldots)}\in\mathcal{S}(F^{\bigidx-1})$ and
$\alpha_i'\in\Delta_{\GL{\bigidx}}$. We plug these formulas into the
integral. Let $k>0$ be such that for all $\eta\in\mathcal{A}_{\pi}$,
$\eta'\in\mathcal{A}_{\tau}$, $y\in F^*$ and $u\in 1+\mathcal{P}^k$,
$\eta(yu)=\eta(y)$, $\eta'(yu)=\eta'(y)$. Assume that $\mu^2$ is a
\nontrivial\ character of $1+\mathcal{P}^k$ (hence so is $\mu$).
Set $a=diag(a_1,\ldots,a_{\smallidx-1})\in A_{\smallidx-1}$ and
$x=diag(b,b^{-1})\in G_1$. We change variables $a_i\mapsto
a_ia_{i+1}$ for all $1\leq i<\smallidx-1$. Assume that $\Lambda$ is
either $G_1^{0,k}$ or $G_1^{\infty,k}$ for some $k>0$. Then either
$\lfloor x\rfloor=b$ for all $x\in\Lambda$, or $\lfloor
x\rfloor=b^{-1}$ for all $x\in\Lambda$. We also change
$a_{\smallidx-1}\mapsto a_{\smallidx-1}\lfloor x\rfloor$. This
changes $\prod_{i=1}^{\smallidx-1}\mu(a_i)\mu(\lfloor x\rfloor)$ to
$\mu(a_1)\prod_{i=2}^{\smallidx-1}\mu^2(a_i)\mu^2(\lfloor
x\rfloor)$. Then by fixing $x$ and all of the coordinates $a_i$
except some $a_{i_0}$, if $|a_{i_0}|$ is small the integral
vanishes. Hence $a_i$ is bounded from below for all $1\leq
i\leq\smallidx-1$. Note that $a_1,\ldots, a_{\smallidx-1}$ are also
bounded from above. Now fixing $a_1,\ldots, a_{\smallidx-1}$, if
$\lfloor x\rfloor$ is small, the integral vanishes whence $[x]$ is
bounded (by definition $[x]\geq1$). This shows that the
$dx$-integration is a finite sum. We conclude that the integral is a
polynomial. If $\Lambda$ is a compact open subgroup, we can ignore
the $dx$-integration altogether and again we see that the
coordinates of $a$ are bounded from above and below.
\end{proof} 

\begin{example}\label{example:choosing tau so that twist has no poles}
Let $\tau'$ be a unitary irreducible supercuspidal representation of
$\GL{\bigidx}$ and take a unitary ramified character $\mu$. Set
$\tau=\tau'\mu$. According to Claim~\ref{claim:properties of
sufficiently ramified characters} we can take $\mu$ such that
$\mu^{2\bigidx}$ is sufficiently highly ramified (with respect to $\tau'$) so
that no unramified twist of $\tau$ (or $\tau^*$) would be self-dual.
Then $L(\tau,Sym^2,2s)=1$ by Shahidi \cite{Sh5} (Theorem 6.2). Now
equality~\eqref{eq:gamma up to units} implies
\begin{align*}
\nintertwiningfull{\tau^*}{1-s}=e(q^{-s},q^{s})L(\tau^*,Sym^2,1-2s)^{-1}\intertwiningfull{\tau^*}{1-s}
\end{align*}
for some $e(q^{-s},q^{s})\in\C[q^{-s},q^s]^*$ and since by
Theorem~\ref{theorem:poles of normalized intertwining for tau
cuspidal}, $L(\tau^*,Sym^2,1-2s)^{-1}\intertwiningfull{\tau^*}{1-s}$
is holomorphic ($\tau^*\isomorphic\dualrep{\tau}$ is irreducible
supercuspidal), $\ell_{\tau^*}(1-s)=1$. Again according to
Claim~\ref{claim:properties of sufficiently ramified characters}, we
can assume that $\mu^{2\bigidx}$ is also sufficiently highly ramified with
respect to $\pi$, then so is $\mu^2$. Then by Proposition~\ref{proposition:gcd for
ramified twist of tau}, $\gcd(\pi\times\tau,s)=1$.
\end{example}

There is an analogous result for a ramified character $\pi$ of
$G_1$.
\begin{proposition}\label{proposition:gcd for l is 1 and pi is ramified}
Let $\pi$ be a representation of $G_1$. If $\pi$ is a sufficiently
highly ramified character or $G_1$ is \quasisplit,
\begin{align*}
\gcd(\pi\times\tau,s)\in \ell_{\tau^*}(1-s)\C[q^{-s},q^s].
\end{align*}
\end{proposition}
\begin{proof}[Proof of Proposition~\ref{proposition:gcd for l is 1 and pi is ramified}] 
If $G_1$ is \quasisplit, the result follows immediately by looking
at \eqref{int:iwasawa decomposition of the integral l<=n split}.
Regarding the split case, as in the proof of
Proposition~\ref{proposition:gcd for ramified twist of tau} consider
\begin{align*}
&\int_{G_1}ch_{\Lambda}(x)\pi(x)W'(diag(\lfloor
x\rfloor,I_{\bigidx-1}))
[x]^{\half\bigidx-s-\half}dx.
\end{align*}
Since $|\lfloor x\rfloor|$ tends to zero as $x$ leaves a compact
subset of $G_1$, we see that by selecting a sufficiently highly
ramified $\pi$, with respect to the functions in
$\mathcal{A}_{\tau}$ as in the proof above (or if $\bigidx=1$, with
respect to $\tau$, which is then a character), the integral vanishes
for all $x$ outside of some compact subset (which depends on $W'$).
\end{proof} 

\section{Preserving poles while increasing $\bigidx$}\label{section:Locating poles in induced representations}
Let $\pi$ and $\tau$ be a pair of representations of $G_{\smallidx}$
and $\GL{\bigidx}$ (resp.). We show that the fractional ideal
$\mathcal{I}_{\pi\times\tau}(s)$ is contained in
$\mathcal{I}_{\pi\times\varepsilon}(s)$ for a representation
$\varepsilon$ induced from $\tau$ and another auxiliary
representation. This result 
will be used in Section~\ref{section:Sub-multiplicativity in the
first variable} to prove Theorem~\ref{theorem:gcd sub multiplicity
first var}. It suggests an inductive passage from $\pi\times\tau$ to
$\pi\times\varepsilon$ which increases $\bigidx$ while preserving
the original poles, hence any upper bound of
$\gcd(\pi\times\varepsilon,s)$ implies a similar bound of
$\gcd(\pi\times\tau,s)$.

\begin{lemma}\label{lemma:sub multiplicativity embed some sections in gcd second var l<=n}
Let $\tau_1$ be a representation of $\GL{m}$ such that
$\smallidx<m+\bigidx$ and
$\varepsilon=\cinduced{P_{m,\bigidx}}{\GL{m+\bigidx}}{\tau_1\otimes\tau}$
is irreducible. For $W\in\Whittaker{\pi}{\psi_{\gamma}^{-1}}$, set
$W_0=W$ if $\smallidx\leq\bigidx$ or $G_{\smallidx}$ is split,
otherwise $W_0=y^{-1}\cdot W$ where $y$ is given by
Lemma~\ref{lemma:integration formula for quotient space H_n G_n+1}.
Then for any $f_s\in\xi_{Q_{\bigidx}}^{H_{\bigidx}}(\tau,hol,s)$
there is
$f_s'\in\xi_{Q_{m+\bigidx}}^{H_{m+\bigidx}}(\varepsilon,hol,s)$,
where $\varepsilon$ is realized in $\Whittaker{\varepsilon}{\psi}$,
such that for all $W\in\Whittaker{\pi}{\psi_{\gamma}^{-1}}$,
\begin{align*}
\Psi(W_0,f'_s,s)=\Psi(W,f_s,s).
\end{align*}
\end{lemma}

\begin{proof}[Proof of Lemma~\ref{lemma:sub multiplicativity embed some sections in gcd second var l<=n}] 
We use the notation and results of
Section~\ref{subsection:Realization of tau for induced
representation} for $\zeta=0$. We will define
$\varphi_s\in\xi_{Q_{m,\bigidx}}^{H_{m+\bigidx}}(\tau_1\otimes\tau,hol,(s,s))$
such that $I_1=\Psi(W_0,\varphi_s,s)$ is absolutely convergent for
$\Re(s)>>0$ and equals $\Psi(W,f_s,s)$. Then
$f_s'=\widehat{f}_{\varphi_s}$ is defined by \eqref{iso:iso 1} and
according to Claim~\ref{claim:deducing varphi and f_s are
interchangeable in psi}, for all $\Re(s)>>0$,
$\Psi(W_0,\widehat{f}_{\varphi_s},s)=I_1=\Psi(W,f_s,s)$. Hence
$\Psi(W_0,\widehat{f}_{\varphi_s},s)=\Psi(W,f_s,s)$, as functions in
$\C(q^{-s})$. The proof depends on the relative sizes of $\smallidx$
and $\bigidx$.
\begin{enumerate}[leftmargin=*]
\item\label{case:embedding tau l<=n} $\smallidx\leq\bigidx$.
In Sections~\ref{subsection:n_1<l<=n_2} and
\ref{subsection:l<=n_1,n_2} we transformed $I_1$ into
\eqref{int:mult l<n_2 before applying eq for inner tau}, which with
the current notation takes the form
\begin{align*}
I_2=\int_{U_m} \int_{\lmodulo{U_{G_{\smallidx}}}{G_{\smallidx}}}W(g)
\int_{R_{\smallidx,\bigidx}}
\varphi_s(w'u,1,w_{\smallidx,\bigidx}r'(\rconj{b_{\bigidx,m}}g),1)
\psi_{\gamma}(r')\psi_{\gamma}(u)dr'dgdu
\end{align*}
($W_0=W$ since $\smallidx\leq\bigidx$). Here
\begin{align*}
w'= \left(\begin{array}{ccc}
&&I_m\\
&b_{\bigidx,m}\\
I_m
\end{array}\right),\qquad b_{\bigidx,m}=diag(I_{\bigidx},(-1)^m,I_{\bigidx}).
\end{align*}
The subgroup $R_{\smallidx,m+\bigidx}$ (in $I_1$) is decomposed as
$R_{\smallidx,\bigidx}\ltimes R_m$ where
$R_m=R_{\smallidx,m+\bigidx}\cap U_m$, $dr=dr'dr_m$ and in $I_2$,
$du=dzdr_m$. Integral $I_2$ is absolutely convergent at $s$ if it is
finite when we replace $W,\varphi_s$ with $|W|,|\varphi_s|$ and drop
the characters. The manipulations of
Section~\ref{subsection:n_1<l<=n_2} prove the following claim.
\begin{claim}\label{claim:embedding tau l<=n integral manipulation}
Integral $I_1$ is absolutely convergent at $s$ if and only if $I_2$
is absolutely convergent at $s$. When either is absolutely
convergent, $I_1=I_2$.
\end{claim}
Using this claim we derive the result. We follow the argument of
Soudry \cite{Soudry2} (Lemma~3.4) 
to find a specific $\varphi_s$
for which $I_2=\Psi(W,f_s,s)$, almost what we need. 
Set $N=\mathcal{N}_{H_{m+\bigidx},k}$, $k>>0$. Let $\varphi_s\in
\xi_{Q_{m,\bigidx}}^{H_{m+\bigidx}}(\tau_1\otimes\tau,hol,(s,s))$ be
such that as a function on $H_{m+\bigidx}$,
$\support{\varphi_s}=Q_mw'N$, $\varphi_s$ is right-invariant by $N$
and for $x\in H_{\bigidx}$ and $y\in\GL{\bigidx}$,
\begin{align*}
\varphi_{s}(w',1,x,y)=f_s^{b_{\bigidx,m}}(x,y)=f_s(\rconj{b_{\bigidx,m}}x,y).
\end{align*}
(The function $f_s^{b_{\bigidx,m}}$ was defined in
Section~\ref{subsection:Twisting the embedding}.) Note that
$\rconj{w'}N=N$.
Then $w'u\in\support{\varphi_s}$ if and only if
$\rconj{(w')^{-1}}u\in N$, because
$\rconj{(w')^{-1}}u\in\overline{U_m}$. 
Formally, for this $\varphi_s$,
\begin{align*}
I_2&=c\int_{\lmodulo{U_{G_{\smallidx}}}{G_{\smallidx}}}W(g)
\int_{R_{\smallidx,\bigidx}}
\varphi_s(w',1,w_{\smallidx,\bigidx}r'(\rconj{b_{\bigidx,m}}g),1)
\psi_{\gamma}(r')dr'dg\\\notag&=c\int_{\lmodulo{U_{G_{\smallidx}}}{G_{\smallidx}}}W(g)\int_{R_{\smallidx,\bigidx}}
f_s^{b_{\bigidx,m}}(\rconj{b_{\bigidx,m}}(w_{\smallidx,\bigidx}r'g),1)
\psi_{\gamma}(r')dr'dg=c\Psi(W,f_s,s).
\end{align*}
Here $c=vol(U_m\cap N)$. The same reasoning implies
\begin{align}\label{integral:first substitution n_1 < l < n_2 first side absolute value}
&\int_{U_m} \int_{\lmodulo{U_{G_{\smallidx}}}{G_{\smallidx}}}
\int_{R_{\smallidx,\bigidx}} |W|(g)
|\varphi_s|(w'u,1,w_{\smallidx,\bigidx}r'(\rconj{b_{\bigidx,m}}g),1)dr'dgdu\\\notag
&=c\int_{\lmodulo{U_{G_{\smallidx}}}{G_{\smallidx}}}\int_{R_{\smallidx,\bigidx}}|W|(g)|f_s|(w_{\smallidx,\bigidx}r'g,1)dr'dg.
\end{align}
Now we use the idea of Jacquet, Piatetski-Shapiro and Shalika
\cite{JPSS} (p.~424). Since $\Psi(W,f_s,s)$ is absolutely convergent
for $\Re(s)>>0$ (i.e. the \rhs\ of \eqref{integral:first
substitution n_1 < l < n_2 first side absolute value} is finite),
the \lhs\ of \eqref{integral:first substitution n_1 < l < n_2 first
side absolute value} is finite. Hence we can apply
Claim~\ref{claim:embedding tau
l<=n integral manipulation} 
to conclude firstly that $I_1$ is absolutely convergent for
$\Re(s)>>0$ and secondly, $I_1=I_2=c\Psi(W,f_s,s)$.
\item\label{case:embedding tau l>n} $\smallidx>\bigidx$. The idea of the proof is the same as in the previous
case but the actual integral manipulations and selection of
$\varphi_s$ differ. Regard $H_{\bigidx}$ as a subgroup of
$G_{\bigidx+1}$, embedded in $G_{\smallidx}$ (which is embedded in
$H_{m+\bigidx}$) through $L_{\smallidx-\bigidx-1}$. Set
$r=\smallidx-\bigidx-1$, $\epsilon=(-1)^{m+\bigidx-\smallidx}$ and
$t_{\bigidx}=diag(\epsilon\gamma^{-1}I_{\bigidx},1,\epsilon\gamma
I_{\bigidx})\in H_{\bigidx}$. Let
\begin{align*}
I_2=&\int_{\lmodulo{H_{\bigidx}Z_{r}V_{r}}{G_{\smallidx}}}
\int_{R_{\smallidx,m+\bigidx}}\int_{\Mat{\bigidx\times
m+\bigidx-\smallidx}} \int_{\lmodulo{U_{H_{\bigidx}}}{H_{\bigidx}}}
(\int_{R^{\smallidx,\bigidx}}W_0(r'w^{\smallidx,\bigidx}h'g)dr')\\\notag&\varphi_s(\omega_{m,\bigidx}\eta
w_{\smallidx,m+\bigidx}rw^{\smallidx,\bigidx}g,1,\rconj{t_{\bigidx}}h',1)\psi_{\gamma}(r)dh'd\eta
drdg.
\end{align*}
Integral $I_2$ is absolutely convergent at $s$ if it is finite with
$|W_0|,|\varphi_s|$ and without $\psi_{\gamma}$. According to
Claims~\ref{claim:n_1 < l < n_2 properties of F} and
\ref{claim:n_2<l<=n_1 properties of F}, for any $s$, $I_1$ is
absolutely convergent if and only if $I_2$ is, and in this case
$I_1=I_2$.

Now we describe the choice of $\varphi_s$. Write
$w_{\smallidx,m+\bigidx}=t_{\gamma}w'_{\smallidx,m+\bigidx}$ for
\begin{align*}
&t_{\gamma}=diag(\gamma
I_{\smallidx},I_{2(m+\bigidx-\smallidx)+1},\gamma^{-1}I_{\smallidx})\in
H_{m+\bigidx},\\
&w'_{\smallidx,m+\bigidx}=\left(\begin{array}{ccccc}&I_{\smallidx}\\&&&&I_{m+\bigidx-\smallidx}\\
&&\epsilon\\I_{m+\bigidx-\smallidx}\\&&&I_{\smallidx}\end{array}\right).
\end{align*}
Further put $t_{\gamma}'=diag(\gamma I_{\bigidx},1,\gamma^{-1}
I_{\bigidx})\in H_{\bigidx}$, $t_{\gamma}''=diag(\gamma
I_{\smallidx-\bigidx},I_{m+\bigidx-\smallidx})\in\GL{m}$ and
$w^{\star}=\omega_{m,\bigidx}w'_{\smallidx,m+\bigidx}w^{\smallidx,\bigidx}$.
Let $y^{\star}=[y]_{\mathcal{E}_{H_{m+\bigidx}}}$ where $y\in
G_{\bigidx+1}$ is either $I_{2(\bigidx+1)}$ if $G_{\smallidx}$ is
split, or in the \quasisplit\ case
\begin{align*}
y=\left(\begin{array}{cccccc}I_{\bigidx-1}\\&1\\&0&1\\&2\rho^{-2}&0&1\\&2\rho^{-3}&0&2\rho^{-1}&1\\&&&&&I_{\bigidx-1}\end{array}\right)
\end{align*}
is the element defined in Lemma~\ref{lemma:integration formula for
quotient space H_n G_n+1}.

Pick $W_{\tau_1}\in\Whittaker{\tau_1}{\psi}$ such that the
restriction of $t_{\gamma}''\cdot W_{\tau_1}$ to the mirabolic
subgroup $Y_m$ of $\GL{m}$ vanishes outside of a small neighborhood
of the identity $\mathcal{N}_{\GL{m},o}\cap Y_m$ where $o>>0$, and
for $v\in\mathcal{N}_{\GL{m},o}\cap Y_m$, $t_{\gamma}''\cdot
W_{\tau_1}(v)=t_{\gamma}''\cdot W_{\tau_1}(I_m)\ne0$. Let
$\varphi_s\in
\xi_{Q_{m,\bigidx}}^{H_{m+\bigidx}}(\tau_1\otimes\tau,hol,(s,s))$ be
such that $\support{\varphi_s}=Q_mw^{\star}y^{\star}N$ with
$N=\mathcal{N}_{H_{m+\bigidx},k}$ and $k>>0$, and for all $a\in
\GL{m}<M_{m}$, $u\in N$, $x\in H_{\bigidx}$ and $b\in\GL{\bigidx}$,
\begin{align*}
\varphi_{s}(aw^{\star}y^{\star}u,1,x,b)=&
c_{\tau,\gamma}\delta_{Q_m}^{\half}(a)\absdet{a}^{s-\half}W_{\tau_1}(a)
((t_{\bigidx}t_{\gamma}')^{-1}\cdot f_s)(x,b).
\end{align*}
Here
$c_{\tau,\gamma}=\omega_{\tau}(\epsilon\gamma)^{-1}|\gamma|^{-\smallidx(\bigidx+\half
m+s-\half)-\half\bigidx(\bigidx+m-2\smallidx)}$. Note that
$w^{\star}$ and $y^{\star}$ normalize $N$, because
$w^{\star},y^{\star}\in K_{H_{m+\bigidx}}$ (recall from
Section~\ref{subsection:groups in study} that we assume
$|\rho|\geq1$). The constant $o$ is selected so that $W$ is
right-invariant by $\mathcal{N}_{G_{\smallidx},o}$. We take $k>>o$
so that $W_{\tau_1}$ is right-invariant by $\mathcal{N}_{\GL{m},k}$,
$W_0$ is right-invariant by $\mathcal{N}_{G_{\smallidx},k}$ and
$k-k_0\geq o$ where $k_0\geq0$ is a constant to be specified later
(depending only on $|2|$).
\begin{claim}\label{claim:embedding second var l>n integration formula}
$I_2=\Psi(W,f_s,s)$.
\end{claim}
As in the case $\smallidx\leq\bigidx$, it follows that $I_1$ is
absolutely convergent for $\Re(s)>>0$ and $I_1=I_2=\Psi(W,f_s,s)$.

\begin{proof}[Proof of Claim~\ref{claim:embedding second var l>n integration formula}] 
We use the formula
\begin{align*}
\int_{\lmodulo{H_{\bigidx}Z_rV_r}{G_{\smallidx}}} F(g)dg=
\int_{\lmodulo{Z_{r}}{\GL{r}}}\int_{\overline{V_r}}\int_{\lmodulo{H_{\bigidx}}{G_{\bigidx+1}}}
F(avg)\delta(a)dgdvda.
\end{align*}
Here $\delta(a)$ is a suitable modulus character. The integral
becomes
\begin{align*}
&\int_{\lmodulo{Z_{r}}{\GL{r}}}\int_{\overline{V_r}}\int_{\lmodulo{H_{\bigidx}}{G_{\bigidx+1}}}
\int_{R_{\smallidx,m+\bigidx}} \int_{\Mat{\bigidx\times
m+\bigidx-\smallidx}}\int_{\lmodulo{U_{H_{\bigidx}}}{H_{\bigidx}}}(\int_{R^{\smallidx,\bigidx}}
W_0(r'w^{\smallidx,\bigidx}h'avg)dr')\\\notag &
\varphi_s(\omega_{m,\bigidx}\eta
w_{\smallidx,m+\bigidx}rw^{\smallidx,\bigidx}avg,1,\rconj{t_{\bigidx}}h',1)\psi_{\gamma}(r)\delta(a)dh'd\eta
 drdgdvda.
\end{align*}

First observe that $a\in\GL{r}$ can be shifted to the second
argument of $\varphi_s$. If
\begin{align*}
&[a]_{\mathcal{E}_{G_{\smallidx}}}=diag(a,I_{2\bigidx+2},a^*),\qquad
[\rconj{(w^{\smallidx,\bigidx})^{-1}}a]_{\mathcal{E}_{G_{\smallidx}}}=diag(I_{\bigidx},a,I_2,a^*,I_{\bigidx}).
\end{align*}
Then $\rconj{(w^{\smallidx,\bigidx})^{-1}}a\in
L_{\smallidx-1}<G_{\smallidx}$ normalizes $R_{\smallidx,\bigidx}$
and fixes $\psi_{\gamma}$. We obtain
\begin{align*}
&\int_{\lmodulo{Z_{r}}{\GL{r}}}\int_{\overline{V_r}}\int_{\lmodulo{H_{\bigidx}}{G_{\bigidx+1}}}
\int_{R_{\smallidx,m+\bigidx}} \int_{\Mat{\bigidx\times
m+\bigidx-\smallidx}}
\int_{\lmodulo{U_{H_{\bigidx}}}{H_{\bigidx}}}(\int_{R^{\smallidx,\bigidx}}
W_0(r'w^{\smallidx,\bigidx}h'avg)dr')\\\notag&\varphi_s(\omega_{m,\bigidx}\eta
w_{\smallidx,m+\bigidx}(\rconj{(w^{\smallidx,\bigidx})^{-1}}a)rw^{\smallidx,\bigidx}vg,1,\rconj{t_{\bigidx}}h',1)\psi_{\gamma}(r)
\absdet{a}^{\smallidx-m-\bigidx}\delta(a)dh'd\eta drdgdvda.
\end{align*}
Now we see that
$\rconj{w_{\smallidx,m+\bigidx}^{-1}}(\rconj{(w^{\smallidx,\bigidx})^{-1}}a)$
commutes with $\eta$ and after conjugating it by
$\omega_{m,\bigidx}$ it lies in $Y_m<\GL{m}$, where $\GL{m}$ is
viewed as a subgroup of $M_{m}<H_{m+\bigidx}$. We get
\begin{align*}
&\int_{\lmodulo{Z_{r}}{\GL{r}}}\int_{\overline{V_r}}\int_{\lmodulo{H_{\bigidx}}{G_{\bigidx+1}}}
\int_{R_{\smallidx,m+\bigidx}} \int_{\Mat{\bigidx\times
m+\bigidx-\smallidx}}
\int_{\lmodulo{U_{H_{\bigidx}}}{H_{\bigidx}}}(\int_{R^{\smallidx,\bigidx}}
W_0(r'w^{\smallidx,\bigidx}h'avg)dr')\\\notag&\varphi_s(\omega_{m,\bigidx}\eta
w_{\smallidx,m+\bigidx}rw^{\smallidx,\bigidx}vg,a,\rconj{t_{\bigidx}}h',1)\psi_{\gamma}(r)\delta_s'(a)dh'd\eta
drdgdvda.
\end{align*}
Here
$\delta_s'(a)=\delta_{Q_m}^{\half}(a)\absdet{a}^{\smallidx-m-\bigidx+s-\half}\delta(a)$.

Put $w_{\smallidx,m+\bigidx}=t_{\gamma}w'_{\smallidx,m+\bigidx}$ and
change variables $\eta\mapsto\rconj{t_{\gamma}^{-1}}\eta$. This
change multiplies $d\eta$ by
$|\gamma|^{\bigidx(m+\bigidx-\smallidx)}$. Denote
\begin{align*}
u^{\star}=(\rconj{\omega_{m,\bigidx}^{-1}}\eta)(\rconj{(\omega_{m,\bigidx}w'_{\smallidx,m+\bigidx})^{-1}}(r(\rconj{(w^{\smallidx,\bigidx})^{-1}}v))).
\end{align*}
Then $\omega_{m,\bigidx}t_{\gamma}\eta
w_{\smallidx,m+\bigidx}'rw^{\smallidx,\bigidx}vg=(\rconj{\omega_{m,\bigidx}^{-1}}t_{\gamma})u^{\star}w^{\star}g$.
We write the coordinates of $u^{\star}$ explicitly. Let
$v\in\overline{V_r}$,
\begin{align*}
&[v]_{\mathcal{E}_{G_{\smallidx}}}=\left(\begin{array}{cccccc}I_r\\v_1&I_{\bigidx}\\v_2&0&1\\v_3&0&0&1\\v_4&0&0&0&I_{\bigidx}\\v_5&v_4'&*&*&v_1'&I_r\\\end{array}\right),\\\notag
&[\rconj{(w^{\smallidx,\bigidx})^{-1}}v]_{\mathcal{E}_{H_{m+\bigidx}}}=\left(\begin{array}{ccccccccc}
I_{m+\bigidx-\smallidx}\\
&I_{\bigidx}&v_1\\
&&I_r\\
&&u_1&1\\
&&u_2&0&1\\
&&u_3&0&0&1\\
&v_4'&v_5&u_3'&u_2'&u_1'&I_r&v_1'\\
&&v_4&&&&&I_{\bigidx}\\
&&&&&&&&I_{m+\bigidx-\smallidx}\end{array}\right),
\end{align*}
\begin{align*}
r=\left(\begin{array}{ccccc}I_{m+\bigidx-\smallidx}&r_1&r_2&0&r_3\\&I_{\smallidx}&&&0\\&&1&&r_2'\\&&&I_{\smallidx}&r_1'\\&&&&I_{m+\bigidx-\smallidx}\\\end{array}\right)\in
R_{\smallidx,m+\bigidx},
\eta=\left(\begin{array}{ccc}I_{\bigidx}&&\eta\\&I_{\smallidx-\bigidx}\\&&I_{m+\bigidx-\smallidx}\\\end{array}\right).
\end{align*}
Also denote $r_1=(r_{1,1},r_{1,2},r_{1,3})$ where $r_{1,1}\in\Mat{m+\bigidx-\smallidx\times\bigidx}$, $r_{1,2}\in\Mat{m+\bigidx-\smallidx\times r}$.
Then
\begin{align*}
u^{\star}=\left(\begin{array}{ccccccccc}
I_r\\
u_1&1\\
0&0&I_{m+\bigidx-\smallidx}\\
v_1&0&\eta&I_{\bigidx}\\
\epsilon u_2&0&\epsilon r_2'&0&1\\
v_4&0&r_{1,1}'&0&0&I_{\bigidx}\\
r_{1,2}+\ldots&r_{1,3}&r_3+\eta' r_{1,1}'&r_{1,1}&\epsilon r_2&\eta'&I_{m+\bigidx-\smallidx}\\
u_3&0&r_{1,3}'&0&0&0&0&1\\
v_5&u_3'&r_{1,2}'&v_4'&\epsilon u_2'&v_1'&0&u_1'&I_r\\\end{array}\right).
\end{align*}

We show that the $dg$-integration can be ignored. Assume first that
$G_{\smallidx}$ is split (then $y^{\star}=I_{2(m+\bigidx)+1}$).
Using the integration formula for
$\lmodulo{H_{\bigidx}}{G_{\bigidx+1}}$ of
Lemma~\ref{lemma:integration formula for quotient space H_n G_n+1},
in the notation of the lemma, the $dg$-integration is replaced with
an integration over $\GL{1}\times\overline{Z_{\bigidx,1}}\times
\Xi_{\bigidx}$.  Write $g=b\lambda\alpha$ with $b\in \Xi_{\bigidx}$,
$\lambda\in\overline{Z_{\bigidx,1}}$ and $\alpha\in \GL{1}$. In
coordinates,
\begin{align*}
b=\left(\begin{array}{cccc}I_{\bigidx}&0&b&0\\&1&0&b'\\&&1&0\\&&&I_{\bigidx}\end{array}\right)
,\qquad \lambda\alpha=\left(\begin{array}{cccc}I_{\bigidx}&0&0&0\\\lambda&\alpha&0&0\\&&\alpha^{-1}&0\\&&\lambda'\alpha^{-1}&I_{\bigidx}\end{array}\right)\quad(\lambda'=-J_{\bigidx}\transpose{\lambda}).
\end{align*}
Then
\begin{align*}
&[g]_{\mathcal{E}_{H_{\bigidx+1}}}=\left(\begin{array}{ccccc}
I_{\bigidx}&-\gamma b\alpha^{-1}&\beta b\alpha^{-1}&b\alpha^{-1}&0\\
\lambda&h_{1,1}&h_{1,2}&h_{1,3}&b'\\
\beta \lambda&h_{2,1}&h_{2,2}&h_{2,3}&\beta b'\\
-\gamma \lambda&h_{3,1}&h_{3,2}&h_{3,3}&-\gamma b'\\
0&-\gamma \lambda'\alpha^{-1}&\beta \lambda'\alpha^{-1}&\lambda'\alpha^{-1}&I_{\bigidx} \\\end{array}\right), \qquad(h_{i,j})\in\Mat{3\times3}.
\end{align*}
Set $g^{\star}=\rconj{(w^{\star})^{-1}}g$.
It can be verified that the last $m$ rows of $u^{\star}g^{\star}$ take the form
\begin{align*}
\left(\begin{array}{ccccccccc}
*&*&*&*&*&*&I_{m+\bigidx-\smallidx}&*&0\\
u_3&h_{3,1}&r_{1,3}'&-\gamma \lambda&\epsilon h_{3,2}&-\gamma b'&0&h_{3,3}&0\\
*&*&*&*&*&*&0&*&I_r\\\end{array}\right).
\end{align*}
In order for
$(\rconj{\omega_{m,\bigidx}^{-1}}t_{\gamma})u^{\star}g^{\star}w^{\star}$
to belong to the support of $\varphi_s$ we must have
$u^{\star}g^{\star}\in Q_mN$. We proceed to argue that in such a
case $g\in\mathcal{N}_{G_{\bigidx+1},k}\cap N$, then the coordinates
of $v$, $\eta$ and $r$ will be small ($v$ will ``almost" belong to
$\mathcal{N}_{G_{\smallidx},k}$, $\eta$ and $r$ will lie in $N$).

Let $x$ be the lower-right $m\times m$ block of
$u^{\star}g^{\star}$, i.e.,
\begin{align*}
x=\left(\begin{array}{ccc}
I_{m+\bigidx-\smallidx}&*&0\\
0&h_{3,3}&0\\
0&*&I_r\\\end{array}\right)\in\Mat{m\times m}.
\end{align*}
If $x\notin\GL{m}$ then $u^{\star}g^{\star}\notin Q_mN$, so assume
that $x$ is invertible, equivalently $h_{3,3}\ne0$. Then
$x_1=diag((x^*)^{-1},I_{2\bigidx+1},x^{-1})\in Q_m$
and the last $m$ rows of $x_1u^{\star}g^{\star}$ are of the form
\begin{align*}
\left(\begin{array}{ccccccccc}
*&*&*&*&*&*&I_{m+\bigidx-\smallidx}&0&0\\
h_{3,3}^{-1}u_3&h_{3,3}^{-1}h_{3,1}&h_{3,3}^{-1}r_{1,3}'&-\gamma h_{3,3}^{-1}\lambda&\epsilon h_{3,3}^{-1}h_{3,2}&-\gamma h_{3,3}^{-1}b'&0&1&0\\
*&*&*&*&*&*&0&0&I_r\\\end{array}\right).
\end{align*}
Now $u^{\star}g^{\star}\in Q_mN$ if and only if
$x_1u^{\star}g^{\star}\in Q_mN$. If $qx_1u^{\star}g^{\star}\in N$
for some $q\in Q_m$ and the Levi part of $q$ is
$diag(m(q)^*,h(q),m(q))$ ($h(q)\in H_{\bigidx}$), then
$m(q)\in\mathcal{N}_{\GL{m},k}$, whence
$m(q)_1=diag((m(q)^*)^{-1},I_{2\bigidx+1},m(q)^{-1})\in N$ and
$m(q)_1qx_1u^{\star}g^{\star}\in N$. Therefore the coordinates of
lower-left $m\times(m+2\bigidx+1)$ block of $x_1u^{\star}g^{\star}$
(the first $m+2\bigidx+1$ columns of the matrix above) belong to
$\mathcal{P}^k$. In particular
\begin{align*}
|h_{3,3}^{-1}h_{3,1}|,|h_{3,3}^{-1}h_{3,2}|\leq q^{-k}.
\end{align*}
Put $c=b'\lambda'$. The last row of $(h_{i,j})$ takes the form
\begin{align*}
&\left(\begin{array}{ccc}
\frac{\gamma}{4\alpha}(-\alpha^2+2c\beta^2-1+2\alpha)&-\frac{\beta}{4\alpha}(\alpha^2+2c\beta^2-1)&
-\frac{1}{4\alpha}(-\alpha^2+2c\beta^2-1-2\alpha)\end{array}\right).
\end{align*}

Our first step is to show $|\alpha|=1$. Since
$|h_{3,3}^{-1}h_{3,1}|\leq q^{-k}$ and $k>>0$,
\begin{align*}
|-\alpha^2+2c\beta^2-1+2\alpha|<|-\alpha^2+2c\beta^2-1-2\alpha|.
\end{align*}
Hence
\begin{align*}
|-\alpha^2+2c\beta^2-1-2\alpha|=|-\alpha^2+2c\beta^2-1-2\alpha-(-\alpha^2+2c\beta^2-1+2\alpha)|=|4\alpha|
\end{align*}
and
\begin{align*}
|h_{3,3}|=|4\alpha|^{-1}|-\alpha^2+2c\beta^2-1-2\alpha|=1.
\end{align*}
Let $\xi=\alpha-(2c\beta^2-1)\alpha^{-1}$. If $|\xi|>|2|$,
\begin{align*}
|h_{3,3}|=|2^{-1}+4^{-1}\xi|=|4^{-1}\xi|>|2|^{-1}\geq1,
\end{align*}
contradiction. Also if $|\xi|<|2|$, $|h_{3,1}|=|\half\gamma|$ contradicting $|h_{3,3}^{-1}h_{3,1}|\leq q^{-k}$.
Thus $|\xi|=|2|$. Also $|h_{3,3}^{-1}h_{3,2}|\leq q^{-k}$ so
\begin{align*}
|\alpha+(2c\beta^2-1)\alpha^{-1}|\leq |4|q^{-k}\leq q^{-k}
\end{align*}
(recall that we assume $|\beta|=1$), whence
$\alpha=-(2c\beta^2-1)\alpha^{-1}+\varpi^k\theta$ with
$|\theta|\leq1$. Then
\begin{align*}
|2|=|\xi|=|2(2c\beta^2-1)\alpha^{-1}-\varpi^k\theta|\leq\max(|2||(2c\beta^2-1)\alpha^{-1}|,q^{-k}).
\end{align*}
Since $k>>0$, we first get $|(2c\beta^2-1)\alpha^{-1}|\geq1$, then
$|2|=|2||(2c\beta^2-1)\alpha^{-1}|$ so
$|(2c\beta^2-1)\alpha^{-1}|=1$ and $|\alpha|=1$.

The second step is to prove that $\lambda$ and $b$ can be made zero.
Because $|\alpha|=1$ and for all $1\leq i\leq\bigidx$, $|\gamma
\lambda_i|=|\gamma h_{3,3}^{-1}\lambda_i|\leq q^{-k}$ and $|\gamma
b'_i|=|\gamma h_{3,3}^{-1}b'_i|\leq q^{-k}$, the element $y\in
G_{\bigidx+1}$ given by
\begin{align*}
&[y]_{\mathcal{E}_{G_{\bigidx+1}}}=\left(\begin{array}{cccc}I_{\bigidx}&\\-\alpha^{-1}\lambda&1\\&&1&\\&&-\alpha^{-1}\lambda'&I_{\bigidx}\\\end{array}\right)
\left(\begin{array}{cccc}I_{\bigidx}&0&-\alpha^{-1}b&0\\&1&0&-\alpha^{-1}b'\\&&1&0\\&&&I_{\bigidx}\\\end{array}\right)
\end{align*}
belongs to $\mathcal{N}_{G_{\bigidx+1},k}$ and also when viewed as an element of $H_{m+\bigidx}$, $y\in N$.
Set $g_1=gy$,
\begin{align*}
&[g_1]_{\mathcal{E}_{G_{\bigidx+1}}}=\left(\begin{array}{cccc}I_{\bigidx}\\&\alpha&&\\&&\alpha^{-1}\\&&&I_{\bigidx}\\\end{array}\right).
\end{align*}

The third step is to show $g_1\in \mathcal{N}_{G_{\bigidx+1},k}\cap
N$. We may replace $g$ with $g_1$ and argue as above, because
$(\rconj{\omega_{m,\bigidx}^{-1}}t_{\gamma})u^{\star}g^{\star}w^{\star}\in
Q_mw^{\star}N$ if and only if
$(\rconj{\omega_{m,\bigidx}^{-1}}t_{\gamma})u^{\star}g_1^{\star}w^{\star}\in
Q_mw^{\star}N$, with $g_1^{\star}=\rconj{(w^{\star})^{-1}}g_1$. We
conclude $|h_{3,3}|=1$ and $|\alpha|=|\alpha|^{-1}=1$. Now
$(h_{i,j})$ is given by \eqref{eq:embedding of SO_2 split} and we
proceed as in the proof of Lemma~\ref{lemma:f with small support}.
Specifically, $|h_{3,2}|\leq q^{-k}$ implies
$\alpha-\alpha^{-1}\in4\mathcal{P}^{k}$ and then
$(\alpha-1)(\alpha+1)\in4\mathcal{P}^{k}$. If
$\alpha+1\in4\mathcal{P}$, we get $|h_{3,1}|=|\gamma|$,
contradiction. Therefore $|\alpha+1|>|4|q^{-1}$ so
$\alpha\in1+\mathcal{P}^k$ and thus
$g_1\in\mathcal{N}_{G_{\bigidx+1},k}$. Looking at
\eqref{eq:embedding of SO_2 split}, the conditions
$|h_{3,1}|,|h_{3,2}|\leq q^{-k}$ imply $g_1\in N$. 
Because
$y,g_1\in \mathcal{N}_{G_{\bigidx+1},k}\cap N$, $g=g_1y^{-1}\in
\mathcal{N}_{G_{\bigidx+1},k}\cap N$.

Now assume that $G_{\smallidx}$ is \quasisplit. By virtue of
Lemma~\ref{lemma:integration formula for quotient space H_n G_n+1}
we can replace the $dg$-integration with an integration over
$\GL{1}\times \overline{Z_{\bigidx-1,1}}\times (V_{\bigidx}\cap
G_2)\times \Xi_{\bigidx}$. Put $g=yb\varsigma\lambda\alpha$ with
$b\in \Xi_{\bigidx}$, $\varsigma\in V_{\bigidx}\cap G_2$,
$\lambda\in\overline{Z_{\bigidx-1,1}}$ and $\alpha\in \GL{1}$. In
coordinates,
\begin{align*}
b=\left(\begin{array}{cccccc}I_{\bigidx-1}&0&0&\varrho&0&A(\varrho)\\&1&&&&0\\&&1&&&0\\&&&1&&\varrho'\\&&&&1&0\\&&&&&I_{\bigidx-1}\end{array}\right),\quad
\varsigma=\left(\begin{array}{cccccc}I_{\bigidx-1}&&&&&\\&1&\varsigma_1&\varsigma_2&*&\\&&1&0&\varsigma_1'\\&&&1&\varsigma_2'\\&&&&1&
\\&&&&&I_{\bigidx-1}\end{array}\right)
\end{align*}
and
\begin{align*}
\lambda\alpha=\left(\begin{array}{cccccc}I_{\bigidx-1}&&&&&\\\lambda&\alpha&&&&\\&&1&\\&&&1\\&&&&\alpha^{-1}&
\\&&&&\lambda'\alpha^{-1}&I_{\bigidx-1}\end{array}\right)\qquad(\lambda'=-J_{\bigidx-1}\transpose{\lambda}).
\end{align*}
Set
$(b\varsigma\lambda\alpha)^{\star}=\rconj{(w^{\star}y^{\star})^{-1}}(b\varsigma\lambda\alpha)$.
Then
\begin{align*}
u^{\star}w^{\star}g=
u^{\star}w^{\star}y^{\star}(b\varsigma\lambda\alpha)=
u^{\star}(b\varsigma\lambda\alpha)^{\star}w^{\star}y^{\star}.
\end{align*}
Now $u^{\star}w^{\star}g\in Q_mw^{\star}y^{\star}N$ if and only if
$u^{\star}(b\varsigma\lambda\alpha)^{\star}\in Q_mN$
($\rconj{(w^{\star}y^{\star})^{-1}}N=N$). Let $x$ be the lower-right
$m\times m$ block of $u^{\star}(b\varsigma\lambda\alpha)^{\star}$.
As in the split case $x\notin\GL{m}$ implies
$u^{\star}(b\varsigma\lambda\alpha)^{\star}\notin Q_mN$. One shows
that the coordinates of $\varsigma_1,\varsigma_2$ and $\lambda$ are
small, and also $|1-\alpha|$ is small. Then the coordinates of
$\varrho$ are seen to be small and
$b\varsigma\lambda\alpha\in\mathcal{N}_{G_{\bigidx+1},k}\cap N$.

We continue with the proof, for both split and \quasisplit\
$G_{\smallidx}$. Going back to the integral, since $W_0$ is
invariant on the right for $\mathcal{N}_{G_{\smallidx},k}$, the
$dg$-integration reduces to a volume constant which may be ignored.
Thus we have
\begin{align*}
&|\gamma|^{\bigidx(m+\bigidx-\smallidx)}\int_{\lmodulo{Z_{r}}{\GL{r}}}\int_{\overline{V_r}}
\int_{R_{\smallidx,m+\bigidx}} \int_{\Mat{\bigidx\times
m+\bigidx-\smallidx}}
\int_{\lmodulo{U_{H_{\bigidx}}}{H_{\bigidx}}}(\int_{R^{\smallidx,\bigidx}}
W_0(r'w^{\smallidx,\bigidx}h'avy)dr')\\\notag&\varphi_s((\rconj{\omega_{m,\bigidx}^{-1}}t_{\gamma})u^{\star}w^{\star}y^{\star},a,\rconj{t_{\bigidx}}h',1)\psi_{\gamma}(r)\delta_s'(a)dh'd\eta
drdvda.
\end{align*}
(In the \quasisplit\ case $y$ and $y^{\star}$ appear as a result of
the application of Lemma~\ref{lemma:integration formula for quotient
space H_n G_n+1}.)

Finally consider $u^{\star}$ again. Let
\begin{align*}
z=\left(\begin{array}{ccc}
I_{m+\bigidx-\smallidx}\\
&1\\
&u_1'&I_r\end{array}\right)\in\GL{m},\qquad
z_1=\left(\begin{array}{ccc}
(z^*)^{-1}\\
&I_{2\bigidx+1}\\
&&z^{-1}\end{array}\right)\in Q_m.
\end{align*}
Then $u^{\star}\in Q_mN$ (i.e.,
$(\rconj{\omega_{m,\bigidx}^{-1}}t_{\gamma})u^{\star}w^{\star}y^{\star}$
belongs to the support of $\varphi_s$) if and only if $z_1u^{\star}\in Q_mN$,
\begin{align*}
z_1u^{\star}=\left(\begin{array}{ccccccccc}
I_r\\
0&1\\
0&0&I_{m+\bigidx-\smallidx}\\
v_1&0&\eta&I_{\bigidx}\\
\epsilon u_2&0&\epsilon r_2'&0&1\\
v_4&0&r_{1,1}'&0&0&I_{\bigidx}\\
r_{1,2}+\ldots&r_{1,3}&r_3+\eta' r_{1,1}'&r_{1,1}&\epsilon r_2&\eta'&I_{m+\bigidx-\smallidx}\\
u_3&0&r_{1,3}'&0&0&0&0&1\\
v_5-u_1'u_3&u_3'&r_{1,2}'-u_1'r_{1,3}'&v_4'&\epsilon u_2'&v_1'&0&0&I_r\\\end{array}\right).
\end{align*}
As above the coordinates of the lower-left $m\times(m+2\bigidx+1)$ block lie in $\mathcal{P}^k$.
Inspecting the other (non-constant) coordinates of $z_1u^{\star}$, they all lie there and $z_1u^{\star}\in N$.
Moreover
\begin{align*}
r_{1,3},r_{1,1},r_2,\eta,u_3,v_4,u_2,v_1\in\mathcal{P}^k,
\end{align*}
where with a minor abuse of notation we write, e.g.,
$v_1\in\mathcal{P}^k$ instead of $v_1\in\MatF{\bigidx\times
r}{\mathcal{P}^k}$. Because $u_2=\beta v_2+\frac1{2\beta}v_3$ and
$u_3=-\gamma v_2+\quarter v_3$ we see that
$v_2\in\half\mathcal{P}^k$, $v_3\in\mathcal{P}^k$ and
$u_1=v_2-\frac1{2\beta^2}v_3\in\half\mathcal{P}^k$. Since
$u_1'u_3\in\half\mathcal{P}^{2k}$ and $k>>0$,
$u_1'u_3\in\mathcal{P}^k$ whence $v_5\in\mathcal{P}^k$. Similarly,
$r_{1,2}\in\mathcal{P}^k$ ($u_1'r_{1,3}'\in\half\mathcal{P}^{2k}$).
Also $r_3\in\mathcal{P}^k$. It follows that $\eta,r\in N$ hence the
$d\eta dr$-integration shrinks into a volume constant, hereby
ignored. Additionally each $v_i$, $i=\intrange{1}{5}$, lies in
$\half\mathcal{P}^k$, thus $v\in\mathcal{N}_{G_{\smallidx},k-k_0}$
where $k_0\geq0$ equals the valuation of $2$ (depends only on
$|2|$).

The integral becomes
\begin{align*}
&c_{\tau,\gamma}^{-1}
\int_{\lmodulo{Z_{r}}{\GL{r}}}\int_{\overline{V_r}\cap\mathcal{N}_{G_{\smallidx},k-k_0}}
\int_{\lmodulo{U_{H_{\bigidx}}}{H_{\bigidx}}}(\int_{R^{\smallidx,\bigidx}}
\\\notag& W(r'w^{\smallidx,\bigidx}h'av)dr')\varphi_s(w^{\star}y^{\star},
\left(\begin{array}{ccc}a\\u_1&1\\&&I_{m+\bigidx-\smallidx}\end{array}\right)t_{\gamma}'',h't_{\bigidx}t_{\gamma}',1)\delta_s'(a)dh'
dvda.
\end{align*}
(Note that
$\rconj{\omega_{m,\bigidx}^{-1}}t_{\gamma}=diag(t_{\gamma}'',t_{\gamma}',(t_{\gamma}'')^*)$.)
By our selection of $W_{\tau_1}$, since $m+\bigidx-\smallidx>0$ it
vanishes unless the coordinates of $u_1$ belong to $\mathcal{P}^{o}$
and $a\in\mathcal{N}_{\GL{r},o}$, in which case $t_{\gamma}''\cdot
W_{\tau_1}$ equals a \nonzero\ constant. Since $W$ is invariant on
the right for $\mathcal{N}_{G_{\smallidx},o}$ and $k-k_0\geq o$, the
$dvda$-integration can be discarded. Ignoring volume constants
independent of $s$, the integral equals
\begin{align*}
\int_{\lmodulo{U_{H_{\bigidx}}}{H_{\bigidx}}}(\int_{R^{\smallidx,\bigidx}}
W(r'w^{\smallidx,\bigidx}h')dr')(t_{\bigidx}t_{\gamma}')^{-1}\cdot
f_s(h't_{\bigidx}t_{\gamma}',1)dh'=\Psi(W,f_s,s).
\end{align*}
The claim is established.
\end{proof} 
\end{enumerate}
The proof of the lemma is complete.
\end{proof}  

Lemma~\ref{lemma:sub multiplicativity embed some sections in gcd
second var l<=n} shows that the poles of $\gcd(\pi\times\tau,s)$
originating from holomorphic sections appear in
$\gcd(\pi\times\varepsilon,s)$. However, $\gcd(\pi\times\tau,s)$ may
also contain poles due to non-holomorphic sections. Under a certain
assumption these poles will also be included in
$\gcd(\pi\times\varepsilon,s)$.
We show,
\begin{corollary}\label{corollary:sub multiplicativity embed all sections in gcd second var l<=n}
Let 
$\tau_1$ and
$\varepsilon=\cinduced{P_{m,\bigidx}}{\GL{m+\bigidx}}{\tau_1\otimes\tau}$
be as in Lemma~\ref{lemma:sub multiplicativity embed some sections
in gcd second var l<=n}. If the operator
\begin{align}\label{op:embed all sections in gcd second var l<=n}
\nintertwiningfull{\tau_1}{s}\nintertwiningfull{\tau_1\otimes\tau^*}{(s,1-s)}
\end{align}
is holomorphic,
\begin{align*}
\gcd(\pi\times\tau,s)\in\gcd(\pi\times\varepsilon,s)\C[q^{-s},q^s].
\end{align*}
\end{corollary}
\begin{remark}
The composition in \eqref{op:embed all sections in gcd second var
l<=n} was described in Section~\ref{subsection:the multiplicativity
of the intertwining operator for tau induced}.
\end{remark}
\begin{proof}[Proof of Corollary~\ref{corollary:sub multiplicativity embed all sections in gcd second var
l<=n}] 
It is enough to show that for any $f_{1-s}\in\xi(\tau^*,hol,1-s)$
there exists $f'_s\in\xi(\varepsilon,good,s)$ such that
\begin{align*}
\Psi(W_0,f'_s,s)=\Psi(W,\nintertwiningfull{\tau^*}{1-s}f_{1-s},s),
\end{align*}
where $W_0$ is defined in Lemma~\ref{lemma:sub multiplicativity
embed some sections in gcd second var l<=n}.

Let
$\varphi_{s,1-s}\in\xi_{Q_{m,\bigidx}}^{H_{m+\bigidx}}(\tau_1\otimes\tau^*,hol,(s,1-s))$
be defined as in Lemma~\ref{lemma:sub multiplicativity embed some
sections in gcd second var l<=n}, using $f_{1-s}$. For example
assume $\smallidx\leq\bigidx$, then $\varphi_{s,1-s}$ is supported
(as a function on $H_{m+\bigidx}$) in $Q_mw'N$, right-invariant by
$N$ and for $x\in H_{\bigidx}$ and $y\in\GL{\bigidx}$,
$\varphi_{s,1-s}(w',1,x,y)=f_{1-s}^{b_{\bigidx,m}}(x,y)$. Denote
$\varphi_s'=\nintertwiningfull{\tau^*}{1-s}\varphi_{s,1-s}\in\xi_{Q_{m,\bigidx}}^{H_{m+\bigidx}}(\tau_1\otimes\tau,rat,(s,s))$.
It is (still) supported in $Q_mw'N$ and invariant on the right by
$N$, but
$\varphi_s'(w',1,x,y)=\nintertwiningfull{\tau^*}{1-s}f_{1-s}^{b_{\bigidx,m}}(x,y)$.

Our assumption on the operator \eqref{op:embed all sections in gcd
second var l<=n} implies that the image of any section in
$\xi_{Q_{m,\bigidx}}^{H_{m+\bigidx}}(\tau_1\otimes\tau^*,hol,(s,1-s))$
under \eqref{op:embed all sections in gcd second var l<=n} lies in
$\xi_{Q_{\bigidx,m}}^{H_{m+\bigidx}}(\tau^*\otimes\tau_1^*,hol,(1-s,1-s))$.
This combined with \eqref{eq:multiplication of normalized
intertwiners} and with the (similar) equality
\begin{align*}
\nintertwiningfull{\tau^*\otimes\tau_1}{(1-s,s)}\nintertwiningfull{\tau_1\otimes\tau^*}{(s,1-s)}=1
\end{align*}
shows that 
we can find
$\Phi_{1-s}\in\xi_{Q_{\bigidx,m}}^{H_{m+\bigidx}}(\tau^*\otimes\tau_1^*,hol,(1-s,1-s))$
such that
\begin{align*}
\nintertwiningfull{\tau^*\otimes\tau_1}{(1-s,s)}\nintertwiningfull{\tau_1^*}{1-s}\Phi_{1-s}=\varphi_{s,1-s}.
\end{align*}
Hence using \eqref{eq:multiplicative property of intertwiners} with
$\varepsilon^*=\cinduced{P_{\bigidx,m}}{\GL{m+\bigidx}}{\tau^*\otimes\tau_1^*}$,
\begin{align*}
\varphi_s'=\nintertwiningfull{\tau^*}{1-s}\nintertwiningfull{\tau^*\otimes\tau_1}{(1-s,s)}\nintertwiningfull{\tau_1^*}{1-s}\Phi_{1-s}=\nintertwiningfull{\varepsilon^*}{1-s}\Phi_{1-s}.
\end{align*}
Define $f_s'=\widehat{f}_{\varphi_s'}$ by \eqref{iso:iso 1}.
According to Claim~\ref{claim:intertwining operator and Jacquet
integral commute},
$f_s'=\nintertwiningfull{\varepsilon^*}{1-s}\widehat{f}_{\Phi_{1-s}}$
and by Claim~\ref{claim:f_s is the image of varphi_s when zeta is
fixed}, $\widehat{f}_{\Phi_{1-s}}\in\xi(\varepsilon^*,hol,1-s)$.
Therefore $f_s'\in\xi(\varepsilon,good,s)$ and as in the lemma
$\Psi(W_0,f'_s,s)=\Psi(W,\nintertwiningfull{\tau^*}{1-s}f_{1-s},s)$.
\end{proof} 

We will resort to Corollary~\ref{corollary:sub multiplicativity
embed all sections in gcd second var l<=n} for given representations
$\pi$ and $\tau$ with the luxury of selecting $\tau_1$. 
The following demonstrates how
to select $\tau_1$ so that the corollary would be applicable.

\begin{corollary}\label{corollary:auxiliary rep to embed all sections in gcd second var l<=n}
Let $\tau$ be an irreducible representation of $\GL{\bigidx}$. For
any $m>\max(\bigidx,\smallidx-\bigidx)$ and unitary irreducible
supercuspidal representation $\tau_1$ of $\GL{m}$ twisted by a
unitary sufficiently highly ramified character, we have
\begin{align*}
M_{\tau_1}(s)=\ell_{\tau^*\otimes\tau_1}(1-s)=\ell_{\tau_1\otimes\tau^*}(s)=1,
\end{align*}
the representation
$\varepsilon=\cinduced{P_{m,\bigidx}}{\GL{m+\bigidx}}{\tau_1\otimes\tau}$
is irreducible and
\begin{align*}
\gcd(\pi\times\tau,s)\in\gcd(\pi\times\varepsilon,s)\C[q^{-s},q^s],
\end{align*}
for any representation $\pi$ of $G_{\smallidx}$.
\end{corollary}
\begin{remark}
Indeed, the restrictions imposed by Lemma~\ref{lemma:sub multiplicativity embed some sections in gcd second var l<=n} and Corollary~\ref{corollary:sub multiplicativity
embed all sections in gcd second var l<=n} refer only to $\tau$ and $\tau_1$ and do not depend on $\pi$.
\end{remark}
\begin{proof}[Proof of Corollary~\ref{corollary:auxiliary rep to embed all sections in gcd second var l<=n}] 
Note that since $\tau$ is irreducible (generic), it is isomorphic to
a representation parabolically induced from a representation
$\eta_1\otimes\ldots\otimes\eta_a$ of
$\GL{\bigidx_1}\times\ldots\times\GL{\bigidx_a}$ where each $\eta_i$
is essentially square-integrable (\cite{Z3}, Section~9). Then since
$m>\bigidx$ and $\tau_1$ is irreducible supercuspidal, the
representation $\varepsilon$ will be irreducible according to
\textit{loc. cit.}

It is enough to show that $\tau_1$ satisfies
$M_{\tau_1}(s)=\ell_{\tau^*\otimes\tau_1}(1-s)=\ell_{\tau_1\otimes\tau^*}(s)=1$.
This is because $M_{\tau_1}(s)=\ell_{\tau_1}(s)\ell_{\tau_1^*}(1-s)$
and according to the definitions, $M_{\tau_1}(s)=1$ if and only if
$\ell_{\tau_1}(s)=\ell_{\tau_1^*}(1-s)=1$, so
$\nintertwiningfull{\tau_1}{s}$ has no poles and similarly
$\nintertwiningfull{\tau_1\otimes\tau^*}{(s,1-s)}$ has
no poles (because $\ell_{\tau_1\otimes\tau^*}(s)=1$). 
Now the result regarding $\gcd(\pi\times\tau,s)$ follows from Corollary~\ref{corollary:sub multiplicativity
embed all sections in gcd second var l<=n}.

As explained in Example~\ref{example:choosing tau so that twist has
no poles} a representation $\tau_1$ with the properties listed
satisfies $L(\tau_1^*,Sym^2,2-2s)=L(\tau_1,Sym^2,2s)=1$ and both
$\nintertwiningfull{\tau_1}{s}$, $\nintertwiningfull{\tau_1^*}{1-s}$
are holomorphic whence $M_{\tau_1}(s)=1$.

Next we show
$\ell_{\tau^*\otimes\tau_1}(1-s)=\ell_{\tau_1\otimes\tau^*}(s)=1$.
Choose
\begin{align*}
\phi^*=\cinduced{P_{\bigidx_1,\ldots,\bigidx_k}}{\GL{\bigidx}}{\phi_1^*\otimes\ldots\otimes\phi_k^*}
\end{align*}
with irreducible supercuspidal representations $\phi_i^*$
such that $\tau^*$ is a sub-representation of $\phi^*$.
Let
\begin{align}\label{op:long Levi}
\intertwiningfull{\phi_1^*\otimes\ldots\otimes\phi_k^*\otimes\tau_1}{(1-s,\ldots,1-s,s)}
\end{align}
be the standard intertwining operator
\begin{align*}
&V_{P_{\bigidx_1,\ldots,\bigidx_k,m}}^{\GL{m+\bigidx}}(\phi_1^*\absdet{}^{\frac{\bigidx}2}\otimes\ldots\otimes\phi_k^*\absdet{}^{\frac{\bigidx}2}\otimes\tau_1\absdet{}^{\frac{\bigidx}2},(1-s,\ldots,1-s,s))\\\notag&
\rightarrow
V_{P_{m,\bigidx_1,\ldots,\bigidx_k}}^{\GL{m+\bigidx}}(\tau_1\absdet{}^{\frac{\bigidx}2}\otimes\phi_1^*\absdet{}^{\frac{\bigidx}2}\otimes\ldots\otimes\phi_k^*\absdet{}^{\frac{\bigidx}2},(s,1-s,\ldots,1-s)).
\end{align*}
Here $\phi_1^*\otimes\ldots\otimes\phi_k^*\otimes\tau_1$ is
considered as an irreducible supercuspidal representation of
$A_{\bigidx_1,\ldots,\bigidx_k,m}$. If \eqref{op:long Levi} is
holomorphic, $\intertwiningfull{\tau^*\otimes\tau_1}{(1-s,s)}$ is
holomorphic (the converse is not true in general), because
$\cinduced{P_{\bigidx,m}}{\GL{m+\bigidx}}{\tau^*\otimes\tau_1}$ is a
sub-representation of
$\cinduced{P_{\bigidx_1,\ldots,\bigidx_k,m}}{\GL{m+\bigidx}}{\phi_1^*\otimes\ldots\otimes\phi_k^*\otimes\tau_1}$.
According to the
multiplicativity of the intertwining operator (\cite{Sh4}
Theorem~2.1.1), operator \eqref{op:long Levi} is a composition of
operators $\intertwiningfull{\phi_i^*\otimes\tau_1}{(1-s,s)}$. By
Theorem~\ref{theorem:poles of normalized intertwining for tau_1
tau_2 cuspidal} for each $1\leq i\leq k$,
\begin{align*}
L(\phi_i^*\otimes\tau_1^*,1-2s)^{-1}\intertwiningfull{\phi_i^*\otimes\tau_1}{(1-s,s)}
\end{align*}
is holomorphic. Since whenever $m>\bigidx$ (and $\tau_1$ is
irreducible supercuspidal) Theorem~\ref{theorem:tau cuspidal n large
L factor of glm gln 1} implies $L(\phi_i^*\times\tau_1^*,1-2s)=1$,
we get that $\intertwiningfull{\phi_i^*\otimes\tau_1}{(1-s,s)}$ is
already holomorphic and so is \eqref{op:long Levi}. Now for some
$e(q^{-s},q^s)\in\C[q^{-s},q^s]^*$,
\begin{align*}
\nintertwiningfull{\tau^*\otimes\tau_1}{(1-s,s)}=e(q^{-s},q^s)
\frac{L(\tau\times\tau_1,2s)}{L(\tau^*\times\tau_1^*,1-2s)}\intertwiningfull{\tau^*\otimes\tau_1}{(1-s,s)}.
\end{align*}
Again by Theorem~\ref{theorem:tau cuspidal n large L factor of glm
gln 1}, 
$L(\tau\times\tau_1,2s)=
1$ concluding that
$\nintertwiningfull{\tau^*\otimes\tau_1}{(1-s,s)}$ is holomorphic
and $\ell_{\tau^*\otimes\tau_1}(1-s)=1$.

The same reasoning applies to
$\nintertwiningfull{\tau_1\otimes\tau^*}{(s,1-s)}$, note that
\begin{align*}
L(\tau_1\times\phi_i,2s-1)^{-1}\intertwiningfull{\tau_1\otimes\phi_i^*}{(s,1-s)}
\end{align*}
is holomorphic and for some $d(q^{-s},q^s)\in\C[q^{-s},q^s]^*$,
\begin{align*}
\nintertwiningfull{\tau_1\otimes\tau^*}{(s,1-s)}=d(q^{-s},q^s)
\frac{L(\tau_1^*\times\tau^*,2-2s)}{L(\tau_1\times\tau,2s-1)}\intertwiningfull{\tau_1\otimes\tau^*}{(s,1-s)}.
\end{align*}
Therefore $\ell_{\tau_1\otimes\tau^*}(s)=1$.
\end{proof} 

\section{A relation between $\gcd(\pi\times\tau,s)$ and $\gcd(\pi\times\tau^*,1-s)$}\label{section:connection pi tau and pi tau star}
We mention a useful, simple observation concerning the g.c.d., which
follows from the multiplicative properties of the $\gamma$-factor -
Theorems~\ref{theorem:multiplicity second var} and
\ref{theorem:multiplicity first var}. The claim (and certain
variations of it) will be used numerously in
Chapter~\ref{chapter:upper boubnds on the gcd}.
\begin{claim}\label{claim:extra P factor in sub multiplicativity is mutual}
Let
$\tau=\cinduced{P_{\bigidx_1,\ldots,\bigidx_a}}{\GL{\bigidx}}{\tau_1\otimes\ldots\otimes\tau_a}$
be irreducible, $a\geq1$. Assume that for some $L\in\C(q^{-s})$,
$\gcd(\pi\times\tau,s)=\prod_{i=1}^a\gcd(\pi\times\tau_i,s)L$.
Then
\begin{align*}
\gcd(\pi\times\tau^*,1-s)\equalun\prod_{i=1}^a\gcd(\pi\times\tau_i^*,1-s)L.
\end{align*}
A similar assertion holds when $\pi$ is induced from a
representation $\sigma\otimes\pi'$ of $\GL{k}\times G_{\smallidx-k}$
($k\leq\smallidx$): if
$\gcd(\pi\times\tau,s)=L(\sigma\times\tau,s)\gcd(\pi'\times\tau,s)L(\sigma^*\times\tau,s)L$,
\begin{align*}
\gcd(\pi\times\tau^*,1-s)\equalun
L(\sigma\times\tau^*,1-s)\gcd(\pi'\times\tau^*,1-s)L(\sigma^*\times\tau^*,1-s)L.
\end{align*}
\end{claim}
\begin{proof}[Proof of Claim~\ref{claim:extra P factor in sub multiplicativity is mutual}] 
In general there exists $L'\in\C(q^{-s})$ 
such that
\begin{align*}
\gcd(\pi\times\tau^*,1-s)=\prod_{i=1}^a\gcd(\pi\times\tau_i^*,1-s)L'.
\end{align*}
Theorem~\ref{theorem:multiplicity second var} implies  
\begin{align*}
\prod_{i=1}^a\frac{\gcd(\pi\times\tau_i^*,1-s)}{\gcd(\pi\times\tau_i,s)}
\equalun\frac{\gcd(\pi\times\tau^*,1-s)}{\gcd(\pi\times\tau,s)}=
\prod_{i=1}^a\frac{\gcd(\pi\times\tau_i^*,1-s)}{\gcd(\pi\times\tau_i,s)}\frac{L'}L.
\end{align*}
Hence $L'\equalun L$. For $\pi$ induced from $\sigma\otimes\pi'$ one uses Theorem~\ref{theorem:multiplicity first var} and \eqref{eq:JPSS relation gamma and friends}.
\end{proof} 

In particular we shall apply Claim~\ref{claim:extra P factor in sub
multiplicativity is mutual} with
$L=M_{\tau_1\otimes\ldots\otimes\tau_a}(s)P$, for some
$P\in\C[q^{-s},q^s]$.

\section{A comment about the unramified case}\label{subsection:comment about the unramified}
Assume that $\pi$ and $\tau$ are irreducible unramified (generic)
representations. In particular they are induced from Borel
subgroups. We describe the upper and lower bounds on the g.c.d. that
we can obtain using our results and compare them to the results of
the theory of $\GL{\smallidx}\times\GL{\bigidx}$. Consider the split
case first. Then
\begin{align*}
\pi=\cinduced{P_{\smallidx}}{G_{\smallidx}}{\sigma},\qquad \sigma=\cinduced{B_{\GL{\smallidx}}}{\GL{\smallidx}}{\sigma_1\otimes\ldots\otimes\sigma_\smallidx}
,\qquad
\tau=\cinduced{B_{\GL{\bigidx}}}{\GL{\bigidx}}{\tau_1\otimes\ldots\otimes\tau_{\bigidx}}.
\end{align*}
Here $\sigma_j,\tau_i$ are unramified characters of $\GL{1}$. 
By Theorem~\ref{theorem:gcd sub multiplicity first var} there is an
upper bound
\begin{align*}
\gcd(\pi\times\tau,s)\in L(\sigma\times\tau,s)L(\sigma^*\times\tau,s)M_{\tau_1\otimes\ldots\otimes\tau_{\bigidx}}(s)\C[q^{-s},q^s]
\end{align*}
(when $\smallidx\leq\bigidx$, to apply Theorem~\ref{theorem:gcd sub
multiplicity first var} rewrite $\tau$ as a representation of
Langlands' type).
Let
$\sigma'=\cinduced{B_{\GL{\smallidx-1}}}{\GL{\smallidx-1}}{\sigma_1\otimes\ldots\otimes\sigma_{\smallidx-1}}$.
Then
$\pi=\cinduced{P_{\smallidx-1}}{G_{\smallidx}}{\sigma'\otimes\pi'}$,
$\pi'=\cinduced{B_{G_1}}{G_1}{\sigma_{\smallidx}}$ and according to
Lemma~\ref{lemma:sub multiplicativity embed some sections in first
second var k<l or k=l>n} we have a lower bound
\begin{align}\label{relation:lower bound unramified split}
L(\sigma'\times\tau,s)\in\gcd(\pi\times\tau,s)\C[q^{-s},q^s].
\end{align}

In the \quasisplit\ case
$\pi=\cinduced{P_{\smallidx-1}}{G_{\smallidx}}{\sigma'\otimes\pi'}$
with $\sigma'$ as above and the trivial representation $\pi'=1$ of
$G_1$. Theorem~\ref{theorem:gcd sub multiplicity first var} implies 
\begin{align*}
\gcd(\pi\times\tau,s)\in
L(\sigma'\times\tau,s)\gcd(\pi'\times\tau,s)L((\sigma')^*\times\tau,s)M_{\tau_1\otimes\ldots\otimes\tau_{\bigidx}}(s)\C[q^{-s},q^s].
\end{align*}
According to Proposition~\ref{proposition:iwasawa decomposition of the integral}, for $\smallidx=1$ and \quasisplit\ $G_{\smallidx}$ the integrals with holomorphic sections are holomorphic (see \eqref{int:iwasawa decomposition of the integral l<=n split}). Then Proposition~\ref{proposition:integral for good section} shows that
$\gcd(\pi'\times\tau,s)^{-1}$ divides $\ell_{\tau^*}(1-s)^{-1}$.
Similarly, $\gcd(\pi'\times\tau^*,1-s)^{-1}$ divides $\ell_{\tau}(s)^{-1}$. Hence
\begin{align*}
\gcd(\pi'\times\tau,s)=\ell_{\tau^*}(1-s)P_0,\qquad \gcd(\pi'\times\tau^*,1-s)=\ell_{\tau}(s)\widetilde{P}_0,
\end{align*}
for some $P_0,\widetilde{P}_0\in\C[q^{-s},q^s]$. Now by
Corollary~\ref{corollary:first var k<l<n exact sub mult} we obtain
\begin{align*}
\gcd(\pi\times\tau,s)\in L(\sigma'\times\tau,s)L((\sigma')^*\times\tau,s)M_{\tau}(s)\C[q^{-s},q^s].
\end{align*}
Relation~\eqref{relation:lower bound unramified split} also holds in
the \quasisplit\ case.

We compare this to the results of the theory of
$\GL{\smallidx}\times\GL{\bigidx}$ for unramified data. Recall that
the $L$-factor $L(\sigma\times\tau,s)$ was defined by Jacquet,
Piatetski-Shapiro and Shalika \cite{JPSS} as a g.c.d. (see
Section~\ref{subsection:func eq gln glm}). In the unramified case,
Theorem~3.1 of \cite{JPSS} and a short calculation of
$\GL{1}\times\GL{1}$ integrals show, that Langlands' $L$-function
for $\sigma\times\tau$ defined using the Satake parameters (see
Section~\ref{subsection:local L functions def}) is an upper bound
for the g.c.d. Additionally, Jacquet and Shalika \cite{JS1}
(Section~2) calculated the Rankin-Selberg integral for
$\sigma\times\tau$ with unramified Whittaker functions and showed
that it is equal to this $L$-function. Therefore, the g.c.d. is
equal to the $L$-function and it is obtained by the integral with
unramified data. Thus the notation $L(\sigma\times\tau,s)$ for the
g.c.d. is compatible, in the sense that in the unramified case the
g.c.d. is actually the $L$-function.

The upper bounds we obtained here imply that in the split (resp.
\quasisplit) case,
$L(\pi\times\tau,s)M_{\tau_1\otimes\ldots\otimes\tau_{\bigidx}}(s)$
(resp. $\det(1-(t'_{\pi}\otimes t_{\tau})q^{-s})^{-1}M_{\tau}(s)$)
is an upper bound for the g.c.d. (see Section~\ref{subsection:local
L functions def} for the definition of $t'_{\pi}$).

Let $W_0\in\Whittaker{\pi}{\psi_{\gamma}^{-1}}$ and
$f_s^0\in\xi(\tau,std,s)$ be the normalized unramified vectors (see Section~\ref{section:Computation of the local integral with unramified data}).
According to Theorem~\ref{theorem:unramified computation}, the
integral $\Psi(W^0,f_s^0,s)$ with all data unramified 
equals
\begin{align*}
\frac{L(\pi\times\tau,s)}{L(\tau,Sym^2,2s)}=
\begin{cases}L(\sigma\times\tau,s)L(\sigma^*\times\tau,s)L(\tau,Sym^2,2s)^{-1}&\text{split $G_{\smallidx}$,}\\
L(\sigma'\times\tau,s)L((\sigma')^*\times\tau,s)L(\tau,\Lambda^2,2s)^{-1}&\text{\quasisplit\
$G_{\smallidx}$.}
\end{cases}
\end{align*}
In
particular the unramified integral may have zeros. Thus in general
it does not represent the g.c.d. and moreover, we can describe
inducing data such that some of the poles of
$L(\sigma'\times\tau,s)$ do not belong to the set of poles of
$\Psi(W^0,f_s^0,s)$ (e.g., take 
$\sigma$ induced from $\sigma_1\otimes\sigma_2$ and $\tau$ induced
from $\sigma_1\otimes\sigma_1^2$, then some poles of
$L(\sigma'\times\tau,s)$ are canceled by zeros of
$L(\tau,Sym^2,2s)$).

In the original definition of Piatetski-Shapiro and Rallis \cite{PS}
(Section~5), the unramified case is dealt with specifically by enlarging the set of good sections (see also
\cite{HKS} Section~6). Following this here we could add
$L(\tau,Sym^2,2s)f_s^0$ 
to the set of good sections in the unramified case. While this
remedies the inconsistency in the sense that
$\Psi(W^0,L(\tau,Sym^2,2s)f_s^0,s)=L(\pi\times\tau,s)$, we are still
unable to show $\gcd(\pi\times\tau,s)=L(\pi\times\tau,s)$, i.e.,
that the intertwining operator does not introduce additional poles.


%% file: chapter_gcd_upper_bounds.tex
\newtheorem{theorem}{Theorem}[section]
\newtheorem{proposition}{Proposition}[section]
\newtheorem{corollary}[proposition]{Corollary}[section]
\newtheorem{lemma}[proposition]{Lemma}[section]
\newtheorem{claim}[proposition]{Claim}[section]
\theoremstyle{remark}
\newtheorem{remark}{Remark}[section]
\newtheorem{example}{Example}[section]
\theoremstyle{definition}
\newtheorem{definition}{Definition}[section]
\numberwithin{equation}{section}
\newcommand{\chapter}{\section} 
\input{thesis_notations}
\end{comment}
\chapter{Upper bounds on the g.c.d.}\label{chapter:upper boubnds on the gcd}
In Chapter~\ref{chapter:the gcd} we calculated the g.c.d. in a few
cases. Here we prove upper bounds that apply in the general case.
The main idea of the proofs is to convert the integral manipulations
of Chapter~\ref{section:gamma_mult} into identities of formal
Laurent series. This approach was developed by Jacquet,
Piatetski-Shapiro and Shalika \cite{JPSS} to establish the upper
bounds of the g.c.d. in their case.
\let\thefootnote\relax\footnote{This chapter is dedicated to Yael and Sophie.}

We briefly recall the construction of \cite{JPSS} (Section~4).
Consider for example a pair of representations $\xi$ and $\phi$, of
$\GL{r}$. With the notation of Section~\ref{subsection:func eq gln
glm},
\begin{align*}
\Psi(W_{\xi},W_{\phi},\Phi,s)=\int_{\lmodulo{Z_{r}}{\GL{r}}}W_{\xi}(a)W_{\phi}(a)\Phi(\eta_r
a)\absdet{a}^sda.
\end{align*}
For $m\in\Integers$, denote
$\GL{r}^m=\setof{b\in\GL{r}}{\absdet{b}=q^{-m}}$. We define a
Laurent series in $\C[[X,X^{-1}]]$,
\begin{align*}
\Sigma(W_{\xi},W_{\phi},\Phi,s)=\sum_{m\in\Integers}X^m\int_{\lmodulo{Z_{r}}{\GL{r}^m}}W_{\xi}(a)W_{\phi}(a)\Phi(\eta_r
a)da.
\end{align*}
If we replace $X$ with $q^{-s}$ this sum equals, for $\Re(s)>>0$,
the value of the integral $\Psi(W_{\xi},W_{\phi},\Phi,s)$. Assume
$\xi=\cinduced{P_{r_1,r_2}}{\GL{r}}{\xi_1\otimes\xi_2}$. The series
$\Sigma(W_{\xi},W_{\phi},\Phi,s)$ is converted into
$\Sigma(\widetilde{W_{\xi}},\widetilde{W_{\phi}},\widehat{\Phi},1-s)$
by applying integration formulas to the integrals over
$\lmodulo{Z_{r}}{\GL{r}^m}$ and utilizing the functional equations
for $\xi_1\times\phi$ and $\xi_2\times\phi$. In this manner both the
multiplicativity of $\gamma(\xi\times\phi,\psi,s)$ and an upper
bound $L(\xi\times\phi,s)\in
L(\xi_1\times\phi,s)L(\xi_2\times\phi,s)\C[q^{-s},q^s]$ are deduced
simultaneously. Note that in order to convert the integral to a
series, one needs only partition the measure space $(\GL{r},dg)$. In
particular the unipotent integration introduced when $\xi$ is a
representation of $\GL{k}$ and $k<r$, as well as the Jacquet
integral used to realize the Whittaker functional on $\xi$
(analogously to Section~\ref{subsection:Realization of tau for
induced representation}), remain intact.

Following this method here, we attach a Laurent series to an integral $\Psi(W,f_s,s)$.
Two noticeable differences between the
$\GL{k}\times\GL{\bigidx}$ integrals and the Rankin-Selberg
constructions for orthogonal groups are relevant in this context. First, the dependence on $s$ of
the space induced from $\tau$, leading to the distinction between
standard, holomorphic and rational sections (in fact in the global
setting of $\GL{\bigidx}\times\GL{\bigidx}$ a special section is
used, see \cite{JS1} Section~4 and \cite{CPS} Section~2.1). Second,
the appearance of the intertwining operators in the integrals and
functional equations. In our setting, in order to associate a series
to $\Psi(W,f_s,s)$ one must also deal with the properties of $f_s$,
not just with those of the measure space. For instance if $f_s$ is a
good section which is not holomorphic, e.g. $f_s=\nintertwiningfull{\tau^*}{1-s}f_{1-s}'$
where $f_{1-s}'\in\xi(\tau^*,hol,1-s)$ and $f_s\notin\xi(\tau,hol,s)$, it is not clear how to define
such a series. We will need to ensure that we always have
holomorphic sections. In the last example this can be done by replacing
$f_s$ with
$\ell_{\tau^*}(1-s)^{-1}f_s$.

In order to define the coefficient of $X^m$ in a series attached to $\Psi(W,f_s,s)$ we will consider, say
if $\smallidx\leq\bigidx$, the set of $(g,r)\in G_{\smallidx}\times
R_{\smallidx,\bigidx}$ such that in the Iwasawa decomposition
$w_{\smallidx,\bigidx}rg=auk$, where $a\in\GL{\bigidx}\isomorphic
M_{\bigidx}$, $u\in U_{\bigidx}$ and $k\in K_{H_{\bigidx}}$,
$\absdet{a}=q^{-m}$. Any integration related to $f_s$ (or
$\varphi_s$ such that $f_s=\widehat{f}_{\varphi_s}$) will need to be
partitioned. Furthermore, it seems difficult to apply our integral manipulations
coefficient-wise.

In Section~\ref{section:integrals and series} 
we develop a framework for converting our integrals into Laurent
series and tools which allow us to translate integral manipulations
into operations on those series. In particular, we prove a certain
version of Fubini's Theorem for series - Lemma~\ref{lemma:double
iterated series}, that will be used to apply a functional equation
to an ``inner series" representing an inner integral. For example if
$\tau=\cinduced{P_{\bigidx_1,\bigidx_2}}{\GL{\bigidx}}{\tau_1\otimes\tau_2}$,
the series representing the $G_{\smallidx}\times\GL{\bigidx}$
integral $\Psi(W,\varphi_s,s)$ will be equal to an ``iterated"
series, where the inner series represents a
$G_{\smallidx}\times\GL{\bigidx_2}$ integral. Then instead of
multiplying the integral by $\gamma(\pi\times\tau_2,\psi,s)$, we
multiply the series by the polynomial
$\gcd(\pi\times\tau_2,s)^{-1}$.

Our results also apply to the Rankin-Selberg integrals of
$SO_{2\smallidx+1}\times\GL{\bigidx}$ studied by Soudry
\cite{Soudry,Soudry2} and imply similar upper bounds. In the case
$\GL{k}\times\GL{\bigidx}$ our construction gives the same series as
\cite{JPSS}, see Example~\ref{example:GL(r) series with our defs}
below, and our tools can be used to reproduce their results.

\section{Integrals and Laurent series}\label{section:integrals and series}
\subsection{Functions and Laurent series}\label{subsection:functions and laurent series}
Let $\Sigma(X)=\C[[X,X^{-1}]]$ be the complex vector space of formal
Laurent series $\sum_{m\in\Integers}a_mX^m$, $a_m\in\C$. It is an
$R(X)=\C[X,X^{-1}]$-module with torsion. A series $\Sigma
=\sum_{m\in\Integers}a_mX^m\in\Sigma(X)$ is said to be absolutely
convergent at $s_0\in\C$ if the complex series obtained from
$\Sigma$ by replacing $X$ with $q^{-s_0}$ is absolutely convergent,
i.e., if $\sum_{m\in\Integers}|a_mq^{-s_0}|<\infty$. Then the value
of $\Sigma$ at $s_0$ is the value of the complex series
$\sum_{m\in\Integers}a_mq^{-s_0}$.

Let $f=f(s)$ be a complex-valued function defined on some domain
$D_f\subset\C$, where a domain will always refer to a subset
containing a non-empty open set. We say that $f$ has a
representation $\Sigma\in \Sigma(X)$ in $D_f$ if for all $s_0\in
D_f$, $\Sigma$ is absolutely convergent at $s_0$ and equals
$f(s_0)$. Note that such a representation is unique, i.e., if
$\Sigma'$ also represents $f$ in $D_f$, $\Sigma=\Sigma'$ in
$\Sigma(X)$. When clear from the context, we omit the domain and say
that $f$ is representable by $\Sigma$.

Let $f_1$ and $f_2$ be complex-valued functions defined on some common
domain $D\subset\C$. Assume that $\Sigma_i\in\Sigma(X)$ represents
$f_i$ in $D$ for $i=1,2$. Then $\Sigma_1+\Sigma_2$ represents
$f_1+f_2$ in $D$. Moreover if $P=P(X,X^{-1})\in R(X)$, $P\Sigma_1\in\Sigma(X)$ represents the function $s\mapsto
P(q^{-s},q^s)f_1(s)$ (in $D$).

\subsection{Representations of integrals as series}\label{subsection:Representations of integrals as series}
Let $\Gamma$ be an $l$-space, i.e. a Hausdorff, locally compact
zero-dimensional topological space (see 
\cite{BZ1}, 1.1), with a Borelian measure $dx$. For a ring $R$,
denote by $C^{\infty}(\Gamma,R)$ the set of locally constant
functions $\phi:\Gamma\rightarrow R$. This is a ring with the
pointwise operations. We usually take $R$ to be the polynomial ring
$\C[q^{-s},q^s]$, or $\C(q^{-s})$. Also recall that $\emptyset$ denotes the empty set and for
$\phi\in C^{\infty}(\Gamma,R)$, $\support{\phi}$ is the support of
$\phi$ (see Section~\ref{subsection:additional notation and symbols}).

\begin{example}\label{example:det is a holomorphic section}
The function $\alpha^s$ belongs to
$C^{\infty}(\GL{\bigidx},\C[q^{-s},q^s])$.
\end{example}

\begin{example}\label{example:f_s is a holomorphic section}
Let
$f_s\in\xi(\tau,hol,s)=\xi_{Q_{\bigidx}}^{H_{\bigidx}}(\tau,hol,s)$
where $\tau$ is realized in $\Whittaker{\tau}{\psi}$. For a fixed
$s$, $f_s\in V(\tau,s)$ and then for a fixed $h$, $b\mapsto
f_s(h,b)$ belongs to $\Whittaker{\tau}{\psi}$ (see
Section~\ref{subsection:the integrals}). Write
$f_s=\sum_{i=1}^mP_i(q^{-s},q^{s})f_s^{(i)}$ where
$P_i\in\C[q^{-s},q^s]$ and $f_s^{(i)}\in\xi(\tau,std,s)$. For $h\in H_{\bigidx}$ write $h=buk$
with $b\in \GL{\bigidx}\isomorphic M_{\bigidx}$, $u\in U_{\bigidx}$
and $k\in K_{H_{\bigidx}}$ according to the Iwasawa decomposition.
Then
$f_s^{(i)}(h,1)=\absdet{b}^{s-\half}\delta_{Q_{\bigidx}}^{\half}(b)f_s^{(i)}(k,b)$
and since $f_s^{(i)}(k,b)$ is independent of $s$,
$f_s^{(i)}(h,1)\in\C[q^{-s},q^s]$. Hence the function $h\mapsto
f_s(h,1)=\sum_{i=1}^mP_i(q^{-s},q^{s})f_s^{(i)}(h,1)$ is a locally
constant function on $H_{\bigidx}$ with values in
$\C[q^{-s},q^{s}]$. Therefore we may regard $f_s$ as an element of
$C^{\infty}(H_{\bigidx},\C[q^{-s},q^{s}])$. Similarly,
$f_s\in\xi(\tau,rat,s)$ may be viewed as an element of
$C^{\infty}(H_{\bigidx},\C(q^{-s}))$ (the function $h\mapsto
f_s(h,1)$ takes values in $\C(q^{-s})$).
\end{example}

For a non-empty finite subset of integers $M$, let
$\mathcal{V}_M=\Span{\C}{\setof{q^{-sj}}{j\in
M}}\subset\C[q^{-s},q^s]$. 
Let $\phi\in
C^{\infty}(\Gamma,\C[q^{-s},q^{s}])$. The $M$-support of $\phi$ is
the set of $x\in\Gamma$ such that $0\ne\phi(x)\in\mathcal{V}_M$ and
for all $\emptyset\ne M'\subsetneq M$,
$\phi(x)\notin\mathcal{V}_{M'}$. We
denote this set (which may be empty) by $supp_M(\phi)$. Since $\phi$
is locally constant, $supp_M(\phi)$ is an open subset and in particular measurable. Note that
$x\in supp_M(\phi)$ means that $M$ is exactly the subset of integers
$m$ such that $q^{-sm}$ appears in the polynomial $\phi(x)$ with a
\nonzero\ coefficient.

\begin{claim}\label{claim:support of phi is divided uniquely}
If $\support{\phi}\ne\emptyset$, there is a unique (necessarily countable) collection of
non-empty finite subsets of integers $\{M_i\}_{i\in I}$, such that
$supp_{M_i}(\phi)\ne\emptyset$ for all $i\in I$ and
$\support{\phi}=\dotcup_{i\in I}supp_{M_i}(\phi)$. 
\end{claim}
\begin{proof}[Proof of Claim~\ref{claim:support of phi is divided
uniquely}] 
For any $x\in\support{\phi}$, $\phi(x)$ is a \nonzero\ polynomial in
$q^{-s},q^s$ whence there is a unique finite set $\emptyset\ne
M\subset\Integers$ containing precisely the integers $m$ such that
$q^{-sm}$ appears in $\phi(x)$ with a \nonzero\ coefficient. Hence
$\phi(x)\in \mathcal{V}_M$ and for all $\emptyset\ne M'\subsetneq
M$, $\phi(x)\notin\mathcal{V}_{M'}$. Thus $x\in supp_{M}(\phi)$. It
follows that $\support{\phi}=\bigcup_{i\in I}supp_{M_i}(\phi)$ for
some collection of sets $\{M_i\}_{i\in I}$.

If $x\in supp_{M}(\phi)\bigcap supp_{N}(\phi)$ for another non-empty
finite subset $N$, then
$\phi(x)\in\mathcal{V}_M\bigcap\mathcal{V}_N$. 
Since $\phi(x)\ne0$ we obtain $M\bigcap N\ne\emptyset$. If $M\bigcap
N=M$, we get $\emptyset\ne M\subsetneq N$ and
$\phi(x)\in\mathcal{V}_M$, contradicting the fact that $x\in
supp_{N}(\phi)$. But now $\emptyset\ne M\cap N\subsetneq M$
satisfies $\phi(x)\in \mathcal{V}_{M\cap N}$ contradicting
$x\in supp_{M}(\phi)$. This shows $\support{\phi}=\dotcup_{i\in
I}supp_{M_i}(\phi)$ and that the collection $\{M_i\}_{i\in I}$ is unique. 
\end{proof} 
We say that $\Gamma$ can be divided into the simple supports of
$\phi$ if $\support{\phi}=\emptyset$ or if the collection
$\{M_i\}_{i\in I}$ of Claim~\ref{claim:support of phi is divided
uniquely} satisfies $M_i=\{m_i\}$ for some $m_i\in\Integers$ for all
$i\in I$.
\begin{example}\label{example:alpha s and simple supports}
The $l$-space $\GL{\bigidx}$ can be divided into the simple supports
of $\alpha^s$. In fact, for any $m\in\Integers$,
\begin{align*}
supp_{\{m\}}(\alpha^s)=\setof{b\in
\GL{\bigidx}}{\absdet{b}=q^{-m}}.
\end{align*}

\end{example}
\begin{example}\label{example:phi which depends mostly on s}
Let $P=\sum_{j\in \Integers}a_jq^{-sj}\in \C[q^{-s},q^s]$ (i.e.,
$a_j=0$ for almost all $j$). The function $\phi(x)=P$ defined on
$\Gamma$ trivially belongs to $C^{\infty}(\Gamma,\C[q^{-s},q^s])$.
Assuming $P\ne0$, $\support{\phi}=\Gamma=supp_M(\phi)$ where
$M=\setof{j\in\Integers}{a_j\ne0}$. This example can be extended by
replacing the coefficients $a_j$ with functions in
$C^{\infty}(\Gamma,\C)$.
\end{example}
\begin{example}\label{example:division of standard section supports}
Let $f_s\in\xi(\tau,std,s)$ be regarded as a function in
$C^{\infty}(H_{\bigidx},\C[q^{-s},q^s])$ as explained in
Example~\ref{example:f_s is a holomorphic section}, i.e., the actual
function is $h\mapsto f_s(h,1)$. For $h\in H_{\bigidx}$,
$f_s(h,1)=\absdet{b}^{s-\half}\delta_{Q_{\bigidx}}^{\half}(b)f_s(k,b)$
where $h=buk$ is written according to Example~\ref{example:f_s is a
holomorphic section}. Since $f_s(k,b)$ is independent of $s$,
\begin{align*}
supp_{\{m\}}(f_s)=\setof{h\in
H_{\bigidx}}{h=buk,\absdet{b}=q^{-m},f_s(k,b)\ne0}.
\end{align*}
This shows that $H_{\bigidx}$ can be divided into the simple
supports of $f_s$.
\end{example}
For $m\in\Integers$ denote by $P_m:\C[q^{-s},q^s]\rightarrow\C$ the mapping
$P_m(\sum_{j\in\Integers}a_jq^{-sj})=a_m$. The function $x\mapsto
P_m(\phi(x))$ is locally constant (hence measurable).
\begin{definition}\label{def:series for bundle}
Let $\phi\in C^{\infty}(\Gamma,\C[q^{-s},q^s])$. Assume $\int_{\Gamma}|P_m(\phi(x))|dx<\infty$, for all $m\in\Integers$. Then we
define a Laurent series in $\Sigma(X)$,
\begin{align*}
\oint_{\Gamma}\phi(x)dx=\sum_{m\in\Integers}X^m\int_{\Gamma}P_m(\phi(x))dx.
\end{align*}
\end{definition}
We define a series as above to be strongly convergent at $s$ if
\begin{align*}
\sum_{m\in\Integers}q^{-\Re(s)m}\int_{\Gamma}|P_m(\phi(x))|dx<\infty,
\end{align*}
a condition stronger than being absolutely convergent at $s$, 
which is just that
\begin{align*}
\sum_{m\in\Integers}q^{-\Re(s)m}\Big|\int_{\Gamma}P_m(\phi(x))dx\Big|<\infty.
\end{align*}
\begin{example}\label{example:phi which does not depend on s}
Let $a\in C^{\infty}(\Gamma,\C)\subset
C^{\infty}(\Gamma,\C[q^{-s},q^s])$. Then
$\support{a}=supp_{\{0\}}(a)$ and provided
$\int_{\Gamma}|a(x)|dx<\infty$,
$\oint_{\Gamma}a(x)dx=\int_{\Gamma}a(x)dx$ is a constant term as an
element of $\Sigma(X)$.
\end{example}
Next we define an integral for $\phi\in
C^{\infty}(\Gamma,\C[q^{-s},q^s])$. Any fixed $s_0\in\C$ induces a
homomorphism $\C[q^{-s},q^s]\rightarrow\C$ by evaluation. Denote
by $[\phi(x)](s_0)$ the value of $\phi(x)$ under this homomorphism.
For example, if $x\in supp_M(\phi)$, $\phi(x)=\sum_{j\in
M}a_j(x)q^{-sj}$ where $0\ne a_j(x)\in\C$ and
$[\phi(x)](s_0)=\sum_{j\in M}a_j(x)q^{-s_0j}\in\C$.

For any fixed $s\in\C$, the integral
$\Phi(s)=\int_{\Gamma}[\phi(x)](s)dx$ is absolutely convergent if
$\int_{\Gamma}|[\phi(x)](s)|dx<\infty$. If there is a domain
$D\subset\C$ such that for all $s\in D$, $\Phi(s)$ is absolutely
convergent, $s\mapsto\Phi(s)$ is a complex-valued function on $D$.
The following results show how to use the series defined above to
represent the complex function $\Phi$ in the sense of
Section~\ref{subsection:functions and laurent series}.

We introduce the following notation that will be used repeatedly
below. For $\phi\in C^{\infty}(\Gamma,\C[q^{-s},q^s])$, denote
$\Phi_{\phi}(s)=\int_{\Gamma}[\phi(x)](s)dx$ and
$\Sigma_{\phi}=\oint_{\Gamma}\phi(x)dx$ (assuming these are
defined). Let $D\subset\C$ be a domain. We write
$\Sigma_{\phi}\sim_D\Phi_{\phi}$ if for all $s\in D$,
$\Phi_{\phi}(s)$ is absolutely convergent, $\Sigma_{\phi}$ is
strongly convergent at $s$ and $\Sigma_{\phi}$ represents
$\Phi_{\phi}$ in $D$.

In order to deduce that $\Sigma_{\phi}$ represents
$\Phi_{\phi}$ in $D$, given that $\Phi_{\phi}(s)$ is absolutely convergent and $\Sigma_{\phi}$ is
strongly convergent, we need only show that substituting $q^{-s}$ for $X$ in $\Sigma_{\phi}$ gives $\Phi_{\phi}(s)$ for $s\in D$.
\begin{lemma}\label{lemma:integral representation for integral with m supports}
Let $\phi\in C^{\infty}(\Gamma,\C[q^{-s},q^s])$ be such that
$\Gamma$ can be divided into its simple supports. Assume that
$\Phi_{\phi}(s)$ is absolutely convergent in a domain $D\subset\C$.
Then $\Sigma_{\phi}$ is defined and
$\Sigma_{\phi}\sim_D\Phi_{\phi}$.
\end{lemma}
\begin{proof}[Proof of Lemma~\ref{lemma:integral representation for integral with m supports}] 
If $x\in supp_{\{m\}}(\phi)$, write $\phi(x)=q^{-sm}a_m(x)$ with
$0\ne a_m(x)\in\C$. In $D$,
\begin{align*}
\infty>\int_{\Gamma}|[\phi(x)](s)|dx&=\sum_{m\in\Integers}\int_{supp_{\{m\}}(\phi)}|[\phi(x)](s)|dx\\\notag&=
\sum_{m\in\Integers}\int_{supp_{\{m\}}(\phi)}q^{-\Re(s)m}|a_m(x)|dx\\\notag
&=\sum_{m\in\Integers}q^{-\Re(s)m}\int_{supp_{\{m\}}(\phi)}|P_m(\phi(x))|dx
\\\notag&=\sum_{m\in\Integers}q^{-\Re(s)m}\int_{\Gamma}|P_m(\phi(x))|dx.
\end{align*}
Note that whenever $supp_{\{m\}}(\phi)\ne\emptyset$, $a_m(x)$ is
defined. Hence
$\int_{\Gamma}|P_m(\phi(x))|dx<\infty$ for all $m\in\Integers$. Therefore $\Sigma_{\phi}$ is defined,
\begin{align*}
\Sigma_{\phi}=\sum_{m\in\Integers}X^m\int_{\Gamma}P_m(\phi(x))dx
\end{align*}
and it is strongly convergent for all $s\in D$.

Additionally for $s\in D$,
\begin{align*}
\Phi_{\phi}(s)&=\sum_{m\in\Integers}\int_{supp_{\{m\}}(\phi)}[\phi(x)](s)dx=
\sum_{m\in\Integers}\int_{supp_{\{m\}}(\phi)}q^{-sm}a_m(x)dx\\\notag
&=\sum_{m\in\Integers}q^{-sm}\int_{supp_{\{m\}}(\phi)}P_m(\phi(x))dx
=\sum_{m\in\Integers}q^{-sm}\int_{\Gamma}P_m(\phi(x))dx.
\end{align*}
Evidently, replacing $X$ with $q^{-s}$ in $\Sigma_{\phi}$ we obtain
$\Phi_{\phi}(s)$, showing that $\Sigma_{\phi}$ represents
$\Phi_{\phi}$ in $D$.
\end{proof} 
\begin{example}\label{example:GL(r) series with our defs}
Consider the $\GL{r}\times\GL{r}$ integral
$\Psi(W_{\xi},W_{\phi},\Phi,s)$ of Jacquet, Piatetski-Shapiro and
Shalika \cite{JPSS} discussed in the beginning of this chapter. Let
$\Gamma=\lmodulo{Z_{r}}{\GL{r}}$ and define $\phi\in
C^{\infty}(\Gamma,\C[q^{-s},q^s])$ by
$\phi(x)=W_{\xi}(x)W_{\phi}(x)\Phi(\eta_rx)\absdet{x}^s$. Then
$\Phi_{\phi}(s)=\Psi(W_{\xi},W_{\phi},\Phi,s)$ is defined in a
suitable right half-plane $D$. Note that
\begin{align*}
supp_{\{m\}}(\phi)=\setof{x\in\Gamma}{\absdet{x}=q^{-m},W_{\xi}(x)W_{\phi}(x)\Phi(\eta_rx)\ne0}.
\end{align*}
Consequently, $\Gamma$ can be divided into the simple supports of $\phi$.
The absolute convergence of $\Phi_{\phi}(s)$ in $D$ implies that for
each $m$,
\begin{align*}
\int_{\Gamma}|P_m(\phi)|dx=\int_{\lmodulo{Z_{r}}{\GL{r}^m}}|W_{\xi}(x)W_{\phi}(x)\Phi(\eta_rx)|dx<\infty.
\end{align*}
Thus $\Sigma_{\phi}$ is defined,
\begin{align*}
\Sigma_{\phi}=\sum_{m\in\Integers}X^m\int_{\lmodulo{Z_{r}}{\GL{r}^m}}W_{\xi}(x)W_{\phi}(x)\Phi(\eta_rx)dx=\Sigma(W_{\xi},W_{\phi},\Phi,s)
\end{align*}
and according to Lemma~\ref{lemma:integral representation for
integral with m supports}, $\Sigma_{\phi}\sim_D\Phi_{\phi}$.
\end{example}
More generally,
\begin{lemma}\label{lemma:integral representation for general integral}
Let $\phi\in C^{\infty}(\Gamma,\C[q^{-s},q^s])$ be such that
$\Sigma_{\phi}$ is defined. Assume that $\Phi_{\phi}(s)$ is
absolutely convergent and $\Sigma_{\phi}$ is strongly convergent in
a domain $D\subset\C$. Then $\Sigma_{\phi}\sim_D\Phi_{\phi}$.
\end{lemma}
\begin{proof}[Proof of Lemma~\ref{lemma:integral representation for general integral}] 
According to the assumptions, it is left to show that
$\Sigma_{\phi}$ represents $\Phi_{\phi}$ in $D$. For $x\in
supp_M(\phi)$ write $\phi(x)=\sum_{j\in M}a_{j}(x)q^{-sj}$ where
$0\ne a_{j}(x)\in\C$. Let $\mathcal{Z}$ be the (countable) set of
finite non-empty subsets of $\Integers$. Observe that for a fixed
$m\in\Integers$,
\begin{align}\label{eq:absolute gamma}
\int_{\Gamma}|P_m(\phi(x))|dx&=\sum_{\setof{M\in\mathcal{Z}}{m\in
M}}\int_{supp_M(\phi)}|P_m(\phi(x))|dx\\\notag&=\sum_{\setof{M\in\mathcal{Z}}{m\in
M}}\int_{supp_M(\phi)}|a_{m}(x)|dx.
\end{align}
(The summation is over all subsets $M\in\mathcal{Z}$ containing
$m$.) Note that $supp_M(\phi)$ may be empty, in which case the
$dx$-integration over $supp_M(\phi)$ vanishes. Also for any two
distinct $M,N\in\mathcal{Z}$, as we proved in
Claim~\ref{claim:support of phi is divided uniquely} the sets
$supp_M(\phi)$, $supp_N(\phi)$ are disjoint.

Let $s\in D$. Since $\Sigma_{\phi}$ is strongly convergent in $D$,
using \eqref{eq:absolute gamma} yields
\begin{align}\label{infty sum:integrals representing series}
\sum_{m\in\Integers}q^{-\Re(s)m}\sum_{\setof{M\in\mathcal{Z}}{m\in
M}}\int_{supp_M(\phi)}|a_{m}(x)|dx<\infty.
\end{align}
We also have
\begin{align}\label{infty sum:phi s}
\Phi_{\phi}(s)=\sum_{M\in\mathcal{Z}}\int_{supp_M(\phi)}[\phi(x)](s)dx
=\sum_{M\in\mathcal{Z}}\int_{supp_M(\phi)}\sum_{j\in
M}a_{j}(x)q^{-sj}dx.
\end{align}
For fixed $M$ and $j\in M$, $\int_{supp_M(\phi)}|a_{j}(x)|dx<\infty$
because it is majorized by $q^{\Re(s)j}$ times \eqref{infty
sum:integrals representing
series}. Hence
\begin{align*}
\int_{supp_M(\phi)}\sum_{j\in
M}|a_{j}(x)q^{-sj}|dx<\infty
\end{align*}
and we may change the order of summation and integration in \eqref{infty
sum:phi s} and obtain
\begin{align*}
\sum_{M\in\mathcal{Z}}\sum_{j\in
M}q^{-sj}\int_{supp_M(\phi)}a_{j}(x)dx.
\end{align*}
Again using \eqref{infty sum:integrals representing series} 
we
change the order of summation,
\begin{align*}
\Phi_{\phi}(s)&=\sum_{M\in\mathcal{Z}}\sum_{j\in
M}q^{-sj}\int_{supp_M(\phi)}a_{j}(x)dx\\\notag&=
\sum_{m\in\Integers}q^{-sm}\sum_{\setof{M\in\mathcal{Z}}{m\in
M}}\int_{supp_M(\phi)}a_{m}(x)dx\\\notag
&=\sum_{m\in\Integers}q^{-sm}\sum_{\setof{M\in\mathcal{Z}}{m\in
M}}\int_{supp_M(\phi)}P_m(\phi(x))dx\\\notag&=\sum_{m\in\Integers}q^{-sm}\int_{\Gamma}P_m(\phi(x))dx.
\end{align*}
This shows that $\Sigma_{\phi}$ represents $\Phi_{\phi}$ in $D$.
\end{proof} 
Next we consider the integral of a sum of holomorphic sections.
\begin{lemma}\label{lemma:sum of series represents sum of integrals}
Let $\phi_1,\phi_2\in C^{\infty}(\Gamma,\C[q^{-s},q^s])$ and put
$\phi=\phi_1+\phi_2$. If $\Sigma_{\phi_i}$ is defined for $i=1,2$, then
$\Sigma_{\phi}$ is defined. Moreover if
$\Sigma_{\phi_i}\sim_D\Phi_{\phi_i}$ for both $i$, then
$\Sigma_{\phi}\sim_D\Phi_{\phi}$ and
$\Sigma_{\phi}=\Sigma_{\phi_1}+\Sigma_{\phi_2}$. 
\end{lemma}
\begin{proof}[Proof of Lemma~\ref{lemma:sum of series represents sum of integrals}] 
The series $\Sigma_{\phi}$ is defined because
$\Sigma_{\phi_1}$ and $\Sigma_{\phi_2}$ are defined and
$P_m(\phi(x))=P_m(\phi_1(x))+P_m(\phi_2(x))$. Now consider $s\in D$.
Since $\Phi_{\phi_1}(s)$ and $\Phi_{\phi_2}(s)$ are absolutely
convergent, so is $\Phi_{\phi}(s)$ and
$\Phi_{\phi}(s)=\Phi_{\phi_1}(s)+\Phi_{\phi_2}(s)$. Also
$\Sigma_{\phi}$ is strongly convergent, because $\Sigma_{\phi_1}$ and
$\Sigma_{\phi_2}$ are. By Lemma~\ref{lemma:integral representation
for general
integral}, $\Sigma_{\phi}\sim_D\Phi_{\phi}$. 
Since the series $\Sigma_{\phi_1}+\Sigma_{\phi_2}$ also represents $\Phi_{\phi}$ in $D$, $\Sigma_{\phi}=\Sigma_{\phi_1}+\Sigma_{\phi_2}$.
\end{proof} 
As in Example~\ref{example:phi which depends mostly on s} let
$P
\in \C[q^{-s},q^s]$ be considered as an element of
$C^{\infty}(\Gamma,\C[q^{-s},q^s])$, then $P\cdot\phi\in
C^{\infty}(\Gamma,\C[q^{-s},q^s])$. This defines a structure of a
$\C[q^{-s},q^s]$-module on
$C^{\infty}(\Gamma,\C[q^{-s},q^s])$. The next lemma 
shows 
that the $\oint$ operation commutes with multiplication by
a polynomial.
\begin{lemma}\label{lemma:interchange integration and polynomial}
Let $\phi\in C^{\infty}(\Gamma,\C[q^{-s},q^s])$ be such that
$\Sigma_{\phi}$ is defined and assume $\Sigma_{\phi}\sim_D\Phi_{\phi}$. Then for any $P\in
\C[q^{-s},q^s]$, $\Sigma_{P\phi}$ is defined,
$\Sigma_{P\phi}\sim_D\Phi_{P\phi}$ and
$P(X,X^{-1})\Sigma_{\phi}=\Sigma_{P\phi}$, i.e.,
\begin{align*}
P(X,X^{-1})\oint_{\Gamma}\phi(x)dx=\oint_{\Gamma}(P\phi)(x)dx.
\end{align*}
\end{lemma}
\begin{proof}[Proof of Lemma~\ref{lemma:interchange integration and polynomial}] 
Write $P=\sum_{j\in M}a_jq^{-sj}$, 
$\phi_j=a_jq^{-sj}\cdot\phi$. 
It follows from the definitions that
$\Sigma_{\phi_j}=\oint_{\Gamma}\phi_j(x)dx$ is defined and in $D$,
$\Sigma_{\phi_j}$ is strongly convergent and
$\Phi_{\phi_j}(s)=\int_{\Gamma}[\phi_j(x)](s)dx$ is absolutely
convergent. By Lemma~\ref{lemma:integral representation for general
integral}, $\Sigma_{\phi_j}\sim_D\Phi_{\phi_j}$. Since
$P\phi=\sum_{j\in M}\phi_j$, Lemma~\ref{lemma:sum of
series represents sum of integrals} shows that 
$\Sigma_{P\phi}=\oint_{\Gamma}(P\phi)(x)dx$ is defined and
$\Sigma_{P\phi}\sim_D\Phi_{P\phi}$. For $s\in D$,
$\Phi_{P\phi}(s)=P(q^{-s},q^s)\Phi_{\phi}(s)$ whence
$P(X,X^{-1})\Sigma_{\phi}$ also represents $\Phi_{P\phi}$ in $D$, proving
$\Sigma_{P\phi}=P(X,X^{-1})\Sigma_{\phi}$.
\end{proof} 
Let $\Gamma\times\Gamma'$ be a product of $l$-spaces (this is also
an $l$-space) and let $\phi\in
C^{\infty}(\Gamma\times\Gamma',\C[q^{-s},q^s])$. For any
$x\in\Gamma$ define a function $\phi(x,\cdot)\in
C^{\infty}(\Gamma',\C[q^{-s},q^s])$
by $x'\mapsto \phi(x,x')$. 
We say that $\phi$ is smooth in $\Gamma$ if 
the mapping $x\mapsto\phi(x,\cdot)$ belongs to \\
$C^{\infty}(\Gamma,C^{\infty}(\Gamma',\C[q^{-s},q^s]))$. Put
$\Phi_{\phi}(s)=\int_{\Gamma\times\Gamma'}[\phi(x,x')](s)d(x,x')$,
$\Sigma_{\phi}=\oint_{\Gamma\times\Gamma'}\phi(x,x')d(x,x')$. We
prove a certain analogue of Fubini's Theorem for series.
\begin{lemma}\label{lemma:double iterated series}
Let $\phi\in C^{\infty}(\Gamma\times\Gamma',\C[q^{-s},q^s])$ be
smooth in $\Gamma$, for which $\Sigma_{\phi}$ is defined. Assume
$\Sigma_{\phi}\sim_D\Phi_{\phi}$. Then for all $x\in\Gamma$, the
series $\Sigma_{\phi(x,\cdot)}=\oint_{\Gamma'}\phi(x,x')dx'$ is
defined and $\Sigma_{\phi(x,\cdot)}\sim_D\Phi_{\phi(x,\cdot)}$,
where $\Phi_{\phi(x,\cdot)}(s)=\int_{\Gamma'}[\phi(x,x')](s)dx'$.
Further assume that for all $x\in\Gamma$, $\Sigma_{\phi(x,\cdot)}\in
R(X)$. Then the function
$\phi_{\Gamma'}(x)=\Sigma_{\phi(x,\cdot)}(q^{-s},q^s)$ belongs to
$C^{\infty}(\Gamma,\C[q^{-s},q^s])$, 
$\Sigma_{\phi_{\Gamma'}}=\oint_{\Gamma}\phi_{\Gamma'}(x)dx$ is
defined, $\Sigma_{\phi_{\Gamma'}}\sim_D\Phi_{\phi}$ and
$\Sigma_{\phi}=\Sigma_{\phi_{\Gamma'}}$.
\end{lemma}

\begin{proof}[Proof of Lemma~\ref{lemma:double iterated series}] 
Let $x\in\Gamma$. Since $\phi$ is smooth in $\Gamma$, there is a
compact open neighborhood $N_x\subset\Gamma$ of $x$, such that
$\phi(x,\cdot)=\phi(n_x,\cdot)$ for all $n_x\in N_x$. Put
$c_x=vol(N_x)^{-1}$. Because $\Sigma_{\phi}$ is defined, for all $m$
we have
\begin{align*}
\int_{\Gamma'}|P_m(\phi(x,x'))|dx'\leq c_x
\int_{\Gamma\times\Gamma'}|P_m(\phi(x,x'))|d(x,x')<\infty.
\end{align*}
Therefore $\Sigma_{\phi(x,\cdot)}$ is defined. The coefficient of
$X^m$ in $\Sigma_{\phi(x,\cdot)}$ is
$\int_{\Gamma'}P_m(\phi(x,x'))dx'$.
In addition since $\Sigma_{\phi}$ is strongly convergent at $s\in
D$,
\begin{align*}
\sum_{m\in\Integers}q^{-\Re(s)m}\int_{\Gamma'}|P_m(\phi(x,x'))|dx\leq
c_x\sum_{m\in\Integers}q^{-\Re(s)m}\int_{\Gamma\times\Gamma'}|P_m(\phi(x,x'))|d(x,x')<\infty.
\end{align*}
Hence $\Sigma_{\phi(x,\cdot)}$ is strongly convergent. Also
$\Phi_{\phi(x,\cdot)}(s)$ is absolutely convergent in $D$. According
to Lemma~\ref{lemma:integral representation for general integral},
$\Sigma_{\phi(x,\cdot)}\sim_D\Phi_{\phi(x,\cdot)}$.

Now assume $\Sigma_{\phi(x,\cdot)}\in R(X)$ for all $x$. 
Then $\phi_{\Gamma'}\in C^{\infty}(\Gamma,\C[q^{-s},q^s])$ because $\phi$ is smooth in $\Gamma$. By Tonelli's
Theorem and using the fact that $\Sigma_{\phi}$ is defined,
\begin{align*}
\int_{\Gamma}|P_m(\phi_{\Gamma'}(x))|dx&=\int_{\Gamma}|\int_{\Gamma'}P_m(\phi(x,x'))dx'|dx\leq\int_{\Gamma\times\Gamma'}|P_m(\phi(x,x'))|d(x,x')<\infty.
\end{align*}
It follows that $\Sigma_{\phi_{\Gamma'}}$
is defined. 
Since $\Sigma_{\phi}$ is strongly convergent for $s\in D$, so is
$\Sigma_{\phi_{\Gamma'}}$. In fact,
\begin{align*}
\sum_{m\in\Integers}q^{-\Re(s)m}\int_{\Gamma}|P_m(\phi_{\Gamma'}(x))|dx\leq
\sum_{m\in\Integers}q^{-\Re(s)m}\int_{\Gamma\times\Gamma'}|P_m(\phi(x,x'))|d(x,x')<\infty.
\end{align*}
Also $\Phi_{\phi_{\Gamma'}}(s)=\int_{\Gamma}[\phi_{\Gamma'}(x)](s)dx$
is absolutely convergent in $D$ because
\begin{align*}
\int_{\Gamma}|[\phi_{\Gamma'}(x)](s)|dx\leq
\int_{\Gamma}\int_{\Gamma'}|[\phi(x,x')](s)|dx'dx=\int_{\Gamma\times\Gamma'}|[\phi(x,x')](s)|d(x,x')<\infty.
\end{align*}

Appealing to Lemma~\ref{lemma:integral representation for general
integral}, $\Sigma_{\phi_{\Gamma'}}\sim_D\Phi_{\phi_{\Gamma'}}$ and
since for $s\in D$,\\
$\Phi_{\phi_{\Gamma'}}(s)=\int_{\Gamma}\int_{\Gamma'}[\phi(x,x')](s)dx'dx$,
by Fubini's
Theorem $\Phi_{\phi_{\Gamma'}}(s)=\Phi_{\phi}(s)$ for all $s\in D$. 
Hence $\Sigma_{\phi_{\Gamma'}}\sim_D\Phi_{\phi}$ and
$\Sigma_{\phi_{\Gamma'}}=\Sigma_{\phi}$.
\end{proof} 
\subsection{Laurent representation for $\Psi(W,f_s,s)$}\label{subsection:Laurent representation for the integral}
We will use the series described in
Section~\ref{subsection:Representations of integrals as series} to
represent the integral $\Psi(W,f_s,s)$. Recall that there is a right half-plane $D\subset\C$ of absolute convergence for $\Psi(W,f_s,s)$,
$f_s\in\xi(\tau,hol,s)$, which depends only on the representations
$\pi$ and $\tau$ (see Section~\ref{subsection:the integrals}). Since
for $f_s\in\xi(\tau,good,s)$ we have
$\ell_{\tau^*}(1-s)^{-1}f_s\in\xi(\tau,hol,s)$, we can assume that
$\Psi(W,f_s,s)$ is absolutely convergent in $D$ for all good
sections, and $D$ still depends only on the representations ($D$ is some other right half-plane).
\begin{lemma}\label{lemma:integral for holomorphic section representable as a series}
Let $W\in\Whittaker{\pi}{\psi_{\gamma}^{-1}}$ and $f_s\in\xi(\tau,hol,s)$. Consider $\Psi(W,f_s,s)$ as a function
of $s$, defined in $D$. It is representable in $\Sigma(X)$ by a
strongly convergent series $\Sigma(W,f_s,s)$ with finitely many
negative coefficients.
\end{lemma}
\begin{proof}[Proof of Lemma~\ref{lemma:integral for holomorphic section representable as a series}] 
We prove the case $\smallidx\leq\bigidx$, the other case being
similar. Since $U_{G_{\smallidx}}\overline{B_{G_{\smallidx}}}$ is a
dense open subset of $G_{\smallidx}$ such that its complement is a
subset of zero measure, we can replace the $dg$-integration in
$\Psi(W,f_s,s)$ with an integration over
$\overline{B_{G_{\smallidx}}}$. Then the integral takes the form
\begin{align*}
\Psi(W,f_s,s)=\int_{\overline{B_{G_{\smallidx}}}}
\int_{R_{\smallidx,\bigidx}}W(g)f_s(w_{\smallidx,\bigidx}rg,1)\psi_{\gamma}(r)drdg.
\end{align*}
Here $dg$ is actually a right-invariant Haar measure on
$\overline{B_{G_{\smallidx}}}$. Let
$\Gamma=\overline{B_{G_{\smallidx}}}\times R_{\smallidx,\bigidx}$.
For $f_s\in\xi(\tau,hol,s)$, set
$\phi_{f_s}(g,r)=W(g)f_s(w_{\smallidx,\bigidx}rg,1)\psi_{\gamma}(r)\in\C[q^{-s},q^s]$
($g\in\overline{B_{G_{\smallidx}}}$, $r\in R_{\smallidx,\bigidx}$). Evidently,
$\phi_{f_s}\in C^{\infty}(\Gamma,\C[q^{-s},q^s])$. Then $\Phi_{\phi_{f_s}}(s)=\Psi(W,f_s,s)$ is absolutely convergent
in $D$.

First assume $f_s\in\xi(\tau,std,s)$. For $(g,r)\in\Gamma$
write $w_{\smallidx,\bigidx}rg=buk$ with $b\in
GL_{\bigidx}\isomorphic M_{\bigidx}$, $u\in U_{\bigidx}$ and $k\in
K_{H_{\bigidx}}$ according to the Iwasawa decomposition. Then
\begin{align*}
\phi_{f_s}(g,r)=W(g)\absdet{b}^{s-\half}\delta_{Q_{\bigidx}}^{\half}(b)f_s(k,b)\psi_{\gamma}(r)
\end{align*}
and since the only factor depending on $s$ is $\absdet{b}^{s}$
(because $f_s(k,b)$ is independent of $s$),
\begin{align*}
supp_{\{m\}}(\phi_{f_s})=\setof{(g,r)\in\Gamma}{w_{\smallidx,\bigidx}rg=buk,\absdet{b}=q^{-m},W(g)f_s(k,b)\ne0}.
\end{align*}
Thus $\Gamma$ can be divided into the simple supports of
$\phi_{f_s}$ (compare to Example~\ref{example:division of
standard section supports}). 
Applying Lemma~\ref{lemma:integral representation for integral with
m supports}, $\Sigma_{\phi_{f_s}}\in\Sigma(X)$ is defined and $\Sigma_{\phi_{f_s}}\sim_D\Phi_{\phi_{f_s}}$. Regarding
$f_s\in\xi(\tau,hol,s)$, write $f_s=\sum_{i=1}^mP_if_s^{(i)}$ for
$P_i\in \C[q^{-s},q^s]$, $f_s^{(i)}\in\xi(\tau,std,s)$. Then $\Sigma_{\phi_{f_s^{(i)}}}$ is defined and
\begin{align*}
\Sigma_{\phi_{f_s^{(i)}}}\sim_D\Phi_{\phi_{f_s^{(i)}}}
\end{align*}
for all $i$ ($\Phi_{\phi_{f_s^{(i)}}}(s)=\Psi(W,f_s^{(i)},s)$ is absolutely convergent in $D$ because $D$ depends only on the representations).
By Lemma~\ref{lemma:interchange integration and polynomial}, $\Sigma_{P_i\phi_{f_s^{(i)}}}$ is defined and
\begin{align*}
\Sigma_{P_i\phi_{f_s^{(i)}}}\sim_D\Phi_{P_i\phi_{f_s^{(i)}}}.
\end{align*}
Since $\sum_{i=1}^mP_i\phi_{f_s^{(i)}}=\phi_{f_s}$,
Lemma~\ref{lemma:sum of series represents sum of integrals} implies that the series
$\Sigma_{\phi_{f_s}}$ is defined and $\Sigma_{\phi_{f_s}}\sim_D\Phi_{\phi_{f_s}}$. We set
$\Sigma(W,f_s,s)=\Sigma_{\phi_{f_s}}$.

Regarding the negative coefficients, decompose $\Psi(W,f_s,s)$ as in
Proposition~\ref{proposition:iwasawa decomposition of the integral}.
We argue that each integral $I_s^{(j)}$ in this decomposition is
representable as a series which has a finite number of negative
coefficients, i.e., a series $\sum_{m=M}^{\infty}a_mX^m$. Looking at \eqref{int:iwasawa decomposition of the integral l<=n
split}, in the split case we write the integral $I_s^{(j)}$ in the form
\begin{align}\label{int:int to series from proposition iwasawa}
\int_{A_{\smallidx-1}}\int_{G_1}(\absdet{a}\cdot[x]^{-1})^s
\theta(x,a)dxda,
\end{align}
where
\begin{align*}
\theta(x,a)=ch_{\Lambda}(x)W^{\diamond}(ax)W'(
diag(a,\lfloor x\rfloor,I_{\bigidx-\smallidx}))(\absdet{a}\cdot[x]^{-1})^{\smallidx-\half\bigidx-\half}\delta_{B_{G_{\smallidx}}}^{-1}(a).
\end{align*}
For $\Re(s)>>0$,
\begin{align*}
\int_{A_{\smallidx-1}}\int_{G_1}(\absdet{a}\cdot[x]^{-1})^{\Re(s)}|\theta(x,a)|dxda<\infty.
\end{align*}
This follows from the proof of Proposition~\ref{proposition:iwasawa decomposition of the integral} but can also be
proved directly, using the expansions of Section~\ref{section:whittaker props}. Hence
if we set
\begin{align*}
(A_{\smallidx-1}\times G_1)^m=\setof{a\in A_{\smallidx-1},x\in G_1}{\absdet{a}\cdot[x]^{-1}=q^{-m}}\qquad(m\in\Integers),
\end{align*}
$\int_{(A_{\smallidx-1}\times G_1)^m}|\theta(a,x)|d(a,x)<\infty$ for all $m\in\Integers$ and \eqref{int:int to series from proposition iwasawa} is equal to
\begin{align}\label{int:int to series from proposition iwasawa to series}
\sum_{m\in\Integers}q^{-sm}\int_{(A_{\smallidx-1}\times G_1)^m}\theta(a,x)d(a,x).
\end{align}
In the \quasisplit\ case we obtain a similar form, without the $dx$-integration.
The series in $\Sigma(X)$ representing \eqref{int:int to series from proposition iwasawa} is obtained by replacing
$q^{-s}$ with $X$ in \eqref{int:int to series from proposition iwasawa to series}. To show that it has a finite number of negative coefficients, we need to prove that $\theta(a,x)$ vanishes when $m$ is small.
Indeed, if $m$ is small, $\absdet{a}$ must be large (in the split case $[x]^{-1}=|\lfloor x\rfloor|\leq1$), then
$\theta(a,x)\equiv0$ because $W^{\diamond}\in\Whittaker{\pi}{\psi_{\gamma}^{-1}}$ vanishes away
from zero (i.e., $W^{\diamond}(ax)=0$ unless each coordinate of $a$ is bounded from above).
Note that we could similarly use the vanishing away from zero of $W'\in\Whittaker{\tau}{\psi}$. When $\smallidx>\bigidx$ we
consider \eqref{int:iwasawa decomposition of the integral l>n}
instead of \eqref{int:iwasawa decomposition of the integral l<=n
split} for this argument.
\end{proof} 
For any $f_s\in\xi(\tau,good,s)$ there exists $P\in
\C[q^{-s},q^s]$ which divides $\ell_{\tau^*}(1-s)^{-1}$, such that
$Pf_s\in\xi(\tau,hol,s)$. According to the above, $\Psi(W,Pf_s,s)$
defined in $D$ is representable in $\Sigma(X)$.

\subsection{Extension to $k$ variables}\label{subsection:Extension to two variables}
We may extend the definitions and results of
Sections~\ref{subsection:functions and laurent series} and
\ref{subsection:Representations of integrals as series} to functions
$f(s_1,\ldots,s_k)$ in $k$ complex variables. This extension will not, however, be used in the sequel.
Consider the ring $C^{\infty}(\Gamma,\C[q^{\mp
s_1},\ldots, q^{\mp s_k}])$. For a finite non-empty subset
$M\subset\Integers^k$ let
\begin{align*}
\mathcal{V}_{M}=\Span{\C}{\setof{q^{-s_1j_1-\ldots-s_kj_k}}{(j_1,\ldots,j_k)\in M}}\subset\C[q^{\mp
s_1},\ldots, q^{\mp s_k}].
\end{align*}
Let $\phi\in C^{\infty}(\Gamma,\C[q^{\mp s_1},\ldots, q^{\mp
s_k}])$. The set $supp_{M}(\phi)$ is defined exactly as in Section~\ref{subsection:Representations of integrals as series}, with respect to the new definition of $\mathcal{V}_{M}$.
The space $\Gamma$ can be divided into the simple supports of $\phi$ if
whenever $M$ is not a singleton, $supp_{M}(\phi)=\emptyset$.
For $\underline{m}=(m_1,\ldots,m_k)\in\Integers^k$ denote
\begin{align*}
P_{\underline{m}}(\sum_{j_1,\ldots,j_k\in\Integers}a_{j_1,\ldots,j_k}q^{-s_1j_1-\ldots-s_kj_k})=a_{m_1,\ldots,m_k}\in\C.
\end{align*}
Provided $\int_{\Gamma}|P_{\underline{m}}(\phi(x))|dx<\infty$ for all $\underline{m}\in\Integers^k$, the Laurent series in $\Sigma(X_1,\ldots,X_k)=\C[[X_1^{\mp1},\ldots,X_k^{\mp1}]]$ is defined by
\begin{align*}
\oint_{\Gamma}\phi(x)dx=\sum_{\underline{m}\in\Integers^k}X_1^{m_1}\cdot\ldots\cdot X_k^{m_k} \int_{\Gamma}P_{\underline{m}}(\phi(x))dx.
\end{align*}
The definition closely resembles Definition~\ref{def:series for
bundle}. Now the results of Section~\ref{subsection:Representations
of integrals as series} readily extend to $k$ variables.

\section{Upper bound in the second
variable}\label{section:Sub-multiplicativity in the second variable}
\subsection{Outline}\label{subsection:second var outline}
Let $\pi$ be a representation of $G_{\smallidx}$.
In this section we
prove Theorem~\ref{theorem:gcd sub multiplicity second var}. Namely,
if
$\tau=\cinduced{P_{\bigidx_1,\ldots,\bigidx_k}}{\GL{\bigidx}}{\tau_1\otimes\ldots\otimes\tau_k}$ is irreducible,
\begin{align*}
\gcd(\pi\times\tau,s)\in
(\prod_{i=1}^k\gcd(\pi\times\tau_i,s))M_{\tau_1\otimes\ldots\otimes\tau_k}(s)\C[q^{-s},q^s].
\end{align*}
We revisit the arguments of Section~\ref{section:2nd variable} where
we proved that $\gamma(\pi\times\tau,\psi,s)$ is multiplicative in
the second variable.
The proof relied on manipulations of integrals, involving the
application of three functional equations - for $\pi\times\tau_2$,
$\tau_1\times\tau_2$ and $\pi\times\tau_1$, to $\Psi(W,f_s,s)$.
Here we reinterpret the passages as passages between Laurent series
and apply the functional equations to the series $\Sigma(W,f_s,s)$
which represents $\Psi(W,f_s,s)$. We utilize the notation and
results of Section~\ref{section:integrals and series} (e.g.
$\Sigma(X)=\C[[X,X^{-1}]]$,
$R(X)=\C[X,X^{-1}]$).

\subsection{Interpretation of functional equations}\label{subsection:Interpretation of functional equations}
The functional equations, at first defined between meromorphic
continuations, can sometimes be interpreted as equations in
$\Sigma(X)$, by substituting $X$ for $q^{-s}$, as observed in 
\cite{JPSS} (Section~4.3). Consider a
typical functional equation,
\begin{align}\label{eq:n_1<l<=n_2 example func eq}
&\gcd(\pi\times\tau,s)^{-1}\Psi(W,f_s,s)\\\notag&=c(\smallidx,\tau,\gamma,s)\epsilon(\pi\times\tau,\psi,s)^{-1}\gcd(\pi\times\tau^*,1-s)^{-1}\Psi(W,\nintertwiningfull{\tau}{s}f_s,1-s).
\end{align}
The integral on the \lhs\ (resp. \rhs) is absolutely convergent in a right (resp. left) half-plane.
This equality is a priori between meromorphic continuations, but by
the definition of the g.c.d. and because of Claim~\ref{claim:epsilon
is a unit}, both sides are actually polynomials. Then \eqref{eq:n_1<l<=n_2 example func eq} may be
reinterpreted as an equality in $R(X)$ by replacing $q^{-s}$ with $X$. If $\gcd(\pi\times\tau,s)^{-1}\Psi(W,f_s,s)=B(q^{-s},q^s)$ and the
\rhs\ equals $\widetilde{B}(q^{-s},q^s)$ for some $B,\widetilde{B}\in \C[q^{-s},q^s]$,
equality~\eqref{eq:n_1<l<=n_2 example func eq} implies $B(X,X^{-1})=\widetilde{B}(X,X^{-1})$ in $R(X)$.

Furthermore, in a right half-plane the meromorphic continuation of $\Psi(W,f_s,s)$ coincides with the integral, hence the equality
$\gcd(\pi\times\tau,s)^{-1}\Psi(W,f_s,s)=B(q^{-s},q^s)$ holds for $\Re(s)>>0$ where $\Psi(W,f_s,s)$ is the integral.
Similarly for $\Re(s)<<0$, the \rhs\ equals $\widetilde{B}(q^{-s},q^s)$ where $\Psi(W,\nintertwiningfull{\tau}{s}f_s,1-s)$ is regarded as the integral.

By definition $\gcd(\pi\times\tau,s)^{-1}$ is a polynomial in $q^{-s}$. Put
$G(q^{-s})^{-1}=\gcd(\pi\times\tau,s)^{-1}$ and denote by $G(X)^{-1}\in
R(X)$ (in fact, $G(X)^{-1}\in\C[X]$) the polynomial obtained by replacing
$q^{-s}$ with $X$. Similarly put
$\widetilde{G}(q^{-s},q^s)^{-1}=c(\smallidx,\tau,\gamma,s)\epsilon(\pi\times\tau,\psi,s)^{-1}\gcd(\pi\times\tau^*,1-s)^{-1}$
and form $\widetilde{G}(X,X^{-1})^{-1}\in R(X)$.

If $f_s\in\xi(\tau,hol,s)$, by Lemma~\ref{lemma:integral for
holomorphic section representable as a series} there exists a series
$\Sigma(W,f_s,s)$ which represents the integral $\Psi(W,f_s,s)$ in some right half-plane. Hence in $\Sigma(X)$ it holds that
$G(X)^{-1}\Sigma(W,f_s,s)=B(X,X^{-1})$.
If also $\nintertwiningfull{\tau}{s}f_s\in\xi(\tau^*,hol,1-s)$,
equality~\eqref{eq:n_1<l<=n_2 example func eq} may be interpreted in
$\Sigma(X)$ as
\begin{align*}
G(X)^{-1}\Sigma(W,f_s,s)=\widetilde{G}(X,X^{-1})^{-1}\Sigma(W,\nintertwiningfull{\tau}{s}f_s,1-s).
\end{align*}
However, it may be the case that $\nintertwiningfull{\tau}{s}f_s$
has poles (due to the intertwining operator). Let $0\ne P\in
\C[q^{-s},q^s]$ be such that
$P(q^{-s},q^s)\nintertwiningfull{\tau}{s}f_s$ is holomorphic. Since
\begin{align*}
P(q^{-s},q^s)\Psi(W,\nintertwiningfull{\tau}{s}f_s,1-s)=\Psi(W,P(q^{-s},q^s)\nintertwiningfull{\tau}{s}f_s,1-s),
\end{align*}
we can multiply both sides of \eqref{eq:n_1<l<=n_2 example func eq},
which are polynomials, by $P$ and reach the equivalent equation
\begin{align*}
&P(q^{-s},q^s)\gcd(\pi\times\tau,s)^{-1}\Psi(W,f_s,s)\\\notag&=c(\smallidx,\tau,\gamma,s)\epsilon(\pi\times\tau,\psi,s)^{-1}\gcd(\pi\times\tau^*,1-s)^{-1}\Psi(W,P(q^{-s},q^s)\nintertwiningfull{\tau}{s}f_s,1-s).
\end{align*}
Now this may be interpreted in $\Sigma(X)$ as
\begin{align*}
&P(X,X^{-1})G(X)^{-1}\Sigma(W,f_s,s)=\widetilde{G}(X,X^{-1})^{-1}\Sigma(W,P(q^{-s},q^s)\nintertwiningfull{\tau}{s}f_s,1-s).
\end{align*}

We also mention that if an integral $\Psi(W,f_s,s)$ for
$f_s\in\xi(\tau,hol,s)$ is already a polynomial, $\Sigma(W,f_s,s)\in
R(X)$ and the analytic continuation of $\Psi(W,f_s,s)$ equals
$(\Sigma(W,f_s,s))(q^{-s},q^s)$.
\subsection{Proof of the bound}\label{subsection:sub mult second var proof body}
The proof of Theorem~\ref{theorem:gcd sub multiplicity second var} is based on the following lemma.
\begin{lemma}\label{lemma:n_1<l<=n_2 basic identity}
Assume that
$\tau=\cinduced{P_{\bigidx_1,\bigidx_2}}{\GL{\bigidx}}{\tau_1\otimes\tau_2}$
is irreducible. Let $P_1,P_2,P_3\in\C[X]$ be normalized by
$P_1(0)=P_2(0)=P_3(0)=1$, and of minimal degree such that
\begin{align*}
&P_1(q^{-s})\gcd(\pi\times\tau_2^*,1-s)^{-1}\nintertwiningfull{\tau_2}{s},\\
&P_2(q^{-s})\nintertwiningfull{\tau_1\otimes\tau_2^*}{(s,1-s)},\\
&P_3(q^{-s})\gcd(\pi\times\tau_1^*,1-s)^{-1}\nintertwiningfull{\tau_1}{s}
\end{align*}
are holomorphic. Set $P_{\pi\times\tau}=P_1P_2P_3\in\C[X]$. Then for
$f_s\in\xi(\tau,hol,s)$,
\begin{align}\label{eq:second var gcd of holomorphic sections}
&\Psi(W,f_s,s)\in\gcd(\pi\times\tau_1,s)\gcd(\pi\times\tau_2,s)P_{\pi\times\tau}(q^{-s})^{-1}\C[q^{-s},q^s],\\
&\label{eq:second var gcd of holomorphic sections 2}
\Psi(W,\nintertwiningfull{\tau}{s}f_s,1-s)\in\gcd(\pi\times\tau_1^*,1-s)\gcd(\pi\times\tau_2^*,1-s)P_{\pi\times\tau}(q^{-s})^{-1}\C[q^{-s},q^s].
\end{align}
\end{lemma}
Before getting to the proof of the lemma we derive the
theorem. First assume $k=2$,
$\tau=\cinduced{P_{\bigidx_1,\bigidx_2}}{\GL{\bigidx}}{\tau_1\otimes\tau_2}$.
By applying Lemma~\ref{lemma:n_1<l<=n_2 basic identity}
to $\tau^*$
, for $f_{1-s}'\in\xi(\tau^*,hol,1-s)$,
\begin{align}\label{eq:second var gcd of holomorphic sections for tau star}
&\Psi(W,\nintertwiningfull{\tau^*}{1-s}f_{1-s}',s)\in\gcd(\pi\times\tau_1,s)\gcd(\pi\times\tau_2,s)P_{\pi\times\tau^*}(q^{s-1})^{-1}\C[q^{-s},q^s].
\end{align}
According to the definitions of $P_{\pi\times\tau}$, $P_1$,
$\ell_{\tau_2}(s)$ ($P_2$, etc.), we have
\begin{align*}
P_{\pi\times\tau}(q^{-s})^{-1}\in
\ell_{\tau_1}(s)\ell_{\tau_1\otimes\tau_2^*}(s)\ell_{\tau_2}(s)\C[q^{-s},q^s].
\end{align*}
Similarly,
\begin{align*}
P_{\pi\times\tau^*}(q^{s-1})^{-1}\in
\ell_{\tau_1^*}(1-s)\ell_{\tau_2^*\otimes\tau_1}(1-s)\ell_{\tau_2^*}(1-s)\C[q^{-s},q^s].
\end{align*}
We also apply the lemma to $\tau$ and $f_s\in\xi(\tau,hol,s)$.
Combining \eqref{eq:second var gcd of holomorphic sections} and
\eqref{eq:second var gcd of holomorphic sections for tau star} we
see that for any $f_s\in\xi(\tau,good,s)$,
\begin{align*}
\Psi(W,f_s,s)&\in\gcd(\pi\times\tau_1,s)\gcd(\pi\times\tau_2,s)M_{\tau_1}(s)
\ell_{\tau_1\otimes\tau_2^*}(s)\ell_{\tau_2^*\otimes\tau_1}(1-s)
M_{\tau_2}(s)\C[q^{-s},q^s]\\\notag&=
\gcd(\pi\times\tau_1,s)\gcd(\pi\times\tau_2,s)M_{\tau_1\otimes\tau_2}(s)\C[q^{-s},q^s].
\end{align*}

Now for the general case let $k>2$. Put 
$G_{\pi\times(\tau_1\otimes\ldots\otimes\tau_k)}(s)=\prod_{i=1}^k\gcd(\pi\times\tau_i,s)$.
Let
$\eta=\cinduced{P_{\bigidx_2,\ldots,\bigidx_k}}{\GL{\bigidx-\bigidx_1}}{\tau_2\otimes\ldots\otimes\tau_k}$ ($\eta$ is irreducible).
Applying induction hypothesis to $\eta^*$ yields
\begin{align*}
\gcd(\pi\times\eta^*,1-s)\in
G_{\pi\times(\tau_k^*\otimes\ldots\otimes\tau_2^*)}(1-s)M_{\tau_k^*\otimes\ldots\otimes\tau_2^*}(1-s)\C[q^{-s},q^s].
\end{align*}

Apply Lemma~\ref{lemma:n_1<l<=n_2 basic identity} to
$\tau=\cinduced{P_{\bigidx_1,\bigidx-\bigidx_1}}{\GL{\bigidx}}{\tau_1\otimes\eta}$.
By the minimality of $P_1$ and since $\ell_{\eta}(s)\in
M_{\tau_k^*\otimes\ldots\otimes\tau_2^*}(1-s)\C[q^{-s},q^s]$
(because of the multiplicativity of $\nintertwiningfull{\eta}{s}$,
see \eqref{eq:multiplicative property of poles of intertwiners}),
\begin{align*}
P_1(q^{-s})^{-1}\gcd(\pi\times\eta^*,1-s)\in
G_{\pi\times(\tau_k^*\otimes\ldots\otimes\tau_2^*)}(1-s)M_{\tau_k^*\otimes\ldots\otimes\tau_2^*}(1-s)\C[q^{-s},q^s].
\end{align*}
Also
$\ell_{\tau_1\otimes\eta^*}(s)\in\prod_{j=2}^k\ell_{\tau_1\otimes\tau_j^*}(s)\C[q^{-s},q^s]$
and $\ell_{\tau_1}(s)\in M_{\tau_1^*}(1-s)\C[q^{-s},q^s]$. Hence
\begin{align}
\notag&P_2(q^{-s})^{-1}P_1(q^{-s})^{-1}\gcd(\pi\times\eta^*,1-s)\\\notag&\qquad\in G_{\pi\times(\tau_k^*\otimes\ldots\otimes\tau_2^*)}(1-s)M_{\tau_k^*\otimes\ldots\otimes\tau_2^*}(1-s)\prod_{j=2}^k\ell_{\tau_1\otimes\tau_j^*}(s)\C[q^{-s},q^s],\notag\\
\label{eq:P inverse is enough for right hand side basic identity
second var}
&P_{\pi\times\tau}(q^{-s})^{-1}\gcd(\pi\times\tau_1^*,1-s)\gcd(\pi\times\eta^*,1-s)\\\notag&\qquad\in
G_{\pi\times(\tau_k^*\otimes\ldots\otimes\tau_1^*)}(1-s)M_{\tau_k^*\otimes\ldots\otimes\tau_1^*}(1-s)\C[q^{-s},q^s].
\end{align}
Then it follows from \eqref{eq:second var gcd of holomorphic
sections 2} ($\eta^*$ replaces $\tau_2^*$) and \eqref{eq:P inverse is
enough for right hand side basic identity second var} that for
$f_s\in\xi(\tau,hol,s)$,
\begin{align}\label{eq:second var gcd general image of intertwining}
\Psi(W,\nintertwiningfull{\tau}{s}f_s,1-s)\in
G_{\pi\times(\tau_k^*\otimes\ldots\otimes\tau_1^*)}(1-s)M_{\tau_k^*\otimes\ldots\otimes\tau_1^*}(1-s)\C[q^{-s},q^s].
\end{align}
Using Theorem~\ref{theorem:multiplicity second var} 
and \eqref{eq:gamma and epsilon},
\begin{align*}
\frac{G_{\pi\times(\tau_k^*\otimes\ldots\otimes\tau_1^*)}(1-s)}{G_{\pi\times(\tau_1\otimes\ldots\otimes\tau_k)}(s)}\equalun\frac{\gcd(\pi\times\tau^*,1-s)}{\gcd(\pi\times\tau,s)}\equalun\frac
{\gcd(\pi\times\tau_1^*,1-s)\gcd(\pi\times\eta^*,1-s)}{\gcd(\pi\times\tau_1,s)\gcd(\pi\times\eta,s)}.
\end{align*}
Then from \eqref{eq:P inverse is enough for
right hand side basic identity second var} we deduce, similarly to the proof of Claim~\ref{claim:extra P factor in sub
multiplicativity is mutual},
\begin{align}\label{eq:deducing relation basic identity}
P_{\pi\times\tau}(q^{-s})^{-1}\gcd(\pi\times\tau_1,s)\gcd(\pi\times\eta,s)\in
G_{\pi\times(\tau_1\otimes\ldots\otimes\tau_k)}(s)M_{\tau_1\otimes\ldots\otimes\tau_k}(s)\C[q^{-s},q^s].
\end{align}
($M_{\tau_1\otimes\ldots\otimes\tau_k}(s)=M_{\tau_k^*\otimes\ldots\otimes\tau_1^*}(1-s)$.)
Combining \eqref{eq:second var gcd of holomorphic sections} with
\eqref{eq:deducing relation basic identity} we assert that for
$f_{s}\in\xi(\tau,hol,s)$,
\begin{align*}
\Psi(W,f_s,s)\in
G_{\pi\times(\tau_1\otimes\ldots\otimes\tau_k)}(s)M_{\tau_1\otimes\ldots\otimes\tau_k}(s)\C[q^{-s},q^s].
\end{align*}
As in the case $k=2$, repeating the arguments above for $\tau^*$ and
taking $f_{1-s}'\in\xi(\tau^*,hol,1-s)$, \eqref{eq:second var gcd
general image of intertwining} becomes
\begin{align*}
\Psi(W,\nintertwiningfull{\tau^*}{1-s}f_{1-s}',s)\in
G_{\pi\times(\tau_1\otimes\ldots\otimes\tau_k)}(s)M_{\tau_1\otimes\ldots\otimes\tau_k}(s)\C[q^{-s},q^s].
\end{align*}
We conclude the theorem,
\begin{align*}
\gcd(\pi\times\tau,s)\in
G_{\pi\times(\tau_1\otimes\ldots\otimes\tau_k)}(s)M_{\tau_1\otimes\ldots\otimes\tau_k}(s)\C[q^{-s},q^s].
\end{align*}

The result we obtain 
is somewhat stronger. 
In the statement of the lemma we see, for instance, that if
$\gcd(\pi\times\tau_2^*,1-s)$ already contains the poles of
$\nintertwiningfull{\tau_2}{s}$, $P_1=1$. Hence the factor
$M_{\tau_2}(s)$ in the upper bound for $k=2$ will be replaced with
just $\ell_{\tau_2^*}(1-s)$. Applying induction more carefully,
keeping track of the poles of intertwining operators which are
canceled by the g.c.d. factors, yields the following refinement of
Theorem~\ref{theorem:gcd sub multiplicity second var}.
\begin{corollary}\label{corollary:refinement gcd second var bound}
Let $\tau=\cinduced{P_{\bigidx_1,\ldots,\bigidx_k}}{\GL{\bigidx}}{\tau_1\otimes\ldots\otimes\tau_k}$ be irreducible.
For all $1\leq i\leq k$, let $A_i,\widetilde{A}_i\in\C[X]$ be
normalized by $A_i(0)=\widetilde{A}_i(0)=1$ and of minimal degree
such that
\begin{align*}
A_i(q^{-s})\gcd(\pi\times\tau_i^*,1-s)^{-1}\nintertwiningfull{\tau_i}{s},\qquad
\widetilde{A}_i(q^{s-1})\gcd(\pi\times\tau_i,s)^{-1}\nintertwiningfull{\tau_i^*}{1-s}
\end{align*}
are holomorphic. Let $Q_i\in\C[q^{-s},q^{s}]$ be a least common
multiple of $A_i(q^{-s})$ and $\widetilde{A}_i(q^{s-1})$. Then
\begin{align*}
\gcd(\pi\times\tau,s)\in(\prod_{i=1}^k\gcd(\pi\times\tau_i,s)Q_i(q^{-s},q^s)^{-1})\prod_{1\leq
i<j\leq
k}\ell_{\tau_i\otimes\tau_j^*}(s)\ell_{\tau_j^*\otimes\tau_i}(1-s)\C[q^{-s},q^s].
\end{align*}
\end{corollary}
\begin{proof}[Proof of Corollary~\ref{corollary:refinement gcd second var bound}]
First assume $k=2$. Let $P_{\pi\times\tau}=P_1P_2P_3$ be as in the lemma applied to $\tau$ and
$f_{s}\in\xi(\tau,hol,s)$, and
$P_{\pi\times\tau^*}=\widetilde{P}_1\widetilde{P}_2\widetilde{P}_3$ (where $\widetilde{P}_i$ is the polynomial
$P_i$ when the lemma is applied to $\tau^*$ and $f_{1-s}\in\xi(\tau^*,hol,1-s)$).
Then
\begin{align*}
&P_1|A_2, \quad P_2(q^{-s})|\ell_{\tau_1\otimes\tau_2^*}(s)^{-1},
\quad  P_3|A_1,\\\notag &\widetilde{P}_1|\widetilde{A}_1,\quad
\widetilde{P}_2(q^{s-1})|\ell_{\tau_2^*\otimes\tau_1}(1-s)^{-1},
\quad \widetilde{P}_3|\widetilde{A}_2.
\end{align*}
Here $|$ denotes polynomial division in either $\C[X]$ or $\C[q^{-s},q^s]$. Also let
$lcm(\cdot,\cdot)$ signify a least common multiple in $\C[q^{-s},q^s]$.
We see that $lcm(P_{\pi\times\tau}(q^{-s}),P_{\pi\times\tau^*}(q^{s-1}))$ divides
\begin{align*}
(\prod_{i=1}^2lcm(A_i(q^{-s}),\widetilde{A}_i(q^{s-1})))(\ell_{\tau_1\otimes\tau_2^*}(s)\ell_{\tau_2^*\otimes\tau_1}(1-s))^{-1}.
\end{align*}
Combining \eqref{eq:second var gcd of holomorphic sections} and
\eqref{eq:second var gcd of holomorphic sections for tau star} we get that for any
$f_s\in\xi(\tau,good,s)$,
\begin{align*}
&\Psi(W,f_s,s)\in\gcd(\pi\times\tau_1,s)\gcd(\pi\times\tau_2,s)
lcm(P_{\pi\times\tau}(q^{-s}),P_{\pi\times\tau^*}(q^{s-1}))^{-1}\C[q^{-s},q^s].
\end{align*}
Therefore
\begin{align*}
\gcd(\pi\times\tau,s)\in
(\prod_{i=1}^2\gcd(\pi\times\tau_i,s)Q_i(q^{-s},q^s)^{-1})\ell_{\tau_1\otimes\tau_2^*}(s)\ell_{\tau_2^*\otimes\tau_1}(1-s)\C[q^{-s},q^s].
\end{align*}

Now assume $k>2$ and use induction on $k$. If
$\eta=\cinduced{P_{\bigidx_2,\ldots,\bigidx_k}}{\GL{\bigidx-\bigidx_1}}{\tau_2\otimes\ldots\otimes\tau_k}$,
by the induction hypothesis
\begin{align*}
&\gcd(\pi\times\eta,s)\\&\quad\in
(\prod_{i=2}^k\gcd(\pi\times\tau_i,s)Q_i(q^{-s},q^s)^{-1})\prod_{2\leq
i<j\leq k}\ell_{\tau_i\otimes\tau_j^*}(s)\ell_{\tau_j^*\otimes\tau_i}(1-s)\C[q^{-s},q^s].
\end{align*}
Then Claim~\ref{claim:extra P factor
in sub multiplicativity is mutual} implies
\begin{align*}
&\gcd(\pi\times\eta^*,1-s)\\\notag&\quad\in
(\prod_{i=2}^k\gcd(\pi\times\tau_i^*,1-s)Q_i(q^{-s},q^s)^{-1})\prod_{2\leq
i<j\leq k}\ell_{\tau_i\otimes\tau_j^*}(s)\ell_{\tau_j^*\otimes\tau_i}(1-s)\C[q^{-s},q^s].
\end{align*}
Apply
Lemma~\ref{lemma:n_1<l<=n_2 basic identity} to
$\tau=\cinduced{P_{\bigidx_1,\bigidx-\bigidx_1}}{\GL{\bigidx}}{\tau_1\otimes\eta}$
and $f_s\in\xi(\tau,hol,s)$. Let $P_1,P_2$ and $P_3$ be the corresponding polynomials of the lemma. According to \eqref{eq:multiplicative property
of poles of intertwiners} with $\eta$ (instead of $\tau$) and because for each $2\leq i\leq k$,
\begin{align*}
\gcd(\pi\times\tau_i^*,1-s)^{-1}Q_i(q^{-s},q^s)\nintertwiningfull{\tau_i}{s}
\end{align*}
is holomorphic,
\begin{align*}
&P_1(q^{-s})^{-1}\gcd(\pi\times\eta^*,1-s)\\\notag&\quad\in
(\prod_{i=2}^k\gcd(\pi\times\tau_i^*,1-s)Q_i(q^{-s},q^s)^{-1})\prod_{2\leq
i<j\leq k}\ell_{\tau_i\otimes\tau_j^*}(s)\ell_{\tau_j^*\otimes\tau_i}(1-s)\C[q^{-s},q^s].
\end{align*}
In addition $P_3|A_1$ and 
\begin{align*}
\ell_{\tau_1\otimes\eta^*}(s)\in\prod_{2\leq
j\leq k}\ell_{\tau_1\otimes\tau_j^*}(s)\ell_{\tau_j^*\otimes\tau_1}(1-s)\C[q^{-s},q^s].
\end{align*}
Thus we get
\begin{align}\label{eq:corollary refinement of upper bound second
variable}
P_{\pi\times\tau}(q^{-s})^{-1}&\gcd(\pi\times\tau_1^*,1-s)\gcd(\pi\times\eta^*,1-s)\\\notag
\in& A_1(q^{-s})^{-1}\prod_{i=1}^k\gcd(\pi\times\tau_i^*,1-s)\prod_{i=2}^kQ_i(q^{-s},q^s)^{-1}\\\notag
&\times\prod_{1\leq i<j\leq
k}\ell_{\tau_i\otimes\tau_j^*}(s)\ell_{\tau_j^*\otimes\tau_i}(1-s)\C[q^{-s},q^s].
\end{align}
Hence from \eqref{eq:second var gcd of holomorphic sections 2}, for
any $W$ and $f_s\in\xi(\tau,hol,s)$,
\begin{align*}
\Psi(W,\nintertwiningfull{\tau}{s}&f_s,1-s)\\\notag
\in& A_1(q^{-s})^{-1}\prod_{i=1}^k\gcd(\pi\times\tau_i^*,1-s)\prod_{i=2}^kQ_i(q^{-s},q^s)^{-1}
\\\notag&\times\prod_{1\leq i<j\leq
k}\ell_{\tau_i\otimes\tau_j^*}(s)\ell_{\tau_j^*\otimes\tau_i}(1-s)\C[q^{-s},q^s].
\end{align*}
In addition from \eqref{eq:corollary refinement of upper bound
second variable} using Theorem~\ref{theorem:multiplicity second var} 
(as in Claim~\ref{claim:extra P factor in sub multiplicativity is mutual}) we obtain
\begin{align*}
P_{\pi\times\tau,s}(q^{-s})^{-1}&\gcd(\pi\times\tau_1,s)\gcd(\pi\times\eta,s)\\\notag
\in &A_1(q^{-s})^{-1}\prod_{i=1}^k\gcd(\pi\times\tau_i,s)\prod_{i=2}^kQ_i(q^{-s},q^s)^{-1}\\\notag
&\times\prod_{1\leq i<j\leq
k}\ell_{\tau_i\otimes\tau_j^*}(s)\ell_{\tau_j^*\otimes\tau_i}(1-s)\C[q^{-s},q^s].
\end{align*}
Therefore \eqref{eq:second var gcd of holomorphic sections} implies that for any $f_s\in\xi(\tau,hol,s)$,
\begin{align}\label{eq:corollary improves sub multiplicativity second var last step}
\Psi(W,f_s,s)\in& A_1(q^{-s})^{-1}\prod_{i=1}^k\gcd(\pi\times\tau_i,s)\prod_{i=2}^kQ_i(q^{-s},q^s)^{-1}\\\notag
&\times\prod_{1\leq i<j\leq
k}\ell_{\tau_i\otimes\tau_j^*}(s)\ell_{\tau_j^*\otimes\tau_i}(1-s)\C[q^{-s},q^s].
\end{align}

Repeating the above for $\tau^*=\cinduced{P_{\bigidx-\bigidx_1,\bigidx_1}}{\GL{\bigidx}}{\eta^*\otimes\tau_1^*}$ and $f_{1-s}'\in\xi(\tau^*,hol,1-s)$, we get
$\widetilde{P}_1|\widetilde{A}_1$ ($\widetilde{P}_1$ - the first
polynomial used in the lemma) and
\begin{align}\label{eq:corollary improves sub multiplicativity second var one step before last}
\Psi(W,\nintertwiningfull{\tau^*}{1-s}&f_{1-s}',s)\\\notag
\in&\widetilde{A}_1(q^{s-1})^{-1}\prod_{i=1}^k\gcd(\pi\times\tau_i,s)\prod_{i=2}^kQ_i(q^{-s},q^s)^{-1}
\\\notag&\times\prod_{1\leq i<j\leq
k}\ell_{\tau_i\otimes\tau_j^*}(s)\ell_{\tau_j^*\otimes\tau_i}(1-s)\C[q^{-s},q^s].
\end{align}

Finally since $Q_1$ is a least common multiple of $A_1(q^{-s})$ and
$\widetilde{A}_1(q^{s-1})$, when we combine \eqref{eq:corollary improves sub multiplicativity second var last step} and
\eqref{eq:corollary improves sub multiplicativity second var one step before last} we obtain
\begin{align*}
&\gcd(\pi\times\tau,s)\\\notag&\quad\in(\prod_{i=1}^k\gcd(\pi\times\tau_i,s)Q_i(q^{-s},q^s)^{-1})\prod_{1\leq
i<j\leq k}\ell_{\tau_i\otimes\tau_j^*}(s)\ell_{\tau_j^*\otimes\tau_i}(1-s)\C[q^{-s},q^s],
\end{align*}
as required.
\end{proof} 

\begin{proof}[Proof of Lemma~\ref{lemma:n_1<l<=n_2 basic identity}] 
We use the notation and results of Section~\ref{subsection:Realization of tau for induced
representation}, replace $\tau$ with $\varepsilon$ and prove the
lemma for $\varepsilon$, keeping $\zeta$ fixed with $\Re(\zeta)>>0$
throughout the proof. Recall that we may assume that $\varepsilon$ is
irreducible.
Eventually we shall put $\zeta=0$ in order to derive the result for $\tau$.
In contrast with the proof of
Theorem~\ref{theorem:multiplicity second var} (see
Section~\ref{subsection:The basic identity tau}),
$\gcd(\pi\times\varepsilon,s)$ is not expected to be rational in
$q^{-\zeta}$ - for example, in a similar setting of $\GL{k}\times\GL{\bigidx}$, the g.c.d. need not be rational in $q^{-\zeta}$.
Thus any data we use will need to be analytic in $\zeta$. This follows the method of 
\cite{JPSS}.
\begin{remark}
Alternatively, we could take $\zeta$ as
a parameter and use Laurent series in two variables. There are minimal
technical differences between these approaches.
\end{remark}
The polynomials $P_i$ are replaced with $P_i^{\zeta}$ as prescribed by the following claim.
\begin{claim}\label{claim:ring homomorphism does not change P_i props}
There exist $P_i^{\zeta}\in\C[X]$, $i=1,2,3$, with coefficients that are polynomial in $q^{\mp\zeta}$, such that
the operators
\begin{align*}
&P_1^{\zeta}(q^{-s})\gcd(\pi\times\varepsilon_2^*,1-s)^{-1}\nintertwiningfull{\varepsilon_2}{s},\\
&P_2^{\zeta}(q^{-s})\nintertwiningfull{\varepsilon_1\otimes\varepsilon_2^*}{(s,1-s)},\\
&P_3^{\zeta}(q^{-s})\gcd(\pi\times\varepsilon_1^*,1-s)^{-1}\nintertwiningfull{\varepsilon_1}{s}
\end{align*}
are holomorphic, and $P_i^0=P_i$ for each $1\leq i\leq3$.
\end{claim}
\begin{proof}[Proof of Claim~\ref{claim:ring homomorphism does not change P_i props}] 
Start with $P_1^{\zeta}$. By definition, the zeros of
$\gcd(\pi\times\varepsilon_2^*,1-s)^{-1}=\gcd(\pi\times\tau_2^*,1-(s-\zeta))^{-1}$
can be written in the form
$(1-aq^{-s+\zeta})$, $a\in\C^*$. The poles of
$\nintertwiningfull{\varepsilon_2}{s}=\nintertwiningfull{\tau_2}{s-\zeta}$
are of the form $(1-aq^{-s+\zeta})^{-1}$. The multiplicity of a
factor $(1-aq^{-s+\zeta})^{\pm1}$ in either
$\gcd(\pi\times\varepsilon_2^*,1-s)^{-1}$ or
$\ell_{\varepsilon_2}(s)$ does not change when we
substitute $0$ for $\zeta$. Then $P_1^{\zeta}$ is chosen by taking
suitable factors $(1-aq^{-s+\zeta})$ and replacing $q^{-s}$ with $X$. The argument for $P_3^{\zeta}$
is similar.

The poles of
$\nintertwiningfull{\varepsilon_1\otimes\varepsilon_2^*}{(s,1-s)}$
appear either in $L(\varepsilon_1^*\times\varepsilon_2^*,2-2s)$ (due
to the normalization of the intertwining operator) or as poles of
$\intertwiningfull{\varepsilon_1\otimes\varepsilon_2^*}{(s,1-s)}$.
An argument as in the proof of Corollary~\ref{corollary:auxiliary
rep to embed all sections in gcd second var l<=n} shows that the
latter poles appear in $L(\phi_i\times\theta_j^*,2s-1)$, where
$\phi_i,\theta_j$ are irreducible supercuspidal representations
independent of $\zeta$ (e.g. $\tau_2^*$ is a sub-representation of a
representation parabolically induced from $\theta_1\otimes\ldots\otimes\theta_m$).
We let $P_2^{\zeta}$ be a product of factors $(1-aX^2)$.
\end{proof} 
Let $W\in\Whittaker{\pi}{\psi_{\gamma}^{-1}}$ and
$f_s\in\xi_{Q_{\bigidx}}^{H_{\bigidx}}(\tau,hol,s)$. Let
$\varphi_s\in\xi_{Q_{\bigidx_1,\bigidx_2}}^{H_{\bigidx}}(\varepsilon_1\otimes\varepsilon_2,hol,(s,s))$
be defined by Lemma~\ref{lemma:replacing a section with another} for $f_s$. Then
$\widehat{f}_{\varphi_s}$, given by \eqref{iso:iso 1}, satisfies $\widehat{f}_{\varphi_s}=f_s$
for $\zeta=0$. The function $h\mapsto \varphi_s(h,1,1,1)$ belongs to
$C^{\infty}(H_{\bigidx},\C[q^{-s},q^s])$ (see
Example~\ref{example:f_s is a holomorphic section}). We replace
$\Psi(W,f_s,s)$ with $\Psi(W,\varphi_s,s)$ (see
\eqref{int:n_1<l<=n_2 example starting form}). Claim~\ref{claim:convergence n_1<l<n_2
starting integral} shows that there is a domain $D$ in $\zeta$ and
$s$, depending only on $\pi$, $\tau_1$ and $\tau_2$, such that
$\Psi(W,\varphi_s,s)$ is absolutely convergent. This
domain is of the form $\setof{\zeta,s\in\C}{\Re(s)>>\Re(\zeta)>>0}$.
Recall that we fix $\zeta$,
hence we refer to this domain as a domain in $s$.

As in the proof of
Lemma~\ref{lemma:integral for holomorphic section representable as a
series}, one shows that $\Psi(W,\varphi_s,s)$ is represented by a
series $\Sigma(W,\varphi_s,s)\in\Sigma(X)$ which is strongly
convergent in $D$. For instance when $\smallidx\leq\bigidx$, we take
$\Gamma=\overline{B_{G_{\smallidx}}}\times
R_{\smallidx,\bigidx}\times Z_{\bigidx_2,\bigidx_1}$ and define
$\phi\in C^{\infty}(\Gamma,\C[q^{-s},q^s])$ using the integrand,
starting with
a standard section, then
using Lemmas~\ref{lemma:interchange integration and polynomial} and
\ref{lemma:sum of series
represents sum of integrals}.

During our series computations,
it is allowed to suppress certain factors in $R(X)^*$ (mostly $\epsilon$-factors). For example, in
the proof of Claim~\ref{claim:ring homomorphism does not change P_i
props} we wrote the zeros of $\gcd(\pi\times\tau_2^*,1-(s-\zeta))^{-1}$
in the form $(1-aq^{-s+\zeta})$ while by definition
$\gcd(\pi\times\tau_2^*,1-(s-\zeta))^{-1}\in\C[q^{s-\zeta-1}]$. These factors must
also be analytic in $\zeta$, we indicate them whenever they occur.

First we apply the functional equation for $\pi\times\varepsilon_2$.
For $h\in H_{\bigidx}$, $b_1\in\GL{\bigidx_1}$, denote by
$\varphi_s(h,b_1,\cdot,\cdot)$ the function $(h_2,b_2)\mapsto
\varphi_s(h,b_1,h_2,b_2)$ ($h_2\in H_{\bigidx_2}$,
$b_2\in\GL{\bigidx_2}$). Then
$\varphi_s(h,b_1,\cdot,\cdot)\in\xi_{Q_{\bigidx_2}}^{H_{\bigidx_2}}(\varepsilon_2,hol,s)$.
The function
\begin{align*}
\nintertwiningfull{\varepsilon_2}{s}\varphi_s\in\xi_{Q_{\bigidx_1,\bigidx_2}}^{H_{\bigidx}}(\varepsilon_1\otimes\varepsilon_2^*,rat,(s,1-s))
\end{align*}
is obtained by applying the intertwining operator to
$\varphi_s(h,b_1,\cdot,\cdot)$ (see Section~\ref{subsection:the multiplicativity of the intertwining operator for tau induced}). Let
\begin{align*}
\varphi_{s,1-s}=P_1^{\zeta}(q^{-s})\gcd(\pi\times\varepsilon_2^*,1-s)^{-1}\nintertwiningfull{\varepsilon_2}{s}\varphi_s\in
\xi_{Q_{\bigidx_1,\bigidx_2}}^{H_{\bigidx}}(\varepsilon_1\otimes\varepsilon_2^*,hol,(s,1-s)).
\end{align*}
We use the notation $\Psi(W,\theta_{s,1-s},(s,1-s))$ to denote the
Rankin-Selberg integral for $G_{\smallidx}\times\GL{\bigidx}$, where
$\theta_{s,1-s}$ is in the space of the induced representation\\
$\cinduced{Q_{\bigidx_1}}{H_{\bigidx}}{(\varepsilon_1\otimes
\cinduced{Q_{\bigidx_2}}{H_{\bigidx_2}}{\varepsilon_2\alpha^{1-s}})\alpha^{s}}$
and integral~\eqref{iso:iso 1} is applied to $\theta_{s,1-s}$. This is
the triple integral \eqref{int:n_1<l<=n_2 example starting form} with $\theta_{s,1-s}$
replacing $\varphi_s$. The integral is absolutely convergent in a
domain $D^*$ of the form
$\setof{\zeta,s\in\C}{\Re(\zeta)>>\Re(s)\geq A_1}$ with some constant $A_1$. Similarly to the domain $D$, since $\zeta$ is fixed,
$D^*$ is a domain in $s$ depending only on $\pi$, $\tau_1$ and $\tau_2$
(the parameters for all of the domains we consider are similar to
those calculated in \cite{Soudry2}, see also Section~\ref{subsection:n_1<l<=n_2}). Because $\varphi_{s,1-s}$ is a holomorphic section, the proof of
Lemma~\ref{lemma:integral for holomorphic section representable as a
series} shows that we have a series
$\Sigma(W,\varphi_{s,1-s},(s,1-s))$ representing
$\Psi(W,\varphi_{s,1-s},(s,1-s))$ in $D^*$.

Let $G(\pi\times\varepsilon_i,X)^{-1}\in R(X)$ be the polynomial
obtained from $\gcd(\pi\times\varepsilon_i,s)^{-1}$ by replacing
$q^{-s}$ with $X$, for $i=1,2$. We have,
\begin{claim}\label{claim:n_1<l<=n_2 func pi and tau_2}
$P_1^{\zeta}(X)G(\pi\times\varepsilon_2,X)^{-1}\Sigma(W,\varphi_s,s)=\Sigma(W,\varphi_{s,1-s},(s,1-s))$.
\end{claim}
\begin{proof}[Proof of Claim~\ref{claim:n_1<l<=n_2 func pi and tau_2}] 
We prove this claim for the case $\smallidx\leq\bigidx_2$.
Essentially, three versions of the proof are needed: for
$\smallidx\leq\bigidx_2$, $\bigidx_2<\smallidx\leq\bigidx$ and
$\smallidx>\bigidx$. The integral manipulations in each case are
different but the method of proof is the same.

According to the transition \eqref{int:mult l<n_2 starting
form}-\eqref{int:mult l<n_2 before applying eq for inner tau}, in
$D$ it holds that
\begin{align}\label{int:mult n_1<l<=n_2 before applying eq for inner tau}
\Psi(W,\varphi_s,s)&=\int_{U_{\bigidx_1}}
\int_{\overline{B_{G_{\smallidx}}}}
\int_{R_{\smallidx,\bigidx_2}}
W(g)\varphi_s(w'u,1,w_{\smallidx,\bigidx_2}r(\rconj{b_{\bigidx_2,\bigidx_1}}g),1)
\psi_{\gamma}(r)\psi_{\gamma}(u)drdgdu.
\end{align}
Here we replaced the $dg$-integration over
$\lmodulo{U_{G_{\smallidx}}}{G_{\smallidx}}$ by an integration
over $\overline{B_{G_{\smallidx}}}$, $w'\in H_{\bigidx}$ is a Weyl
element and
$b_{\bigidx_2,\bigidx_1}=diag(I_{\bigidx_2},(-1)^{\bigidx_1},I_{\bigidx_2})\in\GL{2\bigidx_2+1}$. If we substitute $|W|,|\varphi_s|$ for $W,\varphi_s$ and drop the characters, the \rhs\ of \eqref{int:mult n_1<l<=n_2 before applying eq for inner tau} is
convergent in $D$.
We transform this equality into an equality of series.
Let $\Gamma=U_{\bigidx_1}$,
$\Gamma'=\overline{B_{G_{\smallidx}}}\times
R_{\smallidx,\bigidx_2}$. The function
\begin{align*}
\phi(u,(g,r))=W(g)\varphi_s(w'u,1,w_{\smallidx,\bigidx_2}r(\rconj{b_{\bigidx_2,\bigidx_1}}g),1)
\psi_{\gamma}(r)\psi_{\gamma}(u)
\end{align*}
belongs to $C^{\infty}(\Gamma\times\Gamma',\C[q^{-s},q^s])$. Since
$\varphi_s$ and $\psi_{\gamma}$ are smooth, $\phi$ is smooth in
$\Gamma$. The integral
$\Phi_{\phi}(s)=\int_{\Gamma\times\Gamma'}[\phi(u,(g,r))](s)d(u,(g,r))$ equals
the \rhs\ of \eqref{int:mult n_1<l<=n_2 before applying eq for inner tau}.
As in the proof of Lemma~\ref{lemma:integral for
holomorphic section representable as a series} we see that
$\Sigma_{\phi}=\oint_{\Gamma\times\Gamma'}\phi(u,(g,r))d(u,(g,r))$
is defined and 
$\Sigma_{\phi}\sim_D\Phi_{\phi}$. Since $\Psi(W,\varphi_s,s)=\Phi_{\phi}(s)$ in $D$, we get
$\Sigma(W,\varphi_s,s)=\Sigma_{\phi}$.

Apply Lemma~\ref{lemma:double iterated series} to $\Sigma_{\phi}$.
Then for all $u\in \Gamma$, $\Sigma_{\phi(u,\cdot)}$ is defined and
$\Sigma_{\phi(u,\cdot)}\sim_D\Phi_{\phi(u,\cdot)}$. Note that
$\Phi_{\phi(u,\cdot)}=\psi_{\gamma}(u)\Psi(W,\varphi_s^{\rconj{b_{\bigidx_2,\bigidx_1}}}(w'u,1,\cdot,\cdot),s)$, where
$\varphi_s^{\rconj{b_{\bigidx_2,\bigidx_1}}}(w'u,1,\cdot,\cdot)$ denotes the function
\begin{align*}
(h_2,b_2)\mapsto\varphi_s(w'u,1,\rconj{b_{\bigidx_2,\bigidx_1}}h_2,b_2).
\end{align*}
Set $Q=P_1^{\zeta}(q^{-s})\gcd(\pi\times\varepsilon_2,s)^{-1}\in
\C[q^{-s}]$. By Lemma~\ref{lemma:interchange integration and
polynomial},
\begin{align*}
\Sigma_{Q\phi(u,\cdot)}\sim_DQ(q^{-s})\psi_{\gamma}(u)\Psi(W,\varphi_s^{\rconj{b_{\bigidx_2,\bigidx_1}}}(w'u,1,\cdot,\cdot),s).
\end{align*}
Since
$\varphi_s^{\rconj{b_{\bigidx_2,\bigidx_1}}}(w'u,1,\cdot,\cdot)\in\xi_{Q_{\bigidx_2}}^{H_{\bigidx_2}}(\varepsilon_2,hol,s)$ (because $\varphi_s(w'u,1,\cdot,\cdot)\in\xi_{Q_{\bigidx_2}}^{H_{\bigidx_2}}(\varepsilon_2,hol,s)$),
\begin{align*}
\Sigma_{Q\phi(u,\cdot)}=Q(X)\psi_{\gamma}(u)\Sigma(W,\varphi_s^{\rconj{b_{\bigidx_2,\bigidx_1}}}(w'u,1,\cdot,\cdot),s)\in
R(X).
\end{align*}
Also by Lemma~\ref{lemma:interchange integration and
polynomial}, $Q(X)\Sigma_{\phi}=\Sigma_{Q\phi}$.
Now apply Lemma~\ref{lemma:double iterated series} again, to
$\Sigma_{Q\phi}$ and obtain
\begin{align*}
(Q\phi)_{\Gamma'}(u)=\Sigma_{Q\phi(u,\cdot)}(q^{-s},q^s),\qquad
\Sigma_{Q\phi}=\Sigma_{(Q\phi)_{\Gamma'}}.
\end{align*}
According to the functional equation for $G_{\smallidx}\times
\GL{\bigidx_2}$ and $\pi\times\varepsilon_2$, in $\C[q^{-s},q^s]$ we
have
\begin{align}\label{eq:n_1<l<=n_2 func pi and tau_2 for single u}
&\gcd(\pi\times\varepsilon_2,s)^{-1}\Psi(W,\varphi_s^{\rconj{b_{\bigidx_2,\bigidx_1}}}(w'u,1,\cdot,\cdot),s)
=c(\smallidx,\varepsilon_2,\gamma,s)\epsilon(\pi\times\varepsilon_2,\psi,s)^{-1}\\\notag&\times\gcd(\pi\times\varepsilon_2^*,1-s)^{-1}
\Psi(W,\nintertwiningfull{\varepsilon_2}{s}\varphi_s^{\rconj{b_{\bigidx_2,\bigidx_1}}}(w'u,1,\cdot,\cdot),1-s).
\end{align}
The factors
$c(\smallidx,\varepsilon_2,\gamma,s)=c(\smallidx,\tau_2,\gamma,s-\zeta)$
and
$\epsilon(\pi\times\varepsilon_2,\psi,s)=\epsilon(\pi\times\tau_2,\psi,s-\zeta)$
belong to $\C[q^{-(s-\zeta)},q^{(s-\zeta)}]^*$ and may be ignored (actually since we assume $\smallidx\leq\bigidx_2$,
$c(\smallidx,\varepsilon_2,\gamma,s)=1$). Note that the
integrals $\Psi(\cdot,\cdot,s)$ for $\pi\times\varepsilon_2$ are
defined in some right half-plane $\setof{s\in\C}{\Re(s-\zeta)>>0}$ and
$D$ can be taken so that it intersects this right half-plane in a
domain. Thus for any $s\in D$
we can replace the meromorphic continuation of $\Psi(W,\varphi_s^{\rconj{b_{\bigidx_2,\bigidx_1}}}(w'u,1,\cdot,\cdot),s)$ on the \lhs\ of
\eqref{eq:n_1<l<=n_2 func pi and tau_2 for single u} with the integral $\Psi(W,\varphi_s^{\rconj{b_{\bigidx_2,\bigidx_1}}}(w'u,1,\cdot,\cdot),s)$.

Let
$\widetilde{Q}=P_1^{\zeta}(q^{-s})\gcd(\pi\times\varepsilon_2^*,1-s)^{-1}\in\C[q^{-s},q^s]$.
Since
$\widetilde{Q}(q^{-s},q^s)\nintertwiningfull{\varepsilon_2}{s}$ is a
holomorphic operator, as explained in
Section~\ref{subsection:Interpretation of functional equations} we
can multiply both sides of \eqref{eq:n_1<l<=n_2 func pi and tau_2 for single u} by $P_1^{\zeta}$ and obtain
an equality in $\Sigma(X)$, 
\begin{align}\label{eq:n_1<l<=n_2 func pi and tau_2 for single u as series pre}
Q(X)\Sigma(W,\varphi_s^{\rconj{b_{\bigidx_2,\bigidx_1}}}(w'u,1,\cdot,\cdot),s)=\Sigma(W,\widetilde{Q}(q^{-s},q^s)\nintertwiningfull{\varepsilon_2}{s}
\varphi_s^{\rconj{b_{\bigidx_2,\bigidx_1}}}(w'u,1,\cdot,\cdot),1-s).
\end{align}
Because
\begin{align*}
\nintertwiningfull{\varepsilon_2}{s}\varphi_s^{\rconj{b_{\bigidx_2,\bigidx_1}}}(w'u,1,\cdot,\cdot)
=(\nintertwiningfull{\varepsilon_2}{s}\varphi_s)^{\rconj{b_{\bigidx_2,\bigidx_1}}}(w'u,1,\cdot,\cdot)
\end{align*}
(see Section~\ref{subsection:Twisting the embedding}) and
$\varphi_{s,1-s}=\widetilde{Q}(q^{-s},q^s)\nintertwiningfull{\varepsilon_2}{s}\varphi_s$, we get
\begin{align*}
\widetilde{Q}(q^{-s},q^s)\nintertwiningfull{\varepsilon_2}{s}
\varphi_s^{\rconj{b_{\bigidx_2,\bigidx_1}}}(w'u,1,\cdot,\cdot)=
\varphi_{s,1-s}^{\rconj{b_{\bigidx_2,\bigidx_1}}}(w'u,1,\cdot,\cdot).
\end{align*}
Putting this into \eqref{eq:n_1<l<=n_2 func pi and tau_2 for single u as series pre} gives
\begin{align}\label{eq:n_1<l<=n_2 func pi and tau_2 for single u as series}
Q(X)\Sigma(W,\varphi_s^{\rconj{b_{\bigidx_2,\bigidx_1}}}(w'u,1,\cdot,\cdot),s)=\Sigma(W,\varphi_{s,1-s}^{\rconj{b_{\bigidx_2,\bigidx_1}}}(w'u,1,\cdot,\cdot),1-s).
\end{align}
The \lhs\ of \eqref{eq:n_1<l<=n_2 func pi and tau_2 for
single u as series} equals $\psi_{\gamma}^{-1}(u)\Sigma_{Q\phi(u,\cdot)}$
and the \rhs\ represents
the integral $\Psi(W,\varphi_{s,1-s}^{\rconj{b_{\bigidx_2,\bigidx_1}}}(w'u,1,\cdot,\cdot),1-s)$,
which is a polynomial.

Next, in $D^*$ equality~\eqref{int:mult n_1<l<=n_2 before applying eq for inner tau}
is applicable with $\varphi_{s,1-s}$ replacing $\varphi_{s}$. Define
$\phi^*$ similarly to $\phi$ but with $\varphi_{s,1-s}$ instead of
$\varphi_s$. Then as above Lemma~\ref{lemma:integral for
holomorphic section representable as a series} shows
$\Sigma_{\phi^*}\sim_{D^*}\Phi_{\phi^*}$ and
$\Sigma(W,\varphi_{s,1-s},(s,1-s))=\Sigma_{\phi^*}$. The domain
$D^*$ is taken so that it intersects the left half-plane
$\setof{s\in\C}{\Re(1-s+\zeta)>>0}$ where the integrals
$\Psi(\cdot,\cdot,1-s)$ for $\pi\times\varepsilon_2^*$ on the \rhs\
of \eqref{eq:n_1<l<=n_2 func pi and tau_2 for single u} are defined
(this means taking $\Re(\zeta)$ large enough). Then for $s\in D^*$, we can replace the meromorphic continuation of
$\Psi(W,\nintertwiningfull{\varepsilon_2}{s}\varphi_s^{\rconj{b_{\bigidx_2,\bigidx_1}}}(w'u,1,\cdot,\cdot),1-s)$ on the \rhs\ of
\eqref{eq:n_1<l<=n_2 func pi and tau_2 for single u} with the integral $\Psi(W,\nintertwiningfull{\varepsilon_2}{s}\varphi_s^{\rconj{b_{\bigidx_2,\bigidx_1}}}(w'u,1,\cdot,\cdot),1-s)$.

As above,
\begin{align*}
&\Sigma_{\phi^*(u,\cdot)}\sim_{D^*}\Phi_{\phi^*(u,\cdot)}=\psi_{\gamma}(u)\Psi(W,\varphi_{s,1-s}^{\rconj{b_{\bigidx_2,\bigidx_1}}}(w'u,1,\cdot,\cdot),1-s),\\\notag
&\Sigma_{\phi^*(u,\cdot)}=\psi_{\gamma}(u)\Sigma(W,\varphi_{s,1-s}^{\rconj{b_{\bigidx_2,\bigidx_1}}}(w'u,1,\cdot,\cdot),1-s)\in R(X).
\end{align*}
Then Lemma~\ref{lemma:double iterated series} shows $\Sigma_{\phi^*}=\Sigma_{(\phi^*)_{\Gamma'}}$. Now
\eqref{eq:n_1<l<=n_2 func pi and tau_2 for single u as series}
implies $\Sigma_{Q\phi(u,\cdot)}=\Sigma_{\phi^*(u,\cdot)}$, for all $u$.
Hence $(Q\phi)_{\Gamma'}=(\phi^*)_{\Gamma'}$. Putting the pieces
together,
\begin{equation*}
Q(X)\Sigma(W,\varphi_s,s)=\Sigma_{Q\phi}=\Sigma_{(Q\phi)_{\Gamma'}}=
\Sigma_{(\phi^*)_{\Gamma'}}=\Sigma_{\phi^*}=\Sigma(W,\varphi_{s,1-s},(s,1-s)).\qedhere
\end{equation*}
\end{proof} 
We continue with the functional equation for $\varepsilon_1\times\varepsilon_2$.
In the notation of Section~\ref{subsection:the multiplicativity of the intertwining operator for tau induced},
for any $h\in H_{\bigidx}$,
\begin{align*}
h\cdot\varphi_{s,1-s}(\cdot,\cdot,I_{2\bigidx_2+1},\cdot)\in
\xi_{P_{\bigidx_1,\bigidx_2}}^{\GL{\bigidx}}(\varepsilon_1\absdet{}^{\frac{\bigidx}2}\otimes
\varepsilon_2^*\absdet{}^{\frac{\bigidx}2},hol,(s,1-s)).
\end{align*}
Let
\begin{align*}
\varphi_{1-s,s}'=P_2^{\zeta}(q^{-s})\nintertwiningfull{\varepsilon_1\otimes\varepsilon_2^*}{(s,1-s)}\varphi_{s,1-s}\in
\xi_{Q_{\bigidx_2,\bigidx_1}}^{H_{\bigidx}}(\varepsilon_2^*\otimes\varepsilon_1,hol,(1-s,s)).
\end{align*}
The integral $\Psi(W,\varphi_{1-s,s}',(1-s,s))$ is absolutely
convergent in a domain $D^{**}$ of the form
$\setof{\zeta,s\in\C}{-\Re(\zeta)<<\Re(s)<<\Re(\zeta),\Re(s)\leq A_2}$, for some constant $A_2$.
\begin{claim}\label{claim:n_1<l<=n_2 func tau_1 and tau_2}
$P_2^{\zeta}(X)\Sigma(W,\varphi_{s,1-s},(s,1-s))=\Sigma(W,\varphi_{1-s,s}',(1-s,s))$.
\end{claim}
\begin{proof}[Proof of Claim~\ref{claim:n_1<l<=n_2 func tau_1 and tau_2}] 
The proof is similar to the proof of Claim~\ref{claim:n_1<l<=n_2
func pi and tau_2} and described briefly. For instance assume $\smallidx\leq\bigidx$ (the proof when $\smallidx>\bigidx$ is almost identical). Replace the $dg$-integration in $\Psi(W,\varphi_{s,1-s},(s,1-s))$ with an integration over
$\overline{B_{G_{\smallidx}}}$. Then $\Psi(W,\varphi_{s,1-s},(s,1-s))$ equals
\begin{align*}
\int_{\overline{B_{G_{\smallidx}}}}\int_{R_{\smallidx,\bigidx}}\int_{
Z_{\bigidx_2,\bigidx_1}}W(g)
(w_{\smallidx,\bigidx}rg)\cdot\varphi_{s,1-s}(\omega_{\bigidx_1,\bigidx_2}z,I_{\bigidx_1},I_{2\bigidx_2+1},I_{\bigidx_2})\psi^{-1}(z)\psi_{\gamma}(r)dzdrdg.
\end{align*}
Let $\Gamma=\overline{B_{G_{\smallidx}}}\times
R_{\smallidx,\bigidx}$, $\Gamma'=Z_{\bigidx_2,\bigidx_1}$ and let $D^*$ be as in Claim~\ref{claim:n_1<l<=n_2 func pi and tau_2}. Define $\phi\in C^{\infty}(\Gamma\times\Gamma',\C[q^{-s},q^s])$ by
\begin{align*}
\phi((g,r),z)=W(g)(w_{\smallidx,\bigidx}rg)\cdot
\varphi_{s,1-s}(\omega_{\bigidx_1,\bigidx_2}z,1,1,1)
\psi^{-1}(z)\psi_{\gamma}(r).
\end{align*}
Then
\begin{align*}
\int_{\Gamma'}\phi((g,r),z)dz=W(g)\psi_{\gamma}(r)\whittakerfunctional((w_{\smallidx,\bigidx}rg)\cdot\varphi_{s,1-s}(\cdot,\cdot,I_{2\bigidx_2+1},\cdot)),
\end{align*}
where $\whittakerfunctional$ denotes the Whittaker functional given
by the \lhs\ of \eqref{eq:shahidi func eq} on the space
\begin{align*}
V_{P_{\bigidx_1,\bigidx_2}}^{\GL{\bigidx}}(\varepsilon_1\absdet{}^{\frac{\bigidx}2}\otimes\varepsilon_2^*\absdet{}^{\frac{\bigidx}2},(s,1-s)).
\end{align*}
According to the functional equation \eqref{eq:shahidi func eq}, 
\begin{align}\label{eq:n_1<l<=n_2 shahidi func}
&\whittakerfunctional(
(w_{\smallidx,\bigidx}rg)\cdot\varphi_{s,1-s}(\cdot,\cdot,I_{2\bigidx_2+1},\cdot))\\\notag&=
\whittakerfunctional^*(\nintertwiningfull{\varepsilon_1\otimes\varepsilon_2^*}{(s,1-s)}
(w_{\smallidx,\bigidx}rg)\cdot\varphi_{s,1-s}(\cdot,\cdot,I_{2\bigidx_1+1},\cdot)).
\end{align}
Here $\whittakerfunctional^*$ denotes the integral on the \rhs\ of \eqref{eq:shahidi func eq}, which comprises a Whittaker functional on
\begin{align*}
V_{P_{\bigidx_2,\bigidx_1}}^{\GL{\bigidx}}(\varepsilon_2^*\absdet{}^{\frac{\bigidx}2}\otimes\varepsilon_1\absdet{}^{\frac{\bigidx}2},(1-s,s)).
\end{align*}
Equality~\eqref{eq:n_1<l<=n_2 shahidi func} is a priori between meromorphic continuations in $\C(q^{-s})$, but because
\begin{align*}
\varphi_{s,1-s}(\cdot,\cdot,I_{2\bigidx_2+1},\cdot)\in\xi_{P_{\bigidx_1,\bigidx_2}}^{\GL{\bigidx}}(\varepsilon_1\absdet{}^{\frac{\bigidx}2}\otimes
\varepsilon_2^*\absdet{}^{\frac{\bigidx}2},hol,(s,1-s)),
\end{align*}
it is actually in $\C[q^{-s},q^s]$. In addition there are constants $s_1,s_2>0$, independent of $\zeta$, such that the integral on the \lhs\ (resp. \rhs) of \eqref{eq:n_1<l<=n_2 shahidi func} is absolutely convergent for $\Re(s)>s_1$ (resp. $\Re(s)<-s_2$). By taking $\Re(\zeta)$ large enough, $D^*$ (resp. $D^{**}$) will intersect the domain $\setof{s\in\C}{\Re(s)>s_1}$ (resp. $\setof{s\in\C}{\Re(s)<-s_2}$) in a domain.

We define
$\Gamma''=Z_{\bigidx_1,\bigidx_2}$ and $\phi^*\in
C^{\infty}(\Gamma\times\Gamma'',\C[q^{-s},q^s])$ by
\begin{align*}
\phi^*((g,r),z)=W(g)(w_{\smallidx,\bigidx}rg)\cdot
\varphi_{1-s,s}'(\omega_{\bigidx_2,\bigidx_1}z,1,1,1)
\psi^{-1}(z)\psi_{\gamma}(r).
\end{align*}
In $D^{**}$, for all $(g,r)\in\Gamma$,
\begin{align*}
&\int_{\Gamma'}\phi^*((g,r),z)dz\\\notag&=P_2^{\zeta}(q^{-s})W(g)\psi_{\gamma}(r)\whittakerfunctional^*(\nintertwiningfull{\varepsilon_1\otimes\varepsilon_2^*}{(s,1-s)}(w_{\smallidx,\bigidx}rg)\cdot\varphi_{s,1-s}(\cdot,\cdot,I_{2\bigidx_1+1},\cdot)).
\end{align*}
Proceeding as in Claim~\ref{claim:n_1<l<=n_2 func pi and tau_2}, we
use Lemmas~\ref{lemma:interchange integration and polynomial},
\ref{lemma:double iterated series} and \ref{lemma:integral for
holomorphic section representable as a series} to conclude
\begin{align*}
P_2^{\zeta}(X)\Sigma(W,\varphi_{s,1-s},(s,1-s))&=\Sigma_{P_2^{\zeta}\phi}=\Sigma_{(P_2^{\zeta}\phi)_{\Gamma'}}\\\notag&=\Sigma_{(\phi^*)_{\Gamma''}}=\Sigma_{\phi^*}=\Sigma(W,\varphi_{1-s,s}',(1-s,s)).
\end{align*}
The equality $\Sigma_{(P_2^{\zeta}\phi)_{\Gamma'}}=\Sigma_{(\phi^*)_{\Gamma''}}$ holds because according to
\eqref{eq:n_1<l<=n_2
shahidi func}, \\$(P_2^{\zeta}\phi)_{\Gamma'}=(\phi^*)_{\Gamma''}$.
\end{proof} 
Next we apply the functional equation for $\pi\times\varepsilon_1$. 
Define
\begin{align*}
\varphi_{1-s}^*&=
P_3^{\zeta}(q^{-s})\gcd(\pi\times\varepsilon_1^*,1-s)^{-1}\nintertwiningfull{\varepsilon_1}{s}\varphi_{1-s,s}'\\\notag&\in\xi_{Q_{\bigidx_2,\bigidx_1}}^{H_{\bigidx}}(\varepsilon_2^*\otimes\varepsilon_1^*,hol,(1-s,1-s)).
\end{align*}
According to \eqref{eq:multiplicative property of intertwiners} (see the paragraph after the proof of
Claim~\ref{claim:verification of multiplicativity intertwiners}) we obtain
\begin{align*}
\varphi_{1-s}^*=
P_1^{\zeta}(q^{-s})P_2^{\zeta}(q^{-s})P_3^{\zeta}(q^{-s})\gcd(\pi\times\varepsilon_1^*,1-s)^{-1}
\gcd(\pi\times\varepsilon_2^*,1-s)^{-1}\nintertwiningfull{\varepsilon}{s}\varphi_s.
\end{align*}
The domain $D^{***}$ of absolute convergence of
$\Psi(W,\varphi_{1-s}^*,1-s)$ takes the form $\setof{\zeta,s\in\C}{\Re(1-s)>>\Re(\zeta)>>0}$.
\begin{claim}\label{claim:n_1<l<=n_2 func pi and tau_1}
$P_3^{\zeta}(X)G(\pi\times\varepsilon_1,X)^{-1}\Sigma(W,\varphi_{1-s,s}',(1-s,s))
=\Sigma(W,\varphi_{1-s}^*,1-s)$.
\end{claim}
\begin{proof}[Proof of Claim~\ref{claim:n_1<l<=n_2 func pi and tau_1}] 
The proof is analogous to the proof of Claim~\ref{claim:n_1<l<=n_2
func pi and tau_2}. Three versions are needed, for
$\smallidx\leq\bigidx_1$, $\bigidx_1<\smallidx\leq\bigidx$ and
$\smallidx>\bigidx$. Here we address the case
$\bigidx_1<\smallidx\leq\bigidx$ (together with Claim~\ref{claim:n_1<l<=n_2
func pi and tau_2} this essentially covers all cases except for $\smallidx>\bigidx$,
see Remark~\ref{remark:claims dont prove everything} below). According to the transitions
\eqref{int:mult l<n_2 after shahidi equation}-\eqref{int:mult var 2
before applying eq for outer tau}, in $D^{**}$ the integral
$\Psi(W,\varphi_{1-s,s}',(1-s,s))$ equals
\begin{align}\label{int:mult var 2 before applying eq for outer tau formula in }
&\int_{\lmodulo{H_{\bigidx_1}Z_{\smallidx-\bigidx_1-1}V_{\smallidx-\bigidx_1-1}}{G_{\smallidx}}}
\int_{R_{\smallidx,\bigidx}}\int_{\Mat{\bigidx_1\times\bigidx-\smallidx}}
\int_{\lmodulo{U_{H_{\bigidx_1}}}{H_{\bigidx_1}}}\\\notag
&(\int_{R^{\smallidx,\bigidx_1}}W(r'w^{\smallidx,\bigidx_1}h'g)dr')\varphi_{1-s,s}'(\omega_{\bigidx_2,\bigidx_1}\eta
w_{\smallidx,\bigidx}rw^{\smallidx,\bigidx_1}g,1,\rconj{t_{\bigidx_1}}h',1)\psi_{\gamma}(r)dh'd\eta
drdg.
\end{align}
Here $t_{\bigidx_1}=diag((-1)^{\bigidx-\smallidx}\gamma^{-1}I_{\bigidx_1},1,(-1)^{\bigidx-\smallidx}\gamma I_{\bigidx_1})$. We write the $dg$-integral over
$\overline{B_{\GL{\smallidx-\bigidx_1-1}}}\times\overline{V_{\smallidx-\bigidx_1-1}}\times
(\lmodulo{H_{\bigidx_1}}{G_{\bigidx_1+1}})$, where $\overline{B_{\GL{\smallidx-\bigidx_1-1}}},G_{\bigidx_1+1}<L_{\smallidx-\bigidx_1-1}$. Then we use Lemma~\ref{lemma:integration formula for quotient space H_n
G_n+1} to write the integral over $\lmodulo{H_{\bigidx_1}}{G_{\bigidx_1+1}}$. In the notation of the lemma,
in the split case,
\begin{align*}
\Gamma&=\overline{B_{\GL{\smallidx-\bigidx_1-1}}}\times\overline{V_{\smallidx-\bigidx_1-1}}\times
\GL{1}\times\overline{Z_{\bigidx_1,1}}\times \Xi_{\bigidx_1}\times
R_{\smallidx,\bigidx}\times\Mat{\bigidx_1\times\bigidx-\smallidx}.
\end{align*}
If $G_{\smallidx}$ is \quasisplit,
\begin{align*}
\Gamma&=\overline{B_{\GL{\smallidx-\bigidx_1-1}}}\times\overline{V_{\smallidx-\bigidx_1-1}}
\times\GL{1}\times\overline{Z_{\bigidx_1-1,1}}\times(V_{\bigidx_1}\cap G_2)\times \Xi_{\bigidx_1}\times
R_{\smallidx,\bigidx}\times\Mat{\bigidx_1\times\bigidx-\smallidx}.
\end{align*}
In both cases let $\Gamma'=\overline{B_{H_{\bigidx_1}}}\times
R^{\smallidx,\bigidx_1}$. We define $\phi\in
C^{\infty}(\Gamma\times\Gamma',\C[q^{-s},q^s])$ using the integrand of \eqref{int:mult var 2 before applying eq for outer tau formula in }.
Then $\Sigma(W,\varphi_{1-s,s}',(1-s,s))=\Sigma_{\phi}$.

We make the following remark regarding the usage of $\Xi_{\bigidx_1}$. In the split case, it is a closed subgroup of $G_{\bigidx_1+1}$ and in
particular, an $l$-space. In the \quasisplit\ case $\Xi_{\bigidx_1}$ is homeomorphic to
$F^{\bigidx_1-1}$ and the measure is defined using the measure of $F^{\bigidx_1-1}$, hence it can also be regarded as
an $l$-space with a Borelian measure. 

By virtue of the functional equation for $G_{\smallidx}\times \GL{\bigidx_1}$
and $\pi\times\varepsilon_1$, and equality~\eqref{eq:intertwiner and left and right translation}, for any $g\in G_{\smallidx}$ and $h\in\ H_{\bigidx}$,
\begin{align*}
&\gcd(\pi\times\varepsilon_1,s)^{-1}\Psi(g\cdot
W,(\varphi_{1-s,s}')^{t_{\bigidx_1}}(h,1,\cdot,\cdot),s)\\\notag&=
\epsilon(\pi\times\varepsilon_1,\psi,s)^{-1}\gcd(\pi\times\varepsilon_1^*,1-s)^{-1}\Psi(g\cdot
W,(\nintertwiningfull{\varepsilon_1}{s}\varphi_{1-s,s}')^{t_{\bigidx_1}}(h,1,\cdot,\cdot),1-s).
\end{align*}
The factor $\epsilon(\pi\times\varepsilon_1,\psi,s)\in \C[q^{-s-\zeta},q^{s+\zeta}]^*$ can be ignored.

Let
\begin{align*}
&Q=P_3^{\zeta}(q^{-s})\gcd(\pi\times\varepsilon_1,s)^{-1}\in\C[q^{-s}],\\\notag
&\widetilde{Q}=P_3^{\zeta}(q^{-s})\gcd(\pi\times\varepsilon_1^*,1-s)^{-1}\in\C[q^{-s},q^s].
\end{align*}
Formula \eqref{int:mult var 2 before applying eq for outer tau formula in } stated above
for $\Psi(W,\varphi_{1-s,s}',(1-s,s))$ holds in $D^{***}$ for $\Psi(W,\varphi_{1-s}^*,1-s)$, as explained in Section~\ref{subsection:n_1<l<=n_2} after
\eqref{int:mult var 2 after applying eq for outer tau}. Define $\phi^*\in C^{\infty}(\Gamma\times\Gamma',\C[q^{-s},q^s])$ similarly to $\phi$, with $\varphi_{1-s}^*$ instead of $\varphi_{1-s,s}'$. Then as in Claim~\ref{claim:n_1<l<=n_2 func pi and tau_2},
\begin{equation*}
Q(X)\Sigma(W,\varphi_{1-s,s}',(1-s,s))
=\Sigma_{(Q\phi)_{\Gamma'}}=
\Sigma_{(\phi^*)_{\Gamma'}}
=\Sigma(W,\varphi_{1-s}^*,1-s).\qedhere
\end{equation*}
\end{proof} 

\begin{remark}\label{remark:claims dont prove everything}
In the proofs of Claims~\ref{claim:n_1<l<=n_2 func pi and tau_2} and \ref{claim:n_1<l<=n_2 func pi and tau_1} we translated the integral manipulations of Section~\ref{subsection:n_1<l<=n_2} into transitions between Laurent series. As we have seen in Section~\ref{section:2nd variable}, a combination of these manipulations is used for all cases except that of $\smallidx>\bigidx$. The manipulations in the latter case, given in Section~\ref{subsection:n<l}, are simpler.
\end{remark}

Collecting Claims~\ref{claim:n_1<l<=n_2 func pi and
tau_2}, \ref{claim:n_1<l<=n_2 func tau_1 and tau_2} and
\ref{claim:n_1<l<=n_2 func pi and tau_1},
\begin{align}\label{eq:n_1<l<=n_2 claims combined}
P_1^{\zeta}(X)P_2^{\zeta}(X)P_3^{\zeta}(X)G(\pi\times\varepsilon_1,X)^{-1}G(\pi\times\varepsilon_2,X)^{-1}\Sigma(W,\varphi_{s},s)=\Sigma(W,\varphi_{1-s}^*,1-s).
\end{align}
The final step is to put $\zeta=0$ into this equality. Note that in
general such a substitution in a series $\sum_{m\in\Integers}a_m(\zeta)X^m$ is not
defined, since it is not clear how each coefficient $a_m(\zeta)$ depends on
$\zeta$.

Let
\begin{align*}
f_{1-s}^*=P_{\pi\times\tau}(q^{-s})\gcd(\pi\times\tau_1^*,1-s)^{-1}\gcd(\pi\times\tau_2^*,1-s)^{-1}\nintertwiningfull{\tau}{s}f_s.
\end{align*}
The definition of $P_{\pi\times\tau}$ and \eqref{eq:multiplicative property of poles of intertwiners} imply
$f_{1-s}^*\in\xi_{Q_{\bigidx}}^{H_{\bigidx}}(\tau^*,hol,1-s)$. Let
$\gcd(\pi\times\tau_i,X)^{-1}\in R(X)$ be obtained from
$\gcd(\pi\times\tau_i,s)^{-1}$ by replacing $q^{-s}$ with $X$.
\begin{claim}\label{claim:putting Y=1}
Putting $\zeta=0$ in \eqref{eq:n_1<l<=n_2 claims combined} gives
\begin{align}\label{eq:n_1<l<=n_2 claims combined no zeta}
P_{\pi\times\tau}(X)\gcd(\pi\times\tau_1,X)^{-1}\gcd(\pi\times\tau_2,X)^{-1}\Sigma(W,f_s,s)=\Sigma(W,f_{1-s}^*,1-s).
\end{align}
\end{claim}
\begin{proof}[Proof of Claim~\ref{claim:putting Y=1}] 
Since Claim~\ref{claim:ring homomorphism does not change P_i props}
proves $P_i^{0}=P_i$, it remains to consider 
$\Sigma(W,\varphi_s,s)$ and
$\Sigma(W,\varphi_{1-s}^*,1-s)$. We begin with
$\Sigma(W,\varphi_s,s)$.

The domain $D\subset\C\times\C$ in $\zeta$ and $s$ of absolute
convergence of $\Psi(W,\varphi_s,s)$ contains a domain
$D_0=\setof{\zeta,s\in\C}{\zeta_0\leq\Re(\zeta)\leq\zeta_1,\Re(s)>s_0}$, where $\zeta_0,\zeta_1,s_0$ are constants depending only on $\pi,\tau_1$ and $\tau_2$.
We restrict $\Psi(W,\varphi_s,s)$ to $D_0$. Put
$D_0'=\setof{\zeta\in\C}{\zeta_0\leq\Re(\zeta)\leq\zeta_1}$. The first step is to
show that $\Sigma(W,\varphi_s,s)$ is of the form
$\sum_{m=M}^{\infty}a_m(\zeta)X^m$, where $M$ is independent of
$\zeta$ and for each $m$, $a_m:D_0'\rightarrow\C$ is an analytic
function - a polynomial in $\C[q^{-\zeta},q^{\zeta}]$.

According to Claim~\ref{claim:deducing varphi and f_s are
interchangeable in psi}, in $D_0$ we have
$\Psi(W,\varphi_s,s)=\Psi(W,\widehat{f}_{\varphi_s},s)$ where
$\widehat{f}_{\varphi_s}\in\xi(\Whittaker{\varepsilon}{\psi},hol,s)$ is defined by
\eqref{iso:iso 1} using $\varphi_s$. Since $\varphi_{s}\in\xi(\varepsilon_1\otimes\varepsilon_2,hol,(s,s))$,
we can write $\varphi_{s}=\sum_{i=1}^vP_i\varphi_{s}^{(i)}$ where
$P_i\in\C[q^{-s},q^s]$ and $\varphi_{s}^{(i)}\in\xi(\varepsilon_1\otimes\varepsilon_2,std,(s,s))=
\xi(\tau_1\otimes\tau_2,std,(s+\zeta,s-\zeta))$. Then
$\Psi(W,\widehat{f}_{\varphi_s},s)=\sum_{i=1}^vP_i\Psi(W,\widehat{f}_{\varphi_s^{(i)}},s)$.
Fix $\zeta\in D_0'$ and use
Claim~\ref{claim:iwasawa decomposition of the integral with zeta in tau} 
to write $\Psi(W,\widehat{f}_{\varphi_s^{(i)}},s)$ as a sum of integrals, say if
$\smallidx\leq\bigidx$, of the form \eqref{int:iwasawa decomposition
of the integral l<=n split}. Here $W'\in\Whittaker{\tau}{\psi}$ is
replaced by $W_{\zeta}'\in\Whittaker{\varepsilon}{\psi}$. The
number of negative coefficients is bounded as in
Lemma~\ref{lemma:integral for holomorphic section representable as a
series}, using the fact that $W^{\diamond}\in\Whittaker{\pi}{\psi_{\gamma}^{-1}}$ vanishes away from zero. Now if we
let $\zeta$ vary, this bound remains fixed because $W^{\diamond}$ is
independent of $\zeta$ (it is also possible to use $W_{\zeta}'$ to
obtain a bound independent of $\zeta$). Since the coordinates of $a\in A_{\smallidx-1}$ are bounded from above, if $G_{\smallidx}$ is split
(resp. \quasisplit) for any $m\in\Integers$ the set of $a,x$ (resp.
$a$) such that $\absdet{a}\cdot[x]^{-1}=q^m$ (resp.
$\absdet{a}=q^m$) is compact in $A_{\smallidx-1}\times G_1$ (resp.
$A_{\smallidx-1}$). Then observe that according to
Claim~\ref{claim:iwasawa decomposition of the integral with zeta in tau}, there is a
compact open subgroup $N<\GL{\bigidx}$, independent of $\zeta$, such
that $W_{\zeta}'$ is right-invariant by $N$, and for a fixed $b\in
\GL{\bigidx}$, $W_{\zeta}'(b)\in\C[q^{-\zeta},q^{\zeta}]$. Therefore
the coefficient of $X^m$ is a polynomial
in $q^{\mp\zeta}$. The same applies when $\smallidx>\bigidx$, by
considering \eqref{int:iwasawa decomposition of the integral l>n}.

Let $C\subset\C$ be a compact set containing $0$ and $s_C>0$ be a
constant, depending only on $\pi,\tau_1,\tau_2$ and $C$, such that
in $D_1=\setof{\zeta,s}{\zeta\in C,\Re(s)>s_C}$, $\Psi(W,\widehat{f}_{\varphi_s},s)$
is absolutely convergent and $D_1\cap D_0$ is a domain of
$\C\times\C$. By Lemma~\ref{lemma:integral for holomorphic section
representable as a series}, for any $\zeta\in C$, in the domain
$\setof{s}{\Re(s)>s_C}$, the integral $\Psi(W,\widehat{f}_{\varphi_s},s)$ has a
representation $\Sigma(W,\widehat{f}_{\varphi_s},s)\in\Sigma(X)$. Note that Lemma~\ref{lemma:integral for
holomorphic section representable as a series} is applicable because
$D_1$ does not depend on $\widehat{f}_{\varphi_s}$ or $W$. The computation in the previous paragraph gives
$\Sigma(W,\widehat{f}_{\varphi_s},s)=\sum_{m=M}^{\infty}b_m(\zeta)X^m$
with properties as above.

Since $D_1\cap D_0$ is a domain where both series $\Sigma(W,\varphi_s,s)$ and
$\Sigma(W,\widehat{f}_{\varphi_s},s)$ represent $\Psi(W,\widehat{f}_{\varphi_s},s)$,
$b_m(\zeta)=a_m(\zeta)$ for all $\zeta\in\C$ and $m\geq M$. Additionally for
$\zeta=0$, $\Sigma(W,\widehat{f}_{\varphi_s},s)$ is the series representing
$\Psi(W,f_s,s)$ (because $0\in C$ and for $\zeta=0$, $\widehat{f}_{\varphi_s}=f_s$). 
It follows that putting $\zeta=0$ in $\Sigma(W,\varphi_s,s)$ yields
$\Sigma(W,f_s,s)$.


Regarding $\Sigma(W,\varphi_{1-s}^*,1-s)$, the arguments above
show that when we substitute $0$ for $\zeta$ in $\Sigma(W,\varphi_{1-s}^*,1-s)$ we obtain the series $\Sigma(W,\widehat{f}_{\varphi_{1-s}^*},1-s)$ with $\zeta=0$. Note that by Claim~\ref{claim:intertwining operator and Jacquet integral
commute},
\begin{align*}
\widehat{f}_{\varphi_{1-s}^*}=&(\prod_{i=1}^3P_i^{\zeta}(q^{-s}))\gcd(\pi\times\tau_1^*,1-(s+\zeta))^{-1}\\\notag&\times\gcd(\pi\times\tau_2^*,1-(s-\zeta))^{-1}\nintertwiningfull{\varepsilon}{s}\widehat{f}_{\varphi_s}.
\end{align*}
Equality~\eqref{eq:multiplicative property of intertwiners} and the
determination of the poles in Claim~\ref{claim:ring homomorphism
does not change P_i props} imply that for $\zeta=0$,
$\nintertwiningfull{\varepsilon}{s}=\nintertwiningfull{\tau}{s}$
(see proof of Claim~\ref{claim:gamma is rational in s and zeta}) and $\widehat{f}_{\varphi_{1-s}^*}=f_{1-s}^*$. Hence for $\zeta=0$,
$\Sigma(W,\widehat{f}_{\varphi_{1-s}^*},1-s)=\Sigma(W,f_{1-s}^*,s)$.
\end{proof} 
According to Lemma~\ref{lemma:integral for holomorphic section
representable as a series}, $\Sigma(W,f_s,s)$ (resp.
$\Sigma(W,f_{1-s}^*,1-s)$) is a Laurent series with finitely many
negative (resp. positive) coefficients. Hence both sides of
\eqref{eq:n_1<l<=n_2 claims combined no zeta} are polynomials. Since
\begin{align*}
P_{\pi\times\tau}(q^{-s})\gcd(\pi\times\tau_1,s)^{-1}\gcd(\pi\times\tau_2,s)^{-1}\Psi(W,f_s,s)
\end{align*}
is represented by the \lhs\ of \eqref{eq:n_1<l<=n_2 claims combined
no zeta},
\begin{align*}
\Psi(W,f_s,s)\in\gcd(\pi\times\tau_1,s)\gcd(\pi\times\tau_2,s)P_{\pi\times\tau}(q^{-s})^{-1}\C[q^{-s},q^{s}].
\end{align*}
Additionally $\Psi(W,f_{1-s}^*,1-s)$ is represented by the polynomial
$\Sigma(W,f_{1-s}^*,1-s)$ and according to the definition of $f_{1-s}^*$,
\begin{align*}
\Psi(W,\nintertwiningfull{\tau}{s}&f_s,1-s)\\\notag&\in
\gcd(\pi\times\tau_1^*,1-s)\gcd(\pi\times\tau_2^*,1-s)P_{\pi\times\tau}(q^{-s})^{-1}\C[q^{-s},q^{s}].\qedhere
\end{align*}
\end{proof} 

\section{Upper bound in the first
variable}\label{section:Sub-multiplicativity in the first variable}
\subsection{Outline}\label{subsection:first var outline} In this section we
prove Theorem~\ref{theorem:gcd sub multiplicity first var}. The proof is divided into three parts, according to the relative sizes of $k$, $\smallidx$ and $\bigidx$.
\subsection{The case $k<\smallidx$}\label{subsection:sub mult first var proof body k<l}
Let $\pi=\cinduced{P_k}{G_{\smallidx}}{\sigma\otimes\pi'}$,
$\tau=\cinduced{P_{\bigidx_1,\ldots,\bigidx_a}}{\GL{\bigidx}}{\tau_1\otimes\ldots\otimes\tau_a}$ and assume that $\tau$ is
irreducible (as always, all representations are generic). We prove,
\begin{align*}
\gcd(\pi\times\tau,s)\in
L(\sigma\times\tau,s)(\prod_{i=1}^a\gcd(\pi'\times\tau_i,s))L(\sigma^*\times\tau,s)M_{\tau_1\otimes\ldots\otimes\tau_a}(s)\C[q^{-s},q^s].
\end{align*}
The method of proof is similar to that of Section~\ref{section:Sub-multiplicativity in the second
variable}: the integral manipulations in Section~\ref{section:1st variable}
(the multiplicativity of $\gamma(\pi\times\tau,\psi,s)$ in the first variable)
involving the application of functional equations
for $\sigma\times\tau$, $\pi'\times\tau$ and $\sigma\times\tau^*$,
are rephrased in terms of Laurent series. 
The proof is reduced to the following lemma.
\begin{lemma}\label{lemma:sub multiplicativity first var k<l<n}
Assume $k<\smallidx\leq\bigidx$. Let $P_{\tau},P_{\tau^*}\in\C[X]$ be normalized by
$P_{\tau}(0)=P_{\tau^*}(0)=1$ and of minimal degree such that the operators
\begin{align*}
&P_{\tau}(q^{-s})\gcd(\pi'\times\tau^*,1-s)^{-1}\nintertwiningfull{\tau}{s},\qquad
P_{\tau^*}(q^{s-1})\gcd(\pi'\times\tau,s)^{-1}\nintertwiningfull{\tau^*}{1-s}
\end{align*}
are holomorphic. Let $Q\in \C[q^{-s},q^s]$ be a least common multiple of $P_{\tau}(q^{-s})$ and $P_{\tau^*}(q^{s-1})$ (note that $Q^{-1}\in
M_{\tau}(s)\C[q^{-s},q^s]$). Then
\begin{align*}
\gcd(\pi\times\tau,s)\in
L(\sigma\times\tau,s)\gcd(\pi'\times\tau,s)L(\sigma^*\times\tau,s)Q(q^{-s},q^s)^{-1}\C[q^{-s},q^s].
\end{align*}
In particular Theorem~\ref{theorem:gcd sub multiplicity first var}
holds for $a=1$ and $k<\smallidx\leq\bigidx$.
\end{lemma}

Before proving the lemma we use it to establish the theorem for
$k<\smallidx$ and general $\bigidx$ and $a$. Put
$G_{\pi'\times(\tau_1\otimes\ldots\otimes\tau_a)}(s)=\prod_{i=1}^a\gcd(\pi'\times\tau_i,s)$.
First assume
$\smallidx\leq\bigidx$. 
By Theorem~\ref{theorem:gcd sub multiplicity second var} applied to
$\pi'\times\tau$ and Claim~\ref{claim:extra P factor in sub multiplicativity is mutual}, for some $P_0\in\C[q^{-s},q^s]$,
\begin{align}\label{eq:starting point pi prime}
&\gcd(\pi'\times\tau,s)=G_{\pi'\times(\tau_1\otimes\ldots\otimes\tau_a)}(s)M_{\tau_1\otimes\ldots\otimes\tau_a}(s)P_0,\\\notag
&\gcd(\pi'\times\tau^*,1-s)\equalun G_{\pi'\times(\tau_1^*\otimes\ldots\otimes\tau_a^*)}(1-s)M_{\tau_1\otimes\ldots\otimes\tau_a}(s)P_0.
\end{align}
Since $M_{\tau}(s)\in M_{\tau_1\otimes\ldots\otimes\tau_a}(s)\C[q^{-s},q^s]$ (see Section~\ref{subsection:the factors M tau s}),
we get that the operators
\begin{align*}
P_0\gcd(\pi'\times\tau^*,1-s)^{-1}\nintertwiningfull{\tau}{s},\qquad
P_0\gcd(\pi'\times\tau,s)^{-1}\nintertwiningfull{\tau^*}{1-s}
\end{align*}
are holomorphic. Applying Lemma~\ref{lemma:sub multiplicativity
first var k<l<n} to $\pi\times\tau$, the polynomials $P_{\tau}(q^{-s})$ and $P_{\tau^*}(q^{s-1})$ divide
$P_0$ (in $\C[q^{-s},q^s]$), hence the polynomial $Q$ of the lemma also
divides $P_0$. Therefore
\begin{align*}
\gcd(\pi\times\tau,s) &\in
L(\sigma\times\tau,s)\gcd(\pi'\times\tau,s)
L(\sigma^*\times\tau,s)P_0^{-1}\C[q^{-s},q^s]\\\notag&=L(\sigma\times\tau,s)G_{\pi'\times(\tau_1\otimes\ldots\otimes\tau_a)}(s)
L(\sigma^*\times\tau,s)M_{\tau_1\otimes\ldots\otimes\tau_a}(s)\C[q^{-s},q^s],
\end{align*}
so the result holds.

Now assume $\smallidx>\bigidx$. Let $\eta$ be a unitary irreducible
supercuspidal representation of $\GL{m}$ for $m>\smallidx$ chosen by
Corollary~\ref{corollary:auxiliary rep to embed all sections in gcd
second var l<=n} (see also Example~\ref{example:choosing tau so that
twist has no poles}). Then if
$\varepsilon=\cinduced{P_{m,\bigidx}}{\GL{m+\bigidx}}{\eta\otimes\tau}$
(which is irreducible),
\begin{align}\label{eq:multi first var embedding result k<l l>=n}
\gcd(\pi\times\tau,s)\in\gcd(\pi\times\varepsilon,s)\C[q^{-s},q^s].
\end{align}
According to Proposition~\ref{proposition:gcd for ramified twist of
tau} and since by Corollary~\ref{corollary:auxiliary rep to embed
all sections in gcd second var l<=n}, $M_{\eta}(s)=1$, we may select
$\eta$ which also satisfies
\begin{align*}
\gcd(\pi'\times\eta,s)=\gcd(\pi'\times\eta^*,1-s)=1.
\end{align*}
Corollary~\ref{corollary:auxiliary rep to embed all sections in gcd
second var l<=n} guarantees that
$\ell_{\eta\otimes\tau^*}(s)=\ell_{\tau^*\otimes\eta}(1-s)=1$ and because
\begin{align*}
M_{\varepsilon}(s)\in M_{\eta\otimes\tau}(s)\C[q^{-s},q^s]=M_{\eta}(s)\ell_{\eta\otimes\tau^*}(s)\ell_{\tau^*\otimes\eta}(1-s)M_{\tau}(s)\C[q^{-s},q^s],
\end{align*}
we get
\begin{align*}
M_{\varepsilon}(s)\in
M_{\tau}(s)\C[q^{-s},q^s]\subset M_{\tau_1\otimes\ldots\otimes\tau_a}(s)\C[q^{-s},q^s].
\end{align*}
The proof of the corollary also shows
$M_{\eta\otimes\tau_1\otimes\ldots\otimes\tau_a}(s)=M_{\tau_1\otimes\ldots\otimes\tau_a}(s)$.
Thus according to Theorem~\ref{theorem:gcd sub multiplicity second
var} applied to $\pi'\times\varepsilon$ and by our choice of $\eta$,
for some $Q_0\in \C[q^{-s},q^s]$,
\begin{align}\label{eq:gcd sub mult first var k<l l>=n pi' and P_0}
\gcd(\pi'\times\varepsilon,s)=G_{\pi'\times(\tau_1\otimes\ldots\otimes\tau_a)}(s)M_{\tau_1\otimes\ldots\otimes\tau_a}(s)Q_0.
\end{align}
Similar to the proof of Claim~\ref{claim:extra P factor in sub
multiplicativity is mutual},
\begin{align*}
\gcd(\pi'\times\varepsilon^*,1-s)\equalun G_{\pi'\times(\tau_1^*\otimes\ldots\otimes\tau_a^*)}(1-s)M_{\tau_1\otimes\ldots\otimes\tau_a}(s)Q_0.
\end{align*}
Therefore 
the operators
\begin{align}\label{operators:holomorphic inside proof upper bound k<l}
Q_0\gcd(\pi'\times\varepsilon^*,1-s)^{-1}\nintertwiningfull{\varepsilon}{s},\qquad
Q_0\gcd(\pi'\times\varepsilon,s)^{-1}\nintertwiningfull{\varepsilon^*}{1-s}
\end{align}
are holomorphic. Since $k<\smallidx<m+\bigidx$, Lemma~\ref{lemma:sub
multiplicativity first var k<l<n} is applicable to
$\pi\times\varepsilon$ and $Q$ of the lemma divides $Q_0$. Hence
\begin{align*}
\gcd(\pi\times\varepsilon,s) \in
L(\sigma\times\varepsilon,s)\gcd(\pi'\times\varepsilon,s)
L(\sigma^*\times\varepsilon,s)Q_0^{-1}\C[q^{-s},q^s].
\end{align*}
Since also $k<m$ and $\eta$ is irreducible supercuspidal,
Theorem~\ref{theorem:tau cuspidal n large L factor of glm gln 1} shows
\begin{align*}
L(\sigma\times\eta,s)=L(\sigma^*\times\eta,s)=1.
\end{align*}
According to \cite{JPSS} (Theorem~3.1),
\begin{align*}
L(\sigma\times\varepsilon,s)\in L(\sigma\times\eta,s)L(\sigma\times\tau,s)\C[q^{-s},q^s].
\end{align*}
Therefore
\begin{align*}
\gcd(\pi\times\varepsilon,s) \in
L(\sigma\times\tau,s)\gcd(\pi'\times\varepsilon,s)
L(\sigma^*\times\tau,s)Q_0^{-1}\C[q^{-s},q^s].
\end{align*}
Now using \eqref{eq:multi first var embedding result k<l l>=n} and
\eqref{eq:gcd sub mult first var k<l l>=n pi' and P_0} we complete
the result,
\begin{align*}
\gcd(\pi\times\tau,s) \in
L(\sigma\times\tau,s)G_{\pi'\times(\tau_1\otimes\ldots\otimes\tau_a)}(s)L(\sigma^*\times\tau,s)M_{\tau_1\otimes\ldots\otimes\tau_a}(s)
\C[q^{-s},q^s].
\end{align*}

We summarize a corollary of the proof.
\begin{corollary}\label{corollary:first var k<l<n exact sub mult}
Let $\pi=\cinduced{P_k}{G_{\smallidx}}{\sigma\otimes\pi'}$
($k<\smallidx$) and write
\begin{align}\label{eq:corollary equality for pi prime given in terms of L}
&\gcd(\pi'\times\tau,s)=L\ell_{\tau^*}(1-s)P_0,\qquad
\gcd(\pi'\times\tau^*,1-s)=\widetilde{L}\ell_{\tau}(s)\widetilde{P}_0,
\end{align}
for some polynomials $L^{-1},P_0,\widetilde{L}^{-1},\widetilde{P}_0\in\C[q^{-s},q^s]$. Let $lcm(P_0,\widetilde{P}_0)$ denote a least common multiple of $P_0$ and $\widetilde{P}_0$.
Then
\begin{align*}
\gcd(\pi\times\tau,s)\in L(\sigma\times\tau,s)\gcd(\pi'\times\tau,s)L(\sigma^*\times\tau,s)lcm(P_0,\widetilde{P}_0)^{-1}\C[q^{-s},q^s].
\end{align*}
In particular if $P_0=\widetilde{P}_0$,
\begin{align*}
\gcd(\pi\times\tau,s)\in L(\sigma\times\tau,s)L(\sigma^*\times\tau,s)L\ell_{\tau^*}(1-s)\C[q^{-s},q^s].
\end{align*}
\end{corollary}
\begin{proof}[Proof of Corollary~\ref{corollary:first var k<l<n exact
sub mult}] 
For $\smallidx\leq\bigidx$, the result is an immediate consequence of Lemma~\ref{lemma:sub
multiplicativity first var k<l<n}. Specifically, \eqref{eq:corollary equality for pi prime given in terms of L}
implies that $P_{\tau^*}(q^{s-1})$ divides $P_0$ and $P_{\tau}(q^{-s})$ divides $\widetilde{P}_0$, whence the polynomial $Q$ of the lemma
divides $lcm(P_0,\widetilde{P}_0)$.

For the case $\smallidx>\bigidx$, let
$\varepsilon=\cinduced{P_{m,\bigidx}}{\GL{m+\bigidx}}{\eta\otimes\tau}$
be as above. Since $M_{\eta}(s)=1$, we may apply
Corollary~\ref{corollary:refinement gcd second var bound} to $\pi'\times \varepsilon$ with
$A_1=\widetilde{A}_1=1$, $A_2(q^{-s})$ which divides $\widetilde{P}_0$ and
$\widetilde{A}_2(q^{s-1})$ dividing $P_0$. Because also $\ell_{\eta\otimes\tau^*}(s)=\ell_{\tau^*\otimes\eta}(1-s)=1$ and $\gcd(\pi'\times\eta,s)=1$, the corollary implies
\begin{align*}
&\gcd(\pi'\times\varepsilon,s)\in\gcd(\pi'\times\tau,s)lcm(P_0,\widetilde{P}_0)^{-1}\C[q^{-s},q^s].
\end{align*}
Then using
$\gcd(\pi'\times\eta,s)=\gcd(\pi'\times\eta^*,1-s)=1$, Claim~\ref{claim:extra P factor in sub multiplicativity is mutual} yields
\begin{align*}
&\gcd(\pi'\times\varepsilon,s)=\gcd(\pi'\times\tau,s)lcm(P_0,\widetilde{P}_0)^{-1}Q_0,\\\notag
&\gcd(\pi'\times\varepsilon^*,1-s)\equalun\gcd(\pi'\times\tau^*,1-s)lcm(P_0,\widetilde{P}_0)^{-1}Q_0,
\end{align*}
for some $Q_0\in\C[q^{-s},q^s]$.
Write $lcm(P_0,\widetilde{P}_0)=P_0B=\widetilde{P}_0\widetilde{B}$ for some $B,\widetilde{B}\in\C[q^{-s},q^s]$. Then
\begin{align*}
&\gcd(\pi'\times\varepsilon,s)=L\ell_{\tau^*}(1-s)B^{-1}Q_0,\qquad
\gcd(\pi'\times\varepsilon^*,1-s)\equalun\widetilde{L}\ell_{\tau}(s)\widetilde{B}^{-1}Q_0.
\end{align*} 

According to \eqref{eq:multiplicative property of poles of intertwiners},
$\ell_{\varepsilon}(s)\in \ell_{\tau}(s)\C[q^{-s},q^s]$ and
$\ell_{\varepsilon^*}(1-s)\in \ell_{\tau^*}(1-s)\C[q^{-s},q^s]$. Hence
the operators \eqref{operators:holomorphic inside proof upper bound k<l} are holomorphic. Then we find
\begin{align*}
\gcd(\pi\times\varepsilon,s)\in
L(\sigma\times\tau,s)\gcd(\pi'\times\varepsilon,s)L(\sigma^*\times\tau,s)Q_0^{-1}\C[q^{-s},q^s],
\end{align*}
leading to the result.
\end{proof} 

\begin{proof}[Proof of Lemma~\ref{lemma:sub multiplicativity first var k<l<n}] 
Essentially, the arguments are similar to those of
Lemma~\ref{lemma:n_1<l<=n_2 basic identity} and are described
briefly. Let $W\in\Whittaker{\pi}{\psi_{\gamma}^{-1}}$ and
$f_s\in\xi(\tau,hol,s)$. We use the realization of
$\pi$ described in Section~\ref{subsection:Realization of pi for
induced representation} and select
$\varphi_{\zeta}\in\xi_{\overline{P_k}}^{G_{\smallidx}}(\sigma\otimes\pi',std,-\zeta+\half)$ for
$W$ ($W_{\varphi_0}=W$). Again we fix $\zeta$ with
$\Re(\zeta)>>0$ throughout the proof.

The integral $\Psi(\varphi_{\zeta},f_s,s)$ given by \eqref{int:k<l<n example
starting form} is representable by a strongly convergent series
$\Sigma(\varphi_{\zeta},f_s,s)\in \Sigma(X)$ in a domain $D$ of the form $\{\Re(\zeta)<<\Re(s)<<(1+C_1)\Re(\zeta)\}$, where $C_1>0$.
To construct the series we write the $dg$-integration over
$\overline{B_{G_{\smallidx}}}$, put
$\Gamma=\overline{B_{G_{\smallidx}}}\times V_k\times R_{\smallidx,\bigidx}$ and define
$\phi\in C^{\infty}(\Gamma,\C[q^{-s},q^s])$ by
\begin{align*}
\phi(g,v,r)=\varphi_{\zeta}(vg,I_k,I_{2(\smallidx-k)})\psi_{\gamma}(v)f_s(w_{\smallidx,\bigidx}rg,1)\psi_{\gamma}(r).
\end{align*}
Then (starting with a standard section $f_s$) we construct $\Sigma_{\phi}\in\Sigma(X)$ which represents $\Phi_{\phi}(s)=\Psi(\varphi_{\zeta},f_s,s)$ and set $\Sigma(\varphi_{\zeta},f_s,s)=\Sigma_{\phi}$.

Throughout the proof we use the following notation. 
The function
$W_{\sigma}^{\varphi_{\zeta}}(a)=\varphi_{\zeta}(1,a,1)$
($a\in\GL{k}$) belongs to $\Whittaker{\sigma}{\psi^{-1}}$ and
$W_{\pi'}^{\varphi_{\zeta}}(g')=\varphi_{\zeta}(1,1,g')$ ($g'\in G_{\smallidx-k}$) lies in
$\Whittaker{\pi'}{\psi_{\gamma}^{-1}}$. For a fixed $s$,
the mapping $W_{\tau}^{f_s}(b)=f_s(1,b)$ is in $\Whittaker{\tau}{\psi}$.

In Section~\ref{subsection:1st var l<n} we defined
integral~\eqref{eq:1st var l<n before eq for pi prime}. For
$\varphi_{\zeta}$ and $\theta_s\in\xi(\tau,hol,s)$ it is given
by
\begin{align}\label{eq:1st var l<n before eq for pi prime series}
\Psi_1(\varphi_{\zeta},\theta_s,s)=&\int_{\overline{B_{\GL{k}}}}\int_{V_k}
\int_{\Mat{k\times \bigidx-\smallidx}}
\int_{\overline{B_{G_{\smallidx-k}}}} \int_{R_{\smallidx-k,\bigidx}}\varphi_{\zeta}(vb,1,g')
\\\notag
&\theta_s(w_{\smallidx-k,\bigidx}r'(\rconj{b_{\bigidx,k}}g')w(\rconj{w_{\smallidx,\bigidx}}m)vb,1)\psi_{\gamma}(r')dr'dg'dmdvdb,
\end{align}
where $db$ is a right-invariant Haar measure on $\overline{B_{\GL{k}}}$ (we replaced the $dg$-integration over $\lmodulo{\overline{V_k}Z_kG_{\smallidx-k}}{G_{\smallidx}}$ by the $dvdb$-integration), $G_{\smallidx-k}<G_{\smallidx}<H_{\bigidx}$,
$dg'$ is a right-invariant Haar measure on $\overline{B_{G_{\smallidx-k}}}$, $m$ is the image
in $Q_{\bigidx}$ of
\begin{align*}
\left(\begin{array}{ccc}I_k&&m\\&I_{\smallidx-k}
\\&&I_{\bigidx-\smallidx}\end{array}\right)\in\GL{\bigidx},
\end{align*}
$b_{\bigidx,k}=diag(I_{\bigidx},(-1)^k,I_{\bigidx})$ and
$w=w_{\smallidx-k,\bigidx}^{-1}\omega_{\bigidx-k,k}w_{\smallidx,\bigidx}$.
This integral is absolutely convergent, that is, the integral with $\varphi_{\zeta},\theta_s$ replaced by
$|\varphi_{\zeta}|,|\theta_s|$ and without $\psi_{\gamma}$, is convergent, in a domain depending only
on the representations. For a fixed $\zeta$ with $\Re(\zeta)>>0$
this is a strip in $s$, denoted by $D^*$, of the form
$\{(1-C_2)\Re(\zeta)<<\Re(s)<<\Re(\zeta)\}$ with $1>C_2>0$.
Let
\begin{align*}
\Gamma=
\overline{B_{\GL{k}}}\times V_k\times\Mat{k\times
\bigidx-\smallidx},\qquad\Gamma'=\overline{B_{G_{\smallidx-k}}}\times
R_{\smallidx-k,\bigidx}.
\end{align*}
The integrand defines an element of
$C^{\infty}(\Gamma\times\Gamma',\C[q^{-s},q^s])$, smooth in $\Gamma$.
In turn, $\Psi_1(\varphi_{\zeta},\theta_s,s)$ is represented by a
strongly convergent series which is denoted by
$\Sigma_1(\varphi_{\zeta},\theta_s,s)$. 
The $dr'dg'$-integration in \eqref{eq:1st var l<n before eq for pi
prime series} resembles the integral for $G_{\smallidx-k}\times
\GL{\bigidx}$ and $\pi'\times\tau$, 
\begin{align}\label{eq:1st var l<n before eq for pi prime inner double}
\Psi_1(\varphi_{\zeta},\theta_s,s)=\int_{\overline{B_{\GL{k}}}}\int_{V_k}
\int_{\Mat{k\times \bigidx-\smallidx}}\Psi(W_{\pi'}^{vb\cdot\varphi_{\zeta}},((w(\rconj{w_{\smallidx,\bigidx}}m)vb)\cdot
\theta_s)^{b_{\bigidx,k}},s)dmdvdb.
\end{align}

When we multiply $L(\sigma\times\tau,s-\zeta)^{-1}$, $\gcd(\pi'\times\tau,s)^{-1}$ or similar factors
by a series in $\Sigma(X)$, we regard them as polynomials in $R(X)$ by replacing
$q^{-s}$ with $X$. To simplify notation we still denote, for example, $\gcd(\pi'\times\tau,s)^{-1}$,
instead of $\gcd(\pi'\times\tau,X)^{-1}$.

We turn to the proof of the lemma. First we apply the functional
equation for $\sigma\times\tau$.
\begin{claim}\label{claim:k<l<n func sigma and tau}
$L(\sigma\times\tau,s-\zeta)^{-1}\Sigma(\varphi_{\zeta},f_s,s)=L(\sigma^*\times\tau^*,1-(s-\zeta))^{-1}\Sigma_1(\varphi_{\zeta},f_s,s)$.
\end{claim}
\begin{proof}[Proof of Claim~\ref{claim:k<l<n func sigma and tau}] 
Integral $\Psi(\phi_{\zeta},f_s,s)$ is equal to
integral~\eqref{eq:1st var l<n before func eq sigma}, in its domain
of absolute convergence, while $\Psi_1(\phi_{\zeta},f_s,s)$ equals
\eqref{eq:1st var l<n after func eq sigma}, in a different domain.
We use Lemmas~\ref{lemma:interchange integration and polynomial},
\ref{lemma:double iterated series} and the
functional equation \eqref{eq:modified func eq glm gln}. The $\epsilon$-factor
$\epsilon(\sigma\times\tau,\psi,s-\zeta)\in\C[q^{-(s-\zeta)},q^{s-\zeta}]^*$
and $\omega_{\sigma}(-1)^{\bigidx-1}$ are ignored. We take $\Re(\zeta)>>0$ so that $D$ (resp. $D^*$) intersects
the right half-plane (resp. left half-plane) where the \lhs\ (resp. \rhs) of \eqref{eq:modified func eq glm gln} is defined, in
some domain (see the proof of Claim~\ref{claim:l<n functional eq sigma}). Note that
\eqref{eq:modified func eq glm gln} has an interpretation in
$\Sigma(X)$ according to the definition of the $L$-factor of
Jacquet, Piatetski-Shapiro and Shalika \cite{JPSS}.
\end{proof} 
Now apply the functional equation for $\pi'\times\tau$. Set
\begin{align*}
f_{1-s}^*=P_{\tau}(q^{-s})\gcd(\pi'\times\tau^*,1-s)^{-1}\nintertwiningfull{\tau}{s}f_s\in\xi(\tau^*,hol,1-s).
\end{align*}
\begin{claim}\label{claim:k<l<n func pi' and tau}
$P_{\tau}(X)\gcd(\pi'\times\tau,s)^{-1}\Sigma_1(\varphi_{\zeta},f_s,s)=\Sigma_1(\varphi_{\zeta},f_{1-s}^*,1-s)$.
\end{claim}
\begin{proof}[Proof of Claim~\ref{claim:k<l<n func pi' and tau}] 
This follows from the interpretation of the passage \eqref{eq:1st var
l<n before eq for pi prime}-\eqref{eq:1st var l<n after eq for pi
prime}, using the functional equation for $G_{\smallidx-k}\times
\GL{\bigidx}$ and $\pi'\times\tau$, as series. The factor $\epsilon(\pi'\times\tau,\psi,s)\in\C[q^{-s},q^s]^*$ is ignored. The fact that $f_{1-s}^*$
is a holomorphic section implies the existence of the series
$\Sigma_1(\varphi_{\zeta},f_{1-s}^*,1-s)$.
\end{proof} 
Last we utilize the functional equation for $\sigma\times\tau^*$. By
virtue of Claim~\ref{claim:k<l<n func sigma and tau} with $f_{1-s}^*$ instead of $f_s$
and $(\tau^*,1-s)$ replacing $(\tau,s)$,
\begin{align}\label{eq:k<l<n func sigma and tau^*}
L(\sigma^*\times\tau,s+\zeta)^{-1}\Sigma_1(\varphi_{\zeta},f_{1-s}^*,1-s)=L(\sigma\times\tau^*,1-(s+\zeta))^{-1}\Sigma(\varphi_{\zeta},f_{1-s}^*,1-s).
\end{align}

Combining Claims~\ref{claim:k<l<n func sigma and tau},
\ref{claim:k<l<n func pi' and tau} and equality~\eqref{eq:k<l<n func
sigma and tau^*},
\begin{align*}
&P_{\tau}(X)L(\sigma\times\tau,s-\zeta)^{-1}\gcd(\pi'\times\tau,s)^{-1}L(\sigma^*\times\tau,s+\zeta)^{-1}\Sigma(\varphi_{\zeta},f_s,s)\\
&=L(\sigma^*\times\tau^*,1-(s-\zeta))^{-1}L(\sigma\times\tau^*,1-(s+\zeta))^{-1}\Sigma(\varphi_{\zeta},f_{1-s}^*,1-s).
\end{align*}
As in the proof of Lemma~\ref{lemma:n_1<l<=n_2 basic identity} we
put $\zeta=0$. The integral $\Psi(\varphi_{\zeta},f_s,s)$ is absolutely convergent in $D$, where it is equal to
$\Psi(W_{\varphi_{\zeta}},f_s,s)$. The integral $\Psi(W_{\varphi_{\zeta}},f_s,s)$ is absolutely convergent in $D$ and also in a right half-plane $\Re(s)>>0$. Let
$D''$ be the union of these domains and let $\Sigma(W_{\varphi_{\zeta}},f_s,s)\in\Sigma(X)$ be the series representing
$\Psi(W_{\varphi_{\zeta}},f_s,s)$ in $D''$ (obtained as explained in Lemma~\ref{lemma:integral for holomorphic section representable as a series}).
Then $\Sigma(\varphi_{\zeta},f_s,s)=\Sigma(W_{\varphi_{\zeta}},f_s,s)$ because they represent the same integral in $D$. Therefore we obtain the equality
\begin{align}\label{eq:first var before substituting zero 0}
&P_{\tau}(X)L(\sigma\times\tau,s-\zeta)^{-1}\gcd(\pi'\times\tau,s)^{-1}L(\sigma^*\times\tau,s+\zeta)^{-1}\Sigma(W_{\varphi_{\zeta}},f_s,s)\\\notag
&=L(\sigma^*\times\tau^*,1-(s-\zeta))^{-1}L(\sigma\times\tau^*,1-(s+\zeta))^{-1}\Sigma(W_{\varphi_{\zeta}},f_{1-s}^*,1-s).
\end{align}

As in the proof of Claim~\ref{claim:putting Y=1}, $\Sigma(W_{\varphi_{\zeta}},f_s,s)$ is a series of the form
$\sum_{m=M}^{\infty}a_m(\zeta)X^m$ where $M$ is independent of
$\zeta$ and each $a_m$ is a polynomial in $q^{\mp\zeta}$. Specifically, we use
Proposition~\ref{proposition:iwasawa decomposition of the integral}
to write $\Psi(W_{\varphi_{\zeta}},f_s,s)$ for $s$ in a right half-plane, as a sum of integrals of the form \eqref{int:iwasawa decomposition
of the integral l<=n split}. Here $W\in\Whittaker{\pi}{\psi_{\gamma}^{-1}}$ is
replaced by $W_{\varphi_{\zeta}}\in\Whittaker{\cinduced{\overline{P_k}}{G_{\smallidx}}{(\sigma\otimes\pi')\alpha^{-\zeta+\half}}}{\psi_{\gamma}^{-1}}$. The number of
negative coefficients of $X$ is bounded using $W_{\varphi_{\zeta}}$, this bound will be independent
of $\zeta$ (alternatively use the Whittaker function in
$\Whittaker{\tau}{\psi}$ for this, here $\smallidx\leq\bigidx$). Because $\varphi_{\zeta}$ is a standard section (in
$\zeta$), for any $g\in G_{\smallidx}$ we have
$W_{\varphi_{\zeta}}(g)\in \C[q^{-\zeta},q^{\zeta}]$, whence $a_m\in\C[q^{-\zeta},q^{\zeta}]$.

This holds for any $\zeta$ with $\Re(\zeta)>>0$ and
also for any $\zeta$ varying in a compact set. Also recall that $W_{\varphi_0}=W$. As
in Claim~\ref{claim:putting Y=1}, substituting $0$ for $\zeta$ in \eqref{eq:first var before substituting zero 0}
yields
\begin{align*}
&P_{\tau}(X)L(\sigma\times\tau,s)^{-1}\gcd(\pi'\times\tau,s)^{-1}L(\sigma^*\times\tau,s)^{-1}\Sigma(W,f_s,s)\\
&=L(\sigma^*\times\tau^*,1-s)^{-1}L(\sigma\times\tau^*,1-s)^{-1}\Sigma(W,f_{1-s}^*,1-s).
\end{align*}
As in Lemma~\ref{lemma:n_1<l<=n_2 basic identity} it follows that both sides are polynomials,
\begin{align*}
&\Psi(W,f_s,s)\in
L(\sigma\times\tau,s)\gcd(\pi'\times\tau,s)L(\sigma^*\times\tau,s)P_{\tau}(q^{-s})^{-1}\C[q^{-s},q^s],\\
&\Psi(W,\nintertwiningfull{\tau}{s}f_s,1-s)\\\notag&\qquad\in
L(\sigma^*\times\tau^*,1-s)\gcd(\pi'\times\tau^*,1-s)L(\sigma\times\tau^*,1-s)P_{\tau}(q^{-s})^{-1}\C[q^{-s},q^s].
\end{align*}

These relations resemble \eqref{eq:second var gcd of holomorphic
sections} and \eqref{eq:second var gcd of holomorphic sections 2}, and hold for any irreducible representation $\tau$,
$W\in\Whittaker{\pi}{\psi_{\gamma}^{-1}}$ and $f_s\in\xi(\tau,hol,s)$.
Thus we can replace $\tau$ with $\tau^*$ and obtain
for any $f_{1-s}'\in\xi(\tau^*,hol,1-s)$,
\begin{align*}
&\Psi(W,\nintertwiningfull{\tau^*}{1-s}f_{1-s}',s)\\\notag&\qquad\in
L(\sigma^*\times\tau,s)\gcd(\pi'\times\tau,s)L(\sigma\times\tau,s)P_{\tau^*}(q^{s-1})^{-1}\C[q^{-s},q^s].
\end{align*}
Hence
\begin{equation*}
\gcd(\pi\times\tau,s)\in
L(\sigma\times\tau,s)\gcd(\pi'\times\tau,s)L(\sigma^*\times\tau,s)Q(q^{-s},q^s)^{-1}\C[q^{-s},q^s].\qedhere
\end{equation*}
\end{proof} 

\subsection{The case $k=\smallidx>\bigidx$}\label{subsection:sub mult first var proof body k=l>n}
Let $\pi=\cinduced{P_{\smallidx}}{G_{\smallidx}}{\sigma}$ (or
$\cinduced{\rconj{\kappa}P_{\smallidx}}{G_{\smallidx}}{\sigma}$) and let
$\tau=\cinduced{P_{\bigidx_1,\ldots,\bigidx_a}}{\GL{\bigidx}}{\tau_1\otimes\ldots\otimes\tau_a}$. Assume that $\tau$ is irreducible.
We show the slightly stronger relation
\begin{align*}
\gcd(\pi\times\tau,s)\in
L(\sigma\times\tau,s)L(\sigma^*\times\tau,s)M_{\tau}(s)\C[q^{-s},q^s]
\end{align*}
(stronger because $M_{\tau}(s)\in
M_{\tau_1\otimes\ldots\otimes\tau_a}(s)\C[q^{-s},q^s]$). This is an
immediate corollary of the following lemma, analogous to
Lemma~\ref{lemma:sub multiplicativity first var k<l<n}.
\begin{lemma}\label{lemma:sub multiplicativity first var k=l>n}
Assume $k=\smallidx>\bigidx$. Let $P_{\tau},P_{\tau^*}\in\C[X]$ be normalized by
$P_{\tau}(0)=P_{\tau^*}(0)=1$ and of minimal degree such that the operators
$P_{\tau}(q^{-s})\nintertwiningfull{\tau}{s}$ and\\
$P_{\tau^*}(q^{s-1})\nintertwiningfull{\tau^*}{1-s}$ are holomorphic and let
$Q\in\C[q^{-s},q^s]$ be a least common multiple of $P_{\tau}(q^{-s})$ and $P_{\tau^*}(q^{s-1})$. Then
\begin{align*}
\gcd(\pi\times\tau,s)\in
L(\sigma\times\tau,s)L(\sigma^*\times\tau,s)Q(q^{-s},q^s)^{-1}\C[q^{-s},q^s].
\end{align*}
\end{lemma}
The proof of the lemma follows the same line of arguments of
Lemma~\ref{lemma:sub multiplicativity first var k<l<n}, utilizing
the integral manipulations of Section~\ref{subsection:1st var
k=l>n+1}. The functional equation replacing that of $\pi'\times\tau$ is
Shahidi's functional equation \eqref{eq:Shahidi local coefficient
def}. 
In the analogue of Claim~\ref{claim:putting Y=1} one uses
\eqref{int:iwasawa decomposition of the integral l>n}. In the proof
of Proposition~\ref{proposition:iwasawa decomposition of the
integral}, \eqref{int:iwasawa decomposition of the integral l>n} was
established by noting that the support of the function $r\mapsto
W(ar)$ ($r\in R^{\smallidx,\bigidx}$, $a\in\GL{\bigidx}<L_{\bigidx}$) is contained in a compact set independent of $a$. When
replacing $W$ with $W_{\varphi_{\zeta}}$, this set can also be taken
to be independent of $\zeta$, because $\varphi_{\zeta}$ is a standard section. The proof of the lemma is skipped.

\subsection{The case $k=\smallidx\leq\bigidx$}\label{subsection:sub mult first var proof body k=l}
In Section~\ref{section:1st variable} the multiplicativity of the
$\gamma$-factor when $k=\smallidx\leq\bigidx$ was reduced to the cases $k<\smallidx$ and $k=\smallidx>\bigidx$, hence in this case we do not have integral manipulations to work with.
\begin{remark}
The arguments of Soudry \cite{Soudry3} regarding the
multiplicativity of the $\gamma$-factor for the \archimedean\
integrals of $SO_{2\bigidx+1}\times \GL{\smallidx}$ readily adapt to the
$p$-adic case. They imply a proof, using integral manipulations, that
the $SO_{2\smallidx}\times\GL{\bigidx}$ $\gamma$-factor
$\gamma(\pi\times\tau,\psi,s)$ where $\pi$ is induced from
$P_{\smallidx}$ and $\smallidx\leq\bigidx$, is multiplicative (see
Chapter~\ref{chapter:archimedean results}). It may be possible to use these
manipulations, but more work will be needed, because they include a Fourier transform which needs to
be translated into an operation on Laurent series.
\end{remark}
We prove the following slightly stronger result. Let
$\pi=\cinduced{P_{\smallidx}}{G_{\smallidx}}{\sigma}$ where $\sigma$ is
a (generic) quotient of a representation
$\theta=\cinduced{P_{\smallidx_1,\ldots,\smallidx_m}}{\GL{\smallidx}}{\sigma_1\otimes\ldots\otimes\sigma_m}$
of Langlands' type, i.e. $\sigma_i=\absdet{}^{u_i}\sigma_i'$,
$\sigma_i'$ is tempered, $u_i\in\R$ and $u_1>\ldots>u_m$. Since
$\pi$ is a generic quotient of
$\cinduced{P_{\smallidx}}{G_{\smallidx}}{\theta}$,
$\Whittaker{\pi}{\psi_{\gamma}^{-1}}=\Whittaker{\cinduced{P_{\smallidx}}{G_{\smallidx}}{\theta}}{\psi_{\gamma}^{-1}}$.
The dependence of the integral on $\pi$ is only through its
Whittaker model, hence for the purpose of the g.c.d. we may replace
$\sigma$ with $\theta$, i.e., assume $\pi=\cinduced{P_{\smallidx}}{G_{\smallidx}}{\theta}$. Set
$\mu=\cinduced{P_{\smallidx_1,\ldots,\smallidx_{m-1}}}{\GL{\smallidx-\smallidx_m}}{\sigma_1\otimes\ldots\otimes\sigma_{m-1}}$.

Let
$\tau=\cinduced{P_{\bigidx_1,\ldots,\bigidx_a}}{\GL{\bigidx}}{\tau_1\otimes\ldots\otimes\tau_a}$
be an irreducible representation of Langlands' type and put
$\tau'=\cinduced{P_{\bigidx_2,\ldots,\bigidx_a}}{\GL{\bigidx-\bigidx_1}}{\tau_2\otimes\ldots\otimes\tau_a}$.
By definition the representations $\mu$, $\tau_1$ and $\tau'$ are of Langlands' type and the representations
$\tau_1$ and $\tau'$ are irreducible.
We show
\begin{align}\label{eq:statement of stronger result for k=l<=n}
\gcd(\pi\times\tau,s)\in
L(\theta\times\tau,s)L(\theta^*\times\tau,s)M_{\tau_1\otimes\ldots\otimes\tau_a}(s)\C[q^{-s},q^s].
\end{align}
Since $\sigma$ and $\theta$ can be realized using the same Whittaker model and
$L(\theta\times\tau,s)$ depends only on the Whittaker models of $\theta$ and $\tau$, we conclude
$L(\sigma\times\tau,s)=L(\theta\times\tau,s)$, $L(\sigma^*\times\tau,s)=L(\theta^*\times\tau,s)$ and
\begin{align*}
\gcd(\pi\times\tau,s)\in
L(\sigma\times\tau,s)L(\sigma^*\times\tau,s)M_{\tau_1\otimes\ldots\otimes\tau_a}(s)\C[q^{-s},q^s].
\end{align*}

Any irreducible $\sigma$ is the quotient of a representation $\theta$ of Langlands' type (actually,
since $\sigma$ is the generic irreducible quotient, $\sigma\isomorphic\theta$ by \cite{Z3}), hence the result we prove
in particular implies the theorem in case $k=\smallidx\leq\bigidx$.

The following lemma is used to prove \eqref{eq:statement of stronger result for k=l<=n}.
\begin{lemma}\label{lemma:sub multiplicativity first var k=l>=n+1}
Let $\sigma$ and $\tau$ be essentially tempered. Then
\begin{align*}
\gcd(\cinduced{P_{\smallidx}}{G_{\smallidx}}{\sigma}\times\tau,s)\in
L(\sigma\times\tau,s)L(\sigma^*\times\tau,s)M_{\tau}(s)\C[q^{-s},q^s].
\end{align*}
\end{lemma}
Before proving the lemma, let us use it to derive \eqref{eq:statement of stronger result for k=l<=n}.
If $m>1$, set $\pi'=\cinduced{P_{\smallidx_m}}{G_{\smallidx_m}}{\sigma_m}$ and then
$\pi=\cinduced{P_{\smallidx-\smallidx_m}}{G_{\smallidx}}{\mu\otimes\pi'}$.
Assume $a=1$. If $m=1$, relation~\eqref{eq:statement of
stronger result for k=l<=n} follows directly from
Lemma~\ref{lemma:sub multiplicativity first var k=l>=n+1}. For $m>1$, by Lemma~\ref{lemma:sub multiplicativity first var k=l>=n+1} applied to $\pi'\times\tau$ and
Claim~\ref{claim:extra P factor in sub multiplicativity is mutual}, for some $P_0\in\C[q^{-s},q^s]$,
\begin{align*}
&\gcd(\pi'\times\tau,s)=L(\sigma_m\times\tau,s)L(\sigma_m^*\times\tau,s)M_{\tau}(s)P_0,\\\notag
&\gcd(\pi'\times\tau^*,1-s)\equalun L(\sigma_m\times\tau^*,1-s)L(\sigma_m^*\times\tau^*,1-s)M_{\tau}(s)P_0.
\end{align*}
Putting this into Corollary~\ref{corollary:first var k<l<n exact sub
mult} yields 
\begin{align*}
\gcd(\pi\times\tau,s)\in
L(\mu\times\tau,s)L(\sigma_m\times\tau,s)L(\sigma_m^*\times\tau,s)L(\mu^*\times\tau,s)M_{\tau}(s)\C[q^{-s},q^s].
\end{align*}
Because $\theta$ and $\tau$ are of Langlands' type,
$L(\theta\times\tau,s)=L(\mu\times\tau,s)L(\sigma_m\times\tau,s)$
(\cite{JPSS} Theorem~9.4) and \eqref{eq:statement of stronger result
for k=l<=n} folows.

Now assume $a>1$.
We use induction on $a$. By the induction hypothesis we find that for some $Q_1,Q'\in \C[q^{-s},q^s]$,
\begin{align*}
&\gcd(\pi\times\tau_1,s)=
L(\theta\times\tau_1,s)L(\theta^*\times\tau_1,s)M_{\tau_1}(s)Q_1,\\\notag
&\gcd(\pi\times\tau',s)=
L(\theta\times\tau',s)L(\theta^*\times\tau',s)M_{\tau_2\otimes\ldots\otimes\tau_a}(s)Q'.
\end{align*}
By virtue of Claim~\ref{claim:extra P factor in sub multiplicativity
is mutual}, 
\begin{align*}
&\gcd(\pi\times\tau_1^*,1-s)\equalun
L(\theta\times\tau_1^*,1-s)L(\theta^*\times\tau_1^*,1-s)M_{\tau_1}(s)Q_1,\\\notag
&\gcd(\pi\times(\tau')^*,1-s)\equalun
L(\theta\times(\tau')^*,1-s)L(\theta^*\times(\tau')^*,1-s)M_{\tau_2\otimes\ldots\otimes\tau_a}(s)Q'.
\end{align*}
Apply Corollary~\ref{corollary:refinement gcd second var bound} to $\pi\times\tau$ with
$\tau=\cinduced{P_{\bigidx_1,\bigidx-\bigidx_1}}{\GL{\bigidx}}{\tau_1\otimes\tau'}$. We can take $A_1,\widetilde{A}_1\in\C[X]$
such that $A_1(q^{-s})$ and $\widetilde{A}_1(q^{s-1})$ divide $Q_1$ (in $\C[q^{-s},q^s]$). Since also
$M_{\tau'}(s)\in M_{\tau_2\otimes\ldots\otimes\tau_a}(s)\C[q^{-s},q^s]$ (see Section~\ref{subsection:the factors M tau s}),
we can select $A_2,\widetilde{A}_2\in\C[X]$ such that $A_2(q^{-s})$ and $\widetilde{A}_2(q^{s-1})$ divide $Q'$. The corollary implies
\begin{align*}
\gcd(&\pi\times\tau,s)\\\notag&\in\gcd(\pi\times\tau_1,s)Q_1^{-1}\gcd(\pi\times\tau',s)(Q')^{-1}
\ell_{\tau_1\otimes(\tau')^*}(s)\ell_{(\tau')^*\otimes\tau_1}(1-s)\C[q^{-s},q^s].
\end{align*}
Now \eqref{eq:statement of stronger result for k=l<=n} follows using
the multiplicativity of $L(\theta\times\tau,s)$ in $\tau$
(\cite{JPSS} Theorem~9.4) and the fact that
\begin{align*}
M_{\tau_1}(s)\ell_{\tau_1\otimes(\tau')^*}(s)\ell_{(\tau')^*\otimes\tau_1}(1-s)M_{\tau_2\otimes\ldots\otimes\tau_a}(s)\in
M_{\tau_1\otimes\ldots\otimes\tau_a}(s)\C[q^{-s},q^s].
\end{align*}
The case of $\pi$ induced from $\rconj{\kappa}P_{\smallidx}$ is similar.
\begin{proof}[Proof of Lemma~\ref{lemma:sub multiplicativity first var k=l>=n+1}] 
Set $\pi=\cinduced{P_{\smallidx}}{G_{\smallidx}}{\sigma}$. Assume first that $\tau$ is tempered. Let $\mathcal{U}$ be the set
of all complex numbers $u$ such that $\sigma_u=\absdet{}^u\sigma$ is
tempered. The set $\mathcal{U}$ is a vertical line on the
plane. Let
$\pi_u=\cinduced{P_{\smallidx}}{G_{\smallidx}}{\sigma_u}$. Ignoring
a discrete subset of $\mathcal{U}$, we can assume that $\pi_u$ is
tempered (we take tempered representations to be irreducible, see
\cite{Mu2} Section~3 and \cite{W} Lemma~III.2.3). 
We may apply Theorem~\ref{theorem:gcd for tempered reps} to
$\pi_u\times\tau$ and obtain
\begin{align*}
\gcd(\pi_u\times\tau,s)\in
L(\pi_u\times\tau,s)M_{\tau}(s)\C[q^{-s},q^s],
\end{align*}
where
$L(\pi_u\times\tau,s)$ is the $L$-function defined by
Shahidi.

Since $\pi_u$ and $\tau$ are tempered, the results
of 
Casselman and Shahidi \cite{CSh} (Section~4) and \eqref{eq:JPSS
relation gamma and friends} imply
\begin{align*}
L(\pi_u\times\tau,s)\in
L(\sigma\times\tau,s+u)L(\sigma^*\times\tau,s-u)\C[q^{-s},q^s].
\end{align*}
Thus
\begin{align*}
\gcd(\pi_u\times\tau,s)\in
L(\sigma\times\tau,s+u)L(\sigma^*\times\tau,s-u)M_{\tau}(s)\C[q^{-s},q^s].
\end{align*}
If we could put $u=0$, the result would follow, but this has to be
justified. Currently, for each $u\in\mathcal{U}$ there is some
$P_u\in\C[q^{-s},q^s]$ such that
\begin{align*}
\gcd(\pi_u\times\tau,s)=L(\sigma\times\tau,s+u)L(\sigma^*\times\tau,s-u)M_{\tau}(s)P_u.
\end{align*}
Write $M_{\tau}(s)=eM_{\tau}(s)'$, $P_u=e_uP_u'$ for
$(M_{\tau}(s)')^{-1},P_u'\in\C[q^{-s}]$ with constant terms equal
to
$1$ 
and $e,e_u\in\C[q^{-s},q^s]^*$. Because each of the other factors in this
equation is an inverse of a polynomial in $\C[q^{-s}]$ and its
constant term is equal to $1$, we obtain $e\cdot e_u=1$ and in
particular $e_u$ is independent of $u$. Since
$\gcd(\pi_u\times\tau,s)$ is an inverse of a polynomial, any zero of
$P_u$ must be canceled by some other factor on the \rhs. Therefore
$P_u'$ is uniquely determined by a finite product of factors
appearing in either $L(\sigma\times\tau,s+u)$,
$L(\sigma^*\times\tau,s-u)$ or $M_{\tau}(s)'$, which take the
form $(1-aq^{-s-u})^{-1}$, $(1-aq^{-s+u})^{-1}$ or
$(1-aq^{-s})^{-1}$ (resp.). Hence we may assume (perhaps passing to
a smaller subset $\mathcal{U}'\subset\mathcal{U}$) that there exists
$P\in\C[q^{\mp u},q^{\mp s}]$ such that for all $u\in\mathcal{U}$,
\begin{align}\label{eq:gcd pi tau k=l>=n+1}
\gcd(\pi_u\times\tau,s)=L(\sigma\times\tau,s+u)L(\sigma^*\times\tau,s-u)M_{\tau}(s)P.
\end{align}
Note that $\mathcal{U}$ (still) contains some infinite sequence
which converges to some point in the plane.
Repeating this for $\pi_u\times\tau^*$ we may assume in addition
that there is some $\widetilde{P}\in\C[q^{\mp u},q^{\mp s}]$
satisfying for all $u\in\mathcal{U}$,
\begin{align}\label{eq:gcd pi tau star k=l>=n+1}
\gcd(\pi_u\times\tau^*,1-s)=L(\sigma\times\tau^*,1-s+u)L(\sigma^*\times\tau^*,1-s-u)M_{\tau}(s)\widetilde{P}.
\end{align}
Fix $u\in\mathcal{U}$. According to Theorem~\ref{theorem:multiplicity first var},
\begin{align*}
\gamma(\pi_u\times\tau,\psi,s)=\omega_{\sigma_u}(-1)^{\bigidx}\omega_{\tau}(-1)^{\smallidx}\omega_{\tau}(2\gamma)^{-1}\gamma(\sigma_u\times\tau,\psi,s)\gamma(\sigma_u^*\times\tau,\psi,s)
\end{align*}
and by \eqref{eq:gamma and epsilon} and
\eqref{eq:JPSS relation gamma and friends}, 
\begin{align}\label{eq:gcd epsilon long k=l>=n+1}
&\epsilon(\pi_u\times\tau,\psi,s)
\frac{\gcd(\pi_u\times\tau^*,1-s)}{\gcd(\pi_u\times\tau,s)}
\\\notag&=\epsilon(\sigma\times\tau,\psi,s+u)\epsilon(\sigma^*\times\tau,\psi,s-u)
\frac{L(\sigma^*\times\tau^*,1-s-u)L(\sigma\times\tau^*,1-s+u)}
{L(\sigma\times\tau,s+u)L(\sigma^*\times\tau,s-u)}.
\end{align}
Here the factor
$\omega_{\sigma_u}(-1)^{\bigidx}\omega_{\tau}(-1)^{\smallidx}\omega_{\tau}(2\gamma)^{-1}$
was omitted, it does not impact the argument because it is
independent of $u$ and $s$
($\omega_{\sigma_u}(-1)=\omega_{\sigma}(-1)$). Combining \eqref{eq:gcd pi tau k=l>=n+1}, \eqref{eq:gcd pi tau star
k=l>=n+1} and \eqref{eq:gcd epsilon long k=l>=n+1} we get 
\begin{align}\label{eq:epsilon U k=l>=n+1}
\epsilon(\pi_u\times\tau,\psi,s)=\epsilon(\sigma\times\tau,\psi,s+u)\epsilon(\sigma^*\times\tau,\psi,s-u)P\widetilde{P}^{-1}.
\end{align}

Equality~\eqref{eq:epsilon def} implies that for
all $W_u\in\Whittaker{\pi_u}{\psi_{\gamma}^{-1}}$ and
$f_s\in\xi(\tau,hol,s)$,
\begin{align*}
&c(\smallidx,\tau,\gamma,s)^{-1}
\epsilon(\pi_u\times\tau,\psi,s)\gcd(\pi_u\times\tau,s)^{-1}\Psi(W_u,f_s,s)\\\notag&=
\gcd(\pi_u\times\tau^*,1-s)^{-1}\Psi(W_u,\nintertwiningfull{\tau}{s}f_s,1-s).
\end{align*}
Plugging \eqref{eq:gcd pi tau k=l>=n+1}, \eqref{eq:gcd pi tau
star k=l>=n+1} and \eqref{eq:epsilon U k=l>=n+1} into this equation
yields
\begin{align}\label{eq:func eq integrals k=l>=n+1}
&c(\smallidx,\tau,\gamma,s)^{-1}\epsilon(\sigma\times\tau,\psi,s+u)\epsilon(\sigma^*\times\tau,\psi,s-u)
\\\notag&M_{\tau}(s)^{-1}L(\sigma\times\tau,s+u)^{-1}L(\sigma^*\times\tau,s-u)^{-1}
\Psi(W_u,f_s,s)\\\notag&=
L(\sigma\times\tau^*,1-s+u)^{-1}L(\sigma^*\times\tau^*,1-s-u)^{-1}
\Psi(W_u,M_{\tau}(s)^{-1}\nintertwiningfull{\tau}{s}f_s,1-s).
\end{align}
Note that
$M_{\tau}(s)^{-1}\nintertwiningfull{\tau}{s}f_s\in\xi(\tau^*,hol,1-s)$.
For an arbitrary $W\in\Whittaker{\pi}{\psi_{\gamma}^{-1}}$ take
$W_u$ such that $W_0=W$, by applying a Whittaker functional to a
suitable section of
$\xi_{P_{\smallidx}}^{G_{\smallidx}}(\sigma,std,u+\half)$ (see
Section~\ref{subsection:Realization of pi for induced
representation}).

Let $C\subset\C$ be a compact subset containing $0$, which
intersects $\mathcal{U}$ in a set $D_1$ containing some infinite sequence that
converges to some point in the plane. There exist constants $s_1,s_2>0$ depending only on $C,\sigma$ and $\tau$ such
that the series $\Sigma(W_u,f_s,s)$ (resp.
$\Sigma(W_u,M_{\tau}(s)^{-1}\nintertwiningfull{\tau}{s}f_s,1-s)$)
represents $\Psi(W_u,f_s,s)$ (resp.
$\Psi(W_u,M_{\tau}(s)^{-1}\nintertwiningfull{\tau}{s}f_s,1-s)$)
for all $u\in C$ and $\Re(s)>s_1$ (resp. $\Re(s)<-s_2$). Let $u\in
D_1$. Then both sides of \eqref{eq:func eq integrals k=l>=n+1} are
polynomials in $q^{\mp s}$ and
as explained in Section~\ref{subsection:Interpretation of functional
equations} this equality may be interpreted in $\Sigma(X)$,
\begin{align}\label{eq:func eq integrals k=l>=n+1 in sigma}
&A(u,X)\Sigma(W_u,f_s,s)=B(u,X)\Sigma(W_u,M_{\tau}(s)^{-1}\nintertwiningfull{\tau}{s}f_s,1-s),
\end{align}
where
\begin{align*}
A(u,s)=&c(\smallidx,\tau,\gamma,s)^{-1}\epsilon(\sigma\times\tau,\psi,s+u)\epsilon(\sigma^*\times\tau,\psi,s-u)\\\notag
&M_{\tau}(s)^{-1}L(\sigma\times\tau,s+u)^{-1}L(\sigma^*\times\tau,s-u)^{-1},
\end{align*}
\begin{align*}
&B(u,s)=L(\sigma\times\tau^*,1-s+u)^{-1}L(\sigma^*\times\tau^*,1-s-u)^{-1}
\end{align*}
and $A(u,X),B(u,X)\in R(X)$ are obtained by substituting $X$ for $q^{-s}$.
As in the proof of Claim~\ref{claim:putting Y=1} both series
$\Sigma(W_u,f_s,s)$ and
$\Sigma(W_u,M_{\tau}(s)^{-1}\nintertwiningfull{\tau}{s}f_s,1-s)$
have analytic coefficients in $u$. Hence \eqref{eq:func eq integrals
k=l>=n+1 in sigma} is an equality of the form
$\sum_{m\in\Integers}a_m(u)X^m=\sum_{m\in\Integers}b_m(u)X^m$ valid
for all $u\in D_1$, with analytic functions $a_m,b_m:C\rightarrow\C$. Thus it
is valid also for $u=0$ and we obtain an equality in $\Sigma(X)$,
\begin{align*}
&A(0,X)\Sigma(W_0,f_s,s)=B(0,X)\Sigma(W_0,M_{\tau}(s)^{-1}\nintertwiningfull{\tau}{s}f_s,1-s).
\end{align*}
Because $\Sigma(W_0,f_s,s)$ represents
$\Psi(W_0,f_s,s)=\Psi(W,f_s,s)$, the \lhs\ is a series with finitely
many negative coefficients. Similarly
$\Sigma(W_0,M_{\tau}(s)^{-1}\nintertwiningfull{\tau}{s}f_s,1-s)$
represents
$\Psi(W,M_{\tau}(s)^{-1}\nintertwiningfull{\tau}{s}f_s,1-s)$, so
the \rhs\ is a series with finitely many positive coefficients.
Whence this is also an equality in $R(X)$ and
\begin{align*}
\Psi(W,f_s,s)\in
L(\sigma\times\tau,s)L(\sigma^*\times\tau,s)M_{\tau}(s)\C[q^{-s},q^s].
\end{align*}
A similar relation holds for
$\Psi(W,\nintertwiningfull{\tau^*}{1-s}f_{1-s}',s)$ with $f_{1-s}'\in\xi(\tau^*,hol,1-s)$
(see the last paragraph in the proof of Lemma~\ref{lemma:sub multiplicativity first var k<l<n}) and the result
follows.

Now if $\tau$ is essentially tempered, let $v\in\C$ be such that $\tau_v=\absdet{}^v\tau$ is tempered. According to the result proved above,
there is some $P_v\in\C[q^{-s},q^s]$ such that
\begin{align*}
\gcd(\pi\times\tau_v,s)=
L(\sigma\times\tau_v,s)L(\sigma^*\times\tau_v,s)M_{\tau_v}(s)P_v(q^{-s},q^{s}).
\end{align*}
Since $\gcd(\pi\times\tau_v,s)=\gcd(\pi\times\tau,s+v)$, $L(\sigma\times\tau_v,s)=L(\sigma\times\tau,s+v)$ and
$M_{\tau_v}(s)=M_{\tau}(s+v)$, the following identity
\begin{align*}
\gcd(\pi\times\tau,s+v)=
L(\sigma\times\tau,s+v)L(\sigma^*\times\tau,s+v)M_{\tau}(s+v)P_v(q^{-s},q^{s})
\end{align*}
holds for all $s$. In particular for $s-v$ we obtain
\begin{align*}
\gcd(\pi\times\tau,s)=
L(\sigma\times\tau,s)L(\sigma^*\times\tau,s)M_{\tau}(s)P_v(q^{-s+v},q^{s-v}),
\end{align*}
hence $\gcd(\pi\times\tau,s)\in
L(\sigma\times\tau,s)L(\sigma^*\times\tau,s)M_{\tau}(s)\C[q^{-s},q^s]$.
\end{proof} 

%% file: EyalKaplan_phdthesis.bbl
\begin{thebibliography}{CKPSS04}

\urlstyle{same}

\bibitem[AGRS10]{AGRS}
A.~Aizenbud, D.~Gourevitch, S.~Rallis, and G.~Schiffmann,
\emph{Multiplicity
  one theorems}, Ann. of Math. \textbf{172} (2010), no.~2, 1407--1434.

\bibitem[Ban98]{Banks}
W.~Banks, \emph{A corollary to {B}ernstein's theorem and {W}hittaker
  functionals on the metaplectic group}, Math. Res. Letters \textbf{5} (1998),
  no.~6, 781--790.

\bibitem[BZ76]{BZ1}
I.~N. Bernstein and A.~V. Zelevinsky, \emph{Representations of the
group
  ${GL(n,F)}$ where ${F}$ is a local non-{A}rchimedean field}, Russian Math.
  Surveys \textbf{31} (1976), no.~3, 1--68.

\bibitem[BZ77]{BZ2}
\bysame, \emph{Induced representations of reductive ${p}$-adic
groups {I}},
  Ann. Scient. \'{E}c. Norm. Sup. \textbf{10} (1977), no.~4, 441--472.

\bibitem[Bum05]{Bu}
D.~Bump, \emph{The {R}ankin--{S}elberg method: an introduction and
survey},
  Automorphic representations, ${L}$-functions and applications: progress and
  prospects (J.~W. Cogdell, D.~Jiang, S.~S. Kudla, D.~Soudry, and R.~Stanton,
  eds.), Ohio State Univ. Math. Res. Inst. Publ., 11, de Gruyter, Berlin, 2005,
  pp.~41--73.

\bibitem[BFF97]{BFF}
D.~Bump, S.~Friedberg, and M.~Furusawa, \emph{Explicit formulas for
the
  {W}aldspurger and {B}essel models}, Israel J. Math. \textbf{102} (1997),
  no.~1, 125--177.

\bibitem[Cas95]{Cs}
W.~Casselman, \emph{Introduction to the theory of admissible
representations of
  ${p}$-adic reductive groups}, 1995, preprint,
  \url{http://www.math.ubc.ca/~cass/research/pdf/p-adic-book.pdf}.

\bibitem[CS98]{CSh}
W.~Casselman and F.~Shahidi, \emph{On irreducibility of standard
modules for
  generic representations}, Ann. Scient. \'{E}c. Norm. Sup. \textbf{31} (1998),
  no.~4, 561--589.

\bibitem[CS80]{CS2}
W.~Casselman and J.~A. Shalika, \emph{The unramified principal
series of
  ${p}$-adic groups {II}: the {W}hittaker function}, Compositio Math. \textbf{41}
  (1980), 207--231.

\bibitem[Cog03]{Co}
J.~W. Cogdell, \emph{Dual groups and {L}anglands functoriality}, An
  introduction to the Langlands program (J.~Bernstein and S.~Gelbart, eds.),
  Birkh\"{a}user, Boston, 2003, pp.~251--268.

\bibitem[Cog05]{Cog2}
\bysame, \emph{Converse theorems, functoriality and applications},
Pure Appl.
  Math. Q. \textbf{1} (2005), no.~2, 341--367.

\bibitem[Cog06]{Cog5}
\bysame, \emph{Lectures on integral representations of
${L}$-functions},
  Lecture notes from the {I}nstructional {W}orkshop: {A}utomorphic {G}alois
  representations, ${L}$-functions and {A}rithmetic held at Columbia, June
  17--22, 2006, \url{http://www.math.osu.edu/~cogdell/columbia-www.pdf}.

\bibitem[CKPSS04]{CKPS}
J.~W. Cogdell, H.H. Kim, I.~I. Piatetski-Shapiro, and F.~Shahidi,
  \emph{Functoriality for the classical groups}, Publ. Math. IHES. \textbf{99}
  (2004), 163--233.

\bibitem[CPS]{CPS2}
J.~W. Cogdell and I.~I. Piatetski-Shapiro, \emph{Derivatives and
{L}-functions
  for ${GL_n}$}, 2008, to appear in The Heritage of B. Moishezon, ICMP,
  \url{http://www.math.osu.edu/~cogdell/moish-www.pdf}.

\bibitem[CPS02]{Cog3}
\bysame, \emph{Converse theorems, functoriality and applications to
number
  theory}, ICM (Beijing, 2002), vol.~2, Higher Ed. Press, Beijing, 2002,
  pp.~119--128.

\bibitem[CPS04]{CPS}
\bysame, \emph{Remarks on {R}ankin--{S}elberg convolutions},
Contributions to
  automorphic forms, geometry, and number theory (H.~Hida, D.~Ramakrishnan, and
  F.~Shahid, eds.), Johns Hopkins Univ. Press, Baltimore, MD, 2004,
  pp.~255--278.

\bibitem[FH04]{FH}
W.~Fulton and J.~Harris, \emph{Representation theory: a first
course},
  Springer, New York, 2004.

\bibitem[GGP12]{GGP}
W.~T. Gan, B.~H. Gross, and D.~Prasad, \emph{Symplectic local root
numbers, central critical {L}-values, and restriction problems in the repr\'{e}sentation
theory of classical groups}, Ast\'{e}risque \textbf{346} (2012), 1--109.

\bibitem[GPSR87]{GPS}
S.~Gelbart, I.~I. Piatetski-Shapiro, and S.~Rallis,
\emph{${L}$-functions for
  ${G\times GL(n)}$}, Lecture Notes in Math., vol. 1254, Springer-Verlag, New
  York, 1987.

\bibitem[Gin90]{G}
D.~Ginzburg, \emph{${L}$-functions for ${SO_n\times GL_k}$}, J.
Reine Angew.
  Math. \textbf{405} (1990), 156--180.

\bibitem[GPSR97]{GPSR}
D.~Ginzburg, I.~I. Piatetski-Shapiro, and S.~Rallis,
\emph{${L}$-functions for
  the orthogonal group}, Mem. Amer. Math. Soc., vol. 128, no. 611, Amer. Math.
  Soc., Providence, R. I., 1997.

\bibitem[GRS01]{GRS3}
D.~Ginzburg, S.~Rallis, and D.~Soudry, \emph{Generic automorphic
forms on
  ${SO(2n+1)}$: functorial lift to ${GL(2n)}$, endoscopy, and base change},
  Internat. Math. Res. Notices. \textbf{2001} (2001), no.~14, 729--764.

\bibitem[GRS11]{RGS}
\bysame, \emph{The descent map from automorphic representations of
{GL(n)} to
  classical groups}, World Scientific Publishing, Singapore, 2011.

\bibitem[GW98]{GW}
R.~Goodman and N.~R. Wallach, \emph{Representations and invariants
of the
  classical groups}, Cambridge University Press, Cambridge, 1998.

\bibitem[HKS96]{HKS}
M.~Harris, S.~S. Kudla, and W.~J. Sweet, \emph{Theta dichotomy for
unitary
  groups}, J. Amer. Math. Soc. \textbf{9} (1996), no.~4, 941--1004.

\bibitem[Ike92]{Ik}
T.~Ikeda, \emph{On the location of poles of the triple
{L}-functions}, Compositio Math. \textbf{83} (1992), no.~2, 187--237.

\bibitem[Ike99]{Ik2}
\bysame, \emph{On the gamma factor of the triple ${L}$-function,
{I}}, Duke
  Math. J. \textbf{97} (1999), no.~2, 301--318.

\bibitem[JPSS79]{JPSS2}
H.~Jacquet, I.~I. Piatetski-Shapiro, and J.~A. Shalika,
\emph{Automorphic forms
  on ${GL(3)}$ {I}}, Ann. of Math. \textbf{109} (1979), no.~1, 169--212.

\bibitem[JPSS83]{JPSS}
\bysame, \emph{{R}ankin--{S}elberg convolutions}, Amer. J. Math.
\textbf{105}
  (1983), no.~2, 367--464.

\bibitem[JS81]{JS1}
H.~Jacquet and J.~A. Shalika, \emph{On {E}uler products and the
classification
  of automorphic representations {I}}, Amer. J. Math. \textbf{103} (1981),
  no.~3, 499--558.

\bibitem[JS83]{JS}
\bysame, \emph{The {W}hittaker models of induced representations},
Pacific. J.
  Math. \textbf{109} (1983), no.~1, 107--120.

\bibitem[JS90]{JS3}
\bysame, \emph{{R}ankin--{S}elberg convolutions: {A}rchimedean
theory},
  Festschrift in Honor of I. I. Piatetski–-Shapiro on the occasion of his
  sixtieth birthday, Part I (S.~Gelbert, R.~Howe, and
  P.~Sarnak, eds.), Israel Math. Conf. Proc., 2, Weizmann Science Press of
  Israel, Jerusalem, 1990, pp.~125--207.

\bibitem[JS03]{JSd1}
D.~Jiang and D.~Soudry, \emph{The local converse theorem for
${SO(2n+1)}$ and
  applications}, Ann. of Math. \textbf{157} (2003), no.~3, 743--806.

\bibitem[Kap10]{E}
E.~Kaplan, \emph{An invariant theory approach for the unramified
computation of
  {R}ankin–-{S}elberg integrals for quasi-split ${SO_{2n}\times GL_n}$}, J.
  Number Theory \textbf{130} (2010), 1801--1817.

\bibitem[Kap12]{me2}
\bysame, \emph{The unramified computation of {R}ankin-{S}elberg
integrals
  for ${SO_{2\smallidx}\times GL_{\bigidx}}$}, Israel J. Math. \textbf{191} (2012), no.~1, 137--184.

\bibitem[Kap13a]{me3}
\bysame, \emph{Multiplicativity of the gamma factors of
{R}ankin–-{S}elberg
  integrals for ${SO_{2\smallidx}\times GL_{\bigidx}}$}, Manuscripta Math. \textbf{142} (2013), 307--346.

\bibitem[Kap13b]{me4}
\bysame, \emph{On the gcd of local {R}ankin-{S}elberg integrals for
even orthogonal groups}, Compositio Math. \textbf{149} (2013),  2013, 587--636.

\bibitem[Kap13c]{me5}
\bysame, \emph{Complementary results on the Rankin-Selberg gamma factors of classical groups}, 2013,
DOI: 10.1016/j.jnt.2013.12.002, to appear in J. Number Theory (Special Issue In Honor of Steve Rallis).

\bibitem[Kat78]{K}
S.~Kato, \emph{On an explicit formula for class-1 {W}hittaker
functions on
  split reductive groups over p-adic fields}, preprint, 1978.

\bibitem[LM09]{LM}
E.~Lapid and Z.~Mao, \emph{On the asymptotics of {W}hittaker
functions},
  Represent. Theory \textbf{13} (2009), 63--–81.

\bibitem[MW95]{MW2}
C.~M{\oe}glin and J.-L. Waldspurger, \emph{Spectral decomposition and
{E}isenstein
  series}, Cambridge University Press, New York, 1995.

\bibitem[MW10]{MW}
\bysame, \emph{La conjecture locale de {G}ross-{P}rasad pour les
groupes
  sp\'{e}ciaux orthogonaux: le cas g\'{e}n\'{e}ral}, Ast\'{e}risque \textbf{347} (2012).

\bibitem[Mui01]{Mu3}
G.~Mui\'{c}, \emph{A proof of {C}asselman-{S}hahidi's conjecture for
  quasi-split classical groups}, Canad. Math. Bull. \textbf{44} (2001),
  298--312.

\bibitem[Mui04]{Mu2}
\bysame, \emph{Reducibility of standard representations for
classical
  ${p}$-adic groups}, Functional analysis VIII (Dubrovnik, 2003), Various Publ.
  Ser., vol.~47, Univ. Aarhus, Aarhus, 2004, pp.~132--–145.

\bibitem[Mui08]{Mu}
\bysame, \emph{A geometric construction of intertwining operators
for reductive
  ${p}$-–adic groups}, Manuscripta Math. \textbf{125} (2008), 241--–272.

\bibitem[OO98]{OO}
A.~Okounkov and G.~Olshanski, \emph{Shifted {S}chur functions {II}.
the
  binomial formula for characters of classical groups and its applications},
  Kirillov's Seminar on Representation Theory (G.~I. Olshanski, ed.), Amer.
  Math. Soc. Transl. Ser. 2, vol. 181, Amer. Math. Soc., Providence, R. I.,
  1998, pp.~245--271.

\bibitem[PSR86]{PR2}
I.~I. Piatetski-Shapiro and S.~Rallis, \emph{${\epsilon}$ factor of
  representations of classical groups}, Proc. Natl. Acad. Sci. \textbf{83}
  (1986), 4589--4593.

\bibitem[PSR87]{PS}
\bysame, \emph{{R}ankin triple ${L}$ functions}, Compositio Math.
\textbf{64}
  (1987), no.~1, 31--115.

\bibitem[Rao93]{Rao}
R. Rao, \emph{On some explicit formulas in the theory of Weil representations},
Pacific J. Math.,
\textbf{157} (1993), 335--371.

\bibitem[Ser73]{Ser}
J.-P. Serre, \emph{A course in arithmetic}, Springer-Verlag, New
York, 1973.

\bibitem[Sha78]{Sh2}
F.~Shahidi, \emph{Functional equation satisfied by certain
${L}$-functions},
  Compositio Math. \textbf{37} (1978), 171--208.

\bibitem[Sha81]{Sh4}
\bysame, \emph{On certain ${L}$-functions}, Amer. J. Math.
\textbf{103} (1981),
  no.~2, 297--355.

\bibitem[Sha84]{Sh6}
\bysame, \emph{Fourier transforms of intertwining operators and
{P}lancherel
  measures for ${GL(n)}$}, Amer. J. Math. \textbf{106} (1984), no.~1, 67--111.

\bibitem[Sha90]{Sh3}
\bysame, \emph{A proof of {L}anglands' conjecture on {P}lancherel
measures;
  complementary series of $\mathfrak{p}$-adic groups}, Ann. of Math.
  \textbf{132} (1990), no.~2, 273--330.

\bibitem[Sha92]{Sh5}
\bysame, \emph{Twisted endoscopy and reducibility of induced
representations
  for ${p}$-adic groups}, Duke Math. J. \textbf{66} (1992), no.~1, 1--41.

\bibitem[Sha74]{Sh}
J.~A. Shalika, \emph{The multiplicity one theorem for ${GL_n}$},
Ann. of Math.
  \textbf{100} (1974), 171--193.

\bibitem[Shi76]{S}
T.~Shintani, \emph{On an explicit formula for class-${1}$
``{W}hittaker
  functions" on ${GL_n}$ over $\mathfrak{P}$-adic fields}, Proc. Japan Acad.
  \textbf{52} (1976), 180--182.

\bibitem[Sil79]{Silb}
A.~J. Silberger, \emph{Introduction to harmonic analysis on reductive p-adic
  groups}, Princeton University Press and University of Tokyo Press, Princeton,
  New Jersey, 1979.

\bibitem[Sop07]{Sopranos}
``Made in America", \emph{{T}he {S}opranos}, created by D. Chase, HBO, United States, June 10, 2007.

\bibitem[Sou93]{Soudry}
D.~Soudry, \emph{{R}ankin--{S}elberg convolutions for
${SO_{2l+1}\times GL_n}$:
  local theory}, Mem. Amer. Math. Soc., vol. 105, no. 500, Amer. Math. Soc.,
  Providence, R. I., 1993.

\bibitem[Sou95]{Soudry3}
\bysame, \emph{On the {A}rchimedean theory of {R}ankin--{S}elberg
convolutions
  for ${SO_{2l+1}\times GL_{n}}$}, Ann. Scient. $\acute{E}c.$ Norm. Sup.
  \textbf{28} (1995), no.~2, 161--224.

\bibitem[Sou00]{Soudry2}
\bysame, \emph{Full multiplicativity of gamma factors for
${SO_{2l+1}\times
  GL_n}$}, Israel J. Math. \textbf{120} (2000), no.~1, 511--561.

\bibitem[Sou06]{Soudry4}
\bysame, \emph{{R}ankin--{S}elberg integrals, the descent method,
and
  {L}anglands functoriality}, ICM (Madrid, 2006), EMS, 2006, pp.~1311--1325.

\bibitem[Spr98]{Spr}
T.~A. Springer, \emph{Linear algebraic groups}, second ed.,
Birkh\"{a}user,
  Boston, 1998.

\bibitem[Tam91]{T}
B.~Tamir, \emph{On ${L}$-functions and intertwining operators for
unitary
  groups}, Israel J. Math. \textbf{73} (1991), no.~2, 161--188.

\bibitem[TT75]{TT1}
T.~Ton-That, \emph{On holomorphic representations of symplectic
groups}, Bull.
  Amer. Math. Soc. \textbf{81} (1975), no.~6, 1069--1072.

\bibitem[TT76]{TT2}
\bysame, \emph{Lie group representations and harmonic polynomials of
a matrix
  variable}, Trans. Amer. Math. Soc. \textbf{216} (1976), 1--46.

\bibitem[Wal03]{W}
J.-L. Waldspurger, \emph{La formule de {P}lancherel pour les groupes
p-adiques,
  d'apr\`{e}s {H}arish-{C}handra}, J. Inst. Math. Jussieu \textbf{2} (2003),
  no.~2, 235--333.

\bibitem[Wei64]{Weil}
A. Weil, \emph{Sur certains groupes d'op\'{e}rateurs unitaires},
Acta Math. \textbf{111} (1964), 143--211.

\bibitem[Zel80]{Z3}
A.~V. Zelevinsky, \emph{Induced representations of reductive
${p}$-adic groups
  {II}. on irreducible representations of ${GL(n)}$}, Ann. Scient. \'{E}c.
  Norm. Sup. \textbf{13} (1980), no.~2, 165--210.

\end{thebibliography}
